\documentclass[12pt]{report}
\usepackage{amsmath,amssymb,amsthm,UOPhDthesis}
\usepackage{amsfonts}
\usepackage{lscape}
\usepackage{latexsym}
\usepackage{proof}
\usepackage{epsfig}
\usepackage{epstopdf}
\usepackage[bookmarksopen=true]{hyperref} 

\newcommand{\LL}{{\cal L}}
\newcommand{\CPM}{{\bf CPM}}

\newcommand{\comp}{\text{\footnotesize$\circ$}}
\newcommand{\ot}{\otimes}
\newcommand{\cA}{{\cal A}}
\newcommand{\cB}{{\cal B}}
\newcommand{\cC}{{\cal C}}
\newcommand{\cD}{{\cal D}}

\newcommand{\cF}{{\cal F}}
\newcommand{\cG}{{\cal G}}

\newcommand{\cI}{{\cal I}}
\newcommand{\cJ}{{\cal J}}
\newcommand{\cK}{{\cal K}}

\newcommand{\cP}{{\cal P}}
\newcommand{\cQ}{{\cal Q}}

\newcommand{\cV}{{\cal V}}
\newcommand{\cW}{{\cal W}}

\newcommand{\tr}{\mathop{\rm tr}\nolimits}
\newcommand{\Tr}{\mathop{\rm Tr}\nolimits}  
\newcommand{\Trc}{{\mathbb T}}              
\newcommand{\cpmTr}{\widetilde{\Tr}{}}      
\newcommand{\cpmTrc}{\widetilde{\Trc}}      

\newcommand{\apair}[1]{\langle#1\rangle}

\newcommand{\stt}[1]{\stackrel{#1}{\longrightarrow}}
\newcommand{\imp}{\Rightarrow}          
\newcommand{\funnel}{                           
    \stackrel{}{\textrm{\settowidth{\unitlength}{$\imp$}%
    \begin{picture}(1,.48)
        \put(.925,-.16){\oval(1.7,.57)[tl]}
        \put(.925,.64){\oval(1.7,.57)[bl]}
    \end{picture}}}}
\newcommand{\funnels}{                          
    \stackrel{}{\textrm{\settowidth{\unitlength}{$\imp$}%
    \begin{picture}(1,.48)
        \put(.5,-.16){\oval(.85,.57)[t]}
        \put(.5,.64){\oval(.85,.57)[b]}
    \end{picture}}}}

\newcommand{\iprod}[1]{\langle#1\rangle}

\input xy
\xyoption{all}
\begin{document}

\title{Categorical models of
computation: partially traced categories and presheaf models of
quantum computation}
\author{Octavio Malherbe, B.Sc., M.Sc.}
\copyrightfalse
\copyrightyear{2010}
\figurespagefalse
\tablespagefalse
\beforepreface
\prefacesection{Abstract}

This dissertation has two main parts. The first part deals with
questions relating to Haghverdi and Scott's notion of partially traced
categories. The main result is a representation theorem for such
categories: we prove that every partially traced category can be
faithfully embedded in a totally traced category. Also conversely,
every monoidal subcategory of a totally traced category is partially
traced, so this characterizes the partially traced categories
completely. The main technique we use is based on Freyd's
paracategories, along with a partial version of Joyal, Street, and
Verity's Int construction. Along the way, we discuss some new examples
of partially traced categories, mostly arising in the context of
quantum computation.

The second part deals with the construction of categorical models of
higher-order quantum computation. We construct a concrete semantic
model of Selinger and Valiron's quantum lambda calculus, which has
been an open problem until now. We do this by considering presheaf
categories over appropriate base categories arising from first-order
quantum computation. The main technical ingredients are Day's
convolution theory and Kelly and Freyd's notion of continuity of
functors. We first give an abstract description of the properties
required of the base categories for the model construction to work;
then exhibit a specific example of base categories satisfying these
properties.

    \prefacesection{Acknowledgements}

I want first to express my very deep gratitude to Phil Scott and Peter Selinger.
This thesis would never have come into existence without their advice and encouragement.
I have also benefited from stimulating discussions with Sergey Slavnov, Beno\^it Valiron and Mark Weber.
I am grateful to my examiners, Richard Blute, Robin Cockett, Pieter Hofstra and Benjamin Steinberg, for their helpful
comments and useful suggestions on this work.
I am in addition particularly indebted to the University of Ottawa and Dalhousie University.
Finally, I would like to thank In\'es and Reina for supporting me throughout my time as a student.

    \prefacesection{}
    \begin{center}
    \textit{To Magdalena}
   \end{center}

\afterpreface

\chapter{Introduction}

Quantum computers are computing devices which are based on the laws of quantum physics. While no actual general-purpose quantum computer has yet been built, research in the last two decades indicates that quantum computers would be vastly more powerful than classical computers. For instance, Shor proved in 1994 that the integer factoring problem can be solved in polynomial time on a quantum computer, while no efficient classical algorithm is known.

The goal of this research is to extend existing connections between logic and computation, and to apply them to the field of quantum computation. Logic has been applied to the study of classical computation in many ways. For instance, the lambda calculus, a prototypical programming language invented by Church and Curry in the 1930's, can be simultaneously regarded as a programming language and as a formalism for writing mathematical proofs. This observation has become the basis for the development of several modern programming languages, including ML, Haskell, and Lisp.

Recent research by Selinger, Valiron, and others has shown that the logical system which corresponds most closely to quantum computation is the so-called ``linear logic" of Girard. Linear logic, a resource sensitive logic, formalizes one of the central principles of quantum physics, the so-called ``no-cloning property", which asserts that a given quantum state cannot be replicated. This property is reflected on the logical side by the requirement that a given logical assumption (or ``resource") can only be used once. However, the correspondence between linear logic and quantum computation has only been established at the syntactic level; it is an important open question how to construct semantic models of higher-order quantum computation.

In a series of fundamental works, Girard has examined dynamical models of proofs in linear logic and their evaluation under normalization, using C*-algebras and functional analysis. This program, which he calls ``The Geometry of Interaction", has recently received increased
attention as having deep connections with quantum computation and quantum protocols. See especially the work of Abramsky and Coecke \cite{Ab coecke A} and of
 Haghverdi and Scott \cite{HS05a}, \cite{HS09}, who have given categorical descriptions of it. Using the work of Joyal, Street and Verity they organize these ideas systematically into a theoretical framework based on the abstract notion of a traced monoidal category.
Scott and Haghverdi showed how these techniques could be re-introduced and extended to handle a typed categorical version of Girard's ``Geometry of Interaction" through the notion of a partially traced category.

One of the objectives of this thesis is to systematically
explore this new notion of partially traced category by providing a representation
theorem which establishes a precise correspondence between partially traced categories and
totally traced categories.  Also, we want to use this framework to elucidate how to build new partially traced categories in connection with some standard models of quantum computation.

A second objective of this thesis is to construct mathematical semantical models of higher-order quantum computation. While the algorithmic aspects of quantum computation have been analyzed extensively, the consideration of quantum computation as a programming paradigm in need of a programming language has been explored far less.

One of the most fruitful methods used to
explore the general idea of computational effect in computer science has been the use
of computational monads in the sense of Moggi. We study models that exhibit this feature based on linear logic, taking insights and inspiration from Day's characterization of
convolution in presheaf categories.  In addition we use Freyd-Kelly's notion of continuous functors, as
well as Selinger's models
for first-order quantum computation.

The basic idea is to start from existing low level models of quantum computation, such as the category of superoperators, and to use a Yoneda type
construction to adapt and extend these models to a higher order quantum situation. The tool used to lift this category is Day's theory for obtaining monoidal structure in presheaf categories. Also, this work partly builds on previous research by Benton et al. on categorical models of linear logic.
More precisely, we give a method for constructing models that depends on a family of possible choices.

Specifically, the model construction depends on a sequence of
 categories and functors $\cB \rightarrow \cC\rightarrow \cD$, and on a family $\Gamma$ of cones
 in $\cD$. We use this data to obtain a pair of adjunctions
 \[\xymatrix{
[\cB^{op},{\bf Set}]\ar@<1ex>[rr]^{L}&&[\cC^{op},{\bf Set}]\ar@<1ex>[rr]^{F}\ar@<1ex>[ll]^{\Phi^{*}}_{\bot}&&[\cD^{op},{\bf Set}]_{\Gamma}
\ar@<1ex>[ll]^{G}_{\bot}}\]
and give sufficient conditions on $\cB \rightarrow \cC\rightarrow \cD$ and $\Gamma$ so that the
 resulting structure is a model of the quantum lambda calculus.

This provides a general framework in which one can describe various
 classes of models that depend on the concrete choice of the parameters $\cB$, $\cC$, $\cD$, and $\Gamma$.

\chapter{Some mathematical background}

The aim of this chapter is to review some basic categorical
 background material that is needed to understand this thesis. For a more detailed discussion, see \cite{MacLane 91}, \cite{Borceux94}, and
 \cite{LambekScott86}. The reader who is already familiar with category theory can
 skip this chapter initially, and refer back to it when needed.

\section{Monads and adjunctions}

In what follows, $Id_{\cC}$ is the identity functor on a category $\cC$ and $1_G$ is the identity natural transformation on a functor $G$. Given a category $\cC$, the symbol $\cC(A,B)$ denotes the set of morphisms from $A$ to $B$.
\begin{definition}[Adjunction]
\rm
Let $\cA$ and $\cB$ be categories. An \textit{adjunction} from $\cA$ to $\cB$ is a quadruple $(F,G,\eta,\varepsilon)$ where $F:\cA\rightarrow\cB$ and $G:\cB\rightarrow\cA$ are functors and $\eta:Id_{\cA}\Rightarrow GF$ and $\varepsilon:FG\Rightarrow Id_{\cB}$ are natural transformations such that: $(G\varepsilon)\circ (\eta G)=1_G$ and $(\varepsilon F)\circ (F\eta)=1_F$. The functor $F$ is said to be a \textit{left adjoint} for $G$ or $G$ a \textit{right adjoint} for $F$ and we use the following notation: $F\dashv G$ or $(F,G,\eta,\varepsilon):\mathcal{A}\rightharpoonup \mathcal{B}$ or even more graphically
$$\xymatrix{
\mathcal{A}\ar@<1ex>[r]^{F}& \mathcal{B.} \ar@<1ex>[l]^{G}_{\bot}}$$.
\end{definition}
\begin{definition}[Monads]
\rm
A \textit{monad} or a \textit{triple} on a category $\cC$ is a 3-tuple $(T, \eta, \mu)$ where $T: \cC \rightarrow \cC$ is an endofunctor and $\eta: Id_\cC \Rightarrow T$ (unit law), $\mu: T^2 \Rightarrow T$ (multiplication law) are two natural transformations,  satisfying the following conditions:
$$\xymatrix{T \circ Id_{\cC} \ar@{=>}[r]^{T\eta} \ar@{=}[dr] & T^2 \ar@{=>}[d]^{\mu} & Id_{\cC}\circ T \ar@{=>}[l]_{\eta T} \ar@{=}[dl] & & T^3 \ar@{=>}[r]^{T \mu} \ar@{=>}[d]_{\mu T} & T^2 \ar@{=>}[d]^{\mu} \\  & T & & & T^2 \ar@{=>}[r]^{\mu} & T }$$

\end{definition}
\begin{theorem}[Huber]\label{HUBER}
If $F\dashv G$ with unit $\eta:Id_{\cA}\Rightarrow GF$ and co-unit $\varepsilon:FG\Rightarrow Id_{\cB}$, then $(GF,\eta,G\varepsilon F)$is a monad on $\cA$.
\end{theorem}
\begin{proof}
See Lambek and Scott \cite{LambekScott86}.
\end{proof}

Suppose we have two adjunctions: $(F,G,\eta,\varepsilon):\mathcal{A}\rightharpoonup \mathcal{B}$ and $(F',G',\eta',\varepsilon'):\mathcal{B}\rightharpoonup \mathcal{C}$
\[\xymatrix{
\cA\ar@<1ex>[rr]^{F}&&\cB\ar@<1ex>[rr]^{F'}\ar@<1ex>[ll]^{G}_{\bot}&&\cC\ar@<1ex>[ll]^{G'}_{\bot}}\]
We can consider the composite:
$(F'F,GG',G\eta'F
\circ \eta,\varepsilon'\circ F'\varepsilon
G'):\mathcal{A}\rightharpoonup \mathcal{C}$ yielding an adjunction from $\cA$ to $\cC$. Therefore, by Theorem~\ref{HUBER}
$(T,\tilde{\eta},\tilde{\mu})$ with $T=GG'F'F$,
$\tilde{\eta}=G\eta'F\circ\eta $ and
$\tilde{\mu}=GG'(\varepsilon'\circ F'\varepsilon G')F'F$ is a monad defined by this new adjunction.

Next we recall the comparison theorem for the Kleisli category.
\begin{definition}\label{Kleisli category}
Given a monad $(T,\eta, \mu)$ on a category $\cC$, the {\em Kleisli category} $\cC_{T}$ is determined by the following conditions:
\begin{itemize}
\item[-] $Obj(\cC_{T})=Obj(\cC)$
\item[-] $\cC_{T}(A,B)=\cC(A,TB)$
\item[-] $f^{K}\circ_{K}g^{K}=\mu_{C}\circ T(g)\circ f$  when  $A\stackrel {f^{K}}\rightarrow B$ and $B\stackrel {g^{K}}\rightarrow C$ are arrows in $\cC_{T}$. The identity is given by $1^{K}_{C}=\eta_{C}:C\rightarrow TC$.
\end{itemize}
\end{definition}
There is an adjunction between the category $\cC$ and the Kleisli
category $\cC_T$, given by the following:
\begin{itemize}
\item $F_{T}(A)=A$ and $F_{T}(f)=\eta_{B}\circ f$  if $A\stackrel {f}\rightarrow B$ is an arrow in $\cC$.
\item $G_{T}(B)=T(B)$ and $G_{T}(f^{K})=\mu_{B}\circ T(f)$  if $A\stackrel {f^{K}}\rightarrow B$ is an arrow in $\cC_{T}$.
\end{itemize}

The adjunction
$F_T\dashv G_T$ has the following universal property: given any other
adjunction $F\dashv G$ such that $G\circ F=T$, there exists a unique
functor
$C:\cC_{T}\rightarrow \cD$, called {\em the comparison functor}, with the following properties $C\circ F_{T}=F$ and $G\circ C=G_{T}$.
\begin{itemize}
\item $C(A)=F(A)$ on objects and $F_{T}(f)=\eta_{B}\circ f$  if $A\stackrel {f}\rightarrow B$ is an arrow in $\cC$.
\item $C(f)=\varepsilon_{FB}\circ F(f)$ when $A\stackrel {f^{K}}\rightarrow B$ is an arrow in $\cC_{T}$.
\end{itemize}

$$\xymatrix@=25pt{
 \cC\ar@<1ex>[dd]^{F_{T}}\ar@<1ex>[rr]^{F}&&\cD\ar@<1ex>[ll]^{G}_{\bot} \\
 &&\\
 \cC_{T}\ar@<1ex>[uu]^{G_{T}}_{\vdash}\ar[rruu]_{C}&&
}$$

First we evaluate the identity: $C(1^{K}_{A})=C(\eta_A)=\varepsilon_{FA}\circ F(\eta_A)=1_{FA}$ by definition of the adjoint pair.

Now, suppose we have $A\stackrel {f^{K}}\rightarrow B $ and $B\stackrel {g^{K}}\rightarrow C$ a pair of arrows in $\cC_{T}$ i.e. a pair $A\stackrel {f}\rightarrow GFB $ and $B\stackrel {g}\rightarrow GFC$ in $\cC$. We want to prove that $C(g\circ_{K}f)=C(g)\circ C(f)$. We have that:

$$\xymatrix@=25pt{
 FA\ar[rrdd]_{F(g\circ_{K}f)=F(\mu_{C}T(g)f)}\ar[rr]^{Ff}&&FGFB\ar[d]_{FTg=FGFg}\ar[rr]^{\varepsilon_{FB}} && FB\ar[d]^{Fg}\\
                              &&FGFGFC\ar[d]^{F(\mu_{C})=FG\varepsilon_{FC}}\ar[rr]^{\varepsilon_{FGFC}}&& FGFC\ar[d]^{\varepsilon_{FC}}\\
                              &&FGFC\ar[rr]_{\varepsilon_{FC}}&&FC
}$$
Where the top square commutes by naturality of $\varepsilon$ with $F(g)$ and the bottom square by naturality of $\varepsilon$ with $\varepsilon_{FC}$.
The top leg of the diagram is $C(g)\circ C(f)$ since $C(f)=\varepsilon_{FB}\circ Ff$ and $C(g)=\varepsilon_{FC}\circ Fg$. The bottom leg is $C(g\circ_{K}f)=\varepsilon_{FC}\circ F(g\circ_{K}f)$.

Notice that the comparison functor is fully faithful. The definition of the functor between hom-sets is given by
$$C:\cC_{T}(A,B)\rightarrow \cD(C(A),C(B))$$
$$(A\stackrel {f}\rightarrow TB)\longmapsto (FA\stackrel {F(f)}\rightarrow  FGFB\stackrel {\varepsilon_{FB}}\rightarrow FB)$$
Therefore define a function $C^{-1}$ by
$$C^{-1}:\cD(FA,FB)\rightarrow \cC_{T}(A,B)$$
$$(FA\stackrel {g}\rightarrow FB)\longmapsto (A\stackrel {\eta_{A}}\rightarrow  GFA\stackrel {G(g)}\rightarrow GFB)$$
i.e., $C^{-1}(g)=G(g)\circ\eta_{A}$.

\section{Monoidal categories}
\begin{definition}
\rm
A \textit{monoidal} category, also often called \textit{tensor} category, is a category $\cV$ with a  unit object $I\in\cV$ together with a bifunctor $\otimes:\cV\times\cV\rightarrow \cV$ and natural isomorphisms $\rho:A\otimes I\stackrel{\cong}\rightarrow A$, $\lambda:I\otimes A\stackrel{\cong}\rightarrow A$, $\alpha:A\otimes (B\otimes C)\stackrel{\cong}\rightarrow (A\otimes B)\otimes C$, satisfying the following coherence axioms:

\[
    \xymatrix@=14pt{
            A\otimes (I\otimes B) \ar[dr]_{1\otimes\lambda}  \ar[rr]^{\alpha}& & (A\otimes I) \otimes B \ar[ld]^{\rho\otimes 1}  \\
                              &  A\otimes B &      }
\]
and
\[
    \xymatrix@=24pt{
 A\otimes (B\otimes (C\otimes D))\ar[d]^{\alpha}\ar[r]^{\alpha}&(A\otimes B)\otimes (C\otimes D) \ar[r]^{\alpha} &  ((A\otimes B)\otimes C)\otimes D \ar[d]^{\alpha}\\
 (A\otimes ((B\otimes C)\otimes D)\ar[rr]_{\alpha}&      &    (A\otimes (B\otimes C))\otimes D.}
 \]

\end{definition}

\begin{definition}
\rm
A \textit{symmetric} monoidal category consists of a monoidal category $(\cV,\otimes,I,\alpha,\rho,\lambda)$ with a chosen natural isomorphism $\sigma:A\otimes B\stackrel{\cong}\rightarrow B\otimes A$, called {\em symmetry}, which satisfies the following coherence axioms:
$$\begin{array}{cc}
\xymatrix@=14pt{
A\otimes B \ar[r]^{\sigma}  \ar[rd]_{id}& B \otimes A \ar[d]^{\sigma}  \\
                              &  A\otimes B
} \hspace{3cm}&
\xymatrix@=14pt{
A\otimes I \ar[r]^{\sigma}  \ar[rd]_{\rho}& I\otimes A \ar[d]^{\lambda}  \\
                              &  A
}\end{array}$$
and
\[
    \xymatrix{
            A\otimes (B\otimes C)\ar[d]^{1\otimes\sigma} \ar[r]^{\alpha}  & (A\otimes B)\otimes C\ar[r]^{\sigma}  & C\otimes (A\otimes B) \ar[d]^{\alpha}  \\
    A\otimes (C\otimes B)\ar[r]^{\alpha} &( A\otimes C)\otimes B \ar[r]^{\sigma\otimes 1} & (C\otimes A)\otimes B.      }
\]
\end{definition}

\begin{definition}
\rm
A \textit{symmetric monoidal closed category} category is a symmetric monoidal category $\cV$ for which each functor $-\otimes B:\cV\rightarrow\cV$ has a right adjoint $[B,-]:\cV\rightarrow\cV$, i.e. :  $$\cV(A\otimes B,C)\cong\cV(A,[B,C]).$$
\end{definition}

\begin{definition}\label{MONOIDAL FUNCTOR}
\rm
A \textit{monoidal functor} $(F,m_{A,B},m_I)$ between monoidal categories $(\cV,\otimes,I,\alpha,\rho,\lambda)$ and $(\cW,\otimes',I',\alpha',\rho',\lambda')$ is a functor $F:\cV\rightarrow\cW$ equipped with:
\begin{itemize}
\item[-]
morphisms $m_{A,B}:F(A)\otimes'F(B)\rightarrow F(A\otimes B)$ natural in $A$ and $B$,
\item[-] a morphism $m_I:I'\rightarrow F(I)$,
\end{itemize}
which satisfy the following coherence axioms:
\[
    \xymatrix{
            FA\otimes' (FB\otimes' FC)\ar[d]^{\alpha'} \ar[r]^{1\otimes' m}  & FA\otimes' F(B\otimes C)\ar[r]^{m}  & F(A\otimes (B\otimes C)) \ar[d]^{F\alpha}  \\
    (FA\otimes' FB)\otimes' FC\ar[r]^{m\otimes' 1} &F( A\otimes B)\otimes' FC\ar[r]^{m} & F((A\otimes B)\otimes C)  }
\]

$$\begin{array}{cc}
\xymatrix@=20pt{
FA\otimes' I' \ar[r]^{\rho'}  \ar[d]_{1\otimes'm}& FA   \\
     FA\otimes' FI  \ar[r]_{m}                       &  F(A\otimes I) \ar[u]^{F\rho}
} \hspace{3cm}&
\xymatrix@=20pt{
I'\otimes' FA \ar[d]^{m\otimes'1} \ar[r]^{\lambda'}  &  FA  \\
 FI\otimes'FA   \ar[r]_{m}        &  F(I\otimes A )\ar[u]^{F(\lambda)}
.}\end{array}$$
\end{definition}

A monoidal functor is \textit{strong} when $m_I$ and for every $A$ and $B$ $m_{A,B}$ are isomorphisms. It is said to be \textit{strict} when all the $m_{A,B}$ and $m_I$ are identities.

\begin{remark}
\rm
Throughout the remainder of this exposition whenever we write $(F,m)$ we symbolize a monoidal functor where $m$ not only represents the natural transformation $m_{A,B}:FA\otimes FB\rightarrow F(A\otimes B)$ but also $m_{I}:I\rightarrow FI$ relating the units of the two monoidal categories.
\end{remark}

\begin{definition}
\rm If $\cV$ and $\cW$ are symmetric monoidal categories with natural symmetry maps $\sigma$ and $\sigma'$, a \textit{symmetric monoidal functor} is a monoidal functor $(F,m_{A,B},m_I)$ satisfying the following axiom:
$$\xymatrix@=20pt{
FA\otimes' FB \ar[r]^{\sigma'}  \ar[d]_{m}& FB\otimes'FA \ar[d]^{m}  \\
     F(A\otimes B)  \ar[r]_{F(\sigma)}                       &  F(B\otimes A)
}$$
\end{definition}

\begin{definition}
\label{MONOIDAL NATURAL TRANFORMATION}
\rm
A \textit{monoidal natural transformation} $\theta:(F,m)\rightarrow (G,n)$ between monoidal functors is a natural transformation
$\theta_A:FA\rightarrow GA$ such that the following axioms hold:
$$\begin{array}{cc}
\xymatrix@=20pt{
FA\otimes' FB \ar[rr]^{m}\ar[d]_{\theta_{A}\otimes'\theta_{B}}
  && F(A\otimes B)\ar[d]^{\theta_{A\otimes B}}\\
GA\otimes' GB \ar[rr]_{n}
  && G(A\otimes B)
}\hspace{3cm}&
\xymatrix@=25pt{
I' \ar[r]^{m_I}\ar[rd]_{n_I}
  & FI\ar[d]^{\theta_I}\\
  & GI.
}\end{array}$$
\end{definition}
\section{Monoidal adjunctions and monoidal monads}
\begin{definition}
\rm
A \textit{monoidal adjunction}
\[\xymatrix{
(\cV,\otimes ,I)\ar@<1ex>[r]^{(F,m)}& (\cW,\otimes',I') \ar@<1ex>[l]^{(G,n)}_{\bot}}\]
between two monoidal categories $\cV$ and $\cW$
consists of an adjunction $(F,G,\eta,\varepsilon)$ in which $(F,m)$ and $(G,n)$ are monoidal functors
and the unit $\eta:Id\Rightarrow G\circ F$ and the counit $\varepsilon:F\circ G\Rightarrow Id$ are monoidal natural transformations, as defined in Definition~\ref{MONOIDAL NATURAL TRANFORMATION}.
\end{definition}

\begin{definition}
\rm
Let $(\cV,\otimes,I)$ be a monoidal category. A \textit{monoidal monad} $(T,\eta,\mu,m)$ on $\cV$ is a monad $(T,\eta,\mu)$ such that the endofunctor $T:\cV\rightarrow\cV$ is a monoidal functor $(T,m)$ with $m_{A,B}:TA\otimes TB \rightarrow T(A\otimes B)$ and $m:I \rightarrow TI$ as coherence maps, and the natural transformations $\eta: Id \Rightarrow T$ and $\mu:T^2\Rightarrow T$ are monoidal natural transformations.
\end{definition}

\begin{lemma}\label{2.3.2a}
  Let $T$ be a monoidal monad. Consider the Kleisli adjunction
  $\xymatrix{\mathcal{C}\ar@<1ex>[r]^{F_{T}}&
    \mathcal{C}_{T}\ar@<1ex>[l]^{G_{T}}_{\bot}}$ as in
  Definition~\ref{Kleisli category}. Then $\cC_{T}$ is a monoidal
  category and $F_{T}\dashv G_{T}$ is a monoidal adjunction, where
  \begin{itemize}
  \item[-] $m^{T}_{A,B}:F_{T}(A)\otimes F_{T}(B)\rightarrow F_{T}(A\otimes B)$ is
given by $\eta:A\otimes B\rightarrow T(A\otimes B)$,
  \item[-] $m^{T}_I=\eta_I:I\rightarrow T(I)$,
  \item[-] $n^{T}_{A,B}:G_{T}(A)\otimes G_{T}(B)\rightarrow G_{T}(A\otimes B)$ is
given by $m_{A,B}:T(A)\otimes T(B)\rightarrow T(A\otimes B)$, and
  \item[-] $n^{T}_I=\eta_I:I\rightarrow T(I)$.
  \end{itemize}
\end{lemma}

\begin{definition}
\rm
A \textit{strong monad} $(T,\eta,\mu,t)$ is a monad $(T,\eta,\mu)$ and a natural transformation $t_{A,B}:A\otimes TB\rightarrow T(A\otimes B)$ called a {\em strength} satisfying the following axioms:
$$\begin{array}{cc}
    \xymatrix{
            I\otimes TA \ar[dr]_{\lambda}  \ar[r]^{t_{I,A}} & T(I\otimes A) \ar[d]^{T(\lambda )}  \\
                              &  T(A)      }\hspace{3cm}&
    \xymatrix{
            A\otimes B \ar[dr]_{\eta_{A\otimes B}}  \ar[r]^{1\otimes \eta_B} & A\otimes TB \ar[d]^{t_{A,B}}  \\
                              &  T(A\otimes B)      }\end{array}$$

\[
    \xymatrix{
 (A\otimes B)\otimes TC\ar[rr]^{t_{A\otimes B,C}} \ar[d]_{\alpha_{A,B,TC}}& & T((A\otimes B)\otimes
 C)\ar[d]^{T(\alpha_{A,B,C})}\\
 A\otimes (B\otimes TC)\ar[r]_{1\otimes t_{B,C}}&   A\otimes T(B\otimes C)   \ar[r]_{t_{A,B\otimes C}}     &    T(A\otimes (B\otimes C))}
 \]

\[ \xymatrix{
A\otimes T^2B \ar[d]^{1\otimes \mu_B} \ar[r]^{t_{A,TB}} & T(A\otimes
TB)\ar[r]^{T(t_{A,B})}& T^2(A\otimes B)\ar[d]^{\mu_{A\otimes B}}\\
A\otimes TB\ar[rr]^{t_{A,B}} & & T(A\otimes B).}\]
\end{definition}

\begin{remark}
\rm
Let $(T,\eta,\mu,m)$ be a symmetric monoidal monad. A strong monad can be defined in which the strength $t_{A,B}$ is given by the following formula:
$$A\otimes TB\stackrel{\eta\otimes 1}\longrightarrow TA\otimes TB\stackrel{m_{A,B}}\longrightarrow T(A\otimes B)$$
see Theorem 2.1 in~\cite{KOCK72}.
\end{remark}

We conclude this section with a
theorem by Kelly.

\begin{proposition}[Kelly]
\label{KELLY STRONG-MONOIDAL ADJ}
Let $(F,m):\cC \rightarrow\cC'$ be a monoidal functor. Then $F$ has a right adjoint $G$ for which the adjunction $(F,m)\dashv (G,n)$ is monoidal if and only if $F$ has a right adjoint $F\dashv G$ and $F$ is strong monoidal.
\end{proposition}
\begin{proof}
Here we give a sketch; see~\cite{IMKELLY86},~\cite{KELLY74} or~\cite{MELLIES} for a detailed proof.
Since we have that $\mathcal{C'}(FA,B)\cong \mathcal{C}(A,GB)$ then there is a unique $n_{A,B}$ and $n_I$ such that:
$$\begin{array}{cc}
\xymatrix@=25pt{
F(GA\otimes GB) \ar[rr]^{F(n_{A,B})}\ar[d]_{m_{GA,GB}^{-1}}
  && FG(A\otimes' B) \ar[d]^{\epsilon_{A\otimes B}}\\
 FGA\otimes' FGB\ar[rr]_{\epsilon_A\otimes \epsilon_B}
  && A\otimes' B
}\hspace{2cm} &
\xymatrix@=25pt{
FI \ar[r]^{F(n_I)}\ar[rd]_{m_{I}^{-1}}
  & FGI' \ar[d]^{\epsilon_{I'}}\\
  & I'
}\end{array}$$
Then using the adjunction we check that this candidates satisfy the definition.
\end{proof}

\section{The finite coproduct completion of a category}\label{The finite coproduct completion of a cat}
We recall some properties of the finite coproduct completion of a
 category. A reference can be found in~\cite{CARBONI-JOHNSTONE}.
\begin{definition}
\rm
Let us consider the category ${\bf FinSet}$ whose objects are finite sets $A=\{a_1,\dots,a_n\}$ and whose arrows are functions. To avoid any problem about the size of this category, we assume without loss of generality that all objects of ${\bf FinSet}$ are
 subsets of a given fixed infinite set; thus ${\bf FinSet}$ can be regarded as
 a small category.
\end{definition}
Note that {\bf FinSet} has finite coproducts and products.
\begin{definition}
\label{DEF FINITE COMPL}
\rm Let $\cC$ be a category. The category $\cC^{+}$ has as its objects finite families of objects of $\cC$: $V=\{V_{a}\}_{a\in A}$, with $A$ a finite set. A morphism from $V=\{V_a\}_{a\in A}$ to $W=\{W_b\}_{b\in B}$ consists of the following two items:
\begin{itemize}
\item[-] a function $\phi:A\rightarrow B$
\item[-] a family $f=\{f_a\}_{a\in A}$ of morphisms of $\cC$\\
$f_{a}:V_{a}\rightarrow W_{\phi(a)}$.
\end{itemize}
\end{definition}

\texttt{Notation}: We shall denote a morphism of $\cC^{+}$ as a pair $F=(\phi,f)$. Moreover, sometimes we write $V^{a}_{b}$ instead of  $(V_{a})_{b}$ to emphasize some particular set index subscript, and in the same way for arrows.

Before we study any possible structure in $\cC^{+}$ we observe that this is really a category. The identity map is given by taking $\phi=id_A$ the identity function on $A$ and
$f_{a}=1_{V_a}$,  the identity map in $\cC$,  for every $a\in A$.\\
 Composition is defined by the following rule: if $F=(\phi,f)$ and $G=(\psi,g)$ then $G\circ_{\cC^{+}}F=(\psi\circ\phi,\{g_{\phi(a)}\circ f_{a}\}_{a\in A})$.

To verify the associative law for the composition we have that if $F=(\phi,f)$, $G=(\psi,g)$ and $H=(\lambda,h)$ then:
\begin{center}
$H\circ (G\circ F)=H\circ (\psi\circ\phi,\{g_{\phi(a)}\circ f_{a}\}_{a\in A})=(\lambda\circ (\psi\circ\phi),\{h_{\psi\circ\phi(a)}\circ (g_{\phi(a)}\circ f_{a})\}_{a\in A})=$

$((\lambda\circ \psi)\circ\phi,\{(h_{\psi(\phi(a))}\circ g_{\phi(a)})\circ f_{a}\}_{a\in A})=
(\lambda\circ\psi,\{h_{\psi(b)}\circ g_{b}\}_{b\in B})\circ F=(H\circ G)\circ F.$
\end{center}

\begin{lemma}
$\cC^{+}$ has finite coproducts.\\
\end{lemma}
\begin{proof}
 On objects we have that if $V=\{V_a\}_{a\in A}$, $W=\{W_b\}_{b\in B}$ then $V\oplus W=\{Z_c\}_{c\in C}$ where $C=A+B$ is the coproduct in ${\bf FinSet}$. We take $Z_{in_{1}(a)}=V_{a}$ and $Z_{in_{2}(b)}=W_{b}$ for every $a\in A$, $b\in B$. Thus, $V\oplus W$ is just a concatenation of families of objects of $\cC$.

Injections maps are defined in the following way:
\begin{center}
$\{V_a\}_{a\in A}\stackrel{i^{1}}\longrightarrow\{Z_c\}_{c\in C}$\,\,\, and\,\,\,  $\{W_b\}_{b\in B}\stackrel{i^{2}}\longrightarrow\{Z_c\}_{c\in C}$
\end{center}
where $i^{1}=(in_{1},Id_{A}^{V})$, $i^{2}=(in_{2},Id_{B}^{W})$ are given by:

\begin{center}
$A\stackrel{in_{1}}\longrightarrow A+B$\,\,\,
$B\stackrel{in_{2}}\longrightarrow A+B$ \mbox{injections in ${\bf FinSet}$}
\end{center}

\noindent
and $Id_{A}^{V}=\{1_a^{V}\}_{a\in A}$, $Id_{B}^{W}=\{1_b^{W}\}_{b\in B}$ where
$V_{a}\stackrel{1^{V}_{a}}\longrightarrow V_{a}$ and $W_{b}\stackrel{1^{W}_{b}}\longrightarrow W_{b}$ are identities in $\cC$.

\texttt{Notation}: Sometimes we shall use $V\oplus W$ for $Z$, so we have the following notation $V\oplus W=\{(V\oplus W)_c\}_{c\in A+B}$.

There is also an initial object that we shall denote by $\epsilon$. It is the empty family of objects. The unique morphism $\epsilon \stackrel{\epsilon_{W}}\longrightarrow \{W_b\}_{b\in B}$ is given by $\epsilon_{W}=(\emptyset,\emptyset)$.

\end{proof}

With any category $\cC$, we associate a functor $I:\cC\rightarrow \cC^{+}$ as follows, $I(V)=\{V_{*}\}_{*\in 1}$, $V_{*}=V$ and when there is a $f:V\rightarrow W$ in $\cC$ then $I(f)=(id_{1},\{f_{*}\}_{*\in 1})$ with $f_{*}=f$.

\begin{proposition}
\label{COCOMPLETION OF A CATEGORY}
Given any category $\cA$ with finite coproducts $\,\,\coprod\,\,$ and any functor $F:\cC\rightarrow\cA$, there is a unique finite coproduct preserving functor $G:\cC^{+}\rightarrow \cA$, up to natural isomorphism, such that $G\circ I =F$.
$$\xymatrix@=25pt{
\cC\ar[d]_{I}\ar[r]^{F}&\cA\\
\cC^{+}\ar[ru]_{G}&
}$$
\end{proposition}
\begin{proof}
We shall begin by considering the definition of the functor $G:\cC^{+}\rightarrow \cA$ that assigns to each object $V=\{V_a\}_{a\in A}$ the coproduct $G(\{V_a\}_{a\in A})=\coprod _{a\in A}F(V_a)$ in the category $\cA$. For any arrow $\{V_{a}\}_{a\in A}\stackrel{(\phi,f)}\longrightarrow \{W_{b}\}_{b\in B}$ we define $G(\phi,f)=[i_{F(W_{\phi(a)})}\circ F(f_a)]_{a\in A}$ as the unique arrow in $\cA$ such that the following diagram commutes:

$$\xymatrix@=25pt{
F(V_a)\ar[d]_{i_{F(V_a)}}\ar[r]^{F(f_a)}&F(W_{\phi(a)})\ar[d]^{i_{F(W_{\phi(a)})}}\\
\coprod_{a\in A}F(V_a)\ar@{-->}[r]_{G(\phi,f)}&\coprod_{b\in B}F(W_b)
}$$
We must show that $G$ is a functor. To see this, suppose we have
$$\{V_a\}_{a\in A}\stackrel{(\phi,f)}\longrightarrow \{W_b\}_{b\in B}\stackrel{(\psi,g)}\longrightarrow\{Z_c\}_{c\in C}$$
then by hypothesis
$$\xymatrix@=25pt{
F(W_b)\ar[d]_{i_{F(W_b)}}\ar[r]^{F(g_b)}&F(Z_{\psi(b)})\ar[d]^{i_{F(Z_{\psi(b)})}}\\
\coprod_{b\in B}F(W_b)\ar[r]_{G(\psi,g)}&\coprod_{c\in C}F(Z_c)
}$$
therefore using the case $b=\phi(a)$ we obtain

$$\xymatrix@=25pt{
F(V_a)\ar[d]_{i_{F(V_a)}}\ar[r]^{F(f_a)}&F(W_{\phi(a)})\ar[d]^{i_{F(W_{\phi(a)})}}\ar[r]^{F(g_{\phi(a)})}&F(Z_{\psi(\phi(a))})\ar[d]^{i_{F(Z_{\psi(\phi(a))})}}\\
\coprod_{a\in A}F(V_a)\ar[r]_{G(\phi,f)}&\coprod_{b\in B}F(W_b)\ar[r]_{G(\psi,g)}&\coprod_{c\in C}F(Z_c)
}$$
then by unique existence property of coproducts we have that $G((\psi,g)\circ(\phi,f))=G(\psi,g)\circ G(\phi,f)$. Also by uniqueness it is easily to check that $G(id_A,id^V_A)=id_{\coprod_{a\in A}F(V_a)}$.

The functor $G$ preserves coproducts. To see this let us consider $V^{i}=\{V^{i}_{a}\}_{a\in A_i}$, $i\in I$ then
$$G(\oplus _{i\in I}V^{i})=G(\oplus _{i\in I}\{V^{i}_{a}\}_{a\in A_i})=G(\{Z_c\}_{c\in \oplus_{i\in I}A_i})=\coprod_{c\in \oplus_{i\in I}A_i}F(Z_c)\cong\coprod_{i\in I}(\coprod_{a\in A_i} F(V^i_a))=$$
$$=\coprod_{i\in I}G(\{V^i_a\}_{a\in A_i})=\coprod_{i\in I}G(V^i)$$
with $Z_c=V^i_a$ if  $in_{A_i}(a)=c$.
It remains to verify that $G$ is unique up to natural isomorphism. Suppose there is another $H$ preserving coproducts such that $H\circ I=F$.
Therefore, using the definitions given above of coproduct in $\cC^{+}$, the functor $G$ and the fact that by hypothesis $H$ preserves coproducts, we calculate on objects
$$H(\{V_a\}_{a\in A})\cong H(\oplus_{a\in A}\{V^a_{*}\}_{*\in 1})\cong \coprod_{a\in A}H(\{V^a_{*}\}_{*\in 1})=$$
$$=\coprod_{a\in A}H(I(V_a))=\coprod_{a\in A}F(V_a)=G(\{V_a\}_{a\in A})$$

Suppose we have a morphism $\{V_{a}\}_{a\in A}\stackrel{(\phi,f)}\longrightarrow \{W_{b}\}_{b\in B}$ with $\phi:A\rightarrow B$ and $f=\{f_a\}_{a\in A}$ then using the coproduct in $\cC^{+}$ we consider a decomposition of it, up to isomorphism, in the following way
$$\oplus_{a\in A}\{V^a_{*}\}_{*\in 1}\stackrel{[i_{W_{\phi(a)}}\circ I(f_a)]_{a\in A}}\longrightarrow \oplus_{b\in B}\{W^b_{*}\}_{*\in 1}$$
these morphisms are explicitly given by
$$\{V^a_{*}\}_{*\in 1}\stackrel{I(f_a)}\longrightarrow \{W^{\phi(a)}_{*}\}_{*\in 1}\stackrel{i_{W_{\phi(a)}}}\longrightarrow \oplus_{b\in B}\{W^b_{*}\}_{*\in 1}$$
where $I(f_a)=(id_1,\{f^a_{*}\}_{*\in 1})$, $i_{W_{\phi(a)}}=(in_{\phi(a)},\{1^{W^{\phi(a)}}_{*}\}_{*\in 1})$ with
$1\stackrel{in_{\phi(a)}}\longrightarrow \oplus_{B}1$,  $W_{\phi(a)}\stackrel{1^{W^{\phi(a)}}_{*}=1}\longrightarrow W_{\phi(a)}$ and
$\oplus_{b\in B}\{V^b_{*}\}_{*\in 1}=\{Z_c\}_{c\in \oplus_{B}1}$ with $Z_{in_{b}(*)}=W^{b}_{*}=W_b$.

Since $H$ preserves coproducts
$$H([i_{W_{\phi(a)}}\circ I(f_a)]_{a\in A})\cong[H(i_{W_{\phi(a)}})\circ H(I(f_a))]_{a\in A}=[i_{F(W_{\phi(a)})}\circ F(f_a)]_{a\in A}=G(\phi,f)$$
where the second equality is justified by the following
$$H(\{W^{\phi(a)}_{*}\}_{*\in 1})\stackrel{H(i_{W_{\phi(a)}})}\longrightarrow H(\oplus_{b\in B}\{W^b_{*}\}_{*\in 1})$$
hence using again that $H$ preserves coproducts, up to isomorphism, we have
$$H(\{W^{\phi(a)}_{*}\}_{*\in 1})\stackrel{H(i_{W_{\phi(a)}})}\longrightarrow \coprod_{b\in B}H(\{W^b_{*}\}_{*\in 1})$$
this means by definition of the functor $I$,
$$H(I(W_{\phi(a)}))\stackrel{H(i_{W_{\phi(a)}})}\longrightarrow \coprod_{b\in B}H(I(W_b))$$
but, by hypothesis we know that $H\circ I=F$,
$$F(W_{\phi(a)})\stackrel{i_{F(W_{\phi(a)})}}\longrightarrow \coprod_{b\in B}F(W_b)$$
\end{proof}

\begin{corollary}
$\cC^{+}$ is the free finite coproduct completion generated by $\cC$.
\end{corollary}
\begin{proposition}
\label{MONO FINITE COMPL}
If $\cC$ is a symmetric monoidal category then $\cC^{+}$ is also a symmetric monoidal category.
\end{proposition}
\begin{proof}
Assume that $V=\{V_a\}_{a\in A}$ and $W=\{W_b\}_{b\in B}$ are objects in $\cC^{+}$ then we take $V\otimes_{\cC^{+}} W=\{V_a\otimes W_b\}_{(a,b)\in A\times B}$ where $A\times B$ is the finite product of sets.

The tensor extends to morphisms, if $V\stackrel{F}\longrightarrow X$, $W\stackrel{G}\longrightarrow Y$, with $X=\{X_c\}_{c\in C}$, $Y=\{Y_d\}_{d\in D}$, $F=(\phi,f)$, $G=(\psi,g)$ then $F\otimes G=(\phi\times\psi,f\bar{\otimes} g)$ is given by the following data:
\begin{itemize}
\item[-]$\phi\times \psi:A\times B\rightarrow C\times D$, $(\phi\times \psi)(a,b)=(\phi(a),\psi(b))$
\item[-]$f\bar{\otimes} g=\{(f\bar{\otimes} g)_{(a,b)}\}_{(a,b)\in A\times B}$ where we have that\\
$(f\bar{\otimes} g)_{(a,b)}:(V\otimes W)_{(a,b)}\longrightarrow (X\otimes Y)_{(\phi\times\psi)(a,b)}$
is defined by:
$$f_a\otimes g_b:V_a\otimes W_b\longrightarrow X_{\phi(a)}\otimes Y_{\psi(b)}$$
\end{itemize}
To prove that $-\otimes_{\cC^{+}}-: \cC^{+}\times\cC^{+}\rightarrow\cC^{+}$ is a bifunctor one first calculates the definition by using that $1_{A\times B}=1_A\times 1_B$ and $1_{V_a}\otimes 1_{W_a}=1_{V_a\otimes W_b}$.

Next, we shall prove that $(F\circ F')\otimes (G\circ G')=(F\otimes G)\circ (F'\otimes G')$. Suppose:
$F'=(\phi,f)$, $F= (\eta,h)$, $G'=(\psi,g)$, $G=(\xi,k)$ where
$$\{V_a\}_{a\in A}\stackrel{(\phi,f)}\longrightarrow\{X_c\}_{c\in C} \stackrel{(\eta,h)}\longrightarrow\{Z_e\}_{e\in E}$$
and
$$\{W_b\}_{b\in B}\stackrel{(\psi,g)}\longrightarrow\{Y_d\}_{d\in D} \stackrel{(\xi,k)}\longrightarrow\{H_f\}_{f\in F}$$
Therefore, $(F\circ F')\otimes (G\circ G')=((\eta\circ\phi)\times(\xi\circ\psi),\{(h_{\phi(a)}\circ f_a)\otimes (k_{\psi(b)}\circ g_b)\}_{(a,b)\in A\times B})=((\eta\times\xi)\circ (\phi\times \psi),\{(h_{\phi(a)}\otimes k_{\psi(b)})\circ (f_a\otimes g_b)\}_{(a,b)\in A\times B})=(F\otimes G)\circ (F'\otimes G')$
where we simplify the notation of the tensor symbol.
The unit of the tensor is given by $I=\{I_{*}\}_{* \in \{*\}}$.
The tensor functor is equipped with the following set of isomorphisms:
\begin{itemize}
\item[-] $V\otimes I \stackrel{\bar{\rho}}\longrightarrow V$, and $I\otimes V \stackrel{\bar{\lambda}}\longrightarrow V$ where $V=\{V_a\}_{a\in A}$, $I=\{I_{*}\}_{* \in \{*\}}$ then $V\otimes I=\{V_a\otimes I_{*}\}_{(a,*)\in A\times 1}$.\\
These maps are given by: $\bar{\rho}=(\rho,\textsl{r})$ with $\rho:A\times \{*\}\rightarrow A$, $\rho(a,*)=a$ and with $\textsl{r}=\{r_{(a,*)}\}_{(a,*)\in A\times 1}$ where $r_{(a,*)}=r_{V_a}$,
$ V_a\otimes I \stackrel{r_{V_a}}\longrightarrow  V_a$. In an analogous way is defined $\bar{\lambda}=(\lambda,\textsl{l})$.
\item[-]If $V=\{V_a\}_{a\in A}$ and $W=\{W_b\}_{b\in B}$ then $\bar{\sigma}=(\sigma, \textsl{s})$ with
$\sigma:A\times B\rightarrow B\times A$, $\sigma(x,y)=(y,x)$ and $\textsl{s}=\{s_{(x,y)}\}_{(x,y)\in A\times B}$, where $s_{(x,y)}=s$ i.e,
$$V_x\otimes W_y\stackrel{s}\longrightarrow W_y\otimes V_x$$
\item[-]If $V=\{V_a\}_{a\in A}$, $W=\{W_b\}_{b\in B}$, $Z=\{Z_c\}_{c\in C}$, then $\bar{\alpha}=(\alpha, \textsl{a})$ with $\alpha:A\times (B\times C)\rightarrow (A\times B)\times C$, $\alpha(x,(y,z))=((x,y),z)$ and $\textsl{a}=\{a_{(x,(y,z))}\}_{(x,(y,z))\in A\times(B\times C)}$, where $a_{(x,(y,z))}=a$ i.e.,
$$V_x\otimes (W_y\otimes Z_z)\stackrel{a}\longrightarrow
(V_x\otimes W_y)\otimes Z_z$$
\end{itemize}
Coherence follows by definition, coherence in {\bf FinSet} and coherence in the symmetric monoidal category $\cC$.
\end{proof}
\begin{remark}
\rm
Notice that the distributivity condition $V\otimes(W\oplus Z)\cong(V\otimes W)\oplus (V\otimes Z)$ is satisfied with the map:
$$D:V\otimes(W\oplus Z)\rightarrow(V\otimes W)\oplus (V\otimes Z)$$
where $V=\{V_a\}_{a\in A}$, $W=\{W_b\}_{b\in B}$, $Z=\{Z_c\}_{c\in C}$, $D=(\delta,Id)$ in which $\delta$ is the bijective function $\delta:(A+B)\times C\rightarrow(A\times C)+(B\times C)$ and $Id=\{1_{d}\}_{d\in(A+B)\times C}$.
\end{remark}
\begin{example}
If $\textbf{1}$ is the one object, one arrow strict symmetric monoidal category with the evident monoidal structure then $\textbf{1}^{+}\cong{\bf FinSet}$ and $\otimes_{\textbf{1}^{+}}=\times$ and $I=1$.
\end{example}

\begin{proposition}
\label{MONOIDAL FUNCTOR PSI}
 Under the hypotheses of Proposition~\ref{COCOMPLETION OF A CATEGORY}, assume that the categories $\cC$ and $\cA$ are symmetric monoidal. Then $I$ is a symmetric monoidal functor. If moreover $F$ is a symmetric monoidal functor and tensor distributes over coproducts in $A$, then $G$ is a symmetric monoidal functor. Moreover, if $F$ is strong monoidal then so is $G$.
\end{proposition}
\begin{proof}
We first show that \textsl{I} is a monoidal functor by considering:
$$\textsl{I}(V)\otimes \textsl{I}(W)\stackrel{u}\longrightarrow I(V\otimes W)$$
where $V=\{V_*\}_{*\in 1}$, $W=\{W_*\}_{*\in 1}$ and $u=(\mu,\{1^{V}_{*}\otimes 1^{W}_{*}\}_{(*,*)\in 1\times 1})$ with $\mu:1\times 1\rightarrow 1$ and $1^{V}_{*}\otimes 1^{W}_{*}=1_V\otimes 1_W$. It is easy to check that all the axioms of the definition are satisfied. As an example we have that by routine calculations the following axiom is satisfied:
$$\xymatrix@=25pt{
\{V_*\otimes I_*\}_{(*,*)\in 1\times 1}\ar[d]_{1\otimes 1}\ar[rrr]^{(\rho,\{r_{*,*}\}_{(*,*)\in 1\times 1})}&&&\{V_*\}_{*\in 1}\\
\{V_*\otimes I_*\}_{(*,*)\in 1\times 1}\ar[rrr]^{u}&&&\{(V\otimes I)_*\}_{*\in 1}\ar[u]_{I(r)}
}$$
since $\rho_I=\mu$ and $(r_{V})_{\mu(*,*)}=(r_{V})_{*}=r_V=r_{*,*}$.

Next assuming that $(F,m)$ is monoidal we wish to show that $G$ is also a monoidal functor.

Since we also assumed that the category $(\cA,\rho^{\otimes},\lambda^{\otimes},\alpha^{\otimes},\rho^{\oplus},\lambda^{\oplus},\alpha^{\oplus},\delta,\sigma^{\otimes},\sigma^{\oplus},\lambda^{0},\rho^{0},I,\textsl{0})$ is  symmetric distributive then there exists isomorphisms of type:
$$\phi:(\coprod_{a\in A}F(V_a))\otimes (\coprod_{b\in B}F(W_b))\stackrel{\cong}\rightarrow\coprod_{a\in A,b\in B}F(V_a)\otimes F(W_b)$$
generated by these isomorphisms.\\
We consider the unique arrow $\xi$ given by the universal property of the coproduct:
\begin{equation}
\label{equation: psi}
\xymatrix@=25pt{
F(V_a)\otimes F(W_b)\ar[d]_{m_{V_a,W_b}}\ar[rrrr]^{i_{a,b}}&&&&\coprod_{a\in A,b\in B}F(V_a)\otimes F(W_b)\ar[d]^{\xi=[m_{V_a,W_b}\circ j_{a,b}]_{a\in A,b\in B}}\\
F(V_a\otimes W_b)\ar[rrrr]^{j_{a,b}}&&&&\coprod_{a\in A,b\in B}F(V_a\otimes W_b)
}
\end{equation}
Using these maps we define the mediating arrow $\vartheta:G(V)\otimes G(W)\rightarrow G(V\otimes W)$ as the composition $\vartheta_{V,W}=\xi\circ\phi$. We also have that $\vartheta_{I}:I\rightarrow G(I)$ is given by $m_I$.

To show that $\vartheta$ satisfies the axioms of a symmetric monoidal functor we shall only provide the proof of one of the diagrams. This is justified by obvious coproduct properties: the exterior diagram commutes for every $a\in A$ and this implies that the interior diagram commutes by pre-composing with injections $i$ and using the universal property of coproducts:

$$\xymatrix@=25pt{
F(V_a)\otimes I\ar[rrrr]^{\rho}\ar[ddd]_{1\otimes m_{I}}\ar[dr]_{i}&&&&F(V_a)\ar[ld]^{i}\\
&\coprod_{a\in A}(F(V_a)\otimes I)\ar[d]_{[i\circ (1\otimes m_I)]_{a\in A}}\ar[rr]^{[i_{F(V_a)}\circ\rho]_{a\in A}}&&\coprod_{a\in A}F(V_a)&\\
&\coprod_{a\in A}(F(V_a)\otimes FI)\ar[rr]^{\xi}&&\coprod_{a\in A}F(V_a\otimes I)\ar[u]_{[i_{F(V_a)}\circ F(\rho)]_{a\in A}}&\\
F(V_a)\otimes FI\ar[ru]_{i}\ar[rrrr]_{m_{V_a,I}}&&&&F(V_a\otimes I)\ar[ul]_{i}\ar[uuu]_{F(\rho)}&\\
}$$
Then by coherence \cite{Laplaza72}, distributivity of the tensor through coproduct:
$$\xymatrix@=25pt{
(A\coprod B)\otimes I\ar[rrd]^{\rho}\ar[rr]^{\delta}&&(A\otimes I)\coprod (B\otimes I)\ar[d]^{\rho\coprod\rho}\\
&&A\coprod B
}$$
naturality and by definition of $\vartheta$ we may infer that:
$$\xymatrix@=25pt{
(\coprod F(V_a))\otimes I\ar[rrrd]^{\rho}\ar[ddd]_{(\coprod_{a\in A}1)\otimes m_{I}}\ar[dr]_{\bar{\delta}}&&&&\\
&\coprod_{a\in A}(F(V_a)\otimes I)\ar[d]_{\coprod_{a\in A}(1\otimes m_I)}\ar[rr]_{\coprod_{a\in A}\rho_{F(V_a)}}&&\coprod_{a\in A}F(V_a)&\\
&\coprod_{a\in A}(F(V_a)\otimes FI)\ar[rr]^{\xi}&&\coprod_{a\in A}F(V_a\otimes I)\ar[u]_{\coprod_{a\in A}F(\rho_{V_a})}&\\
(\coprod F(V_a))\otimes FI\ar[ru]^{\phi}\ar[rrru]_{\vartheta_{V,I}}&&&&&\\
}$$
commutes,
which turns to be:
$$\xymatrix@=25pt{
G(V)\otimes I\ar[d]^{1\otimes \vartheta_I}\ar[rr]^{\rho}&&G(V)\\
G(V)\otimes G(I)\ar[rr]^{\vartheta_{V,I}}&&G(V\otimes I)\ar[u]_{G(\rho)}
}$$
Similarly one could prove the rest of the axioms.

Notice that if the mediating arrows $m_{V_a,W_b}$ are isomorphisms in diagram~(\ref{equation: psi}) above then $\xi$ is an isomorphism. Therefore this implies that $\vartheta_{V,W}$ is an isomorphism for every $V$ and $W$ i.e., $G$ is a strong functor.
\end{proof}

\section{The functor $\Phi:{\bf FinSet}\rightarrow \cC^{+}$.}\label{THE FUNCTOR PHI FROM SETS}
Now we turn to prove that when $\cC$ is affine, there exists a functor $\Phi:{\bf FinSet}\rightarrow \cC^{+}$ which is fully faithful and preserves tensor and coproduct.

\begin{definition}\label{AFFINE CATEGORY}
 A monoidal category $\cC$ is called {\em affine} if
 the tensor unit $I$ is a terminal object.
\end {definition}
\begin{lemma}\label{fully-faithful strong monoidal functor PHI}
Let $\cC$ be an affine category. Then there exists a fully-faithful strong monoidal functor  $\,\,\,\Phi:({\bf FinSet},\times,1)\rightarrow (\cC^{+},\otimes_{\cC^{+}},I)$ that preserves coproducts.
\end{lemma}
\begin{proof}
We shall begin by considering the functor $\Phi$ which assigns to each finite set $A$ a family
$\Phi(A)=\{C_a\}_{a\in A}$, such that for every $a\in A$, $C_a=I$ is the unit of the category $\cC$.\\
Now let $A\stackrel{\phi}\longrightarrow B$ be a function in ${\bf FinSet}$, then $$\Phi(A)\stackrel{\Phi(\phi)}\longrightarrow \Phi(B)\,\,\mbox{with}\,\,\,
\Phi(\phi)=(\phi,Id_{A})\,\,\mbox{and}\,\,\,\,Id_{A}=\{1_a\}_{a\in A},\,\,\,\,\,I\stackrel{1_a=id_{I}}\longrightarrow I.$$
The kind of functor obtained in this way has been motivated in order to satisfy the following properties which are essential for the model.

\textsl{$\Phi$ is faithful:}
The way we define morphisms in $\cC^{+}$ allows us to infer that if $\Phi(\phi)=\Phi(\psi)$ then $\phi=\psi$.

\textsl{$\Phi$ is full:}
Suppose we have a pair $(\phi,f)\in \cC^{+}(\Phi(A),\Phi(B))$ then $f=\{f_{a}\}_{a\in A}$ with $I\stackrel{f_a}\longrightarrow I$; since $I$ is a terminal object this implies that $f_a=1_a=\,\,\,!\,\,\,$ for every $a\in A$. Therefore $\Phi(\phi)=(\phi,f)$.

\textsl{$\Phi$ preserves coproducts:}

Take objects $A$ and $B$; then by definition we have that
$$\Phi(A\oplus B)=\{C_c\}_{c\in A\oplus B}=\{C_a\}_{a\in A}\oplus\{C_b\}_{b\in B}=\Phi(A)\oplus \Phi(B).$$

Suppose we have two arrows $A\stackrel{\phi}\longrightarrow C$, $B\stackrel{\psi}\longrightarrow D$ then:
$$\Phi(\phi\oplus \psi)=(\phi\oplus \psi,Id_{A\oplus B})=(\phi\oplus \psi,Id_{A}\oplus Id_{B})\stackrel{def.}=(\phi,Id_A)\oplus (\psi,Id_B)=\Phi(\phi)\oplus \Phi(\psi)$$

\textsl{$\Phi$ preserves tensor product:}

Assuming $A$ and $B$ are finite sets then
$$\Phi(A\times B)=\{C_{(a,b)}\}_{(a,b)\in A\times B}=\{C_{a}\otimes C_{b}\}_{(a,b)\in A\times B}=\{C_{a}\}_{a\in A}\otimes \{C_{b}\}_{b\in B}=\Phi(A)\otimes\Phi(B)$$
at the level of objects.
If $A\stackrel{\phi}\longrightarrow C$, $B\stackrel{\psi}\longrightarrow D$ then we have that naturality is satisfied:
$$\Phi(\phi\times \psi)=(\phi\times \psi,Id_{A\times B})=(\phi\times \psi,Id_{A}\bar{\otimes} Id_{B})=(\phi,Id_A)\otimes (\psi,Id_B)=\Phi(\phi)\otimes \Phi(\psi)$$
since $Id_{A}\bar{\otimes} Id_{B}=\{(1\otimes 1)_{(a,b)}\}_{(a,b)\in A\times B}=\{1_a\otimes 1_b\}_{(a,b)\in A\times B}=\{1_{(a,b)}\}_{(a,b)\in A\times B}=Id_{A\times B}$.

Also, $\Phi(1)=\Phi(\{*\})=\{C_{*}\}_{*\in 1}=I_{\cC^{+}}$.

This implies that $\Phi$ is a monoidal functor with identity $id:\Phi(A)\otimes\Phi(B)\rightarrow\Phi(A\times B)$, $id:I\rightarrow\Phi(1)$ as mediating natural transformations. It is a routine exercise to show that the remaining equations of a monoidal functor, involving the structural maps $\alpha$, $\rho$ and $\lambda$, are satisfied.

For example, the diagram
$$\xymatrix@=25pt{
\Phi(B)\otimes I\ar[d]_{1\otimes 1}\ar[r]^{\bar{\rho}}&\Phi(B)\\
\Phi(B)\otimes \Phi(1)\ar[r]^{1}&\Phi(B\times 1)\ar[u]_{\Phi(\rho)}
}$$
is satisfied. To see this, we calculate $\Phi(\rho)=(\rho,\{1_{(a,*)}\}_{(a,*)\in A\times 1})$. On the other hand by definition we have that
$\bar{\rho}=(\rho,\textsl{r})$ with $\rho:A\times \{*\}\rightarrow A$, $\rho(a,*)=a$ and with $\textsl{r}=\{r_{(a,*)}\}_{(a,*)\in A\times 1}$ where $r_{(a,*)}=r_{V_a}$,
$ V_a\otimes I \stackrel{r_{V_a}}\longrightarrow  V_a$ but since $V_a=I$ this implies $I\stackrel{r_{V_a}=1_{I}}\longrightarrow I$. Hence, these two arrows are equal.
\end{proof}

\section{Affine monoidal categories}\label{FREE AFFINE MONOIDAL CATEGORY}
Recall from Definition~\ref{AFFINE CATEGORY} that a monoidal category is affine when
 the tensor unit $I$ is a terminal object. The following construction is well-known.

\begin{definition}[Free affine symmetric monoidal category]
\rm
Let $\cK$ be a category. The \textit{free affine symmetric monoidal category} $\cF wm(\cK)$ is the category defined as follows:
\begin{itemize}
\item[(a)] objects are finite sequences of objects of $\cK$:
$$\{V_{i}\}_{i\in [n]}=\{V_1,\dots,V_n\}$$
\item[(b)] maps $(\phi,\{f_i\}_{i\in [m]}):\{V_{i}\}_{i\in [n]}\longrightarrow \{W_{i}\}_{i\in [m]}$ are determined by:
\begin{itemize}
\item[-] an injective function $\phi:[m]\rightarrow [n]$
\item[-] a family of morphism $f_{i}:V_{\phi(i)}\rightarrow W_{i}$ in the category $\cK$
\end{itemize}
\item[(c)] composition $(\phi,\{f_i\}_{i\in [m]})\circ(\psi,\{g_i\}_{i\in [s]})=(\psi\circ\phi,\{f_i\circ g_{\phi(i)}\}_{i\in [s]})$
\item[(d)] the unit  is given by the empty sequence.
\item[(e)] the tensor $\otimes$ is given by concatenation of sequences of objects and arrows:
$$\{V_{i}\}_{i\in [n]}\otimes\{W_{i}\}_{i\in [m]}=\{Z_{i}\}_{i\in [n+m]}$$
where $Z_i=V_i$ if $1\leq i\leq n$ and $Z_i=W_{i-n}$ if $n+1\leq i\leq n+m$
$$\{V_{i}\}_{i\in [n]}\otimes\{W_{i}\}_{i\in [m]}=\{P_{i}\}_{i\in [n+m]}\stackrel{(\phi,f)\otimes(\psi,g)}\longrightarrow \{Q_{i}\}_{i\in [s+t]}=\{X_{i}\}_{i\in [s]}\otimes\{Y_{i}\}_{i\in [t]}$$
given by $(\phi,f)\otimes(\psi,g)=(\phi+\psi,f+g)$ where $\phi+\psi:[s+t]\rightarrow[n+m]$ is defined by $(\phi+\psi)(i)=\phi(i)$ if $1\leq i\leq s$ and $(\phi+\psi)(i)=\psi(i-s)+n\,\,\,$ if $s+1\leq i\leq s+t$ and $f+g=\{(f+g)_j\}_{j\in [s+t]}$ where $(f+g)_j:P_{(\phi+\psi)(j)}\rightarrow Q_j$ is defined by $(f+g)_j=f_j$ if $1\leq j\leq s$ and $(f+g)_j=g_{j-s}$ if $s+1\leq j\leq s+t$

\item[(f)] the canonical isomorphisms are strict given by $\textsl{l}=\textsl{r}=1$, $a=1$ and symmetries by $s=(\sigma,1)$ with $\sigma:[n+m]\rightarrow[n+m]$ such that $\sigma(i)=i+n$ if $1\leq i\leq m$ and $\sigma(i)=i-m$ if $m+1\leq i\leq n+m$.
\end{itemize}
\end{definition}
\begin{remark}
\rm
The tensor unit of $\cF wm(\cK)$ is a terminal object:
$$\{V_{i}\}_{i\in [n]}\stackrel{(\emptyset,\emptyset)}\longrightarrow \{\}$$
for every $\{V_{i}\}_{i\in [n]}$ object in $\cK$. In addition, notice that $\cF wm(\cK)(\{\},\{V_{i}\}_{i\in [n]})=\emptyset $ if $\{V_{i}\}_{i\in [n]}\neq \{\}$.
\end{remark}
\begin{proposition}
\label{FREE WEEK SYM MON}
Given any symmetric monoidal category $\cA$ whose tensor unit is terminal and any functor $F:\cK\rightarrow\cA$, there is a unique strong monoidal functor $G:\cF wm(\cK)\rightarrow \cA$, up to isomorphism, such that $G\circ I =F$.
$$\xymatrix@=25pt{
\cK\ar[d]_{I}\ar[r]^{F}&\cA\\
\cF wm(\cK)\ar[ru]_{G}&
}$$
\end{proposition}
\begin{proof}(sketch)
The functor $G$ is defined on objects by:
$G(\{\})=I$ and $G(\{V_{i}\}_{i\in [n]})=(\ldots(F(V_1)\otimes F(V_2)\otimes F(V_3)  )\ldots\otimes F(V_n))$.

Let $(\phi,\{f_i\}_{i\in [m]}):\{V_{i}\}_{i\in [n]}\rightarrow \{W_{i}\}_{i\in [m]}$ be a map in $\cF wm(\cK)$ then
$$G(\phi,\{f_i\}_{i\in [m]}):G(\{V_{i}\}_{i\in [n]})\rightarrow G(\{W_{i}\}_{i\in [m]})$$
is given by

$$\xymatrix@=4pt{
(F(V_1)\otimes F(V_2))\otimes\ldots\otimes F(V_n)=G(\{V_{i}\}_{i\in [n]})\ar[dddd]^{(x_1\otimes x_2)\ldots\otimes x_n}\ar@{..>}[rr]^{}&&(F(W_1)\otimes F(W_2))\ldots F(W_m)=G(\{W_{i}\}_{i\in [m]})\\
&&\\
&&\\
&&\\
(X_1\otimes X_2)\otimes\ldots\otimes X_n\ar[rr]_{\cong}&&(F(V_{\phi(1)})\otimes F(V_{\phi(2)}))\ldots F(V_{\phi(m)})\ar[uuuu]_{(F(f_1)\otimes F(f_2))\ldots\otimes F(f_m)}&
}$$
where $x_i=1_{F(V_i)}:F(V_i)\rightarrow F(V_i)$ if $i\in \phi([m])$ and $x_i=\,\,\,!:F(V_i)\rightarrow I$ if $i\in [n]-\phi([m])$.

Using coherence of the category $\cA$ we prove that $G$ is a strong functor:  the mediating isomorphism is given by the unique morphism that shifts all the parenthesis to the left:
$$G(\{V_{i}\}_{i\in [n]})\otimes G(\{W_{i}\}_{i\in [m]})\stackrel{m}\longrightarrow G(\{V_{i}\}_{i\in [n]}\otimes\{W_{i}\}_{i\in [m]})$$
and
$$I\stackrel{m^0=1}\longrightarrow G(\{\}).$$

To prove uniqueness we use the fact that
$x_i=\,\,\,!:F(V_i)\rightarrow I$ transforms into $x_i=\,\,\,!:G\circ I(V_i)\rightarrow G\{\}$ if $i\in [n]-\phi([m])$ and also that the coherence structure is preserved, up to isomorphism, for any functor satisfying these conditions.
\end{proof}

\begin{corollary}
$Fwm(\cK)$ is the free affine symmetric monoidal category generated by $\cK$.
\end{corollary}

\begin{example}
To illustrate the definition of the functor $G$ in the proof of Proposition~\ref{FREE WEEK SYM MON}, let us consider $(\phi,\{f_i\}_{i\in [2]}):\{V_1,V_2,V_3\}\rightarrow \{W_1,W_2\}$ with $\phi:[2]\rightarrow [3]$, $\phi(1)=3,\phi(2)=1$ then
$$G(\phi,\{f_i\}_{i\in [2]}):G(\{V_1,V_2,V_3\})\rightarrow G(\{W_1,W_2\})$$ is given by
$$\xymatrix@=25pt{
(F(V_1)\otimes F(V_2))\otimes F(V_3)=G(\{V_1,V_2,V_3\})\ar[d]_{F(f_2)\otimes !\otimes F(f_1)}\ar@{..>}[rr]^{G(\phi,\{f_i\}_{i\in [2]})}&&F(W_1)\otimes F(W_2)=G(\{W_1,W_2\})\\
(F(W_2)\otimes I)\otimes F(W_1)\ar[rr]_{\rho\otimes 1}&&F(W_2)\otimes F(W_1).\ar[u]_{\sigma}&
}$$

\end{example}

\section{Traced monoidal categories}\label{TRACED MONOIDAL CATEGORIES}
We recall the definition of a trace from~\cite{JSV96}.

\begin{definition}
\rm A \textit{trace} for a symmetric monoidal category $(\cC,\otimes,I,\rho,\lambda,s)$ consists of a family of functions
$$\Tr^U_{A,B}:\cC(A\otimes U,B\otimes U)\rightarrow \cC(A,B)$$
natural in $A$, $B$, and dinatural in $U$, satisfying the following axioms:

\textbf{Vanishing I:}\\
$\Tr^{I}_{X,Y}(f)=f$,\\

 \textbf{Vanishing II:}\\
$\Tr^{U\otimes V}_{X,Y}(g)=\Tr^{U}_{X,Y}(\Tr^{V}_{X\otimes U,Y\otimes U}(g))$,\\

\textbf{Superposing:}\\
$\Tr^U_{A\otimes C,B\otimes D}((1_B\otimes \sigma^{-1}_{D,U})\circ (f\otimes g)\circ (1_A\otimes \sigma_{C,U}))=\Tr^U_{A,B}(f)\otimes g=\\
\Tr^U_{A\otimes C,B\otimes D}((1_B\otimes \sigma_{U,D})\circ (f\otimes g)\circ (1_A\otimes \sigma^{-1}_{U,C}))$,\\

\textbf{Yanking:}\\
 For every $U$, we have $\Tr^{U}_{U,U}(\sigma_{U,U})=1_{U}$.\\

Explicitly, naturality and
dinaturality mean the following

\textbf{Naturality in $A$ and $B$:}\\
For any $g:X'\rightarrow X$ and $h:Y\rightarrow Y'$ we have that
$$\Tr^{U}_{X',Y'}((h\otimes 1_{U})\circ f\circ(g\otimes 1_{U})=h \circ \Tr^{U}_{X,Y}(f)\circ g.$$\\

\textbf{Dinaturality in $U$:}\\
 For any $f:X\otimes U\rightarrow Y\otimes U'$, $g:U'\rightarrow U$ we have that
$$\Tr^{U}_{X,Y}((1_{Y}\otimes g)\circ f)=\Tr^{U'}_{X,Y}(f\circ (1_{X}\otimes g)).$$\\

\end{definition}

\begin{definition}
\rm
Suppose we have two traced monoidal categories $(\cV,\Tr)$ and $(\cW,\widehat{\Tr})$. We say that a strong monoidal functor $(F,m):\cV\rightarrow\cW$ is \textit{traced monoidal} when it preserves the trace operator in the following way: for $f:A\otimes U\rightarrow B\otimes U$
$$\widehat{\Tr}^{FU}_{FA,FB}(m^{-1}_{A,U}\circ F(f)\circ m_{A,U})=F(\Tr^U_{A,B}(f)):FA\rightarrow FB.$$
\end{definition}

\section{Graphical language}

 Graphical calculi are an important tool for reasoning about monoidal
 categories, dating back at least to the work of Penrose~\cite{Pen}.
 There are various graphical languages which are provably complete for
 reasoning about diagrams in different kinds of monoidal
 categories. They allow efficient geometrical and topological insights
 to be used in a kind of calculus of ``wirings'', which simplifies
 diagrammatic reasoning. See~\cite{Sel 2009} for a detailed survey of
 such graphical languages.

 In particular, there is a graphical language for traced monoidal
 categories, which was already used in the original paper of Joyal,
 Street, and Verity~\cite{JSV96}. The axioms of traced monoidal
 categories are represented in the following way.

Naturality:\hspace{0.8in}\includegraphics[height=0.6in]{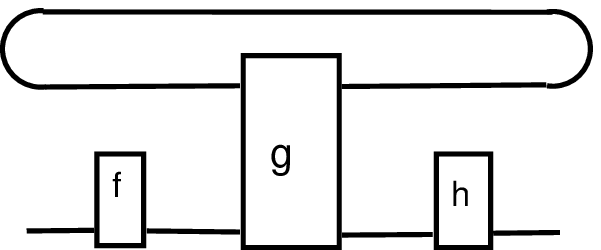}=\includegraphics[height=0.6in]{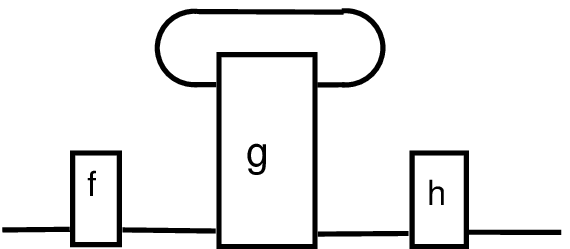}\\

Dinaturality:\hspace{0.8in}\includegraphics[height=0.6in]{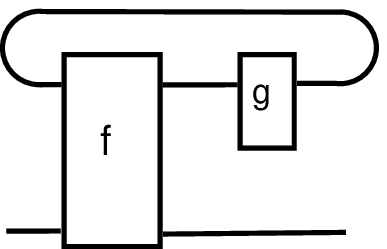}=\includegraphics[height=0.6in]{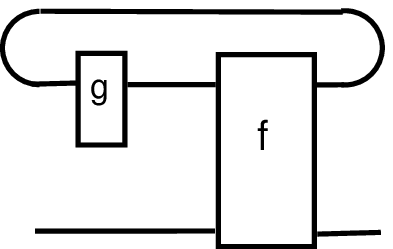}\\

Vanishing I:\hspace{0.8in} \includegraphics[height=0.6in]{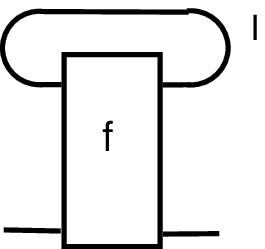}=\includegraphics[height=0.6in]{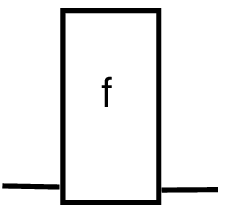}\\

Vanishing II:\hspace{0.8in} \includegraphics[height=0.6in]{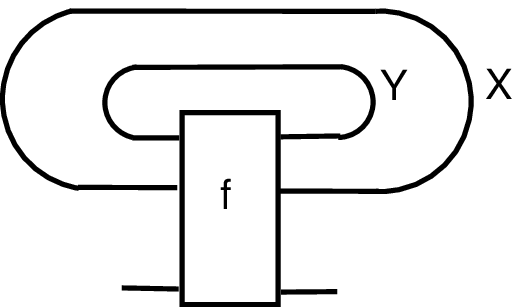}=\includegraphics[height=0.6in]{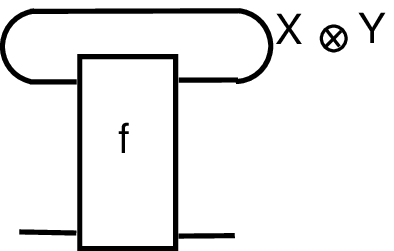}\\

Superposing (equivalent formulation):\hspace{0.5in}
\includegraphics[height=0.6in]{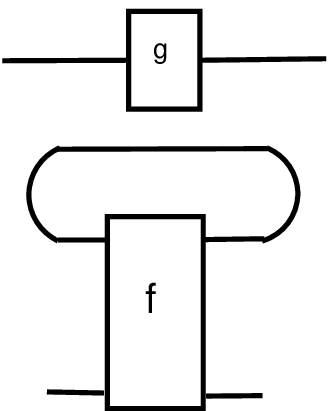}=\includegraphics[height=0.6in]{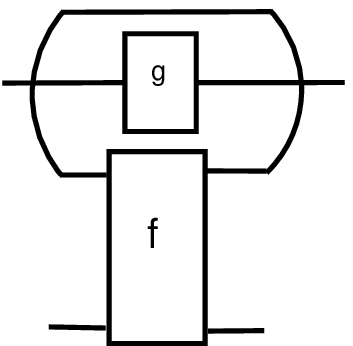}\\

Yanking:\hspace{0.8in}
\includegraphics[height=0.6in]{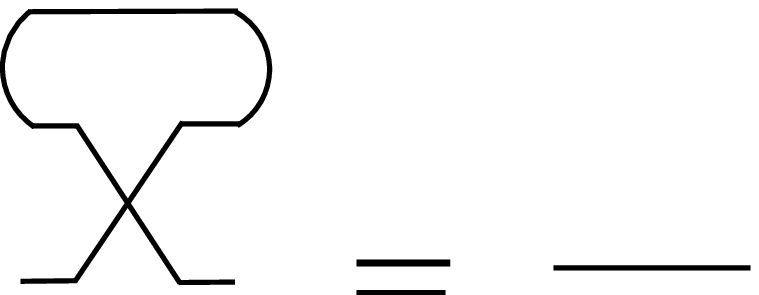}\\

Strength (equivalent formulation of superposing):\hspace{0.5in}
\includegraphics[height=0.6in]{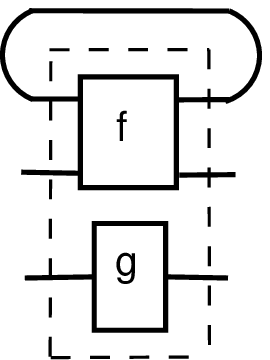}=\includegraphics[height=0.6in]{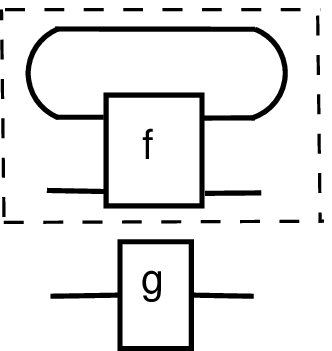}\\

The following theorem
shows the validity of such diagrammatic reasoning in compact closed categories:

\begin{theorem}[Coherence, see~\cite{Sel 2009}]
A well-formed equation between morphisms in the language of symmetric traced categories follows from the
axioms of symmetric traced categories if and only if it holds, up to isomorphism of diagrams,
in the graphical language.
\end{theorem}

Here by isomorphism of diagrams we mean a bijective correspondence between wires and boxes in which the structure of the graph is preserved.

\section{Compact closed categories}

\begin{definition}
\rm
A \textsl{compact closed} category is a symmetric monoidal category $\cV$ for which every object $A$ has assigned another object $A^{*}$, called  the dual, and a pair of arrows $\eta:I\rightarrow A^{*}\otimes A$ (unit), $\varepsilon:A\otimes A^{*}\rightarrow I$ (counit) such that the following diagrams commute:
$$
\xymatrix@=24pt{
  A\ar[d]_{1}\ar[rr]^{\rho}&&  A\otimes I \ar[rr]^{1\otimes \eta}& & A\otimes (A^* \otimes A) \ar[d]^{\alpha}\\
                 A             &&  I\otimes A \ar[ll]^{\lambda^{-1}}& &    (A\otimes A^*) \otimes A\ar[ll]^{\varepsilon\otimes 1}}
$$
and also,
$$
\xymatrix@=24pt{
  A^{*}\ar[d]_{1}\ar[rr]^{\lambda}&&  I\otimes A^{*} \ar[rr]^{\eta\otimes 1}& & (A^{*}\otimes A) \otimes A^{*} \ar[d]^{\alpha^{-1}}\\
                 A^{*}            &&  A^{*}\otimes I \ar[ll]^{\rho^{-1}}& &    A^{*}\otimes (A \otimes A^{*}).\ar[ll]^{1\otimes \varepsilon}}
$$

\end{definition}

In a compact closed category we can define a functor $(-)^*:\cV^{op}\rightarrow \cV$ where if $f:A\rightarrow B$ then $f^{*}:B^*\rightarrow A^*$ is given by:
$$B^*\stackrel{\lambda}\longrightarrow I\otimes B^* \stackrel{\eta\otimes 1}\longrightarrow A^*\otimes A\otimes B\stackrel{1\otimes f\otimes 1}\longrightarrow A^*\otimes B\otimes B^* \stackrel{1\otimes \varepsilon}\longrightarrow A^*\otimes I\stackrel{\rho^{-1}}\longrightarrow A^* .$$

\begin{proposition}
Let $(\cV,\otimes,\eta,\varepsilon)$ be a compact closed category. There exists a trace, which we call the \textit{canonical trace}, defined by:
$$\Tr^U_{A,B}(f)=(1\otimes\varepsilon\sigma)(f\otimes 1)(1\otimes\eta).$$
Moreover every symmetric strong monoidal functor between compact categories is traced monoidal with respect to the canonical trace.
\end{proposition}
\begin{proof}
See~\cite{JSV96}.
\end{proof}
\begin{proposition}\label{uniqueness compact}
Let $\cC$ be a compact closed category.  Then $\cC$ has a unique trace, i.e., the canonical trace
$$\Tr^U_{A,B}(f)=(1\otimes\varepsilon\sigma)(f\otimes 1)(1\otimes\eta).$$
\end{proposition}
\begin{proof}
Appendix B of~\cite{Has08}.
\end{proof}

\chapter{Categories of completely positive maps}

\section{Completely positive maps}\label{COMP POSITIVE MAPS}
\begin{definition}
  Let $H$ be a finite dimensional Hilbert space, i.e., a finite
  dimensional complex inner product space. Let us write $\LL(H)$ for
  the space of linear functions $\rho:H\to H$. Equivalently, we can
  write $\LL(H) = H^*\otimes H$.
 \end{definition}

Recall that the {\em adjoint} of a linear function $F:H\to K$ is
 defined to be the unique function $F^{\dagger}:K\to H$ such that
 $\iprod{F^{\dagger}v,w} = \iprod{v,Fw}$, for all $v\in K$ and $w\in
 H$.

 \begin{definition}
   Let $H, K$ be finite dimensional Hilbert spaces. A linear function
   $F:\LL(H)\to\LL(K)$ is said to be {\em completely positive} if it
   can be written in the form
   \[ F(\rho) = \sum_{i=1}^{m} F_i\rho F_i^{\dagger},
   \]
   where $F_i:H\to K$ is a linear function for $i=1,\ldots,m$.
 \end{definition}

 \begin{definition}
   The category $\CPM_s$ of {\em simple completely positive maps} has
   finite dimensional Hilbert spaces as objects, and the morphisms
   $F:H\to K$ are completely positive maps $F:\LL(H)\to\LL(K)$.
 \end{definition}

 \begin{definition}\label{COMPLETE POSITIVE MAPS}
   The category $\CPM$ of {\em completely positive maps} is defined as
   $\CPM=\CPM_s^{\oplus}$, the biproduct completion of $\CPM_s$.
   Specifically, the objects of $\CPM$ are finite sequences
   $(H_1,\ldots,H_n)$ of finite-dimensional Hilbert spaces, and a
   morphism $F:(H_1,\ldots,H_n)\to (K_1,\ldots,K_m)$ is a matrix
   $(F_{ij})$, where each $F_{ij}:H_j\to K_i$ is a completely positive
   map. Composition is defined by matrix multiplication.
 \end{definition}

 \begin{remark}
   In quantum mechanics, completely positive maps correspond to
   general transformations between quantum systems. Two special cases
   are of note: first, $F(\rho) = U\rho U^{\dagger}$, where $U$ is a
   unitary transformation. This represents the unitary evolution of an
   isolated quantum system. Second,
   \[ F(\rho) = (P_1\rho P_1^{\dagger}, \ldots, P_m\rho P_m^{\dagger}),
   \]
   where $P_1,\ldots,P_m$ is a system of commuting self-adjoint
   projections. This corresponds to measurement with possible outcomes
   $1,\ldots,m$. For more details on the physical interpretation, see
   e.g.~\cite{nielsen} or~\cite{Sel 2004}.
 \end{remark}

 \begin{remark}
   Note that the category $\CPM$ is the same (up to equivalence) as
   the category ${\bf W}$ of~\cite{Sel 2004} and the category ${\bf CPM}({\bf
     FdHilb})^{\oplus}$ of~\cite{Sel 2005}.
 \end{remark}

 Note that for any two finite dimensional Hilbert spaces $V$ and $W$,
 there is a canonical isomorphism
 $\phi_{V,W} : \LL(V\otimes W) \to \LL(V)\otimes \LL(W).$

 \begin{remark}
 The categories ${\bf CPM}_s$ and ${\bf CPM}$ are symmetric monoidal.
 For ${\bf CPM}_s$, the tensor product is given on objects by the tensor
 product defined on Hilbert spaces $V\bar{\otimes} W = V\otimes W$, and
 on morphisms by the following map $f\bar{\otimes} g$:

 $$\xymatrix@=25pt{
\LL(V\otimes W) \ar[rr]^{f\bar{\otimes} g}\ar[d]_{\phi_{V,W}}& &\LL(X \otimes Y)\ar[d]^{\phi_{X,Y}}\\
 \LL(V)\otimes \LL(W)\ar[rr]^{f\otimes g} & &  \LL(X)\otimes \LL(Y).
 }$$

 The left and right unit, associativity, and symmetry maps are
 inherited from the symmetric monoidal structure of Hilbert spaces.
 For the symmetric monoidal structure on ${\bf CPM}$, define
$$(V_i)_{i\in I} \otimes (W_j)_{j\in J} = (V_i\otimes W_j)_{i\in I, j\in J}.$$
 This extends to morphisms in an obvious way. For details, see~\cite{Sel 2004}.
\end{remark}

\section{Superoperators}\label{SUPEROPERATORS}

\begin{definition}
\rm We say that a linear map $F:\LL(V)\rightarrow \LL(W)$
is a \textit{trace preserving} linear function when it satisfies
\begin{equation}
\tr_W(F(\rho))=\tr_V(\rho)\label{TRACE PRESERVING CONDITION}
\end{equation}
for all positive $\rho\in \LL(V).$
$F$ is called {\em trace
non-increasing} when it satisfies
\begin{equation}
\tr_W(F(\rho))\leq \tr_V(\rho)\label{TRACE NON-INCREASING CONDITION}
\end{equation}
for all positive $\rho \in \LL(V).$
\end{definition}

\begin{definition}
 A linear function $F :\LL(V)\rightarrow \LL(W)$ is called a
 {\em trace preserving superoperator} if it is completely positive and
 trace preserving, and it is called a {\em trace non-increasing
 superoperator} if it is completely positive and trace non-increasing.
 \end{definition}

\begin{definition}
  A completely positive map $F:(H_1,\ldots,H_n) \to
 (K_1,\ldots,K_m)$ in the category ${\bf CPM}$ is called a {\em trace
 preserving superoperator} if for all $j$ and all positive $\rho\in \LL(H_j)$,
 \[
       \sum_i \tr(F_{ij}(\rho)) = \tr(\rho),
 \]
 and a {\em trace non-increasing superoperator} if for all $j$ and all
 positive $\rho\in \LL(H_j)$,
 \[
       \sum_i \tr(F_{ij}(\rho)) \leq \tr(\rho).
 \]

 \end{definition}

 \begin{definition}\label{Q OPLUS CATEGORY}\label{Q SIMPLE SUPEROPERATORS}\label{TRACE PRESERVING CATEGORY}
 We define four symmetric monoidal categories of
 superoperators. All of them are symmetric monoidal subcategories of ${\bf CPM}$.
 \begin{itemize}
 \item[-] ${\bf Q}$ and $\cQ'$ have the same objects as ${\bf CPM}$, and ${\bf Q}_s$ and $\cQ'_s$ have the
 same objects as ${\bf CPM}_s$.
 \item[-] The morphisms of ${\bf Q}$  and ${\bf Q}_s$ are trace non-increasing superoperators,
   and the morphisms of $\cQ'$ and $\cQ'_s$ are trace preserving
   superoperators.
 \end{itemize}
\end{definition}

The six categories defined in this chapter are summarized in the
 following table:
 \[
 \begin{array}{|c|c|c|}
   \hline
   & \mbox{simple} & \mbox{non-simple} \\\hline
   \mbox{no trace condition} & {\bf CPM}_s & {\bf CPM} \\\hline
   \mbox{trace non-increasing} & {\bf Q}_s & {\bf Q} \\\hline
   \mbox{trace preserving} & \cQ'_s & \cQ' \\\hline
 \end{array}
 \]

\begin{remark} The categories ${\bf Q}$, ${\bf Q}_s$, $\cQ'$, and $\cQ'_s$ are all symmetric
 monoidal. The symmetric monoidal structure is as in ${\bf CPM}$ and ${\bf CPM}_s$, and
 it is easy to check that all the structural maps are trace
 preserving.
 \end{remark}

\begin{lemma}
${\bf Q}$ and $\cQ'$ have finite coproducts.
\end{lemma}
\begin{proof}

The injection and copairing maps are as in ${\bf CPM}$; we only need
 to show that they are trace preserving. But this is trivially true.
\end{proof}

\chapter{Partially traced categories}

Traced monoidal categories were introduced by Joyal, Street and Verity~\cite{JSV96} as an attempt to organize properties from different fields of mathematics, such as algebraic topology and computer science. This abstraction has been useful in formulating new insights in concrete topics of theoretical computer science such as feedback, fixed-point operators, the execution formula in Girard's Geometry
of Interaction (GoI) \cite{GIRARD 88}, etc.
In this spirit, an axiomatization for partially traced symmetric monoidal categories was introduced by Haghverdi and Scott \cite{HS05a} providing an appropriate framework for a typed version of the Geometry of Interaction.

 An important part of the treatment of the dynamics of proofs in the Geometry of Interaction relies on the expressiveness of its model: proofs are interpreted as linear operators in Hilbert spaces and an invariant for the cut-elimination process is modelled by a convergent sum in some linear space. Haghverdi and Scott \cite{HS05a} have demonstrated that the categorical notion of partially traced category is a useful tool for capturing the dynamic behavior of all of these conceptual ideas as described by Girard. The word ``partial" here refers to the fact that the trace operator is defined on a subset of the set of morphisms $Hom(A\otimes U,B\otimes U)$ called the {\em trace class}. A large portion of Haghverdi and Scott's work is concerned with constructing the appropriate abstract notion of a typed GoI aided by the idea of orthogonality in the sense of Hyland and Schalk. Partial traces play a central role in Haghverdi and Scott's
 work. For example, their analysis of the idea of an abstract
 algorithm concerns the interplay with the execution formula defined
 in terms of a partially traced category. The categorical formula agrees with the original formula of Girard in some concrete
 Hilbert spaces and the execution formula in this new setting is an invariant of the cut-elimination process.

In this chapter, we give some examples of partially traced categories,
 including an example in the context of quantum computation. We also
 provide a method for constructing more examples by proving that each
 monoidal subcategory of a (totally or partially) traced category is
 partially traced.

\section{Partially traced categories}

We recall the definition of a monoidal partially traced category from \cite{HS05a}.

\begin{definition}\label{Kleene equality}
 Let $f$ and $g$ be partially defined operations. We write
 $f(x)\downarrow\,\,\,$ if $f(x)$ is defined, and $f(x)\uparrow$ if it is
 undefined. Following Freyd and Scedrov~\cite{FREYD SCEDROV Categories Allegories},
 we also write $f(x)\funnels g(x)$ if $f(x)$ and $g(x)$ are either
 both undefined, or else they are both defined and equal. The relation
 ``$\funnels$'' is known as {\em Kleene equality}. We also write
 $f(x)\funnel g(x)$ if either $f(x)$ is undefined, or else $f(x)$ and
 $g(x)$ are both defined and equal. The relation ``$\funnel$'' is
 known as {\em directed Kleene equality}.
\end{definition}
 \begin{definition} Suppose $(\cC,\otimes,I,\rho,\lambda,s)$  is a symmetric monoidal
 category. A {\em partial trace} is given by a family of partial
 functions $\Tr^U_{X,Y} : \cC(X\otimes U, Y\otimes U) \rightharpoonup
 \cC(X,Y)$, satisfying the following axioms:

\noindent
{\bf Naturality:}\\
For any $f:X\otimes U\rightarrow Y\otimes U$, $g:X'\rightarrow X$ and $h:Y\rightarrow Y'$ we have that
$$h \Tr^{U}_{X,Y}(f)g\funnel \Tr^{U}_{X',Y'}((h\otimes 1_{U}) f(g\otimes 1_{U})).$$

\noindent
\textbf{Dinaturality:}\\
 For any $f:X\otimes U\rightarrow Y\otimes U'$, $g:U\rightarrow U'$ we have
$$\Tr^{U}_{X,Y}((1_{Y}\otimes g)f)\funnels \Tr^{U'}_{X,Y}(f(1_{X}\otimes g)).$$

 \noindent
\textbf{Vanishing I:}\\
For every $f:X\otimes I\rightarrow Y\otimes I$ we have
$$\Tr^{I}_{X,Y}(f)\funnels \rho_{Y}f\rho^{-1}_{X}.$$

\noindent
\textbf{Vanishing II:}\\
 For every $g:X\otimes U\otimes V\rightarrow Y\otimes U\otimes V$, if
 $$\Tr^V_{X\otimes U,Y\otimes U}(g)\downarrow, $$ then
 $$\Tr^{U\otimes V}_{X,Y}(g)\funnels \Tr^{U}_{X,Y}(\Tr^{V}_{X\otimes U,Y\otimes U}(g)).$$

 \noindent
\textbf{Superposing:}\\
 For any $f:X\otimes U\rightarrow Y\otimes U$ and $g:W\rightarrow Z$,
 $$g\otimes \Tr^{U}_{X,Y}(f)\funnel \Tr^{U}_{W\otimes X,Z\otimes Y}(g\otimes f).$$

 \noindent
\textbf{Yanking:}\\
 For any $U$,
$$\Tr^{U}_{U,U}(\sigma_{U,U})\funnels 1_{U}.$$
\end{definition}

\begin{definition}\label{DEFINITION TRACE CLASS}
A {\em partially traced} category is a symmetric monoidal
 category with a partial trace.
\end{definition}

 \begin{remark} Comparing this to the definition of a traced monoidal
 category in Section~\ref{TRACED MONOIDAL CATEGORIES}, we see that a traced monoidal category is
 exactly the same as a partially traced category where the trace
 operation happens to be total. We sometimes refer to traced monoidal
 categories as {\em totally} traced monoidal categories, when we want
 to emphasize that they are not partial.
\end{remark}
 \begin{definition}
 The subset of $\cC(X\otimes U,Y\otimes U)$ where
 $\Tr^U_{X,Y}$ is defined is sometimes called the {\em trace class}, and
 is written
 \[
          \Trc^U_{X,Y} = \{ f:X\otimes U\rightarrow Y\otimes U \mid \Tr^U_{X,Y}(f)\downarrow \}.
 \]
\end{definition}

\begin{lemma}
Let $(\mathcal{C},\otimes,I,\Tr,s)$ be a partially traced category. The
 superposition axioms is equivalent to the following axiom (called
 {\em strength}):\\
 For $f:A\otimes U\rightarrow B\otimes U$ and $g:C\rightarrow D$,
\begin{center}
$\Tr^{U}_{A,B}(f)\otimes g\funnel  \Tr^{U}_{A\otimes C,B\otimes D}((1_B\otimes s_{U,D})\circ (f\otimes g)\circ (1_A\otimes s_{C,U})).$
\end{center}
\end{lemma}

\begin{proof}
$(\Rightarrow)$ First, from the original version we shall prove this second version.\\
By hypothesis and by naturality of the symmetries we have:
\begin{center}
$g\otimes f\in \Trc^U_{C\otimes A,D\otimes B}$ and\\
$s_{DB}\circ \Tr^U_{C\otimes A,D\otimes B}(g\otimes f)\circ s_{AC}=s_{DB}\circ (g\otimes \Tr^U_{A,B}(f))\circ s_{AC}=\Tr^U_{AB}(f)\otimes g$.
\end{center}
Thus by the naturality axiom we have that since $g\otimes f\in \Trc^U_{C\otimes A,D\otimes B}$:
\begin{center}
$(s_{DB}\otimes 1_U)\circ(g\otimes f)\circ (s_{AC}\otimes 1_U)\in\Trc^U_{A\otimes C, B\otimes D}$ and\\

$\Tr^U_{A\otimes C, B\otimes D}(s_{DB}\otimes 1_U)\circ (g\otimes f)\circ (s_{AC}\otimes 1_U)=s_{DB}\circ \Tr^U_{C\otimes A,D\otimes B}(g\otimes f)\circ s_{AC}$.
\end{center}
Finally by coherence we obtain:
\begin{center}
$(s_{DB}\otimes 1_U)\circ(g\otimes f)\circ (s_{AC}\otimes 1_U)=(1_B\otimes s_{UD})\circ (f\otimes g)\circ (1_A\otimes s_{CU})$
\end{center}
$(\Leftarrow)$ Conversely by hypothesis and composing with symmetries we get:
\begin{center}
$(1_B\otimes s_{UD})\circ (f\otimes g)\circ (1_A\otimes s_{CU})\in \Trc^U_{A\otimes C,B\otimes D}$ and

$s_{BD}\circ \Tr^{U}_{A\otimes C,B\otimes D}((1_B\otimes s_{U,D})\circ (f\otimes g)\circ (1_A\otimes s_{C,U}))\circ s_{CA}=s_{BD}\circ (\Tr^{U}_{A,B}(f)\otimes g)\circ s_{BD}$.
\end{center}
Which implies by the naturality axiom that:
\begin{center}
$\alpha=(s_{BD}\otimes 1_U)\circ (1_B\otimes s_{U,D})\circ (f\otimes g)\circ (1_A\otimes s_{C,U})\circ (s_{CA}\otimes 1_U)\in \Trc^U_{C\otimes A,D\otimes B}$ and\\

$\Tr^U_{C\otimes A,D\otimes B}(\alpha)=g\otimes \Tr^U_{A,B}(f).$
\end{center}
But by coherence $\alpha=g\otimes f$.
\end{proof}

\section{Examples of partially traced categories}
\subsection{Finite dimensional vector spaces}
Among the examples that motivated this notion of partially traced category in Definition~\ref{DEFINITION TRACE CLASS} a particularly important one~\cite{HS05a},~\cite{HS09} is the category $(\textbf{Vect}_{fn},\oplus,\textbf{0})$ of finite dimensional vector spaces and linear transformations, with biproduct $\oplus$ as the tensor product.

We recall that in an additive category a morphism $f: X\oplus U\rightarrow Y\oplus V$ is characterized by compositions with injections and projections: $f_{ij}=\pi_{i} \comp f\comp in_{j}$, $1\leq i,j\leq 2$. We denote $f$ by a matrix of morphisms of type
$\left[\begin{array}{llll}
f_{11} & f_{12}\\
f_{21} & f_{22} \\
\end{array}\right]$ where composition corresponds to multiplication of matrices.
\begin{definition}\label{PARTIAL TRACE IN FINITE VECTOR SPACE}
\rm
The \textsl{trace class in} $(\textbf{Vect}_{fn},\oplus,\textbf{0})$ is defined as follows:
we say that $f:X\oplus U\rightarrow Y\oplus U\in \Trc^U_{X,Y}$ iff $I-f_{22}$ is invertible, where $I=id$ on $U$.\\
When this is the case we define $\Tr^U_{X,Y}(f)=f_{11}+f_{12}(I-f_{22})^{-1}f_{21}$.
\end{definition}

\begin{proposition}
 With the operation defined in
 Definition~\ref{PARTIAL TRACE IN FINITE VECTOR SPACE}, the category of finite dimensional vector spaces is
 partially traced.
 \end{proposition}
\begin{proof}~\cite{HS05a},~\cite{HS09}.
\end{proof}

\subsection{Stochastic relations}

In order to capture classical probabilistic computation (as a
 stepping stone towards quantum computation), we now describe a trace
 class in the category ${\bf Srel}$ of stochastic relations. In fact, this partial trace arises from the canonical total trace on $({\bf Vect}_{fn},\otimes)$ by a general construction that we will examine in detail in Section~\ref{TRACE IN A MONOIDAL SUBCATEGORY}. Note that it differs from the trace on ${\bf Srel}$ given by Abramsky~\cite{Abr96},
~\cite{THESIS HAGVERDI}. Abramsky's trace is with respect to the coproduct structure $\oplus$ and is total; here we discuss a partial trace with respect to the tensor structure $\otimes$.

The category of stochastic relations attempts to model the probability of a bit being in states 0 or 1, or more
generally, of a variable taking a specific value in a finite set of
possible values. Morphisms in this category correspond to the behaviours of finitary probabilistic systems.
The  general category of stochastic relations, ${\bf Srel}$, is described in~\cite{Abr96} and~\cite{AHS02}. It arises as the Kleisli category of the Giry Monad~\cite{GIRY}. We look at the special case where the objects are finite sets.

\begin{definition}\label{STOCHASTIC RELATION}
The category $\textbf{Srel}_{fn}$ of \textit{finite
stochastic relations} consists of the following data:
\begin{itemize}
\item[-] objects are finite sets: $A$, $B$\ldots
\item[-] morphisms:  $\xymatrix{A \ar[r]^{f} & B}$ are finite matrices $f:B\times A\rightarrow [0,1]$ such that $\forall a\in A$
$$\sum_{b\in B}f(b,a)\leq 1.$$
\end{itemize}

The composite of two morphisms is defined by matrix multiplication:

If $\xymatrix{A \ar[r]^{f} & B}$ and $\xymatrix{B \ar[r]^{g} & C}$ then $g\circ f:C\times A\rightarrow [0,1]$ is:
 \[(g\circ f)(c,a)= \sum_{b\in B}g(c,b).f(b,a).\]
It is immediate that composition as defined above is associative, with identities
$1_A:A\times A\rightarrow [0,1]$, defined $1_A(x,y)=\left \{ \begin{array}{ll}
1 & \textrm{if $x=y$}\\
0 & \textrm{if $x\neq y$.}
\end{array} \right.$
\end{definition}

\begin{remark}
 Note that we allow $\sum_{b\in B}f(b,a)\leq 1$, rather
 than requiring equality. This is also called a ``partial" stochastic
 relation. A probability that is less than 1 corresponds to a
 computational process that may not terminate.
 \end{remark}

One obtains a symmetric monoidal category $(\textbf{Srel}_{fn},\otimes, I)$ where the tensor product on objects is given by the set product $A\otimes B=A\times B$. For arrows $ f:A\rightarrow B$ and $g:C\rightarrow D$, i.e., $f:B\times A\rightarrow [0,1]$ and $g:D\times C\rightarrow [0,1]$ then we have $f\otimes g:A\otimes C\rightarrow B\otimes D$ is given by a map of type $f\otimes g:B\times D\times A\times C \rightarrow [0,1]$, where
$$(f\otimes g)(b,d,a,c)=f(b,a)\cdot g(d,c)\,.$$
Let $A,B$ be finite sets. There is a canonical way to encode a function $f:A\rightarrow B$ as a stochastic map:  we write $\hat{f}:B\times A\rightarrow [0,1]$ where $\hat{f}(b,a)=1$ if $f(a)=b$ and $\hat{f}(b,a)=0$ otherwise. We define the symmetric monoidal coherence isomorphisms by applying this codification to the coherence structure of the cartesian category {\bf FinSet} of finite sets.
\begin{definition}\label{TRACE STOCH FORMULA}
Let $f:X\otimes U \rightarrow Y\otimes U$ be a stochastic map. We define
the following trace class $\Trc^{U}_{X,Y}\subseteq
\textbf{Srel}_{fn}(X\otimes U,Y\otimes U)$ for all $X$ and $Y$:
\begin{center}
$f\in \Trc^{U}_{X,Y}$ iff $\sum_{y\in Y}\sum_{u\in U}f(y,u,x,u)\leq
1, \forall x\in X$
\end{center}
and a partial trace:
\begin{center}
$\Tr^{U}_{X,Y}:\Trc^{U}_{X,Y}\rightarrow \textbf{Srel}_{fn}(X,Y)$ with
$\Tr^{U}_{X,Y}(f)(y,x)=\sum_{u\in U}f(y,u,x,u).$
\end{center}
\end{definition}
\begin{proposition}
 The formula given in Definition~\ref{TRACE STOCH FORMULA} defines a partial trace on ${\bf Srel}_{fn}$.
 \end{proposition}
 \begin{proof}
 We check the axioms of partial trace.\\
\textbf{Naturality:}\\
Let $f\in \Trc^{U}_{X,Y}$ and $g:X'\rightarrow X$ and
$h:Y\rightarrow Y'$ be stochastic maps, first we want to prove that
\begin{center}
$(h\otimes 1_U)f(g\otimes 1_U)\in\Trc^{U}_{X',Y'}$ with $(h\otimes
1_U)f(g\otimes 1_U):X'\otimes U\rightarrow Y'\otimes U.$
\end{center}
Since we have a map of type
$(h\otimes 1_U)f:X\otimes U\rightarrow Y'\otimes U$
we evaluate:
\begin{center}
$(h\otimes 1_U)f(y',u,x,v) =\sum_{y\in Y,u'\in U}(h\otimes
1_U)(y',u,y,u')f(y,u',x,v)=\sum_{y\in Y,u'\in
U}h(y',y)1(u,u')f(y,u',x,v)=\sum_{y\in Y}h(y',y)f(y,u,x,v).$
\end{center}

Now we compose again:
\begin{center}
$(h\otimes 1_U)f(g\otimes 1_U)(y',u,x',v)=\sum_{x\in X,u'\in U}(h\otimes
1_U)f(y',u,x,u')(g\otimes 1_U)(x,u',x',v)=\sum_{x\in X,u'\in U}(\sum_{y\in
Y}h(y',y)f(y,u,x,u')).g(x,x').1_U(u',v)=\sum_{x\in X,y\in
Y}h(y',y)f(y,u,x,v)g(x,x').$
\end{center}

Thus $(h\otimes 1_U)f(g\otimes 1_U)\in\Trc^{U}_{X',Y'}$  iff
$\,\sum_{y'\in Y',u\in U}(h\otimes 1_U)f(g\otimes 1_U)(y',u,x',u)\leq 1\,,
\forall x'\in X'.$\\
We know by hypothesis that $f\in \Trc^{U}_{X,Y}$ which implies that
$\sum_{y\in Y,u\in U}f(y,u,x,u)\leq 1\,,\forall x\in X.$ On the other hand
we also know that $\sum_{x\in X}g(x,x')\leq 1$ $\forall x'\in X$ and
$\sum_{y'\in Y'}h(y',y)\leq 1$ $\forall y\in Y$ since $g:X'\rightarrow X$
and $h:Y\rightarrow Y'$ are stochastic maps.\\
Thus,
\begin{center}
$\sum_{x\in X}(\sum_{y\in Y,u\in U}f(y,u,x,u))g(x,x')\leq \sum_{x\in
X}1.g(x,x')\leq 1$ $\forall x'\in X'.$
\end{center}
Therefore,
\begin{center}
$\sum_{x\in X,y\in Y,u\in U}f(y,u,x,u))g(x,x')\leq 1$ $\forall x'\in X'.$
\end{center}
Now using this and the fact that $\sum_{y'\in Y'}h(y',y)\leq 1$:
\begin{center}
$\sum_{x\in X,y\in Y,u\in U}(\sum_{y'\in
Y'}h(y',y)).f(y,u,x,u))g(x,x')\leq$\\
$ \sum_{x\in X,y\in Y,u\in U}1.f(y,u,x,u))g(x,x')\leq 1$  $\forall x'\in
X'.$
\end{center}
This implies the following:
\begin{center}
$\sum_{x\in X,y\in Y,u\in U,y'\in Y'}h(y',y))f(y,u,x,u))g(x,x')\leq 1$
$\forall x'\in X'.$
\end{center}
Therefore,
\begin{center}
$\sum_{y'\in Y',u\in U}(\sum_{x\in X,y\in
Y}h(y',y))f(y,u,x,u))g(x,x')=\sum_{y'\in Y',u\in U}(h\otimes
1_U)f(g\otimes 1_U)(y',u,x',u)\leq 1$ $\forall x'\in X'$
\end{center}
which implies that the following assertion holds:
\begin{center}
$(h\otimes 1_U)f(g\otimes 1_U)\in\Trc^{U}_{X',Y'}.$
\end{center}
Next, we preliminary compute the partial trace. For that purpose, we first need some previous
calculations:
\begin{center}
$\Tr^U_{X',Y'}((h\otimes 1_U)f(g\otimes 1_U))(y',x')=\sum_{u\in
U}(\sum_{x\in X,y\in Y}h(y',y)f(y,u,x,u))g(x,x')).$
\end{center}
If we apply the definition of partial trace to $f$ and compose with $h$
then this comes down to
\begin{center}
$h\circ \Tr^U_{X,Y}(f)(y',x)=\sum_{y\in Y}h(y',y).(\sum_{u\in
U}f(y,u,x,u))=\sum_{y\in Y,u\in U}h(y',y)f(y,u,x,u).$
\end{center}
Similarly, we compose with $g$
\begin{center}
$((h\Tr^U_{X,Y}(f)) g)(y',x')=\sum_{x\in
X}(h\Tr^U_{X,Y}(f))(y',x).g(x,x')=\sum_{x\in X}(\sum_{y\in Y,u\in
U}h(y',y)f(y,u,x,u)).g(x,x')=\sum_{x\in X,y\in Y,u\in
U}h(y',y)f(y,u,x,u)g(x,x')$
\end{center}
which proves that both previous calculations are equal.

\textbf{Yanking:}\\
Let $\sigma :A\otimes B\rightarrow B\otimes A$ be defined as the matrix
$\sigma :B\times A\times A\times B\rightarrow [0,1]$ with
\begin{center}
$\sigma(b,a,a',b')=1$ iff $b=b'$ and $a=a'$ otherwise is $0$.
\end{center}
It may be seen immediately that if $\sigma :U\otimes U\rightarrow U\otimes
U$
\begin{center}
$\Tr^U_{U,U}(\sigma )(u,v)=\sum_{x\in U}\sigma (u,x,v,x)=1$ if and only if
$u=x=v$ otherwise is $0$.
\end{center}
Then, since $1_U(u,v)=1$ if and only if $u=v$, otherwise it is $0$ we obtain
that
\begin{center}
$\Tr^U_{U,U}(\sigma )(u,v)=1_U(u,v)$ for every $u$ and $v$.
\end{center}

\textbf{Dinaturality:}\\
Consider the stochastic maps $f:X\otimes U\rightarrow Y\otimes U'$ and
$g:U'\rightarrow U$. First we want to prove that
\begin{center}
$(1_Y \otimes g)f\in \Trc^U_{X,Y}$ if and only if $f(1_X\otimes
g)\in \Trc^{U'}_{X,Y}.$
\end{center}
By definition of trace class we know that
\begin{center}
$(1_Y \otimes g)f\in \Trc^U_{X,Y}$ if and only if $\sum_{y\in Y,u\in
U}(1_Y \otimes g)f(y,u,x,u)\leq 1$  $\forall x\in X.$
\end{center}
Also, by definition of composition in the category $\textbf{Srel}_{fn}$:
\begin{center}
$(1_Y \otimes g)f(y,u,x,v)=\sum_{y'\in Y,u'\in U'}(1_Y \otimes
g)(y,u,y',u')f(y',u',x.v)=\sum_{y'\in Y,u'\in
U'}1_Y(y,y')g(u,u')f(y',u',x,v)=\sum_{u'\in U'}g(u,u')f(y,u',x,v).$
\end{center}
Thus, we have seen that
\begin{center}
$(1_Y \otimes g)f\in \Trc^U_{X,Y}$ if and only if $\sum_{y\in Y,u\in
U}(\sum_{u'\in U'}g(u,u')f(y,u',x,u))\leq 1$ $\forall x\in X.$
\end{center}
Following a similar argument we have that
\begin{center}
$f(1_X\otimes g)\in \Trc^{U'}_{X,Y}$ if and only if $\sum_{y\in
Y,u'\in U'}f(1_X \otimes g)(y,u',x,u')\leq 1$  $\forall x\in X.$
\end{center}
But, again by definition of composition
\begin{center}
$f(1_X\otimes g)(y,u',x,v')=\sum_{x'\in X,u\in U}f(y,u',x',u)(1_X\otimes
g)(x',u,x,v')=\sum_{x'\in X,u\in U}f(y,u',x',u)1_X(x',x)g(u,v')=\sum_{u\in
U}f(y,u',x,u)g(u,v').$
\end{center}
This means that
\begin{center}
$f(1_X\otimes g)\in \Trc^{U'}_{X,Y}$ if and only if $\sum_{y\in
Y,u'\in U'}(\sum_{u\in U}f(y,u',x,u)g(u,u'))\leq 1$  $\forall x\in X.$
\end{center}
This implies that the condition on the trace class is satisfied. Next, it remains to calculate
the corresponding partial traces.
\begin{center}
$\Tr^U_{X,Y}((1_Y \otimes g)f)(y,x)=\sum_{u\in U}(1_Y \otimes
g)f(y,u,x,u)=\sum_{u\in U}(\sum_{u'\in U'}g(u,u')f(y,u',x,u))=\sum_{u\in
U,u'\in U'}g(u,u')f(y,u',x,u)=\sum_{u'\in U',u\in
U}f(y,u',x,u))g(u,u')=\sum_{u'\in U'}f(1_X\otimes
g)(y,u',x,u')=\Tr^{U'}_{X,Y}(f(1_X\otimes g))(y,x).$
\end{center}
\textbf{Vanishing I:}\\
Let $f:X\otimes I\rightarrow Y\otimes I$ be a stochastic map. Therefore, this
implies by definition
\begin{center}
$\sum_{y\in Y,u\in \{\ast\}}f(y,u,x,\ast)=\sum_{y\in
Y}f(y,\ast,x,\ast)\leq 1$ for every $x \in X.$
\end{center}
Thus, this is equivalent to
\begin{center}
$\sum_{y\in Y,u\in \{\ast\}}f(y,u,x,u)\leq 1$ for every $x \in X$
\end{center}
which is the condition $f\in \Trc^I_{X,Y}.$\\
Now, we compute the partial traces.
Let us consider the following composition
\begin{center}
$X\stackrel{\rho^{-1}_X}{\rightarrow}X\otimes
I\stackrel{f}{\rightarrow}Y\otimes I\stackrel{\rho_Y}{\rightarrow}Y.$
\end{center}
We have
\begin{center}
\begin{eqnarray*}
f\rho^{-1}_X(y,\ast,x)&=&\sum_{x'\in X,u\in
I}f(y,\ast,x',u)\rho^{-1}_X(x',u,x)\\
&=&f(y,\ast,x,\ast)
\end{eqnarray*}

$f\rho^{-1}_X(y,\ast,x)=\sum_{x'\in X,u\in
I}f(y,\ast,x',u)\rho^{-1}_X(x',u,x)=f(y,\ast,x,\ast).$
\end{center}

\noindent
Now, we compose with $\rho_Y$ to get:
\begin{eqnarray*}
\rho_Y(f\rho^{-1}_X)(y,x)&=&\sum_{y'\in Y,u\in
I}\rho_Y(y,y',u)(f\rho^{-1}_X)(y',u,,x)\\
&=&
f\rho^{-1}_X(y,\ast,x) = f(y,\ast,x,\ast)
\end{eqnarray*}

\begin{center}
$\rho_Y(f\rho^{-1}_X)(y,x)=\sum_{y'\in Y,u\in
I}\rho_Y(y,y',u)(f\rho^{-1}_X)(y',u,,x)=
f\rho^{-1}_X(y,\ast,x)=f(y,\ast,x,\ast)$
\end{center}
which clearly means that
\begin{center}
$\Tr^I_{X,Y}(f)(y,x)=f(y,\ast,x,\ast)=\rho_Yf\rho^{-1}_X(y,x)$ for every
$x\in X$ and $y\in Y.$
\end{center}
Thus, we proved that $\Tr^I_{X,Y}(f)=\rho_Y f\rho^{-1}_X$.

\noindent
\textbf{Vanishing II:}\\
Suppose we have a stochastic map $g:X\otimes U\otimes V\rightarrow
Y\otimes U\otimes V$ such that $g\in \Trc^V_{X\otimes U, Y\otimes
U}.$ We need to check that
\begin{center}
$g\in \Trc^{U\otimes V}_{X, Y}$ if and only if $\Tr^V_{X\otimes U,
Y\otimes U}(g)\in \Trc^{U}_{X,Y}.$
\end{center}
By definition, it follows that
\begin{center}
$g\in \Trc^{U\otimes V}_{X, Y}$ if and only if $\sum_{y\in
Y,(u,v)\in U\times V}g(y,u,v,x,u,v)\leq 1.$
\end{center}
On the other hand we have
\begin{center}
$\Tr^V_{X\otimes U, Y\otimes U}(g)(y,u,x,u')=\sum_{v\in V}g(y,u,v,x,u',v).$
\end{center}
We obtain
\begin{center}
$\Tr^V_{X\otimes U, Y\otimes U}(g)\in \Trc^{U}_{X,Y}$ if and only if
$\sum_{y\in Y}(\sum_{u\in U}\Tr^V_{X\otimes U, Y\otimes
U}(g)(y,u,x,u)=\sum_{y\in Y,u\in U,v\in V}g(y,u,v,x,u,v)\leq 1.$
\end{center}
Thus, we have shown that both conditions are equivalent. Now we move to the
calculation of the partial traces.
\begin{eqnarray*}
\Tr^{U\otimes V}_{X, Y}(g)(y,x) & = & \sum_{(u,v)\in U\times
V}g(y,u,v,x,u,v)\\
& = & \sum_{u\in U}\sum_{v\in V}g(y,u,v,x,u,v) \\
& = &\sum_{u\in U}\Tr^{V}_{X\otimes U,Y\otimes U}(g)(y,u,x,u)\\
& = & \Tr^{U}_{X,Y}(\Tr^{V}_{X\otimes U,Y\otimes U}(g))(y,x).
\end{eqnarray*}
In conclusion we obtain that
\begin{center}
$\Tr^{U\otimes V}_{X, Y}(g)=\Tr^{U}_{X,Y}(\Tr^{V}_{X\otimes U,Y\otimes
U}(g)).$
\end{center}

\textbf{Superposing:}\\
Consider the stochastic maps $f:X\otimes U\rightarrow Y\otimes U$ with
$f\in \Trc^U_{X,Y}$  and $g:W\rightarrow Z$. First, we want to prove
that
\begin{center}
$g\otimes f\in \Trc^{U}_{W\otimes X,Z\otimes Y}.$
\end{center}
In order to prove this we have that
$$
\begin{array}{ll}
g\otimes f\in \Trc^{U}_{W\otimes X,Z\otimes Y}& \\
\mbox{if and only if}&  \sum_{(z,y)\in Z\times Y,u\in U}g\otimes
f(z,y,u,w,x,u)\leq 1   \forall w\in W ,  \forall x\in X    \\
\mbox{if and only if}&   \sum_{z\in Z,y\in Y,u\in U}g(z,w)f(y,u,x,u)\leq 1
\forall w\in W ,  \forall x\in X \\
  \mbox{if and only if}&   \sum_{z\in Z}g(z,w)\sum_{y\in Y,u\in
U}f(y,u,x,u)\leq 1 ,  \forall w\in W ,  \forall x\in X.
\end{array}
$$

\noindent
Here the last equivalence is true since $g$ is stochastic i.e.,
$\sum_{z\in Z}g(z,w)\leq 1$, $\forall w\in W$. Since we have that $f\in
\Trc^U_{X,Y}$ this implies $\sum_{y\in Y,u\in U}f(y,u,x,u)\leq 1$,
$\forall x\in X.$
We show now that the partial traces are equal.
\begin{center}
$\Tr^{U}_{W\otimes X,Z\otimes Y}(g\otimes f)(z,y,w,x)=\sum_{u\in
U}(g\otimes f)(z,y,u,w,x,u)=\sum_{u\in
U}g(z,w).f(y,u,x,u)=g(z,w).\sum_{u\in
U}f(y,u,x,u)=g(z,w).\Tr^{U}_{X,Y}(f)(y,x)=g\otimes
\Tr^{U}_{X,Y}(f)(z,y,w,x).$
\end{center}
This means that
\begin{center}
$\Tr^{U}_{W\otimes X,Z\otimes Y}(g\otimes f)=g\otimes \Tr^{U}_{X,Y}(f).$
\end{center}
\end{proof}

\subsection{Total trace on completely positive maps with $\otimes$}

 In this section, we define a total trace on the category $\textbf{CPM}_s$ of
 simple completely positive maps (see Section~\ref{COMP POSITIVE MAPS}). As a matter of
 fact, this category is compact closed, and therefore is has a unique
 total trace. Here, we describe it explicitly via a Kraus operator-sum representation.
\newcommand{\fHilb}{{\bf fHilb}}

Recall that the category $\fHilb$ of finite dimensional Hilbert spaces and
 linear maps is compact closed, and therefore (totally)
 traced. Let, $\cal H_A,\cal H_B$ and $\cal H_C$  be finite dimensional
 Hilbert spaces with orthonormal bases $\{e_i\},\{f_i\}$ and $\{w_i\}$,
 respectively, and let $F:\cal H_A\otimes\cal H_B\rightarrow\cal H_C\otimes\cal H_B$ be a linear function, i.e.,
$$F=\sum_{j,l,k,m}F_{j,l,k,m}|w_j,f_k\rangle\langle e_l,f_m|.$$
 Then $\tr_B (F) =\sum_{j,l,k}F_{j,l,k,k}|w_j\rangle\langle e_l|$ defines a total trace on $\fHilb$.

\begin{proposition}
Let $\textit{F}:\cal L(\cal H_A)\otimes \cal L(\cal H_B)\rightarrow \cal L(\cal H_C)\otimes \cal L(\cal H_B)$ be a complete positive map with representation $\textit{F}=\sum_{j=1}^n F_j\rho F_j^{\dagger}$. Then $\Tr^{A,C}_B(\textit{F})(\rho)=\sum_{j=1}^n \tr_B F_j \rho\  \tr_B F_j^{\dagger}$ defines a
(total) trace on the category ${\bf CPM}_s$.
\end{proposition}
\begin{proof}
Suppose we take two representations of
\begin{center}
$\textit{F}(\rho)=\sum_{i=1}^n E_i\rho E_i^{\dagger}=\sum_{j=1}^n F_j\rho F_j^{\dagger}.$
\end{center}
Then
\begin{center}
$\Tr^{A,C}_B(\textit{F})(\rho)$=$\sum_{i=1}^n \tr_B F_i\, \rho\, \tr_B F_i^{\dagger}$=$\sum_{i=1}^n \tr_B (\sum_j U_{i,j}F_j) \rho\  \tr_B (\sum_j U_{i,j}F_j)^{\dagger}
=\sum_{i=1}^n  (\sum_j U_{i,j}\tr_B F_j) \rho\ (\sum_j U_{i,j}^{*}\tr_B F_j^{\dagger})=
\sum_{i,j,k}U_{i,j}U_{i,k}^{*}\tr_B F_j\,\, \rho\,\, \tr_B F_k^{\dagger}=\sum_{j,k}(\sum_{i}U_{i,j}U_{i,k}^{*})\tr_B F_j \,\rho\, \tr_B F_k^{\dagger}=\sum_{j,k}(\sum_{i}U_{k,i}^{\dagger}U_{i,j})\tr_B F_j \rho\ \tr_B F_k^{\dagger}=\sum_{j,k}\delta_{k,j}\tr_B F_j \rho\ \tr_B F_k^{\dagger}=\sum_{j}\tr_B F_j \rho\ \tr_B F_j^{\dagger}$
\end{center}
since $U$ is unitary.\\
Now we check all the axioms.\\

\textbf{Naturality:}\\
Let us consider $f=\sum_iU_i-U_i^{\dagger}$ and $g=\sum_jV_j-V_j^{\dagger}$ where $\textit{f}:\cal{L}(\cal{H}_A)\otimes \cal{L}(\cal{H}_B)\rightarrow \cal{L}(\cal{H}_C)\otimes \cal{L}(\cal{H}_B)$ and
$\textit{g}:\cal{L}(\cal{H}_{A'})\rightarrow \cal{L}(\cal{H}_A)$. \\
Since $f(g\otimes id)=(\lambda\rho\sum_iU_i\rho U_i^{\dagger})(\lambda\rho\sum_jV_j\rho V_j^{\dagger}\otimes id)=\lambda\rho\sum_{i,j}U_i(V_j\otimes I)\rho (V_j^{\dagger}\otimes I)U_i^{\dagger}$ therefore, we have:
\begin{center}
$\Tr_{A'C}^B(f(g\otimes id))=\lambda\rho\sum_{i,j}\tr_B(U_i(V_j\otimes I))\rho\, \tr_B((V_j^{\dagger}\otimes I)U_i^{\dagger})=\lambda\rho\sum_{i,j}(\tr_B U_i) V_j\rho V_j^{\dagger}(\tr_B U_i^{\dagger})=\lambda\rho\sum_{i,j}(\tr_B U_i)\rho (\tr_B U_i^{\dagger})\comp\lambda\rho\sum_jV_j\rho V_j^{\dagger}=\Tr_{AC}^B(f)\comp g$.
\end{center}

\textbf{Dinaturality:}\\
Suppose we have $f=\sum_iU_i-U_i^{\dagger}$ and $g=\sum_jV_j-V_j^{\dagger}$ where $\textit{f}:\cal{L}(\cal{H}_A)\otimes \cal{L}(\cal{H}_B)\rightarrow \cal{L}(\cal{H}_C)\otimes \cal{L}(\cal{H}_{B'})$ and
$\textit{g}:\cal{L}(\cal{H}_{B'})\rightarrow \cal{L}(\cal{H}_B)$. \\
Then
\begin{center}
$\Tr_{AC}^B((1\otimes g)f)=\Tr_{AC}^B((\lambda\rho\sum_j(I\otimes V_j)\rho(I\otimes V_j^{\dagger}))\comp(\lambda\rho\sum_i U_i\rho U_i^{\dagger}))= \Tr_{AC}^B(\lambda\rho\sum_{i,j}(I\otimes V_j)(U_i\rho U_i^{\dagger})(I\otimes V_j^{\dagger}))=\sum_{i,j}\tr_B((I\otimes V_j)U_i)\rho\, \tr_B(U_i^{\dagger}(I\otimes V_j^{\dagger}))=\sum_{i,j}\tr_{B'}(U_i(I\otimes V_j))\rho\, \tr_{B'}((I\otimes V_j^{\dagger})U_i^{\dagger})=
\Tr_{AC}^{B'}(\lambda\rho\sum_{i,j}U_i(I\otimes V_j)\rho (I\otimes V_j^{\dagger})U_i^{\dagger})=
\Tr_{AC}^{B'}((\lambda\rho\sum_i U_i\rho U_i^{\dagger})\comp(\lambda\rho\sum_j(I\otimes V_j)\rho(I\otimes V_j^{\dagger})))=
\Tr_{AC}^{B'}(f(1\otimes g)).$
\end{center}

\textbf{Vanishing I:}\\
Consider the map  $\textit{f}:\cal{L}(\cal{H}_A)\otimes \cal{L}(\cal{H}_I)\rightarrow \cal{L}(\cal{H}_B)\otimes \cal{L}(\cal{H}_{I})$ with the following representation $f=\sum_iU_i-U_i^{\dagger}$, so \\
$\Tr_{A,B}^I(f)=\lambda\rho \sum_i \tr_{I}U_i\rho\,\tr_{I}U_i^{\dagger}=\sum_i U_i\rho\,U_i^{\dagger}.$

\textbf{Vanishing II:}\\
Let us consider $\textit{g}:\cal{L}(\cal{H}_X)\otimes \cal{L}(\cal{H}_U)\otimes \cal{L}(\cal{H}_V)\rightarrow \cal{L}(\cal{H}_Y)\otimes\cal{L}(\cal{H}_U)\otimes \cal{L}(\cal{H}_{V})$ with representation $g=\sum_iE_i-E_i^{\dagger}$ then:
\begin{center}
$T_{X,Y}^{U}(T_{X\otimes U, Y\otimes U}^{V}(g))=T_{X,Y}^{U}(\lambda \rho \sum_i \tr_V E_i\rho\, \tr_V E_i^{\dagger})=
\lambda \rho \sum_i \tr_U (\tr_V (E_i))\rho\, \tr_U (\tr_V (E_i^{\dagger}))=\lambda \rho \sum_i \tr_{U\otimes V} (E_i))\rho\, \tr_{U\otimes V} (E_i^{\dagger})=T_{X,Y}^{U\otimes V}(g).$
\end{center}

\textbf{Yanking:}\\
Before we study the proof of this axiom we consider a representation of the symmetric isomorphism:
\begin{center}
$\sigma_{N,M}:\cal{L}(\cal{H}_N)\otimes \cal{L}(\cal{H}_M)\rightarrow \cal{L}(\cal{H}_M)\otimes \cal{L}(\cal{H}_{N}).$
\end{center}
Let $\{e^n_{i}\}$, $\{e^m_{j}\}$ be an orthonormal basis for $\cal{H}_N$ and $\cal{H}_M$ respectively. Then $\{E^n_{i,j}\}$ and $\{E^m_{k,l}\}$ are orthonormal basis for $\cal{L}(\cal{H}_N)$ and $\cal{L}(\cal{H}_N)$ respectively  with $E^n_{i,j}=e^n_{i}e_{j}^{n\dagger}$, $E^{m}_{k,l}=e^{m}_{k}e^{n\dagger}_{l}$ and $\apair{A,B}=\tr(A^{\dagger}B)$ as a inner product.\\
Thus we have:
\begin{center}
$\sigma (E^n_{i,j}\otimes E^m_{k,l})=\sigma (|e^n_{i}\rangle\langle e^n_{j}|\otimes |e^m_{k}\rangle\langle e^m_{l}|)=\sigma (|e^n_{i}\rangle |e^m_{k}\rangle \otimes \langle e^n_{j}|\langle e^m_{l}|)=U(|e^n_{i}\rangle | e^m_{k}\rangle \otimes \langle e^n_{j}|\langle e^m_{l}|)U^{\dagger}= U|e^n_{i},e^m_{k}\rangle\otimes(U|e^n_{j},e^m_{l}\rangle)^{\dagger}=|e^m_{k},e^n_{i}\rangle\otimes(|e^m_{l},e^n_{j}\rangle)^{\dagger}=|e^m_{k},e^n_{i}\rangle\otimes \langle e^m_{l},e^n_{j}|=|e^m_{k}\rangle\langle e^m_{l}|\otimes |e^n_{i}\rangle\langle e^n_{j}|=E^m_{k,l}\otimes E^n_{i,j}$
\end{center}
for every vector basis where the action $U$ is defined by $U|e^n_i\otimes e^m_j\rangle=|e^m_j\otimes e^n_i \rangle $ on the basis of the tensor space. This implies that $\sigma(A)=UAU^{\dagger}$ for every $A\in \cal{L}(\cal{H}_N)\otimes \cal{L}(\cal{H}_M)$.

Now, let $\sigma_{N,N}:\cal{L}(\cal{H}_N)\otimes \cal{L}(\cal{H}_N)\rightarrow \cal{L}(\cal{H}_N)\otimes \cal{L}(\cal{H}_{N})$ be the symmetric natural isomorphism with the representation  $\sigma_{N,N}=U-U^{\dagger}$,  $\sigma_{N,N}:\cal{H}_N\otimes\cal{H}_N\rightarrow\cal{H}_N\otimes\cal{H}_{N}$ where
$U=\sum_{k,l}|e^m_l\rangle\langle e^n_k|\otimes |e^n_k\rangle\langle e^m_l|$ and $U^{\dagger}=\sum_{k,l}|e^n_k\rangle\langle e^n_l|\otimes |e^n_l\rangle\langle e^n_k|$. Thus we have that $\tr_NU=\sum_{k,l}|e^n_l\rangle\langle e^n_k|\otimes \tr(|e^n_k\rangle\langle e^n_l|)=\sum_{k,l}|e^n_l\rangle\langle e^n_k|\otimes \langle e^n_k| e^n_l\rangle=\sum_{k=l}|e^n_l\rangle\langle e^n_k|=\sum_{l}|e^n_l\rangle\langle e^n_l|=Id_N$. In an analogous way we trace $U^{\dagger}$ obtaining the identity. Hence
$\Tr^N_{N,N}(\sigma_{N,N})(\rho)=\tr_NU\rho\,\tr_NU^{\dagger}=Id_N\rho Id_N=\rho$.
\end{proof}
\begin{remark} The category $\textbf{CPM}_s$ is compact closed, due to the
 existence of a monoidal functor $F:\textbf{fHilb} \rightarrow\textbf{CPM}_s$ which is onto
 objects. (This functor takes each object to itself, and each linear
 map $f$ to $F(\rho) = f\rho f^{\dagger})$. This already implies that this
 category is traced, and moreover that the trace is unique by
 Proposition~\ref{uniqueness compact}.  It is easy to check that the trace is indeed
 computed as above.
 \end{remark}

\subsection{Partial trace in the category $\textbf{Vect}$}\label{PARTIAL TRACE KERN-IMAGE CONDITION}
In Definition~\ref{PARTIAL TRACE IN FINITE VECTOR SPACE}, we considered a partial trace on
the category of finite dimensional vector spaces with $\oplus$ as a tensor product. Now, we relax conditions on the definition of the trace class and we define another partial trace on vector spaces for not necessarily finite dimensions.
\begin{definition}\label{DEFINITION RELAX CONDITION}
\rm
Let $(\textbf{Vect},\oplus,0)$ be the symmetric monoidal category of vector spaces and linear transformations with the monoidal tensor taken to be the direct
sum. We define a trace class in the following way. Given a map $f:V\oplus U\rightarrow W\oplus U$ we say $f\in \Trc^U_{V,W}$ iff
\begin{itemize}\label{X}
\item $im f_{21}\subseteq im (I-f_{22})$ and
\item $ker(I-f_{22})\subseteq ker f_{12}$,
\end{itemize}
where $I$ is the identity map. Whenever these conditions are satisfied we define $\Tr^{U}_{V,W}(f):V\rightarrow W$:
\begin{center}
$\Tr^{U}_{V,W}(f)(v)=f_{11}(v)+f_{12}(u)$ for some $u\in U$ such that $(I-f_{22})(u)=f_{21}(v).$
\end{center}
\end{definition}
To show that this is well-defined, suppose $u'$ is another candidate satisfying $$(I-f_{22})(u')=f_{21}(v).$$

Then $(I-f_{22})(u-u')=0$ which implies by the second condition of Definition~\ref{DEFINITION RELAX CONDITION} that
$f_{12}(u)=f_{12}(u')$. This shows that the value of the trace does not depend on the choice of the pre-image, but on its existence.\\

\begin{remark}
Notice that the partial trace of Definition~\ref{DEFINITION RELAX CONDITION}
 generalizes that of Definition~\ref{PARTIAL TRACE IN FINITE VECTOR SPACE}. Indeed, if $I-f_{22}$ is
 invertible, then $im(I-f_{22})=U$ and $ker(I-f_{22})=0$, which implies that Definition~\ref{DEFINITION RELAX CONDITION} is trivially satisfied and in this case, $\Tr^{U}_{V,W}(f) = f_{11} + f_{12}
(I-f_{22})^{-1} f_{21}$ (where $u=(I-f_{22})^{-1}f_{21}(v)$).
Moreover, Definition~\ref{DEFINITION RELAX CONDITION} is strictly more general than Definition~\ref{PARTIAL TRACE IN FINITE VECTOR SPACE}, because the identity maps are traceable in Definition~\ref{DEFINITION RELAX CONDITION},
 but not in Definition~\ref{PARTIAL TRACE IN FINITE VECTOR SPACE}.
 \end{remark}

\begin{theorem}
 The formula given in Definition~\ref{DEFINITION RELAX CONDITION} is a partial trace.
\end{theorem}
\begin{proof}
$\textbf{Naturality:}$\\
Let $f\in \Trc^{U}_{X,Y}$, $g:X'\rightarrow X$ and $h:Y\rightarrow Y'$ be linear maps. First, we want to prove that
\begin{center}
$(h\oplus 1_U)f(g\oplus 1_U)\in\Trc^{U}_{X',Y'}$ with $(h\oplus
1_U)f(g\oplus 1_U):X'\oplus U\rightarrow Y'\oplus U.$
\end{center}
The following equations are satisfied by naturality on injections and projections:\\
\begin{itemize}
\item $((h\oplus 1_U)f(g\oplus 1_U))_{11}=hf_{11}g$
\item $((h\oplus 1_U)f(g\oplus 1_U))_{12}=hf_{12}$
\item $((h\oplus 1_U)f(g\oplus 1_U))_{21}=f_{21}g$
\item $((h\oplus 1_U)f(g\oplus 1_U))_{22}=f_{22}.$
\end{itemize}

Thus, we have $$im((h\oplus 1_U)f(g\oplus 1_U))_{21}=im f_{21}g\subseteq im f_{21}\subseteq im (I-f_{22})=im (I-((h\oplus 1_U)f(g\oplus 1_U))_{22})$$ by the hypotheses, properties of the image,  and  the equations above.\\
Also, $$ker(I-((h\oplus 1_U)f(g\oplus 1_U))_{22})=ker(I-f_{22})\subseteq ker f_{12}\subseteq ker\, hf_{12}=ker ((h\oplus 1_U)f(g\oplus 1_U))_{12}$$ by the equations above, by the hypothesis, the properties of the  kernel  and the equations above again.\\
Now, we want to check the value of the trace. In view of the definition, we may write:
\begin{center}
$\Tr^{U}_{X',Y'}((h\oplus 1_U)f(g\oplus 1_U))(x)=((h\oplus 1_U)f(g\oplus 1_U))_{11}(x)+((h\oplus 1_U)f(g\oplus 1_U))_{12}(u)$
\end{center}
for some $u\in U$ such that $(I-((h\oplus 1_U)f(g\oplus 1_U))_{22})(u)=(h\oplus 1_U)f(g\oplus 1_U))_{21}(x).$
But, this implies using the equations above that:
\begin{center}
$\Tr^{U}_{X',Y'}((h\oplus 1_U)f(g\oplus 1_U))(x)=hf_{11}g(x)+hf_{12}(u)=h(f_{11}(g(x))+f_{12}(u))=h\Tr^{U}_{X,Y}(f)(g(x))=h\Tr^{U}_{X,Y}(f)g(x)$
\end{center}
for some $u\in U$ such that $(I-f_{22})(u)=f_{21}(g(x))$.\\

$\textbf{Dinaturality}:$\\
For any $f:X\otimes U\rightarrow Y\otimes U'$, $g:U'\rightarrow U$ we must prove that
\begin{center}
$(1_Y\otimes g)f \in \Trc^U_{X,Y}$ iff $f(1_X\otimes g)\in \Trc^{U'}_{X,Y}$
\end{center}
and also we need to check:
$\Tr^{U}_{X,Y}((1_{Y}\otimes g)f)=\Tr^{U'}_{X,Y}(f(1_{X}\otimes g)).$\\
On the one hand, we know by naturality on injections and projections that we have the following equations:
\begin{itemize}
\item $((1_Y\oplus g)f)_{11}=f_{11}$
\item $((1_Y\oplus g)f)_{12}=f_{12}$
\item $((1_Y\oplus g)f)_{21}=gf_{21}$
\item $((1_Y\oplus g)f)_{22}=gf_{22}.$
\end{itemize}
On the other hand we know:
\begin{itemize}
\item $(f(1_X\oplus g))_{11}=f_{11}$
\item $(f(1_X\oplus g))_{12}=f_{12}g$
\item $(f(1_X\oplus g))_{21}=f_{21}$
\item $(f(1_X\oplus g))_{22}=f_{22}g.$
\end{itemize}
First, let us now prove the following equivalence:
\begin{center}
$im((1_Y\oplus g)f)_{21}\subseteq im(I-((1_Y\oplus g)f)_{22})$ iff $im (f(1_X\oplus g))_{21}\subseteq im (I-(f(1_X\oplus g))_{22}).$
\end {center}
By the equations above, it corresponds to the following equivalence:
\begin{center}
$im\, gf_{21}\subseteq im(I-gf_{22})$ iff $im f_{21}\subseteq im (I-f_{22}g).$
\end {center}
($\Rightarrow$) Given $x=f_{21}(z)$ for some $z$ we want to prove that $x\in im (I-f_{22}g).$\\
Since, by hypothesis $g(x)=g(f_{21}(z))\in im(I-gf_{22})$ then $g(f_{21}(z))=z'-g(f_{22}(z'))$ for some $z'$, which implies that $g(f_{21}(z)+f_{22}(z'))=z'.$
Thus, now choose $v=f_{21}(z)+f_{22}(z')$ allowing us to obtain:
\begin{center}
$v-f_{22}(g(v))=f_{21}(z)+f_{22}(z')-f_{22}(g(v))=f_{21}(z)+f_{22}(z')-f_{22}(z')=f_{21}(z)=x.$
\end {center}
($\Leftarrow$) Given $y=g(f_{21}(u))$ for some $u$ we want to prove $y\in im(I-gf_{22}).$\\
Since by hypothesis there is a $z$ such that $f_{21}(u)=z-f_{22}(g(z))$ consider $v=g(z)$; then we get the following:
\begin{center}
$(I-gf_{22})(v)=g(z)-g(f_{22}(g(z)))=g(z-f_{22}(g(z)))=g(f_{21}(u))=y.$
\end {center}
Next, we want to check the following:
\begin{center}
$ker (I-((1_Y\oplus g)f)_{22})\subseteq ker ((1_Y\oplus g)f)_{12}$ iff $ker(I-(f(1_X\oplus g))_{22})\subseteq ker(f(1_X\oplus g))_{12}$
\end{center}
which by the equations above is equivalent to:
\begin{center}
$ker (I-gf_{22})\subseteq ker f_{12}$ iff $ker (I-f_{22}g)\subseteq ker f_{12}g.$
\end{center}
($\Rightarrow $) If $z=f_{22}g(z)$ then  $g(z)=g(f_{22}(g(z)))$ which implies that $g(z)\in ker (I-gf_{22})$ and by hypothesis that $f_{12}(g(z))=0$ i.e., $z\in ker f_{12}g$. \\
($\Leftarrow $) If $v-gf_{22}(v)=0$ then choosing $z=f_{22}(v)$ there is a $z$ such that $g(z)=v$. But, clearly $z\in ker (I-f_{22}g)$ since $v=gf_{22}(v)$ implies:
\begin{center}
$(I-f_{22}g)(z)=f_{22}(v)-f_{22}g(f_{22}(v))=f_{22}(v)-f_{22}(g(f_{22}(v)))=f_{22}(v)-f_{22}(v)=0.$
\end{center}
Then by hypothesis $z\in ker f_{12}g$, which means that $f_{12}g(z)=0$ i.e., $f_{12}(v)=0$.
Hence, we proved that if $v-gf_{22}(v)=0$ then $f_{12}(v)=0.$\\
Now we are ready to check the values of the traces.
\begin{center}
$\Tr^{U}_{X,Y}((1_Y\oplus g)f)(u)=((1_Y\oplus g)f)_{11}(u)+((1_Y\oplus g)f)_{12}(v)$ for some $v$ with $(I-((1_Y\oplus g)f)_{22})(v)=((1_Y\oplus g)f)_{21}(u)$
\end{center}
which by the equations above we get:
\begin{center}
$\Tr^{U}_{X,Y}((1_Y\oplus g)f)(u)=f_{11}(u)+f_{12}(v)$ for some $v$ such that $(I-gf_{22})(v)=gf_{21}(u).$
\end{center}
On the other hand we have that:
\begin{center}
$\Tr^{U'}_{X,Y}(f(1_X\oplus g))(u)=(f(1_{X}\oplus g))_{11}(u)+(f(1_{X}\oplus g))_{12}(v')$ for some $v'$ such that $(I-(f(1_{X}\oplus g))_{22})(v')=(f(1_{X}\oplus g))_{21}(u)$
\end{center}
and again by the equations above:
\begin{center}
$\Tr^{U'}_{X,Y}(f(1_X\oplus g))(u)=f_{11}+f_{12}g(v')$ for some $v'$ such that $(I-f_{22}g)(v')=f_{21}(u).$
\end{center}
($\Rightarrow$) Given $v$ as above there is a $v'$ such that $g(v')=v$ since we have $(I-gf_{22})(v)=gf_{21}(u)$ then $v=g(f_{22}(v)+f_{21}(u))$ so choose $v'=f_{22}(v)+f_{21}(u)$ and this vector satisfies the condition required since
$(I-f_{22}g)(v')=v'-f_{22}g(v')=f_{22}(v)+f_{21}(u)-f_{22}g(v')=f_{21}(u).$\\
($\Leftarrow$) Choose $v=g(v')$ and then we get $(I-gf_{22})(v)=(I-gf_{22})(g(v'))=g(v')-gf_{22}g(v')=g(v'-f_{22}g(v'))=g(I-f_{22}g)(v'))=g(f_{21}(u))=gf_{21}(u).$

$\textbf{Vanishing I}:$\\
Now, we want to check that:
\begin{center}
$\widehat{\Trc}^{I}_{X,Y}=C(X\otimes I,Y\otimes I)$ and
$\Tr^{I}_{X,Y}(f)=\rho_{Y}f\rho^{-1}_{X}.$
\end{center}
Let us consider $f:X\oplus I\rightarrow Y\oplus I$, we notice first that $im f_{21}=im\,0=0\subseteq im(I-f_{22})$ and
$ker (I-f_{22})=ker I=0\subseteq ker f_{12}$ since $f_{12}$, $f_{21}$, $f_{22}$ are constant $0$ functions.\\
Next, we move to the value of the trace:
\begin{center}
$\Tr^{0}_{X,Y}(f)=f_{11}(u)+f_{12}(v)$ for some $v$ such that $(I-f_{22})(v)=f_{21}(u).$
\end{center}
Therefore, since $f_{21}=0$ we choose $v=0$ as a representative and we obtain:
\begin{center}
$\Tr^{0}_{X,Y}(f)=f_{11}(u)=\pi_{11}f\, in_{11}(u)=\rho_Y\,f\,\rho^{-1}_{X}(u)$.
\end{center}
since injection, projection and $\rho$ isomorphism coincide in this case.\\

$\textbf{Vanishing II}:$\\
For any $g:X\otimes U\otimes V\rightarrow Y\otimes U\otimes V$, with $g \in\Trc^{V}_{X\otimes U,Y\otimes U}$ we want to prove the following equivalence:
\begin{center}
$g \in\Trc^{U\otimes V}_{X,Y}$ iff $\Tr^V_{X\otimes U,Y\otimes U}(g)\in \Trc^{U}_{X,Y}.$
\end{center}
We are going to represent $g$ using matrix notation:
\begin{center}
$g=\left(
\begin{array}{ccc}
g_{11}&g_{12}&g_{13}\\
g_{21}&g_{22}&g_{23}\\
g_{31}&g_{32}&g_{33}.
\end{array}\right)$
\end{center}
First, we translate the general hypothesis $g \in\Trc^{V}_{X\otimes U,Y\otimes U}$ in terms of this matrix representation.
\begin{itemize}
\item $\widehat{g_{11}}=\left(
\begin{array}{ccc}
g_{11}&g_{12}\\
g_{21}&g_{22}
\end{array}\right):X\oplus U\rightarrow Y\oplus U$
\item $\widehat{g_{21}}=\left(
\begin{array}{ccc}
g_{31}&g_{32}
\end{array}\right):X\oplus U\rightarrow V$
\item $\widehat{g_{12}}=\left(
\begin{array}{ccc}
g_{13}\\
g_{23}
\end{array}\right):V\rightarrow Y\oplus U$
\item
$\widehat{g_{22}}=\left(
\begin{array}{ccc}
g_{33}
\end{array}\right):V\rightarrow V.$
\end{itemize}
Thus the condition $im\,\widehat{g_{21}}\subseteq im\,(I-\widehat{g_{22}})$ is actually $\,im\left(
\begin{array}{ccc}
g_{31}&g_{32}
\end{array}\right)\subseteq im\,(I-g_{33})$
which implies that:
$\forall x\in X,\forall u\in U,\exists v\in V:g_{31}(x)+g_{32}(u)+g_{33}(v)=v.$
On the other hand, the condition $ker(I-\widehat{g_{22}})\subseteq ker\,\widehat{g_{12}}$ is $ker(I-g_{33})\subseteq ker\left(
\begin{array}{ccc}
g_{13}\\
g_{23}
\end{array}\right)$
which implies that:
$\forall v\in V$ such that $g_{33}(v)=v$ then $g_{13}(v)+g_{23}(v)=0$.\\
We are now ready to translate the condition $g \in\Trc^{U\otimes V}_{X,Y}$ in terms of the matrix representation of $g$.
\begin{itemize}
\item $\tilde{g}_{11}=\left(
\begin{array}{ccc}
g_{11}
\end{array}\right):X\rightarrow Y$
\item $\tilde{g}_{21}=\left(
\begin{array}{ccc}
g_{21}\\
g_{31}
\end{array}\right):X\rightarrow U\oplus V$
\item $\tilde{g}_{12}=\left(
\begin{array}{ccc}
g_{12}&g_{13}
\end{array}\right):U\oplus V\rightarrow Y$
\item
$\tilde{g}_{22}=\left(
\begin{array}{ccc}
g_{22}&g_{23}\\
g_{32}&g_{33}
\end{array}\right):U\oplus V\rightarrow U\oplus V.$
\end{itemize}
Thus the condition $im\,\tilde{g}_{21}\subseteq im\,(I-\tilde{g}_{22})$ is actually $\,im\left(
\begin{array}{ccc}
g_{21}\\
g_{31}
\end{array}\right)\subseteq im\,(I-\left(
\begin{array}{ccc}
g_{22}&g_{23}\\
g_{32}&g_{33}
\end{array}\right))$
which implies that:
$\forall x\in X,\exists u\in U,\exists v\in V:
\left\{
\begin{array}{ccc}
g_{21}(x)+g_{22}(u)+g_{23}(v)=u \\
g_{31}(x)+g_{32}(u)+g_{33}(v)=v
\end{array}\right)$ \\
On the other hand, the condition $ker(I-\tilde{g}_{22})\subseteq ker\,\tilde{g}_{12}$ is $ker(I-\left(
\begin{array}{ccc}
g_{22}&g_{23}\\
g_{32}&g_{33}
\end{array}\right))\subseteq ker\left(
\begin{array}{ccc}
g_{12}&g_{13}
\end{array}\right )$
which implies that:
$ \forall u\in U, \forall v\in V$ such that $u=g_{22}(u)+g_{23}(v)$ and $v=g_{32}(u)+g_{33}(v)$ then $g_{12}(u)+g_{13}(v)=0.$\\

Now we express  $\Tr^V_{X\oplus U,Y\oplus U}(g)\in \Trc^U_{X,Y}$ in terms of the components of $g$\\
$\Tr^V_{X\oplus U,Y\oplus U}(g)(x,u)=\widehat{g_{11}}(x,u)+\widehat{g_{12}}(v)$ for some $v\in V$ such that $(I-\widehat{g_{22}})(v)=\widehat{g_{21}}(x,u)$ which implies:
\begin{center}
$\Tr^V_{X\oplus U,Y\oplus U}(g)(x,u)=(g_{11}(x)+g_{12}(u),g_{21}(x)+g_{22}(u))+(g_{13}(v),g_{23}(v))$ for some $v\in V$ such that $v-g_{33}(v)=g_{31}(x)+g_{32}(u).$
\end{center}
Now we renamed $\bar{g}=\Tr^V_{X\oplus U,Y\oplus U}(g)$ and compose with injections and projections.
\begin{itemize}
\item $\bar{g}_{11}=\pi_{1}\bar{g}\,in_{1}:X\rightarrow Y$, $\bar{g}_{11}(x)=g_{11}(x)+g_{13}(v_1)$ with $v_1$ such that $v_1-g_{33}(v_1)=g_{31}(x)$
\item $\bar{g}_{21}=\pi_{2}\bar{g}\,in_{1}:X\rightarrow U$, $\bar{g}_{21}(x)=g_{21}(x)+g_{23}(v_1)$ with $v_1$ such that $v_1-g_{33}(v_1)=g_{31}(x)$
\item $\bar{g}_{12}=\pi_{1}\bar{g}\,in_{2}:U\rightarrow Y$, $\bar{g}_{12}(u)=g_{12}(u)+g_{13}(v_2)$ with $v_2$ such that $v_2-g_{33}(v_2)=g_{32}(u)$
\item $\bar{g}_{22}=\pi_{2}\bar{g}\,in_{2}:U\rightarrow U$, $\bar{g}_{22}(u)=g_{22}(u)+g_{23}(v_2)$ with $v_2$ such that $v_2-g_{33}(v_2)=g_{32}(u).$
\end{itemize}
Thus we have that:
\begin{center}
$\bar{g}\in\Trc^U_{X,Y}$ iff $im\, \bar{g}_{21}\subseteq im\,(I-\bar{g}_{22})$ and $ker(I-\bar{g}_{22})\subseteq ker\bar{g}_{12}.$
\end{center}
By the equations above the condition $im\, \bar{g}_{21}\subseteq im\,(I-\bar{g}_{22})$ implies that
\begin{center}
$\forall x\in X, \forall v_1\in V$ such that $v_1-g_{33}(v_1)=g_{31}(x), \exists u\in U, \exists v_2\in V$ such that\\ $v_2-g_{33}(v_2)=g_{32}(u)$ and $g_{21}(x)+g_{23}(v_1)+g_{22}(u)+g_{23}(v_2)=u.$
\end{center}
On the other hand, the condition
$ker(I-\bar{g}_{22})\subseteq ker\bar{g}_{12}$ implies by the equations above that
\begin{center}
$\forall u\in U, \forall v_2 \in V$ such that $v_2-g_{33}(v_2)=g_{32}(u)$, if $g_{22}(u)+g_{23}(v_2)=u$ then $g_{12}(u)+g_{13}(v_2)=0.$
\end{center}
Now since we have all the conditions in term of $g$ we can prove the equivalence.

\vspace{1ex}

\noindent
$(\Rightarrow )$ We have by general hypothesis that
the condition $im\,\widehat{g_{21}}\subseteq im\,(I-\widehat{g_{22}})$ is actually
$\forall x\in X,\forall u\in U,\exists v\in V:g_{31}(x)+g_{32}(u)+g_{33}(v)=v$. We also have now as hypothesis that
the condition $im\,\tilde{g}_{21}\subseteq im\,(I-\tilde{g}_{22})$ is
$\forall x\in X,\exists u\in U,\exists v\in V:
g_{21}(x)+g_{22}(u)+g_{23}(v)=u $ and
$g_{31}(x)+g_{32}(u)+g_{33}(v)=v$.\\
By the equations above we want to prove that:
$$\forall x\in X, \forall v_1 \in V  \ \  \mbox{such that} \ \  v_1 -g_{33}(v_1)=g_{31}(x)\qquad (\alpha)$$ then $\exists u\in U, \exists v_2 \in V$ such that the following two
equations hold: $$v_2 -g_{33}(v_2 )=g_{32}(u)\qquad (\beta)$$   $$ g_{21}(x)+g_{23}(v_1 )+g_{22}(u)+g_{23}(v_2 )=u \qquad (\gamma).$$
By hypothesis given $x\in X$, let us consider $u,v$ such that $g_{21}(x)+g_{22}(u)+g_{23}(v)=u $ and
$g_{31}(x)+g_{32}(u)+g_{33}(v)=v$. Now choose $v_2=v-v_1$; then we have that $g_{21}(x)+g_{23}(v_1)+g_{22}(u)+g_{23}(v-v_1)=g_{21}(x)+g_{22}(u)+g_{23}(v)=u $ which proves equation $(\gamma)$ using  the first of the equations above. It can be seen that:\\
$v_2=v-v_1=g_{31}(x)+g_{32}(u)+g_{33}(v)-v_1=g_{31}(x)+g_{32}(u)+g_{33}(v)-(g_{33}(v_1)+g_{31}(x))=g_{32}(u)+g_{33}(v)-g_{33}(v_1)=g_{32}(u)+g_{33}(v-v_1)=g_{32}(u)+g_{33}(v_2) $ which proves equation $(\beta)$ using the equations above.

\noindent
$(\Leftarrow)$
Now assume the same general hypothesis as before: $\forall x\in X,\forall u\in U,\exists v\in V:g_{31}(x)+g_{32}(u)+g_{33}(v)=v$. We know by hypothesis that:\\
$\forall x\in X, \forall v_1 \in V$ such that $v_1 -g_{33}(v_1)=g_{31}(x), \exists u\in U, \exists v_2 \in V$ such that\\ $v_2 -g_{33}(v_2 )=g_{32}(u)$ and $g_{21}(x)+g_{23}(v_1 )+g_{22}(u)+g_{23}(v_2 )=u$.\\
We want to prove that:
$\forall x\in X, \exists u \in U, \exists v \in V  \mbox{such that}$ $$g_{21}(x)+g_{22}(u)+g_{23}(v)=u \qquad (\star) $$
$$g_{31}(x)+g_{32}(u)+g_{33}(v)=v \qquad (\star\star)$$
Using the general hypothesis with $u=0$ we obtain:
$\forall x\in X,\exists v_1\in V:$$$g_{31}(x)+g_{32}(0)+g_{33}(v_1)=v_1 . \qquad (\star\star\star) $$  Now by hypothesis we have:
given $ x\in X, v_1 \in V$, since $v_1=g_{31}(x)+g_{33}(v_1)$ we have that $ \exists u\in U, \exists v_2 \in V$ such that $$v_2 -g_{33}(v_2 )=g_{32}(u) \qquad(\textbf{1})$$  $$g_{21}(x)+g_{23}(v_1 )+g_{22}(u)+g_{23}(v_2 )=u \qquad (\textbf{2}).$$
Now consider $v=v_1+v_2$ we have by the equation above $(\textbf{2})$ that:
$g_{21}(x)+g_{23}(v_1+v_2 )+g_{22}(u)=u$ which proves $(\star)$. We also have that $v_1+v_2=g_{31}(x)+g_{33}(v_1)+g_{33}(v_2 )+g_{32}(u)$ by adding equations $(\textbf{1})$ and $(\star\star\star)$. Thus $v_1+v_2=g_{31}(x)+g_{33}(v_1+v_2 )+g_{32}(u)$ which proves $(\star\star)$.

Now we move to checking that the condition on kernels is also satisfied.
It follows from the general hypothesis that:
$\forall v\in V$ such that $g_{33}(v)=v$ then $g_{13}(v)+g_{23}(v)=0$.

\vspace{1ex}

\noindent
$(\Rightarrow)$ By hypothesis we know that the two equations $$\forall u\in U ,  \forall v\in V    u=g_{22}(u)+g_{23}(v)\qquad (\star)_1$$  $$v=g_{32}(u)+g_{33}(v)\qquad (\star)_2$$ imply $g_{12}(u)+g_{13}(v)=0$.
We want to prove that: $\forall u\in U$ if $\exists v_2\in V: v_2=g_{33}(v_2)+g_{32}(u)$ with $g_{22}(u)+g_{23}(v_2)=u$ then $g_{12}(u)+g_{13}(v_2)=0\qquad\qquad (\star\star).$\\
So, given $u\in U, v_2\in V: v_2=g_{33}(v_2)+g_{32}(u)$ with $g_{22}(u)+g_{23}(v_2)=u$ then by hypothesis since $u=g_{22}(u)+g_{23}(v_2)$ and $v_2=g_{32}(u)+g_{33}(v_2)$ so $(\star)_1$ and $(\star)_2$ are satisfied with $v=v_2$ which implies $g_{12}(u)+g_{13}(v_2)=0\,\,(\star\star)$.

\noindent
$(\Leftarrow)$By a similar argument with $v=v_2$.\\[1ex]
The values of the trace are conditioned by the implications above.
According to these equations we have that:
$\Tr^V_{X\oplus U,Y\oplus V}(g)(x,u)=\widehat{g_{11}}(x,u)+\widehat{g_{12}}(v)$ for some $v$ such $(I-\widehat{g_{22}})(v)=\widehat{g_{21}}(x,u)$. If we apply $\Tr^U_{X,Y}$ to this function, it is equivalent in terms of the $g$ to $\Tr^U_{X,Y}(\Tr^V_{X\oplus U,Y\oplus V}(g))(x)=g_{11}(x)+g_{13}(v_1+v_2)+g_{12}(u)$
with $u\in U$, $v_1\in V$, $v_2\in V$ such that: $u=g_{21}(x)+g_{23}(v_1+v_2)+g_{22}(u)$, $v_1=g_{31}(x)+g_{33}(v_1)$ and $v_2=g_{32}(u)+g_{33}(v_2)$.

On the other hand, we may also calculate $\Tr^{U\otimes V}_{X,Y}(g)(x)=\tilde{g}_{11}(x)+\tilde{g}_{12}(u,v)$ for some $u\in U$, $v\in V$ such that $(I-\tilde{g}_{22}(u,v))=\tilde{g}_{21}(x)$ and we get by the equations above:
$\Tr^{U\otimes V}_{X,Y}(g)(x)=g_{11}(x)+g_{13}(v)+g_{12}(u)$ with $u=g_{21}(x)+g_{23}(v)+g_{22}(u)$, $v=g_{31}(x)+g_{33}(v)+g_{32}(u)$.
In both implications we obtain the same value of the trace. Notice that the value is independent of the choice of the vectors that satisfy the auxiliary conditions.  When we chose $v=v_1+v_2$ we have:\\
$\Tr^U_{X,Y}(\Tr^V_{X\oplus U,Y\oplus V}(g))(x)=g_{11}(x)+g_{13}(v_1+v_2)+g_{12}(u)=\Tr^{U\otimes V}_{X,Y}(g)(x)$\\
and when we chose $v_2=v-v_1$ we have \\
$\Tr^U_{X,Y}(\Tr^V_{X\oplus U,Y\oplus V}(g))(x)=g_{11}(x)+g_{13}(v_1+v_2)+g_{12}(u)=g_{11}(x)+g_{13}(v_1+v-v_1)+g_{12}(u)=g_{11}(x)+g_{13}(v)+g_{12}(u)= \Tr^{U\otimes V}_{X,Y}(g)(x)$.\\

$\textbf{Superposing}$:\\
Suppose now that $f\in \widehat{\Trc}^{U}_{X,Y}$ and $g:W\rightarrow Z$; we want to prove that $g\oplus
f\in \widehat{\Trc}^{U}_{W\oplus X,Z\oplus Y}$.\\
First, we start writing the matrix representation of $g\oplus f$ in terms of $g$.
\begin{itemize}
\item $(g\oplus f)_{11}=g\oplus f_{11}:W\oplus X\rightarrow Z\oplus Y$
\item $(g\oplus f)_{21}=\left(
\begin{array}{ccc}
0&f_{21}
\end{array}\right):W\oplus X\rightarrow U$
\item $(g\oplus f)_{12}=\left(
\begin{array}{ccc}
0\\
f_{12}
\end{array}\right):U\rightarrow Z\oplus Y$
\item $(g\oplus f)_{22}=f_{22}:U\rightarrow U.$
\end{itemize}
If $z\in im\,\,(g\oplus f)_{21}$ then $z=0w+f_{21}(x)$ for some $w\in W$, $x\in X$ which by hypothesis and the equation above  implies that $f_{21}(x)\in  im\,\,(I-f_{22})=im\,\,(I-(g\oplus f)_{22})$.\\
On the other hand, we have $ker (I-(g\oplus f)_{22})=ker (I- f_{22})\subseteq ker f_{12}\subseteq ker \left(
\begin{array}{ccc}
0\\
f_{12}
\end{array}\right)=ker (g\oplus f)_{12}$
by hypothesis and properties of kernels.\\
Now we evaluate the traces:
\begin{center}
$\Tr^U_{W\oplus X,Z\oplus Y}(g\oplus f)(w,x)=(g\oplus f)_{11}(w,x)+(g\oplus f)_{12}(u)=g\oplus f_{11}(w,x)+\left(
\begin{array}{ccc}
0\\
f_{12}
\end{array}\right)(u)=(g(w),f_{11}(x))+(0,f_{12}(u))=(g(w),f_{11}(x)+f_{12}(u))=(g(w),\Tr^U_{X,Y}(f)(x)=(g\oplus \Tr^U_{X,Y}(f))(w,x)$
\end{center}
with
$u-f_{22}(u)=\left(
\begin{array}{ccc}
0&f_{21}
\end{array}\right)(w,x)$ which by the equations above is equivalent to $u-f_{22}(u)=f_{21}(x).$
Thus
\begin{center}
$\Tr^U_{W\oplus X,Z\oplus Y}(g\oplus f)=g\oplus \Tr^U_{X,Y}(f).$
\end{center}
$\textbf{Yanking:}$\\
We want to prove that $\sigma_{U,U}\in \Trc^{U}_{U,U}$, and also $\Tr^U_{U,U}(\sigma_{U,U})=1_U$ where $\sigma_{U,U}:U\oplus U\rightarrow U\oplus U$ is the coherent isomorphism.
\begin{itemize}
\item $\sigma_{11}=\pi_{1}\sigma_{UU}\,in_{1}:U\rightarrow U$, with $\sigma_{11}=0$
\item $\sigma_{21}=\pi_{2}\sigma_{UU}\,in_{1}:U\rightarrow U$, with $\sigma_{21}=id_U$
\item $\sigma_{12}=\pi_{1}\sigma_{UU}\,in_{2}:U\rightarrow U$, with $\sigma_{12}=id_{U}$
\item $\sigma_{22}=\pi_{2}\sigma_{UU}\,in_{2}:U\rightarrow U$, with $\sigma_{22}=0.$
\end{itemize}
Thus, we have $\sigma_{21}(u)=u=(I-\sigma_{22})(u)$ which means that $im\,\,\sigma_{21}\subseteq im\,\,(I-\sigma_{22})$.\\
On the other hand we have that if $u=\sigma_{22}(u)$ then $u=0$. This means that $ker (I-\sigma_{22})\subseteq ker\,\, \sigma_{12}$.\\
The value of the trace is the following:
\begin{center}
$\Tr^U_{U,U}(\sigma_{UU})(u)=\sigma_{11}(u)+\sigma_{12}(v)=0+v=v$
\end{center}
with the condition: $(I-\sigma_{22})(v)=\sigma_{21}(u)$ for some $v\in U$. But this implies by the equations above that $v=u$. Thus $\Tr^U_{U,U}(\sigma_{UU})(u)=u$, i.e., $\Tr^U_{U,U}(\sigma_{UU})=id_U$.

\end{proof}

\subsection{Completely positive maps with $\oplus$}\label{4.2.5}
\begin{definition}\label{PARTIAL TRACE IN CPM+}
\rm
On the category $\textbf{CPM}$ with monoidal structure $\oplus$, we define a partial
trace as follows. We say that $f\in \cpmTrc^{U}_{X,Y}$ for some objects $X$, $Y$, $U$ iff
\begin{itemize}
\item[(a)]$(I-f_{22})$ is invertible as linear function and
\item[(b)] the inverse map $(I-f_{22})^{-1}$ is a completely positive map.
\end{itemize}
We define $\cpmTr^{U}_{X,Y}(f)=f_{11}+f_{12}(I-f_{22})^{-1}f_{21}$ where $I$ is the identity map.
\end{definition}
Thus, we are demanding that $(I-f_{22})^{-1}$ should be regarded as an inverse in the category \textbf{CPM}.
\begin{lemma}\label{INVERSE MATRIX}
Let $M=\left[\begin{array}{ll}
A & B\\
C& D
\end{array} \right]$
be a partitioned matrix with sub-block $A\in Mat_{m\times m}$, $B\in Mat_{m\times n}$, $C\in Mat_{n\times m}$ and $D\in Mat_{n\times n}$. Assume $D$ is
invertible. Then $M$ is invertible if and only if $A-BD^{-1}C$ is invertible.
\end{lemma}
\begin{lemma}\label{INVERSE COMPOSITION MATRIX}
Let us consider $A\in Mat_{m\times n}$ and $B\in Mat_{n\times m}$. Then $(I_{m}-AB)$ is invertible if and only if $(I_{n}-BA)$ is invertible and $(I_{m}-AB)^{-1}A=A(I_{n}-BA)^{-1}$.
\end{lemma}
\begin{proposition}
$( {\bf CPM},\oplus, 0)$ is a partially traced category with respect to Definition~\ref{PARTIAL TRACE IN CPM+}.
\end{proposition}
\begin{proof}
The partial trace axioms, restricted to condition (a) of Definition~\ref{PARTIAL TRACE IN CPM+}, are basically proved in \cite{HS05a}. This picture is completed by adding the proof of the trace axioms for the positiveness condition (b) of Definition~\ref{PARTIAL TRACE IN CPM+}.

\textbf{Vanishing I:}\\
This follows from the definition of the unit $I$ as the empty list and the fact that the identity map is an invertible map where its inverse is a completely positive map. Thus $\cpmTrc^{I}_{X,Y}=\mbox{\textbf{CPM}}(X,Y)$ and $\cpmTr^{I}_{X,Y}(f)=f$ for every $f\in \cpmTrc^{I}_{X,Y}$.\\

$\textbf{Superposing:}$\\
Let us consider $ f\in \cpmTrc^{U}_{X,Y}$ and $g:W\rightarrow Z$ then $g\oplus f\in \cpmTrc^{U}_{W\oplus X,Z\oplus Y}$ since $(g\oplus f)_{22}=f_{22}$.
We also have:
$\cpmTr^{U}_{W\oplus X,Z\oplus Y}(g\oplus f)=
\left[\begin{array}{llll}
g & 0\\
0&  f_{11}+f_{12}(I-f_{22})^{-1}f_{21}\\
\end{array}\right]=
g\oplus \cpmTr^{U}_{X,Y}(f)$.
\\

$\textbf{Naturality:}$\\
If $f \in \cpmTrc^{U}_{X,Y}$ and we have two arrows $g:X'\rightarrow X$, $h:Y\rightarrow Y'$ then since
$$((h\oplus id_U)f(g\oplus id_U))_{22}=f_{22}$$
 always is satisfied for linear maps since composition computes as matrix product i.e.,\\

$\left[\begin{array}{llll}
h & 0\\
0& 1_{u}
\end{array}\right]$.
$\left[\begin{array}{llll}
f_{11} & f_{12}\\
f_{21} & f_{22}
\end{array}\right]$.
$\left[\begin{array}{llll}
g & 0\\
0& 1_{u}
\end{array}\right]=
\left[\begin{array}{llll}
hf_{11}g & hf_{12}\\
f_{21}g & f_{22}
\end{array}\right]$\\

Thus then the conditions remain exactly the same, meaning that
$$((h\oplus id_u)f(g\oplus id_u))_{22}=f_{22}\in \cpmTrc^{U}_{X',Y'}.$$
Moreover
$$\cpmTr^{U}_{X',Y'}((h\oplus 1_{u}) f(g\oplus 1_{u})=hf_{11}g+(hf_{12})(f_{22})(f_{21}g) =h(f_{11}+f_{12}f_{22}f_{21})g=
h \cpmTr^{U}_{X,Y}(f)g.$$

$\bf Yanking:$\\
Note that $s_{U,U}\in \cpmTrc^{U}_{U,U}$ since $(s_{U,U})_{2,2}=0$ which implies that $I-(s_{U,U})_{2,2}$ is invertible and $(I-0)^{-1}$
is a completely positive map.
Moreover $\cpmTr^{U}_{U,U}(\sigma_{U,U})=1_{u}$ since $(s_{U,U})_{1,1}=(s_{U,U})_{2,2}=0$ and $(s_{U,U})_{1,2}=(s_{U,U})_{2,1}=1_{u}$.\\

$\bf Vanishing\,\,\,II:$\\
Let us consider $g:X \oplus U \oplus V \rightarrow Y \oplus U \oplus V$, we write using matrix notation $g=
\left[\begin{array}{llll}
a & b &  c\\
d & e &  f\\
m & n & p\\
\end{array}\right]$.

Now, assuming by hypothesis that $g \in \cpmTrc^{V}_{X\oplus U,Y\oplus U}$, i.e., $I-p$ is invertible and $(I-p)^{-1}$ is a completely positive map we must show that $g \in \cpmTrc^{U\oplus V}_{X,Y}$ iff $\cpmTr^{V}_{X\oplus U,Y\oplus U}(g)\in \cpmTrc^{U}_{X,Y}$.

 First, we analyze the conditions of definition \ref{PARTIAL TRACE IN CPM+} in terms of its matrix term components. If we represent functions using matrix notation we have:

 $\cpmTr^{V}_{X\oplus U,Y\oplus U}(g)
\left(\begin{array}{llll}
x\\
u\\
\end{array}\right)
=\left[
\left(
\begin{array}{llll}
a & b \\
d & e \\
\end{array}
\right)
+
\left(\begin{array}{llll}
c\\
f\\
\end{array}\right)
\comp (1-p)^{-1} \comp
 \left(
\begin{array}{llll}
m& n\\
\end{array}
\right)
\right]
\left(\begin{array}{llll}
x\\
u\\
\end{array}\right)
$
and we obtain
\begin{center}
$(\cpmTr^{V}_{X\oplus U,Y\oplus U}(g))_{22}(u)=(e+f(1-p)^{-1}n)(u)$
\end{center}
by composing with the second injection and the second projection.\\
Thus we know by definition:\\
$\cpmTr^{V}_{X\oplus U,Y\oplus U}(g)\in \cpmTrc^{U}_{X,Y}$ i.e., $I-e-f(1-p)^{-1}n$ is invertible and $(I-e-f(1-p)^{-1}n)^{-1}$ is a complete positive map.\\
On the other hand, $g \in \cpmTrc^{U\oplus V}_{X,Y}$ i.e., $I-\left(\begin{array}{llll}
 e &  f\\
 n & p\\
\end{array}\right)$ is invertible and $(I-\left(\begin{array}{llll}
 e &  f\\
 n & p\\
\end{array}\right))^{-1}$ is a complete positive map.
Also we obtain the explicit inverse by

\[
  (I-\left(\begin{array}{llll}
      e &  f\\
      n & p\\
    \end{array}\right))^{-1}=
  \left[\begin{array}{llll}
      (I-e-fqn)^{-1} &  (I-e-fqn)^{-1}(f)q\\
      qn(I-e-fqn)^{-1} & q+qn(I-e-fqn)^{-1}fq
    \end{array}\right]
\]
where $q = (I-p)^{-1}.$
Now we prove the equivalence on the trace class:

First of all, by Lemma~\ref{INVERSE MATRIX} above we get:
if we know that $(I-p)$ is invertible then
$(I-\left(\begin{array}{llll}
 e &  f\\
 n & p\\
\end{array}\right))$ is invertible iff $I-e-f(1-p)^{-1}n$ is invertible, which means that the first part of the definition is satisfied.
Also from Lemma~\ref{INVERSE MATRIX} we have that the equation on traces
$\cpmTr^{U\oplus V}_{X,Y}(g)=\cpmTr^{U}_{X,Y}(\cpmTr^V_{X\oplus U,Y\oplus U}(g))$ is satisfied by
using matrix multiplication and the explicit inverse
$(I-\left(\begin{array}{llll}
 e &  f\\
 n & p\\
\end{array}\right))^{-1}$  written above.\\

Positiveness condition (b):\\
$(\Leftarrow )$
Injections and projections are completely positive maps and by the fact that $g$ is a completely positive map this implies by definition that $n$ and $f$ are complete positive maps. Also, $(I-e-f(1-p)^{-1}n)^{-1}=(I-(\cpmTr^{V}_{X\oplus U,Y\oplus U}(g))_{22})^{-1}$ is a completely positive map by conditional hypothesis and $(I-p)^{-1}$ is also a completely positive map by the general hypothesis. This implies that $(I-\left(\begin{array}{llll}
 e &  f\\
 n & p\\
\end{array}\right))^{-1}$
is a completely positive since each component of the matrix is obtained by sum and composition of completely positive maps.

$(\Rightarrow)$
If $(I-\left(\begin{array}{llll}
 e &  f\\
 n & p\\
\end{array}\right))^{-1}$ is a completely positive map then
$\pi _1 \comp (I-\left(\begin{array}{llll}
 e &  f\\
 n & p\\
\end{array}\right))^{-1}
\comp i_1= (I-e-f(I-p)^{-1}n)^{-1}$ is a completely positive map where $\pi_1$ and $i_1$ are the first projection and first injection. Therefore, we showed that $(I-e-f(I-p)^{-1}n)^{-1}=(I-(e+f(1-p)^{-1}n))^{-1}=(I-(\cpmTr^{v}_{X\oplus U,Y\oplus U}(g))_{22})^{-1}$ is a completely positive map.\\

$\textbf{Dinaturality:} $\\
Now, suppose $f:X\oplus U\rightarrow Y\oplus U'$ and $g:U'\rightarrow U$ are completely positive maps then we want to prove that $(id_y\oplus g)\comp f \in \cpmTrc^{U}_{X,Y}$ iff $f\comp (id_x\oplus g) \in \cpmTrc^{U'}_{X,Y}$ which means that $((id_y\oplus g)\comp f)_{22}=g\comp f_{22}$ satisfies conditions (a) and (b) of Definition~\ref{PARTIAL TRACE IN CPM+} if and only if $(f\comp (id_x\oplus g))_{22}=f_{22}\comp g$ does.

Given $f:X\oplus U\rightarrow Y\oplus U'$ and $g:U'\rightarrow U$ by Lemma~\ref{INVERSE COMPOSITION MATRIX} above,  $I-g\comp f_{22}$ is invertible if and only if $I-f_{22}\comp g$ is invertible and we have that
$$\cpmTr^{U}_{X,Y}((1_{y}\oplus g)f)=f_{11}+f_{12}(I-gf_{22})^{-1}gf_{21}=f_{11}+f_{12}g(I-f_{22}g)^{-1}f_{21}=\cpmTr^{U'}_{X,Y}(f(1_{x}\oplus g)).$$

Therefore, it suffices to prove the following:
if $I_{u}-g\comp f$ is invertible and $(I_{u}-g\comp f)^{-1}$ is a completely positive map then
$(I_{u'}-f\comp g)^{-1}$ is a completely positive map, where $f:U\rightarrow U'$ and $g:U'\rightarrow U$.\\
We know by hypothesis that
$$\forall \tau, \forall A'\in V_{\tau}\otimes V_{u},\,\,\, \mbox{if}\,\,\, A'\geq 0 \,\,\,\mbox{then}\,\,\, (Id_{\tau}\otimes (I_{u}-g\comp f)^{-1})(A')\geq 0$$
and we want to prove that
$$\forall \tau, \forall A\in V_{\tau}\otimes V_{u'},\,\,\,\mbox{if}\,\,\, A\geq 0\,\,\,\mbox{then}\,\,\, (Id_{\tau}\otimes (I_{u'}-f\comp g))^{-1})(A)\geq 0.$$
Suppose we name $(Id_{\tau}\otimes (I_{u'}-f\comp g)^{-1})(A)=B$ then by hypothesis
\begin{equation}\label{EQUATION POSITIVENESS}
A=(Id_{\tau}\otimes (I_{u'}-f\comp g))(B)\geq 0.
\end{equation}
Since $g$ is a completely positive map this implies that: if $A\geq 0$ then $(Id_{\tau}\otimes g)(A)\geq 0$; next we apply this property to equation~(\ref{EQUATION POSITIVENESS}).

\noindent
So, we get:\\
$0\leq (Id_{\tau}\otimes g)\comp (Id_{\tau}\otimes (I_{u'}-f\comp g))(B)=(Id_{\tau}\otimes g(I_{u'}-f\comp g))(B)=\\(Id_{\tau}\otimes (g-g\comp f\comp g))(B) =(Id_{\tau}\otimes (I_{u}-g\comp f)g)(B)=(Id_{\tau}\otimes (I_{u}-g\comp f))\comp (Id_{\tau}\otimes g)(B)\,\,\,$
which implies (rename it $C$)
$$(Id_{\tau}\otimes (I_{u}-g\comp f))((Id_{\tau}\otimes g)(B))=C \geq 0.$$
Thus we have
$$(Id_{\tau}\otimes g)(B)=(Id_{\tau}\otimes (I_{u}-g\comp f))^{-1}(C)=(Id_{\tau}\otimes (I_{u}-g\comp f)^{-1})(C)\geq 0$$
since $(I_{u}-g\comp f)^{-1}$ is a completely positive map by hypothesis. Therefore,  $(Id_{\tau}\otimes g)(B)\geq 0$ and on the other hand $f$ is a completely positive map which implies $(Id_{\tau}\otimes f)((Id_{\tau}\otimes g)(B))\geq 0$, which means $(Id_{\tau}\otimes f\comp g)(B)\geq 0$.

Finally, since if $A\geq 0$ implies $(Id_{\tau}\otimes I_{u'})(B)- (Id_{\tau}\otimes f\comp g)(B)\geq 0$ by equation~(\ref{EQUATION POSITIVENESS}) this implies that $B\geq (Id_{\tau}\otimes f\comp g)(B)$ hence $B=(Id_{\tau}\otimes (I_{u'}-f\comp g)^{-1})(A)\geq 0$ by transitivity for every $\tau$.
For the converse implication we repeat this argument interchanging $f$ and $g$.
\end{proof}

\section{Partial trace in a monoidal subcategory of a partially traced category}
\label{TRACE IN A MONOIDAL SUBCATEGORY}
The aim of this section is to provide a general construction of partially traced categories as
subcategories of other partially (or totally) traced categories.

Suppose $(\cD,\otimes,I,s,\Tr)$ is a partially traced category with trace
$$\Tr^{U}_{X,Y}:\cD(X\otimes U,Y\otimes U)\rightharpoonup \cD(X,Y).$$
Given a monoidal subcategory $\cC\subseteq \cD$, we get a partial trace on $\cC$,
defined by
$\widehat{\Tr}^U_{X,Y}(f)=\Tr^U_{X,Y}(f)$ if $\Tr^U_{X,Y}(f)\downarrow$
and $\Tr^U_{X,Y}(f)\in\cC(X,Y)$, and undefined otherwise.

More generally, we shall show a method of constructing one partially traced category from another in such a way that the first one is faithfully embedded in the second.\\
\begin{proposition}
\label{PROPOSITION SUBCATEGORY}
Let $F:\cC\rightarrow \cD$ be a faithful strong symmetric monoidal functor with $(\cD,\otimes,I,s,\Tr)$ a partially traced category and $(\cC,\otimes,I,s)$ a symmetric monoidal category. Then we obtain a partial trace $\widehat{\Tr}$ on $\cC$ as follows.
 For $f:X\otimes U\rightarrow Y\otimes U$, we define
 $\widehat{\Tr}^U_{X,Y}(f)=g$ if there exists some (necessarily unique)
 $g:X\rightarrow Y$ such that $F(g) = \Tr^{FU}_{FX,XY}(m^{-1}_{Y,U}\circ F(f)\circ
 m_{X,U})$ is defined, and $\widehat{\Tr}^U_{X,Y}(f)$ undefined otherwise.
\end{proposition}

\begin{proof}
To clarify the notation used here we recall that there are two partial functions:

 \[  \widehat{\Tr}^U_{X,Y} : \cC(X\otimes U, Y\otimes U)\rightharpoonup \cC(X,Y)
 \]
 and
 \[  \Tr^U_{X,Y}: \cD(X\otimes U, Y\otimes U) \rightharpoonup \cD(X,Y).
 \]

Then we have two maps
$$\widehat{\Tr}^{U}_{X,Y}:\widehat{\Trc}^U_{X,Y}\rightarrow \cC(X,Y)$$
where $\widehat{\Trc}^{U}_{X,Y}\subseteq \cC(X\otimes U,Y\otimes U)$ and we also have
$$\Tr^{U}_{X,Y}:\Trc^U_{X,Y}\rightarrow \cD(X,Y)$$
where $\Trc^{U}_{X,Y}\subseteq \cD(X\otimes U,Y\otimes U).$\\

$\textbf{Naturality}:$\\

For any $X$, $Y$, $U$ objects in $\cC$, $f \in \widehat{\Trc}^{U}_{X,Y}$ and $g:X'\rightarrow X$, $h:Y\rightarrow Y'$ arrows in $\cC$. We want to prove that the two conditions given above hold:\\
\textbf{(1)} we must prove that
$$m^{-1}_{Y'U}\circ F(h\otimes 1_U)F(f)F(g\otimes 1_U)\circ m_{X',U}\in \Trc^{FU}_{FX',FY'}$$
By naturality of the map $m^{-1}$ with $h$, $g$ and identities we have:
\begin{equation}
\label{NATURALITY OF M}
m^{-1}_{Y'U}\circ F(h\otimes 1_U)=(F(h)\otimes 1_{FU})\circ m^{-1}_{YU}\,\,\,\mbox{and also}\,\,\,F(g\otimes 1_U)\circ m_{X',U}=m_{X,U}\circ (F(g)\otimes 1_{FU}).
\end{equation}
Consequently, we need to prove that
$$(F(h)\otimes 1_{FU})\circ m^{-1}_{YU}\circ F(f)\circ m_{X,U}\circ (F(g)\otimes 1_{FU})\in \Trc^{FU}_{FX',FY'}.$$
Notice that by hypothesis
$$m^{-1}_{YU}\circ F(f)\circ m_{X,U}\in \Trc^{FU}_{FX,FY}.$$
Then by the naturality axiom in the category $\cD$ we have that
$$(F(h)\otimes 1_{FU})\circ m^{-1}_{YU}\circ F(f)\circ m_{X,U}\circ (F(g)\otimes 1_{FU})\in \Trc^{FU}_{FX',FY'}$$
and also

$\Tr^{FU}_{FX',FY'}((F(h)\otimes 1_{FU})\circ m^{-1}_{YU}\circ F(f)\circ m_{X,U}\circ (F(g)\otimes 1_{FU}))=$

$=F(h)\circ \Tr^{FU}_{FX,FY}(m^{-1}_{YU}\circ F(f)\circ m_{X,U})\circ F(g).$

\textbf{(2)} Since by hypothesis there exists an arrow $p_1:X\rightarrow Y$ such that $$F(p_1)=\Tr^{FU}_{FX,FY}(m^{-1}_{YU}\circ F(f)\circ m_{X,U})$$
 then

$\Tr^{FU}_{FX',FY'}( m^{-1}_{Y'U}\circ F((h\otimes 1_U)\circ f\circ (g\otimes 1_U))\circ m_{X',U})=\,\,\,\,\,\mbox{(equation~(\ref{NATURALITY OF M}) above)}$

$=\Tr^{FU}_{FX',FY'}((F(h)\otimes 1_{FU})\circ m^{-1}_{YU}\circ F(f)\circ m_{X,U}\circ (F(g)\otimes 1_{FU}))=\,(\mbox{naturality axiom in}\,\, \cD\,\,\mbox{and hyp.})$

$=F(h)\circ F(p_1)\circ F(g)=F(h\circ p_1\circ g).$

This means that we can choose $p_2=h\circ p_1\circ g.$

Now we are able to compute the trace:
$$\widehat{\Tr}^{U}_{X',Y'}((h\otimes 1_U)f(g\otimes 1_U))=p_2=h\circ p_1\circ g=h\circ\widehat{\Tr}^{U}_{X,Y}(f)\circ g.$$

\textbf{Dinaturality}:\\
For any $f:X\otimes U\rightarrow Y\otimes U'$, $g:U'\rightarrow U$ where $f$ and $g$ are in $\cC$ we must prove that
\begin{center}
$(1_Y\otimes g)f \in \widehat{\Trc}^U_{X,Y}$ iff $f(1_X\otimes g)\in \widehat{\Trc}^{U'}_{X,Y}.$
\end{center}
We must check condition (1) and (2).

\textbf{(1)} By definition we have
\begin{equation}
\label{TRACE CLASS MEMBER}
(1_Y\otimes g)f \in\widehat{\Trc}^U_{X,Y}\,\,\,\,\mbox{implies}\,\,\,m^{-1}_{Y,U}\circ F((1_Y\otimes g)f)\circ m_{X,U}\in \Trc^{FU}_{FX,FY}.
\end{equation}
But in view of the naturality of $m$ it follows that $m^{-1}_{Y,U}\circ F(1_Y\otimes g)=(F(1)\otimes F(g))\circ m^{-1}_{Y,U'}.$ Then we can replace it in~(\ref{TRACE CLASS MEMBER}) obtaining:
$$m^{-1}_{Y,U}\circ F((1_Y\otimes g)f)\circ m_{X,U}=m^{-1}_{Y,U}\circ F(1_Y\otimes g)\circ Ff\circ m_{X,U}=$$
$$(1_{FY}\otimes F(g))\circ m^{-1}_{Y,U'}\circ Ff\circ m_{X,U}\in\Trc^{FU}_{FX,FY}.$$
It now follows by the dinaturality axiom of the category $\cD$ that this condition is equivalent to proving:
$$m^{-1}_{Y,U'}\circ F(f)\circ m_{X,U}\circ (1_{FX}\otimes F(g))\in \Trc^{FU'}_{FX,FY}$$
and again by naturality of $m$ we have that
$m_{X,U}\circ (1_{FX}\otimes F(g))=(F(1\otimes g))\circ m_{X,U'}$
and we replace it:
$$m^{-1}_{Y,U'}\circ F(f)\circ F((1\otimes g)\circ m_{X,U'}=m^{-1}_{Y,U'}\circ F(f(1\otimes g))\circ m_{X,U'}\in\Trc^{FU'}_{FX,FY}$$
which is condition (1) in the definition $\,\,\,\,f\circ(1_X\otimes g)\in\widehat{\Trc}^{U'}_{X,Y}$. In the same way we prove the converse.

\textbf{(2)} Also there is an arrow $p_1$ such that $F(p_1)=\Tr^{FU}_{FX,FY}(m^{-1}_{Y,U}\circ F((1_Y\otimes g)f)\circ m_{X,U})$ if and only if there is an arrow $F(p_2)=\Tr^{FU'}_{FX,FY}(m^{-1}_{Y,U'}\circ F(f(1_X\otimes g))\circ m_{X,U'})$. Since the value of the trace remains invariant under the dinaturality axiom and all the transformations made in part (1) then it is enough to take $p_1=p_2$.

\textbf{Vanishing I}:\\
Now we want to check that: $\widehat{\Trc}^{I}_{X,Y}=\cC(X\otimes I,Y\otimes I)$.
Given any $f:X\otimes I\rightarrow Y\otimes I$ we want to prove that $f\in\widehat{\Trc}^{I}_{X,Y}$ by verifying conditions (1) and (2).\\
\textbf{(1)} Let us consider $g=(1_{FY}\otimes m^{-1}_I)\circ m^{-1}_{Y,I}\circ F(f)\circ m_{X,I}\circ (1_{FX}\otimes m_I)$.
By the vanishing I axiom in the category $\cD$ we know that $g\in \widehat{\Trc}^{I}_{FX,FY}$. Then, since $(1_{FY}\otimes m^{-1}_I)\circ (1_{FY}\otimes m_I)\circ g=g\in\widehat{\Trc}^{I}_{FX,FY}$ we can apply the dinaturality axiom in $\cD$ to conclude that $(1_{FY}\otimes m_I)\circ g \circ (1_{FX}\otimes m^{-1}_I)\in\widehat{\Trc}^{FI}_{FX,FY}$ but we have that  $(1_{FY}\otimes m_I)\circ g \circ (1_{FX}\otimes m^{-1}_I)=m^{-1}_{Y,I}\circ F(f)\circ m_{X,I}$. So we proved that
$m^{-1}_{Y,I}\circ F(f)\circ m_{X,I}\in \widehat{\Trc}^{FI}_{FX,FY}$.

\textbf{(2)} Since $g\in \cD(FX\otimes I,FY\otimes I)$ we can say also, by the dinaturality axiom, that
$$\Tr^I_{FX,FY}(g)=\Tr^{FI}_{FX,FY}(m^{-1}_{Y,I}\circ F(f)\circ m_{X,I})$$
but on the other hand we know that
$$\Tr^I_{FX,FY}(g)=\rho_{FY}\circ g\circ \rho^{-1}_{FX}$$
by vanishing I in $\cD$ which implies that

$\Tr^I_{FX,FY}(g)=\rho_{FY}\circ (1_{FY}\otimes m^{-1}_I)\circ m^{-1}_{Y,I}\circ F(f)\circ m_{X,I}\circ (1_{FX}\otimes m_I)\circ \rho^{-1}_{FX}= \mbox{(since  $F$ is monoidal )} ~~
  F(\rho_Y)\circ F(f)\circ F(\rho^{-1}_X)=F(\rho_Y\circ f\circ \rho^{-1}_X) . $

Thus there exists a $p=\rho_Y\circ f\circ \rho^{-1}_X$ such that $F(p)=\Tr^{FI}_{FX,FY}(m^{-1}_{Y,I}\circ F(f)\circ m_{X,I})$. Also notice that we prove that $\widehat{\Tr}^{I}_{X,Y}(f)=p=\rho_Y\circ f\circ \rho^{-1}_X$, which is the equation of the trace value in the category $\cC$.

\textbf{Vanishing II}:\\
Let $g:X\otimes U\otimes V\rightarrow Y\otimes U\otimes V$ be an arrow in the category $\cC$. By hypothesis, we are given $g \in\widehat{\Trc}^{V}_{X\otimes U,Y\otimes U}$ (general hypothesis) and we want to prove the following equivalence:
\begin{center}
$g \in\widehat{\Trc}^{U\otimes V}_{X,Y}$ iff $\widehat{\Tr}^V_{X\otimes U,Y\otimes U}(g)\in \widehat{\Trc}^{U}_{X,Y}.$
\end{center}
According to the general hypothesis there is a map:
$$F(X\otimes U)\otimes FV\stackrel{m_{X\otimes U,V}}\longrightarrow F(X\otimes U\otimes V)\stackrel{Fg}\longrightarrow F(Y\otimes U\otimes V)\stackrel{m^{-1}_{Y\otimes U,V}}\longrightarrow F(Y\otimes U)\otimes FV\in\widehat{\Trc}^{FV}_{F(X\otimes U),F(Y\otimes U)}$$
and also there exists $p_1:X\otimes U\rightarrow Y\otimes U$ such that
$$F(p_1)=\Tr^{FV}_{F(X\otimes U),F(Y\otimes U)}(m^{-1}_{Y\otimes U,V}\circ F(g)\circ m_{X\otimes U,V})\,\,\,\mbox{i.e., by definition}\,\,\,p_1=\widehat{\Tr}^{V}_{X\otimes U,Y\otimes U}(g).$$

$(\Rightarrow)$ We have a conditional hypothesis $g \in\widehat{\Trc}^{U\otimes V}_{X,Y}$ which asserts that the map:
$$FX\otimes F(U\otimes V)\stackrel{m_{X, U\otimes V}}\longrightarrow F(X\otimes U\otimes V)\stackrel{Fg}\longrightarrow F(Y\otimes U\otimes V)\stackrel{m^{-1}_{Y,U\otimes V}}\longrightarrow FY\otimes F(U\otimes V)\in\widehat{\Trc}^{F(U\otimes V)}_{FX,FY}$$
and also that there exists an $p_2:X\rightarrow Y$ such that
$$F(p_2)=\Tr^{F(U\otimes V)}_{FX,FY}(m^{-1}_{Y,U\otimes V}\circ F(g)\circ m_{X,U \otimes V})\,\,\,\mbox{i.e., by definition}\,\,\,p_2=\widehat{\Tr}^{U\otimes V}_{X,Y}(g).$$

Recalling that $p_1=\widehat{\Tr}^{V}_{X\otimes U,Y\otimes U}(g)$, we want to prove that $p_1\in \widehat{\Trc}^{U}_{X,Y}$. For that purpose, we shall prove the two conditions that characterize the trace class definition which are the following:

(1) the map
$$FX\otimes FU\stackrel{m_{X, U}}\longrightarrow F(X\otimes U)\stackrel{F(p_1)}\longrightarrow F(Y\otimes U)\stackrel{m^{-1}_{Y,U}}\longrightarrow FY\otimes FU)\in\Trc^{FU}_{FX,FY}$$

(2) there exists an $p_3:X\rightarrow Y$ such that
$$F(p_3)=\Tr^{FU}_{FX,FY}(m^{-1}_{Y,U}\circ F(p_1)\circ m_{X,U})\,\,\,\mbox{i.e., by definition}\,\,\,$$
$$p_3=\widehat{\Tr}^{U}_{X,Y}(p_1)=\widehat{\Tr}^{U}_{X,Y}(\widehat{\Tr}^{V}_{X\otimes U,Y\otimes U}(g)).$$

\textbf{(1)} To prove condition (1) we notice that since by definition
$$F(p_1)=\Tr^{FV}_{F(X\otimes U),F(Y\otimes U)}(m^{-1}_{Y\otimes U,V}\circ F(g)\circ m_{X\otimes U,V})$$
then we must prove that
$$FX\otimes FU\stackrel{m_{X, U}}\longrightarrow F(X\otimes U)\stackrel{\Tr^{FV}_{F(X\otimes U),F(Y\otimes U)}(m^{-1}_{}\circ F(g)\circ m_{})}\longrightarrow F(Y\otimes U)\stackrel{m^{-1}_{Y,U}}\longrightarrow FY\otimes FU)\in\widehat{\Trc}^{FU}_{FX,FY}.$$
But since
$$m^{-1}_{Y\otimes U,V}\circ F(g)\circ{m_{X\otimes U,V}}\in\Trc^{FV}_{F(X\otimes U),F(Y\otimes U)}$$
this condition allows us to apply the naturality axiom in the category $\cD$:\\
$$FX\otimes FU\otimes FV\stackrel{m_{X, U}\otimes 1_{FV}}\longrightarrow F(X\otimes U)\otimes FV\stackrel{m^{-1}_{Y\otimes U,V}\circ F(g)\circ m_{X\otimes U,V}}\longrightarrow F(Y\otimes U)FV\stackrel{m^{-1}_{Y,U}\otimes FV}\longrightarrow FY\otimes FU\otimes FV$$

$\in\widehat{\Trc}^{FV}_{FX\otimes FU,FY\otimes FU}.$

\noindent
And also by the same axiom we have that:\\

$\Tr^{FV}_{FX\otimes FU,FY\otimes FU}((m^{-1}_{Y,U}\otimes 1_{FV})\circ m^{-1}_{Y\otimes U,V}\circ F(g)\circ m_{X\otimes U,V}\circ (m_{X, U}\otimes 1_{FV}))=$\\

$=m^{-1}\circ \Tr^{FV}_{F(X\otimes U),F(Y\otimes U)}(m^{-1}_{Y\otimes U,V}\circ F(g)\circ m_{X\otimes U,V})\circ m.$

\noindent
Hence, this is equivalent to proving that:
$$\Tr^{FV}_{FX\otimes FU,FY\otimes FU}((m^{-1}_{Y,U}\otimes 1_{FV})\circ m^{-1}_{Y\otimes U,V}\circ F(g)\circ m_{X\otimes U,V}\circ (m_{X, U}\otimes 1_{FV}))\in \Tr^{FU}_{FX,FY}.$$
Consequently, by vanishing II in the category $\cD$, it would be enough that the map
$$\lambda=(m^{-1}_{Y,U}\otimes 1_{FV})\circ m^{-1}_{Y\otimes U,V}\circ F(g)\circ m_{X\otimes U,V}\circ (m_{X, U}\otimes 1_{FV})\in \Trc^{FU\otimes FV}_{FX,FY}$$
since we know that $\lambda\in\Trc^{FV}_{FX\otimes FU,FY\otimes FU}$.
But by coherence of monoidal functors we have:\\

$(m^{-1}_{Y,U}\otimes 1_{FV})\circ m^{-1}_{Y\otimes U,V}\circ F(g)\circ m_{X\otimes U,V}\circ (m_{X, U}\otimes 1_{FV})=$\\

$=(1_{FX}\otimes m^{-1}_{U,V})\circ m^{-1}_{Y,U\otimes V}\circ F(g)\circ m_{X,U\otimes V}\circ (1_{FX}\otimes m_{U, V}).$

Therefore, by the dinaturality axiom in the category $\cD$:
$$(1_{FX}\otimes m^{-1}_{U,V})\circ m^{-1}_{Y,U\otimes V}\circ F(g)\circ m_{X,U\otimes V}\circ (1_{FX}\otimes m_{U, V})\in \Trc^{FU\otimes FV}_{FX,FY}$$
if and only if
$$(m^{-1}_{Y,U\otimes V}\circ F(g)\circ m_{X,U\otimes V})\circ (1_{FX}\otimes m_{U, V})\circ (1_{FX}\otimes m^{-1}_{U,V})\in \Trc^{F(U\otimes V)}_{FX,FY} $$
which is valid since this is the conditional hypothesis.\\

\textbf{(2)} We shall prove that there exists an arrow $p_3:X\rightarrow Y$ in $\cC$ such that $$F(p_3)=\Tr^{FU}_{FX,FY}(m^{-1}_{Y,U}\circ F(p_1)\circ m_{X,U}).$$
For that purpose, take $p_3=p_2$. Hence by the conditional hypothesis if $g\in \widehat{\Trc}^{U\otimes V}_{X,Y}$ holds then there is $p_2$ with
$$F(p_2)=\Tr^{F(U\otimes V)}_{FX,FY}(m^{-1}_{Y,U\otimes V}\circ F(g)\circ m_{X,U \otimes V}).$$
Therefore, this is equal to,\\

$\Tr^{F(U\otimes V)}_{FX,FY}((m^{-1}_{Y,U\otimes V}\circ F(g)\circ m_{X,U\otimes V})\circ (1_{FX}\otimes m_{U, V})\circ (1_{FX}\otimes m^{-1}_{U,V}))=$\\

$\Tr^{FU\otimes FV}_{FX,FY}((1_{FX}\otimes m^{-1}_{U,V})\circ m^{-1}_{Y,U\otimes V}\circ F(g)\circ m_{X,U\otimes V}\circ (1_{FX}\otimes m_{U, V}))= (\mbox{dinaturality})$\\

$\Tr^{FU}_{FX,FY}(\Tr^{FV}_{FX\otimes FU,FY\otimes FU}((1_{FX}\otimes m^{-1}_{U,V})\circ m^{-1}_{Y,U\otimes V}\circ F(g)\circ m_{X,U\otimes V}\circ (1_{FX}\otimes m_{U, V})))=(\mbox{vanishing II})$\\

$\Tr^{FU}_{FX,FY}(\Tr^{FV}_{FX\otimes FU,FY\otimes FU}((m^{-1}_{Y,U}\otimes 1_{FV})\circ m^{-1}_{Y\otimes U,V}\circ F(g)\circ m_{X\otimes U,V}\circ (m_{X, U}\otimes 1_{FV})))= (\mbox{coherence})$\\

$\Tr^{FU}_{FX,FY}(m^{-1}_{Y,U}\circ \Tr^{FV}_{F(X\otimes U),F(Y\otimes U)}(m^{-1}_{Y\otimes U,V}\circ F(g)\circ m_{X\otimes U,V}) \circ m_{X,U})= \,\,\,\mbox{(naturality axiom)}$\\

$\Tr^{FU}_{FX,FY}(m^{-1}_{Y,U}\circ F(p_1)\circ m_{X,U})=\,\,\, \mbox{(definition of general hypothesis)}.$\\

\noindent
So we have proved that:
$$F(p_2)=\Tr^{FU}_{FX,FY}(m^{-1}_{Y,U}\circ F(p_1)\circ m_{X,U})$$
which means that
$$\widehat{\Tr}^{U}_{X,Y}(\widehat{\Tr}^{V}_{X\otimes U,Y\otimes U}(g))=\widehat{\Tr}^{U}_{X,Y}(p_1)=p_3=p_2=\widehat{\Tr}^{U\otimes V}_{X,Y}(g).$$

\noindent
$(\Leftarrow)$ Similarly, we prove the converse. The proof is just a matter of using the converse hypothesis of vanishing II in the category $\cD$.

\textbf{Superposing:}\\
Suppose $f\in \widehat{\Trc}^{U}_{X,Y}$ and $g:W\rightarrow Z$ with $g\in \cC$ we want to prove that $g\otimes
f\in \widehat{\Trc}^{U}_{W\otimes X,Z\otimes Y}$ by checking conditions (1) and (2). Also we want to show that
$$\widehat{\Tr}^U_{W\otimes X,Z\otimes Y}(g\otimes f)=g\otimes \widehat{\Tr}^U_{X,Y}(f).$$
\textbf{(1)} By hypothesis we know that
 $$FX\otimes FU\stackrel{m_{X, U}}\longrightarrow F(X\otimes U)\stackrel{F(f)}\longrightarrow F(Y\otimes U)\stackrel{m^{-1}_{Y,U}}\longrightarrow FY\otimes FU)\in\widehat{\Trc}^{FU}_{FX,FY}$$
and also there exists an arrow $p_1:X\rightarrow Y$ such that
$$F(p_1)=\Tr^{FU}_{FX,FY}(m^{-1}_{Y,U}\circ F(f)\circ m_{X,U}).$$
Then by the superposing axiom in the category $\cD$ it follows that
$$F(g)\otimes (m^{-1}_{Y,U}\circ F(f)\circ m_{X,U})\in \Trc^{FU}_{FW\otimes FX,FW\otimes FY}$$
and the trace value turns out to be
$$\Tr^{FU}_{FW\otimes FX,FZ\otimes FY}(F(g)\otimes (m^{-1}_{Y,U}\circ F(f)\circ m_{X,U}))=F(g)\otimes \Tr^{FU}_{FX,FY}(m^{-1}_{Y,U}\circ F(f)\circ m_{X,U}).$$
But by functoriality of the tensor we obtain
$$F(g)\otimes (m^{-1}_{Y,U}\circ F(f)\circ m_{X,U})=(1_{FZ}\otimes m^{-1}_{Y,U})\circ (F(g)\otimes F(f))\circ (1_{FW}\otimes m_{X,U})=\beta $$
(To simplify notation, we name this equation $\beta$).\\
We can apply the naturality axiom in the category $\cD$ and we obtain:
$$(m_{Z,Y}\otimes 1_{FU})\circ \beta\circ(m^{-1}_{W,X}\otimes 1_{FU})\in \Trc^{FU}_{F(W\otimes X),F(Z\otimes Y)}$$
and
$$\Tr^{FU}_{F(W\otimes X),F(Z\otimes Y)}((m_{Z,Y}\otimes 1_{FU})\circ \beta\circ(m^{-1}_{W,X}\otimes 1_{FU}))=m_{Z,Y}\circ \Tr^{FU}_{FW\otimes FX,FZ\otimes FY}(\beta)\circ m^{-1}_{W,X}$$
but by naturality and monoidal functor axioms we have that
$$(m_{Z,Y}\otimes 1_{FU})\circ \beta\circ(m^{-1}_{W,X}\otimes 1_{FU})=m_{Z\otimes Y,U}\circ F(g\otimes f)\circ m_{W\otimes X,U}.$$
Therefore, we proved that $m^{-1}\circ F(g\otimes f)\circ m\in \Trc^{FU}_{F(W\otimes X),F(Z\otimes Y}.$

\textbf{(2)}  Let us consider $\beta=(1_{FZ}\otimes m^{-1}_{Y, U})\circ (Fg\otimes Ff)\circ(1_{FW}\otimes m_{X,U})$.
It follows that

$\Tr^{FU}_{F(W\otimes X),F(Z\otimes Y)}((m_{Z,Y}\otimes 1_{FU})\circ\beta\circ (m^{-1}_{W,X}\otimes 1_{FU}))=\,\,\,\mbox{(naturality axiom)}$\\

$=m_{Z,Y}\circ \Tr^{FU}_{FW\otimes FX,FZ\otimes FY}(\beta)\circ m^{-1}_{W,X}=\mbox{(functoriality of the tensor)}$\\

$=m_{Z,Y}\circ \Tr^{FU}_{FW\otimes FX,FZ\otimes FY}(F(g)\otimes(m^{-1}_{Y,U}\circ Fg\circ m_{X,U}))\circ m^{-1}_{W,X}=\mbox{(superposing)}$\\

$=m_{Z,Y}\circ(F(g)\otimes \Tr^{FU}_{FW\otimes FX,FZ\otimes FY}(m^{-1}_{Y,U}\circ Fg\circ m_{X,U}))\circ m^{-1}_{W,X}=\mbox{(by hypothesis)}$\\

$=m_{Z,Y}\circ(F(g)\otimes F(p_1))\circ m^{-1}_{W,X}=\mbox{(by naturality of $m$)}$\\

$=F(g\otimes p_1)$.

Thus, we proved that there exists an arrow $p_2=g\otimes p_1$ such that
$$F(p_2)=\Tr^{FU}_{F(W\otimes X),F(Z\otimes Y)}((m_{Z,Y}\otimes 1_{FU})\circ\beta\circ (m^{-1}_{W,X}\otimes 1_{FU})).$$

On the other hand, we have by naturality and the fact that $F$ is a monoidal functor:
$$(m_{Z,Y}\otimes 1_{FU})\circ\beta\circ (m^{-1}_{W,X}\otimes 1_{FU})=m^{-1}_{Z\otimes Y,U}\circ F(g\otimes f)\circ m_{W\otimes X,U}$$
which means, according to our definition, that
$$\widehat{\Tr}^{U}_{W\otimes X,Z\otimes Y}(g\otimes f)=p_2=g\otimes p_1=g\otimes \widehat{\Tr}^{U}_{X,Y}(f).$$

\textbf{Yanking:}\\
Let us consider  $\sigma:U\otimes U\rightarrow U\otimes U$; we want to prove that $\sigma_{U,U}\in \widehat{\Trc}^{U}_{U,U}$ and $\widehat{\Tr}^U_{U,U}(\sigma_{U,U})=1_U$. To show that $\sigma_{U,U}\in \widehat{\Trc}^{U}_{U,U}$ we recall from the trace class definition that we must check two conditions:

\textbf{(1)} First, we notice that since $F$ is a symmetric monoidal  functor and by the yanking axiom in the category $\cD$:
$$m^{-1}_{U,U}\circ F(\sigma_{U,U})\circ m_{U,U}=\sigma_{FU,FU}\in \Trc^{U}_{U,U}.$$
From which it follows that $\Tr^{FU}_{FU,FU}(\sigma_{FU,FU})=1_{FU}=F(1_U)$.

\textbf{(2)} Therefore there exists an arrow $p=1_U$ such that
$$F(1_U)=\Tr^{FU}_{FU,FU}(m^{-1}_{U,U}\circ F(\sigma_{U,U})\circ m_{U,U}).$$
Hence, we are saying that $\widehat{\Tr}^{U}_{U,U}(\sigma_{U,U})=p=1_{U}$.
\end{proof}

\section{Another partial trace on completely
positive maps with $\oplus$}
\begin{definition}
Consider the forgetful functor $F:(\textbf{CPM},\oplus)\rightarrow(\textbf{Vect}_{fn},\oplus)$, where $(\textbf{Vect}_{fn},\oplus,\textbf{0},\Trc)$ is partially traced by Definition~\ref{PARTIAL TRACE IN FINITE VECTOR SPACE}, i.e., $\textbf{CPM}$  is a monoidal subcategory of $\textbf{Vect}_{fn}$. We define a partial trace $\widehat{\Tr}$ with trace class given by $\widehat{\Trc}$ on $\textbf{CPM}$ by the method of Section~\ref{TRACE IN A MONOIDAL SUBCATEGORY}.
\end{definition}

\begin{remark}
Comparing this with the partial trace (on the same category) defined in Section~\ref{4.2.5}, we note that if $f$ and $(I-f_{22})^{-1}$ are completely positive then
$$f_{11}+f_{12}(I-f_{22})^{-1}f_{21}$$
is a completely positive map. This implies that $\cpmTrc^{U}_{X,Y}$ as in Definition~\ref{PARTIAL TRACE IN CPM+} satisfies: $\cpmTrc^{U}_{X,Y}\subseteq\widehat{\Trc}^{U}_{X,Y}$. However, consider the $\textbf{CPM}$-map $f:U\oplus U\rightarrow U\oplus U$ given by the following matrix:
$$\left(\begin{array}{llll}
 I &  0\\
 0 & 2I\\
\end{array}\right)$$
We have $f_{11}=I$, $f_{21}=f_{12}=0$ and $f_{22}=2I$. Then $I-f_{22}=I-2I=(-1)I$ is an invertible map with inverse $(-1)I$ but is not a positive map. On the other hand, $f_{11}+f_{12}(I-f_{22})^{-1}f_{21}=I+0((-1)I)0=I$ is a $\textbf{CPM}$-map, i.e., $f\in\widehat{\Trc}^U_{U,U}$ but $f\notin \cpmTrc^U_{U,U}$.
\end{remark}

\section{Partial trace on superoperators with $\oplus$ and $\otimes$}

As an application of the construction of Section~\ref{TRACE IN A MONOIDAL SUBCATEGORY}, we now focus on the category \textbf{Q} which is not a compact closed category. We discuss examples
of partial traces in connection with its two monoidal structures.

\begin{example}
$(\textbf{\cQ},\oplus)$ has a total trace operator $\Tr^{u}_{x,y}:\textbf{Q}(x\oplus u,y\oplus u)\rightarrow \textbf{Q}(x,y)$ defined by
$\Tr^u_{x,y}(f)=f_{11}+\sum_{i=0}^{\infty}f_{12}f_{22}^{i}f_{21}\,$, see \cite{Sel 2004} for details.
\end{example}
\begin{example}
 By Proposition~\ref{PROPOSITION SUBCATEGORY}, $(\textbf{\cQ},\oplus)$ has a partial trace
 $\Tr^{u}_{x,y}:\textbf{Q}(x\oplus u,y\oplus u)\rightharpoonup\textbf{Q}(x,y),$ given by $\Tr^u_{x,y}(f)\funnels f_{11}+f_{12}(I-f_{22})^{-1}f_{21}$.
 \end{example}
 \begin{example}
 Another partial trace on $(\textbf{\cQ},\oplus)$ is given by considering the forgetful functor from $\textbf{Q}$ to the category of vector spaces $(\textbf{Vect},\oplus)$ with the kernel-image partial trace of Definition~\ref{DEFINITION RELAX CONDITION} given in Section~\ref{PARTIAL TRACE KERN-IMAGE CONDITION}. Notice that the identity is a superoperator satisfying Definition~\ref{DEFINITION RELAX CONDITION} which implies that these two partial traces still remain different on $ \textbf{Q}$.
 \end{example}
 \begin{example}
We can consider the category $(\textbf{Q}_s,\otimes)$ of simple superoperators as a subcategory of the compact closed category $(\textbf{CPM}_s,\otimes)$, see Definition~\ref{Q SIMPLE SUPEROPERATORS}. It has a partial trace $\Tr$ given by Proposition~\ref{PROPOSITION SUBCATEGORY} where $\Tr^{U}_{X,Y}: \textbf{Q}_s(X\otimes U, Y\otimes U) \rightharpoonup \textbf{Q}_s(X,Y)$ is the canonical trace on $\textbf{CPM}_s$. Since linear maps $f$ in the category of finite dimensional vector spaces are continuous functions we can prove that for every completely positive map there exists a $0<\lambda\leq 1$ such that $\lambda f$ is a superoperator. Then, for every unit map $\eta_{U}:I\rightarrow U^{*}\otimes U$ in \textbf{CPM}
 there exists a $\lambda_{U}$ such that $\lambda_U.\eta_U$ is a superoperator. Therefore, if $\lambda_{U}^{-1}.f$ is a superoperator then $f\in\Trc^U_{X,Y}$.
\end{example}

\chapter[A representation theorem]
         {A representation theorem for partially traced categories}

The goal of this chapter is to prove a strong converse to Proposition~\ref{TRACE IN A MONOIDAL SUBCATEGORY}, i.e.: every partially traced category arises as a monoidal subcategory of a totally traced category. More precisely, we show
 that every partially traced category can be faithfully embedded in a
 compact closed category in such a way that the trace is preserved.

 Our construction uses a partial version of the $``Int"$ construction of
 Joyal, Street, and Verity~\cite{JSV96}. When we try to apply the $Int$
 construction to a partially traced category $\cC$, we find that the
 composition operation in ${\rm Int}(C)$ is a well-defined operation only if
 the trace is total. We therefore consider a notion of ``categories"
 with partially defined composition, namely, Freyd's paracategories~\cite{HerMat03}. Specifically, we introduce the notion of a strict symmetric
 compact closed paracategory.

 We first show that every partially traced category can be fully and
 faithfully embedded in a compact closed paracategory, by an analogue
 of the $Int$ construction.  We then show that every compact closed
 paracategory can be embedded (faithfully, but not necessarily fully)
 in a compact closed (total) category, using a construction similar to
 Freyd's. Finally, every compact closed category is (totally) traced,
 yielding the desired result.

\section{Paracategories}

The aim of this section is to recall Freyd's notion of paracategory.
 A reference on this subject is~\cite{HerMat03}. Informally, a paracategory is a
 category with partially defined composition.

\begin{definition}
\rm A \textit{(directed) graph} $\mathcal{C}$ consists of:\\
$\bullet$  a class of elements called objects $obj (\mathcal{C})$\\
$\bullet$  for every pair of objects $A,B$ a set $\mathcal{C}(A,B)$ called arrows from $A$ to $B$. Let $Arrow(\mathcal{C})$ be the class of all the arrows in $\mathcal{C}$.
\end{definition}
\begin{definition}
\rm Let $\mathcal{C}$ be a graph. We define $\mathcal{P(C)}$, the {\em path category} of $\mathcal{C}$, by $obj(\mathcal{P(C)})=obj(\mathcal{C})$ and arrows from $A_0$ to $A_n$ are finite sequences $(A_0,f_0,A_1,f_1,\ldots,A_n)$ of alternating objects and arrows of the graph $\mathcal{C}$, where $n\geq 0$. We say that $n$ is the {\em length} of
the path. Two arrows are equal when the sequences coincide. Composition is defined by concatenation and the identity arrow at $A$ is the path of zero length $(A)$ with an object $A$. We write $\epsilon_A=(A)$ for the identity arrow.
\end{definition}

\texttt{Notation:} For the sake of simplification, we often write $\vec{f}=f_1,f_2,\ldots,f_n$ for a path and the symbol $``;"$ or $``,"$ for concatenation.

Recall the definition of Kleene equality
``$\funnels$'' and directed Kleene equality ``$\funnel$'' from
Definition~\ref{Kleene equality}.

We write $\phi (f)\downarrow$ to say a partial function $\phi$ is defined on input $f$.
\begin{definition}\label{DEF PARACATEGORY}
\rm A \textit{paracategory} $(\mathcal{C},[-])$ consists of a directed graph $\mathcal{C}$ and a partial operation $[-]:Arrow(\mathcal{P(C)})\rightharpoonup Arrow(\cC)$ called {\em composition}, which satisfies the following axioms:
\begin{itemize}
\item[(a)] for all $A$, $[\epsilon_A]\downarrow$, i.e.,
$[-]$ is a total operation on empty paths
\item[(b)] for paths of length one, $[f]\downarrow$ and $[f]=f$
\item[(c)] for all paths $\vec{r}:A\to B$, $\vec{f}:B\to C$, and
$\vec{s}:C\to D$, if $[\vec{f}]\downarrow$ then
$$[\vec{r},[\vec{f}],\vec{s}]\funnels[\vec{r},\vec{f},\vec{s}].$$
\end{itemize}
\end{definition}
We introduce the following notation:
\begin{itemize}
\item[-] $1_A=[(A)]=[\epsilon_A]$ for every object $A$ in $\mathcal{C}$.
\item[-] for a path $\vec{f}=f_1,f_2,\ldots,f_n$ and an operation $\otimes$, defined on $\cC$ (see Definition~\ref{DEFINITION SSMPC}), we extend it to the category of paths using the following notation:
$$ 1\otimes_{p}\vec{f}=1\otimes f_1,1\otimes f_2,\ldots,1\otimes f_n$$
and in the same way: $\vec{f}\otimes_{p} 1$. We drop the symbol $p$ when it is clear from the context.
\end{itemize}

\begin{definition}
\rm Let $(\mathcal{C},[-])$ and $(\mathcal{D},[-]')$ be two paracategories. A \textit{functor} between paracategories is a graph morphism $F:Obj(\mathcal{C})\rightarrow Obj(\mathcal{D})$, $F:\mathcal{C}(A,B)\rightarrow\mathcal{D}(FA,FB)$ such that when $[\vec{p}]\downarrow$ then $F[\,\vec{p}\,]=[\,\vec{Fp}\,]'$. Let ${\bf PCat}$ be the category of (small) paracategories and functors.\\
We say that such a functor is \textit{faithful} if it is faithful as a morphism of graphs.\\
\end{definition}

\begin{remark} Every category $\cC$ can be regarded as a paracategory with
$[f_1,\ldots,f_n] = f_n \circ \ldots \circ f_1$. In this case,
composition is a totally defined operation. This yields a forgetful functor ${\bf Cat}\rightarrow {\bf PCat}$.
\end{remark}

\section{Symmetric monoidal paracategories}

\begin{definition}
\label{DEFINITION SSMPC}
\rm A \textit{strict symmetric monoidal paracategory} $(\mathcal{C},[-],\otimes,I,\sigma)$, also called an \textit{ssmpc}, consists of:
\begin{itemize}
\item a paracategory $(\mathcal{C},[-])$
\item a total operation $\otimes:\mathcal{C}\times\mathcal{C}\rightarrow\mathcal{C}$ which satisfies:\\
$(A\otimes B)\otimes C=A\otimes (B\otimes C)$ on objects, $(f\otimes g)\otimes h=f\otimes (g\otimes h)$ on arrows (associative); there is an object $I$ such that $A\otimes I=I\otimes A=A$ and $f\otimes 1_I=1_I\otimes f=f$ for every object $A$ and arrow $f$ (unit). Subject to the following conditions:
\begin{itemize}
\item[(a)] $1_A\otimes 1_B=1_{A\otimes B}$.\\
\item[(b)] $[f,f']\otimes [g,g']\funnel [f\otimes g,f'\otimes g']$ where $f,g,f',g'$ are arrows of $\mathcal{C}$ and  $\funnel$ denotes Kleene directed equality.\\
\item[(c)] $1\otimes [\, \vec{p}\, ]\funnel [\, 1\otimes \vec{p}\, ]$ and $[ \,\vec{p}\, ]\otimes 1 \funnel [\, \vec{p}\otimes 1 \,]$
\end{itemize}
\item for all objects $A$ and $B$ there is an arrow $\sigma_{A,B}:A\otimes B\rightarrow B\otimes A$ such that:
\begin{itemize}
\item[-] for every $f:B\otimes A\rightarrow X$, $g:Y\rightarrow A\otimes B$, $[\sigma_{A,B},f]\downarrow$ and $[g,\sigma_{A,B}]\downarrow$

\item[-] for every $f:A\rightarrow A'$ and $g:B\rightarrow B'$: $[f\otimes 1_B,\sigma]=[\sigma,1_B\otimes f]$ and $[1_A\otimes g,\sigma]=[\sigma,g\otimes 1_A]$
\item[-] for every $A$ and $B$: $[\sigma_{A,B},\sigma_{B,A}]=1_{A\otimes B}$
\item[-] for every $A$, $B$, and $C$: $[\sigma_{A,B}\otimes 1_C, \sigma_{B,A\otimes C}]=1_A\otimes \sigma_{B,C}.$
\end{itemize}
\end{itemize}
\end{definition}

\begin{remark}
\rm Conditions (b) and (c) are equivalent to the condition $[f_1,\dots,f_n]\otimes [g_1\dots,g_n]\funnel [f_1\otimes g_1,\dots,f_n\otimes g_n]$ for all natural numbers $n$.
\end{remark}

\begin{proposition}\label{PARACAT2}
Let $(\mathcal{C},[-],\otimes,I,\sigma)$ be a ssmpc. Then for paths $p$, $q$ of length one we have that $[p\otimes 1_C,1_B\otimes q]\downarrow$, $[1_A\otimes q,p\otimes 1_D]\downarrow$ and are equal to $p\otimes q$. Moreover, for paths $\vec{p}$ and $\vec{q}$
$[\,1_A\otimes \vec{q},\vec{p}\otimes 1_D\,]\funnels[\,\vec{p}\otimes 1_C,1_B\otimes \vec{q}\,]$.
\end{proposition}
\begin{proof}
Let us first prove the result for paths of length 1, say $p:A\rightarrow B$, $q:C\rightarrow D$. Observe that $[p,1_B]=[p,[(B)]]=[p,(B)]=[p]=p$ since the last equation is defined and by the axioms. In the same way $p=[1_A,p]$, $[1_C,q]=q=[q,1_D]$. Then $p\otimes q=[p,1_B]\otimes [1_C,q]\funnel [p\otimes 1_C,1_B\otimes q]$ and $p\otimes q=[1_A,p]\otimes [q,1_D]\funnel[1_A\otimes q,p\otimes 1_D]$, by condition (b) of Definition~\ref{DEFINITION SSMPC}, which implies that $[1_A\otimes q,p\otimes 1_D]\downarrow$, $[p\otimes 1_C,1_B\otimes q]\downarrow$ and $[1_A\otimes q,p\otimes 1_D]=[p\otimes 1_C,1_B\otimes q]= p \otimes q$.\\
Now since we have already proved that $[1_A\otimes q,p\otimes 1_D]\downarrow$, $[p\otimes 1_C,1_B\otimes q]\downarrow$ and that they are equal we can use the axioms of paracategories and extend this to $[1_A\otimes \vec{q},\vec{p}\otimes 1_D]\funnels[\vec{p}\otimes 1_C,1_B\otimes \vec{q}]$ by iterating this procedure in the following way:\\
\begin{align*}
[1_A\otimes \vec{q},\vec{p}\otimes 1_D] &= [1_A\otimes q_1,\dots,1_A\otimes q_n,p_1\otimes 1_D,\dots,p_m\otimes 1_D]\\
                    &\funnels [1_A\otimes q_1,\dots,[1_A\otimes q_n,p_1\otimes 1_D],\dots,p_m\otimes 1_D] \\
                    &\funnels [1_A\otimes q_1,\dots,[p_1\otimes 1_C,1_B\otimes q_n],\dots,p_m\otimes 1_D]\\
                    &\funnels [1_A\otimes q_1,\dots,p_1\otimes 1_C,1_B\otimes q_n,\dots,p_m\otimes 1_D]\\
                    &\funnels\dots \mbox{we move}\,\,\, p_1\otimes 1_C\,\,\, \mbox{to the first position}\\
                    &\funnels [p_1\otimes 1_C,1_B\otimes q_1,\dots ,1_A\otimes q_n,p_2\otimes 1_C,\dots,p_m\otimes 1_D]\\
                    &\funnels\dots \mbox{we iterate this procedure $m-1$ times}\\
                    &\funnels[\vec{p}\otimes 1_C,1_B\otimes \vec{q}].
\end{align*}

\end{proof}
\begin{definition}
\rm
Let $(\mathcal{C},[-],\otimes,I,\sigma)$  and $(\mathcal{D},[-]',\otimes',I',\sigma')$ be two ssmpcs. A functor between them is \textit{strict monoidal} when $F(A)\otimes' F(B)=F(A\otimes B)$, $F(I)=I'$ on objects, $F(f)\otimes' F(g)=F(f\otimes g)$ and $F(\sigma)=\sigma'$ on arrows.
\end{definition}

\section{The completion of symmetric monoidal paracategories}

From now on $\cC$ denotes a ssmpc. We wish to prove the following theorem:
\begin{theorem}\label{PARACAT7}
Every strict symmetric monoidal paracategory can be faithfully embedded in a strict symmetric monoidal category.
\end{theorem}

\begin{definition}\label{PARACAT3}
\rm A {\em congruence relation} $\mathcal{S}$ on $\mathcal{P(C)}$ is given as
 follows: for every pair of objects $A$, $B$, an equivalence relation
   $\sim_{\cal S}^{A,B}$ on the hom-set $\mathcal{P(C)}(A,B)$, satisfying the following
   axioms. We usually omit the superscripts when they are clear from
   the context.

\begin{itemize}
\item[(1)] If $\vec p \sim_{\cal S} \vec p\,'$ and $\vec q \sim_{\cal S} \vec q\,'$, then
    $\vec p;\vec q \sim_{\cal S} \vec p\,';\vec q\,'$.
\item[(2)] Whenever $[\vec{p}\,]\downarrow$, then $\vec{p} \sim_{\cal S} [\vec{p}\,]$.
\item[(3)] If $\vec p \sim_{\cal S} \vec q$, then $\vec p\otimes 1 \sim_{\cal S} \vec
     q\otimes 1$ and $1\otimes \vec p \sim_{\cal S} 1\otimes \vec q$.

\end{itemize}
\end {definition}
\begin{remark}
Technically Definition~\ref{PARACAT3} can be regarded as a ``congruence subcategory" on $\mathcal{P(C)}$, i.e., $\mathcal{S}$ is a subcategory of $\mathcal{P(C)}\times\mathcal{P(C)}$ satisfying axioms $(2)$ and $(3)$.
\end{remark}
\begin{definition}
 We define a particular congruence relation $\hat{\mathcal{S}}$ as follows: $\vec{p} \sim_{\hat{\mathcal{S}}} \vec{q}$ if and only if $\forall\vec{r},\vec{s},\forall A,B\in Obj(\mathcal{C})\,\, [\,\vec{r},1_A\otimes \vec{p}\otimes 1_B,\vec{s}\,]\funnels[\,\vec{r},1_A\otimes \vec{q}\otimes 1_B,\vec{s}\,]$.
\end{definition}

\begin{remark}\label{PARACAT1}
\rm
It should be observed that $\vec{p}\sim_{\hat{\mathcal{S}}}\vec{q}$ implies $[\,\vec{p}\,]\funnels[\,\vec{q}\,]$ by letting $\vec{r}, \vec{s}$ be empty lists and $A = B = I$.
\end{remark}
Let us check that $\hat{\mathcal{S}}$ is a congruence relation.

\begin{lemma}
$\hat{\mathcal{S}}$ is a congruence relation.
\end{lemma}
\begin{proof}
We need to show axioms $(1)$, $(2)$,
and $(3)$. To show $(1)$, assume $\vec{p}\sim_{\hat{\mathcal{S}}} \vec{q}$ and $\vec{u}\sim_{\hat{\mathcal{S}}} \vec{t}$ we have to check that $\vec{p};\vec{u}\sim_{\hat{\mathcal{S}}} \vec{q};\vec{t}$. Consider arbitrary $\vec{r}$, $\vec{s}$, $A$, $B$. We have:
\begin{center}
$[\,\vec{r},1_A\otimes (\vec{p};\vec{u})\otimes 1_B,\vec{s}\,]\funnels[\,\vec{r},1_A\otimes \vec{p}\otimes 1_B,1_A\otimes \vec{u}\otimes 1_B,\vec{s}\,]\funnels[\,\vec{r},1_A\otimes \vec{q}\otimes 1_B,1_A\otimes \vec{u}\otimes 1_B,\vec{s}\,].$
\end{center}
The first equation is by definition of the tensor $\otimes_p$ on paths, the second equation is because by hypothesis we have that: $\vec{p}\sim_{\hat{\mathcal{S}}}\vec{q}$ implies $[\,\vec{r'},1_A\otimes \vec{p}\otimes 1_B,\vec{s'}\,]\funnels[\,\vec{r'},1_A\otimes \vec{q}\otimes 1_B,\vec{s'}\,]$ with $\vec{r}=\vec{r'}$ and $\vec{s'}=1_A\otimes \vec{u}\otimes 1_B,\vec{s}$.
In a similar way we have that:
\begin{center}
$[\,\vec{r},1_A\otimes (\vec{q};\vec{t})\otimes 1_B,\vec{s}\,]\funnels[\,\vec{r},1_A\otimes \vec{q}\otimes 1_B,1_A\otimes \vec{t}\otimes 1_B,\vec{s}\,]\funnels[\,\vec{r},1_A\otimes \vec{q}\otimes 1_B, 1_A\otimes \vec{u}\otimes 1_B,\vec{s}\,].$
\end{center}
It follows that $\vec{p};\vec{u}\sim_{\hat{\mathcal{S}}} \vec{q};\vec{t}$.
To prove $(2)$, assume $[\vec p]\downarrow$, and let
$\vec r,\vec s,A,B$ be given.
We observe first that $1_A\otimes [\, \vec{p}\, ]\otimes 1_B\funnel [\, 1_A\otimes \vec{p}\,\otimes 1_B ]$ by $(c)$ in the definition of a ssmpc. Then $[\,\vec{p}\,]\downarrow$ implies that $1_A\otimes [\, \vec{p}\, ]\otimes 1_B\downarrow$ and then $[\, 1_A\otimes \vec{p}\,\otimes 1_B ]\downarrow$ and and they are equal.
Thus we have by one of the axioms of paracategory that:
\begin{center}
$[\,\vec{r},1_A\otimes [\, \vec{p}\, ]\otimes 1_B, \vec{s}\,]\funnels[\,\vec{r},[\, 1_A\otimes \vec{p}\,\otimes 1_B\,] ,\vec{s}\,]\funnels[\,\vec{r},1_A\otimes \vec{p} \otimes 1_B, \vec{s}\,].$
\end{center}
To prove $(3)$, assume $\vec p\sim_{\hat{\mathcal{S}}} \vec p'$.
We observe that this implies
for every $C\in Obj(\mathcal{C})$, $[\,\vec{r},1_A\otimes \vec{p}\otimes 1_C\otimes 1_B,\vec{s}\,]\funnels[\,\vec{r},1_A\otimes \vec{p'}\otimes 1_C\otimes 1_B,\vec{s}\,]$, $\forall\vec{r},\vec{s},\forall A,B\in Obj(\mathcal{C})$, therefore $\vec{p}\otimes 1\sim_{\hat{\mathcal{S}}}\vec{p'}\otimes 1$. In a similar way we get the other equation.
\end{proof}
\begin{definition}
\rm
 Let $\sim$ be the smallest congruence relation on $\mathcal{P(C)}$,
i.e., the intersection of all congruence relations.
\end{definition}

\begin{proposition}\label{PARACAT6}
$\vec{p}\sim\vec{q}$ implies $[\,\vec{p}\,]\funnels[\,\vec{q}\,]$.
\end{proposition}
\begin{proof}
Since $(\vec{p},\vec{q})$ is in the intersection of all congruence relations then in particular  $\vec{p}\sim_{\hat{\mathcal{S}}}\vec{q}$ which implies that $[\,\vec{p}\,]\funnels[\,\vec{q}\,]$ by Remark~\ref{PARACAT1}.\\
\end{proof}
\begin{corollary}\label{PARACAT6b} For paths $p,q:A\to B$ of length 1, $p\sim q$ iff
 $p=q$.
\end{corollary}
 \begin{proof} Obvious from Proposition~\ref{PARACAT6} and axiom (b) of paracategories.
 \end{proof}

We now introduce the following notation:
$$ \vec{f}\otimes_{p}\vec{g}=(\vec{f}\otimes_{p}1),(1\otimes_{p}\vec{g}).$$

Note that, as a path, this is not equal to $(1\otimes_p \vec{g}),(\vec{f}\otimes_p 1)$. However, we will
show that they are congruent.
When it is clear from the context we drop the letter $p$.

\begin{lemma}\label{PARACAT5}
Let $\mathcal{S}$ be a congruence relation of $\mathcal{P(C)}$. Then if $\vec{f}\sim_{\mathcal{S}}\vec{f'}$ and $\vec{g}\sim_{\mathcal{S}}\vec{g'}$ then $\vec{f}\otimes\vec{g} \sim_{\mathcal{S}}\vec{f'}\otimes\vec{g'}$.
\end{lemma}
\begin{proof}
By assumption $\vec{f}\sim_{\cal S}\vec{f'}$ therefore by $(3)$ we have $\vec{f}\otimes 1\sim_{\cal S}\vec{f'}\otimes 1$. Similarly $1\otimes\vec{q}\sim_{\cal S}1\otimes\vec{q'}$. Therefore by $(1)$, we have: $\vec{f}\otimes 1,1\otimes\vec{q}\sim_{\cal S}\vec{f'}\otimes 1,1\otimes\vec{q'}$ .\\
\end{proof}

\begin{lemma}\label{PARACAT4}
Let $\mathcal{S}$ be a congruence relation of $\mathcal{P(C)}$. Then $$\vec{f}\otimes 1,1\otimes \vec{g}\sim_{\cal S} 1\otimes\vec{g},\vec{f}\otimes 1.$$
\end{lemma}
\begin{proof}
Given $\vec{f}=f_1,\dots, f_n$ and $\vec{g}=g_1,\dots,g_m$ we have that by Proposition~\ref{PARACAT2} above $[f_n\otimes 1,1\otimes  g_1]\downarrow$, $[1\otimes g_1,f_n\otimes 1]\downarrow$ and are equal to $f_n\otimes g_1$. This yields, by Definition~\ref{PARACAT3} of congruence relation, the following sequence of equivalences:
\begin{center}
$f_n\otimes 1,1\otimes g_1\,\,\,\sim_{\cal S} \,\,\,[f_n\otimes 1,1\otimes g_1]=[1\otimes g_1,f_n\otimes 1]\,\,\,\sim_{\cal S}\,\,\, 1\otimes g_1,f_n\otimes  1$.
\end{center}
Which implies by composition:
\begin{center}
$f_1\otimes 1,\dots, f_{n-1}\otimes 1,\overbrace{(f_n\otimes 1,1\otimes g_1)},1\otimes g_2,\dots 1\otimes g_m \,\,\,\sim_ {\cal S}
f_1\otimes 1,\dots, f_{n-1}\otimes 1,\underbrace{(1\otimes g_1,f_n\otimes 1)}, 1\otimes g_2,\dots 1\otimes g_m $.
\end{center}
By iterating this procedure we end up moving $1\otimes g_1$ into the first place. We finish the proof by repeating this $m-1$ times.
\end{proof}

From now, $``;"$ denotes composition in the quotient category written in diagrammatic order (here this means concatenation of paths).
\begin{lemma}
Let $\mathcal{S}$ be a congruence relation defined on a strict symmetric monoidal paracategory $(\mathcal{C},[-],\otimes,I,\sigma)$. Then the quotient $(\mathcal{P(C)}/\mathcal{S},\hat{\otimes},I,s)$ is a strict symmetric monoidal category, where $\hat{\otimes}$ is the obvious tensor and $s=\overline{\sigma}$.
\end{lemma}
\begin{proof}
Let $(\mathcal{C},[-],\otimes,I,\sigma)$ be a strict symmetric monoidal paracategory. It induces a strict symmetric monoidal category $(\mathcal{P(C)}/\mathcal{S},\hat{\otimes},I,s)$ in the following way:\\
The objects of $\mathcal{P(C)}/\mathcal{S}$ are the same as the objects of the graph $\mathcal{C}$ and the arrows $\overline{\vec{f}}=\overline{f_1\ldots,f_n}$ are $\mathcal{S}$-equivalence classes of paths. Composition on classes is induced by composition on paths by axiom $(1)$ of congruences. The identity is the class of the identity of the path category.\\
A bifunctor $\hat{\otimes}:P(C)/\mathcal{S} \times P(C)/\mathcal{S}\rightarrow P(C)/\mathcal{S}$ is  defined by $\overline{\vec{f}}\hat{\otimes}\overline{\vec{g}}=\overline{\vec{f}\otimes_{p}1,1\otimes_{p}\vec{g}}$. The tensor is well-defined by the Lemma~\ref{PARACAT5} above.\\
We must check the interchange law: $$(\overline{\vec{f}}\hat{\otimes}\overline{\vec{g}});(\overline{\vec{f'}}\hat{\otimes}\overline{\vec{g'}})=(\overline{\vec{f}};\overline{\vec{f'}})\hat{\otimes}(\overline{\vec{g}};\overline{\vec{g'}})$$
We have:

\begin{center}
$\vspace{3mm}(\overline{\vec{f}}\hat{\otimes}\overline{\vec{g}});(\overline{\vec{f'}}\hat{\otimes}\overline{\vec{g'}})=(\overline{\vec{f}\otimes 1,1\otimes \vec{g}});(\overline{\vec{f'}\otimes 1,1\otimes \vec{g'}})=
\overline{\vec{f}\otimes 1,1\otimes \vec{g},\vec{f'}\otimes 1,1\otimes \vec{g'}}=
\vspace{3mm}\overline{\vec{f}\otimes 1};\overline{1\otimes \vec{g},\vec{f'}\otimes 1};\overline{1\otimes \vec{g'}}\stackrel{(*)}=
\overline{\vec{f}\otimes 1};\overline{\vec{f'}\otimes 1,1\otimes \vec{g}};\overline{1\otimes \vec{g'}}=
\overline{\vec{f}\otimes 1,\vec{f'}\otimes 1,1\otimes \vec{g},1\otimes \vec{g'}}=
\overline{(\vec{f},\vec{f'})\otimes 1,1\otimes (\vec{g},\vec{g'})}=
\overline{(\vec{f},\vec{f'})\otimes_p (\vec{g},\vec{g'})}=
\overline{(\vec{f},\vec{f'})}\hat{\otimes}\overline{(\vec{g},\vec{g'})}
=(\overline{\vec{f}};\overline{\vec{f'}})\hat{\otimes}(\overline{\vec{g}};\overline{\vec{g'}})$.
\end{center}

Where in $(*)$ we used the property of the Lemma~\ref{PARACAT4} above: $\overline{\,1\otimes \vec{g},\vec{f'}\otimes 1\,}=\overline{\,\vec{f'}\otimes 1,1\otimes \vec{g}}$.\\
Also we want to check that
$\overline{\epsilon_{A\otimes B}}=\overline{\epsilon_A}\hat{\otimes}\overline{\epsilon_B}$.

$$\overline{\epsilon_A}\hat{\otimes}\overline{\epsilon_B}=\overline{\epsilon_A \otimes_p 1_B,1_A\otimes_p \epsilon_B}=\overline{\epsilon_{A\otimes B},\epsilon_{A\otimes B}}=\overline{\epsilon_{A\otimes B}}$$

since $[\,1_{A\otimes B},1_{A\otimes B}\,]\downarrow $.\\
Given paths $\vec{f}:A\rightarrow B$, $\vec{g}:C\rightarrow D$ and $\vec{h}:E\rightarrow F$ we check the associative property:
\begin{center}
$\vspace{3mm}(\overline{\vec{f}}\hat{{\otimes}}\overline{\vec{g}})\hat{\otimes}\overline{\vec{h}}=
(\overline{\vec{f}\otimes_p\vec{g}})\hat{\otimes}\overline{\vec{h}}=
\overline{(\vec{f}\otimes 1_C,1_B\otimes \vec{g})}\hat{\otimes}\overline{\vec{h}}=
\overline{(\vec{f}\otimes 1_C, 1_B\otimes \vec{g})\otimes_p\vec{h}}=
\vspace{3mm}\overline{(\vec{f}\otimes 1_C,1_B\otimes \vec{g})\otimes 1_E,1_{B\otimes D}\otimes\vec{h}}=
\overline{\vec{f}\otimes 1_C\otimes 1_E,1_B\otimes \vec{g}\otimes 1_E, 1_{B\otimes D}\otimes\vec{h}}=
\vspace{3mm}\overline{\vec{f}\otimes 1_{C\otimes E},1_B\otimes \vec{g}\otimes 1_E, 1_B\otimes 1_{D}\otimes\vec{h}}=
\overline{\vec{f}\otimes 1_{C\otimes E},1_B\otimes (\vec{g}\otimes 1_E, 1_{D}\otimes\vec{h})}=
\overline{\vec{f}\otimes_p (\vec{g}\otimes 1_E, 1_{D}\otimes\vec{h})}=
\overline{\vec{f}\otimes_p (\vec{g}\otimes_p\vec{h})}=
\overline{\vec{f}}\hat{\otimes} (\overline{\vec{g}\otimes_p\vec{h}})=
\overline{\vec{f}}\hat{\otimes} (\overline{\vec{g}}\hat{\otimes}\overline{\vec{h}}).$
\end{center}

Also if $\vec{f}:A\rightarrow B$ and $1_I:I\rightarrow I$ then:
\begin{center}
$\overline{\vec{f}}\hat{\otimes}\overline{1_I}=\overline{\vec{f}\otimes_p 1_I}=\overline{\vec{f}\otimes 1_I,1_B\otimes 1_I}=\overline{\vec{f},1_B}=\overline{\vec{f}}.$
\end{center}
Since $\vec{f}\otimes 1_I=\overrightarrow{f\otimes 1_I}=\vec{f}$ and $1_B\otimes 1_I=1_B$. In the same way we get $\overline{{1_I}}\hat{\otimes}\overline{\vec{f}}=\overline{\vec{f}}$.\\
The symmetry is defined as $s_{A,B}:A\hat{\otimes} B\rightarrow B\hat{\otimes}A$, $s_{A,B}=\overline{\sigma_{A,B}}$.
This arrow is an isomorphism since $[\sigma_{A,B},\sigma_{B,A}]\downarrow$ implies $\sigma_{A,B},\sigma_{B,A}\sim_{\cal S}[\sigma_{A,B},\sigma_{B,A}]$ and then:
$$s_{A,B}; s_{B,A}=\overline{\sigma_{A,B}};\overline{\sigma_{B,A}}=\overline{\sigma_{A,B},\sigma_{B,A}}=\overline{[\sigma_{A,B},\sigma_{B,A}]}=\overline{1_{A\otimes B}}.$$
Similarly  $s_{B,A};s_{A,B}=\overline{1_{B\otimes A}}$.\\
Next, we check the following coherence diagram: $(s_{A,B}\otimes 1_C); s_{B,A\otimes C}=1_A\otimes s_{B,C}$.\\
\begin{center}
$\vspace{3mm}(s_{A,B}\otimes 1_C); s_{B,A\otimes C}=(\overline{\sigma_{A,B}}\hat{\otimes} \overline{1_C});\overline{\sigma_{B,A\otimes C}}=\overline{(\sigma_{A,B}\otimes 1_C)};\overline{\sigma_{B,A\otimes C}}=\overline{(\sigma_{A,B}\otimes 1_C),\sigma_{B,A\otimes C}}=\overline{[(\sigma_{A,B}\otimes 1_C),\sigma_{B,A\otimes C}]}=\overline{1_A\otimes \sigma_{B,C}}=\overline{1_A}\hat{\otimes}\overline{\sigma_{B,C}}=1_A\hat{\otimes}s_{B,C}.$
\end{center}
Next we prove naturality of the map $s_{A,B}:A\hat{\otimes} B\rightarrow B\hat{\otimes}A$. To see this, it is enough to prove it on simple path of length one and then extend it by composition.  Let us consider $\overline{f}:A\rightarrow A'$ and since $\overline{[\sigma_{A,B},1\otimes f]}\downarrow$  \\
$\vspace{3mm}s_{A,B}; (\overline{1}\hat{\otimes}\overline{f})=\overline{\sigma_{A,B}};(\overline{1}\hat{\otimes}\overline{f})=\overline{\sigma_{A,B}};\overline{1\otimes f}=\overline{\sigma_{A,B},1\otimes f}=\overline{[\sigma_{A,B},1\otimes f]}=\overline{[f\otimes 1,\sigma_{A',B}]}=\overline{f\otimes 1,\sigma_{A',B}}=\overline{f\otimes 1};\overline{\sigma_{A',B}}=(\overline{f}\hat{\otimes}\overline{1}); s_{A',B}.$\\

For the general case we iterate this, applying the above equation several times.

\end{proof}

{\bf Proof of Theorem~\ref{PARACAT7}}
\begin{proof}

A functor between paracategories $F:(\mathcal{C},[-],\otimes,I,\sigma)\rightarrow (\mathcal{P(C)}/\mathcal{S},\hat{\otimes},I,s)$, where the category  $\mathcal{P(C)}/\mathcal{S}$ is taken as a (total) paracategory, is defined in the following way:
\begin{itemize}
\item[-]on objects as the identity and
\item[-] on arrows $F(f)=\overline{f}$ as the projection on classes.
\end{itemize}
Observe that $F$ preserves identities and composition when $[\vec{f}]$ is defined:
$$ F[\vec{f}]=\overline{[\vec{f}]}=\overline{\vec{f}}=\overline{f_1,\ldots,f_n}=\overline{f_1};\ldots;\overline{f_n}= Ff_1;\ldots; Ff_n.$$
Following the definition, we have that $F$ preserves symmetries: $F(\sigma)=\overline{\sigma}=s$.\\
In addition, if $f:A\rightarrow C$ and $g:B\rightarrow D$ then
$$F(f\otimes g)=\overline{f\otimes g}=\overline{[f\otimes 1_B,1_C\otimes g]}=\overline{f\otimes 1_B,1_C\otimes g}=\overline{f\otimes_p g}=Ff\hat{\otimes}Fg$$
where the last sequence of equations is justified by Proposition~\ref{PARACAT2}, the property above, axioms and by definition of congruence relation.\\
Moreover, if $\mathcal{S}$ is the smallest congruence relation, or indeed any
congruence relation satisfying $\mathcal{S}\subseteq \hat{\mathcal{S}}$, then $F$ is faithful by Corollary~\ref{PARACAT6b}.
\end{proof}

\section{Compact closed paracategories}

\begin{definition}
\rm A \textit{(strict symmetric) compact closed paracategory} $(\mathcal{C},[-],\otimes,I,\sigma,\eta,\epsilon)$ is a strict symmetric monoidal paracategory such that for every object $A$ there is an object $A^*$ and arrows $\eta_A:I\rightarrow A\otimes A^*$, $\epsilon_A :A^*\otimes A\rightarrow I$ such that $[\eta_A\otimes 1_A,1_A\otimes\epsilon_A ]\downarrow$, $[1_{A^*}\otimes\eta_A,\epsilon_A\otimes 1_{A^*}]\downarrow$ and  $[\eta_A\otimes 1_A,1_A\otimes\epsilon_A ]=1_A$, $[1_{A^*}\otimes\eta_A,\epsilon_A\otimes 1_{A^*}]=1_{A^*}$.
\end{definition}

\begin{theorem}\label{FAITHFUL EMBED COMP CLOSED PARA}
Every compact closed paracategory can be
faithfully embedded in a compact closed category.
\end{theorem}
\begin{proof}
Let us consider the paracategory $(\mathcal{C},[-],\otimes,I,\sigma,\eta,\epsilon)$.\\
As a result of the proof of Theorem~\ref{PARACAT7} above, it suffices to show that $(\mathcal{P(C)}/\mathcal{S},\hat{\otimes},I,s,\eta',\epsilon')$ is compact closed, where $\eta'=\overline{\eta}$ and $\epsilon'=\overline{\epsilon}$. Notice that by definition the functor $F$ preserves $\eta$ and $\epsilon$.
Consequently, the compactness diagrams are satisfied, since the condition $[\eta\otimes 1, 1\otimes \epsilon]\downarrow$ implies:\\
\begin{align*}
\overline{\eta}\hat{\otimes}\overline{1_A};\overline{1_A}\hat{\otimes}\overline{\epsilon}    &= \overline{\eta\otimes 1_A}; \overline{1_A\otimes \epsilon}\\
                    &= \overline{\eta\otimes 1_A,1_A\otimes \epsilon}\\
                    &= \overline{[\eta\otimes 1_A,1_A\otimes \epsilon]}\\
                    &= \overline{1_A}.\\
                    &
\end{align*}
In the same way,  $\overline{1_{A^*}}\hat{\otimes}\overline{\eta};\overline{\epsilon}\hat{\otimes}\overline{1_{A^*}}= \overline{1_{A^*}}$

\end{proof}

\section{Freeness}

We can strengthen Theorem~\ref{PARACAT7} by noting that the faithful
 embedding satisfies a universal property.
\begin{theorem}\label{Freeness}
The category $(\mathcal{P(C)/{\sim}},\hat{\otimes},I,s)$ satisfies the following property: for any strict symmetric monoidal category $\mathcal{D}$ and any strict symmetric monoidal functor $G:\mathcal{C}\rightarrow \mathcal{D}$ between paracategories,
there exists a unique strict symmetric monoidal functor $L:\mathcal{P(C)/{\sim}}\rightarrow \mathcal{D}$ such that $L\circ F=G$, where $F$ is the inclusion map defined in Theorem~\ref{PARACAT7} above.

$$\xymatrix@=25pt{
\mathcal{C}\ar[rr]^{F}\ar[rrd]_{G}
  && \mathcal{P(C)/{\sim}} \ar[d]^{L}\\
  && \mathcal{D}
}$$
\end{theorem}
\begin{proof}
Consider the set $\mathcal{S}$:
$$\{(\vec{f},\vec{g})\in\mathcal{P(C)}\times \mathcal{P(C)}:\,\,\,G(f_1)\circ\dots\circ G(f_n)=G(g_1)\circ\dots\circ G(g_m)\,\}$$
where $\vec{f}=f_1,\ldots ,f_n$ and $\vec{g}=g_1,\ldots,g_m.$

We claim that $\mathcal{S}$ is a congruence relation in the sense of Definition~\ref{PARACAT3} stipulated above.
Clearly, it is an equivalence relation. To show that it satisfies axiom $(1)$, assume
\begin{center}
$p_1,\dots p_n\sim_{\mathcal{S}} q_1,\dots,q_m$ and $f_1,\dots f_s \sim_{\mathcal{S}}\,q_1,\dots,g_t$, then by hypothesis $G(p_1)\circ\dots\circ G(p_n)=G(q_1)\circ\dots\circ G(q_m)$ and $G(f_1)\circ\dots\circ G(f_s)=G(g_1)\circ\dots\circ G(g_t)$.
\end{center}
Then by composing the left hand side and the right hand side we get the condition
$$p_1,\dots ,p_n,f_1,\dots ,f_s\sim_{\mathcal{S}}q_1,\dots,q_m,g_1,\dots,g_t.$$
To show $(3)$, assume
$p_1,\dots,p_n\sim_{\hat{\mathcal{S}}}q_1,\dots ,q_m$, then $G(p_1)\circ\dots\circ G(p_n)=G(q_1)\circ\dots\circ G(q_m)$ which in $\mathcal{C}$ implies that  $(G(p_1)\circ\dots\circ G(p_n))\otimes G1=(G(q_1)\circ\dots\circ G(q_m))\otimes G1$ and by the tensor property of $G$ and functoriality we obtain $G(p_1\otimes 1)\circ\dots\circ G(p_n\otimes 1)=G(q_1\otimes 1)\circ\dots\circ G(q_m\otimes 1)$, which means $p_1\otimes 1,\dots,p_n\otimes 1\sim_{\mathcal{S}}q_1\otimes 1,\dots ,q_m\otimes 1.$
In the same way $\vec{p}\sim_{\mathcal{S}}\vec{q}$ implies $1\otimes\vec{p}\sim_{\mathcal{S}}1\otimes\vec{q}$.\\
To show $(2)$, since $G$ is a functor between paracategories, we have  $G(p_1)\circ\dots G(p_n)=G([\,\vec{p}\,])$ when $[\,\vec{p}\,]\downarrow$ hence $\vec{p}\sim_{\mathcal{S}}[\,\vec{p}\,]$.\\
Now we define the functor $L$ in the following way:
\begin{center}
$L(A)=G(A)$ on objects and $L(\overline{\vec{p}})=G(p_1)\circ\dots\circ G(p_n) $, where $\vec{p}=p_1,\dots , p_n$.
\end{center}
It should be apparent that $F$ is well-defined since when $\vec{p}\sim\vec{q}$ then in particular it is true that $\vec{p}\sim_{\mathcal{S}}\vec{q}$ and this implies $L(\overline{\vec{p}})=L(\overline{\vec{q}})$.\\
We check functoriality:

\begin{align*}
 L(\overline{p_1,\dots ,p_n};\overline{q_1,\dots,q_m})   &= L(\overline{p_1,\dots ,p_n,q_1,\dots,q_m}) \\
                    &= G(p_1)\circ\dots\circ G(p_n)\circ G(q_1)\circ\dots\circ G(q_m) \\
                    &= L(\overline{p_1,\dots ,p_n})\circ L(\overline{q_1,\dots ,q_m} )
\end{align*}

and $$L(\overline{(A)})=L(\overline{1_A})=G(1_A)=1_{GA}.$$

Furthermore $L$ is strict symmetric monoidal:\\
\begin{align*}
 L(\overline{\vec{p}}\hat{\otimes}\overline{\vec{q}})   &= L(\overline{\vec{p}\otimes 1,1\otimes\vec{q}})\\
                    &= G(p_1\otimes 1)\circ\dots\circ G(p_n\otimes 1)\circ G(1\otimes q_1)\circ\dots\circ G(1\otimes q_m)\\
                    &= (G(p_1)\otimes G1)\circ\dots\circ (G(p_n)\otimes G1)\circ (G1\otimes G(q_1))\circ\dots\circ (G1\otimes G(q_m))\\
                    &= (G(p_1)\otimes 1)\circ\dots\circ (G(p_n)\otimes 1)\circ (1\otimes G(q_1))\circ\dots\circ (1\otimes G(q_m))\\
                    &= (G(p_1)\circ\dots\circ G(p_n))\otimes 1)\circ (1\otimes (G(q_1)\circ\dots\circ G(q_m))\\
                    &=(G(p_1)\circ\dots\circ G(p_n))\otimes (G(q_1)\circ\dots\circ G(q_m))\\
                    &=L(\overline{\vec{p}})\otimes L(\overline{\vec{q}})
\end{align*}

Finally, since $G$ is strict symmetric $L(s)=L(\overline{\sigma})=G(\sigma)=\sigma'$ where $\sigma '$ is the symmetry of the category $\mathcal{D}$.

\end{proof}

\section{Partially traced categories and the partial Int construction}

Joyal, Street, and Verity proved in \cite{JSV96} that every (totally) traced
 monoidal category $\cC$ can be faithfully embedded in a compact closed
 category $Int(\cC)$. Here, we give a similar construction for {\em
 partially} traced categories. We call the corresponding construction
 the {\em partial Int construction}, or the $Int^p$ construction for
 short. When $\cC$ is a partially traced category, $Int^p(\cC)$ will be a
 compact closed {\em para}category.

\begin{definition}
\rm Let $(\mathcal{C},\otimes,I,\sigma)$ be a symmetric monoidal category. There is a graph  $Int^{p}(\mathcal{C})$ associated to this category defined in the following way:
\begin{itemize}
\item objects: are a pair of object $(A^+,A^-)$ of the category $\mathcal{C}$.
\item arrows: $f^{Int^p}:(A^+,A^-)\rightarrow (B^+,B^-)$ are arrows of type $f:A^+\otimes B^-\rightarrow B^+\otimes A^-$ in the category $\mathcal{C}$.
\end{itemize}
\end{definition}
When it is clear from the context we drop the symbol $Int^p$ on the arrows of $Int^p(\mathcal{C})$.

We want to define a partial composition on this graph. For that purpose, consider the following natural transformation,
uniquely induced by the symmetric monoidal structure, for $n\geq 0$:
\begin{center}
$\gamma_n:A_1\otimes A_2\otimes\dots A_{n-1}\otimes A_{n}\rightarrow A_n\otimes A_{n-1}\dots A_2\otimes A_1.$
\end{center}
Also, given a path $\vec{p}=p_1,\dots,p_m\in \mathcal{P}(Int^{p}(\cC))$, using graphical language of symmetric monoidal
categories, we shall define an arrow $\epsilon(\vec{p})\in \mathcal{C}$ in the following way:
if $\vec{p}=p_1,\dots,p_m$ then $\epsilon(\vec{p})$ pictorially is equal to:

For $m=1$ arrow:\\

\[\xymatrix@!0{
&*{}\ar@{-}[rd]&*{}\ar@{-}[r]&*{}\ar@{-}[rd]&*{}\\
&*{}\ar@{-}[ru]&*{}\ar@{}[rd]|{\fbox{\rule[-.6cm]{0cm}{1cm}$p_1$}}&*{}\ar@{-}[ru]&*{}&*{}\\
&*{}\ar@{-}[r]&*{}&*{}\ar@{-}[r]&*{}
}\]

For $m=2$ arrows:\\
\[\xymatrix@!0{
&*{}\ar@{-}[rd]&*{}\ar@{-}[r]&*{}\ar@{-}[rd]&*{}\ar@{-}[r]&*{}\ar@{-}[rd]&*{}\\
&*{}\ar@{-}[ru]&*{}\ar@{}[rd]|{\fbox{\rule[-.6cm]{0cm}{1cm}$p_1$}}&*{}\ar@{-}[ru]&*{}\ar@{}[rd]|{\fbox{\rule[-.6cm]{0cm}{1cm}$p_2$}}&*{}\ar@{-}[ru]&*{}\\
&*{}\ar@{-}[r]&*{}&*{}\ar@{-}[r]&*{}&*{}\ar@{-}[r]&*{}
}\]

For $m=3$ arrows:\\

\[\xymatrix@!0{
&*{}&*{}\ar@{-}[rd]&*{}\ar@{-}[r]&*{}\ar@{-}[rd]&*{}\ar@{-}[r]&*{}\ar@{-}[rd]&*{}&*{}\\
&*{}\ar@{-}[rd]&*{}\ar@{-}[ru]&*{}\ar@{-}[rd]&*{}\ar@{-}[ru]&*{}\ar@{-}[rd]&*{}\ar@{-}[ru]&*{}\ar@{-}[rd]&*{}\\
&*{}\ar@{-}[ru]&*{}&*{}\ar@{-}[ru]&*{}&*{}\ar@{-}[ru]&*{}&*{}\ar@{-}[ru]&*{}\\
&*{}\ar@{-}[r]&*{}\ar@{}[ru]|{\fbox{\rule[-.6cm]{0cm}{1cm}$p_1$}}&*{}\ar@{-}[r]&*{}\ar@{}[ru]|{\fbox{\rule[-.6cm]{0cm}{1cm}$p_2$}}&*{}\ar@{-}[r]&*{}\ar@{}[ru]|{\fbox{\rule[-.6cm]{0cm}{1cm}$p_3$}}&*{}\ar@{-}[r]&*{}
}\]
and so on.

In order to get $\epsilon(p_1,\dots,p_m)$ we form a pyramid of $m-1$ layers of symmetries.
\begin{definition}\label{DEFINITION INT PARACATEGORY}
\rm Let $(\mathcal{C},\otimes,I,\Tr,s)$ be a symmetric monoidal partially traced category. We turn the graph $Int^p(C)$ into a paracategory
by defining a partial composition operation $[\vec p\,]$. First of all, when it is applied to an empty path it will be defined as the identity arrow i.e., $[(A^+,A^-)]=1_{A^+\otimes A^-}.$ On path of length one it will be by definition the same arrow, i.e., $[f]=\Tr^{U}(\epsilon(f)(1_{X_1^+}\otimes \sigma_{X^{-}_2,X^{-}_1}))=f$ with $U=X^{-}_1.$\\
Suppose that we have a family of arrows $f_i^{Int^p}:(X_i^+,X_i^-)\rightarrow (X_{i+1}^+,X_{i+1}^-)$ with $1\leq i\leq n$ ($n\geq 2$) in the graph $Int^p(\mathcal{C})$ such that $dom(f_{i+1})=cod(f_i)$ and $1\leq i\leq n-1$.
Let $U=X^-_n\otimes X^-_{n-1}\otimes\dots\otimes X_3^-\otimes X_{2}^-$ and the permutation $\gamma$
$$X^-_n\otimes X^-_{n-1}\otimes\dots\otimes X_3^-\otimes X_{2}^-\stackrel{\gamma}\longrightarrow X^-_2\otimes X^-_{3}\otimes\dots \otimes X_{n-1}^-\otimes X_{n}^-.$$

We define the following operation for $n\geq 2$:
\begin{center}
$[f_1,\dots,f_n]\funnels \Tr^{U}(\epsilon(f_1,\dots,f_n)(1_{X_1^+}\otimes 1_{X_{n+1}^-}\otimes \gamma_{n-1})).$
\end{center}
\end{definition}

Note that therefore, $[f_1,\ldots,f_n]$ is defined if and only if
$$\epsilon(f_1,\dots,f_n)(1_{X_1^+}\otimes 1_{X_{n+1}^-}\otimes \gamma)\in \Trc^U.$$

We show now that the operation $[-]$ satisfies the axioms required in order to be a paracategory.

\begin{lemma}
Let $(\mathcal{C},\otimes,I,\Tr,s)$ be a strict symmetric monoidal partially traced category. The operation defined in Definition~\ref{DEFINITION INT PARACATEGORY} determines a paracategory $(Int^p(\mathcal{C}),[-])$.
\end{lemma}
\begin{proof}
Properties $(a)$ and $(b)$ of Definition~\ref{DEF PARACATEGORY} hold by definition. The goal is to prove
$(c)$, i.e., if $[\vec{g}]\downarrow$ then $[\vec{f},[\vec{g}],\vec{h}]\funnels[\vec{f},\vec{g},\vec{h}]$ for every $\vec{f}$ and $\vec{h}$. The value of the trace remains always invariant or follows the variations that the axioms trace dictate.\\
Without loss of generality we are going to represent these paths using graphical language in a concrete situation. Therefore, suppose we have $\vec{f}=f_1,f_2$, $\vec{g}=g_1,g_2,g_3,g_4$ and $\vec{h}=h_1,h_2,h_3$. The most general case follows the same pattern.\\

The fact that $[\vec{g}]\downarrow$ means that the map:

\begin{equation}\label{g equation}
\xymatrix@R=10pt@C=18pt{
&*{}\ar@/_1pc/@{.}[dd]|{V}
\ar@{-}[rdd]&*{}\ar@{-}[rr]&*{}&*{}\ar@{-}[rd]&*{}\ar@{-}[r]&*{}\ar@{-}[rd]&*{}\ar@{-}[r]&*{}\ar@{-}[rd]&*{}\ar@{-}[rr]&*{}&*{}\ar@/^1pc/@{.}[dd]|{V}
\\
&*{}\ar@{-}[r]&*{}\ar@{-}@{-}[r]&*{}\ar@{-}[rd]&*{}\ar@{-}[ru]&*{}\ar@{-}[rd]&*{}\ar@{-}[ru]&*{}\ar@{-}[rd]&*{}\ar@{-}[ru]&*{}\ar@{-}[rd]&*{}\ar@{-}[r]&*{}\\
&*{}\ar@{-}[ruu]&*{}\ar@{-}[rd]&*{}\ar@{-}[ru]&*{}\ar@{-}[rd]&*{}\ar@{-}[ru]&*{}\ar@{-}[rd]&*{}\ar@{-}[ru]&*{}\ar@{-}[rd]&*{}\ar@{-}[ru]&*{}\ar@{-}[rd]&*{}\\
&*{}\ar@{-}[r]&*{}\ar@{-}[ru]&*{}&*{}\ar@{-}[ru]&*{}&*{}\ar@{-}[ru]&*{}&*{}\ar@{-}[ru]&*{}&*{}\ar@{-}[ru]&*{}\\
&*{}\ar@{-}[rr]&*{}\ar@{-}[r]&*{}\ar@{-}[ru]|{\fbox{\rule[-.6cm]{0cm}{1cm}$g_1$}}&*{}\ar@{-}[r]&*{}\ar@{-}[ru]|{\fbox{\rule[-.6cm]{0cm}{1cm}$g_2$}}&*{}\ar@{-}[r]&*{}\ar@{-}[ru]|{\fbox{\rule[-.6cm]{0cm}{1cm}$g_3$}}&*{}\ar@{-}[r]&*{}\ar@{-}[ru]|{\fbox{\rule[-.6cm]{0cm}{1cm}$g_4$}}&*{}\ar@{-}[r]&*{}
}
\end{equation}

(without the dotted lines) is in the trace class $\Trc^V$. We symbolize that it is in the trace class of this type with these dotted lines. Moreover, $[\vec{f},[\vec{g}],\vec{h}]\downarrow$ means that:

\begin{equation}\label{f g h equation}
\xymatrix@R=10pt@C=18pt{
\ar@/_1pc/@{.}[dddd]|{U}\ar@{-}[rdddd]&*{}\ar@{-}[rrrr]&*{}&*{} &*{}
&*{}\ar@{-}[rd]&*{}\ar@{-}[r] &*{}\ar@{-}[rd]&*{}\ar@{-}[r]
&*{}\ar@{-}[rd]&*{}\ar@{-}[rrrrr]  &*{}  &*{}  &*{}  &*{}  &*{}
\ar@/^1pc/@{.}[dddd]|{U}&*{}\\
\ar@{-}[rdd]&*{}\ar@{-}[rrr]&*{}&*{}
&*{}\ar@{-}[rd]&*{}\ar@{-}[ru]&*{}\ar@{-}[rd]&*{}\ar@{-}[ru]&*{}\ar@{-}[rd]&*{}\ar@{-}[ru]&*{}\ar@{-}[rd]&*{}\ar@{-}[rrrr]
 &*{}  &*{}  &*{}  &*{}  &*{}\\
\ar@{-}[r]&*{}\ar@{-}[rr]&*{}&*{}\ar@{-}[rd]&*{}\ar@{-}[ru]&*{}\ar@{-}[rd]&*{}\ar@{-}[ru]&*{}\ar@{-}[rd]&*{}\ar@{-}[ru]&*{}\ar@{-}[rd]&*{}\ar@{-}[ru]&*{}\ar@{-}[rd]&*{}\ar@{-}[rrr]
 &*{}  &*{}  &*{}  &*{}\\
\ar@{-}[ruu]&*{}\ar@{-}[r]&*{}\ar@{-}[rd]&*{}\ar@{-}[ru]&*{}\ar@{-}[rd]&*{}\ar@{-}[ru]&*{}\ar@{-}[rd]&*{}\ar@{-}[ru]&*{}\ar@{-}[rd]&*{}\ar@{-}[ru]&*{}\ar@{-}[rd]&*{}\ar@{-}[ru]&*{}\ar@{-}[rd]&*{}\ar@{-}[rr]&*{}&*{}&*{}\\
\ar@{-}[ruuuu]&*{}\ar@{-}[rd]&*{}\ar@{-}[ru]&*{}\ar@{-}[rd]&*{}\ar@{-}[ru]&*{}\ar@{-}[rd]&*{}\ar@{-}[ru]&*{}\ar@{-}[rd]&*{}\ar@{-}[ru]&*{}\ar@{-}[rd]&*{}\ar@{-}[ru]&*{}\ar@{-}[rd]&*{}\ar@{-}[ru]&*{}\ar@{-}[rd]&*{}\ar@{-}[r]&*{}&*{}\\
\ar@{-}[r]  &*{}  \ar@{-}[ru]&*{}   &*{}\ar@{-}[ru]&*{}   &*{}\ar@{-}[ru]&*{}
&*{}\ar@{-}[ru]&*{}  &*{}\ar@{-}[ru]&*{}    &*{}\ar@{-}[ru]&*{} &*{}\ar@{-}[ru]
&*{}\ar@{-}[r]&*{} &*{}\\
\ar@{-}[rr]  &*{} &*{}\ar@{}[ru]|{\fbox{\rule[-.6cm]{0cm}{1cm}$f_1$}} &*{}\ar@{-}[r]
 &*{} \ar@{}[ru]|{\fbox{\rule[-.6cm]{0cm}{1cm}$f_2$}}
&*{}\ar@{-}[r]&*{}\ar@{}[ru]|{\fbox{\rule[-.6cm]{0cm}{1cm}$[\vec{g}]$}}&*{}\ar@{-}[r]
      &*{}\ar@{}[ru]|{\fbox{\rule[-.6cm]{0cm}{1cm}$h_1$}}&*{}\ar@{-}[r]
&*{}\ar@{}[ru]|{\fbox{\rule[-.6cm]{0cm}{1cm}$h_2$}}&*{}\ar@{-}[r]&*{}\ar@{}[ru]|{\fbox{\rule[-.6cm]{0cm}{1cm}$h_3$}}
&*{} \ar@{-}[rr]&*{}  &*{}&*{}
}
\end{equation}

without the dotted lines is in trace class $\Trc^U$. We want to obtain $[\vec{f},\vec{g},\vec{h}]$. So, for that purpose, we start by replacing the first diagram~(\ref{g equation}) traced on $V$ into the second diagram~(\ref{f g h equation}). Then we apply superposition, and the naturality axiom and we get the following diagram:

\begin{equation}\label{f g h UV equation}
\xymatrix@R=10pt@C=16pt{
\ar@/_1pc/@{.}[dd]|{V}\ar@{-}[rrrrrr]&*{}&*{}&*{}&*{}&*{}&*{}\ar@{-}[rddddd]&*{}\ar@{-}[rrrrrrrrrr]&*{}&*{}&*{}&*{}&*{}&*{}&*{}&*{}&*{}&*{}\ar@{-}[rddd]&*{}\ar@{-}[rrrrrrr]&*{}&*{}&*{}&*{}&*{}&*{}&*{}\ar@/^1pc/@{.}[dd]|{V}&*{}\\
\ar@{-}[rrrrrr]&*{}&*{}&*{}&*{}&*{}&*{}\ar@{-}[rddddd]&*{}\ar@{-}[rrrr]&*{}&*{}&*{}&*{}\ar@{-}[rd]&*{}\ar@{-}[rrrrr]&*{}&*{}&*{}&*{}&*{}\ar@{-}[rddd]&*{}\ar@{-}[rrrrrrr]&*{}&*{}&*{}&*{}&*{}&*{}&*{}&*{}\\
\ar@{-}[rrrrrr]&*{}&*{}&*{}&*{}&*{}&*{}\ar@{-}[rddddd]&*{}\ar@{-}[rrrr]&*{}&*{}&*{}&*{}\ar@{-}[ru]&*{}\ar@{-}[rrrrr]&*{}&*{}&*{}&*{}&*{}\ar@{-}[rddd]&*{}\ar@{-}[rrrrrrr]&*{}&*{}&*{}&*{}&*{}&*{}&*{}\\
\ar@/_1pc/@{.}[dddd]|{U}\ar@{-}[rdddd]&*{}\ar@{-}[rrrr]&*{}&*{} &*{}
&*{}\ar@{-}[rd]&*{}\ar@{-}[ruuu] &*{}
\ar@{-}[rrrr]&*{}&*{}&*{}&*{}\ar@{-}[rd]&*{}\ar@{-}[rrrrr]&*{}&*{}&*{}&*{}&*{}\ar@{-}[rddd]&*{}\ar@{-}[rd]&*{}\ar@{-}[r]&*{}\ar@{-}[rd]&*{}\ar@{-}[rrrr]&*{}&*{}&*{}&*{}\ar@/^1pc/@{.}[dddd]|{U}&*{}\\
\ar@{-}[rdd]&*{}\ar@{-}[rrr]&*{}&*{}
&*{}\ar@{-}[rd]&*{}\ar@{-}[ru]&*{}\ar@{-}[ruuu]&*{}\ar@{-}[rrrr]&*{}&*{}&*{}&*{}\ar@{-}[ru]&*{}\ar@{-}[rrrrr]&*{}&*{}&*{}&*{}&*{}\ar@{-}[rddd]&*{}\ar@{-}[ru]&*{}\ar@{-}[rd]&*{}\ar@{-}[ru]&*{}\ar@{-}[rd]&*{}\ar@{-}[rrr]&*{}&*{}&*{}&*{}\\
\ar@{-}[r]&*{}\ar@{-}[rr]&*{}&*{}\ar@{-}[rd]&*{}\ar@{-}[ru]&*{}\ar@{-}[rd]&*{}\ar@{-}[ruuu]&*{}
\ar@{-}[rdd]&*{}\ar@{-}[rr]&*{}&*{}\ar@{-}[rd]&*{}\ar@{-}[r]&*{}\ar@{-}[rd]&*{}\ar@{-}[r]&*{}\ar@{-}[rd]&*{}\ar@{-}[rr]&*{}&*{}\ar@{-}[ruuuuu]&*{}\ar@{-}[rd]&*{}\ar@{-}[ru]&*{}\ar@{-}[rd]&*{}\ar@{-}[ru]&*{}\ar@{-}[rd]&*{}\ar@{-}[rr]&*{}&*{}&*{}\\
\ar@{-}[ruu]&*{}\ar@{-}[r]&*{}\ar@{-}[rd]&*{}\ar@{-}[ru]&*{}\ar@{-}[rd]&*{}\ar@{-}[ru]&*{}\ar@{-}[ruuu]&*{}\ar@{-}[r]&*{}\ar@{-}@{-}[r]&*{}\ar@{-}[rd]&*{}\ar@{-}[ru]&*{}\ar@{-}[rd]&*{}\ar@{-}[ru]&*{}\ar@{-}[rd]&*{}\ar@{-}[ru]&*{}\ar@{-}[rd]&*{}\ar@{-}[r]&*{}\ar@{-}[ruuuuu]&*{}\ar@{-}[ru]&*{}\ar@{-}[rd]&*{}\ar@{-}[ru]&*{}\ar@{-}[rd]&*{}\ar@{-}[ru]&*{}\ar@{-}[rd]&*{}\ar@{-}[r]&*{}&*{}\\
\ar@{-}[ruuuu]&*{}\ar@{-}[rd]&*{}\ar@{-}[ru]&*{}\ar@{-}[rd]&*{}\ar@{-}[ru]&*{}\ar@{-}[rd]&*{}\ar@{-}[ruuu]&*{}
\ar@{-}[ruu]&*{}\ar@{-}[rd]&*{}\ar@{-}[ru]&*{}\ar@{-}[rd]&*{}\ar@{-}[ru]&*{}\ar@{-}[rd]&*{}\ar@{-}[ru]&*{}\ar@{-}[rd]&*{}\ar@{-}[ru]&*{}\ar@{-}[rd]&*{}\ar@{-}[ruuuuu]&*{}\ar@{-}[rd]&*{}\ar@{-}[ru]&*{}\ar@{-}[rd]&*{}\ar@{-}[ru]&*{}\ar@{-}[rd]&*{}\ar@{-}[ru]&*{}\ar@{-}[rd]&*{}&*{}\\
\ar@{-}[r]  &*{}  \ar@{-}[ru]&*{}   &*{}\ar@{-}[ru]&*{}
&*{}\ar@{-}[ru]&*{}\ar@{-}[r]   &*{}
\ar@{-}[r]&*{}\ar@{-}[ru]&*{}&*{}\ar@{-}[ru]&*{}&*{}\ar@{-}[ru]&*{}&*{}\ar@{-}[ru]&*{}&*{}\ar@{-}[ru]&*{}\ar@{-}[r]&*{}\ar@{-}[ru]&*{}&*{}\ar@{-}[ru]&*{}&*{}\ar@{-}[ru]&*{}&*{}\ar@{-}[ru]&*{}&*{}\\
\ar@{-}[rr]  &*{} &*{}\ar@{}[ru]|{\fbox{\rule[-.6cm]{0cm}{1cm}$f_1$}} &*{}\ar@{-}[r]
 &*{} \ar@{}[ru]|{\fbox{\rule[-.6cm]{0cm}{1cm}$f_2$}} &*{}\ar@{-}[r]&*{}
\ar@{-}[r]&*{}\ar@{-}[rr]&*{}\ar@{-}[r]&*{}\ar@{-}[ru]|{\fbox{\rule[-.6cm]{0cm}{1cm}$g_1$}}&*{}\ar@{-}[r]&*{}\ar@{-}[ru]|{\fbox{\rule[-.6cm]{0cm}{1cm}$g_2$}}&*{}\ar@{-}[r]&*{}\ar@{-}[ru]|{\fbox{\rule[-.6cm]{0cm}{1cm}$g_3$}}&*{}\ar@{-}[r]&*{}\ar@{-}[ru]|{\fbox{\rule[-.6cm]{0cm}{1cm}$g_4$}}&*{}\ar@{-}[r]&*{}\ar@{-}[r]&*{}\ar@{-}[r]&*{}\ar@{-}[ru]|{\fbox{\rule[-.6cm]{0cm}{1cm}$h_1$}}&*{}\ar@{-}[r]&*{}\ar@{-}[ru]|{\fbox{\rule[-.6cm]{0cm}{1cm}$h_2$}}&*{}\ar@{-}[r]&*{}\ar@{-}[ru]|{\fbox{\rule[-.6cm]{0cm}{1cm}$h_3$}}&*{}\ar@{-}[r]&*{}&*{}
}
\end{equation}

Let us call this map $\alpha$ (without the dotted lines). Notice that since $[\vec{g}]\downarrow$ and after applying superposing and the naturality axioms we have that $\alpha\in \Trc^V$. This turns out to be the general condition that we need in order to use the Vanishing II axiom, i.e.,
if we consider $\alpha\in \Trc^V$ as a general hypothesis then the equivalence
$$\alpha\in\Trc^{U\otimes V}\,\, \Leftrightarrow \,\,\,\Tr^V(\alpha)\in \Trc^{U}$$
is precisely the condition required to apply the Vanishing II axiom in which the condition $[\vec{f},[\vec{g}],\vec{h}]$ translates into $\Tr^V(\alpha)\in \Trc^U$ and $[\vec{f},\vec{g},\vec{h}]$ into $\alpha\in\Trc^{U\otimes V}$.  Thus we can replace the previous diagram by the next one:

\[
\xymatrix@R=10pt@C=16pt{
\ar@/_1pc/@{.}[ddddddd]|{U\otimes
V}\ar@{-}[rrrrrr]&*{}&*{}&*{}&*{}&*{}&*{}\ar@{-}[rddddd]&*{}\ar@{-}[rrrrrrrrrr]&*{}&*{}&*{}&*{}&*{}&*{}&*{}&*{}&*{}&*{}\ar@{-}[rddd]&*{}\ar@{-}[rrrrrrr]&*{}&*{}&*{}&*{}&*{}&*{}&*{}\ar@/^1pc/@{.}[ddddddd]|{U\otimes
V}&*{}\\
\ar@{-}[rrrrrr]&*{}&*{}&*{}&*{}&*{}&*{}\ar@{-}[rddddd]&*{}\ar@{-}[rrrr]&*{}&*{}&*{}&*{}\ar@{-}[rd]&*{}\ar@{-}[rrrrr]&*{}&*{}&*{}&*{}&*{}\ar@{-}[rddd]&*{}\ar@{-}[rrrrrrr]&*{}&*{}&*{}&*{}&*{}&*{}&*{}&*{}\\
\ar@{-}[rrrrrr]&*{}&*{}&*{}&*{}&*{}&*{}\ar@{-}[rddddd]&*{}\ar@{-}[rrrr]&*{}&*{}&*{}&*{}\ar@{-}[ru]&*{}\ar@{-}[rrrrr]&*{}&*{}&*{}&*{}&*{}\ar@{-}[rddd]&*{}\ar@{-}[rrrrrrr]&*{}&*{}&*{}&*{}&*{}&*{}&*{}\\
\ar@{-}[rdddd]&*{}\ar@{-}[rrrr]&*{}&*{} &*{} &*{}\ar@{-}[rd]&*{}\ar@{-}[ruuu] &*{}
\ar@{-}[rrrr]&*{}&*{}&*{}&*{}\ar@{-}[rd]&*{}\ar@{-}[rrrrr]&*{}&*{}&*{}&*{}&*{}\ar@{-}[rddd]&*{}\ar@{-}[rd]&*{}\ar@{-}[r]&*{}\ar@{-}[rd]&*{}\ar@{-}[rrrr]&*{}&*{}&*{}&*{}&*{}\\
\ar@{-}[rdd]&*{}\ar@{-}[rrr]&*{}&*{}
&*{}\ar@{-}[rd]&*{}\ar@{-}[ru]&*{}\ar@{-}[ruuu]&*{}\ar@{-}[rrrr]&*{}&*{}&*{}&*{}\ar@{-}[ru]&*{}\ar@{-}[rrrrr]&*{}&*{}&*{}&*{}&*{}\ar@{-}[rddd]&*{}\ar@{-}[ru]&*{}\ar@{-}[rd]&*{}\ar@{-}[ru]&*{}\ar@{-}[rd]&*{}\ar@{-}[rrr]&*{}&*{}&*{}&*{}\\
\ar@{-}[r]&*{}\ar@{-}[rr]&*{}&*{}\ar@{-}[rd]&*{}\ar@{-}[ru]&*{}\ar@{-}[rd]&*{}\ar@{-}[ruuu]&*{}
\ar@{-}[rdd]&*{}\ar@{-}[rr]&*{}&*{}\ar@{-}[rd]&*{}\ar@{-}[r]&*{}\ar@{-}[rd]&*{}\ar@{-}[r]&*{}\ar@{-}[rd]&*{}\ar@{-}[rr]&*{}&*{}\ar@{-}[ruuuuu]&*{}\ar@{-}[rd]&*{}\ar@{-}[ru]&*{}\ar@{-}[rd]&*{}\ar@{-}[ru]&*{}\ar@{-}[rd]&*{}\ar@{-}[rr]&*{}&*{}&*{}\\
\ar@{-}[ruu]&*{}\ar@{-}[r]&*{}\ar@{-}[rd]&*{}\ar@{-}[ru]&*{}\ar@{-}[rd]&*{}\ar@{-}[ru]&*{}\ar@{-}[ruuu]&*{}\ar@{-}[r]&*{}\ar@{-}@{-}[r]&*{}\ar@{-}[rd]&*{}\ar@{-}[ru]&*{}\ar@{-}[rd]&*{}\ar@{-}[ru]&*{}\ar@{-}[rd]&*{}\ar@{-}[ru]&*{}\ar@{-}[rd]&*{}\ar@{-}[r]&*{}\ar@{-}[ruuuuu]&*{}\ar@{-}[ru]&*{}\ar@{-}[rd]&*{}\ar@{-}[ru]&*{}\ar@{-}[rd]&*{}\ar@{-}[ru]&*{}\ar@{-}[rd]&*{}\ar@{-}[r]&*{}&*{}\\
\ar@{-}[ruuuu]&*{}\ar@{-}[rd]&*{}\ar@{-}[ru]&*{}\ar@{-}[rd]&*{}\ar@{-}[ru]&*{}\ar@{-}[rd]&*{}\ar@{-}[ruuu]&*{}
\ar@{-}[ruu]&*{}\ar@{-}[rd]&*{}\ar@{-}[ru]&*{}\ar@{-}[rd]&*{}\ar@{-}[ru]&*{}\ar@{-}[rd]&*{}\ar@{-}[ru]&*{}\ar@{-}[rd]&*{}\ar@{-}[ru]&*{}\ar@{-}[rd]&*{}\ar@{-}[ruuuuu]&*{}\ar@{-}[rd]&*{}\ar@{-}[ru]&*{}\ar@{-}[rd]&*{}\ar@{-}[ru]&*{}\ar@{-}[rd]&*{}\ar@{-}[ru]&*{}\ar@{-}[rd]&*{}&*{}\\
\ar@{-}[r]  &*{}  \ar@{-}[ru]&*{}   &*{}\ar@{-}[ru]&*{}
&*{}\ar@{-}[ru]&*{}\ar@{-}[r]   &*{}
\ar@{-}[r]&*{}\ar@{-}[ru]&*{}&*{}\ar@{-}[ru]&*{}&*{}\ar@{-}[ru]&*{}&*{}\ar@{-}[ru]&*{}&*{}\ar@{-}[ru]&*{}\ar@{-}[r]&*{}\ar@{-}[ru]&*{}&*{}\ar@{-}[ru]&*{}&*{}\ar@{-}[ru]&*{}&*{}\ar@{-}[ru]&*{}&*{}\\
\ar@{-}[rr]  &*{} &*{}\ar@{}[ru]|{\fbox{\rule[-.6cm]{0cm}{1cm}$f_1$}} &*{}\ar@{-}[r]
 &*{} \ar@{}[ru]|{\fbox{\rule[-.6cm]{0cm}{1cm}$f_2$}} &*{}\ar@{-}[r]&*{}
\ar@{-}[r]&*{}\ar@{-}[rr]&*{}\ar@{-}[r]&*{}\ar@{-}[ru]|{\fbox{\rule[-.6cm]{0cm}{1cm}$g_1$}}&*{}\ar@{-}[r]&*{}\ar@{-}[ru]|{\fbox{\rule[-.6cm]{0cm}{1cm}$g_2$}}&*{}\ar@{-}[r]&*{}\ar@{-}[ru]|{\fbox{\rule[-.6cm]{0cm}{1cm}$g_3$}}&*{}\ar@{-}[r]&*{}\ar@{-}[ru]|{\fbox{\rule[-.6cm]{0cm}{1cm}$g_4$}}&*{}\ar@{-}[r]&*{}\ar@{-}[r]&*{}\ar@{-}[r]&*{}\ar@{-}[ru]|{\fbox{\rule[-.6cm]{0cm}{1cm}$h_1$}}&*{}\ar@{-}[r]&*{}\ar@{-}[ru]|{\fbox{\rule[-.6cm]{0cm}{1cm}$h_2$}}&*{}\ar@{-}[r]&*{}\ar@{-}[ru]|{\fbox{\rule[-.6cm]{0cm}{1cm}$h_3$}}&*{}\ar@{-}[r]&*{}&*{}
}
\]

By coherence we can replace this part of the diagram:

\[\xymatrix@R=10pt@C=16pt{
\ar@{-}[rddddd]&*{}\ar@{-}[r]&*{}\\
\ar@{-}[rddddd]&*{}\ar@{-}[r]&*{}\\
\ar@{-}[rddddd]&*{}\ar@{-}[r]&*{}\\
\ar@{-}[ruuu]&*{}\ar@{-}[r]&*{}\\
\ar@{-}[ruuu]&*{}\ar@{-}[r]&*{}\\
\ar@{-}[ruuu]&*{}\ar@{-}[rdd]&*{}\\
\ar@{-}[ruuu]&*{}\ar@{-}[r]&*{}\\
\ar@{-}[ruuu]&*{}\ar@{-}[ruu]&*{}\\
\ar@{-}[r]&*{}\ar@{-}[r]&*{}\\
\ar@{-}[r]&*{}\ar@{-}[r]&*{}
}\]

by this one\\

\[\xymatrix@R=10pt@C=16pt{
\ar@{-}[rdd]&*{}\ar@{-}[rddddd]&*{}\\
\ar@{-}[r]&*{}\ar@{-}[rddddd]&*{}\\
\ar@{-}[ruu]&*{}\ar@{-}[rddddd]&*{}\\
\ar@{-}[r]&*{}\ar@{-}[ruuu]&*{}\\
\ar@{-}[r]&*{}\ar@{-}[ruuu]&*{}\\
\ar@{-}[r]&*{}\ar@{-}[ruuu]&*{}\\
\ar@{-}[r]&*{}\ar@{-}[ruuu]&*{}\\
\ar@{-}[r]&*{}\ar@{-}[ruuu]&*{}\\
\ar@{-}[r]&*{}\ar@{-}[r]&*{}\\
\ar@{-}[r]&*{}\ar@{-}[r]&*{}
}\]

So, by this substitution and functoriality we get:\\

\[
\xymatrix@R=10pt@C=16pt{
\ar@{-}[rdd]&*{}\ar@{-}[rrrrr]&*{}&*{}&*{}&*{}&*{}\ar@{-}[rddddd]&*{}\ar@{-}[rrrrrrrrrr]&*{}&*{}&*{}&*{}&*{}&*{}&*{}&*{}&*{}&*{}\ar@{-}[rddd]&*{}\ar@{-}[rrrrrrr]&*{}&*{}&*{}&*{}&*{}&*{}&*{}&*{}\\
\ar@{-}[r]&*{}\ar@{-}[rrrrr]&*{}&*{}&*{}&*{}&*{}\ar@{-}[rddddd]&*{}\ar@{-}[rrrr]&*{}&*{}&*{}&*{}\ar@{-}[rd]&*{}\ar@{-}[rrrrr]&*{}&*{}&*{}&*{}&*{}\ar@{-}[rddd]&*{}\ar@{-}[rrrrrrr]&*{}&*{}&*{}&*{}&*{}&*{}&*{}&*{}\\
\ar@{-}[ruu]&*{}\ar@{-}[rrrrr]&*{}&*{}&*{}&*{}&*{}\ar@{-}[rddddd]&*{}\ar@{-}[rrrr]&*{}&*{}&*{}&*{}\ar@{-}[ru]&*{}\ar@{-}[rrrrr]&*{}&*{}&*{}&*{}&*{}\ar@{-}[rddd]&*{}\ar@{-}[rrrrrrr]&*{}&*{}&*{}&*{}&*{}&*{}&*{}\\
\ar@{-}[rdddd]&*{}\ar@{-}[rrrr]&*{}&*{} &*{} &*{}\ar@{-}[rd]&*{}\ar@{-}[ruuu] &*{}
\ar@{-}[rrrr]&*{}&*{}&*{}&*{}\ar@{-}[rd]&*{}\ar@{-}[rrrrr]&*{}&*{}&*{}&*{}&*{}\ar@{-}[rddd]&*{}\ar@{-}[rd]&*{}\ar@{-}[r]&*{}\ar@{-}[rd]&*{}\ar@{-}[rrrr]&*{}&*{}&*{}&*{}&*{}\\
\ar@{-}[rdd]&*{}\ar@{-}[rrr]&*{}&*{}
&*{}\ar@{-}[rd]&*{}\ar@{-}[ru]&*{}\ar@{-}[ruuu]&*{}\ar@{-}[rrrr]&*{}&*{}&*{}&*{}\ar@{-}[ru]&*{}\ar@{-}[rrrrr]&*{}&*{}&*{}&*{}&*{}\ar@{-}[rddd]&*{}\ar@{-}[ru]&*{}\ar@{-}[rd]&*{}\ar@{-}[ru]&*{}\ar@{-}[rd]&*{}\ar@{-}[rrr]&*{}&*{}&*{}&*{}\\
\ar@{-}[r]&*{}\ar@{-}[rr]&*{}&*{}\ar@{-}[rd]&*{}\ar@{-}[ru]&*{}\ar@{-}[rd]&*{}\ar@{-}[ruuu]&*{}
\ar@{-}[r]&*{}\ar@{-}[rr]&*{}&*{}\ar@{-}[rd]&*{}\ar@{-}[r]&*{}\ar@{-}[rd]&*{}\ar@{-}[r]&*{}\ar@{-}[rd]&*{}\ar@{-}[rr]&*{}&*{}\ar@{-}[ruuuuu]&*{}\ar@{-}[rd]&*{}\ar@{-}[ru]&*{}\ar@{-}[rd]&*{}\ar@{-}[ru]&*{}\ar@{-}[rd]&*{}\ar@{-}[rr]&*{}&*{}&*{}\\
\ar@{-}[ruu]&*{}\ar@{-}[r]&*{}\ar@{-}[rd]&*{}\ar@{-}[ru]&*{}\ar@{-}[rd]&*{}\ar@{-}[ru]&*{}\ar@{-}[ruuu]&*{}\ar@{-}[r]&*{}\ar@{-}@{-}[r]&*{}\ar@{-}[rd]&*{}\ar@{-}[ru]&*{}\ar@{-}[rd]&*{}\ar@{-}[ru]&*{}\ar@{-}[rd]&*{}\ar@{-}[ru]&*{}\ar@{-}[rd]&*{}\ar@{-}[r]&*{}\ar@{-}[ruuuuu]&*{}\ar@{-}[ru]&*{}\ar@{-}[rd]&*{}\ar@{-}[ru]&*{}\ar@{-}[rd]&*{}\ar@{-}[ru]&*{}\ar@{-}[rd]&*{}\ar@{-}[r]&*{}&*{}\\
\ar@{-}[ruuuu]&*{}\ar@{-}[rd]&*{}\ar@{-}[ru]&*{}\ar@{-}[rd]&*{}\ar@{-}[ru]&*{}\ar@{-}[rd]&*{}\ar@{-}[ruuu]&*{}
\ar@{-}[r]&*{}\ar@{-}[rd]&*{}\ar@{-}[ru]&*{}\ar@{-}[rd]&*{}\ar@{-}[ru]&*{}\ar@{-}[rd]&*{}\ar@{-}[ru]&*{}\ar@{-}[rd]&*{}\ar@{-}[ru]&*{}\ar@{-}[rd]&*{}\ar@{-}[ruuuuu]&*{}\ar@{-}[rd]&*{}\ar@{-}[ru]&*{}\ar@{-}[rd]&*{}\ar@{-}[ru]&*{}\ar@{-}[rd]&*{}\ar@{-}[ru]&*{}\ar@{-}[rd]&*{}&*{}\\
\ar@{-}[r]  &*{}  \ar@{-}[ru]&*{}   &*{}\ar@{-}[ru]&*{}
&*{}\ar@{-}[ru]&*{}\ar@{-}[r]   &*{}
\ar@{-}[r]&*{}\ar@{-}[ru]&*{}&*{}\ar@{-}[ru]&*{}&*{}\ar@{-}[ru]&*{}&*{}\ar@{-}[ru]&*{}&*{}\ar@{-}[ru]&*{}\ar@{-}[r]&*{}\ar@{-}[ru]&*{}&*{}\ar@{-}[ru]&*{}&*{}\ar@{-}[ru]&*{}&*{}\ar@{-}[ru]&*{}&*{}\\
\ar@{-}[rr]  &*{} &*{}\ar@{}[ru]|{\fbox{\rule[-.6cm]{0cm}{1cm}$f_1$}} &*{}\ar@{-}[r]
 &*{} \ar@{}[ru]|{\fbox{\rule[-.6cm]{0cm}{1cm}$f_2$}} &*{}\ar@{-}[r]&*{}
\ar@{-}[r]&*{}\ar@{-}[rr]&*{}\ar@{-}[r]&*{}\ar@{-}[ru]|{\fbox{\rule[-.6cm]{0cm}{1cm}$g_1$}}&*{}\ar@{-}[r]&*{}\ar@{-}[ru]|{\fbox{\rule[-.6cm]{0cm}{1cm}$g_2$}}&*{}\ar@{-}[r]&*{}\ar@{-}[ru]|{\fbox{\rule[-.6cm]{0cm}{1cm}$g_3$}}&*{}\ar@{-}[r]&*{}\ar@{-}[ru]|{\fbox{\rule[-.6cm]{0cm}{1cm}$g_4$}}&*{}\ar@{-}[r]&*{}\ar@{-}[r]&*{}\ar@{-}[r]&*{}\ar@{-}[ru]|{\fbox{\rule[-.6cm]{0cm}{1cm}$h_1$}}&*{}\ar@{-}[r]&*{}\ar@{-}[ru]|{\fbox{\rule[-.6cm]{0cm}{1cm}$h_2$}}&*{}\ar@{-}[r]&*{}\ar@{-}[ru]|{\fbox{\rule[-.6cm]{0cm}{1cm}$h_3$}}&*{}\ar@{-}[r]&*{}&*{}
}
\]

From now, we are going to permute the objects that are traced in order to get the formula $[\vec{f},\vec{g},\vec{h}]$. The dinaturality axiom allows us to commute the objects that are traced by composing with a permutation and pre-composing with its inverse. For that purpose we define a permutation which will impose an order at the level of objects in such a way that creates a sequence where the objects that are connected to $\vec{g}$ follow the objects connected to $\vec{f}$ and the objects of $\vec{h}$ follow the objects of $\vec{g}$:

\begin{center}
$\tau:A_1\otimes A_2\dots\otimes A_n\otimes B_m\otimes C_1\otimes\dots\otimes C_{s-1}\otimes B_1\otimes\dots\otimes B_{m-1}\rightarrow A_1\otimes A_2\dots\otimes A_n\otimes B_1\otimes\dots\otimes B_{m-1} B_m\otimes C_1\otimes\dots\otimes C_{s-1}.$
\end{center}
Also, by definition of our product we have the permutations associated with $[\vec{f},[\vec{g}],\vec{h}]$:
$$\gamma':C_{s-1}\otimes C_{s-2}\otimes\dots C_2\otimes C_1\otimes B_m\otimes A_n\dots A_2\otimes A_1\rightarrow
A_1\otimes A_2\dots A_n\otimes B_m\otimes C_1\otimes C_2\dots C_{s-1}$$
and with $[\vec{g}]$:
$$\gamma'':B_{m-1}\otimes B_{m-2}\otimes\dots\otimes B_2\otimes B_1\rightarrow B_1\otimes B_2\otimes\dots\otimes B_{m-2}\otimes B_{m-1}$$
and with $[\vec{f},\vec{g},\vec{h}]$:
$$C_{s-1}\otimes C_{s-2}\otimes\dots C_1\otimes B_m\otimes B_{m-1}\otimes\dots B_1\otimes A_n\otimes A_{n-1}\dots\otimes A_1\stackrel{\gamma}\rightarrow A_1\otimes\dots A_n\otimes B_1\dots\otimes B_{m}\otimes C_1\dots C_{s-1}$$

As we said, we want to compose and pre-compose with a permutation, let us call it $y$, for our purpose this permutation should satisfy:
$$ y;(\gamma'\otimes\gamma'')=\gamma;\tau^{-1}.$$
Thus, since all this map are invertible we define:
$$ y=\gamma;\tau^{-1};(\gamma'^{-1}\otimes\gamma''^{-1}).$$

In our concrete graphical description after applying dinaturality we get:

\[
\xymatrix@R=10pt@C=16pt{
\ar@{}[rddddddd]|{y}\ar@{.}[r]\ar@{.}[ddddddd]&*{}\ar@/^2pc/@{.}[r]|{\gamma'\otimes
\gamma''}\ar@{.}[ddddddd]\ar@{-}[rdd]&*{}\ar@{-}[rrrrr]&*{}&*{}&*{}&*{}&*{}\ar@{-}[rddddd]&*{}\ar@{-}[rrrrrrrrrr]&*{}&*{}&*{}&*{}&*{}&*{}&*{}&*{}&*{}&*{}\ar@{-}[rddd]&*{}\ar@{-}[rrrrrrr]&*{}&*{}&*{}&*{}&*{}&*{}&*{}\ar@{.}[ddddddd]\ar@{.}[r]&*{}\ar@{.}[ddddddd]\ar@{}[lddddddd]|{y^{-1}}\\
&*{}\ar@{-}[r]&*{}\ar@{-}[rrrrr]&*{}&*{}&*{}&*{}&*{}\ar@{-}[rddddd]&*{}\ar@{-}[rrrr]&*{}&*{}&*{}&*{}\ar@{-}[rd]&*{}\ar@{-}[rrrrr]&*{}&*{}&*{}&*{}&*{}\ar@{-}[rddd]&*{}\ar@{-}[rrrrrrr]&*{}&*{}&*{}&*{}&*{}&*{}&*{}&*{}\\
&*{}\ar@{-}[ruu]&*{}\ar@{-}[rrrrr]&*{}&*{}&*{}&*{}&*{}\ar@{-}[rddddd]&*{}\ar@{-}[rrrr]&*{}&*{}&*{}&*{}\ar@{-}[ru]&*{}\ar@{-}[rrrrr]&*{}&*{}&*{}&*{}&*{}\ar@{-}[rddd]&*{}\ar@{-}[rrrrrrr]&*{}&*{}&*{}&*{}&*{}&*{}&*{}\\
&*{}\ar@{-}[rdddd]&*{}\ar@{-}[rrrr]&*{}&*{} &*{} &*{}\ar@{-}[rd]&*{}\ar@{-}[ruuu]
&*{}
\ar@{-}[rrrr]&*{}&*{}&*{}&*{}\ar@{-}[rd]&*{}\ar@{-}[rrrrr]&*{}&*{}&*{}&*{}&*{}\ar@{-}[rddd]&*{}\ar@{-}[rd]&*{}\ar@{-}[r]&*{}\ar@{-}[rd]&*{}\ar@{-}[rrrr]&*{}&*{}&*{}&*{}&*{}\\
&*{}\ar@{-}[rdd]&*{}\ar@{-}[rrr]&*{}&*{}
&*{}\ar@{-}[rd]&*{}\ar@{-}[ru]&*{}\ar@{-}[ruuu]&*{}\ar@{-}[rrrr]&*{}&*{}&*{}&*{}\ar@{-}[ru]&*{}\ar@{-}[rrrrr]&*{}&*{}&*{}&*{}&*{}\ar@{-}[rddd]&*{}\ar@{-}[ru]&*{}\ar@{-}[rd]&*{}\ar@{-}[ru]&*{}\ar@{-}[rd]&*{}\ar@{-}[rrr]&*{}&*{}&*{}&*{}\\
&*{}\ar@{-}[r]&*{}\ar@{-}[rr]&*{}&*{}\ar@{-}[rd]&*{}\ar@{-}[ru]&*{}\ar@{-}[rd]&*{}\ar@{-}[ruuu]&*{}
\ar@{-}[r]&*{}\ar@{-}[rr]&*{}&*{}\ar@{-}[rd]&*{}\ar@{-}[r]&*{}\ar@{-}[rd]&*{}\ar@{-}[r]&*{}\ar@{-}[rd]&*{}\ar@{-}[rr]&*{}&*{}\ar@{-}[ruuuuu]&*{}\ar@{-}[rd]&*{}\ar@{-}[ru]&*{}\ar@{-}[rd]&*{}\ar@{-}[ru]&*{}\ar@{-}[rd]&*{}\ar@{-}[rr]&*{}&*{}&*{}\\
&*{}\ar@{-}[ruu]&*{}\ar@{-}[r]&*{}\ar@{-}[rd]&*{}\ar@{-}[ru]&*{}\ar@{-}[rd]&*{}\ar@{-}[ru]&*{}\ar@{-}[ruuu]&*{}\ar@{-}[r]&*{}\ar@{-}@{-}[r]&*{}\ar@{-}[rd]&*{}\ar@{-}[ru]&*{}\ar@{-}[rd]&*{}\ar@{-}[ru]&*{}\ar@{-}[rd]&*{}\ar@{-}[ru]&*{}\ar@{-}[rd]&*{}\ar@{-}[r]&*{}\ar@{-}[ruuuuu]&*{}\ar@{-}[ru]&*{}\ar@{-}[rd]&*{}\ar@{-}[ru]&*{}\ar@{-}[rd]&*{}\ar@{-}[ru]&*{}\ar@{-}[rd]&*{}\ar@{-}[r]&*{}&*{}\\
\ar@{.}[r]&*{}\ar@{-}[ruuuu]&*{}\ar@{-}[rd]&*{}\ar@{-}[ru]&*{}\ar@{-}[rd]&*{}\ar@{-}[ru]&*{}\ar@{-}[rd]&*{}\ar@{-}[ruuu]&*{}
\ar@{-}[r]&*{}\ar@{-}[rd]&*{}\ar@{-}[ru]&*{}\ar@{-}[rd]&*{}\ar@{-}[ru]&*{}\ar@{-}[rd]&*{}\ar@{-}[ru]&*{}\ar@{-}[rd]&*{}\ar@{-}[ru]&*{}\ar@{-}[rd]&*{}\ar@{-}[ruuuuu]&*{}\ar@{-}[rd]&*{}\ar@{-}[ru]&*{}\ar@{-}[rd]&*{}\ar@{-}[ru]&*{}\ar@{-}[rd]&*{}\ar@{-}[ru]&*{}\ar@{-}[rd]&*{}\ar@{.}[r]&*{}\\
\ar@{.}[r]&*{}\ar@{-}[r]  &*{}  \ar@{-}[ru]&*{}   &*{}\ar@{-}[ru]&*{}
&*{}\ar@{-}[ru]&*{}\ar@{-}[r]   &*{}
\ar@{-}[r]&*{}\ar@{-}[ru]&*{}&*{}\ar@{-}[ru]&*{}&*{}\ar@{-}[ru]&*{}&*{}\ar@{-}[ru]&*{}&*{}\ar@{-}[ru]&*{}\ar@{-}[r]&*{}\ar@{-}[ru]&*{}&*{}\ar@{-}[ru]&*{}&*{}\ar@{-}[ru]&*{}&*{}\ar@{-}[ru]&*{}\ar@{.}[r]&*{}\\
\ar@{.}[r]&*{}\ar@{-}[rr]  &*{} &*{}\ar@{}[ru]|{\fbox{\rule[-.6cm]{0cm}{1cm}$f_1$}}
&*{}\ar@{-}[r]  &*{} \ar@{}[ru]|{\fbox{\rule[-.6cm]{0cm}{1cm}$f_2$}}
&*{}\ar@{-}[r]&*{}
\ar@{-}[r]&*{}\ar@{-}[rr]&*{}\ar@{-}[r]&*{}\ar@{-}[ru]|{\fbox{\rule[-.6cm]{0cm}{1cm}$g_1$}}&*{}\ar@{-}[r]&*{}\ar@{-}[ru]|{\fbox{\rule[-.6cm]{0cm}{1cm}$g_2$}}&*{}\ar@{-}[r]&*{}\ar@{-}[ru]|{\fbox{\rule[-.6cm]{0cm}{1cm}$g_3$}}&*{}\ar@{-}[r]&*{}\ar@{-}[ru]|{\fbox{\rule[-.6cm]{0cm}{1cm}$g_4$}}&*{}\ar@{-}[r]&*{}\ar@{-}[r]&*{}\ar@{-}[r]&*{}\ar@{-}[ru]|{\fbox{\rule[-.6cm]{0cm}{1cm}$h_1$}}&*{}\ar@{-}[r]&*{}\ar@{-}[ru]|{\fbox{\rule[-.6cm]{0cm}{1cm}$h_2$}}&*{}\ar@{-}[r]&*{}\ar@{-}[ru]|{\fbox{\rule[-.6cm]{0cm}{1cm}$h_3$}}&*{}\ar@{-}[r]&*{}\ar@{.}[r]&*{}
}
\]
and using the equation that we defined above
$y;(\gamma'\otimes\gamma'')=\gamma;\tau^{-1}$ we replace it and we obtain:\\

\[
\xymatrix@R=10pt@C=16pt{
\ar@/^2pc/@{.}[r]|{\gamma}\ar@{.}[ddddddddd]\ar@{.}[rr]\ar@{-}[rddddddd]&*{}\ar@/^2pc/@{.}[r]|{\tau^{-1}}\ar@{-}[rddd]&*{}\ar@{.}[ddddddddd]\ar@{-}[rrrr]&*{}&*{}&*{}&*{}\ar@{-}[rddddd]&*{}\ar@{-}[rrrrrrrrrr]&*{}&*{}&*{}&*{}&*{}&*{}&*{}&*{}&*{}&*{}\ar@{-}[rddd]&*{}\ar@{-}[rrrrrrr]&*{}&*{}&*{}&*{}&*{}&*{}&*{}\ar@{.}[ddddddddd]\ar@{.}[r]\ar@{-}[rdd]&*{}\ar@{.}[ddddddddd]&*{}\\
\ar@{-}[rddddd]&*{}\ar@{-}[rddd]&*{}\ar@{-}[rrrr]&*{}&*{}&*{}&*{}\ar@{-}[rddddd]&*{}\ar@{-}[rrrr]&*{}&*{}&*{}&*{}\ar@{-}[rd]&*{}\ar@{-}[rrrrr]&*{}&*{}&*{}&*{}&*{}\ar@{-}[rddd]&*{}\ar@{-}[rrrrrrr]&*{}&*{}&*{}&*{}&*{}&*{}&*{}\ar@{-}[rdd]&*{}&*{}\\
\ar@{-}[rddd]&*{}\ar@{-}[rddd]&*{}\ar@{-}[rrrr]&*{}&*{}&*{}&*{}\ar@{-}[rddddd]&*{}\ar@{-}[rrrr]&*{}&*{}&*{}&*{}\ar@{-}[ru]&*{}\ar@{-}[rrrrr]&*{}&*{}&*{}&*{}&*{}\ar@{-}[rddd]&*{}\ar@{-}[rrrrrrr]&*{}&*{}&*{}&*{}&*{}&*{}&*{}\ar@{-}[rdd]&*{}\\
\ar@{-}[rd]&*{}\ar@{-}[ruuu]&*{}\ar@{-}[rrr]&*{} &*{}
&*{}\ar@{-}[rd]&*{}\ar@{-}[ruuu]&*{}
\ar@{-}[rrrr]&*{}&*{}&*{}&*{}\ar@{-}[rd]&*{}\ar@{-}[rrrrr]&*{}&*{}&*{}&*{}&*{}\ar@{-}[rddd]&*{}\ar@{-}[rd]&*{}\ar@{-}[r]&*{}\ar@{-}[rd]&*{}\ar@{-}[rrrr]&*{}&*{}&*{}&*{}\ar@{-}[ruuu]&*{}&*{}\\
\ar@{-}[ru]&*{}\ar@{-}[ruuu]&*{}\ar@{-}[rr]&*{}
&*{}\ar@{-}[rd]&*{}\ar@{-}[ru]&*{}\ar@{-}[ruuu]&*{}\ar@{-}[rrrr]&*{}&*{}&*{}&*{}\ar@{-}[ru]&*{}\ar@{-}[rrrrr]&*{}&*{}&*{}&*{}&*{}\ar@{-}[rddd]&*{}\ar@{-}[ru]&*{}\ar@{-}[rd]&*{}\ar@{-}[ru]&*{}\ar@{-}[rd]&*{}\ar@{-}[rrr]&*{}&*{}&*{}\ar@{-}[ruuu]&*{}&*{}\\
\ar@{-}[ruuu]&*{}\ar@{-}[ruuu]&*{}\ar@{-}[r]&*{}\ar@{-}[rd]&*{}\ar@{-}[ru]&*{}\ar@{-}[rd]&*{}\ar@{-}[ruuu]&*{}
\ar@{-}[r]&*{}\ar@{-}[rr]&*{}&*{}\ar@{-}[rd]&*{}\ar@{-}[r]&*{}\ar@{-}[rd]&*{}\ar@{-}[r]&*{}\ar@{-}[rd]&*{}\ar@{-}[rr]&*{}&*{}\ar@{-}[ruuuuu]&*{}\ar@{-}[rd]&*{}\ar@{-}[ru]&*{}\ar@{-}[rd]&*{}\ar@{-}[ru]&*{}\ar@{-}[rd]&*{}\ar@{-}[rr]&*{}\ar@{-}[r]&*{}\ar@{-}[r]&*{}&*{}\\
\ar@{-}[ruuuuu]&*{}\ar@{-}[r]&*{}\ar@{-}[rd]&*{}\ar@{-}[ru]&*{}\ar@{-}[rd]&*{}\ar@{-}[ru]&*{}\ar@{-}[ruuu]&*{}\ar@{-}[r]&*{}\ar@{-}@{-}[r]&*{}\ar@{-}[rd]&*{}\ar@{-}[ru]&*{}\ar@{-}[rd]&*{}\ar@{-}[ru]&*{}\ar@{-}[rd]&*{}\ar@{-}[ru]&*{}\ar@{-}[rd]&*{}\ar@{-}[r]&*{}\ar@{-}[ruuuuu]&*{}\ar@{-}[ru]&*{}\ar@{-}[rd]&*{}\ar@{-}[ru]&*{}\ar@{-}[rd]&*{}\ar@{-}[ru]&*{}\ar@{-}[rd]&*{}\ar@{-}[r]&*{}\ar@{-}[r]&*{}&*{}\\
\ar@{-}[ruuuuuuu]&*{}\ar@{-}[r]&*{}\ar@{-}[ru]&*{}\ar@{-}[rd]&*{}\ar@{-}[ru]&*{}\ar@{-}[rd]&*{}\ar@{-}[ruuu]&*{}
\ar@{-}[r]&*{}\ar@{-}[rd]&*{}\ar@{-}[ru]&*{}\ar@{-}[rd]&*{}\ar@{-}[ru]&*{}\ar@{-}[rd]&*{}\ar@{-}[ru]&*{}\ar@{-}[rd]&*{}\ar@{-}[ru]&*{}\ar@{-}[rd]&*{}\ar@{-}[ruuuuu]&*{}\ar@{-}[rd]&*{}\ar@{-}[ru]&*{}\ar@{-}[rd]&*{}\ar@{-}[ru]&*{}\ar@{-}[rd]&*{}\ar@{-}[ru]&*{}\ar@{-}[rd]&*{}\ar@{-}[r]&*{}&*{}\\
\ar@{-}[r]  &*{}  \ar@{-}[r]&*{}   &*{}\ar@{-}[ru]&*{}
&*{}\ar@{-}[ru]&*{}\ar@{-}[r]   &*{}
\ar@{-}[r]&*{}\ar@{-}[ru]&*{}&*{}\ar@{-}[ru]&*{}&*{}\ar@{-}[ru]&*{}&*{}\ar@{-}[ru]&*{}&*{}\ar@{-}[ru]&*{}\ar@{-}[r]&*{}\ar@{-}[ru]&*{}&*{}\ar@{-}[ru]&*{}&*{}\ar@{-}[ru]&*{}&*{}\ar@{-}[ru]&*{}\ar@{-}[r]&*{}&*{}\\
\ar@{.}[rr]\ar@{-}[r]  &*{}\ar@{-}[r]
&*{}\ar@{}[ru]|{\fbox{\rule[-.6cm]{0cm}{1cm}$f_1$}} &*{}\ar@{-}[r]  &*{}
\ar@{}[ru]|{\fbox{\rule[-.6cm]{0cm}{1cm}$f_2$}} &*{}\ar@{-}[r]&*{}
\ar@{-}[r]&*{}\ar@{-}[rr]&*{}\ar@{-}[r]&*{}\ar@{-}[ru]|{\fbox{\rule[-.6cm]{0cm}{1cm}$g_1$}}&*{}\ar@{-}[r]&*{}\ar@{-}[ru]|{\fbox{\rule[-.6cm]{0cm}{1cm}$g_2$}}&*{}\ar@{-}[r]&*{}\ar@{-}[ru]|{\fbox{\rule[-.6cm]{0cm}{1cm}$g_3$}}&*{}\ar@{-}[r]&*{}\ar@{-}[ru]|{\fbox{\rule[-.6cm]{0cm}{1cm}$g_4$}}&*{}\ar@{-}[r]&*{}\ar@{-}[r]&*{}\ar@{-}[r]&*{}\ar@{-}[ru]|{\fbox{\rule[-.6cm]{0cm}{1cm}$h_1$}}&*{}\ar@{-}[r]&*{}\ar@{-}[ru]|{\fbox{\rule[-.6cm]{0cm}{1cm}$h_2$}}&*{}\ar@{-}[r]&*{}\ar@{-}[ru]|{\fbox{\rule[-.6cm]{0cm}{1cm}$h_3$}}&*{}\ar@{-}[r]&*{}\ar@{-}[r]&*{}&*{}
}
\]

Now we split the diagram in two sets of different types of symmetries, those which are functorially free from the set of arrows $\{f_i,g_j,h_k:i,j,k\}$ and those that are not.
Here, in the next diagram, the dotted boxes contain part of the free ones:

\[
\xymatrix@R=10pt@C=16pt{
\ar@{-}[rddddddd]&*{}\ar@{.}[ddddd]\ar@{.}[rrrrrrrrrr]\ar@{-}[rddd]&*{}\ar@{-}[rrrr]&*{}&*{}&*{}&*{}\ar@{-}[rddddd]&*{}\ar@{-}[rrrrrrrrrr]&*{}&*{}&*{}&*{}\ar@{..}[dddddd]&*{}&*{}\ar@{.}[ddddd]&*{}&*{}&*{}&*{}\ar@{-}[rddd]&*{}\ar@{-}[rrrrrrr]&*{}&*{}&*{}&*{}&*{}&*{}&*{}\ar@{-}[rdd]&*{}\ar@{.}[dddd]\ar@{.}[lllllllllllll]\\
\ar@{-}[rddddd]&*{}\ar@{-}[rddd]&*{}\ar@{-}[rrrr]&*{}&*{}&*{}&*{}\ar@{-}[rddddd]&*{}\ar@{-}[rrrr]&*{}&*{}&*{}&*{}\ar@{-}[rd]&*{}\ar@{-}[rrrrr]&*{}&*{}&*{}&*{}&*{}\ar@{-}[rddd]&*{}\ar@{-}[rrrrrrr]&*{}&*{}&*{}&*{}&*{}&*{}&*{}\ar@{-}[rdd]&*{}&*{}&*{}\\
\ar@{-}[rddd]&*{}\ar@{-}[rddd]&*{}\ar@{-}[rrrr]&*{}&*{}&*{}&*{}\ar@{-}[rddddd]&*{}\ar@{-}[rrrr]&*{}&*{}&*{}&*{}\ar@{-}[ru]&*{}\ar@{-}[rrrrr]&*{}&*{}&*{}&*{}&*{}\ar@{-}[rddd]&*{}\ar@{-}[rrrrrrr]&*{}&*{}&*{}&*{}&*{}&*{}&*{}\ar@{-}[rdd]&*{}&*{}\\
\ar@{-}[rd]&*{}\ar@{-}[ruuu]&*{}\ar@{-}[rrr]&*{} &*{}
&*{}\ar@{-}[rd]&*{}\ar@{-}[ruuu]&*{}
\ar@{-}[rrrr]&*{}&*{}&*{}&*{}\ar@{-}[rd]&*{}\ar@{-}[rrrrr]&*{}&*{}&*{}&*{}&*{}\ar@{-}[rddd]&*{}\ar@{-}[rd]&*{}\ar@{-}[r]&*{}\ar@{-}[rd]&*{}\ar@{-}[rrrr]&*{}&*{}&*{}&*{}\ar@{-}[ruuu]&*{}&*{}&*{}\\
\ar@{-}[ru]&*{}\ar@{-}[ruuu]&*{}\ar@{-}[rr]&*{}
&*{}\ar@{-}[rd]&*{}\ar@{-}[ru]&*{}\ar@{-}[ruuu]&*{}\ar@{-}[rrrr]&*{}&*{}&*{}&*{}\ar@{-}[ru]&*{}\ar@{-}[rrrrr]&*{}&*{}&*{}&*{}&*{}\ar@{-}[rddd]&*{}\ar@{-}[ru]&*{}\ar@{-}[rd]&*{}\ar@{-}[ru]&*{}\ar@{-}[rd]&*{}\ar@{.}[d]\ar@{.}[rrrr]\ar@{-}[rrr]&*{}&*{}&*{}\ar@{-}[ruuu]&*{}&*{}&*{}\\
\ar@{-}[ruuu]&*{}\ar@{.}[r]\ar@{-}[ruuu]&*{}\ar@{.}[rrrddd]\ar@{-}[r]&*{}\ar@{-}[rd]&*{}\ar@{-}[ru]&*{}\ar@{-}[rd]&*{}\ar@{-}[ruuu]&*{}
\ar@{-}[r]&*{}\ar@{-}[rr]&*{}&*{}\ar@{-}[rd]&*{}\ar@{-}[r]&*{}\ar@{-}[rd]&*{}\ar@{-}[r]\ar@{.}[rrrrdddd]&*{}\ar@{-}[rd]&*{}\ar@{-}[rr]&*{}&*{}\ar@{-}[ruuuuu]&*{}\ar@{-}[rd]&*{}\ar@{-}[ru]&*{}\ar@{-}[rd]&*{}\ar@{-}[ru]&*{}\ar@{-}[rd]&*{}\ar@{-}[rr]&*{}\ar@{-}[r]&*{}\ar@{-}[r]&*{}&*{}&*{}\\
\ar@{-}[ruuuuu]&*{}\ar@{-}[r]&*{}\ar@{-}[rd]&*{}\ar@{-}[ru]&*{}\ar@{-}[rd]&*{}\ar@{-}[ru]&*{}\ar@{-}[ruuu]&*{}\ar@{-}[r]&*{}\ar@{-}@{-}[r]&*{}\ar@{-}[rd]&*{}\ar@{-}[ru]&*{}\ar@{-}[rd]&*{}\ar@{-}[ru]&*{}\ar@{-}[rd]&*{}\ar@{-}[ru]&*{}\ar@{-}[rd]&*{}\ar@{-}[r]&*{}\ar@{-}[ruuuuu]&*{}\ar@{-}[ru]&*{}\ar@{-}[rd]&*{}\ar@{-}[ru]&*{}\ar@{-}[rd]&*{}\ar@{-}[ru]&*{}\ar@{-}[rd]&*{}\ar@{-}[r]&*{}\ar@{-}[r]&*{}&*{}&*{}\\
\ar@{-}[ruuuuuuu]&*{}\ar@{-}[r]&*{}\ar@{-}[ru]&*{}\ar@{-}[rd]&*{}\ar@{-}[ru]&*{}\ar@{-}[rd]&*{}\ar@{-}[ruuu]&*{}
\ar@{-}[r]&*{}\ar@{-}[rd]&*{}\ar@{-}[ru]&*{}\ar@{-}[rd]&*{}\ar@{-}[ru]&*{}\ar@{-}[rd]&*{}\ar@{-}[ru]&*{}\ar@{-}[rd]&*{}\ar@{-}[ru]&*{}\ar@{-}[rd]&*{}\ar@{-}[ruuuuu]&*{}\ar@{-}[rd]&*{}\ar@{-}[ru]&*{}\ar@{-}[rd]&*{}\ar@{-}[ru]&*{}\ar@{-}[rd]&*{}\ar@{-}[ru]&*{}\ar@{-}[rd]&*{}\ar@{-}[r]&*{}&*{}&*{}\\
\ar@{-}[r]  &*{}  \ar@{-}[r]&*{}   &*{}\ar@{-}[ru]&*{}
&*{}\ar@{-}[ru]\ar@{.}[rrd]&*{}\ar@{-}[r]   &*{}
\ar@{-}[r]&*{}\ar@{-}[ru]&*{}\ar@{.}[rruu]&*{}\ar@{-}[ru]&*{}&*{}\ar@{-}[ru]&*{}&*{}\ar@{-}[ru]&*{}&*{}\ar@{-}[ru]&*{}\ar@{-}[r]&*{}\ar@{-}[ru]&*{}&*{}\ar@{-}[ru]&*{}&*{}\ar@{-}[ru]&*{}&*{}\ar@{-}[ru]&*{}\ar@{-}[r]&*{}&*{}&*{}\\
\ar@{-}[r]  &*{}\ar@{-}[r] &*{}\ar@{}[ru]|{\fbox{\rule[-.6cm]{0cm}{1cm}$f_1$}}
&*{}\ar@{-}[r]  &*{} \ar@{}[ru]|{\fbox{\rule[-.6cm]{0cm}{1cm}$f_2$}}
&*{}\ar@{-}[r]&*{}
\ar@{-}[r]&*{}\ar@{.}[rru]\ar@{-}[rr]&*{}\ar@{-}[r]&*{}\ar@{-}[ru]|{\fbox{\rule[-.6cm]{0cm}{1cm}$g_1$}}&*{}\ar@{-}[r]&*{}\ar@{-}[ru]|{\fbox{\rule[-.6cm]{0cm}{1cm}$g_2$}}&*{}\ar@{-}[r]&*{}\ar@{-}[ru]|{\fbox{\rule[-.6cm]{0cm}{1cm}$g_3$}}&*{}\ar@{-}[r]&*{}\ar@{-}[ru]|{\fbox{\rule[-.6cm]{0cm}{1cm}$g_4$}}&*{}\ar@{-}[r]&*{}\ar@{-}[r]&*{}\ar@{.}[rrrruuuu]\ar@{-}[r]\ar@{-}[r]&*{}\ar@{-}[ru]|{\fbox{\rule[-.6cm]{0cm}{1cm}$h_1$}}&*{}\ar@{-}[r]&*{}\ar@{-}[ru]|{\fbox{\rule[-.6cm]{0cm}{1cm}$h_2$}}&*{}\ar@{-}[r]&*{}\ar@{-}[ru]|{\fbox{\rule[-.6cm]{0cm}{1cm}$h_3$}}&*{}\ar@{-}[r]&*{}\ar@{-}[r]&*{}&*{}&*{}
}
\]
So, we replace this box:

\[
\xymatrix@R=10pt@C=16pt{
\ar@{-}[rddd]\ar@{.}[ddddd]\ar@{.}[r]&*{}\ar@{.}[rrrrr]\ar@{-}[rrrr]&*{}&*{}&*{}&*{}\ar@{-}[rddddd]&*{}\ar@{-}[rrrr]&*{}&*{}&*{}&*{}\ar@{..}[dddddd]&*{}&*{}&*{}&*{}&*{}&*{}\\
\ar@{-}[rddd]&*{}\ar@{-}[rrrr]&*{}&*{}&*{}&*{}\ar@{-}[rddddd]&*{}\ar@{-}[rrrr]&*{}&*{}&*{}&*{}&*{}&*{}&*{}&*{}&*{}&*{}\\
\ar@{-}[rddd]&*{}\ar@{-}[rrrr]&*{}&*{}&*{}&*{}\ar@{-}[rddddd]&*{}\ar@{-}[rrrr]&*{}&*{}&*{}&*{}&*{}&*{}&*{}&*{}&*{}&*{}\\
\ar@{-}[ruuu]&*{}\ar@{-}[rrr]&*{} &*{} &*{}\ar@{-}[rd]&*{}\ar@{-}[ruuu]&*{}
\ar@{-}[rrrr]&*{}&*{}&*{}&*{}&*{}&*{}&*{}&*{}&*{}&*{}\\
\ar@{-}[ruuu]&*{}\ar@{-}[rr]&*{}
&*{}\ar@{-}[rd]&*{}\ar@{-}[ru]&*{}\ar@{-}[ruuu]&*{}\ar@{-}[rrrr]&*{}&*{}&*{}&*{}&*{}&*{}&*{}&*{}&*{}&*{}\\
\ar@{.}[r]\ar@{-}[ruuu]&*{}\ar@{.}[rrrddd]\ar@{-}[r]&*{}\ar@{-}[rd]&*{}\ar@{-}[ru]&*{}\ar@{-}[rd]&*{}\ar@{-}[ruuu]&*{}
\ar@{-}[r]&*{}\ar@{-}[rr]&*{}&*{}\ar@{-}[rd]&*{}&*{}&*{}&*{}&*{}&*{}&*{}\\
&*{}&*{}\ar@{-}[ru]&*{}\ar@{-}[rd]&*{}\ar@{-}[ru]&*{}\ar@{-}[ruuu]&*{}\ar@{-}[r]&*{}\ar@{-}@{-}[r]&*{}\ar@{-}[rd]&*{}\ar@{-}[ru]&*{}&*{}&*{}&*{}&*{}&*{}&*{}\\
&*{}&*{}&*{}\ar@{-}[ru]&*{}\ar@{-}[rd]&*{}\ar@{-}[ruuu]&*{}
\ar@{-}[r]&*{}\ar@{-}[rd]&*{}\ar@{-}[ru]&*{}&*{}&*{}&*{}&*{}&*{}&*{}&*{}\\
  &*{}   &*{}&*{}   &*{}\ar@{-}[ru]\ar@{.}[rrd]&*{}\ar@{-}[r]   &*{}
\ar@{-}[r]&*{}\ar@{-}[ru]&*{}\ar@{.}[rruu]&*{}&*{}&*{}&*{}&*{}&*{}&*{}&*{}\\
&*{}&*{} &*{} &*{}&*{}&*{}\ar@{.}[rru]&*{}&*{}&*{}&*{}&*{}&*{}&*{}&*{}&*{}&*{}
}
\]
By this one:

\[
\xymatrix@R=10pt@C=16pt{
\ar@{-}[rrrrr]&*{}&*{}&*{}&*{}&*{}\ar@{-}[rd]&*{}\ar@{-}[r]&*{}&*{}\\
\ar@{-}[rrrr]&*{}&*{}&*{}&*{}\ar@{-}[rd]&*{}\ar@{-}[ru]&*{}\ar@{-}[r]&*{}&*{}\\
\ar@{-}[rrr]&*{}&*{}&*{}\ar@{-}[rd]&*{}\ar@{-}[ru]&*{}\ar@{-}[rd]&*{}\ar@{-}[r]&*{}&*{}\\
\ar@{-}[rr]&*{}&*{}\ar@{-}[rd]&*{}\ar@{-}[ru]&*{}\ar@{-}[rd]&*{}\ar@{-}[ru]&*{}\ar@{-}[r]&*{}&*{}\\
\ar@{-}[r]&*{}\ar@{-}[rd]&*{}\ar@{-}[ru]&*{}\ar@{-}[rd]&*{}\ar@{-}[ru]&*{}\ar@{-}[rd]&*{}\ar@{-}[r]&*{}&*{}\\
\ar@{-}[rd]&*{}\ar@{-}[ru]&*{}\ar@{-}[rd]&*{}\ar@{-}[ru]&*{}\ar@{-}[rd]&*{}\ar@{-}[ru]&*{}\ar@{-}[r]&*{}&*{}\\
\ar@{-}[ru]&*{}\ar@{-}[rd]&*{}\ar@{-}[ru]&*{}\ar@{-}[rd]&*{}\ar@{-}[ru]&*{}\ar@{-}[rr]&*{}&*{}&*{}\\
\ar@{-}[r]&*{}\ar@{-}[ru]&*{}\ar@{-}[rd]&*{}\ar@{-}[ru]&*{}\ar@{-}[rrr]&*{}&*{}&*{}&*{}\\
\ar@{-}[rr]&*{}&*{}\ar@{-}[ru]&*{}\ar@{-}[rrrr]&*{}&*{}&*{}&*{}&*{}\\
\ar@{-}[rrrrrrr]&*{}&*{}&*{}&*{}&*{}&*{}&*{}&*{}
}
\]
and this one:

\[
\xymatrix@R=10pt@C=16pt{
\ar@{-}[rrrr]\ar@{.}[ddddd]&*{}&*{}&*{}&*{}\ar@{-}[rddd]&*{}\ar@{-}[rrrrrrr]&*{}&*{}&*{}&*{}&*{}&*{}&*{}\ar@{-}[rdd]&*{}\ar@{.}[dddd]\ar@{.}[lllllllllllll]\\
\ar@{-}[rrrr]&*{}&*{}&*{}&*{}\ar@{-}[rddd]&*{}\ar@{-}[rrrrrrr]&*{}&*{}&*{}&*{}&*{}&*{}&*{}\ar@{-}[rdd]&*{}&*{}&*{}\\
\ar@{-}[rrrr]&*{}&*{}&*{}&*{}\ar@{-}[rddd]&*{}\ar@{-}[rrrrrrr]&*{}&*{}&*{}&*{}&*{}&*{}&*{}\ar@{-}[rdd]&*{}&*{}\\
\ar@{-}[rrrr]&*{}&*{}&*{}&*{}\ar@{-}[rddd]&*{}\ar@{-}[rd]&*{}\ar@{-}[r]&*{}\ar@{-}[rd]&*{}\ar@{-}[rrrr]&*{}&*{}&*{}&*{}\ar@{-}[ruuu]&*{}&*{}&*{}\\
\ar@{-}[rrrr]&*{}&*{}&*{}&*{}\ar@{-}[rddd]&*{}\ar@{-}[ru]&*{}\ar@{-}[rd]&*{}\ar@{-}[ru]&*{}\ar@{-}[rd]&*{}\ar@{-}[rrr]\ar@{.}[rrrr]&*{}&*{}&*{}\ar@{-}[ruuu]&*{}&*{}&*{}\\
\ar@{.}[rrrrdddd]\ar@{-}[r]&*{}\ar@{-}[rd]&*{}\ar@{-}[rr]&*{}&*{}\ar@{-}[ruuuuu]&*{}\ar@{-}[rd]&*{}\ar@{-}[ru]&*{}\ar@{-}[rd]&*{}\ar@{-}[ru]&*{}\ar@{.}[u]&*{}&*{}&*{}&*{}&*{}&*{}\\
&*{}\ar@{-}[ru]&*{}\ar@{-}[rd]&*{}\ar@{-}[r]&*{}\ar@{-}[ruuuuu]&*{}\ar@{-}[ru]&*{}\ar@{-}[rd]&*{}\ar@{-}[ru]&*{}&*{}&*{}&*{}&*{}&*{}&*{}\\
&*{}&*{}\ar@{-}[ru]&*{}\ar@{-}[rd]&*{}\ar@{-}[ruuuuu]&*{}\ar@{-}[rd]&*{}\ar@{-}[ru]&*{}&*{}&*{}&*{}&*{}&*{}&*{}&*{}&*{}\\
&*{}&*{}&*{}\ar@{-}[ru]&*{}\ar@{-}[r]&*{}\ar@{-}[ru]&*{}&*{}&*{}&*{}&*{}&*{}&*{}&*{}&*{}\\
&*{}&*{}&*{}&*{}\ar@{-}[r]&*{}\ar@{.}[rrrruuuu]&*{}&*{}&*{}&*{}&*{}&*{}&*{}&*{}&*{}
}
\]
By this other one:\\

\[
\xymatrix@R=10pt@C=16pt{
\ar@{-}[rd]&*{}\ar@{-}[r]&*{}\ar@{-}[rd]&*{}\ar@{-}[rrrr]&*{}&*{}&*{}&*{}&*{}\\
\ar@{-}[ru]&*{}\ar@{-}[rd]&*{}\ar@{-}[ru]&*{}\ar@{-}[rd]&*{}\ar@{-}[rrr]&*{}&*{}&*{}&*{}\\
\ar@{-}[rd]&*{}\ar@{-}[ru]&*{}\ar@{-}[rd]&*{}\ar@{-}[ru]&*{}\ar@{-}[rd]&*{}\ar@{-}[rr]&*{}&*{}&*{}\\
\ar@{-}[ru]&*{}\ar@{-}[rd]&*{}\ar@{-}[ru]&*{}\ar@{-}[rd]&*{}\ar@{-}[ru]&*{}\ar@{-}[rd]&*{}\ar@{-}[r]&*{}&*{}\\
\ar@{-}[rd]&*{}\ar@{-}[ru]&*{}\ar@{-}[rd]&*{}\ar@{-}[ru]&*{}\ar@{-}[rd]&*{}\ar@{-}[ru]&*{}\ar@{-}[rd]&*{}&*{}\\
\ar@{-}[ru]&*{}\ar@{-}[rd]&*{}\ar@{-}[ru]&*{}\ar@{-}[rd]&*{}\ar@{-}[ru]&*{}\ar@{-}[rd]&*{}\ar@{-}[ru]&*{}&*{}\\
\ar@{-}[r]&*{}\ar@{-}[ru]&*{}\ar@{-}[rd]&*{}\ar@{-}[ru]&*{}\ar@{-}[rd]&*{}\ar@{-}[ru]&*{}\ar@{-}[r]&*{}&*{}\\
\ar@{-}[rr]&*{}\ar@{-}[r]&*{}\ar@{-}[ru]&*{}\ar@{-}[rd]&*{}\ar@{-}[ru]&*{}\ar@{-}[rr]&*{}&*{}&*{}\\
\ar@{-}[rrr]&*{}\ar@{-}[rr]&*{}&*{}\ar@{-}[ru]&*{}\ar@{-}[rrr]&*{}&*{}&*{}&*{}\\
\ar@{-}[rrrrrrr]&*{}&*{}&*{}&*{}&*{}&*{}&*{}&*{}
}
\]
Finally, we get the desired diagram:

\[
\xymatrix@R=10pt@C=16pt{
\ar@{-}[rddddddd]
&*{}\ar@{-}[rrrrrrr]&*{}&*{}&*{}&*{}&*{}&*{}&*{}\ar@{-}[rd]&*{}\ar@{-}[r]&*{}\ar@{-}[rd]&*{}\ar@{-}[r]&*{}\ar@{-}[rd]&*{}\ar@{-}[rrrrrrr]&*{}&*{}&*{}&*{}&*{}&*{}&*{}&*{}\\
\ar@{-}[rddddd]
&*{}\ar@{-}[rrrrrr]&*{}&*{}&*{}&*{}&*{}&*{}\ar@{-}[rd]&*{}\ar@{-}[ru]&*{}\ar@{-}[rd]&*{}\ar@{-}[ru]&*{}\ar@{-}[rd]&*{}\ar@{-}[ru]&*{}\ar@{-}[rd]&*{}\ar@{-}[rrrrrr]&*{}&*{}&*{}&*{}&*{}&*{}&*{}\\
\ar@{-}[rddd]
&*{}\ar@{-}[rrrrr]&*{}&*{}&*{}&*{}&*{}\ar@{-}[rd]&*{}\ar@{-}[ru]&*{}\ar@{-}[rd]&*{}\ar@{-}[ru]&*{}\ar@{-}[rd]&*{}\ar@{-}[ru]&*{}\ar@{-}[rd]&*{}\ar@{-}[ru]&*{}\ar@{-}[rd]&*{}\ar@{-}[rrrrr]&*{}&*{}&*{}&*{}&*{}&*{}\\
\ar@{-}[rd]
&*{}\ar@{-}[rrrr]&*{}&*{}&*{}&*{}\ar@{-}[rd]&*{}\ar@{-}[ru]&*{}\ar@{-}[rd]&*{}\ar@{-}[ru]&*{}\ar@{-}[rd]&*{}\ar@{-}[ru]&*{}\ar@{-}[rd]&*{}\ar@{-}[ru]&*{}\ar@{-}[rd]&*{}\ar@{-}[ru]&*{}\ar@{-}[rd]&*{}\ar@{-}[rrrr]&*{}&*{}&*{}&*{}&*{}\\
\ar@{-}[ru]
&*{}\ar@{-}[rrr]&*{}&*{}&*{}\ar@{-}[rd]&*{}\ar@{-}[ru]&*{}\ar@{-}[rd]&*{}\ar@{-}[ru]&*{}\ar@{-}[rd]&*{}\ar@{-}[ru]&*{}\ar@{-}[rd]&*{}\ar@{-}[ru]&*{}\ar@{-}[rd]&*{}\ar@{-}[ru]&*{}\ar@{-}[rd]&*{}\ar@{-}[ru]&*{}\ar@{-}[rd]&*{}\ar@{-}[rrr]&*{}&*{}&*{}&*{}\\
\ar@{-}[ruuu]
&*{}\ar@{-}[rr]&*{}&*{}\ar@{-}[rd]&*{}\ar@{-}[ru]&*{}\ar@{-}[rd]&*{}\ar@{-}[ru]&*{}\ar@{-}[rd]&*{}\ar@{-}[ru]&*{}\ar@{-}[rd]&*{}\ar@{-}[ru]&*{}\ar@{-}[rd]&*{}\ar@{-}[ru]&*{}\ar@{-}[rd]&*{}\ar@{-}[ru]&*{}\ar@{-}[rd]&*{}\ar@{-}[ru]&*{}\ar@{-}[rd]&*{}\ar@{-}[rr]&*{}&*{}&*{}\\
\ar@{-}[ruuuuu]
&*{}\ar@{-}[r]&*{}\ar@{-}[rd]&*{}\ar@{-}[ru]&*{}\ar@{-}[rd]&*{}\ar@{-}[ru]&*{}\ar@{-}[rd]&*{}\ar@{-}[ru]&*{}\ar@{-}[rd]&*{}\ar@{-}[ru]&*{}\ar@{-}[rd]&*{}\ar@{-}[ru]&*{}\ar@{-}[rd]&*{}\ar@{-}[ru]&*{}\ar@{-}[rd]&*{}\ar@{-}[ru]&*{}\ar@{-}[rd]&*{}\ar@{-}[ru]&*{}\ar@{-}[rd]&*{}\ar@{-}[r]&*{}&*{}\\
\ar@{-}[ruuuuuuu]
&*{}\ar@{-}[rd]&*{}\ar@{-}[ru]&*{}\ar@{-}[rd]&*{}\ar@{-}[ru]&*{}\ar@{-}[rd]&*{}\ar@{-}[ru]&*{}\ar@{-}[rd]&*{}\ar@{-}[ru]&*{}\ar@{-}[rd]&*{}\ar@{-}[ru]&*{}\ar@{-}[rd]&*{}\ar@{-}[ru]&*{}\ar@{-}[rd]&*{}\ar@{-}[ru]&*{}\ar@{-}[rd]&*{}\ar@{-}[ru]&*{}\ar@{-}[rd]&*{}\ar@{-}[ru]&*{}\ar@{-}[rd]&*{}&*{}\\
                      \ar@{-}[r]&*{}\ar@{-}[ru]&*{}&*{}\ar@{-}[ru]&*{}&*{}\ar@{-}[ru]&*{}&*{}\ar@{-}[ru]&*{}&*{}\ar@{-}[ru]&*{}&*{}\ar@{-}[ru]&*{}&*{}\ar@{-}[ru]&*{}&*{}\ar@{-}[ru]&*{}&*{}\ar@{-}[ru]&*{}&*{}\ar@{-}[ru]&*{}&*{}\\
                      \ar@{-}[rr]&*{}&*{}\ar@{-}[ru]|{\fbox{\rule[-.6cm]{0cm}{1cm}$f_1$}}&*{}\ar@{-}[r]&*{}\ar@{-}[ru]|{\fbox{\rule[-.6cm]{0cm}{1cm}$f_2$}}&*{}\ar@{-}[r]&*{}\ar@{-}[ru]|{\fbox{\rule[-.6cm]{0cm}{1cm}$g_1$}}&*{}\ar@{-}[r]&*{}\ar@{-}[ru]|{\fbox{\rule[-.6cm]{0cm}{1cm}$g_2$}}&*{}\ar@{-}[r]&*{}\ar@{-}[ru]|{\fbox{\rule[-.6cm]{0cm}{1cm}$g_3$}}&*{}&*{}\ar@{-}[ru]|{\fbox{\rule[-.6cm]{0cm}{1cm}$g_4$}}&*{}\ar@{-}[r]&*{}\ar@{-}[ru]|{\fbox{\rule[-.6cm]{0cm}{1cm}$h_1$}}&*{}\ar@{-}[r]&*{}\ar@{-}[ru]|{\fbox{\rule[-.6cm]{0cm}{1cm}$h_2$}}&*{}\ar@{-}[r]&*{}\ar@{-}[ru]|{\fbox{\rule[-.6cm]{0cm}{1cm}$h_3$}}&*{}\ar@{-}[r]&*{}&*{}
}
\]

To go from $[\vec{f},\vec{g},\vec{h}]$ to $[\vec{f},[\vec{g}],\vec{h}]$ we use the same arguments in the
reverse order since $[\vec{g}]\downarrow$.
\end{proof}
Next, we wish to show that the paracategory $Int^p(C)$ is strict symmetric monoidal.
\begin{definition}
\rm Let $(\mathcal{C},\otimes,I,\Tr,s)$ be a symmetric monoidal partially traced category, the tensor in the graph $Int^p(\mathcal{C})$ is defined as follows:
\begin{itemize}
\item The unit is $(I,I)$
\item on objects: $(A^+,A^-)\otimes (B^+,B^-)=(A^+\otimes B^+,B^-\otimes A^-)$
\item on arrows: given $f^{Int^p}:(A^+,A^-)\rightarrow (C^+,C^-)$ and $g^{Int^p}:(B^+,B^-)\rightarrow (D^+,D^-)$ then $(f\otimes g)^{Int^p}:(A^+,A^-)\otimes (B^+,B^-)\rightarrow (C^+,C^-)\otimes (D^+,D^-)$ is defined by
\begin{center}
$A^+\otimes B^+\otimes D^-\otimes C^-\stackrel{s\otimes s}{\rightarrow} B^+\otimes A^+\otimes C^-\otimes D^-\stackrel{1\otimes f\otimes 1}{\rightarrow}  B^+\otimes C^+\otimes A^-\otimes D^-\stackrel{s\otimes s}{\rightarrow}  C^+\otimes B^+\otimes D^-\otimes A^-\stackrel{1\otimes g\otimes 1}{\rightarrow} C^+\otimes D^+\otimes B^-\otimes A^-.$
\end{center}
\end{itemize}
\end{definition}

Let us derive some immediate consequences of this definition:
\begin{itemize}
\item[(i)] $$1_{(A^+,A^-)}\otimes 1_{(B^+,B^-)}= \includegraphics[height=0.25in]{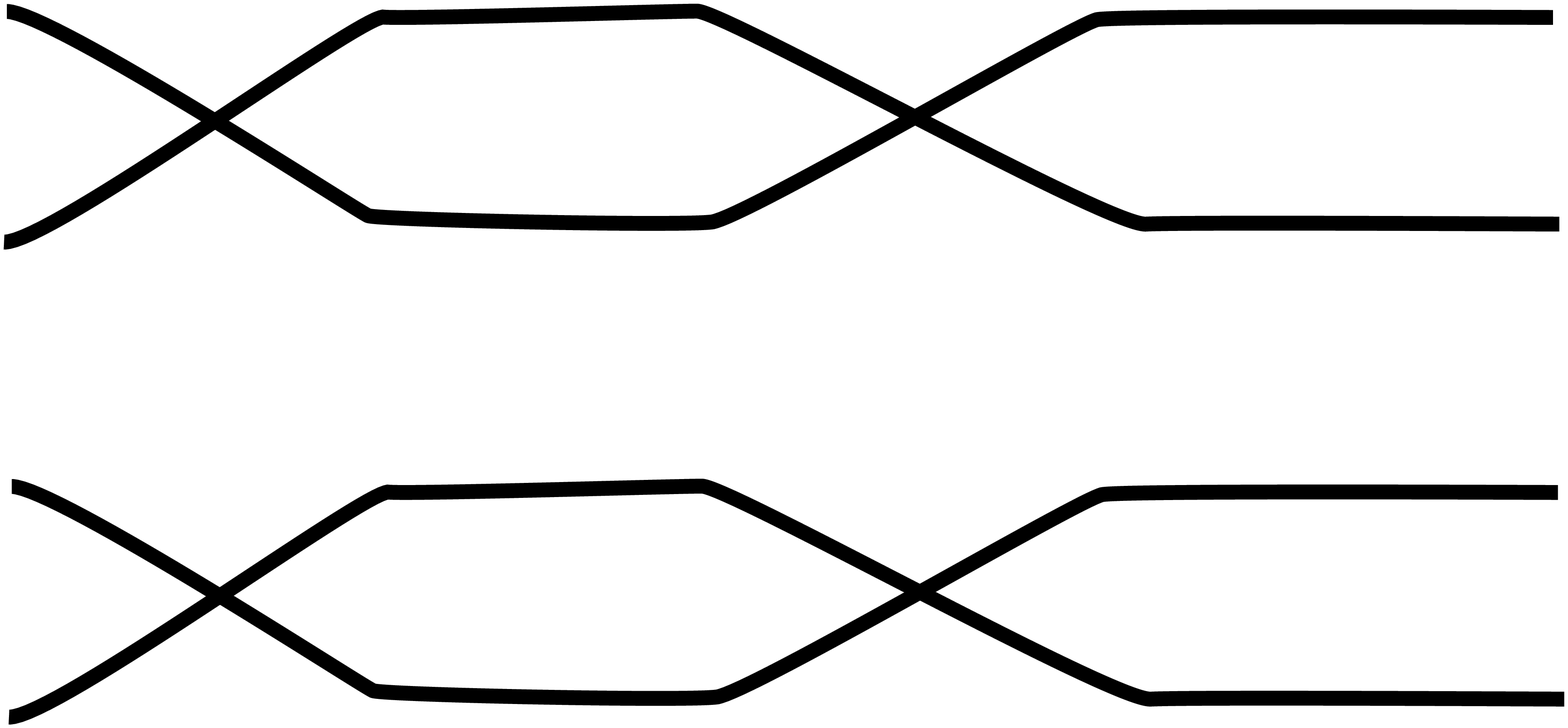}= \includegraphics[height=0.25in]{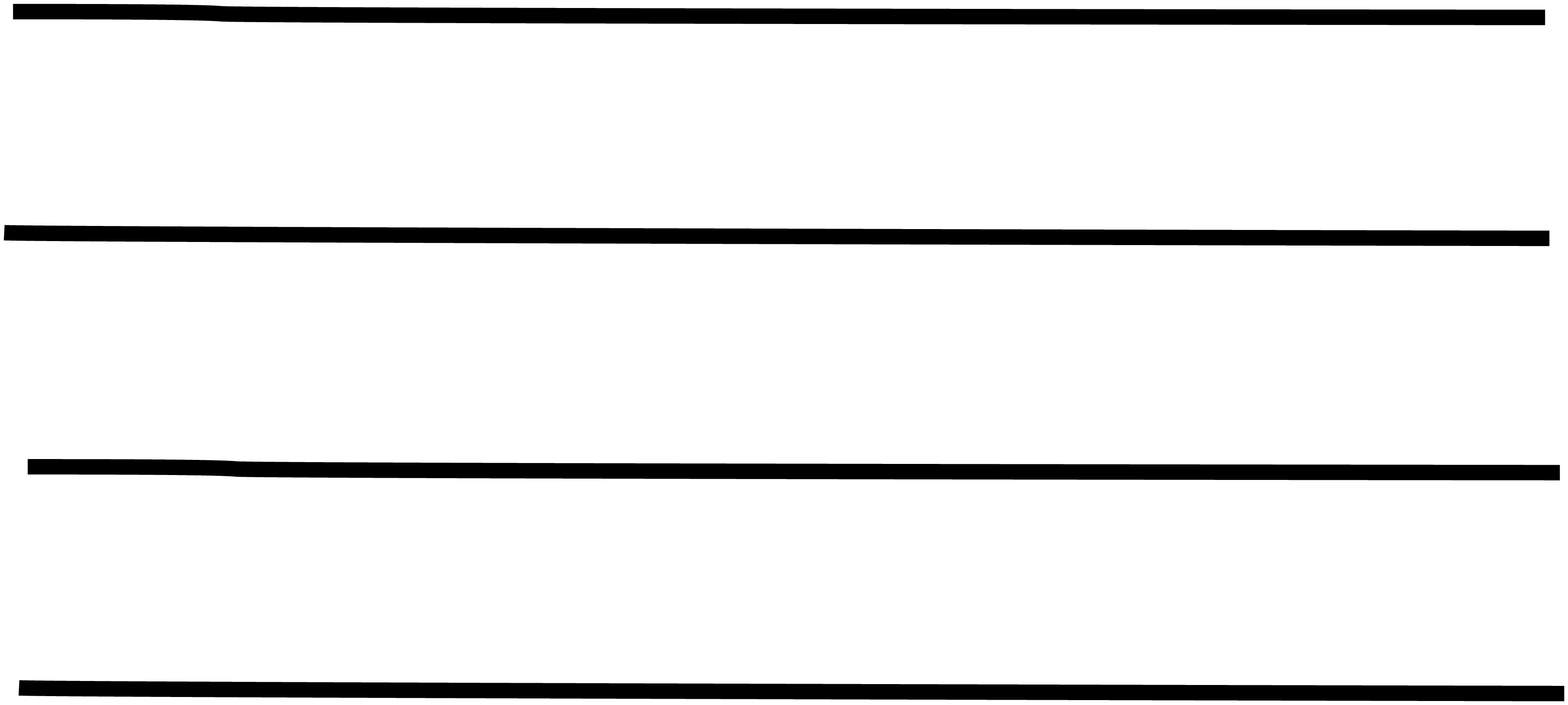}=1_{(A^+\otimes B^+,B^-\otimes A^-)}$$
\item[(ii)] $(A^+,A^-)\otimes (I,I)=(A^+\otimes I,I\otimes A^-)=(A^+,A^-)$
and $(I,I)\otimes (A^+,A^-)=(A^+,A^-)$.
\item[(iii)]$(A^+,A^-)\otimes ((B^+,B^-)\otimes (C^+,C^-))=(A^+\otimes B^+\otimes C^+,C^-\otimes B^-\otimes A^-)=((A^+,A^-)\otimes (B^+,B^-))\otimes (C^+,C^-).$
\end{itemize}

\begin{definition}
\rm The symmetry $(A^+,A^-)\otimes (B^+,B^-)\stackrel{\sigma}{\rightarrow} (B^+,B^-)\otimes (A^+,A^-)$  is defined in $Int^p(\mathcal{C})$ by the following formula: $\sigma=s_{A^+B^+}\otimes s_{A^-B^-}$.
\end{definition}
\begin{lemma}
\label{EXISTENCE SYMM COMP}
Let $(\mathcal{C},\otimes,I,\Tr,s)$ be a symmetric monoidal partial traced category. Given $f^{Int^p}:(Y^+,Y^-)\rightarrow (C^+,C^-)\otimes (D^+,D^-)$, and  $g^{Int^p}:(A^+,A^-)\otimes (B^+,B^-)\rightarrow (X^+,X^-)$ then $[f,\sigma]\downarrow$ and $[\sigma,g]\downarrow$.
\end{lemma}
\begin{proof}
To simplify the notation we use the symbol ``;" for the composition in the category $\cC$ with the order given by graphical concatenation.

We first consider the composition of $f:Y^+\otimes D^-\otimes C^-\rightarrow C^+\otimes D^+\otimes Y^-$ with the following symmetries and identities in the category $\mathcal{C}$: $(1_{Y^+}\otimes s_{C^-D^-});f;(s_{C^+D^+}\otimes 1_{Y^-})$.\\
Next, since by the yanking axiom $s_{D^-\otimes C^-,D^-\otimes C^-}\in \Trc^{D^-\otimes C^-}_{D^-\otimes C^-,D^-\otimes C^-}$ and
$\Tr^{D^-\otimes C^-}_{D^-\otimes C^-,D^-\otimes C^-}(s_{D^-\otimes C^-,D^-\otimes C^-})=1_{D^-\otimes C^-}$  then by superposing axiom we have that  $1_{Y^+}\otimes s_{D^-\otimes C^-,D^-\otimes C^-}\in \Trc^{D^-\otimes C^-}_{Y^+\otimes D^-\otimes C^-, Y^+\otimes D^-\otimes C^-}$ and
\begin{center}
 $1_{Y^+}\otimes \Tr^{D^-\otimes C^-}_{D^-\otimes C^-,D^-\otimes C^-}(s_{D^-\otimes C^-, D^-\otimes C^-})=
  \Tr^{D^-\otimes C^-}_{Y^+\otimes D^-\otimes C^-, Y^+\otimes D^-\otimes C^-}(1_{Y^+}\otimes (s_{D^-\otimes C^-, D^-\otimes C^-})).$
 \end{center}
Therefore, by naturality we have that:\\
 $(1_{Y^+}\otimes s_{C^-D^-});\Tr^{D^-\otimes C^-}_{Y^+\otimes D^-\otimes C^-, Y^+\otimes D^-\otimes C^-}(1_{Y^+}\otimes (s_{D^-\otimes C^-, D^-\otimes C^-}));f;(s_{C^+D^+}\otimes 1_{Y^-})=$\\
 $\Tr^{D^-\otimes C^-}_{Y^+\otimes C^-\otimes D^-, Y^+\otimes D^-\otimes C^-}((1_{Y^+}\otimes s_{C^-D^-}\otimes 1_{D^-\otimes C^-});(1_{Y^+}\otimes s_{D^-\otimes C^- D^-\otimes C^-}));f;(s_{C^+D^+}\otimes 1_{Y^-})=$\\
by coherence:\\
 $\Tr^{D^-\otimes C^-}_{Y^+\otimes C^-\otimes D^-, Y^+\otimes D^-\otimes C^-}((1_{Y^+}\otimes s_{C^-\otimes D^-, D^-\otimes C^-});(1_{Y^+}\otimes 1_{D^-\otimes C^-}\otimes s_{C^- D^-}));f;(s_{C^+D^+}\otimes 1_{Y^-})=$\\

by the naturality axiom:\\
 $\Tr^{D^-\otimes C^-}_{Y^+\otimes C^-\otimes D^-, D^+\otimes C^+\otimes Y^-}((1_{Y^+}\otimes s_{C^-\otimes D^-, D^-\otimes C^-});(1_{Y^+}\otimes 1_{D^-\otimes C^-}\otimes s_{C^- D^-});f\otimes 1_{D^-\otimes C^-};(s_{C^+D^+}\otimes 1_{Y^-}\otimes 1_{D^-\otimes C^-}))=$\\

and by functoriality:\\
 $\Tr^{D^-\otimes C^-}_{Y^+\otimes C^-\otimes D^-, D^+\otimes C^+\otimes Y^-}((1_{Y^+}\otimes s_{C^-\otimes D^-, D^-\otimes C^-});(f\otimes 1_{C^-\otimes D^-});(s_{C^+ D^+}\otimes 1_{Y^-}\otimes s_{C^-D^-}))=$\\

Now by coherence, we can replace:\\
$$s_{C^+ D^+}\otimes 1_{Y^-}\otimes s_{C^-D^-}$$
by the following
$$(1_{C^+\otimes D^+}\otimes s_{Y^-, C^-\otimes D^-});(s_{C^+D^+}\otimes s_{C^-D^-}\otimes 1_{Y^-});(1_{D^+\otimes C^+}\otimes s_{D^-\otimes C^-, Y^-}).$$\\
Which, by definition, is $[f,s_{C^+D^+}\otimes s_{C^-D^-}]$, i.e., we proved that $[f,\sigma_{(C^+,C^-),(D^+,D^-)}]\downarrow$.\\
After repeating a similar argument as above, we have: $[s_{A^+B^+}\otimes s_{A^-B^-},f]\downarrow.$

Now we repeat the proof using graphical language. The purpose of this is to persuade the reader of the advantages of using this methodology.
We start with the following diagram
$$\includegraphics[height=0.4in]{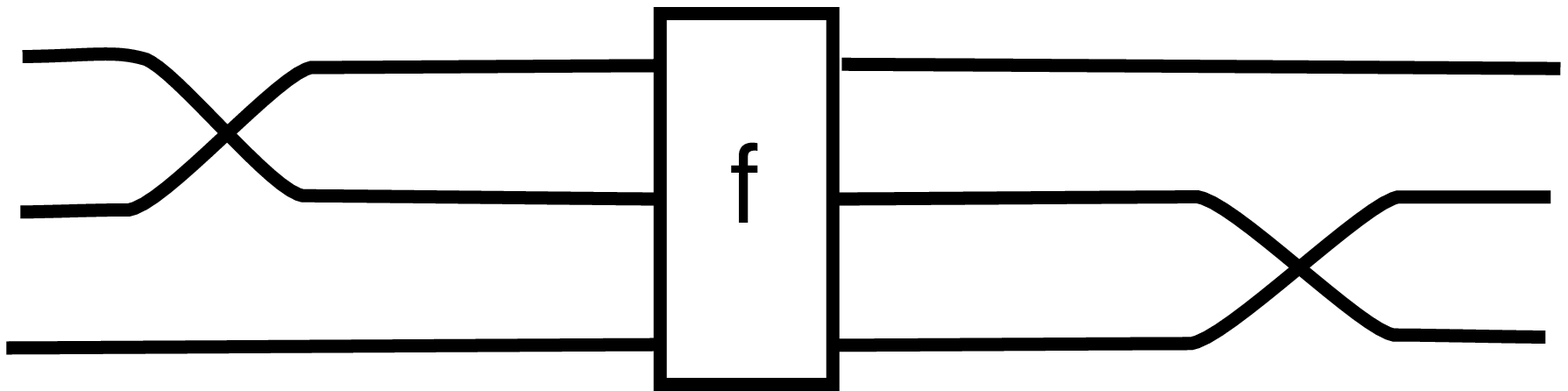}$$
by the yanking axiom the graphic inside the box is in the trace class
$$\includegraphics[height=1in]{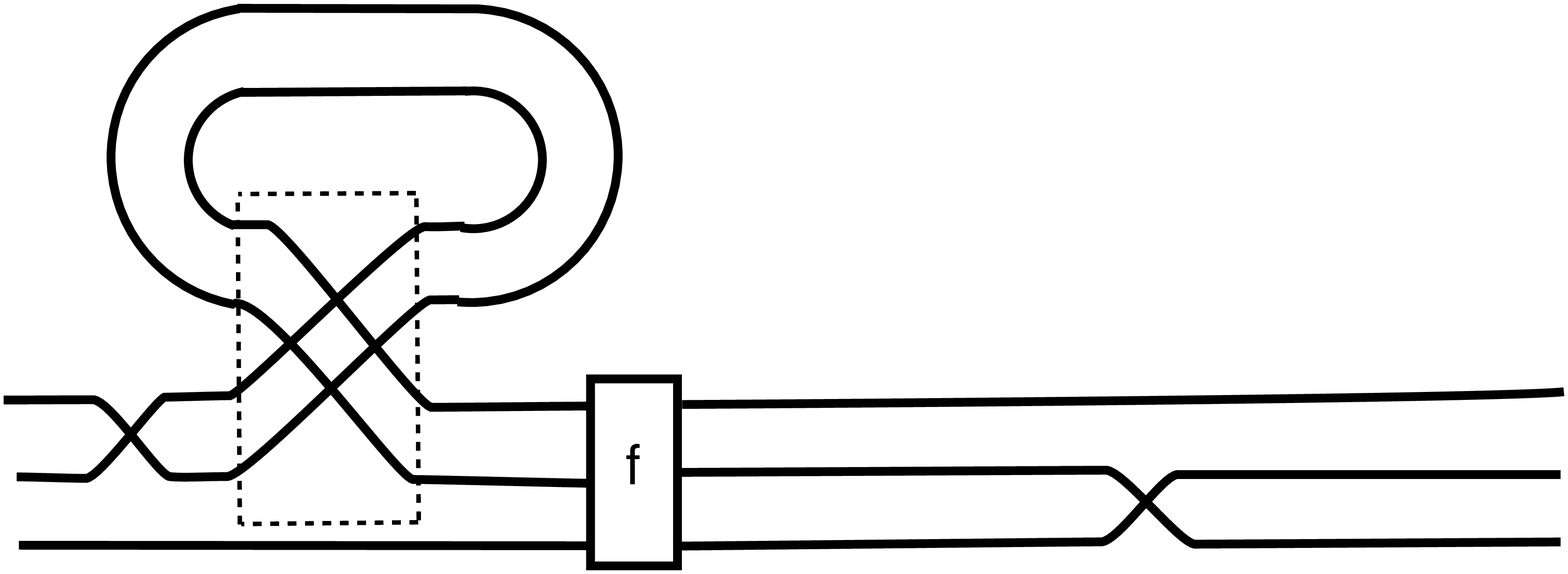}$$
by the superposition axiom
$$\includegraphics[height=1in]{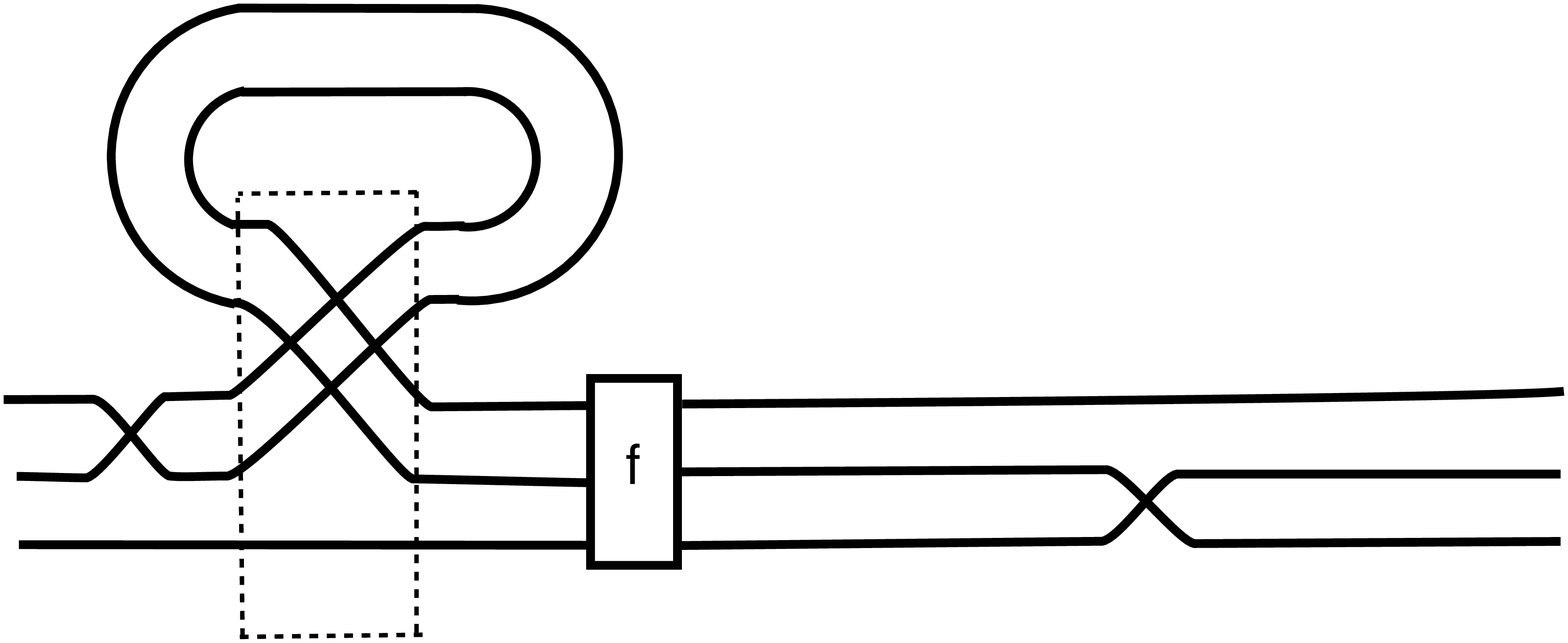}$$
naturality axiom
$$\includegraphics[height=1in]{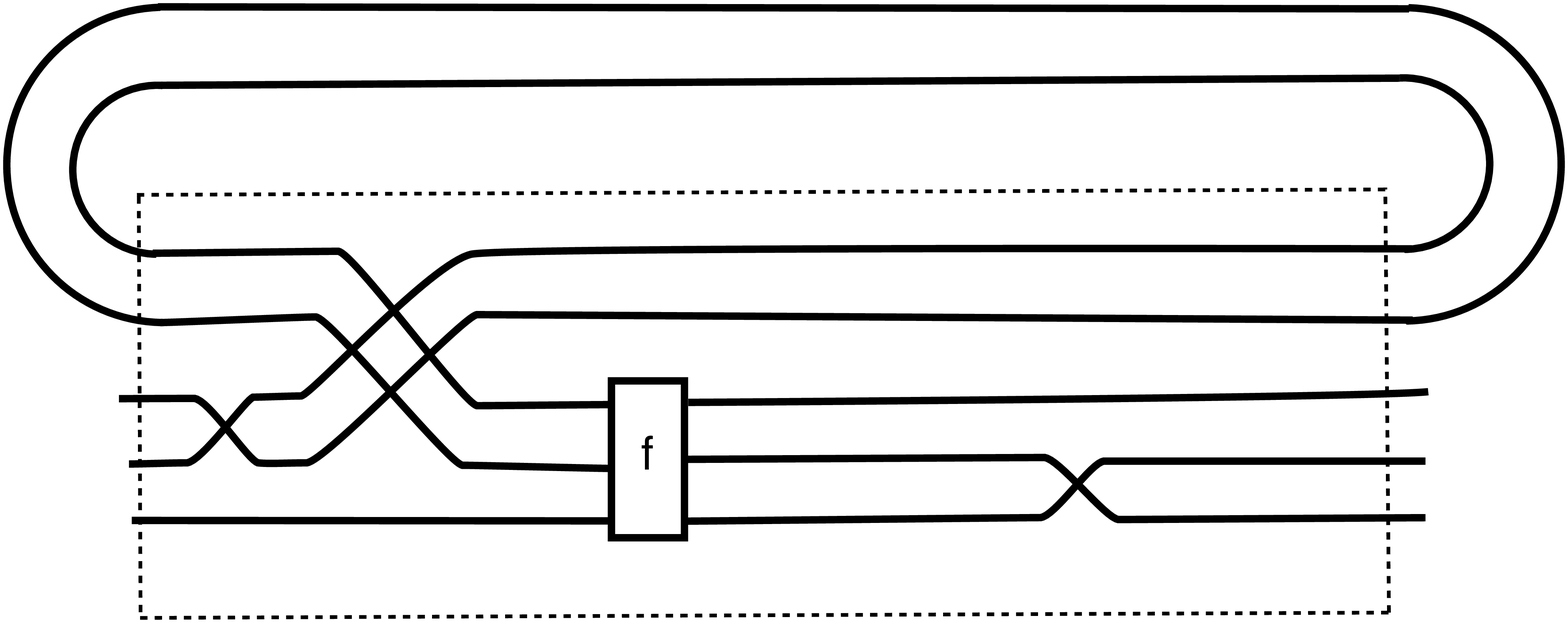}$$
naturality of the symmetry $\sigma$
$$\includegraphics[height=1in]{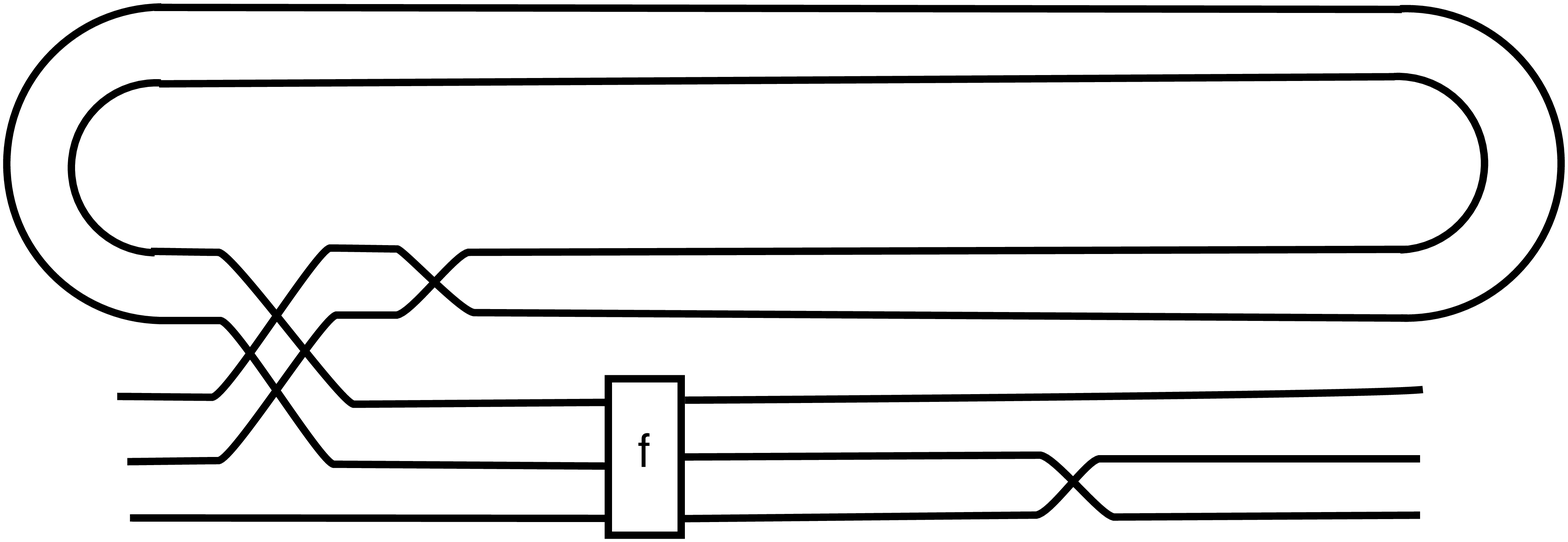}$$
functoriality
$$\includegraphics[height=1in]{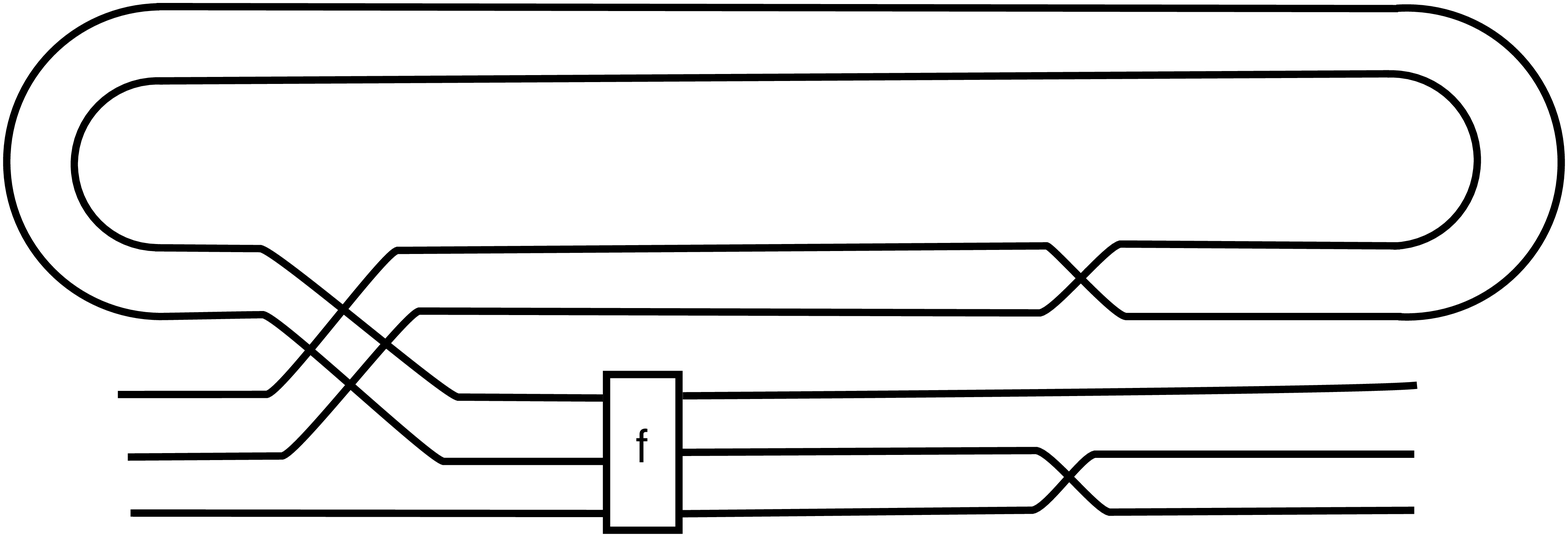}$$
by coherence given by naturality of $\sigma$ and coherence axiom in $\cC$
$$\includegraphics[height=1in]{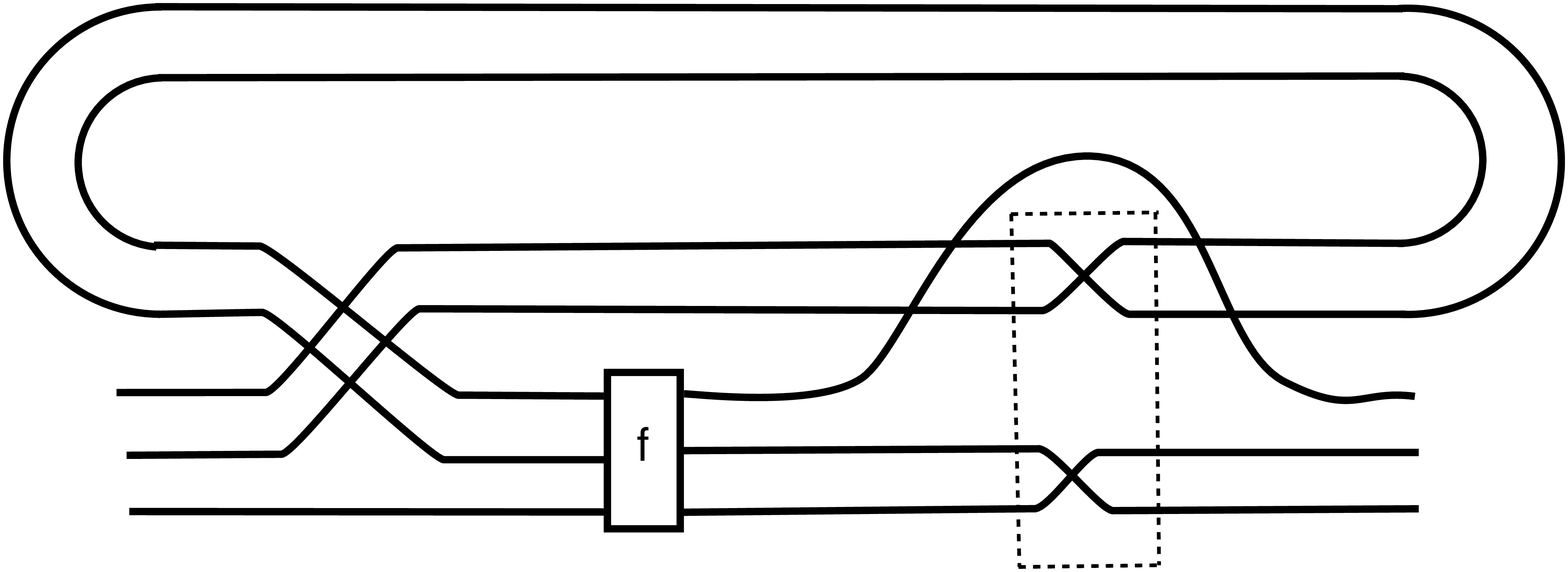}$$
 From now we use the graphical language systematically.
\end{proof}

\begin{lemma}
\label{SIGMA IS A NATURAL TRANSFORMATION}
$\sigma$ is a natural transformation.
\end{lemma}
\begin{proof}
We want to prove that $\sigma\circ (f\otimes g)=(f\otimes g)\circ\sigma$. Notice that we have already proved that $(f\otimes g)\circ\sigma\downarrow$ and $\sigma\circ (f\otimes g)\downarrow$ by Lemma~\ref{EXISTENCE SYMM COMP}. We have by assumption that $\sigma\circ (f\otimes g)$ is defined. In the graphical language this means that $h\in\Trc^{U\otimes V}$, where h is the following diagram:
$$\begin{array}{cc}
\includegraphics[height=0.7in]{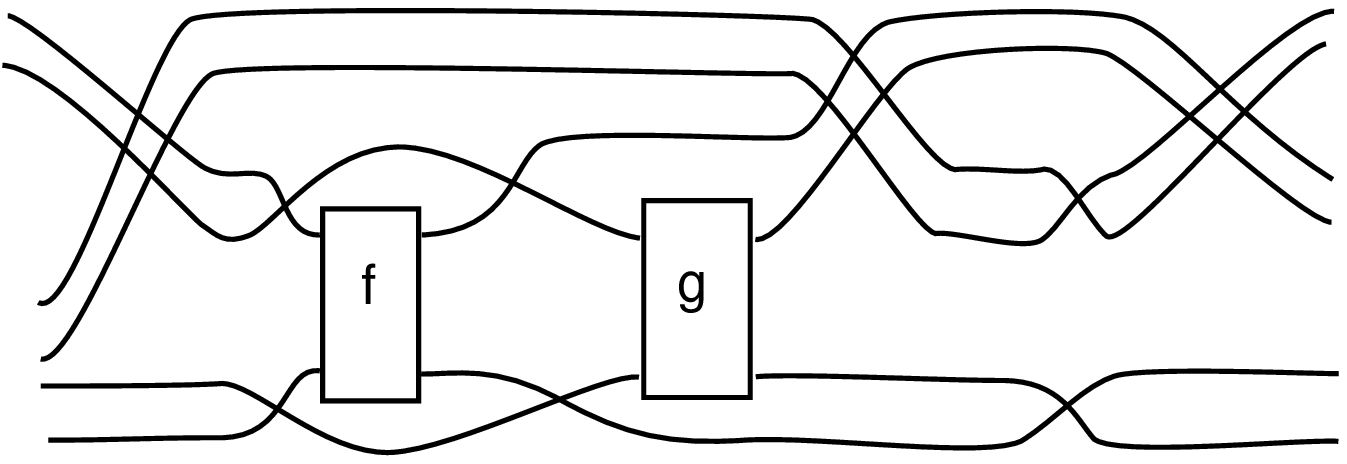} \hspace{0.1cm}&\mbox{the trace given by}\,\,\Tr^{U\otimes V}(h)=
\includegraphics[height=0.7in]{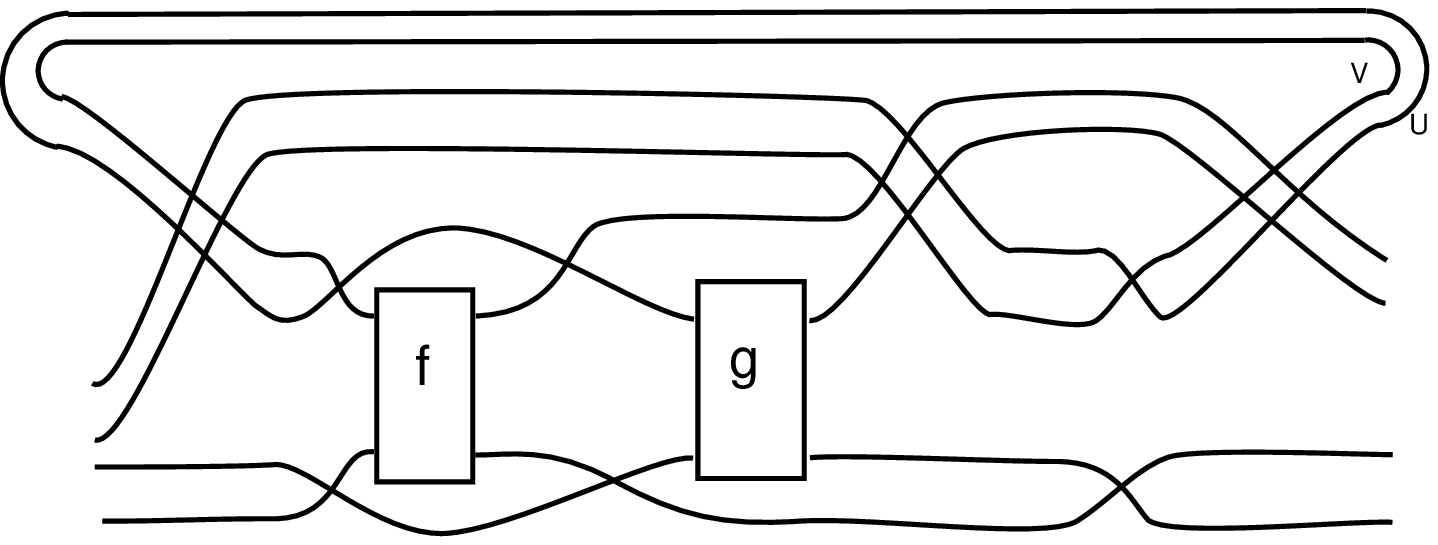}
\end{array}$$
Here the issue is to justify the use of the Vanishing II axiom. Putting the matter schematically without to much emphasis on the name of the objects, we want to split the trace over $U\otimes V$ by using a general hypothesis of type $h\in\Trc^V$ and a conditional hypothesis of type $h\in\Trc^{U\otimes V}$ and we must prove that $\Tr^{V}(h)\in\Trc^{U}$.
This is the kind of back and forward process of proof that we have repeatedly used before where the justification of the use of the axiom is also the proof that we need.
Let us start by considering the following diagram:
\begin{equation}
\includegraphics[height=0.6in]{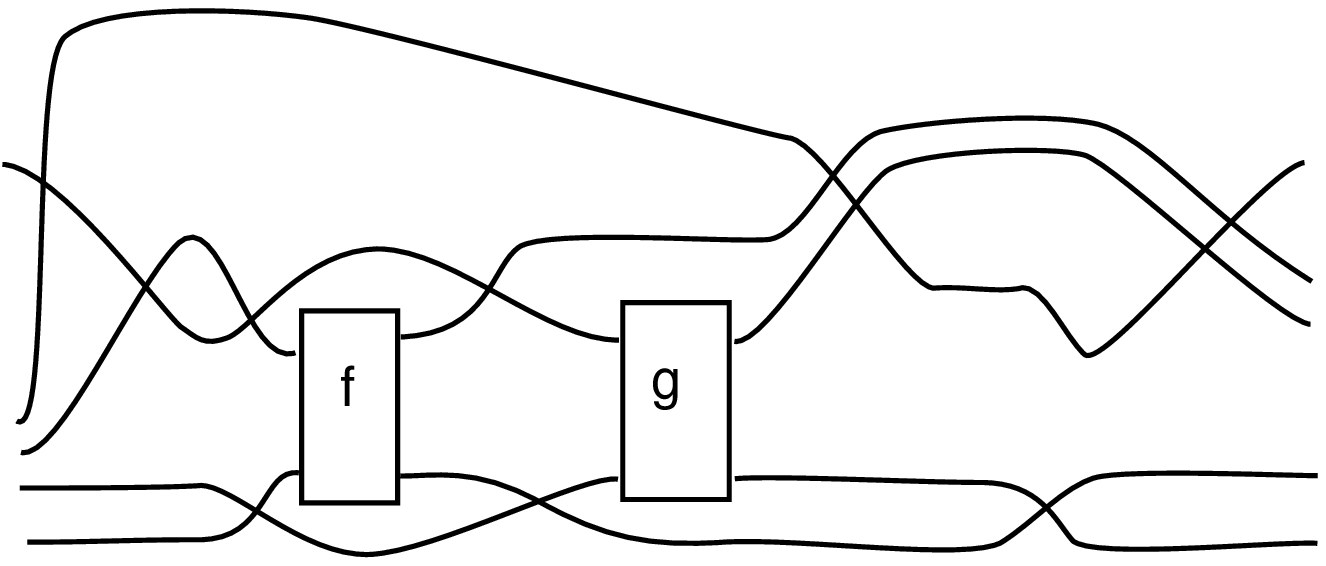}.
\end{equation}
 Then by the yanking axiom, which is totally defined: $\sigma\in\Trc^V$ and we can replace the former graph by this one $$\includegraphics[height=0.6in]{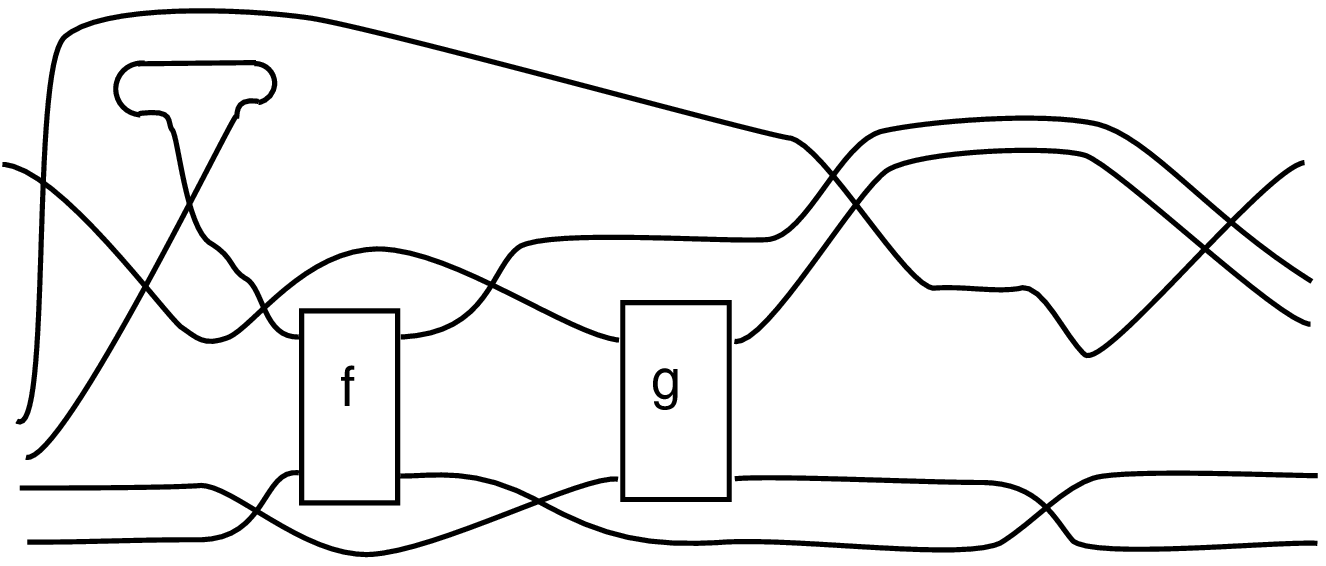},$$
and by the superposition axiom (both versions) we have that locally the diagram satisfies that it is part of the trace class $\Trc^V$ and the graph, after tracing it out is given by
$$\includegraphics[height=0.6in]{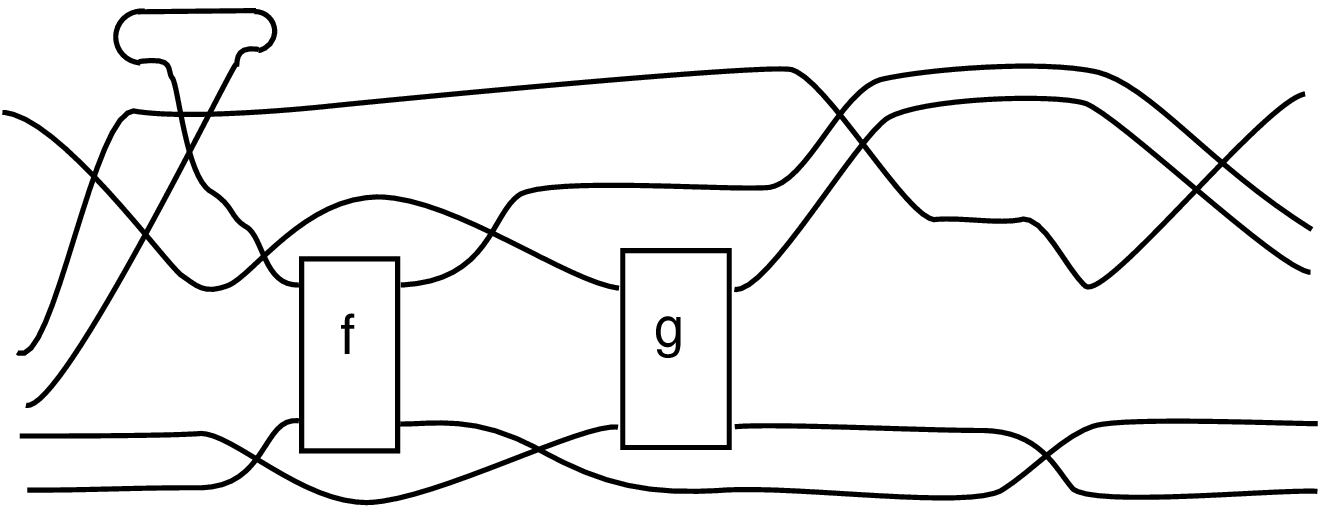}.$$
Then the naturality axiom allows us to include the full diagram in the trace class $\Trc^V$ and we are allowed also to trace it: $$\includegraphics[height=0.6in]{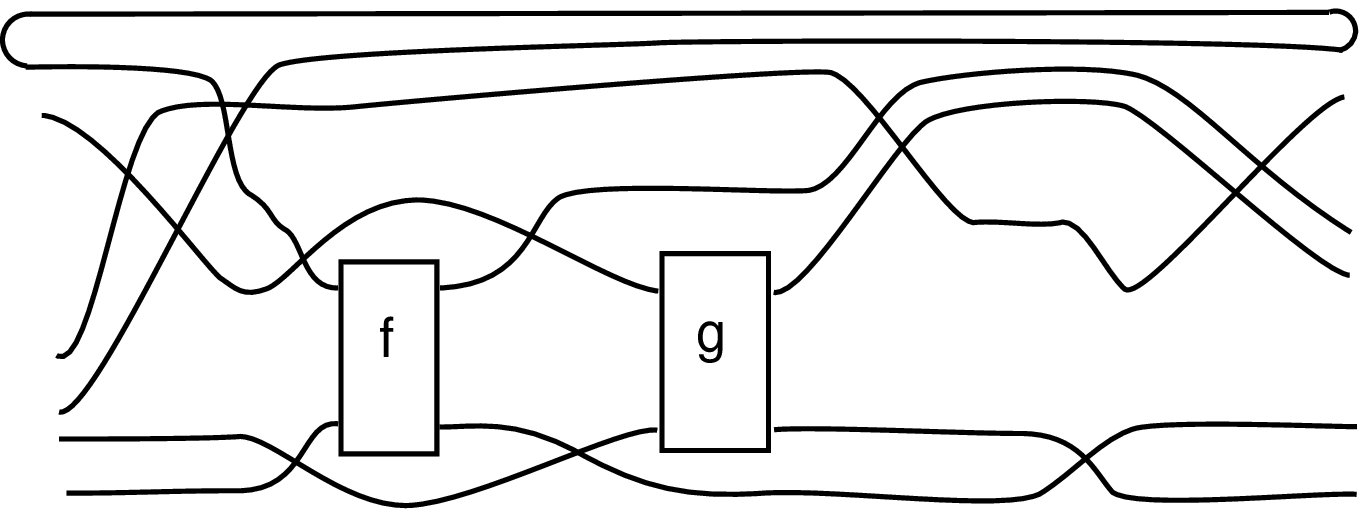}.$$
Finally, by coherence
$$\begin{array}{cc}
\includegraphics[height=0.6in]{naturaA2.eps}\in\Trc^{V} \hspace{0.1cm}&\mbox{and the trace is represented by\,\,\,\,\,\,\,\,} \includegraphics[height=0.6in]{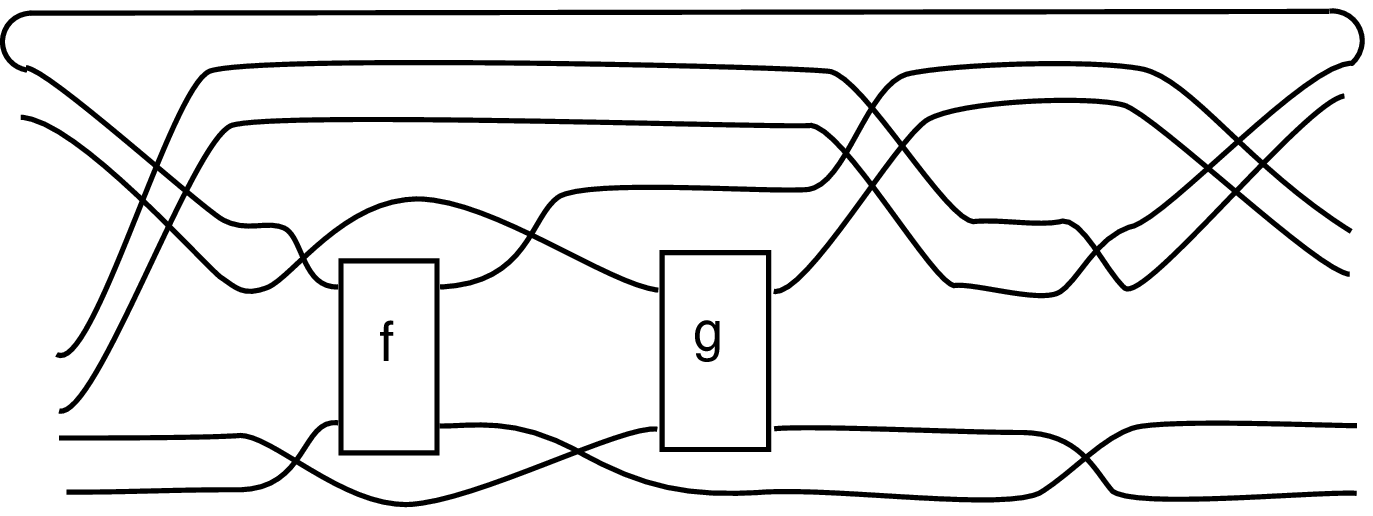}.
\end{array}$$
Now we are in a position where we can use the vanishing II axiom and to conclude that $\Tr^V(h)\in\Trc^U$ and of course the value of the trace is given by $\Tr^U(\Tr^V(h))=\Tr^{U\otimes V}(h)$.

After justifying the use of the vanishing II axiom we move to ensure that both diagram are equal.
First notice that the following diagrams are equivalent :
$$\begin{array}{cc}
\includegraphics[height=0.6in]{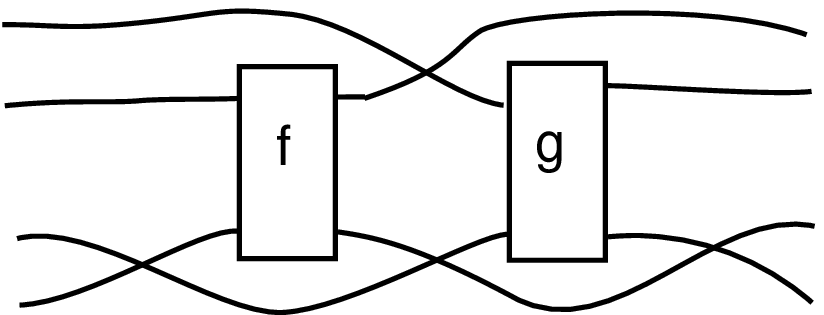}
\hspace{0.3cm}&\mbox{by coherence\,\,\,\,\,\,} \includegraphics[height=0.6in]{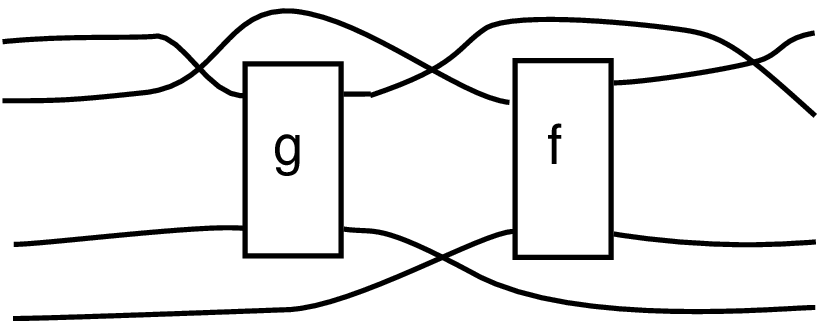}.
\end{array}$$

Starting with the last diagram and applying the axioms, where the existence of the trace is justified by the axiom that we are mentioning, we obtain:
$$\includegraphics[height=0.5in]{natura13.eps}
\,\mbox{coh.}\,\,
\includegraphics[height=0.6in]{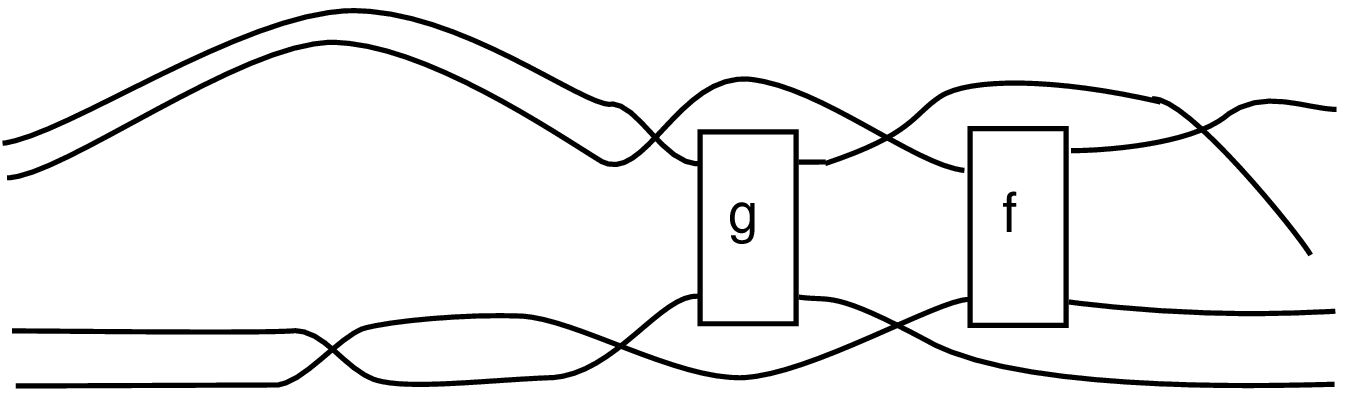}
\,\,\mbox{yank.}\,
\includegraphics[height=0.6in]{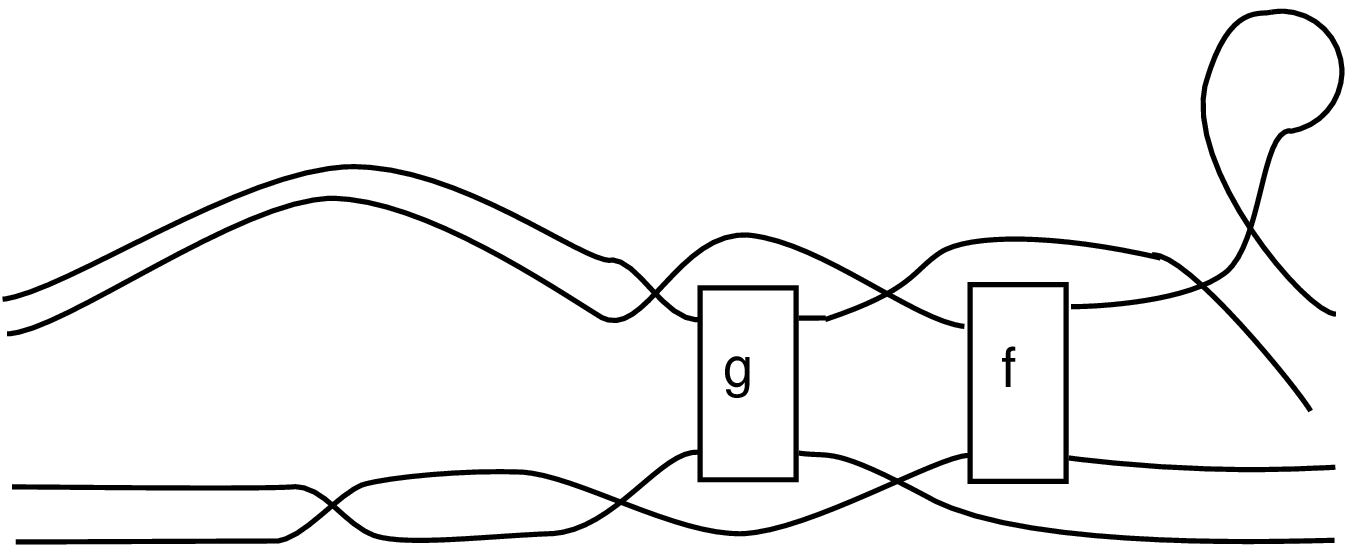}
\,\mbox{nat.}\,\,\,\,$$
\includegraphics[height=0.6in]{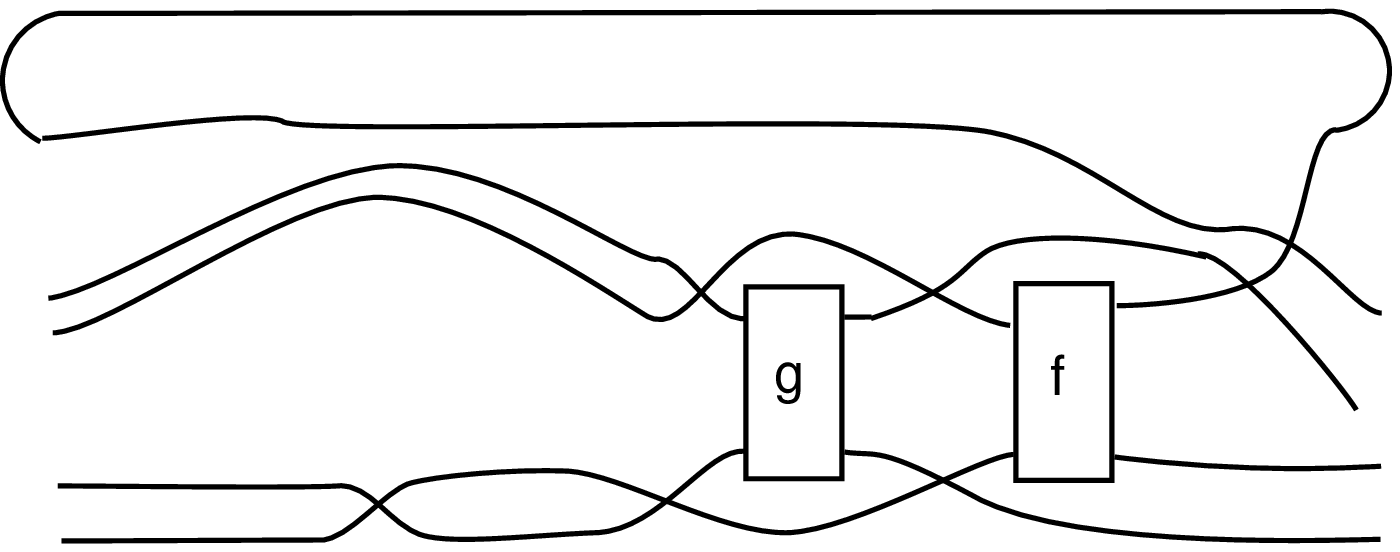}
$\,\,\,\,\mbox{coh.}\,\,\,\,$
\includegraphics[height=0.6in]{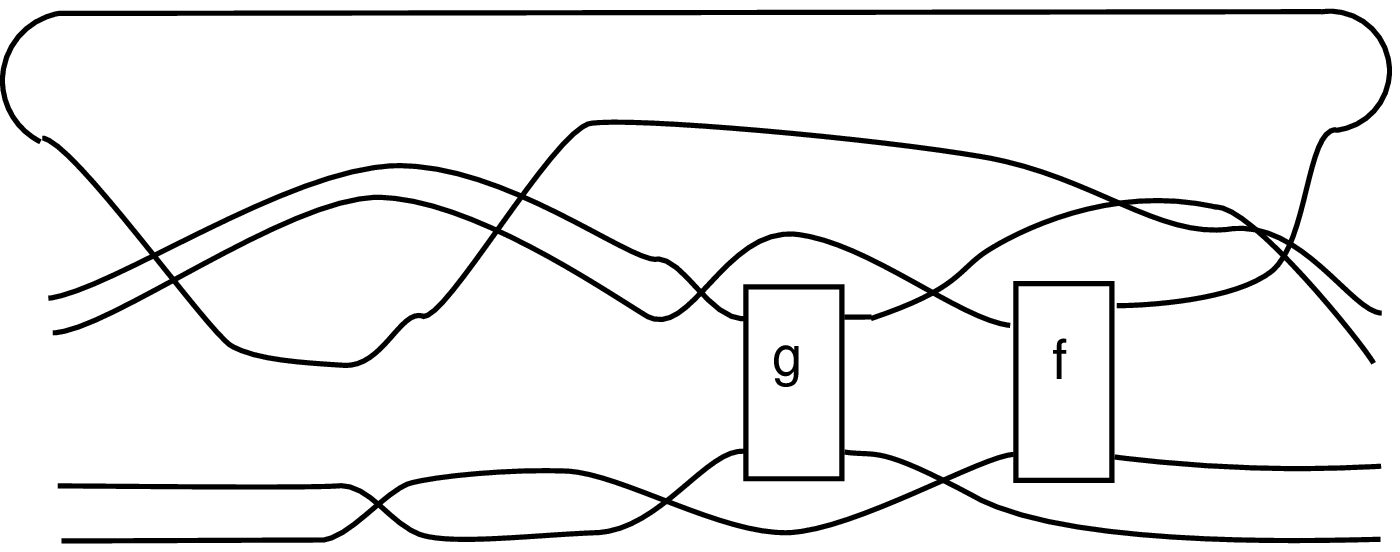}
$\,\,\,\,\mbox{yank., sup.}\,$
\includegraphics[height=0.6in]{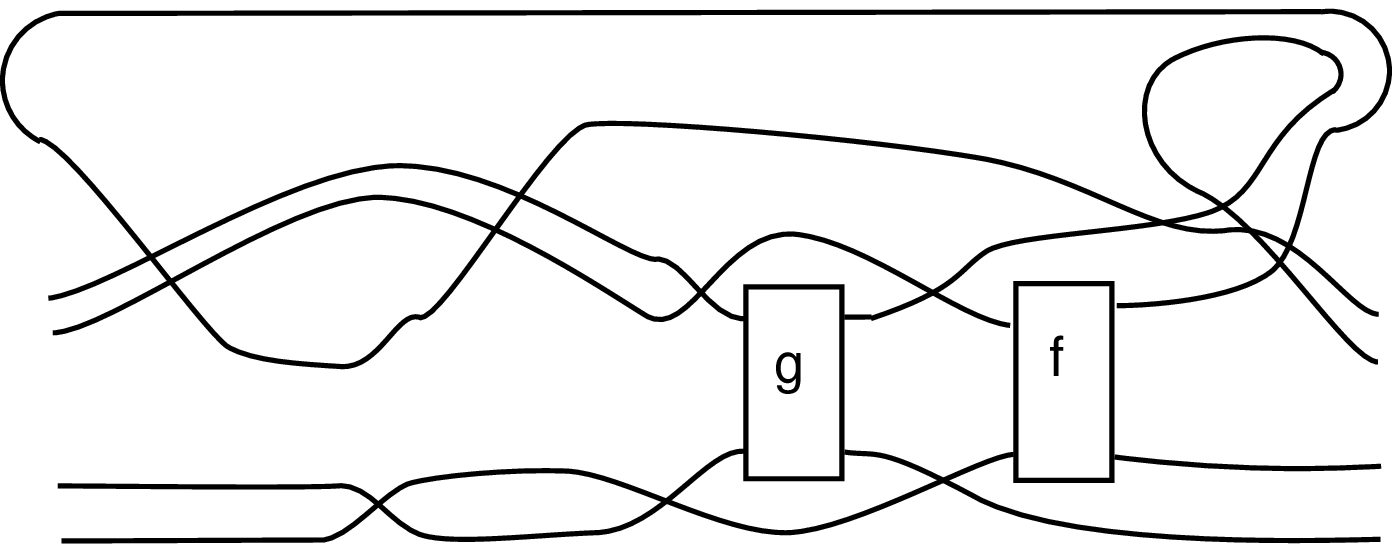}
$\,\mbox{coh.}\,$

\includegraphics[height=0.6in]{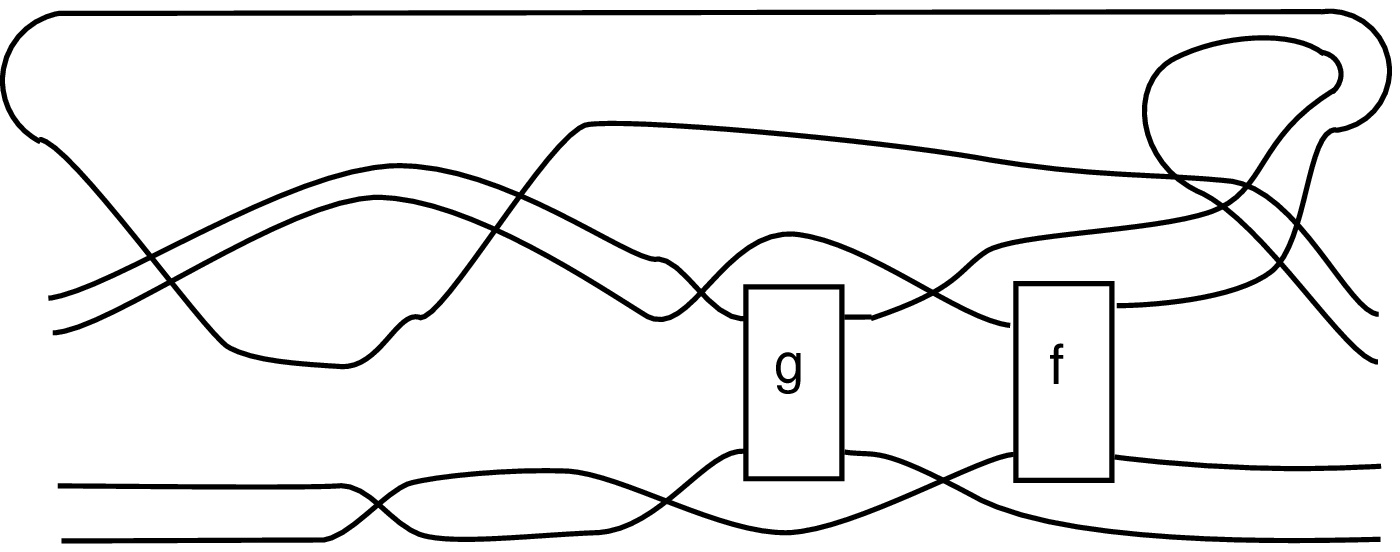}
$\,\,\,\,\mbox{nat.}\,\,\,$
\includegraphics[height=0.6in]{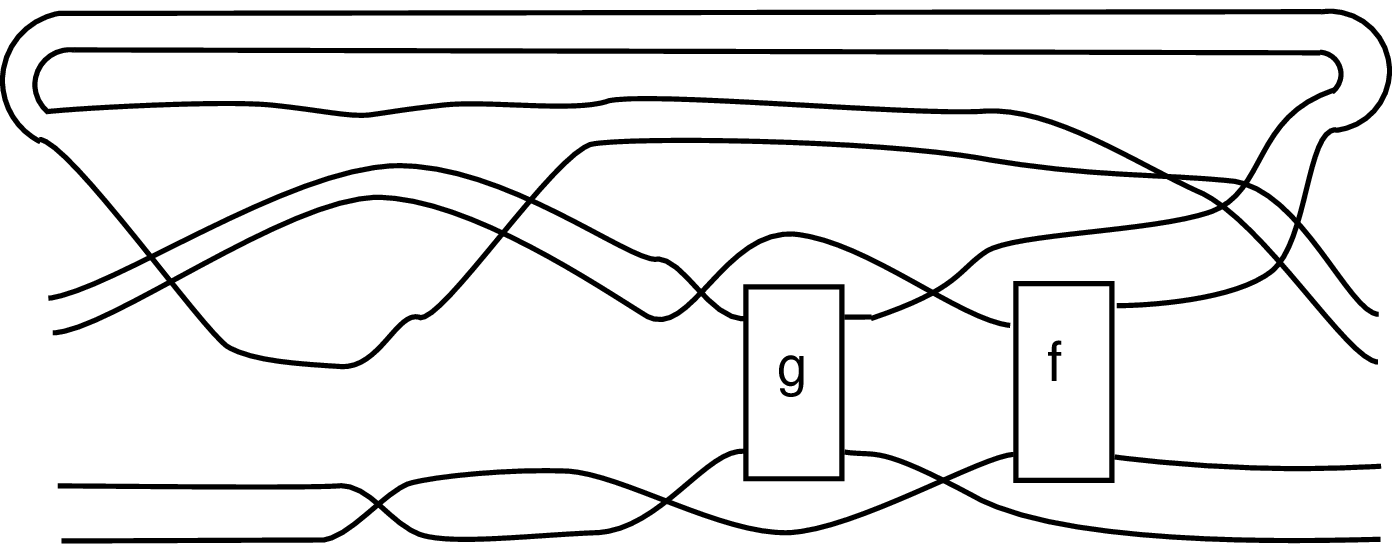}
$\,\,\,\,\mbox{coh.}\,\,\,\,$
\includegraphics[height=0.6in]{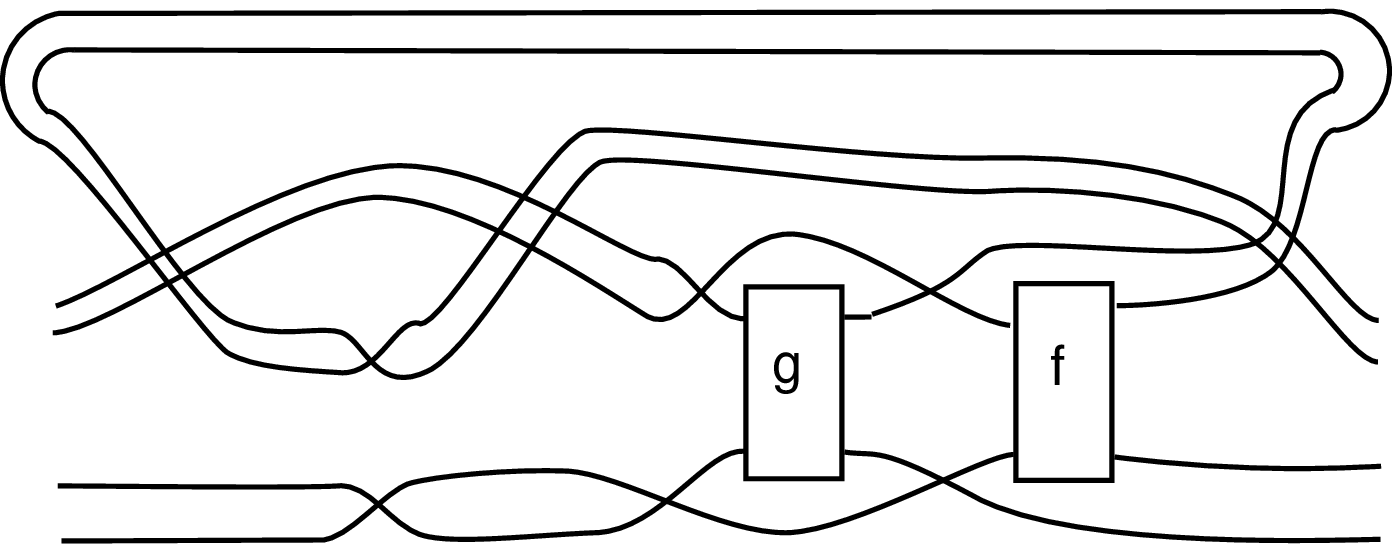}

where the last diagram is of type $\Tr^U(\Tr^V(h'))$. Therefore, by the same reasoning as given at the beginning of the proof we find that $\Tr^{U\otimes V}(h')=\Tr^U(\Tr^V(h'))$ also for this diagram.
As before we repeat our arguments to justify the existence and value of the trace for the case when we start with the graph
$$\includegraphics[height=0.6in]{naturaA12dia.eps}$$
obtaining the following diagram:
$$\includegraphics[height=0.8in]{naturaA1.eps}.$$
\end{proof}

\begin{lemma}
$[\sigma\otimes 1,\sigma]=1\otimes \sigma$.
\end{lemma}
\begin{proof}
Here again, as in Lemma~\ref{SIGMA IS A NATURAL TRANSFORMATION}, the key point is to justify the use of the Vanishing II axiom. We will apply this strategy twice. Since by Lemma~\ref{EXISTENCE SYMM COMP}  $[\sigma\otimes 1,\sigma]\downarrow$ we want to be able to use an scheme proof of type:  $g\in \Trc^{U\otimes V\otimes W}$ iff $\Tr^W(g)\in \Trc^{U\otimes V}$, but for this we need an hypothesis of type $g\in \Trc^{W}$. To justify this,
we want to split the trace over $U\otimes V$ by using a general hypothesis of type $h\in\Trc^V$ and a conditional hypothesis of type $h\in\Trc^{U\otimes V}$ and we must prove that $\Tr^{V}(h)\in\Trc^{U}$.

We start with the following diagram that represents $1\otimes\sigma$:
$$\includegraphics[height=1in]{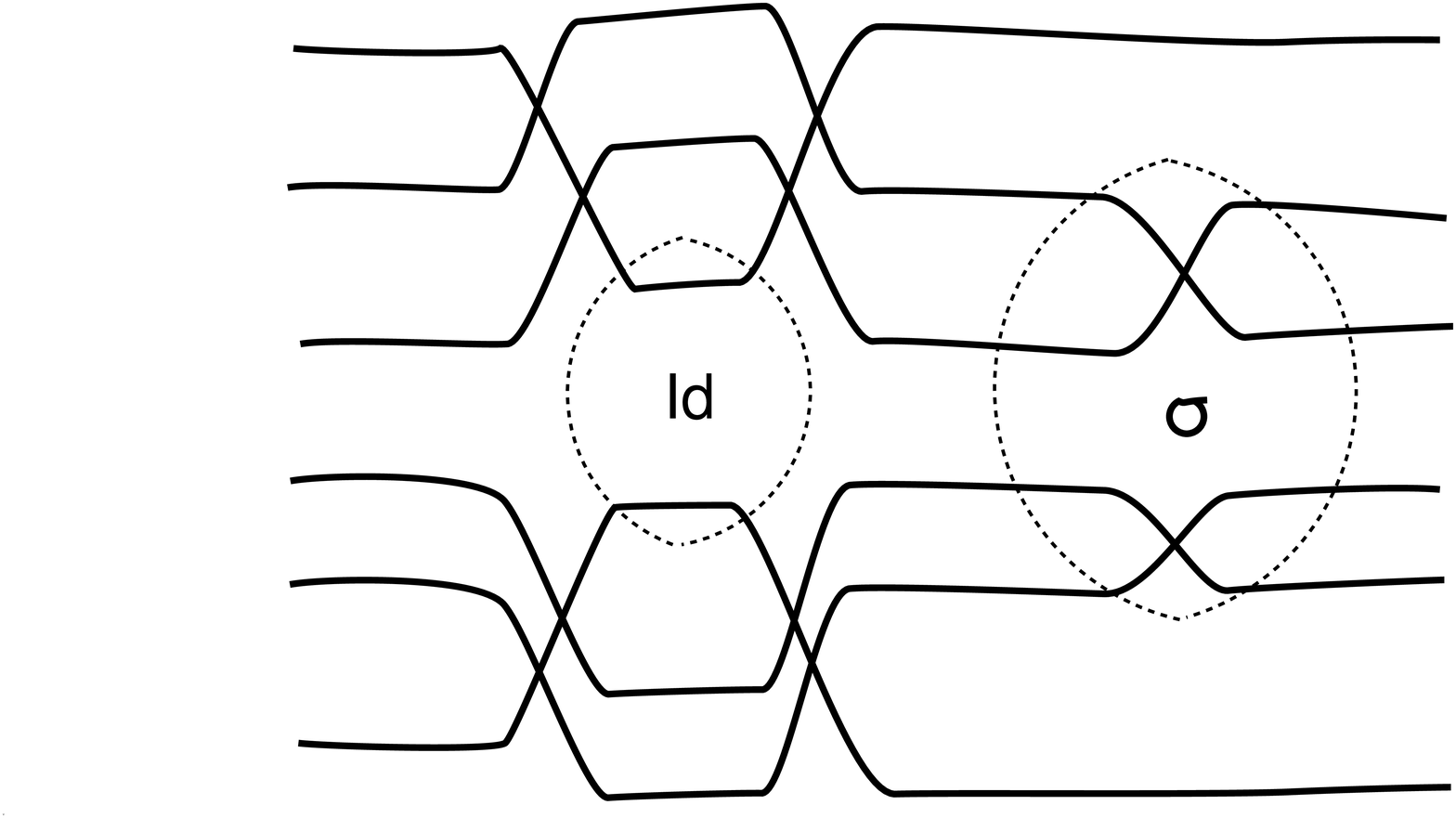}$$
which by coherence is equivalent to the following

$$\includegraphics[height=1in]{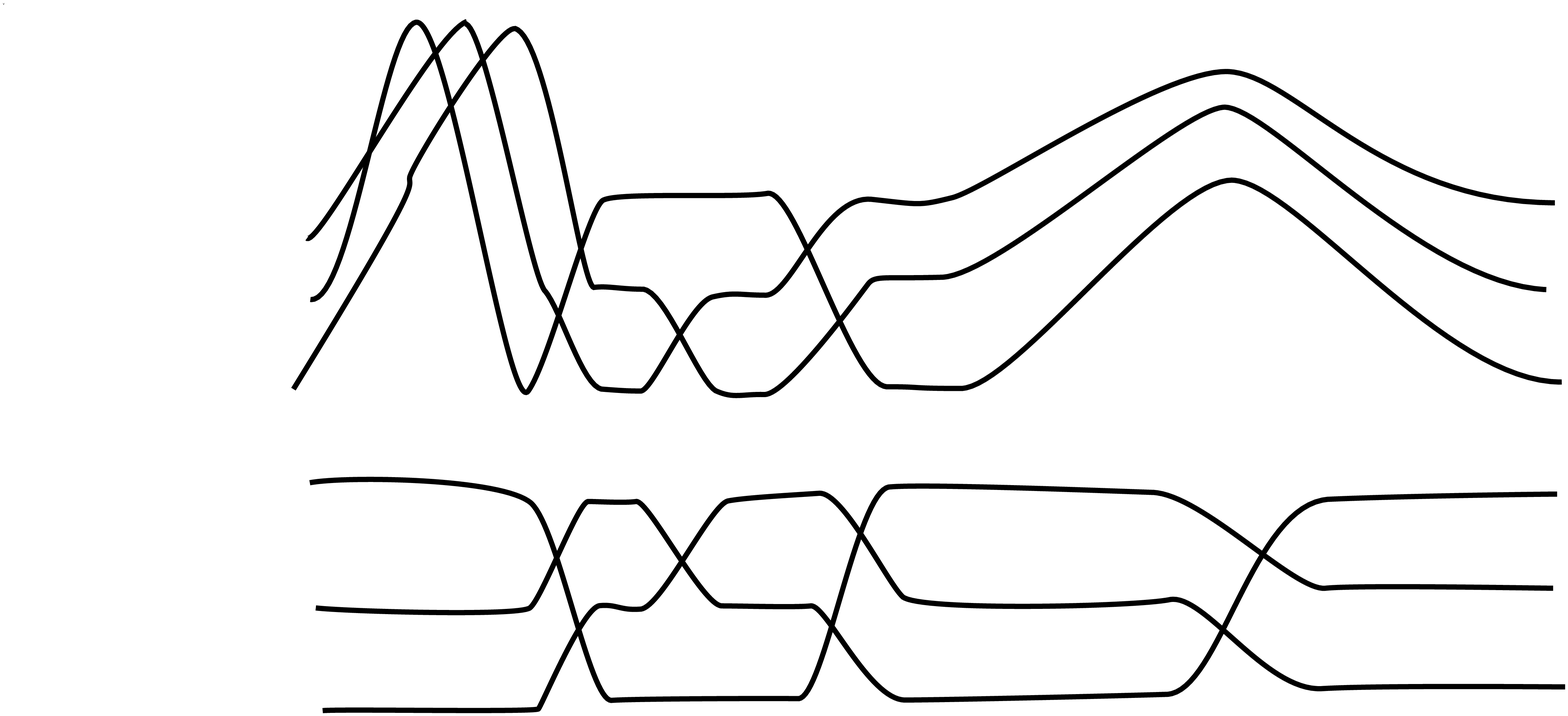}$$
 Therefore, by yanking, (let us name the variable $U$) we have

$$\includegraphics[height=1in]{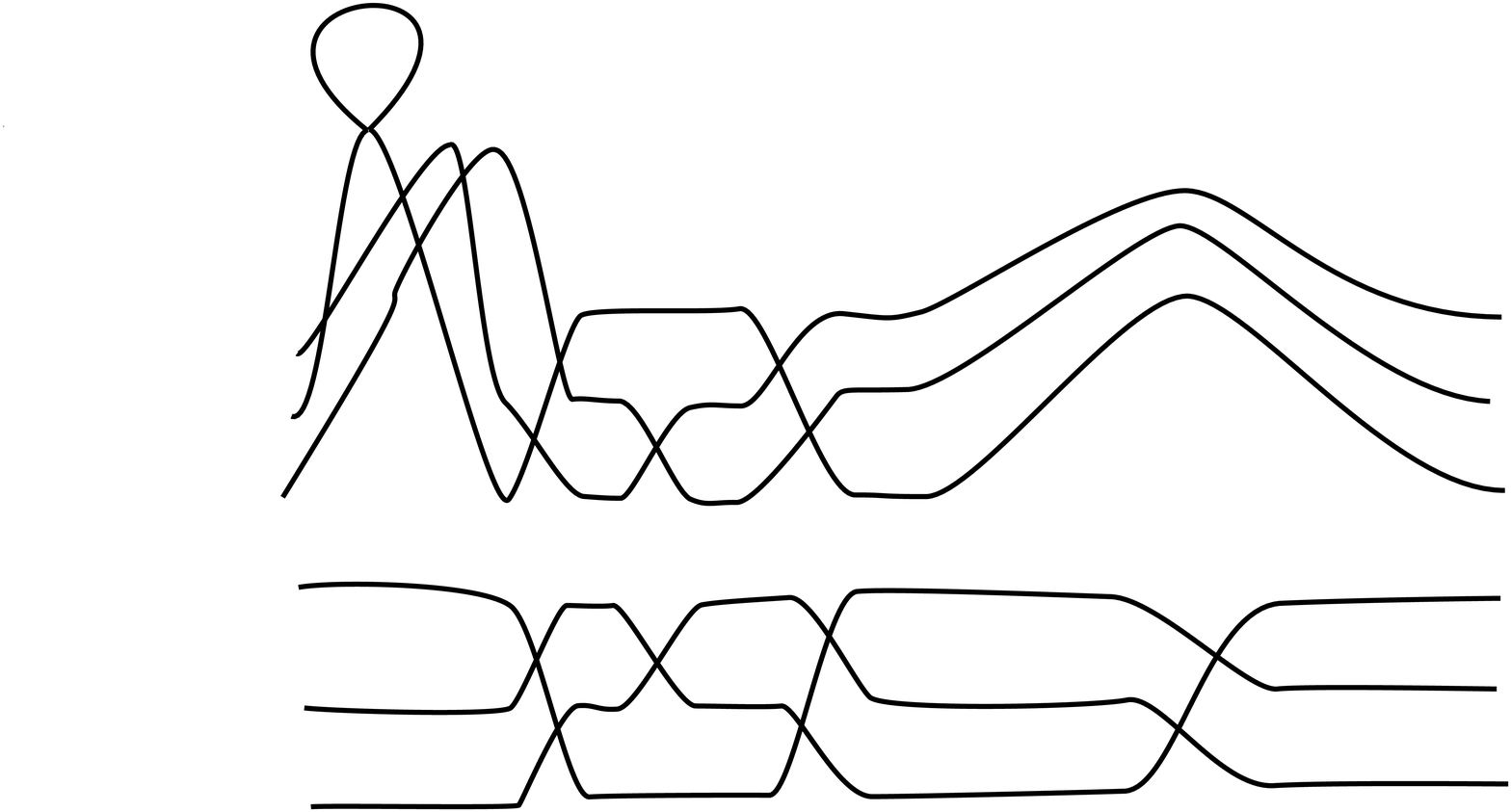}$$

Then by the naturality and the superposition axioms we obtain that it is equal to the trace represented by:

$$\includegraphics[height=1in]{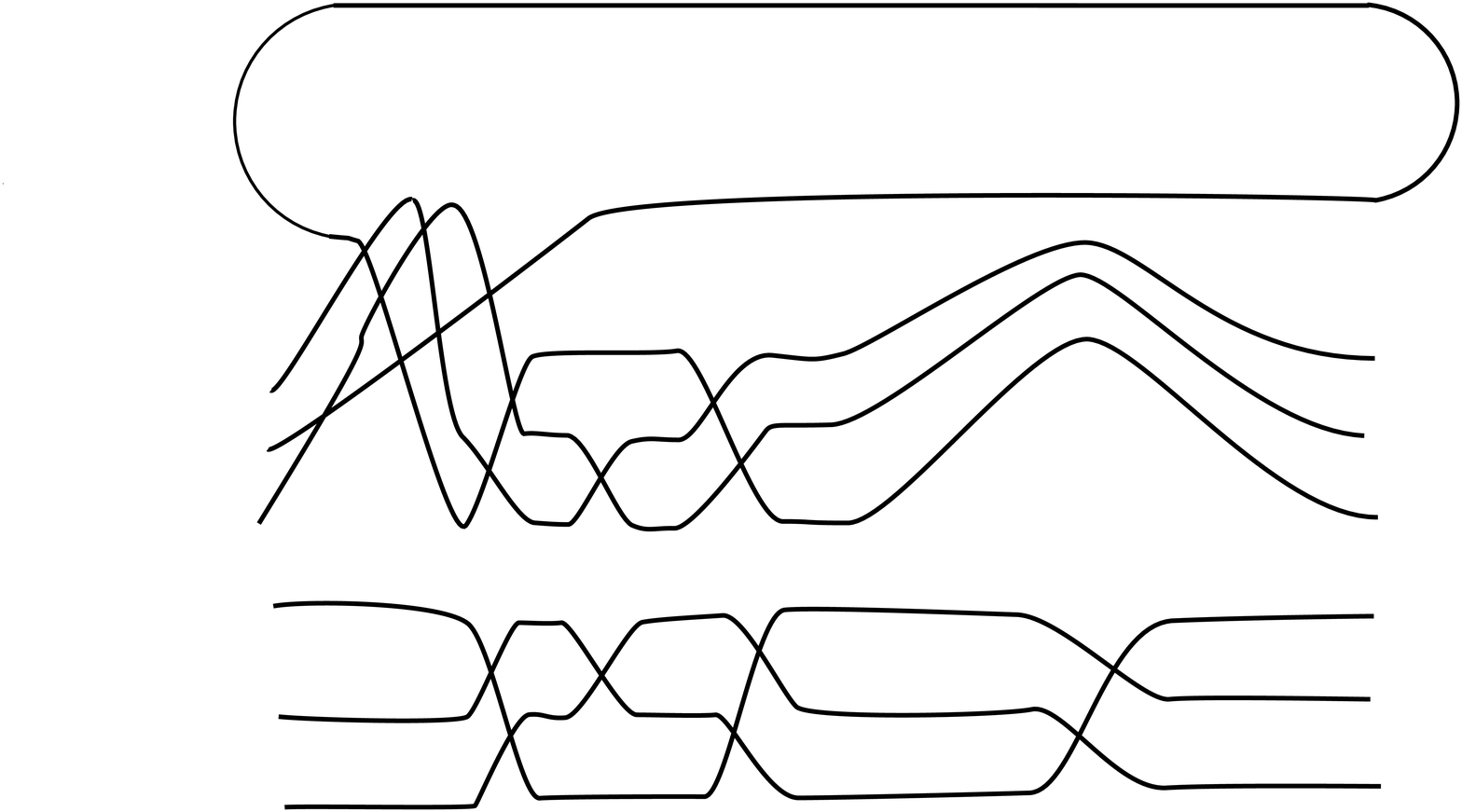}$$
in which the diagram below the trace, let us call it $h$, satisfies $h\in\Trc^U$. Notice that this is true because our axioms of partially traced category allow us to entail this last statement.

Now, by coherence we have that is equal to
$$\includegraphics[height=1in]{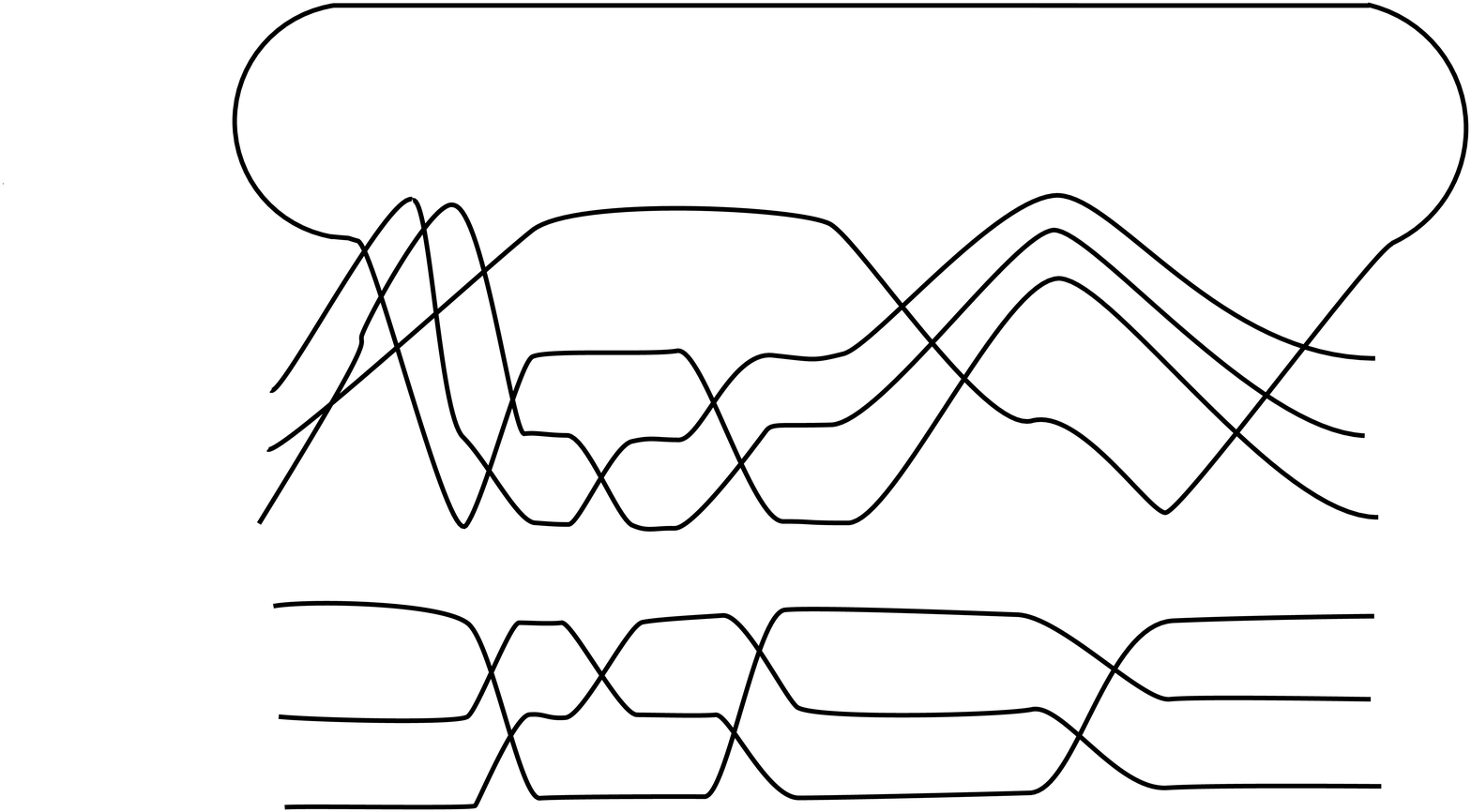}$$
Let us still call $h$ the new graph below the trace.
By yanking with respect to a variable $V$
$$\includegraphics[height=1in]{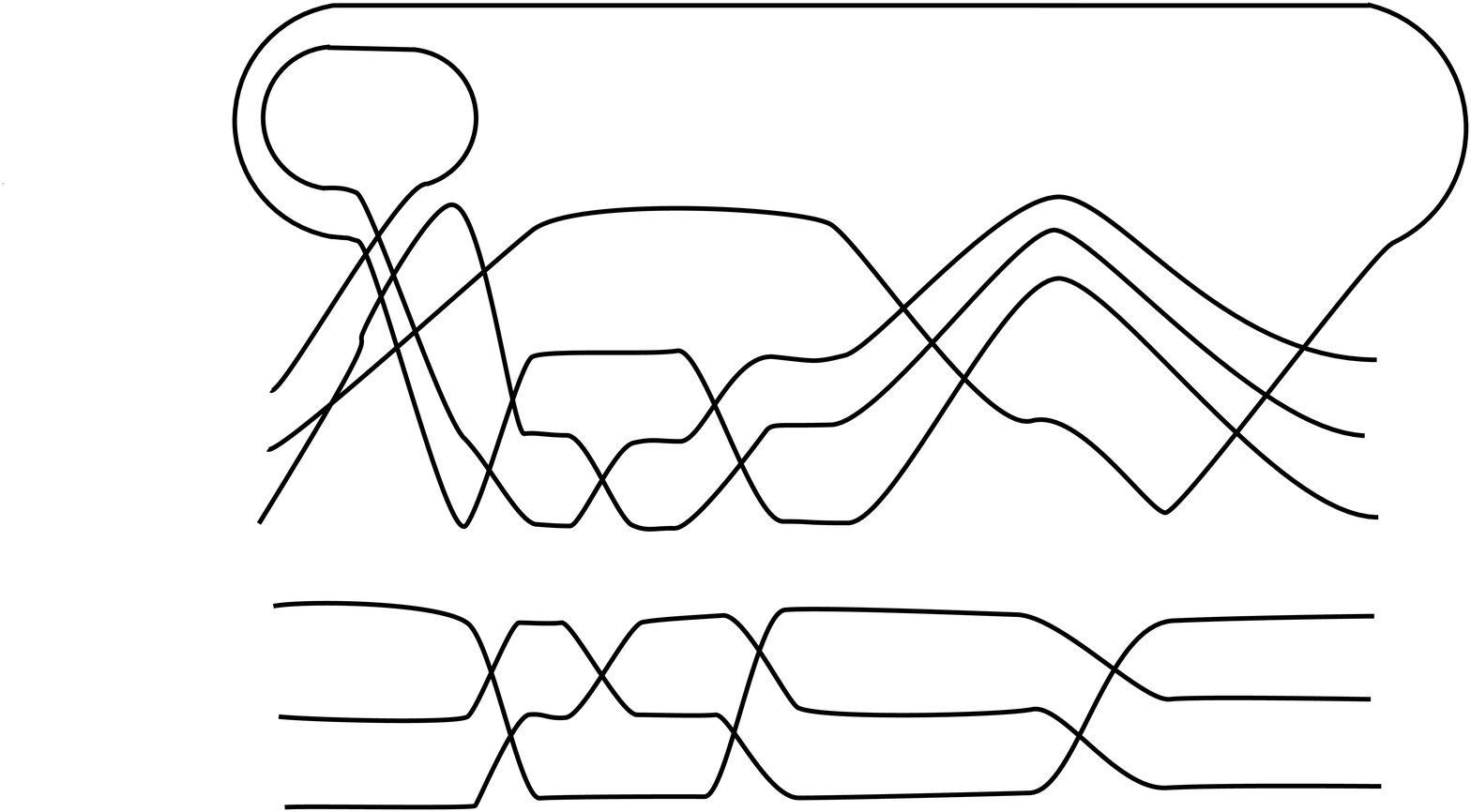}$$

Now again by naturality, superposition and coherence we conclude that the graph below the trace, name it $h'$, is in the trace class $\Trc^V$. Moreover, the value of the trace along $V$ is equal to $h$, i.e., $\Tr^V(h')=h$, which implies that is in the trace class $\Trc^U$, this means that we are allowed to use vanishing II and to conclude that $h'\in\Trc^{U\otimes V}$:

$$\includegraphics[height=0.8in]{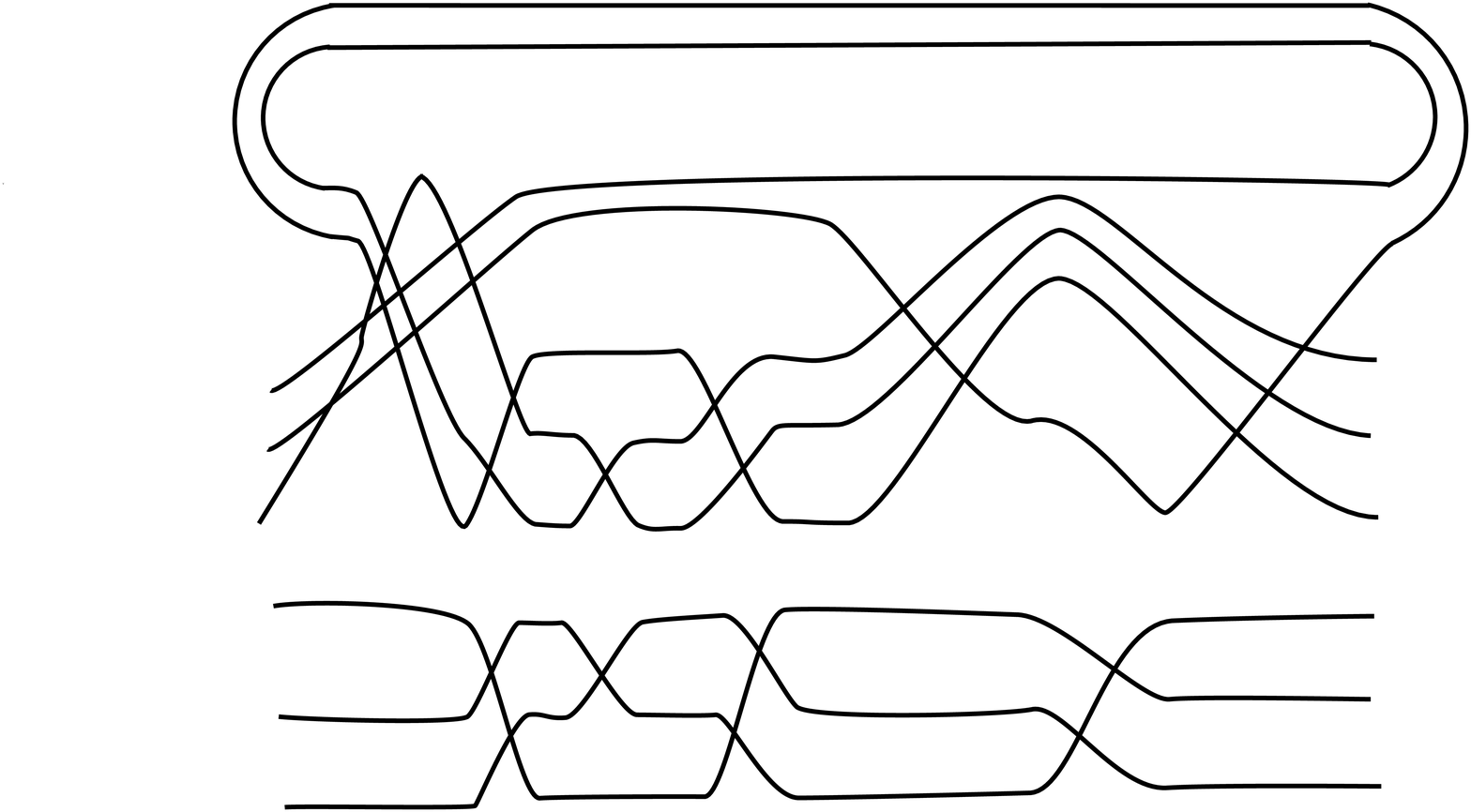}\,\,\,\,\,\,\mbox{and coherence}\,\,\,\,\,\includegraphics[height=0.8in]{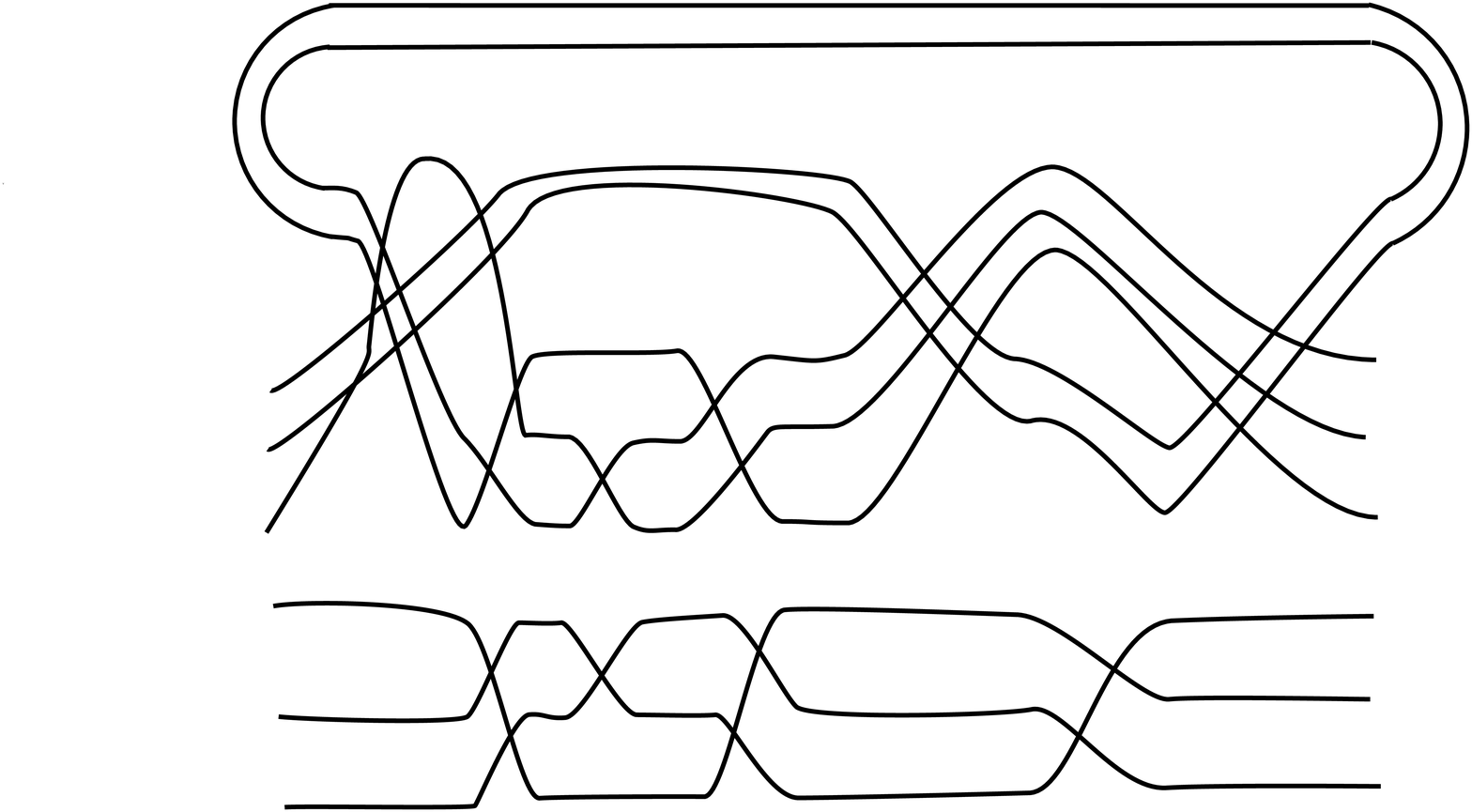}$$
Now we repeat the idea with a new parameter $W$.
$$\includegraphics[height=1in]{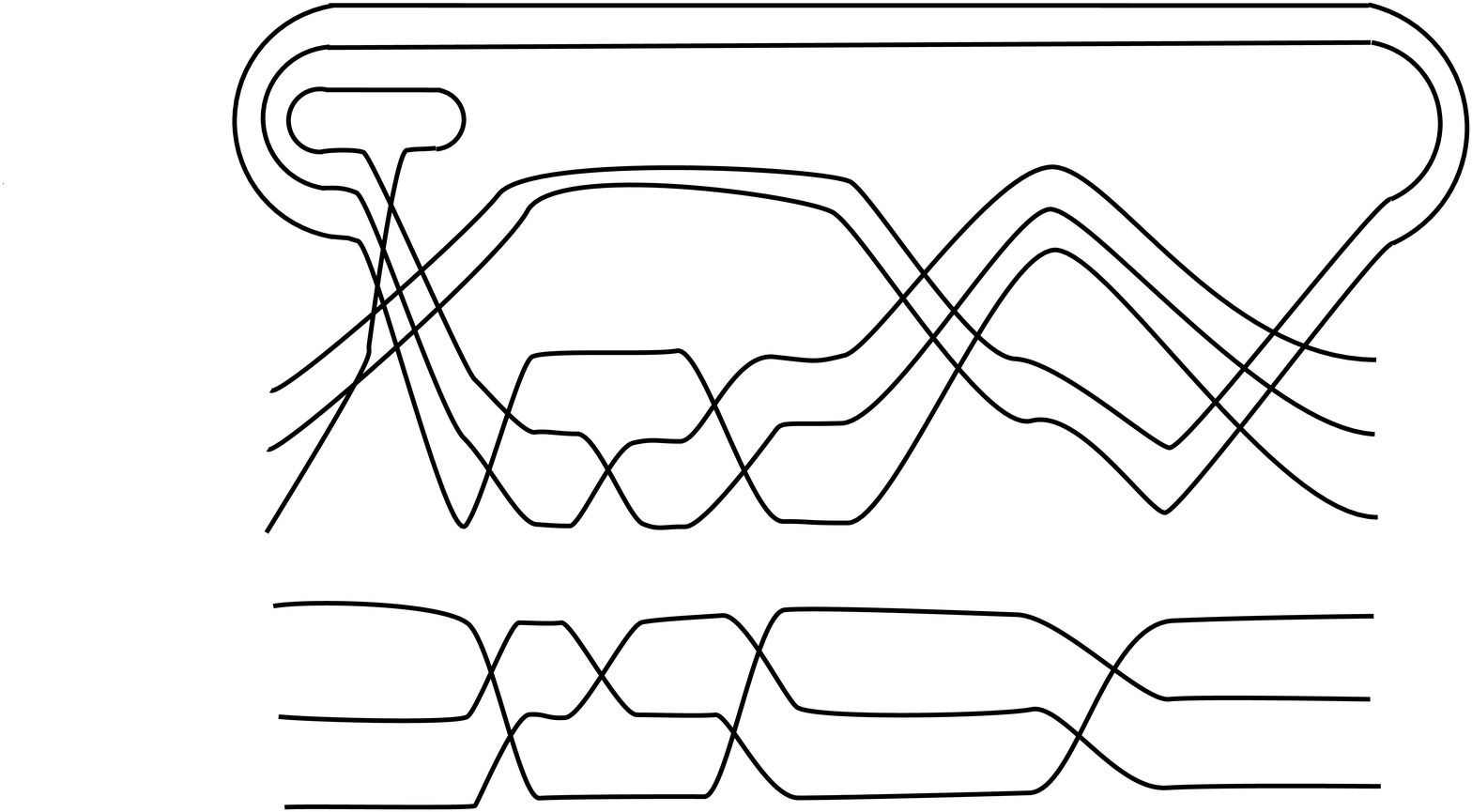}
\,\,\,\,\,\mbox{and}\,\,\,\,\,
\includegraphics[height=1in]{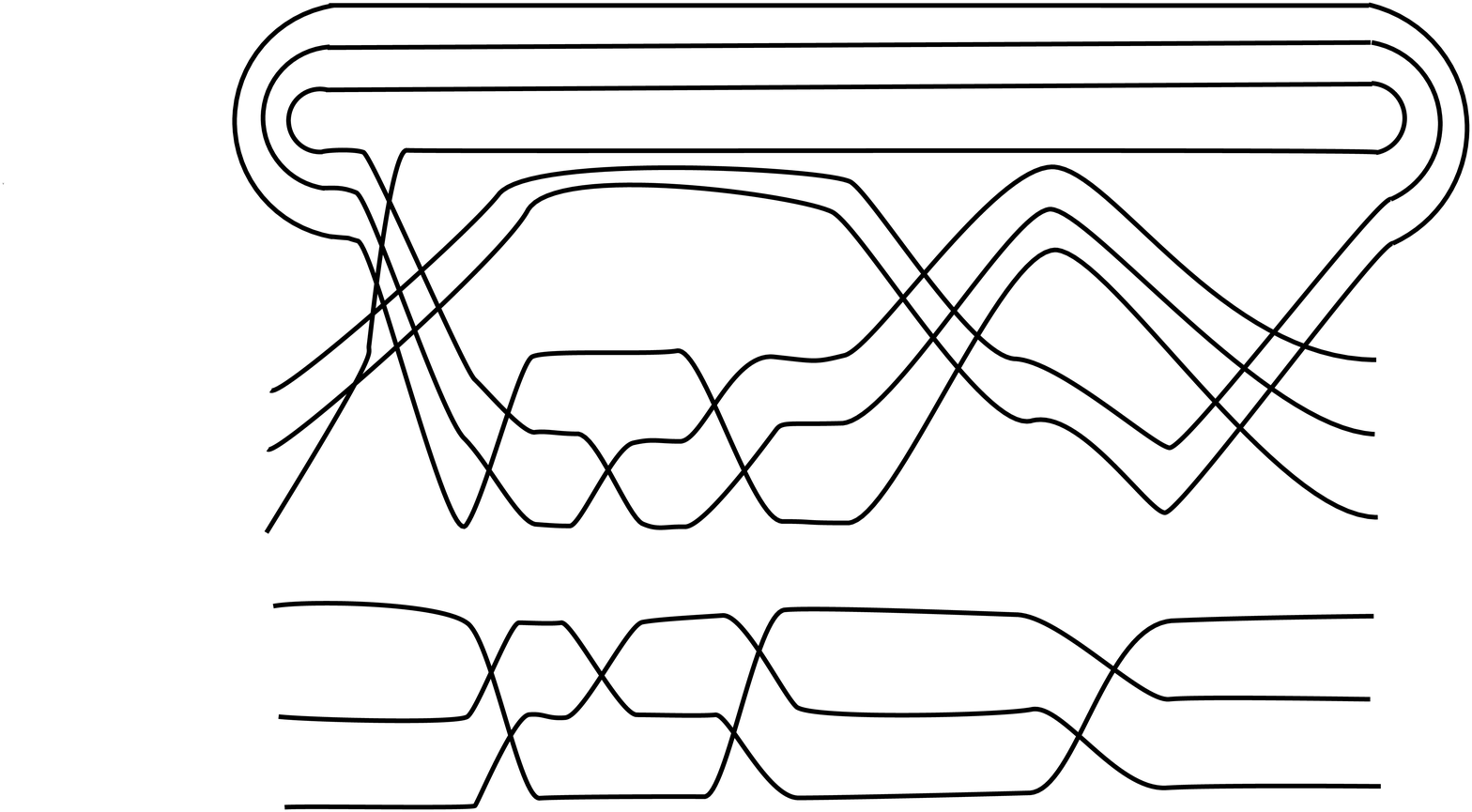}$$
Hence, this yields after applying vanishing II again
$$\includegraphics[height=1.2in]{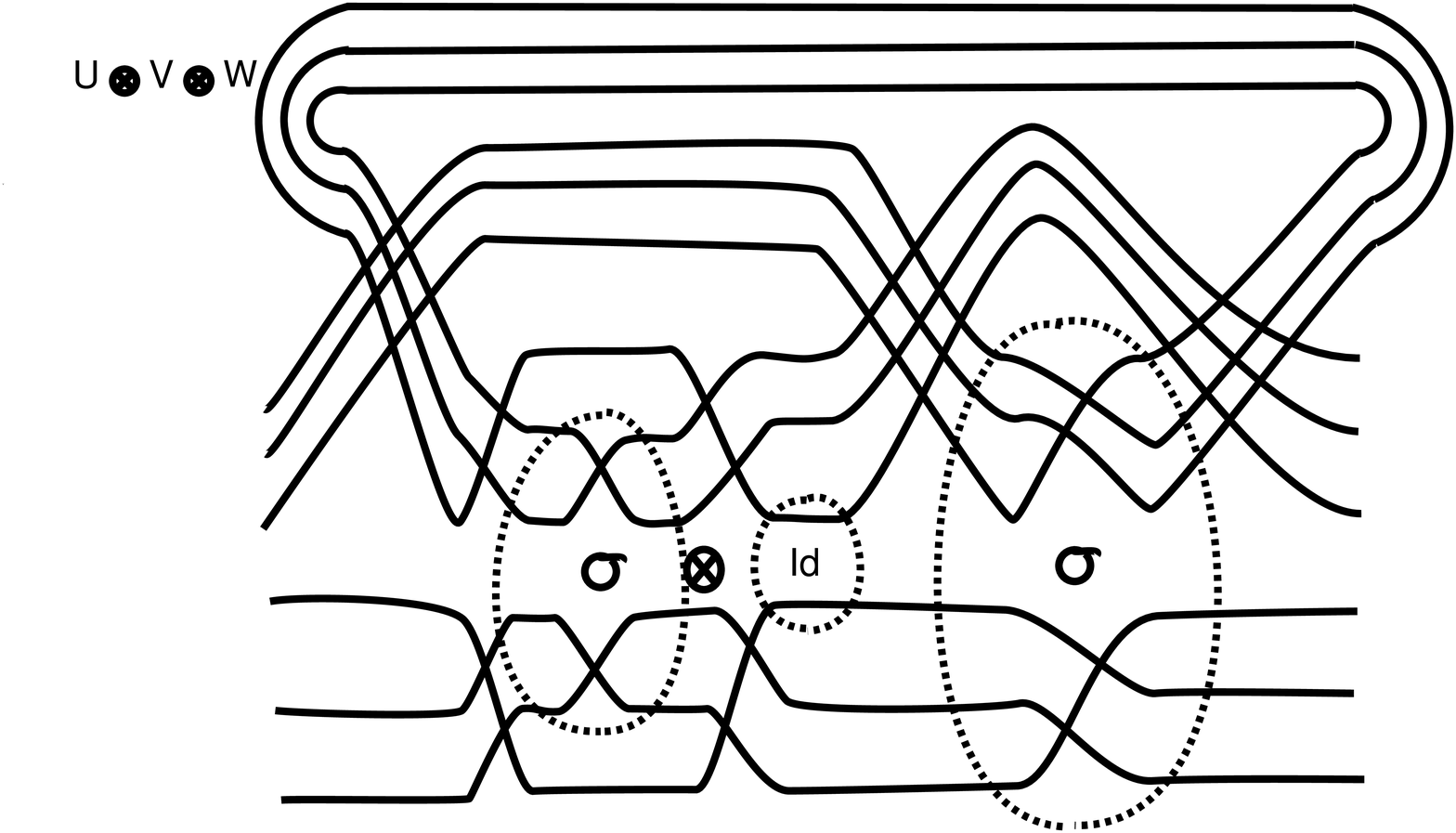}$$
which represents $[\sigma\otimes 1,\sigma]$.
\end{proof}
\begin{lemma}
 $1\otimes [\, \vec{p}\, ]\funnel [\, 1\otimes \vec{p}\, ]$ and $[ \,\vec{p}\, ]\otimes 1 \funnel [\, \vec{p}\otimes 1 \,]$.
\end{lemma}
\begin{proof}
Without loss of generality we consider the case when $\vec{p}=p_1,p_2,p_3,p_4$.
By definition $[ \,\vec{p}\, ]\otimes 1$ is equal to:
\begin{center}
\includegraphics[height=1.2in]{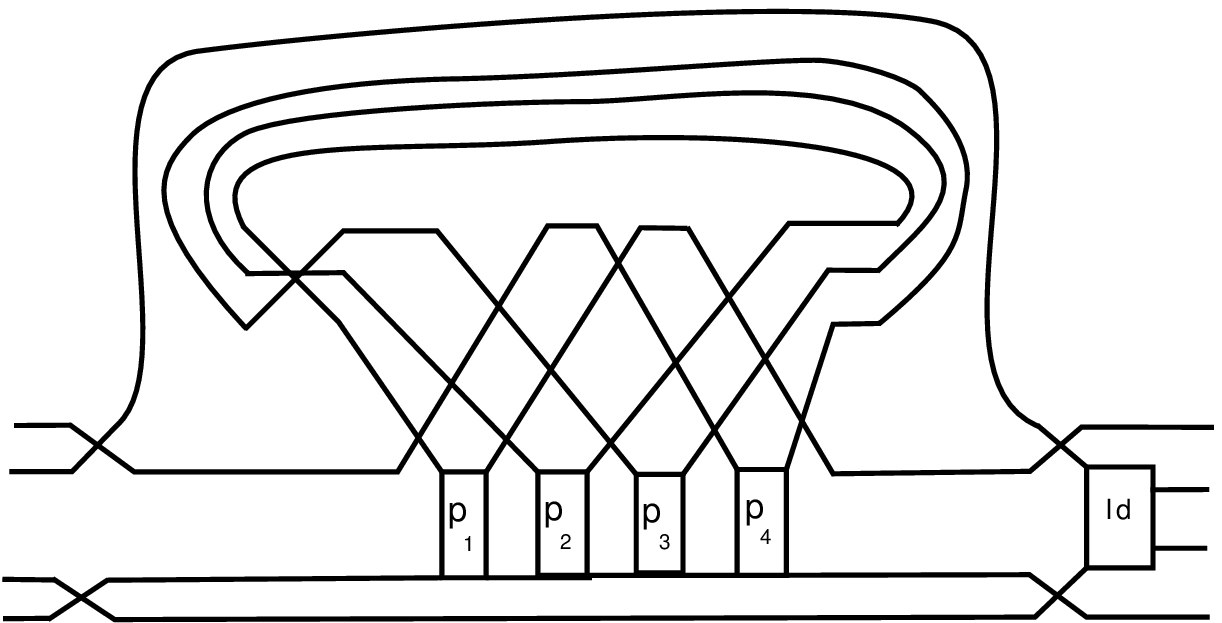}
\end{center}

Then using superposing axiom we obtain:
\begin{center}
\includegraphics[width=2in]{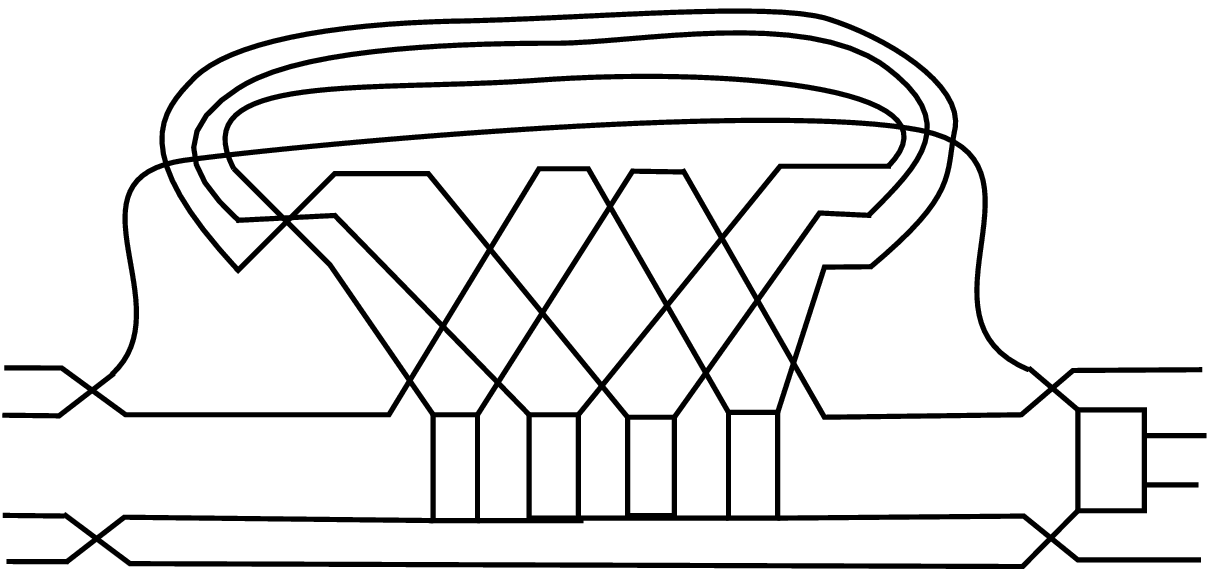}
\end{center}
and since by the yanking axiom $\Tr(\sigma)=1$, we have that:
\begin{center}
\includegraphics[width=2in]{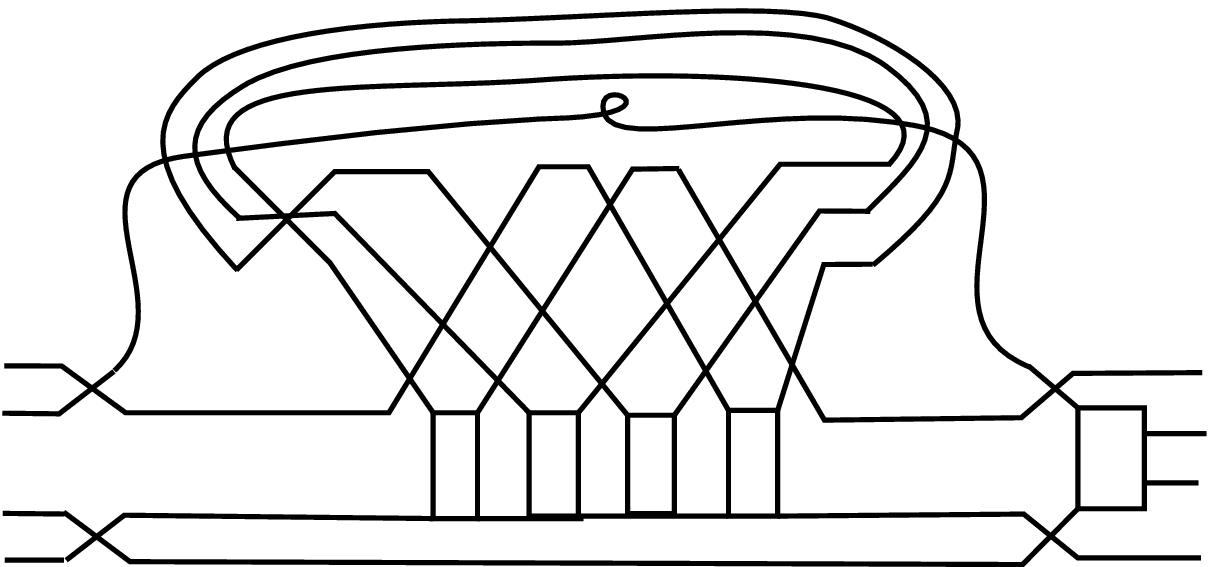}
\end{center}

Now by the fact that the trace is defined on symmetries this is the hypothesis that I need in order to apply superposing (equivalent version) axiom, thus by the same reason we can apply also the naturality axiom:
\begin{center}
\includegraphics[width=2in]{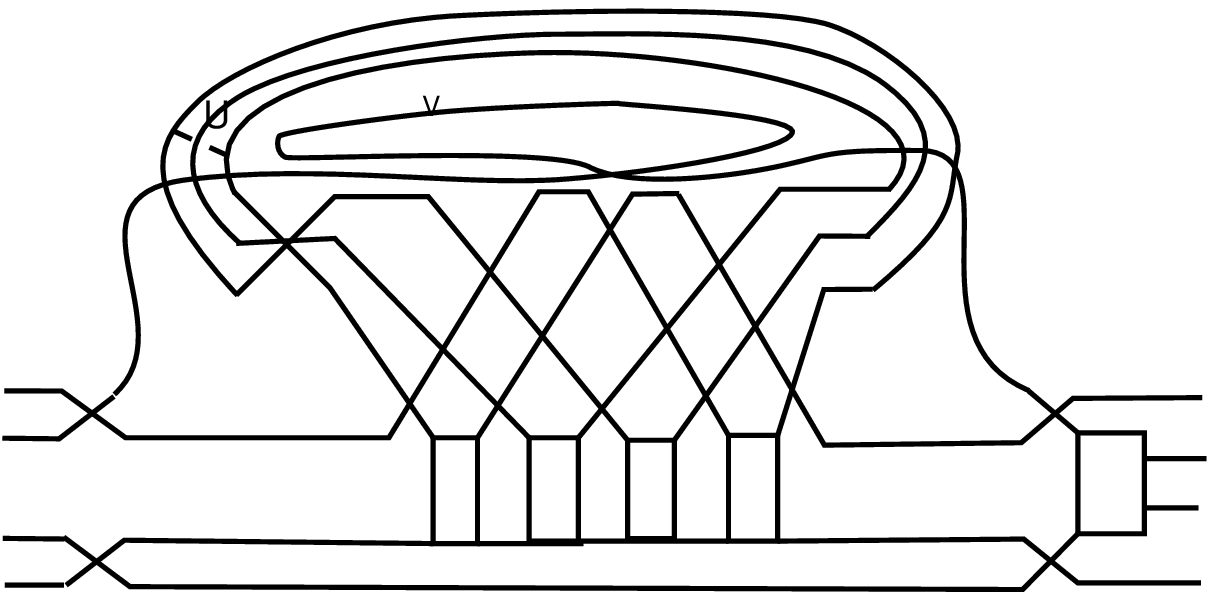}
\end{center}
We name $g$ the diagram without being traced, i.e., $g$ is
\begin{center}
\includegraphics[width=2in]{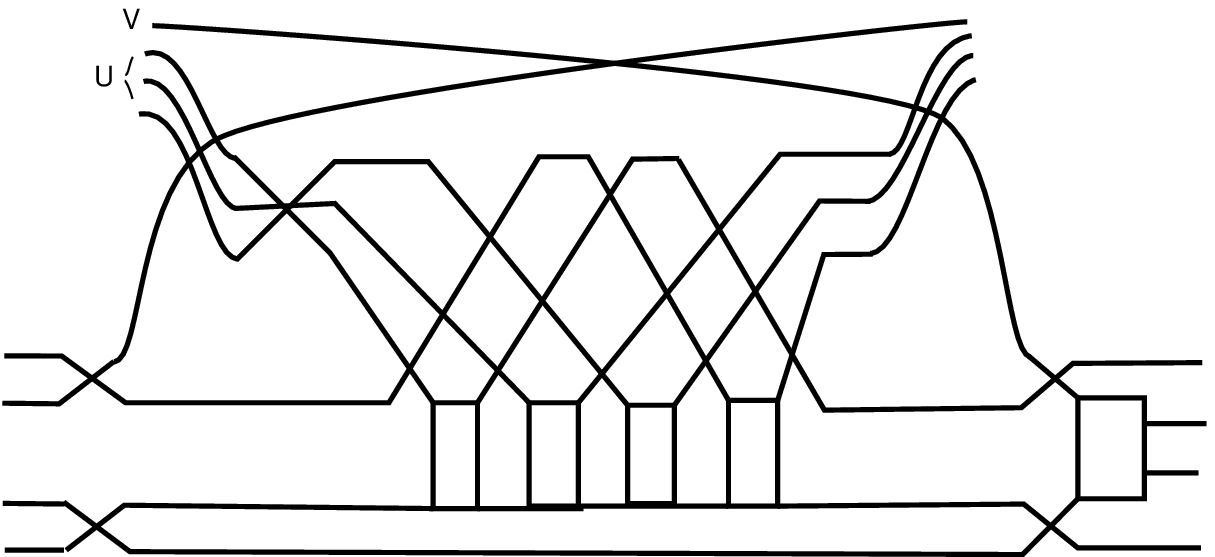}
\end{center}

Then $g\in \Trc^V$ by the reasons given above and if we reverse this procedure in fact we are showing that $\Tr^V(g)\in \Trc^U$  (after applying superposition, yanking and naturality and returning to the very beginning of the proof) thus we are satisfying the hypothesis of Vanishing II which means that $g\in \Trc^{U\otimes V}$. \\
Now we are allowed to apply the dinaturality axiom in order to permute the order of the objects that are going to be traced out:
\begin{center}
\includegraphics[width=2in]{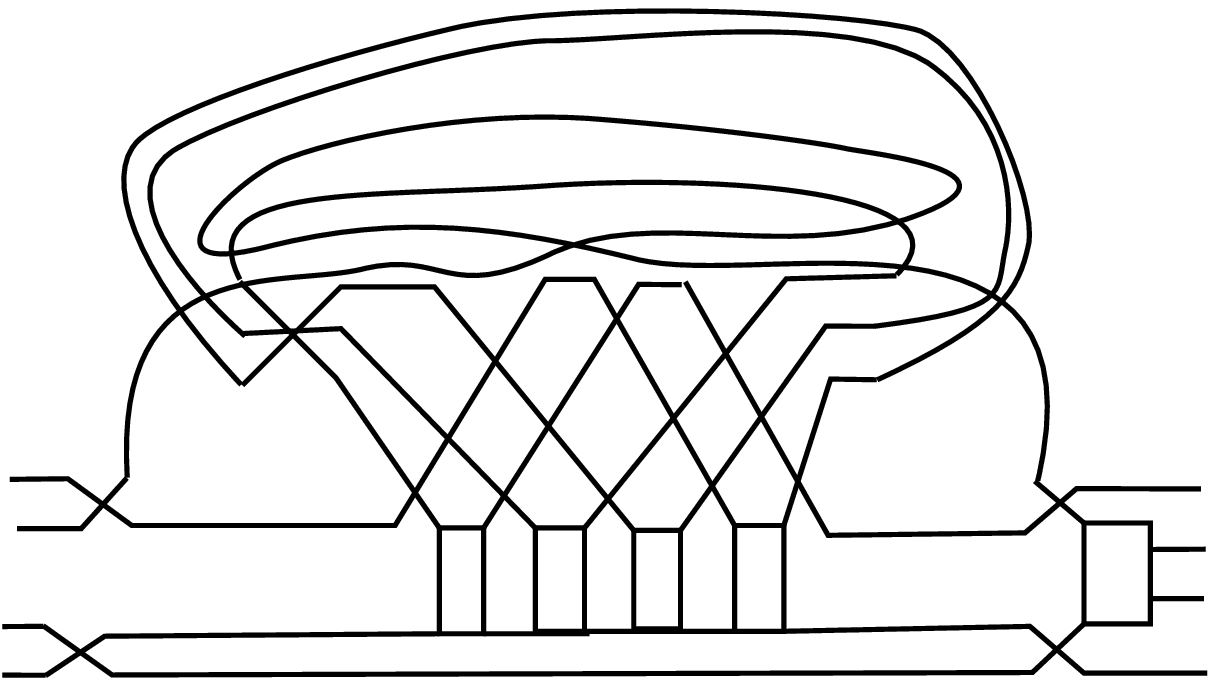}
\end{center}
Thus by coherence we have that:
\begin{center}
\includegraphics[width=2in]{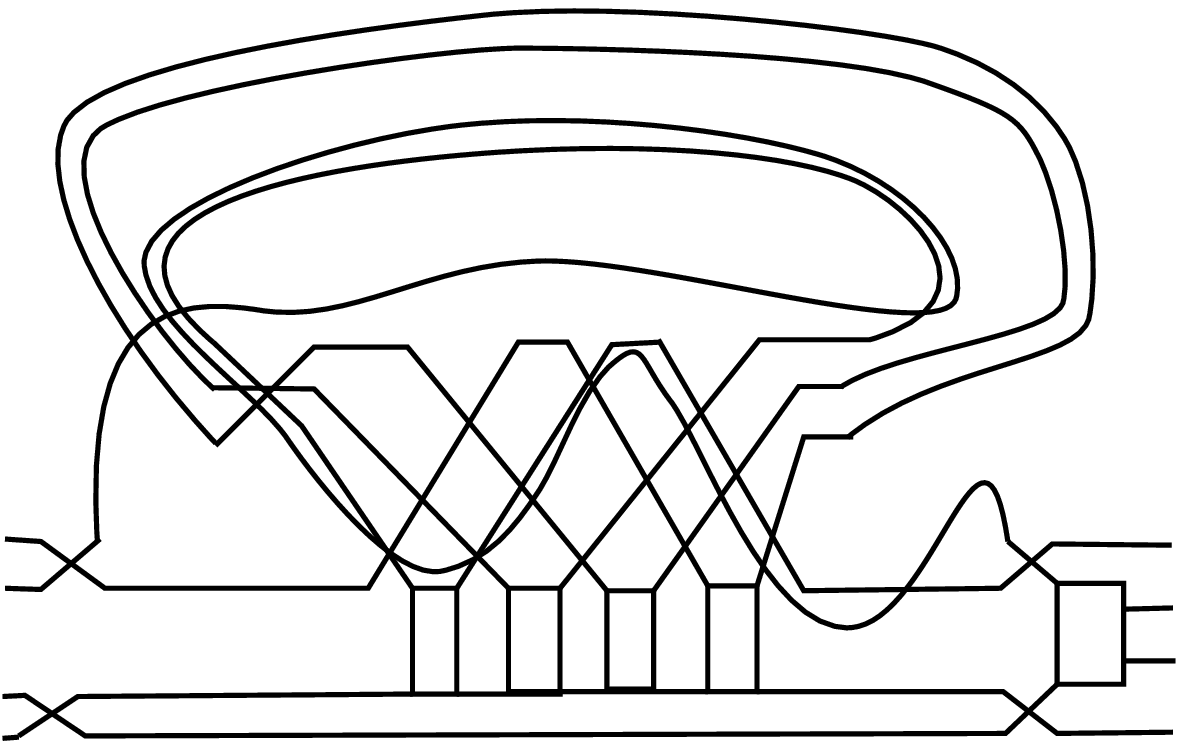}
\end{center}
Again by coherence:
\begin{center}
\includegraphics[width=2in]{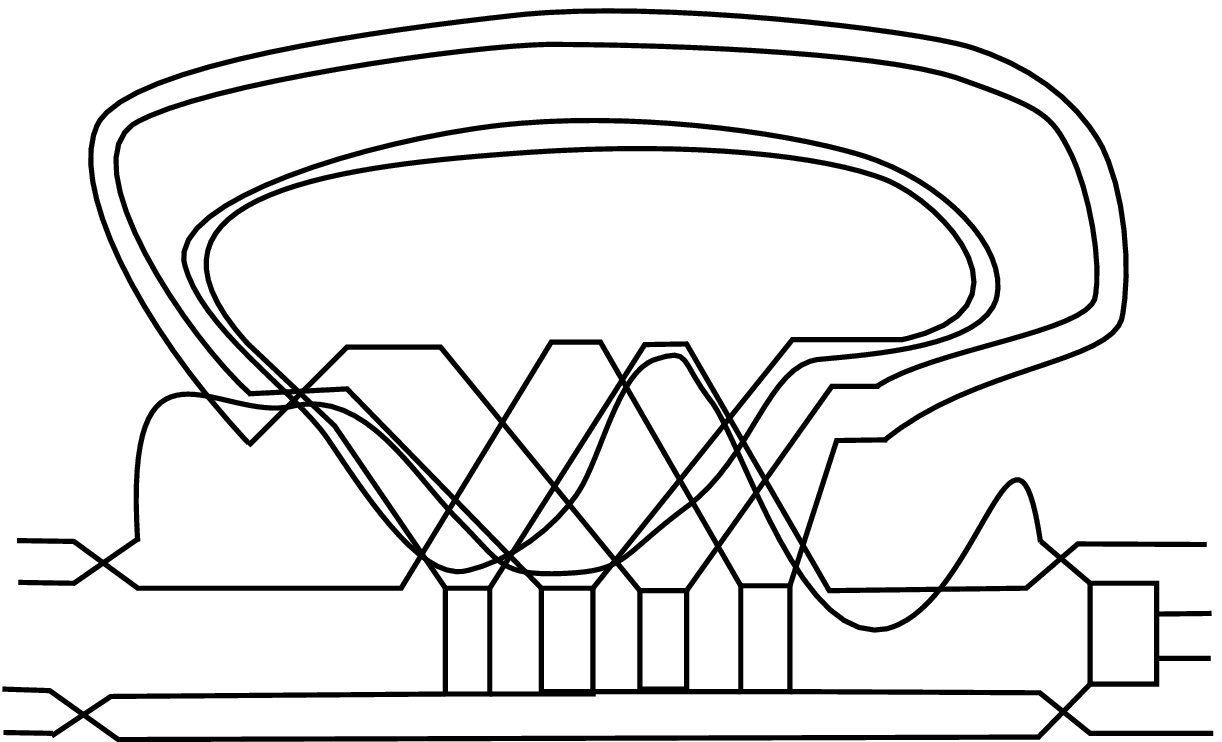}
\end{center}
Now by coherence and the yanking axiom:

\begin{center}
\includegraphics[width=2in]{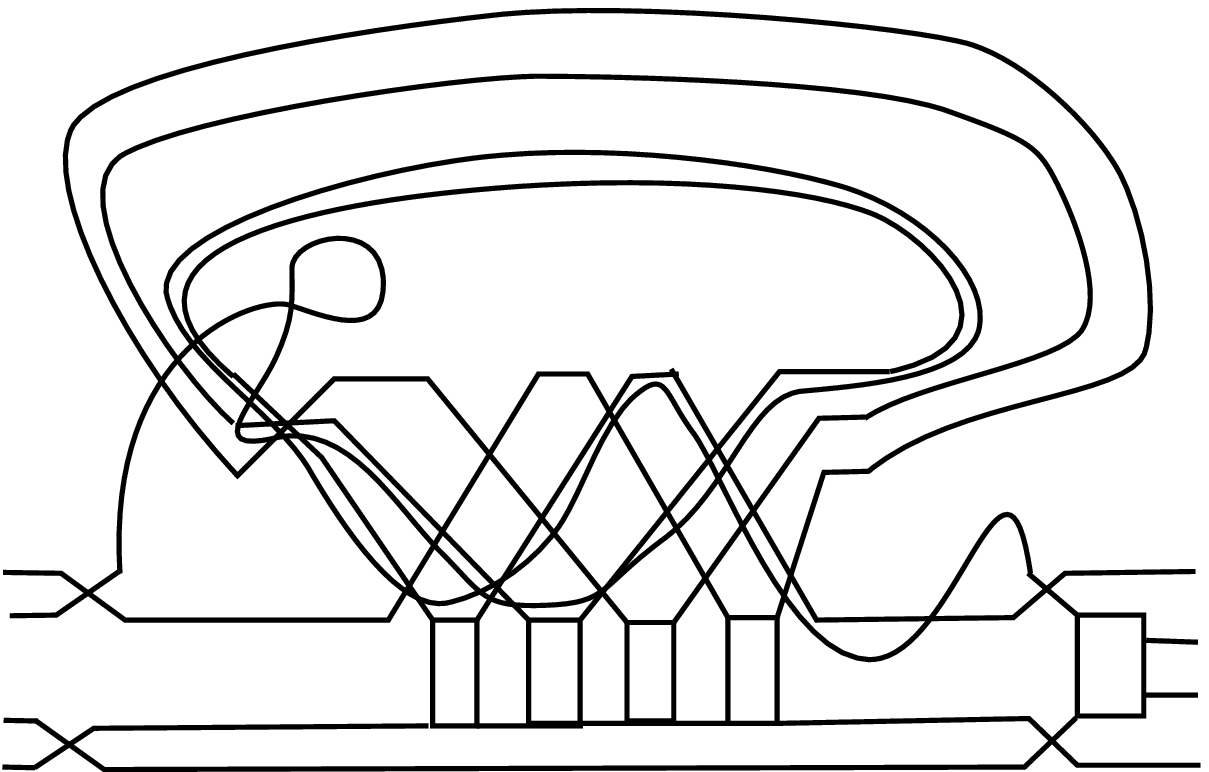}
\end{center}
Again since the trace is total on symmetries, and after applying superposing (equivalent version), the naturality axiom shows that:
\begin{center}
\includegraphics[width=2in]{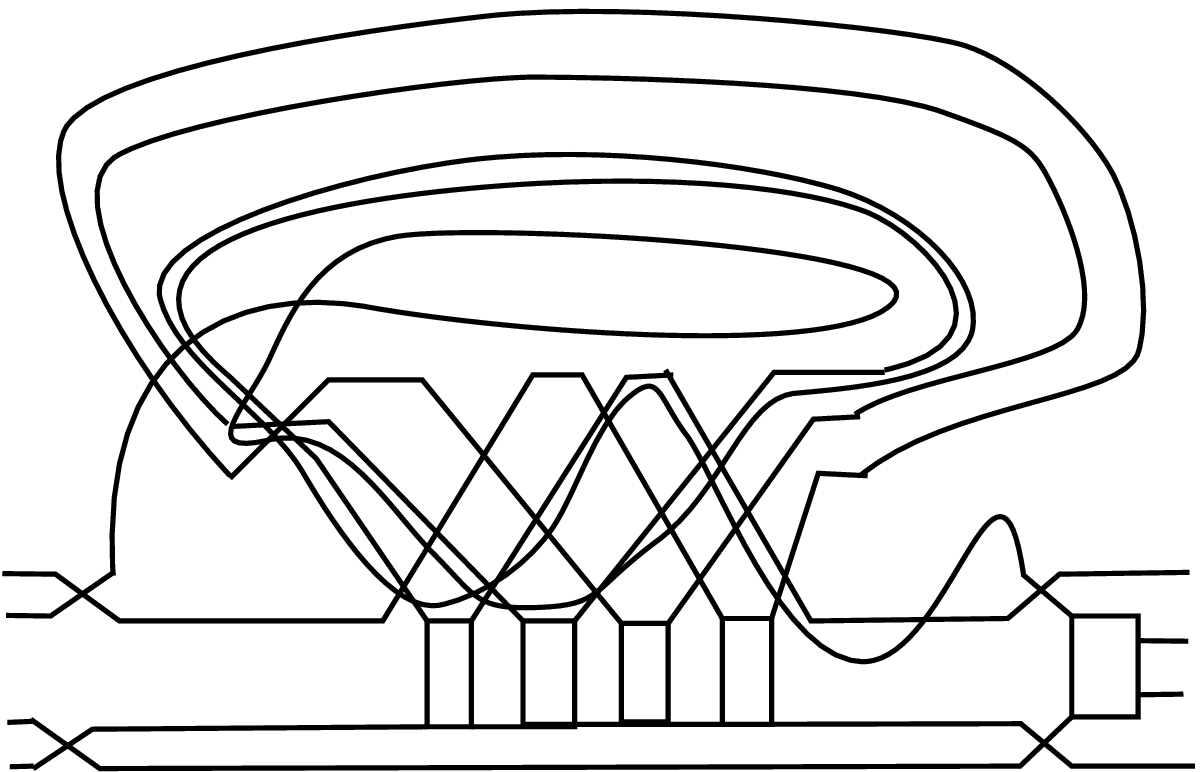}
\end{center}
In the same way as before we repeat what we did but now applied to the second line:
\begin{center}
\includegraphics[width=2in]{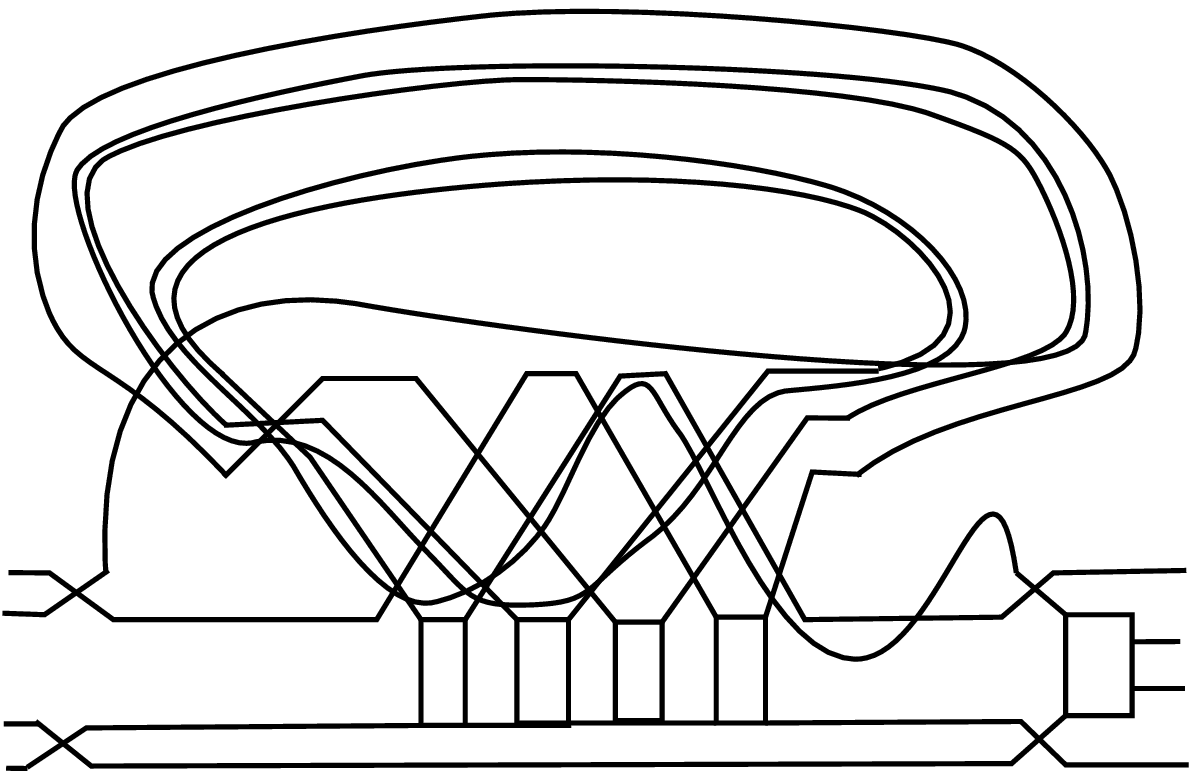}
\end{center}
Because the map involve are coherence maps:
\begin{center}
\includegraphics[width=2in]{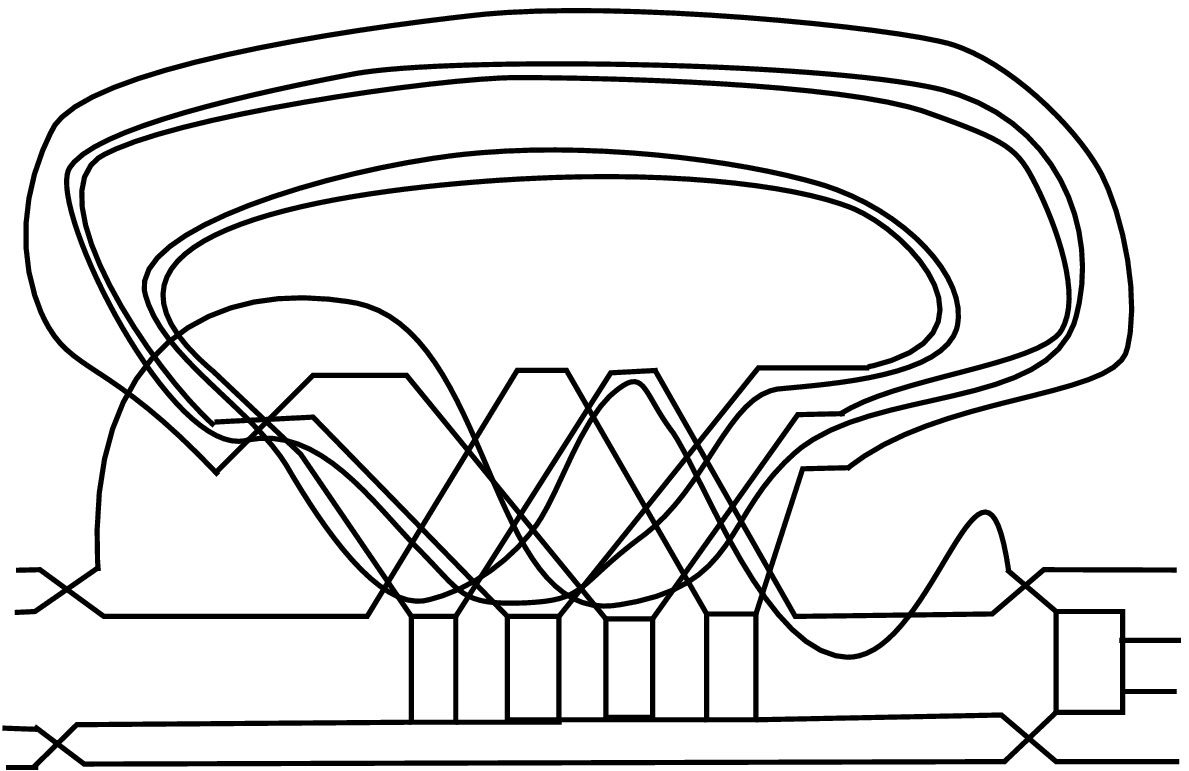}
\end{center}
Coherence:
\begin{center}
\includegraphics[width=2in]{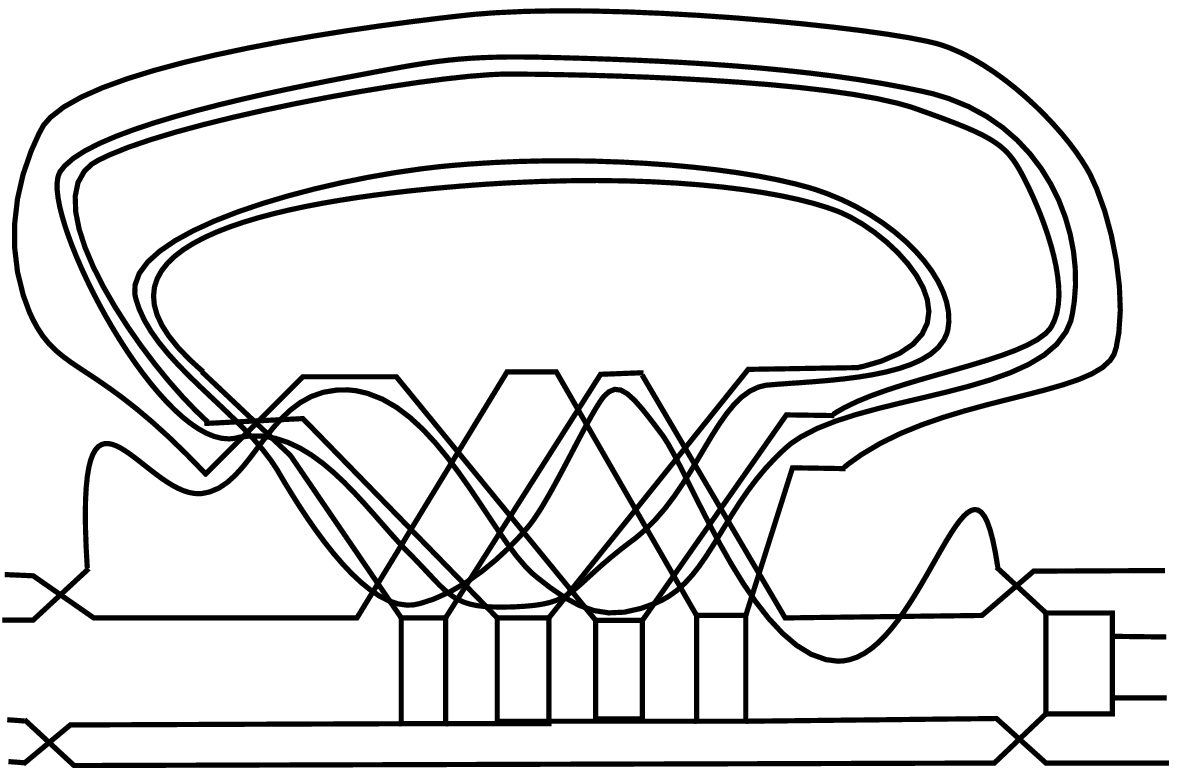}
\end{center}

Coherence and the yanking axiom:
\begin{center}
\includegraphics[width=2in]{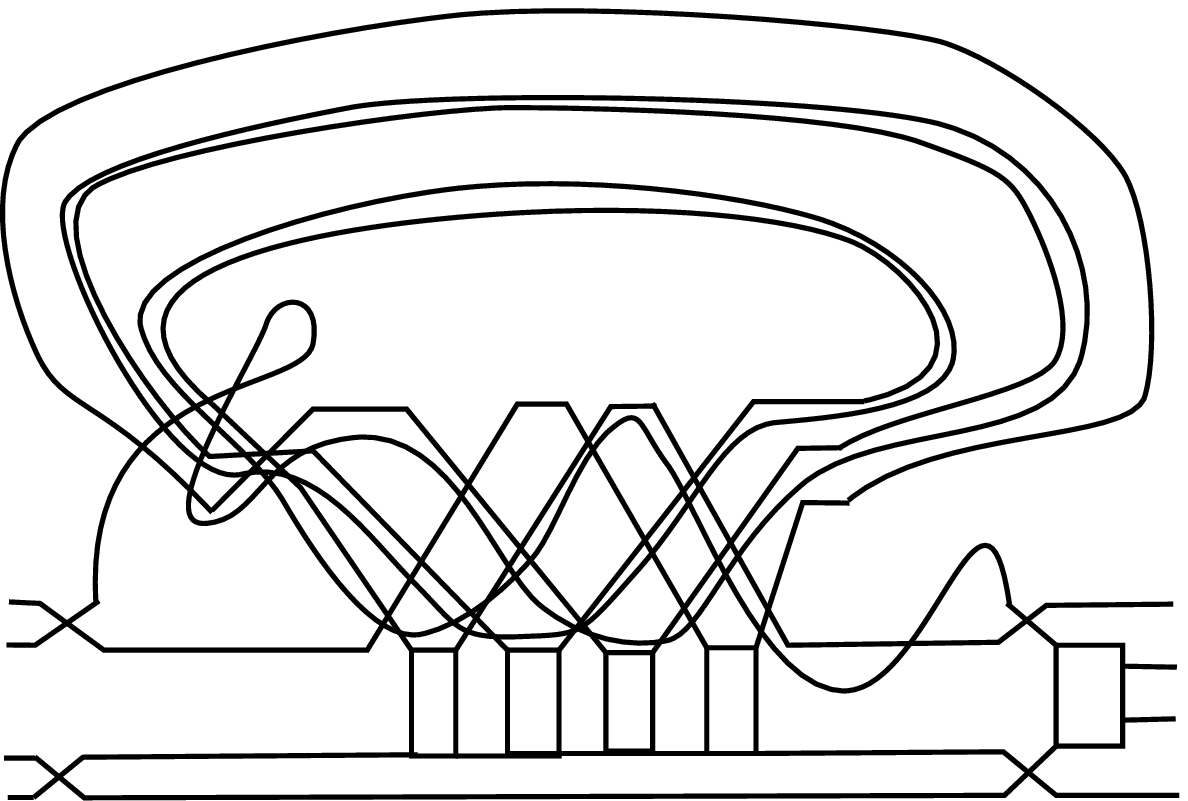}
\end{center}

Superposition and naturality:
\begin{center}
\includegraphics[width=2in]{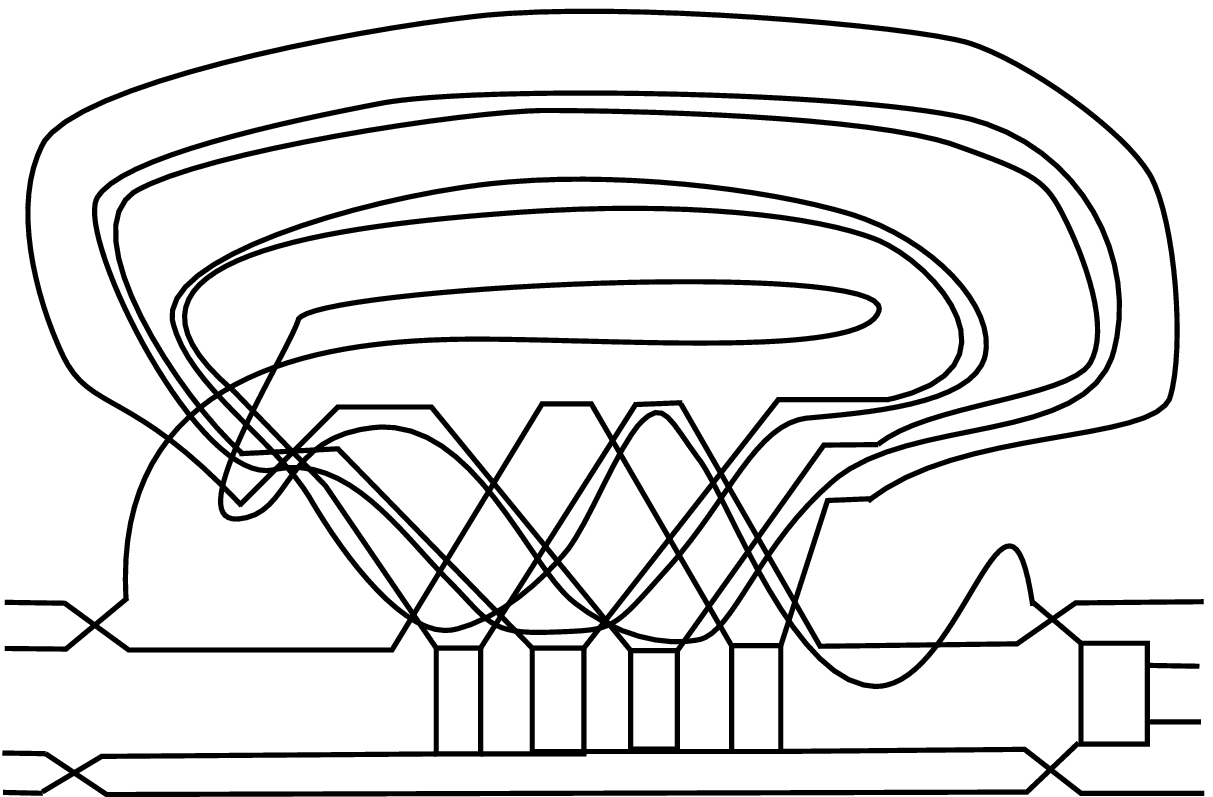}
\end{center}
Same argument as before applied to the third line:
\begin{center}
\includegraphics[width=2in]{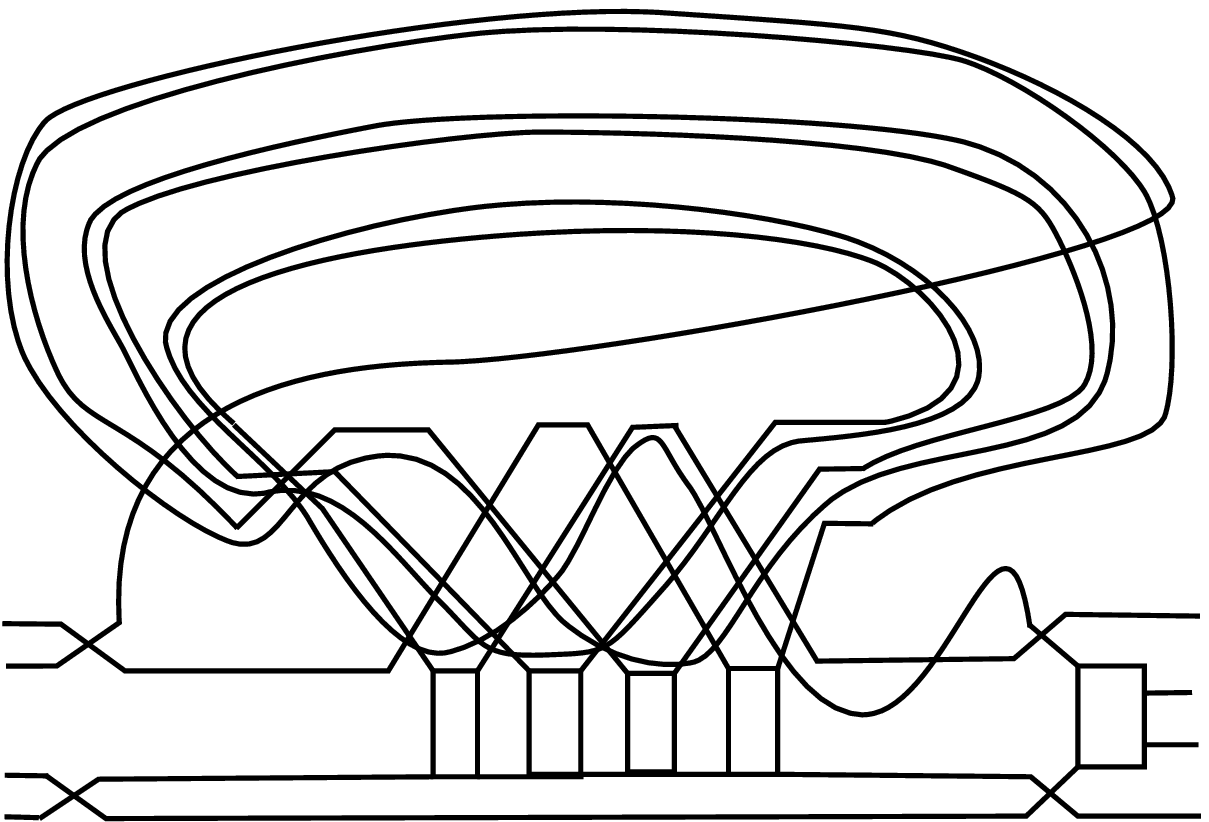}
\end{center}
Coherence allows us to express:
\begin{center}
\includegraphics[width=2in]{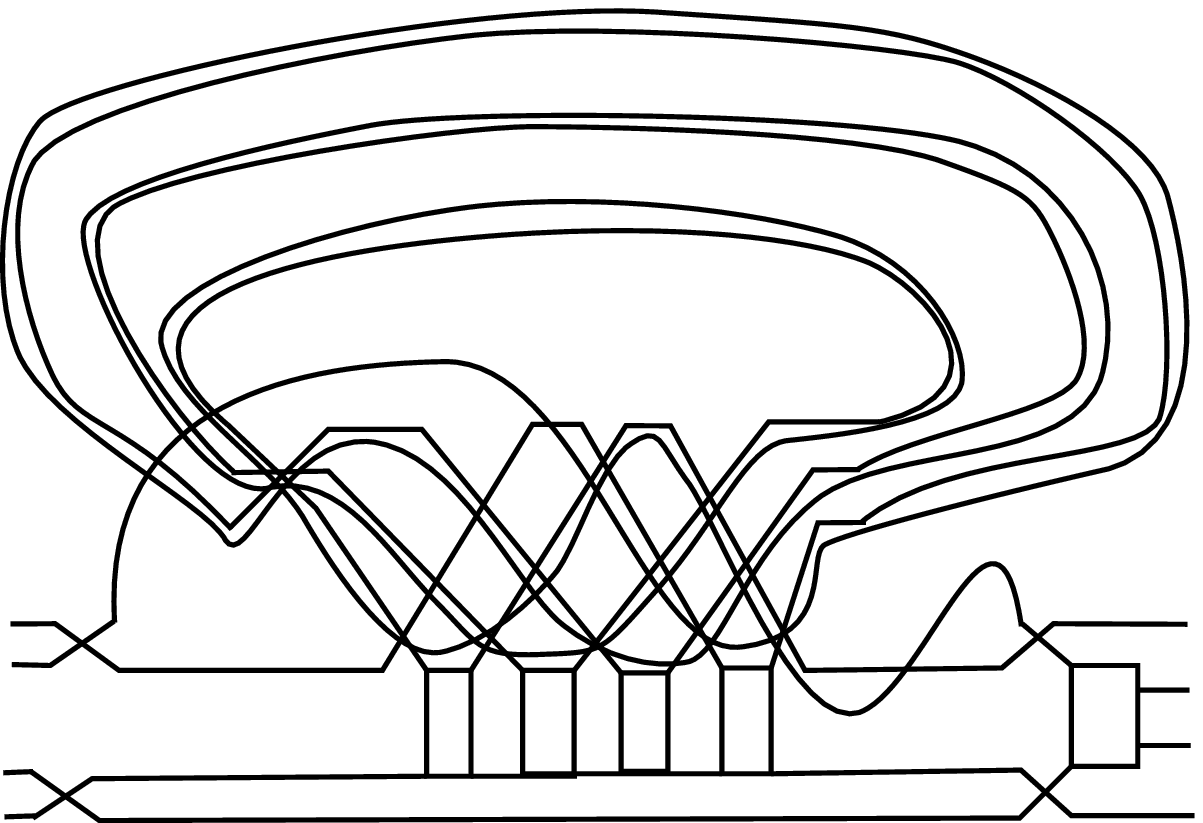}
\end{center}
We therefore have again by coherence:
\begin{center}
\includegraphics[width=2in]{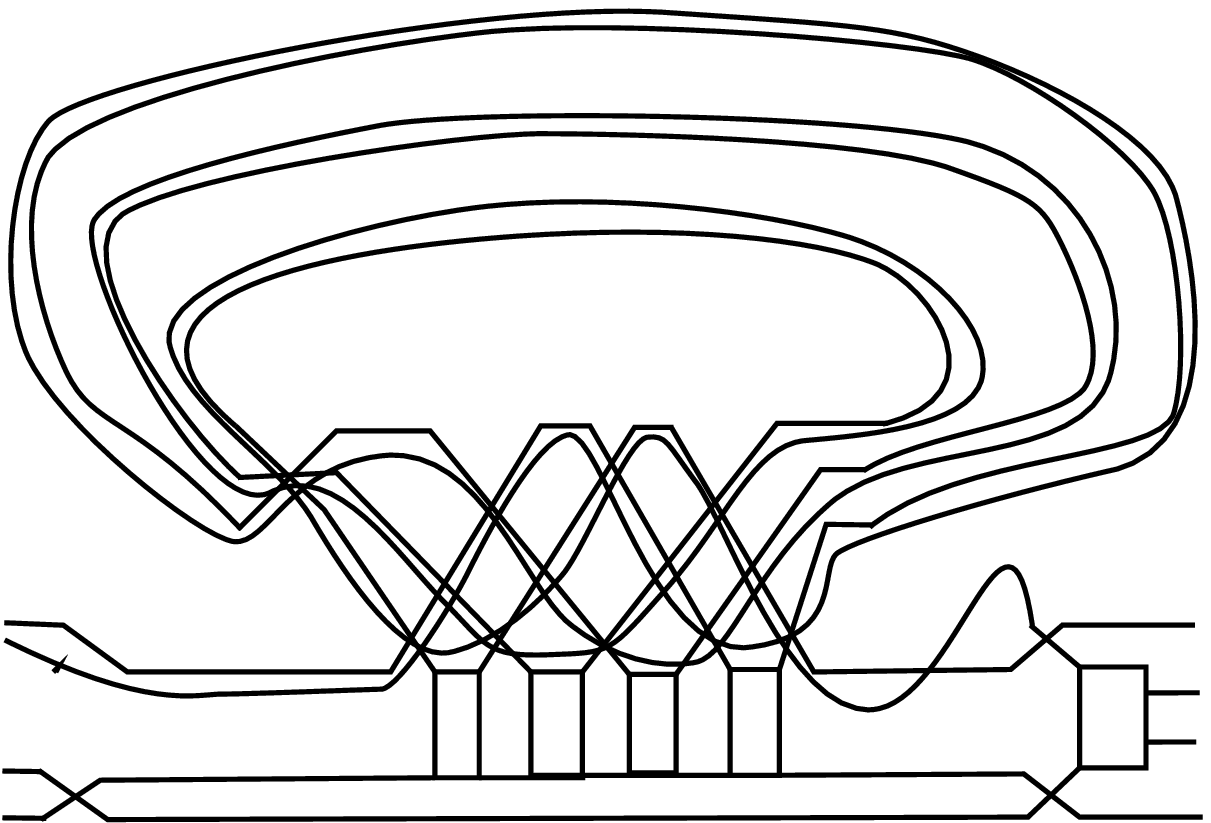}
\end{center}

Since $\sigma\circ \sigma^{-1}=1$:
\begin{center}
\includegraphics[width=2in]{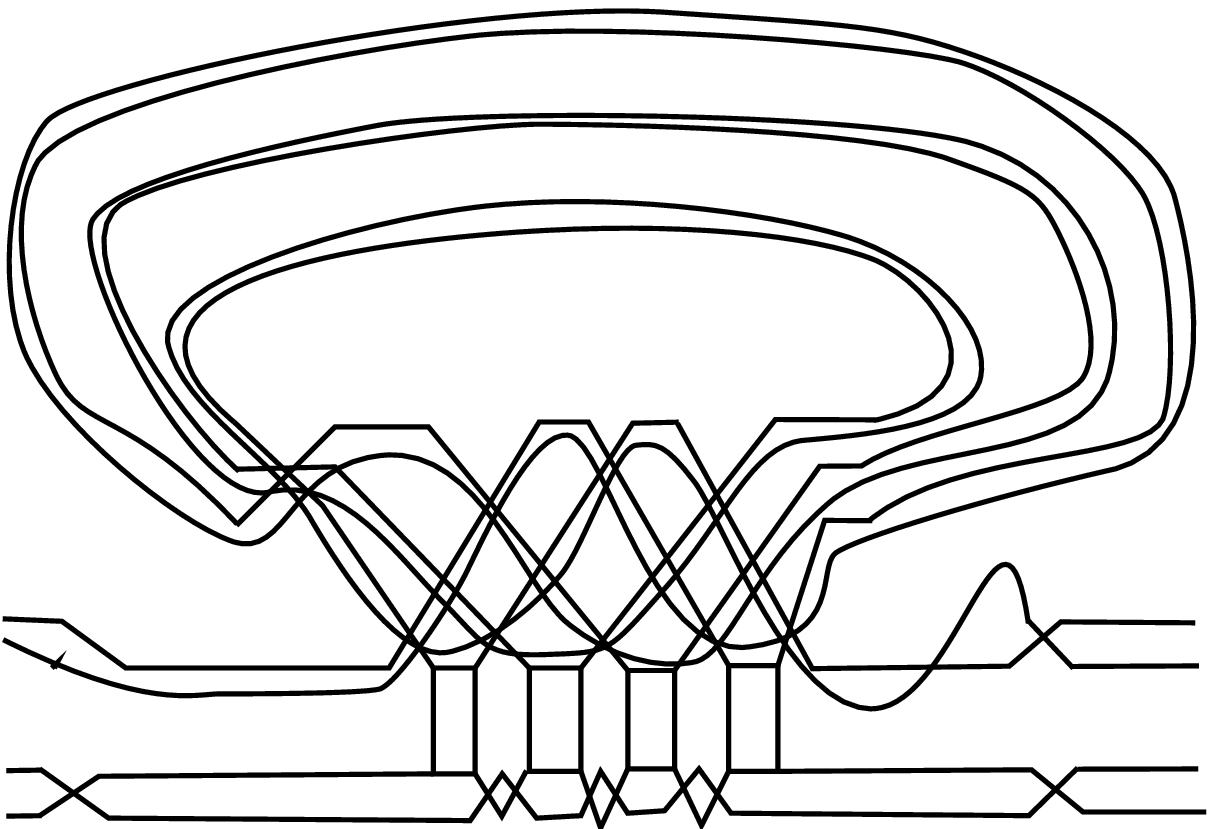}
\end{center}
Finally by coherence we get:
\begin{center}
\includegraphics[width=2in]{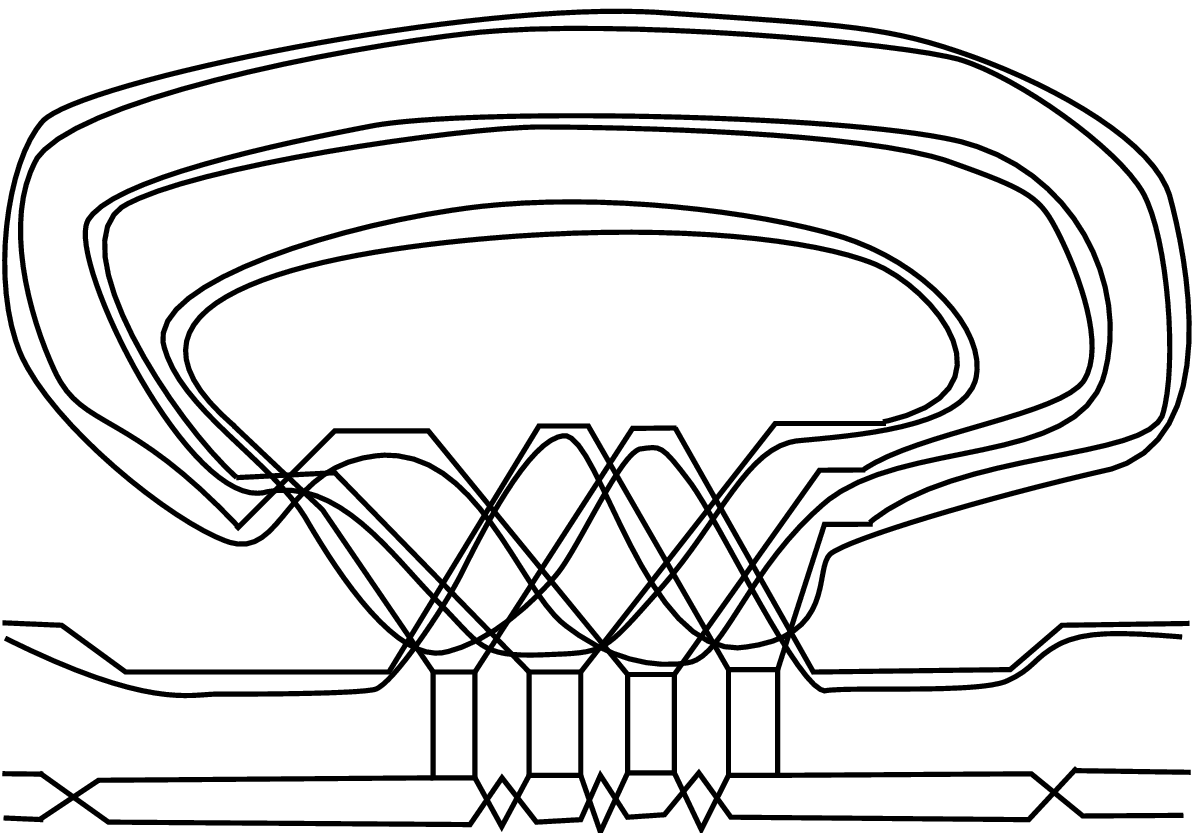}
\end{center}
Which is by definition $[\, \vec{p}\otimes 1 \,]$.
\end{proof}

\begin{theorem}\label{INT ES ssmpc}
Let $(\mathcal{C},\otimes,I,\Tr,s)$ be a symmetric monoidal partially traced category. The operation defined above $[-]$ determines a ssmpc $(Int^p(\mathcal{C}),[-],\otimes,I,\sigma)$.
\end{theorem}
\begin{proof}
It follows from the previous lemmas.
\end{proof}

 Next, we wish to show that $Int^p(\cC)$ is a compact closed paracategory. Let $(I,I)\stackrel{\eta}\longrightarrow (A,B)\otimes (A,B)^*$ and $(A,B)^*\otimes (A,B)\stackrel{\varepsilon}\longrightarrow (I,I)$ be the unit and counit associated to the paracategory $Int^p(\cC)$. Actually, since $\cC$ is a strict category, we can regard these morphisms as $id:I\otimes A\otimes B\rightarrow A\otimes B\otimes I$ and $id:B\otimes A\otimes B\rightarrow I\otimes B\otimes A$ respectively.
\begin{lemma}\label{COMPACT EQUATION INT}
$[\eta\otimes 1,1\otimes\varepsilon]\downarrow$, $[1\otimes\eta,\varepsilon\otimes1]\downarrow$ and $[\eta\otimes 1,1\otimes\varepsilon]=1$, $[1\otimes\eta,\varepsilon\otimes1]=1$.
\end{lemma}
\begin{proof}
Notice that $\sigma_{A,I}=id_A$ for every object $A\in \cC$.
We start with the identity map $(A,B)\stackrel{1}\longrightarrow (A,B)$ which is the map $1:A\otimes B\rightarrow A\otimes B$ in $\cC$.
Since,
$$1_{A\otimes B}=\includegraphics[height=0.4in]{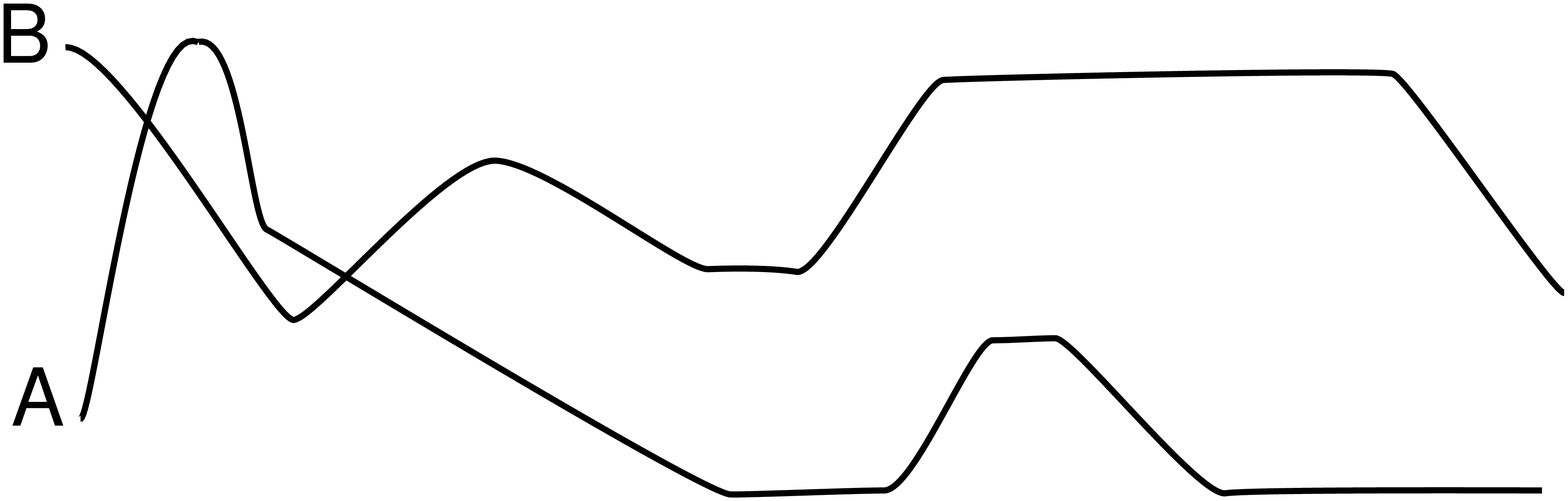}$$
holds by coherence, using the yanking axiom
$$\includegraphics[height=0.8in]{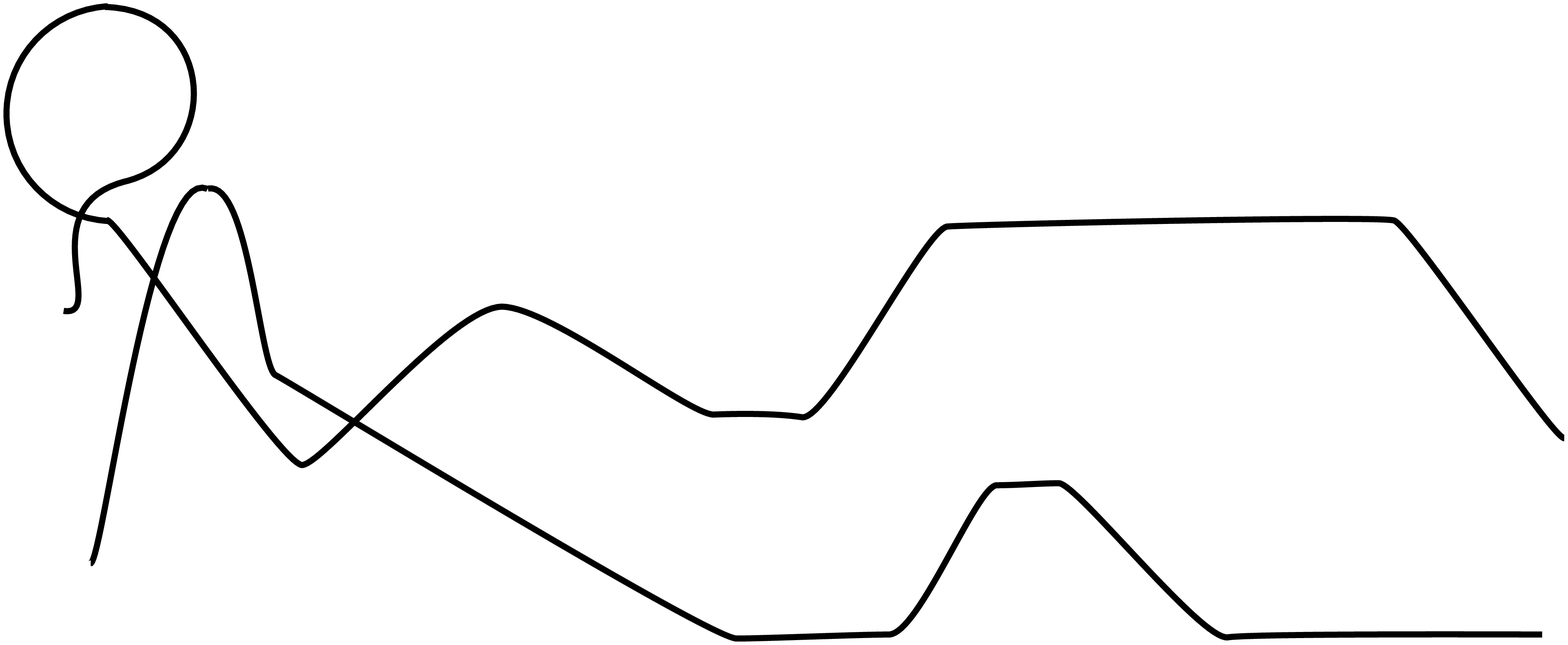}\,\,\,\,\,\,\,\mbox{and naturality}\,\,\,\,\,\,\,\includegraphics[height=0.8in]{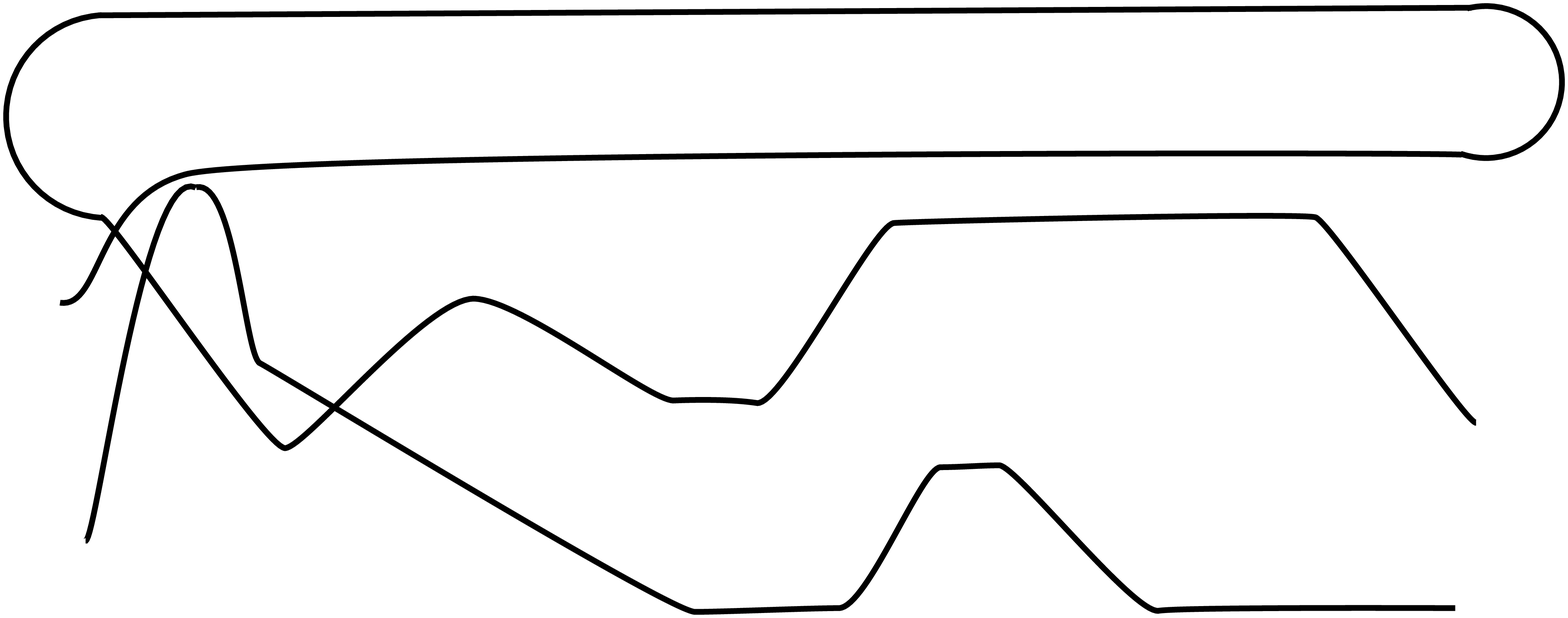}.$$
Notice that, all along this proof, we implicitly claim that the graph below the trace is in the corresponding trace class.
For instance, in the last diagram from the naturality axiom it follows that
$$\includegraphics[height=0.4in]{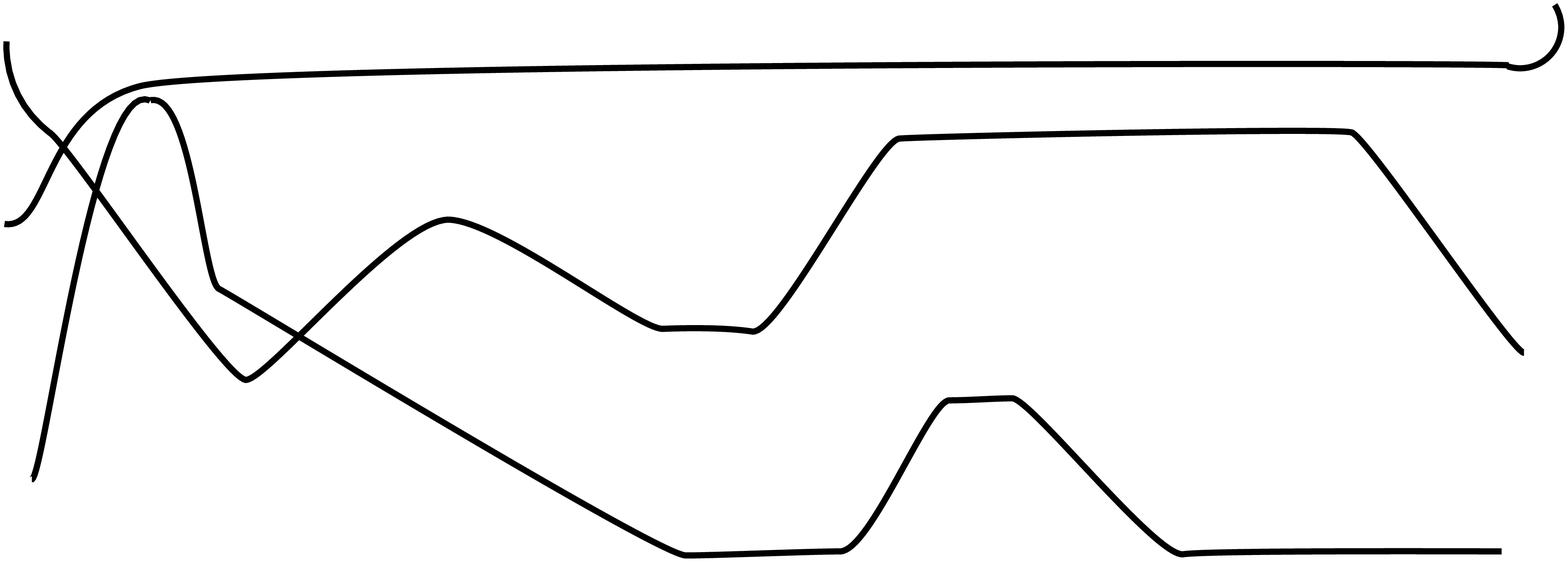}\in\Trc^B.$$
Then by superposing axiom and coherence
$$\includegraphics[height=0.8in]{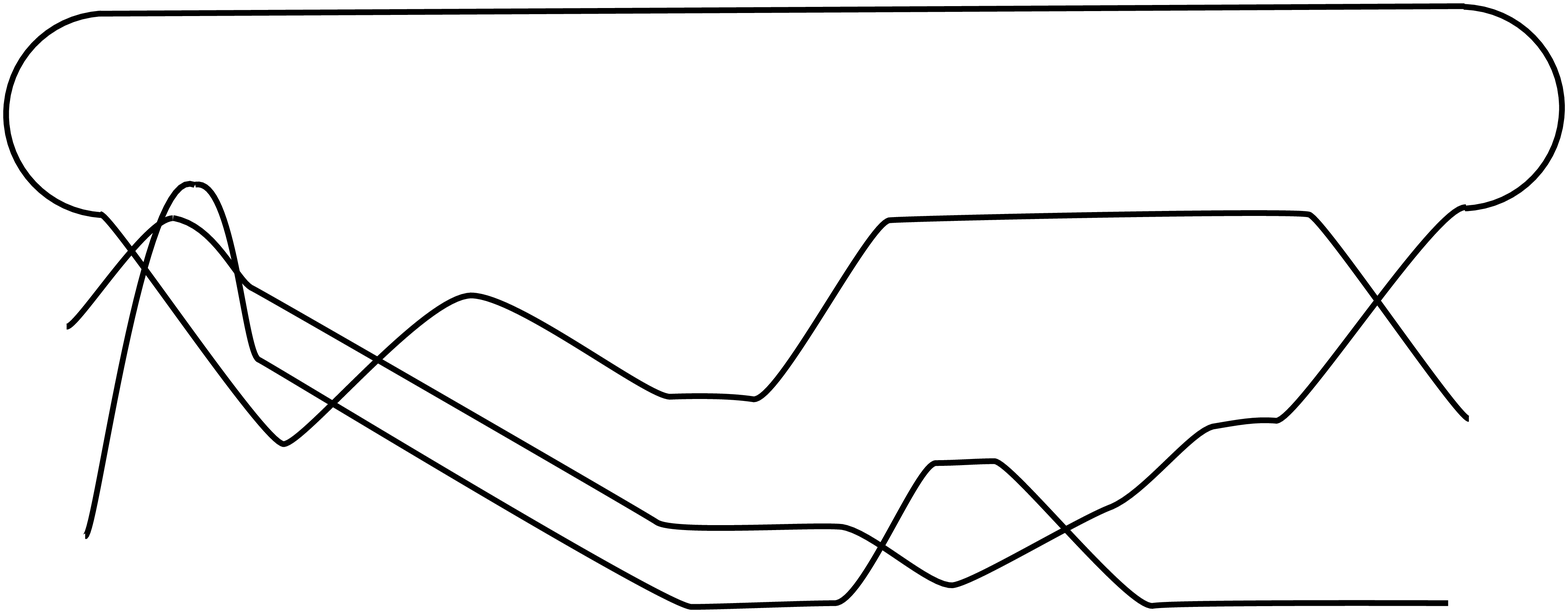}.$$
Therefore, by applying yanking and naturality again we obtain
$$\includegraphics[height=0.8in]{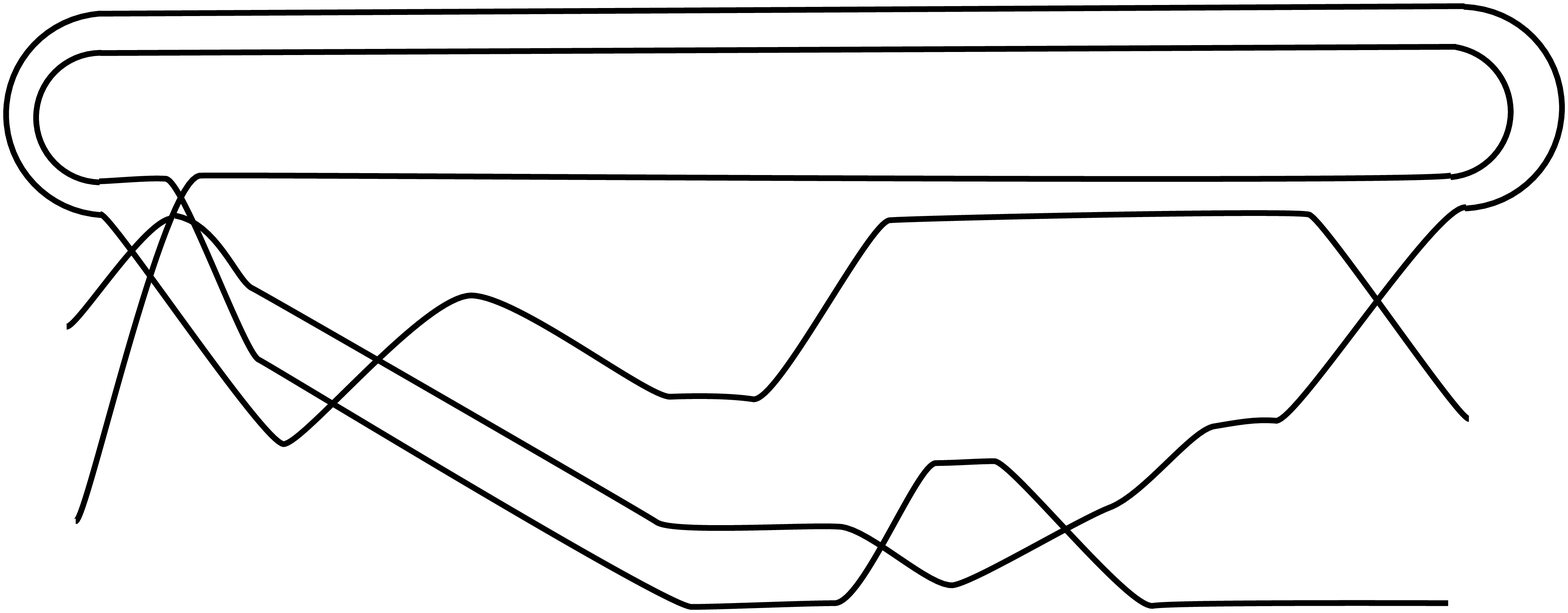}$$
where the graph below this new trace is in the trace class $\Trc^A$ and this, of course, will be preserved by any coherent modification of the graph.
We have from superposition and coherence axioms that
$$\includegraphics[height=0.8in]{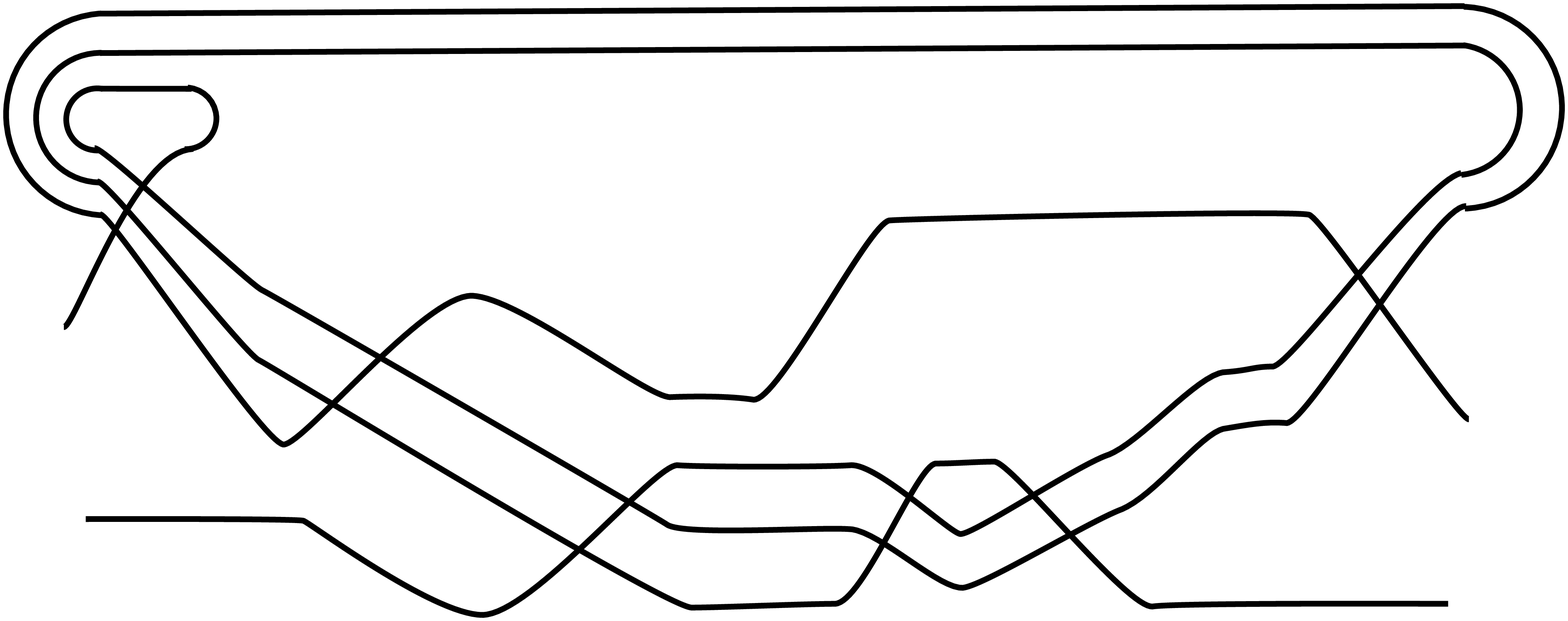}.$$
Again, naturality, superposition and coherence gives us
$$\includegraphics[height=0.8in]{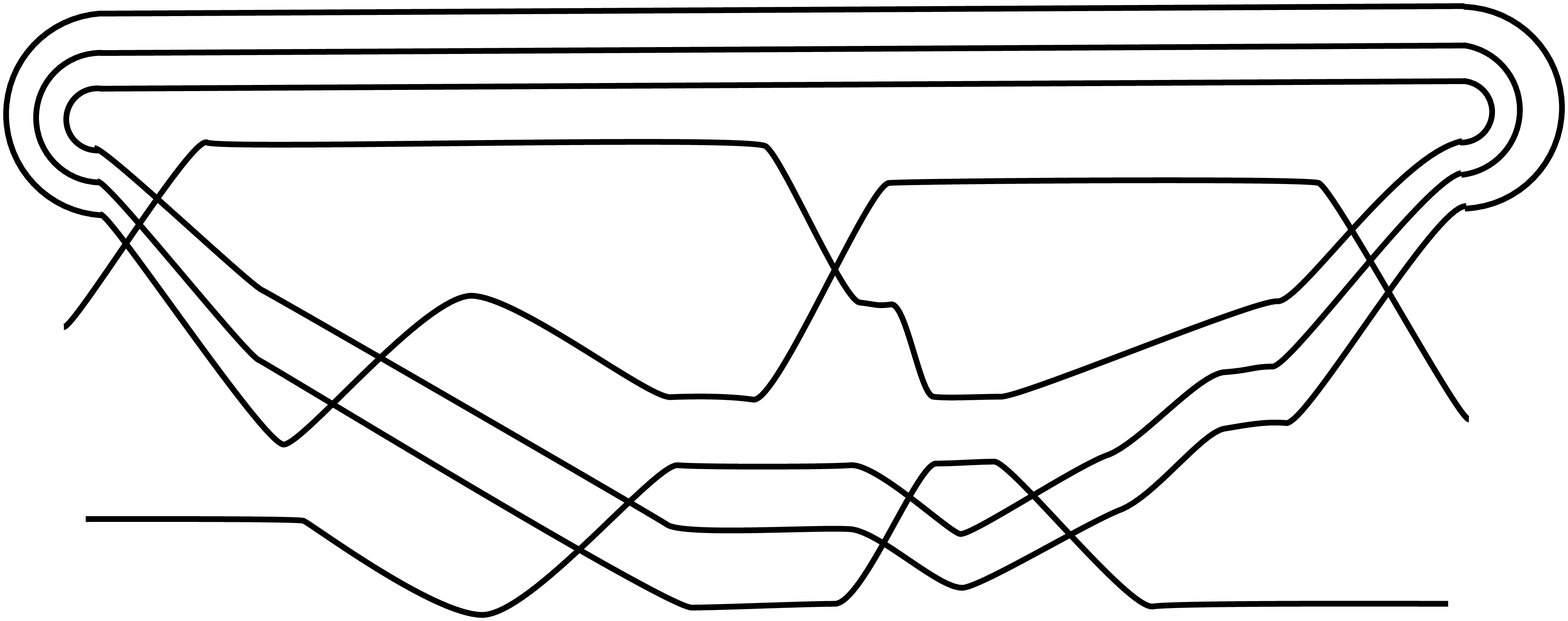}$$
where the graph below the trace is in the trace class $\Trc^B$. Since $\cC$ is a strict category then is equal to
$$\includegraphics[height=0.8in]{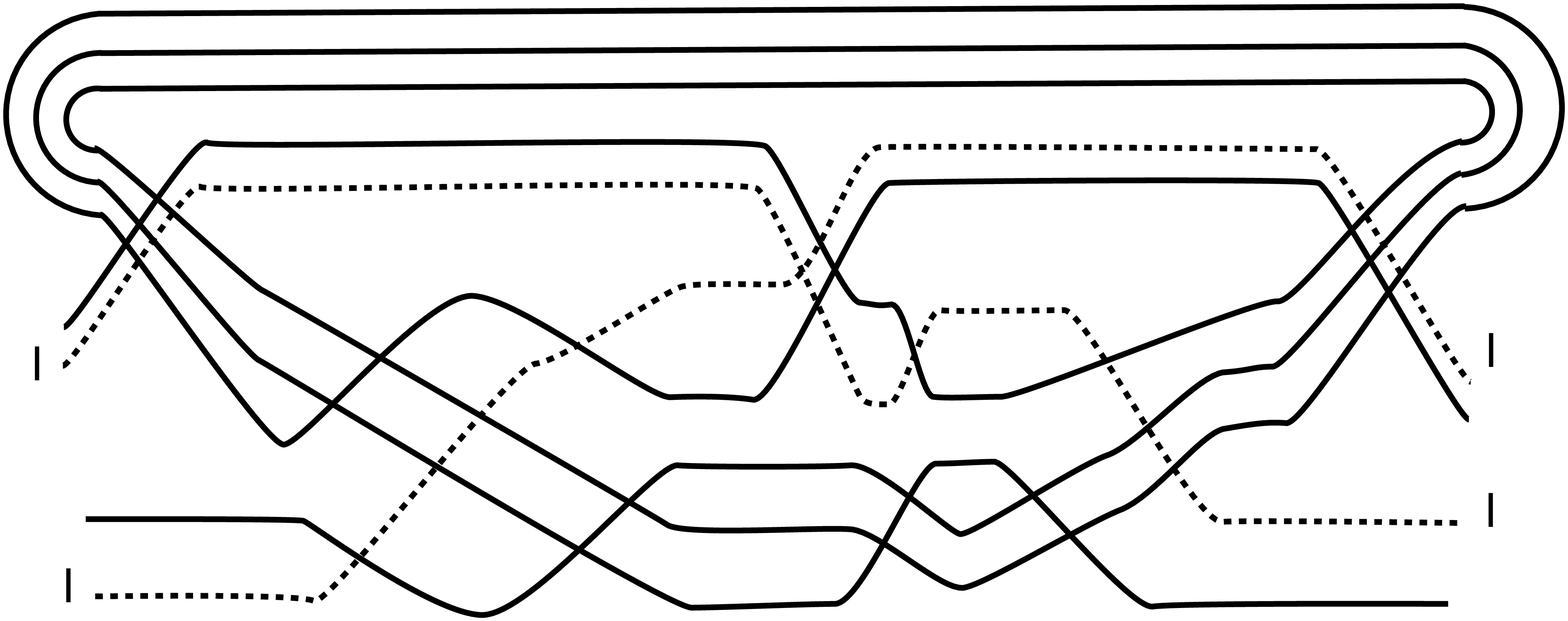}.$$
Finally, since the trace class conditions for applying vanishing II are satisfied, we apply the vanishing II axiom twice and we obtain that is equal to
$$\includegraphics[height=0.8in]{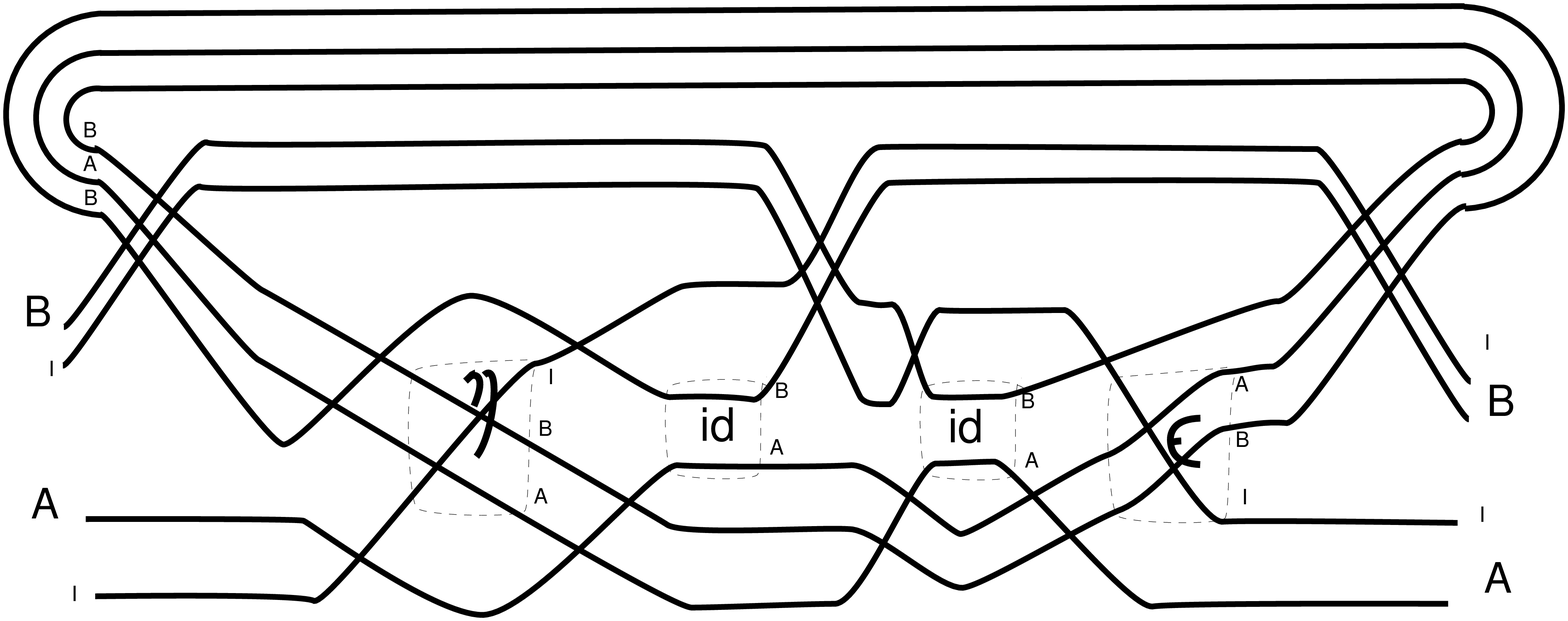}=[\eta\otimes 1,1\otimes\varepsilon].$$

In the same way as before we prove that $[1\otimes\eta,\varepsilon\otimes1]=1$.

We sketch schematically the rest of the proof leaving details to the reader. We start with the identity $1_{B\otimes A}$. \\

coherence:
$\includegraphics[height=0.3in]{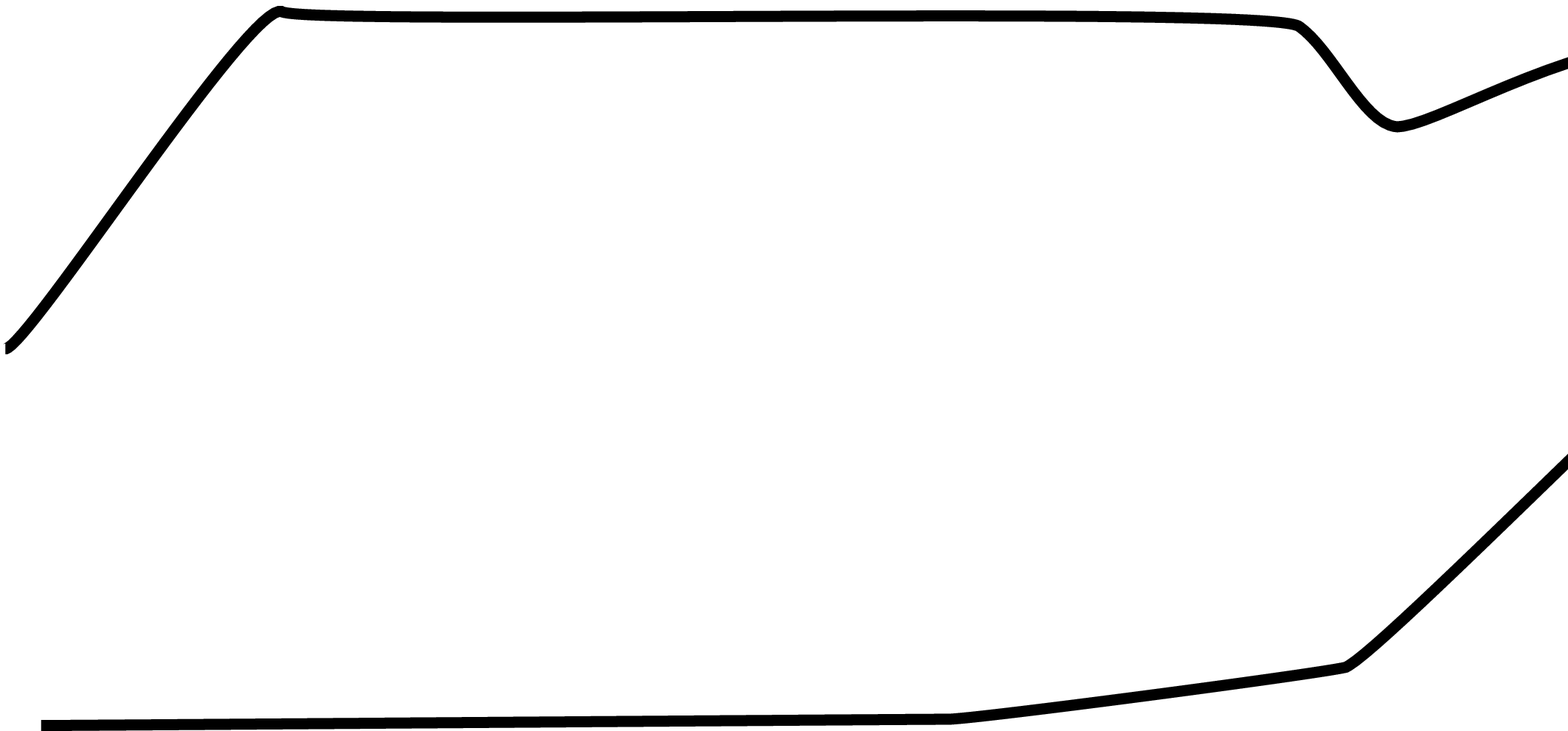}$
yanking:
$\includegraphics[height=0.4in]{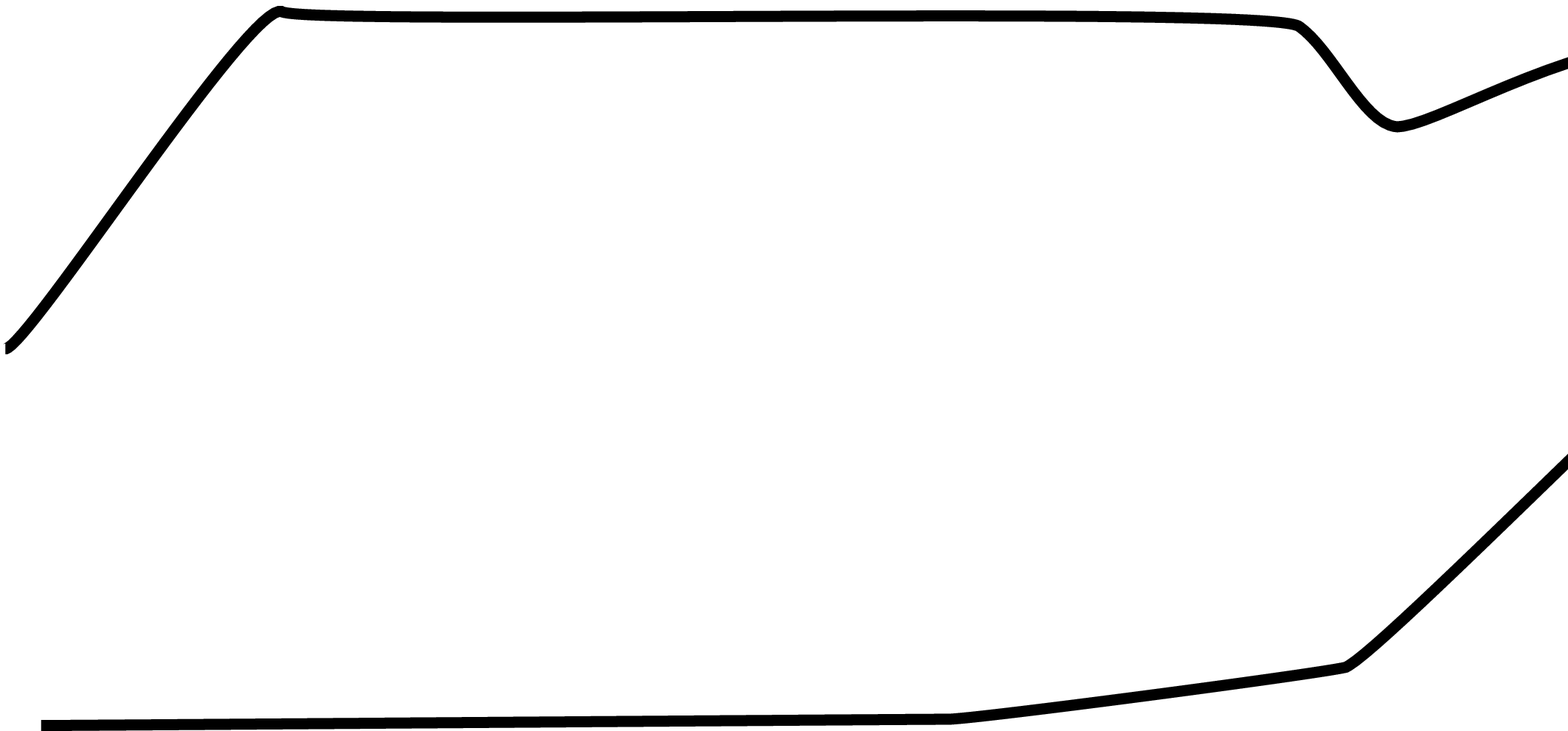}$
naturality:
$\includegraphics[height=0.4in]{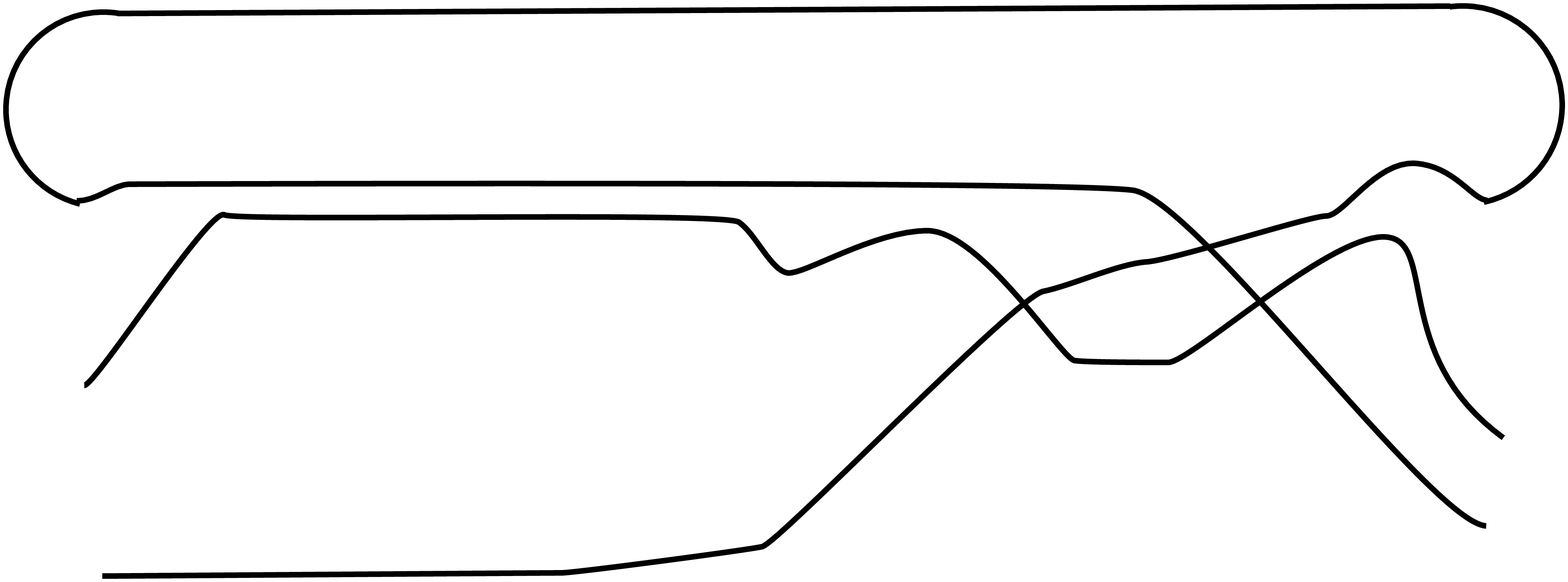}$\\

coh.:
$\includegraphics[height=0.4in]{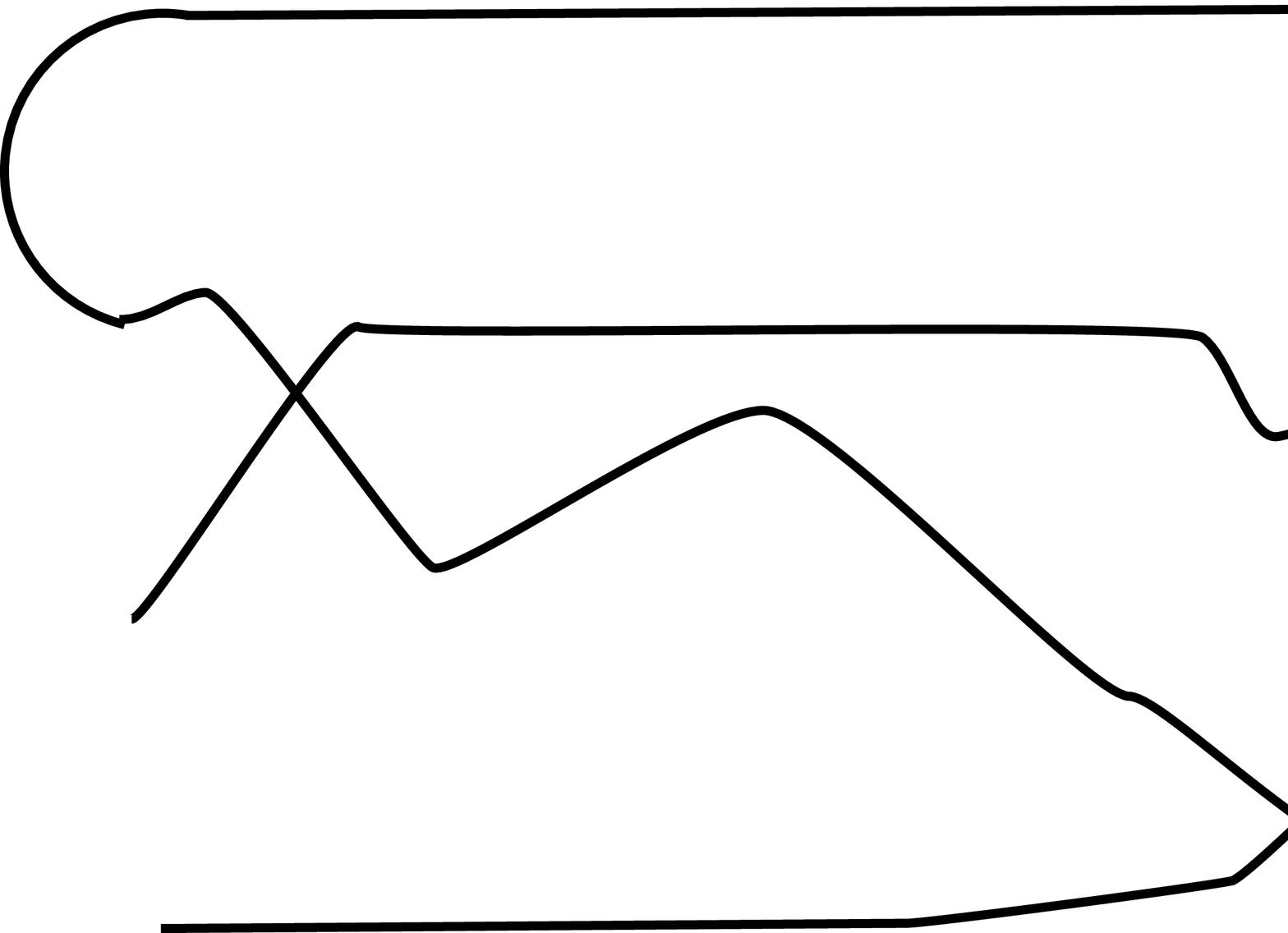}$
yanking:
$\includegraphics[height=0.4in]{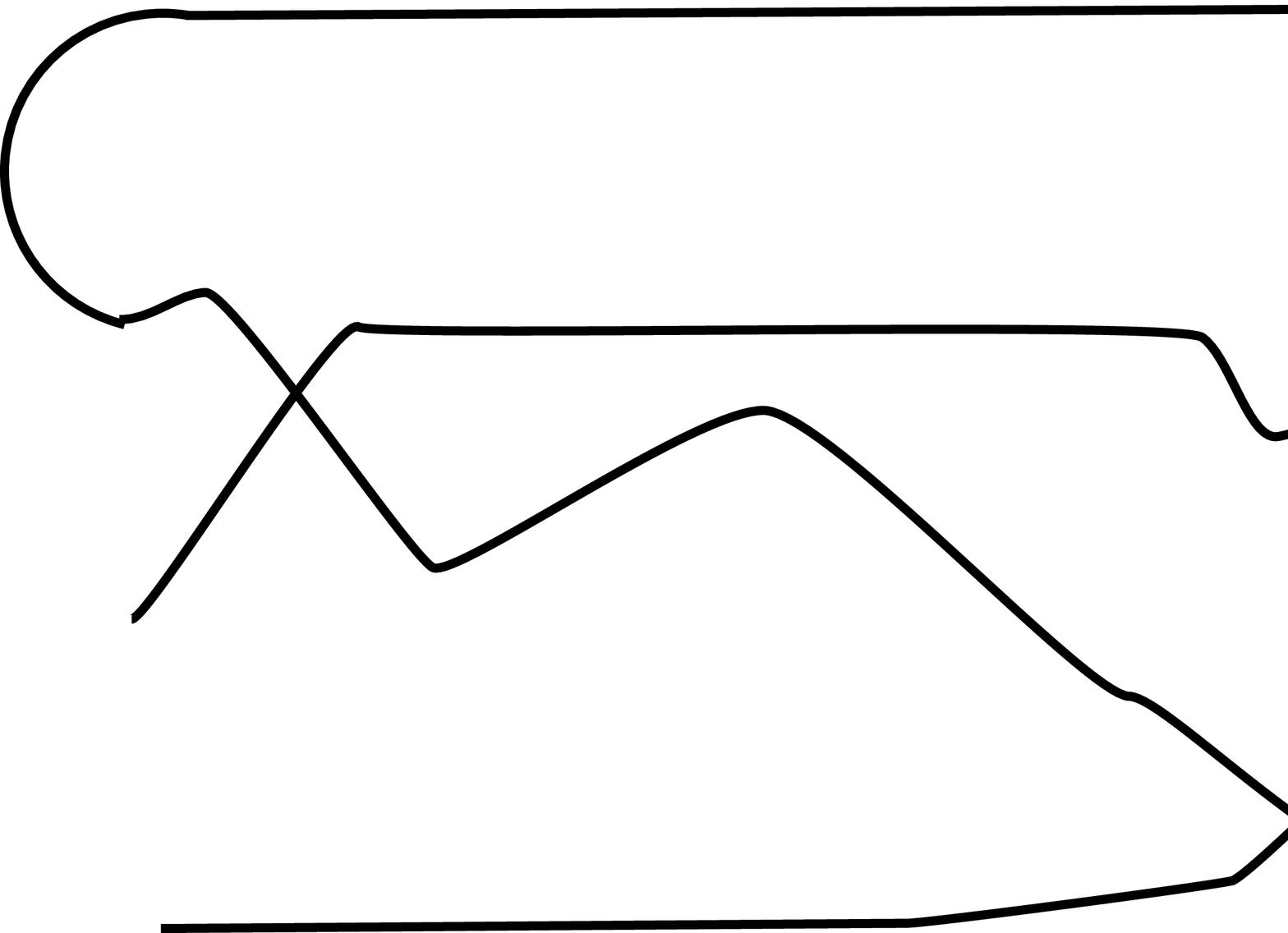}$
superposing:
$\includegraphics[height=0.4in]{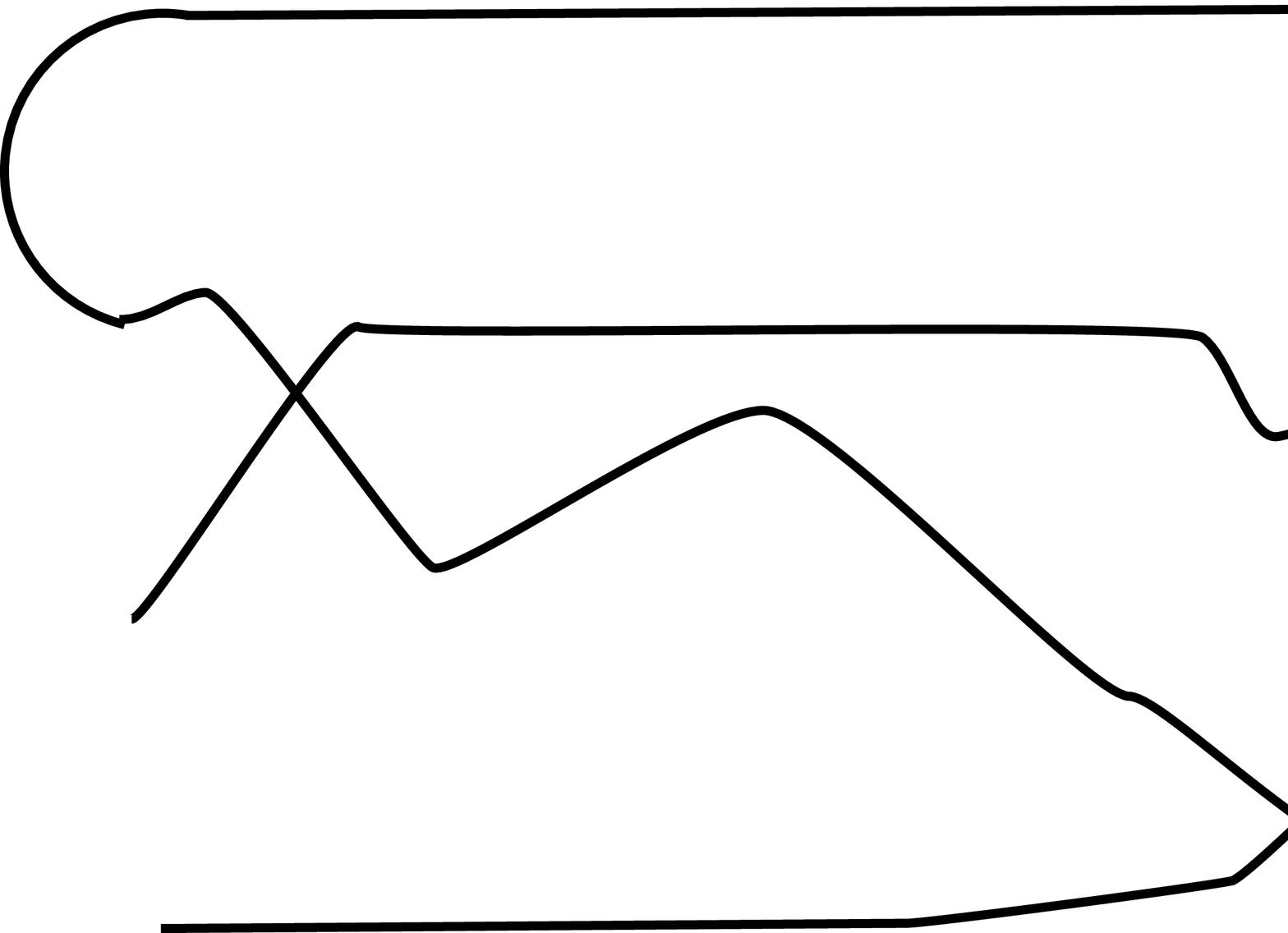}$\\

nat.:
$\includegraphics[height=0.4in]{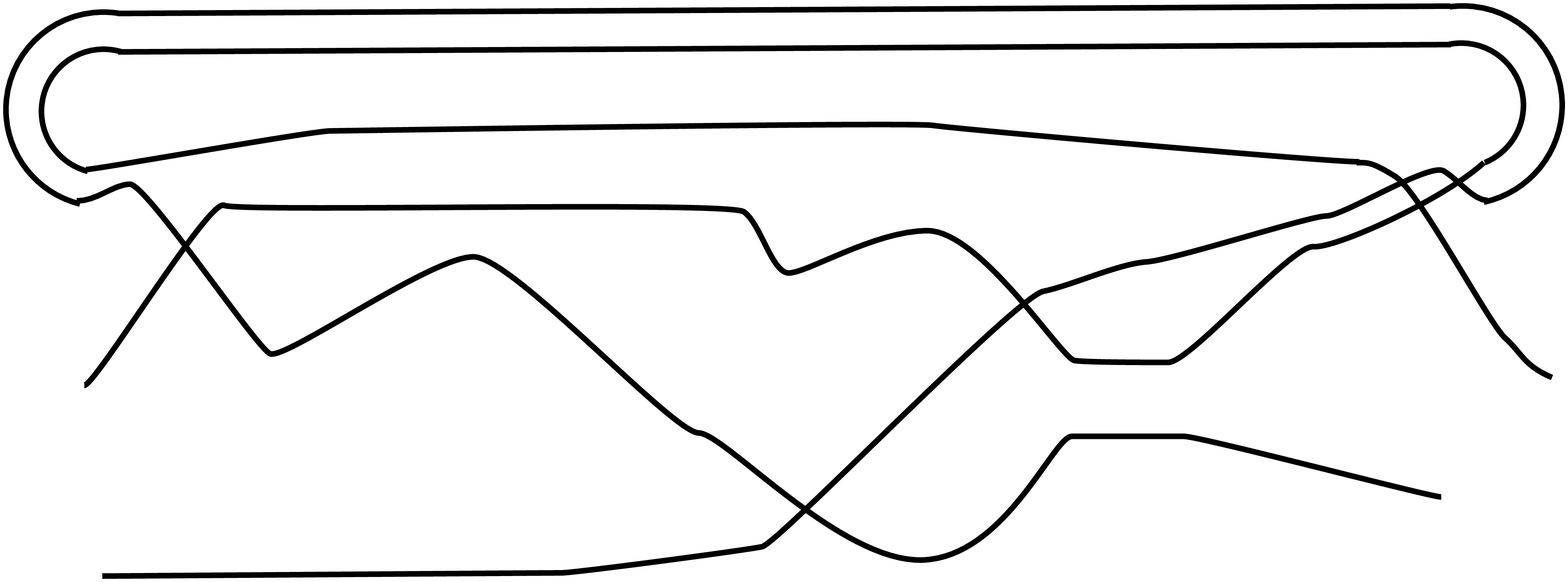}$
coh.:
$\includegraphics[height=0.4in]{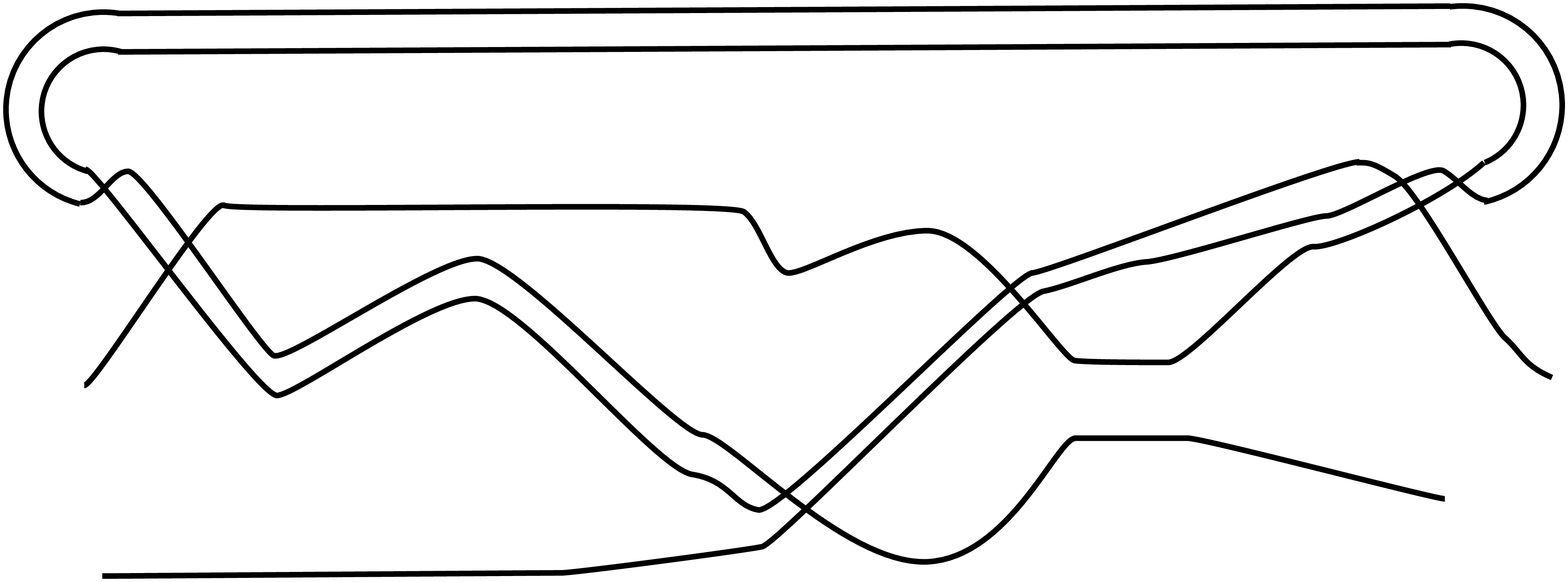}$
dinaturality:
$\includegraphics[height=0.4in]{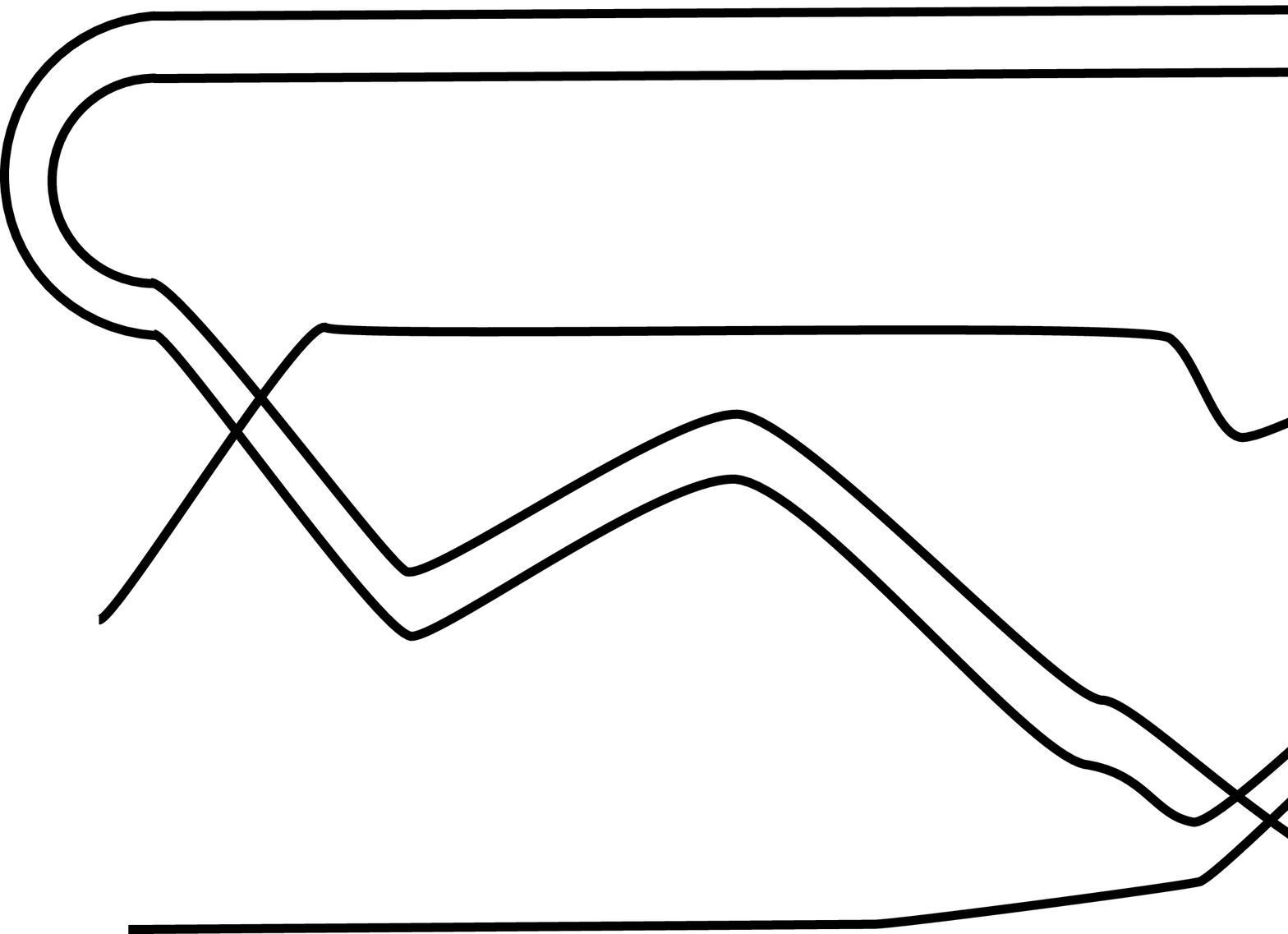}$\\

yanking:
$\includegraphics[height=0.4in]{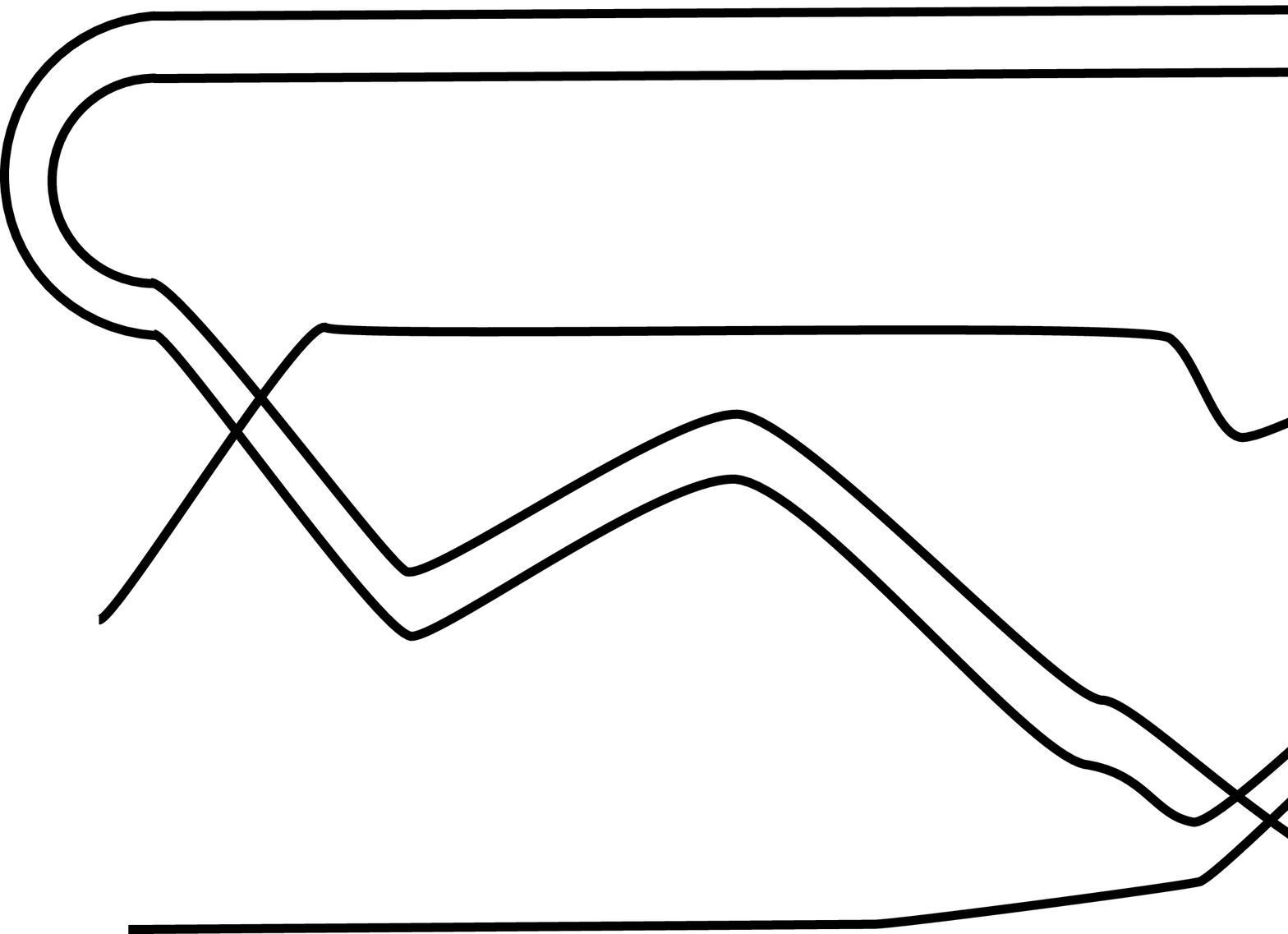}$
naturality:
$\includegraphics[height=0.4in]{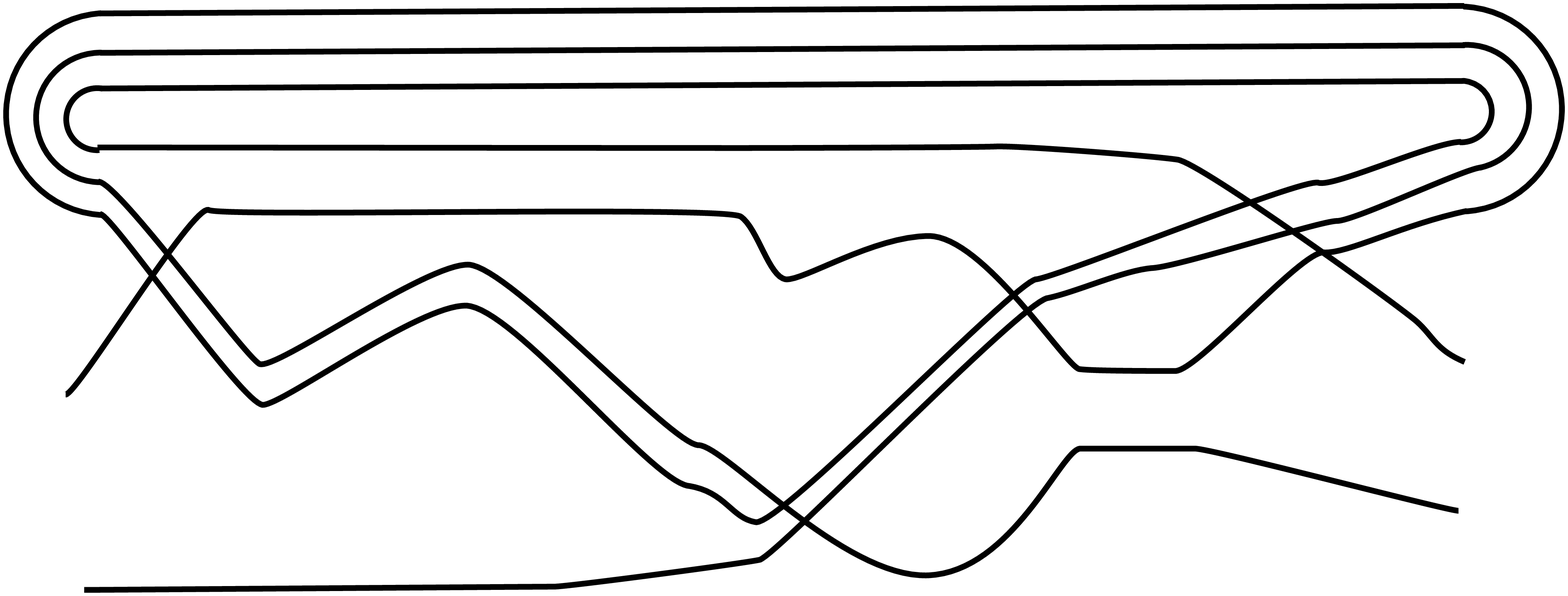}$
coherence:
$\includegraphics[height=0.4in]{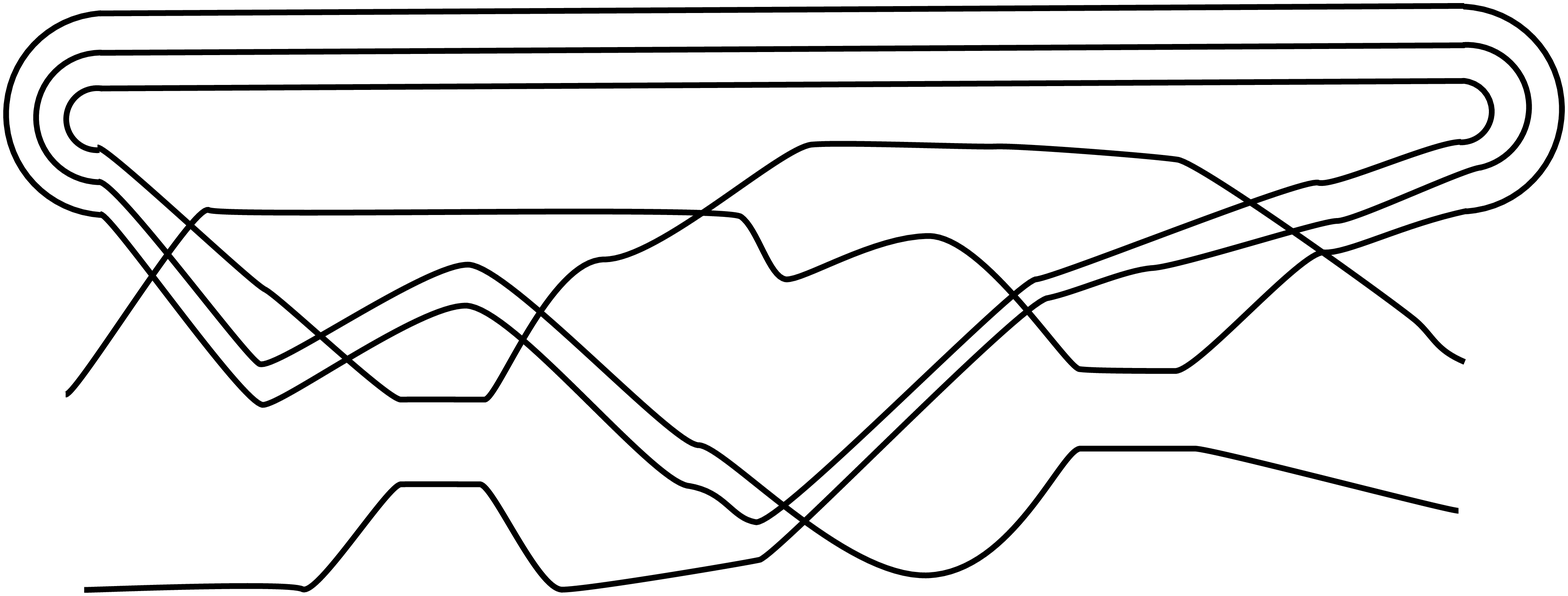}$\\

vanishing II:
$$\includegraphics[height=0.8in]{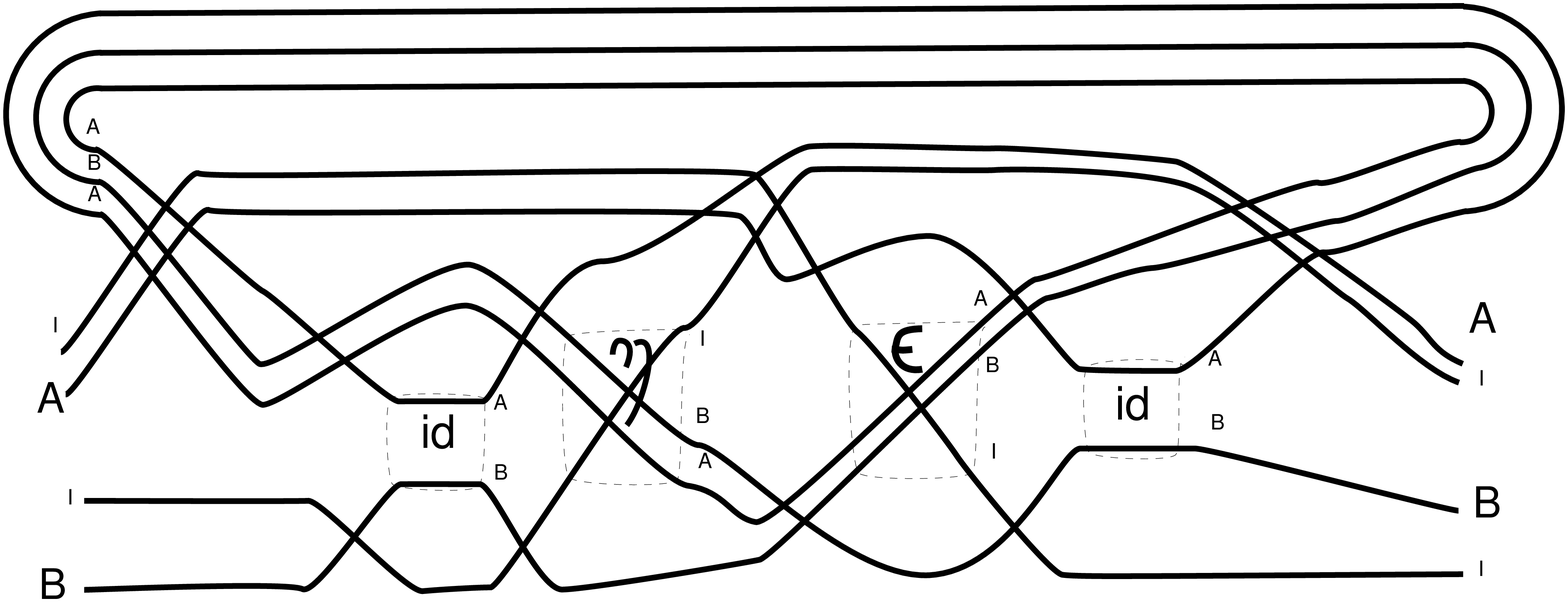}$$

\end{proof}

\begin{corollary}
Let $\cC$ be partially traced. Then $Int^p(\cC)$ is a compact closed paracategory.
\end{corollary}
\begin{proof}
This is a consequence of Lemma~\ref{COMPACT EQUATION INT}.
\end{proof}

Our final result for this section is that there exists a full and faithful, trace preserving functor from $\cC$ to $Int^p(\cC)$.

\begin{definition}
In a similar way as done in~\cite{JSV96}, we define a fully faithful functor between paracategories $N:\cC\rightarrow Int^{p}(\cC)$ defined by $N(A)=(A,I)$ and $N(f)=f$ by strictness of the category $\cC$.
\end{definition}
\begin{lemma}\label{N faithful}
 $N$ is a well-defined, full and faithful functor of paracategories.
\end{lemma}
\begin{proof} To prove well-definedness, note that we are considering the category $\cC$ as a paracategory with composition $[f_1,\ldots,f_n] = f_n\circ\ldots\circ f_1$ as its partial operation, and $[-]'$ the partial composition defined in $Int^{p}(\cC)$. Thus, $N([\vec{f}])=\vec{[N(f)]}'$, since by the Vanishing I axiom, the trace operator is totally defined when we restrict it to this type of arrows i.e., $\Trc^I_{A,B}=\cC(A\otimes I,B\otimes I)$ and $\Tr^I_{A,B}(f)=f$.

By definition, $N(f)=N(g)$ implies $f=g$, which proves faithfulness. If we take and arrow in $Int^p((A,I),(B,I))$, let us say for example $f:(A,I)\rightarrow (B,I)$, which really means in $\cC$  an arrow of type $f:A\otimes I\rightarrow B\otimes I$, then we just choose the same $f$ obtaining $Nf=f$. This proves fullness.
\end{proof}

\begin{lemma}\label{N PRESERVES TRACE}
The functor $N:\cC\rightarrow Int^p(\cC)$ preserves the trace, i.e., if $f:A\otimes U\rightarrow B\otimes U$ is in $\Trc^U_{A,B}$ then $N(\Tr^U_{A,B}(f))=\Tr^{NU}_{NA,NB}(Nf):(A,I)\rightarrow (B,I)$ which means
$$N(\Tr^U_{A,B}(f))=[1\otimes \eta;f\otimes 1;1\otimes \sigma;1\otimes\varepsilon].$$
\end{lemma}
\begin{proof}
Let us start with $N(\Tr^U_{A,B}(f)):A\otimes I \rightarrow B\otimes I$ in $\cC$ which is represented by

$$\includegraphics[height=0.9in]{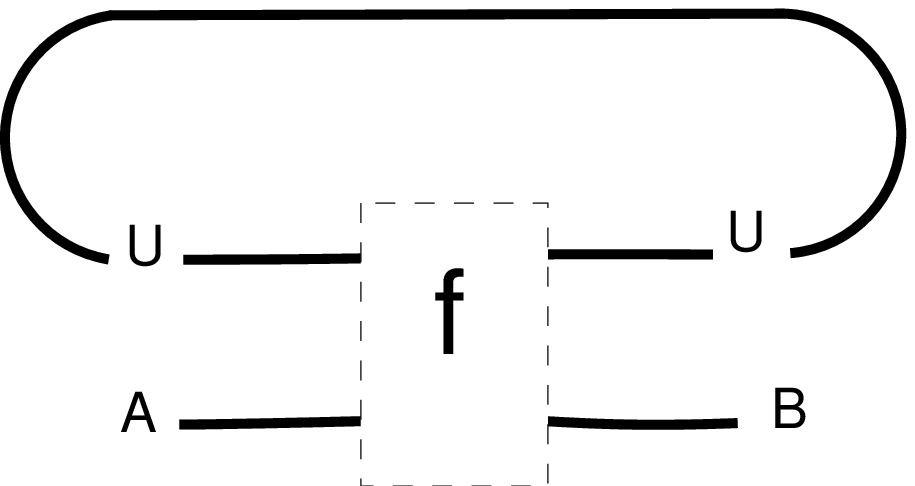}$$

Notice that by hypothesis we have $f\in\Trc^U$. Let us call this hypothesis: condition $(A)$.

By the yanking axiom $\sigma_{U,U}\in\Trc^U$ where the trace is locally represented by  $$\includegraphics[height=1.2in]{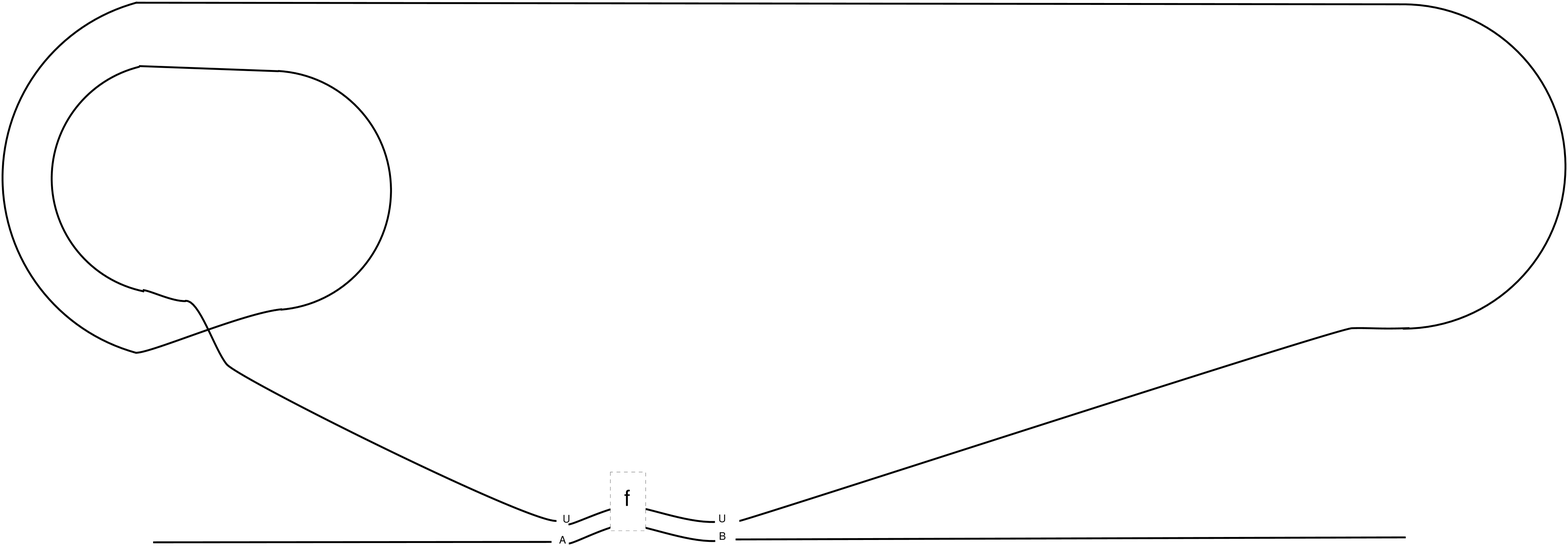}$$

and by applying superposing axiom $\sigma\otimes 1_A\in\Trc^U$ and then by applying the naturality axiom we obtain that the full diagram below this trace is in $\Trc^U$ (let us call it condition $(B)$), i.e.,

$$\includegraphics[height=0.6in]{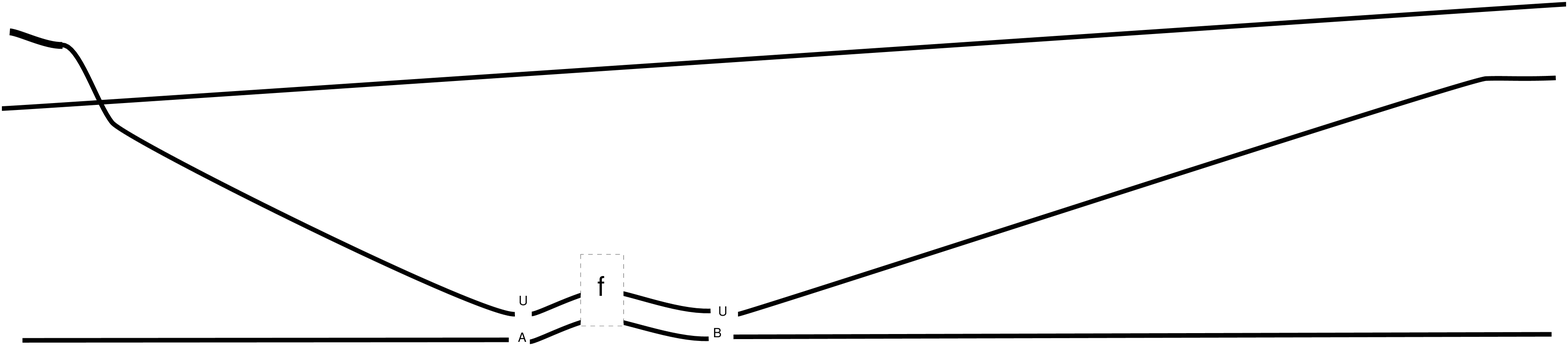}\,\,\,\in \Trc^U.$$

The trace of this graph is equal to $f$ which implies by condition $(A)$ that is in $\Trc^U$ i.e.,
$$\includegraphics[height=0.8in]{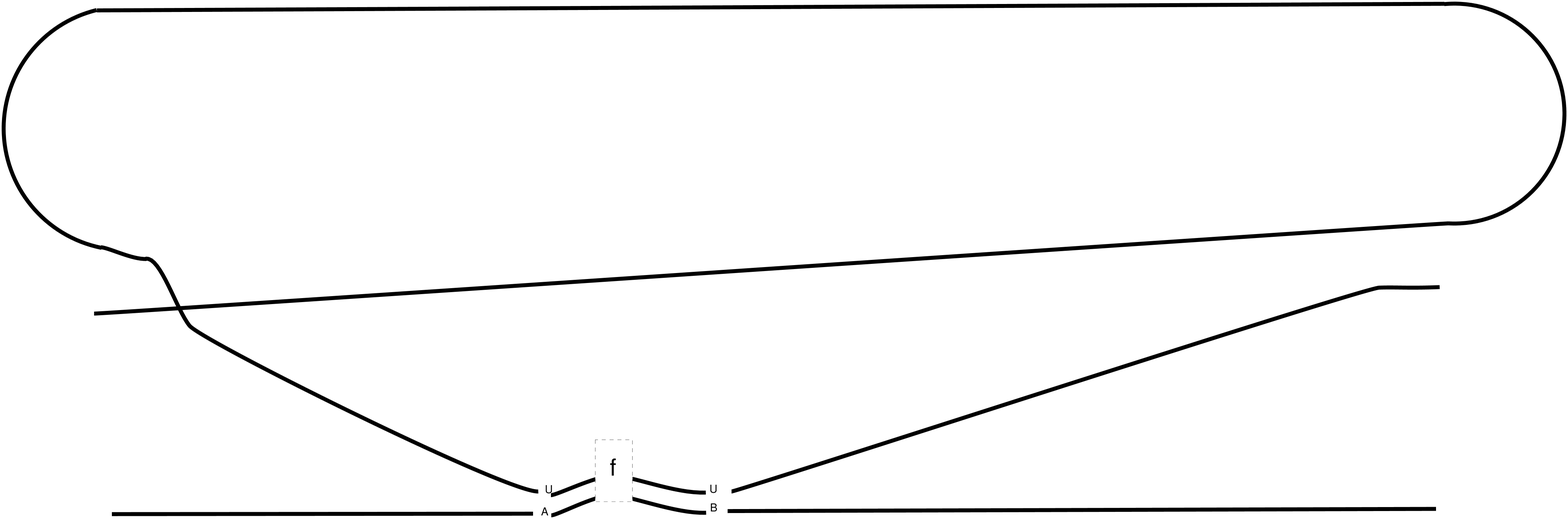}\,\,\,\in \Trc^U.$$

From condition $(A)$ and $(B)$  and the vanishing II axiom we conclude that

$$\includegraphics[height=0.6in]{tracepreservation15BB.eps}\,\,\,\in \Trc^{U\otimes U}$$

(let us call it condition $A+B$) and the trace is represented by:
$$\includegraphics[height=1.4in]{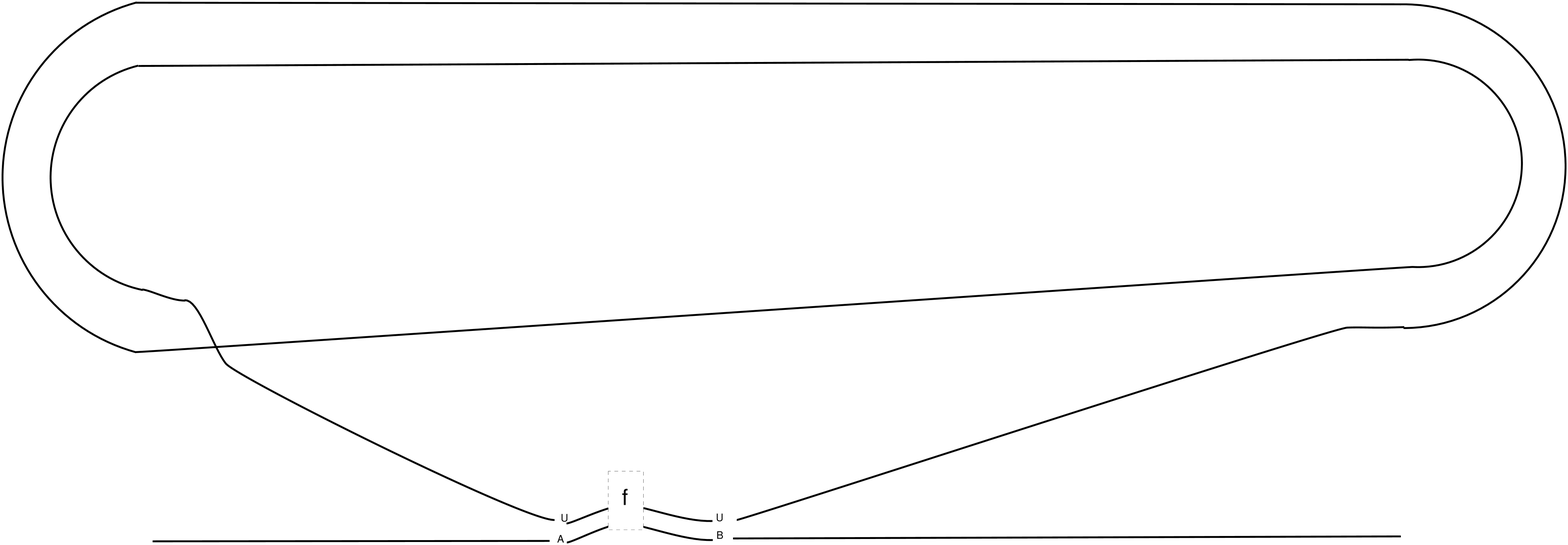}.$$
We repeat this operation, by yanking:$$\includegraphics[height=1.4in]{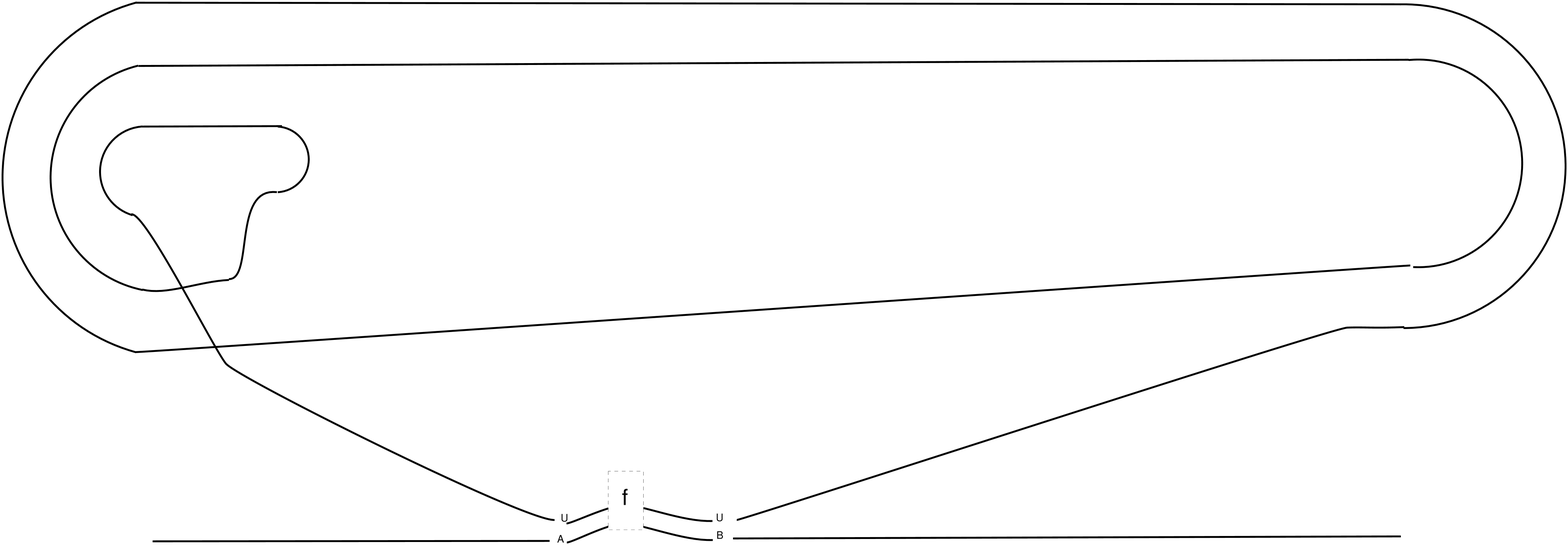}$$

and naturality we obtain that the diagram in the dotted box:$$\includegraphics[height=1.4in]{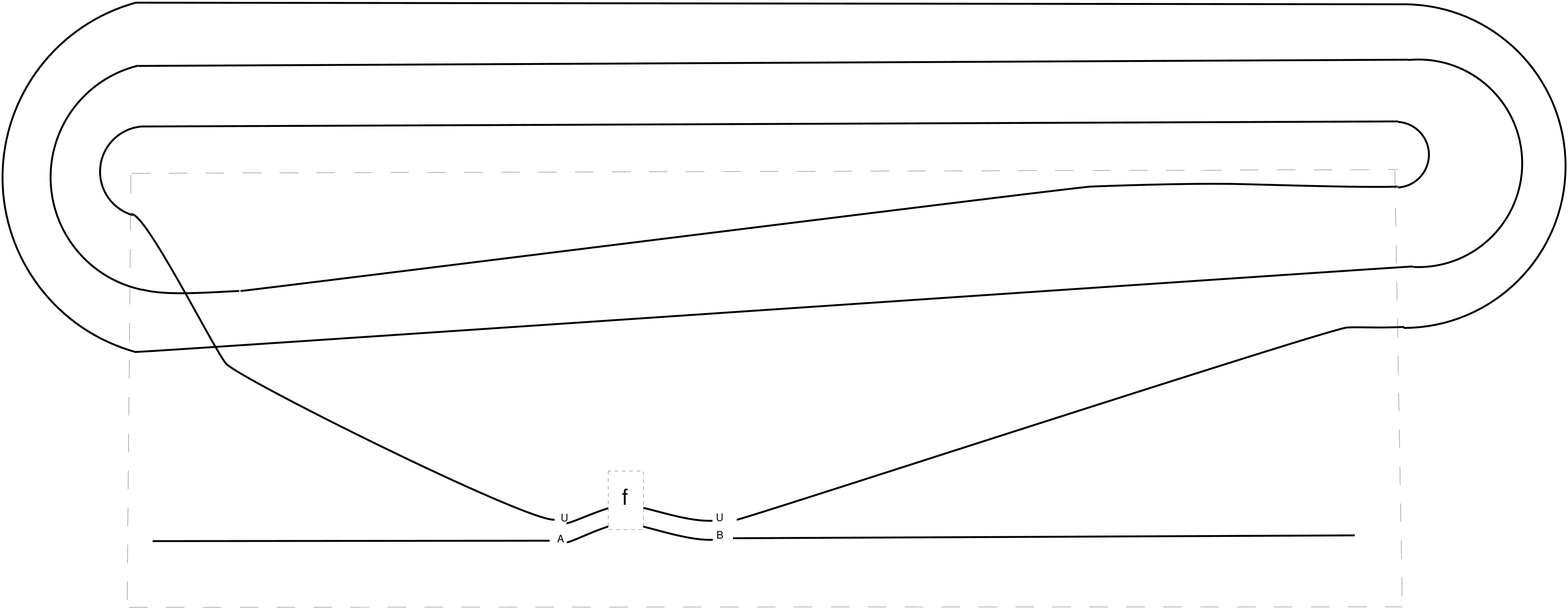}$$
is in $\Trc^U$.

Hence, after any further coherent change we made in the graph, it will remain in the trace class $\Trc^U$. Let us call it condition $(C)$; where the trace will be represented by  $$\includegraphics[height=1.4in]{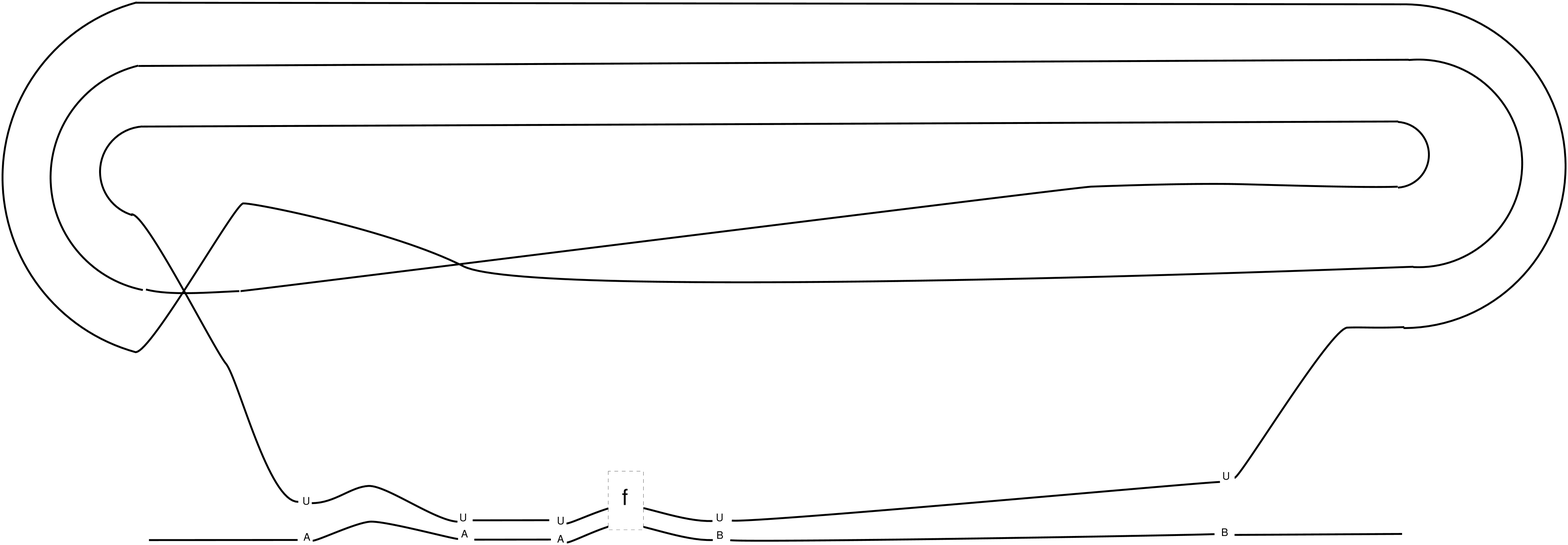}$$

coherence:
$$\includegraphics[height=0.8in]{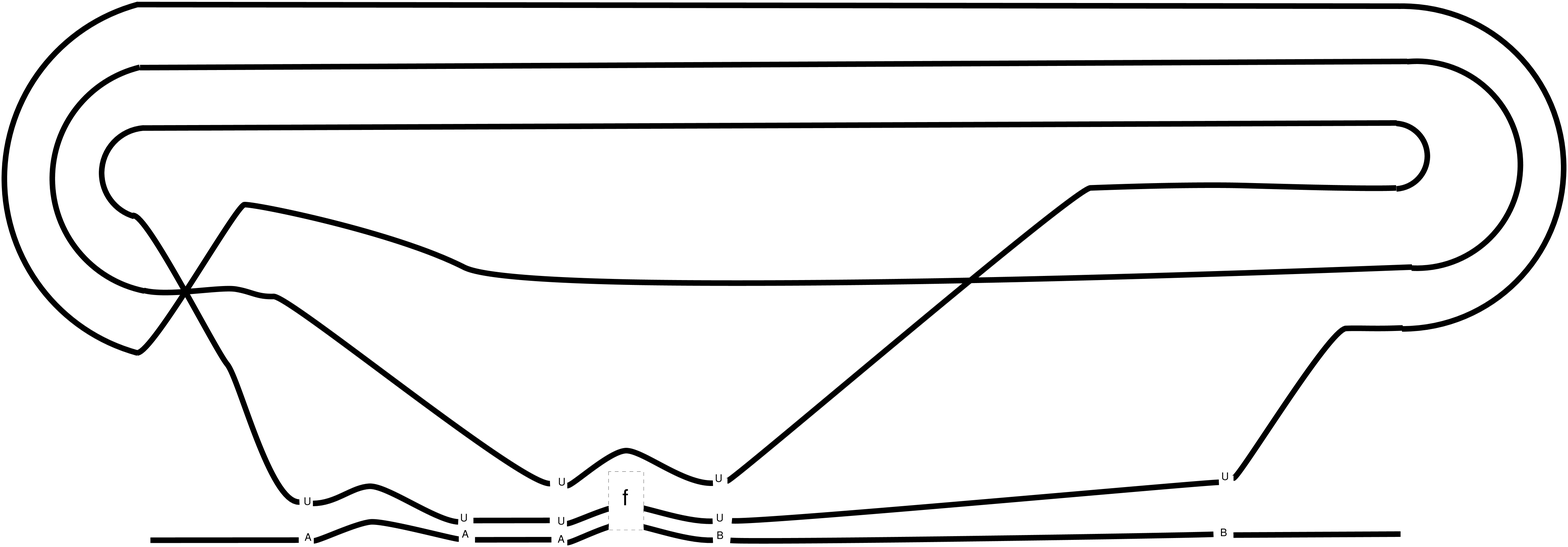}\,\,\, \mbox{and}\,\,\,\includegraphics[height=0.8in]{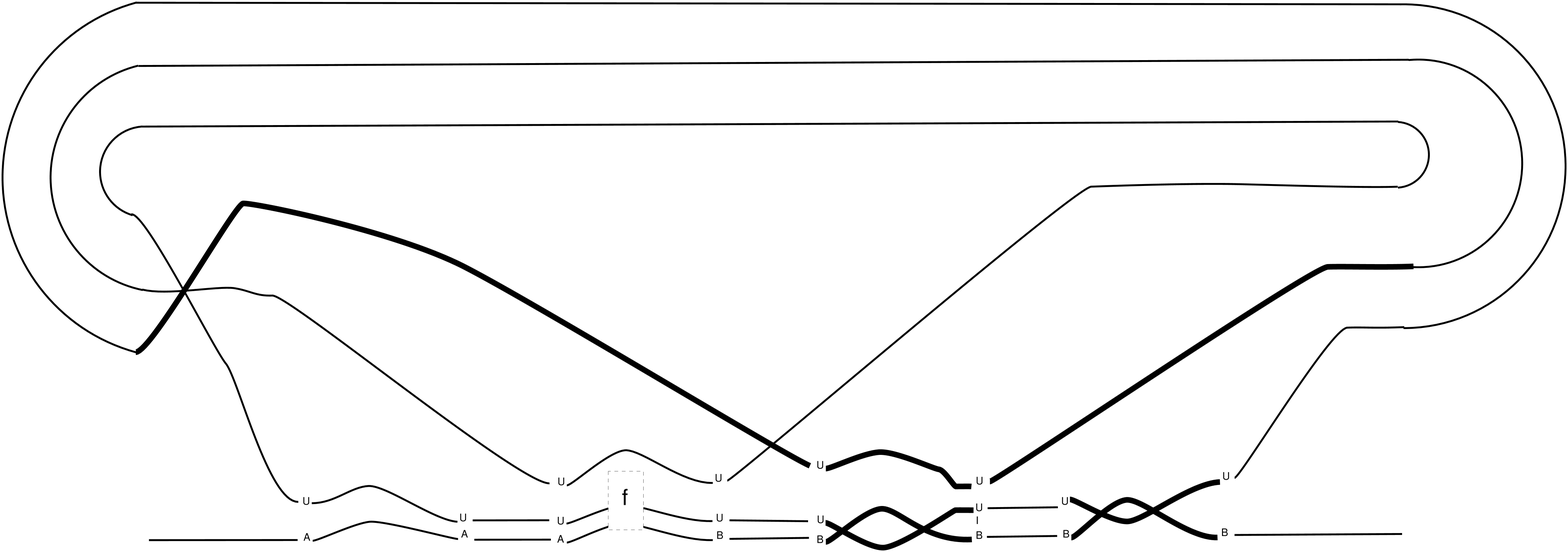}.$$

In the same way as above: by condition $A+B$, $C$ and the vanishing II axiom we obtain that the graph is in the trace class $\Trc^{U\otimes U\otimes U}$ and the trace given by
$$\includegraphics[height=2in]{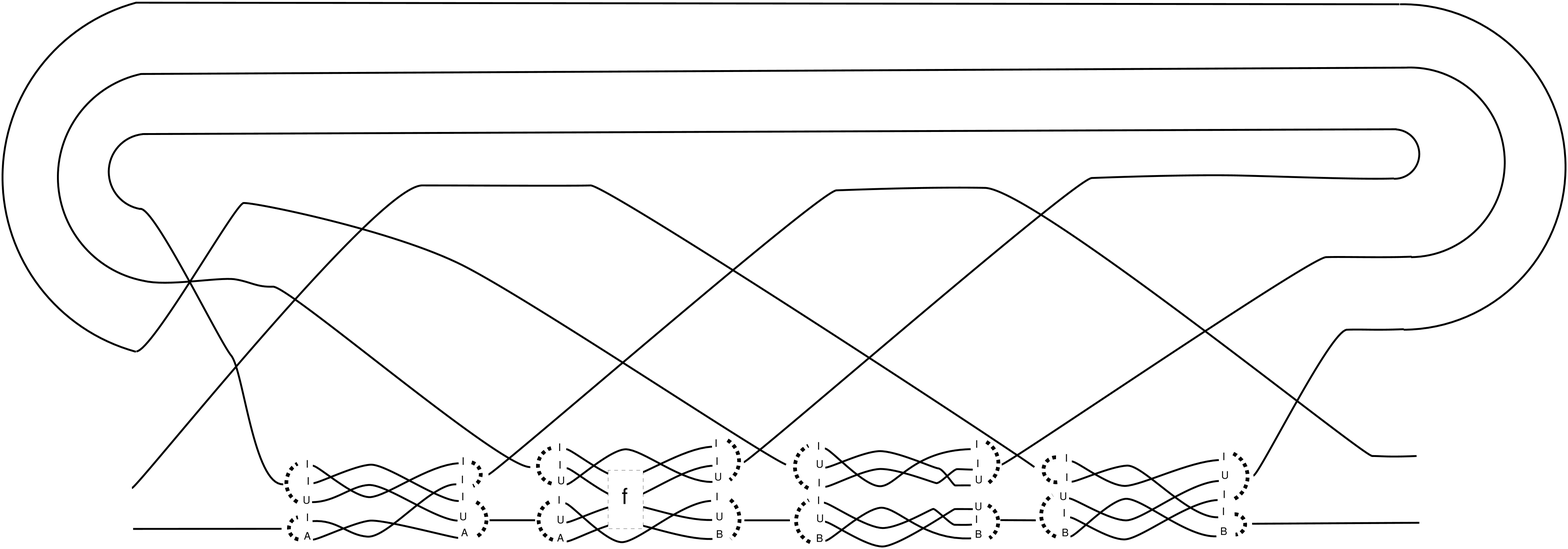}$$

Now, since $\cC$ is a strict category we can represent the last diagram in the following way:
$$\includegraphics[height=2in]{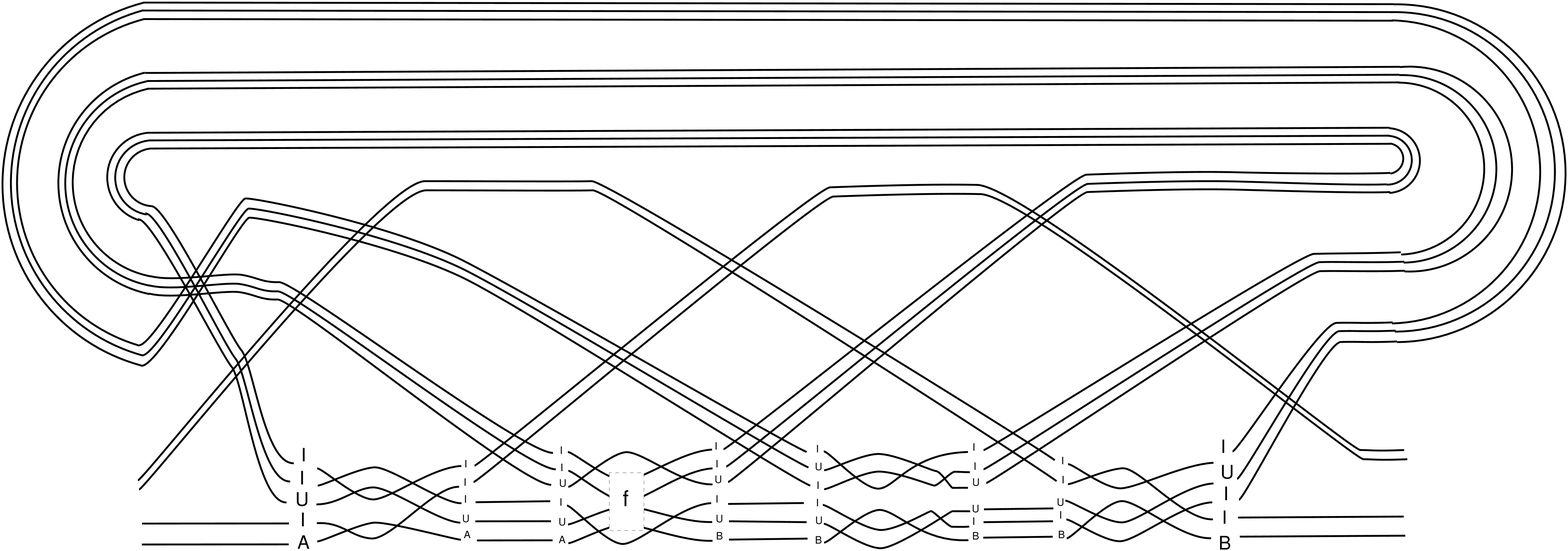}$$

which is equal to $[1\otimes \eta;f\otimes 1;1\otimes \sigma;1\otimes\varepsilon].$

\end{proof}

\section{Representation theorem for partially traced categories}

\begin{theorem}
Every (strict) symmetric partially traced category can be faithfully embedded in a totally traced category.
\end{theorem}
\begin{proof} This follows from the various lemmas.
Let $\cC$ be a strict symmetric partially traced category. By Lemmas~\ref{N faithful} and~\ref{N PRESERVES TRACE}, $\cC$
 can be faithfully embedded in a compact closed paracategory $Int^p(\cC)$,
 and the embedding is trace preserving. By Lemma~\ref{FAITHFUL EMBED COMP CLOSED PARA}, $Int^p(\cC)$ can be
 faithfully embedded in a compact closed category $\cP(Int^p(\cC))/{\sim}$
 (and the embedding preserves the compact closed structure, hence the
 trace). Since $\cP(Int^p(\cC))/{\sim}$ is compact closed, it is totally
 traced, which proves the theorem.
\end{proof}

\begin{remark}
Notice that by the Lemma~\ref{N PRESERVES TRACE} above if $f:A\otimes U\rightarrow B\otimes U$ is in $\Trc^U_{A,B}$ then $[1\otimes \eta;f\otimes 1;1\otimes \sigma;1\otimes\varepsilon]\downarrow$; therefore the projection functor
$$F:Int^p(\cC)\rightarrow \cP(Int^p(\cC))/{\sim}$$
also preserves the trace $F(\Tr^U_{A,B}(f))=\Tr^{FU}_{FA,FB}(Ff)$
since we have that

$F(\Tr^U_{A,B}(f))=F[1\otimes \eta;f\otimes 1;1\otimes \sigma;1\otimes\varepsilon]=\overline{[1\otimes \eta;f\otimes 1;1\otimes \sigma;1\otimes\varepsilon]}=\overline{1\otimes \eta;f\otimes 1;1\otimes \sigma;1\otimes\varepsilon}=\overline{1\otimes \eta}\circ \overline{f\otimes 1}\circ\overline{1\otimes \sigma}\circ\overline{1\otimes\varepsilon}=\overline{1}\hat{\otimes} \overline{\eta}\circ\overline{f}\hat{\otimes} \overline{1}\circ\overline{1}\hat{\otimes} \overline{\sigma}\circ\overline{1}\hat{\otimes}\overline{\varepsilon}=\Tr^{FU}_{FA,FB}(Ff).$
\end{remark}

\section{Universal property}

The category $(\cP(Int^p(\cC))/{\sim},\hat{\otimes},I,s)$ satisfies the following universal property.

\begin{proposition}
Let $\cC$ be a partially traced category and $\cD$ a compact closed category. If $F:\cC\rightarrow\cD$ is a strict monoidal traced functor then there exists a unique monoidal functor $L:\cP(Int^p(\cC))/{\sim}\rightarrow\cD$ such that
$$\xymatrix@=25pt{
\mathcal{C}\ar[rrr]^{\hat{N}}\ar[rrrd]_{F}
  &&& \cP(Int^p(\cC))/{\sim} \ar[d]^{L}\\
  &&& \mathcal{D}
}$$
where $\hat{N}$ is $\cC\stackrel{N}\longrightarrow Int^p(\cC)\stackrel{\pi}\longrightarrow\cP(Int^p(\cC))/{\sim}$
\end{proposition}
\begin{proof}
We first construct a monoidal functor $K:Int^p(\cC)\rightarrow\mathcal{D}$ such that $K\circ N=F$. This functor is defined in the same way as in~\cite{JSV96}, and is in fact unique.

On objects $K(A,U)=FA\otimes (FU)^*$ and given $(A,U)\stackrel{f}\longrightarrow (B,V)$ we define $K(f)$ as
$$FA\otimes FU^* \stackrel{1\otimes\eta\otimes 1}\longrightarrow FA\otimes FV\otimes FV^*\otimes FU^* \stackrel{Ff\otimes 1}\longrightarrow FB\otimes FU\otimes FV^*\otimes FU^* \stackrel{(1\otimes\sigma\otimes 1)\circ(1\otimes \varepsilon\sigma)}\longrightarrow FB\otimes FV^*$$

Graphically this is represented by the following diagram

$$\includegraphics[height=0.8in]{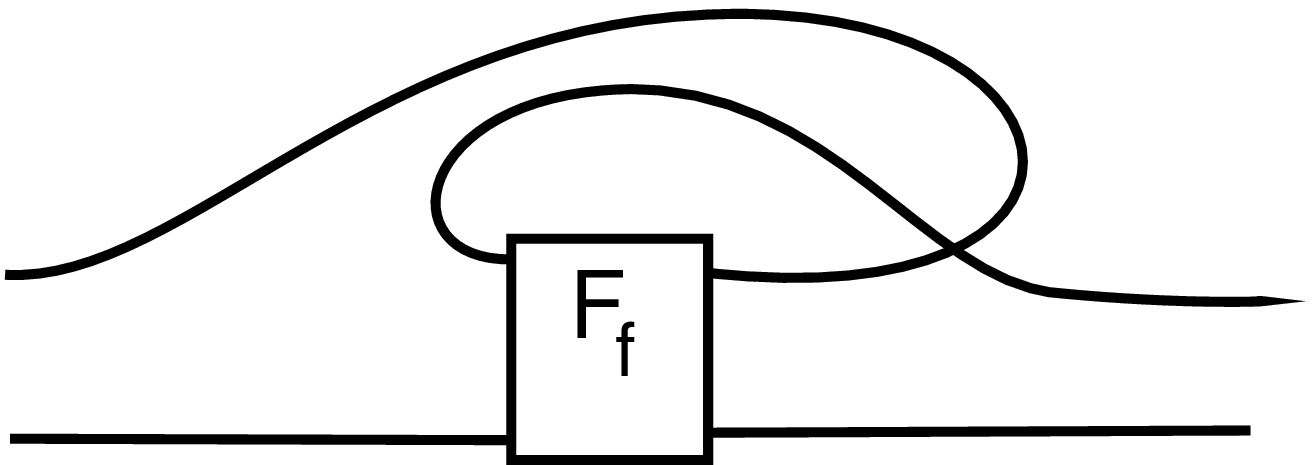}.$$

We need to prove that $K$ is a functor between paracategories, i.e., if $[f_1,\dots f_n]\downarrow$ then
$K[f_1,\dots f_n]=[Kf_1,\dots Kf_n]$. The remaining properties of $K$ are proved as in~\cite{JSV96}.

Without loss of generality we take $n=4$.
Therefore we have
\begin{equation}\label{K FUNCTOR}
K[f_1,\dots f_4]=\includegraphics[height=0.5in]{universal16.eps}.
\end{equation}
where
$$f=\includegraphics[height=1.2in]{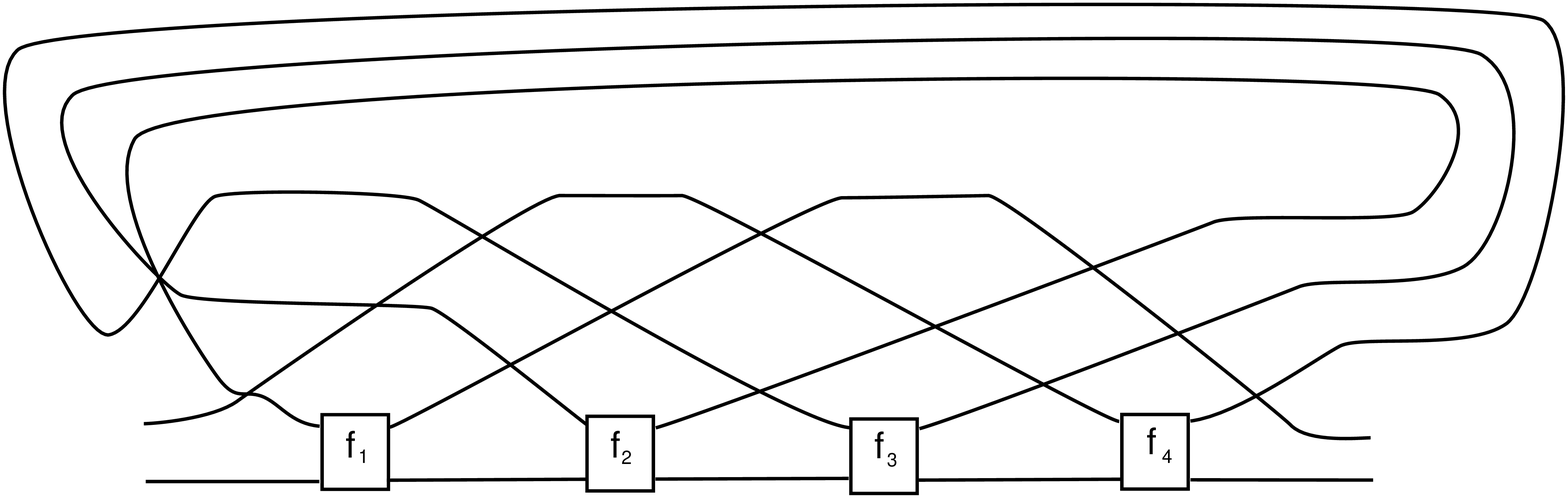}$$

Since $F$ preserves the trace, composition and symmetries we have that equation~(\ref{K FUNCTOR}) is equal to the following diagram

$$\includegraphics[height=1.5in]{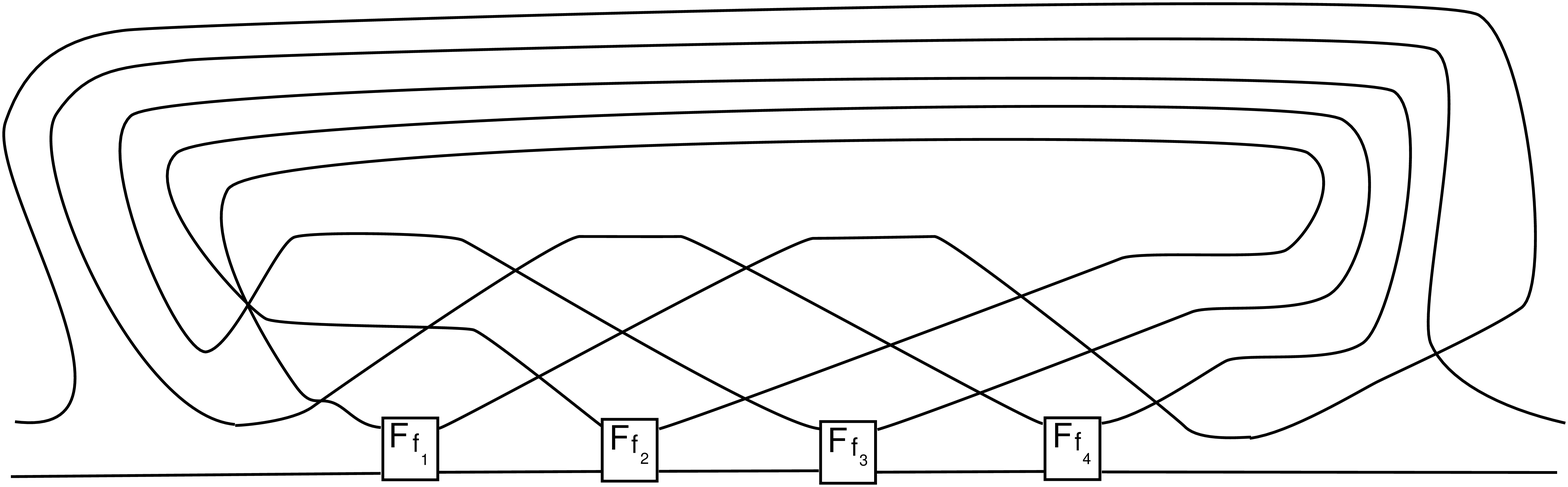}.$$

Notice that the category $\cD$ is compact closed and its trace is totally defined and given by composition of unit $\eta$, counit $\varepsilon$, symmetries $\sigma$ and arrows $Ff_i$ in $\cD$, $i=1\dots 4$. Therefore, by coherence in $\cD$, we transform the previous diagram into

$$\includegraphics[height=0.7in]{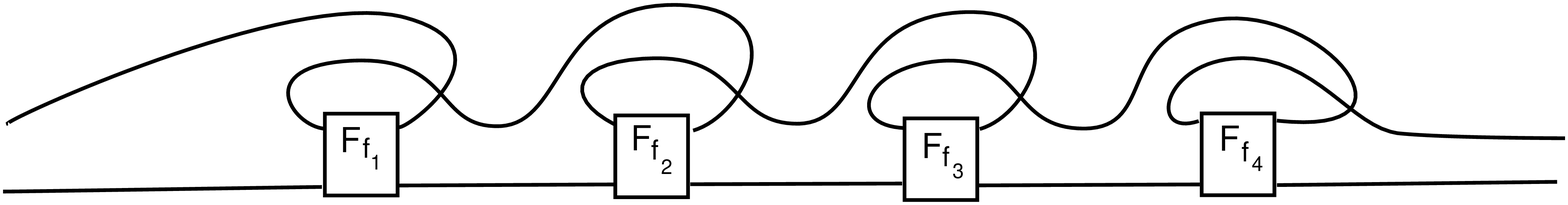}$$

i.e., $[Kf_1,\dots Kf_4]$.

Given $K$, we use Theorem~\ref{Freeness} to obtain a unique $L$ such that:
$$\xymatrix@=25pt{
\mathcal{C}\ar[r]^{N}\ar[rrd]_{F}&Int^p(\cC)\ar[r]^{\pi}\ar[dr]^{K}& \cP(Int^p(\cC))/{\sim} \ar[d]^{L}\\
  && \mathcal{D}
}$$

Uniqueness: Suppose $L':\cP(Int^p(\cC))/{\sim}\rightarrow\mathcal{D}$ is another monoidal functor such that $L'\circ\pi\circ N=F$. Then $K'=L'\circ\pi$ satisfies $K'\circ N=F$ so by uniqueness of $K$, it follows that $K=K'$. But then $L'\circ\pi=K$, and by uniqueness of $L$, we have $L=L'$.

\end{proof}

\chapter{Background material on presheaf categories}
Here we review some of the basic and advanced concepts of functor categories that will be used in Chapters~\ref{PREASHEAVES MODELS OF Q C} and~\ref{A CONCRETE MODEL}.
For additional details, see \cite{MacLane 91}, \cite{Borceux94}, \cite{Lambek66}, \cite{KELLY82}.

\section{Universal arrows, representable functors, and the Yoneda Lemma}
\begin{definition}
\rm Let $F:\cA\rightarrow\cB$ be a functor and $B\in\cB$. A pair $(A,f)$ where $A\in\cA$ and $f:B\rightarrow F(A)$ is said to be a \textit{universal arrow} from $B$ to $F$ when for every arrow $f':B\rightarrow F(A')$ there is a unique arrow $g:A\rightarrow A'$ in the category $\cA$ such that
$$\xymatrix@=25pt{
B\ar[r]^{f}\ar[rd]_{f'} & F(A)\ar[d]^{F(g)}\\
  & F(A')
}$$
is a commutative diagram.
\end{definition}
\begin{definition}
\rm A \textit{universal element} of the functor $F:\cA\rightarrow{\bf Set}$ is an object $A\in\cA$ and an element
$x\in F(A)$ such that for any other pair $A'\in\cA$ and $x'\in F(A')$ there exists a unique $f:A\rightarrow A'$ that satisfies $F(f)(x)=x'$.
\end{definition}

\begin{definition}\label{CATEGORY OF ELEMENTS}
\rm
Let $F:\cC\rightarrow {\bf Set}$ be a functor. The category $El(F)$ of {\em elements} is given by the following data:
\begin{itemize}
\item[(a)] objects of $El(F)$ are pairs $(C,x)$ where $x\in FC$ and $C\in |\cC|$.
\item[(b)] morphisms  $f:(C,x)\rightarrow (D,y)$ are arrows $f:C\rightarrow D$ in the category $\cC$ such that $Ff(x)=y$.
\end{itemize}
\end{definition}
\begin{definition}
\rm
An object $A\in\cA$ is said to be the {\em representing object} of a functor $F:\cA\rightarrow {\bf Set}$ when there is a natural isomorphism $\phi$:
$$\cA(A,-)\stackrel {\phi}\rightarrow F.$$
When this occurs we said that $F$ is a {\em representable functor}. There is a distinguished element of this isomorphism $\phi_A(1_A)\in F(A)$, which is called the {\em unit} of the representation.
\end{definition}
\begin{theorem}[The Yoneda Lemma]
\label{YONEDA LEMMA}
Let $F:\cA\rightarrow {\bf Set}$ be a functor, $A\in\cA$. There exists a bijection
$$\vartheta_{F,A}:[\cA,{\bf Set}](\cA(A,-),F)\rightarrow F(A)$$
which is natural in $A$ and if $\cA$ is a small category $\vartheta$ is natural in $F$.
\end{theorem}
\begin{proof}
\cite{Borceux94}
\end{proof}

\begin{theorem}
Let $F:\cA\rightarrow {\bf Set}$ be a functor, $F$ is representable iff it has a universal element.
\end{theorem}
\begin{proof}
\cite{MacLane 91}
\end{proof}

\section{Limits and colimits}

Let $\cA$ and $\cB$ be categories. For every object $A\in \cA$ the constant functor is defined to be $\Delta_A:\cB\rightarrow\cA$ with $\Delta_A(B)=A$ and $\Delta_A(f)=1_A$ when $B\stackrel {f}\rightarrow B'$. If $A\stackrel {g}\rightarrow A'$ is an arrow in $\cA$ there is a natural transformation $\Delta(g):\Delta _A\Rightarrow\Delta_{A'}$ defined $(\Delta(g))(B)=g$. These functors and natural transformations define a functor $\Delta:\cA\rightarrow [\cB,\cA]$.\\

Let $F:\cJ\rightarrow\cA$ be a functor. The definition of limits and colimits can be characterized by objects that represent the following functors:
\begin{equation}
\cA(-,\mbox{lim}\, F)\cong [\cJ,\cA](\Delta-,F):\cA^{op}\rightarrow {\bf Set} \label{REPRESLIMIT}
\end{equation}
and
\begin{equation}
\cA(\mbox{colim}\, F,-)\cong [\cJ,\cA](F,\Delta-):\cA\rightarrow {\bf Set}. \label{REPRESCOLIMIT}
\end{equation}
To see this, suppose we have $\cA(-,\mbox{lim}\, F)\stackrel {\phi}\rightarrow [\cJ,\cA](\Delta -,F)$. Then $\phi_{limF}(1_{limF}):\Delta_{limF}\Rightarrow F$ is a cone determined by the universal element. If $\Delta\stackrel {\alpha}\Rightarrow F$ is another cone then $\phi^{-1}_A(\alpha):A\rightarrow lim\,F$ is an arrow on the category $\cA$ such that by naturality we have:

$$\xymatrix@=25pt{
 \cA(\mbox{lim}\,F,\mbox{lim}\, F)\ar[rr]^{\phi_{lim F}}\ar[d]_{\cA(\phi^{-1}_A(\alpha),\mbox{lim}\, F)}& & [\cJ,\cA](\Delta \mbox{lim}\,F,F)\ar[d]^{[\cJ,\cA](\Delta (\phi^{-1}_A(\alpha)),F)}\\
 \cA(A,\mbox{lim}\, F)\ar[rr]^{\phi_{A}} && [\cJ,\cA](\Delta A,F)
}$$
which implies by evaluating at $1_{\mbox{lim}\, F}$ that:
$$\phi_{limF}(1_{limF})\circ \Delta (\phi^{-1}_A(\alpha))=\phi_{A}(\phi^{-1}_{A}(\alpha)):\Delta A\Rightarrow F.$$
Graphically:
$$\xymatrix@=25pt{
 \Delta A\ar@{=>}[r]^{\alpha}\ar@{=>}[rd]_{\Delta (\phi^{-1}_A(\alpha))} & F\\
 & \Delta \mbox{lim}\,F\ar@{=>}[u]_{\phi_{limF}(1_{limF})}
}$$
Therefore, evaluating at $i\in\cJ$:
$$\xymatrix@=25pt{
  A\ar[r]^{\alpha_i}\ar[rd]_{\phi^{-1}_A(\alpha)} & F(i)\\
 & \mbox{lim}\,F\ar[u]_{(\phi_{limF}(1_{limF}))(i)}
}$$

\section{Dinatural transformations, ends, and co-ends}

Next, we recall the notion of dinatural transformation. The case which interests us the most is when one of the functors involved is a constant functor.
\begin{definition}[Dinatural transformation]
\rm
Suppose we have two functors $F,G:\cA^{op}\times\cA\rightarrow \cB$, a family of maps $\alpha:F\stackrel{..}\longrightarrow G=\{\alpha_A:F(A,A)\rightarrow G(A,A)\}_{A\in |\cA|}$ is called a \textit{dinatural transformation} when for every  arrow $f:A\rightarrow B$ the following holds:
$$\xymatrix@=25pt{
&F(A,A)\ar[r]^{\alpha_A}&G(A,A)\ar[rd]^{G(1,f)}&\\
F(B,A)\ar[ru]^{F(f,1)}\ar[rd]_{F(1,f)}&&&G(A,B)\\
&F(B,B)\ar[r]_{\alpha_B}&G(B,B)\ar[ru]_{G(f,1)}&
}$$
\end{definition}
\begin{example} \label{EXAMPLE COEND AS DINATURAL MAP}
Let $S:\cA^{op}\rightarrow {\bf Set}$ be a functor, and let $B\in |\cA|$. There are two functors $F,G:\cA^{op}\times\cA\rightarrow {\bf Set}$ defined by $F(A',A)=S(A')\times \cA(B,A)$, and $G=\Delta(S(B))$, the constant functor. Let us consider maps of type $\lambda_A:S(A)\times \cA(B,A)\rightarrow S(B)$ with $\lambda_A(x,f)=S(f)(x)$. Then $\lambda:F\to G$ is
a dinatural transformation: for all $f:A'\to A$,
$$\xymatrix@=25pt{
&S(A)\times \cA(B,A)\ar[rd]^{\lambda_A}&\\
S(A)\times\cA(B,A')\ar[ru]^{1\times \cA(B,f)}\ar[rd]_{S(f)\times 1}&&S(B)\\
&S(A')\times \cA(B,A')\ar[ru]_{\lambda_{A'}}&
}.$$
\end{example}
\begin{definition}[Wedge]
\rm
Given a functor $F:\cA^{op}\times\cA\rightarrow \cB$, a \textit{wedge} is a dinatural transformation from a constant  functor to $F$,
$$\lambda:\Delta(E)\stackrel{..}\longrightarrow F.$$
\end{definition}
\begin{definition}[End]
\rm
Given a functor $F:\cA^{op}\times\cA\rightarrow \cB$, an \textit{end} is a wedge
$$\lambda:\Delta(E)\stackrel{..}\longrightarrow F$$
satisfying a universal property: if there is another wedge $\alpha:\Delta(A)\stackrel{..}\longrightarrow F$ then there is a unique $g:A\rightarrow E$ with $\lambda_A\circ g=\alpha_A$ for every $A\in \cA$.
\end{definition}
In an analogous way we define the notion of co-end.\\
\begin{example}
In the example above we have that $S(B)$ with component $\lambda$ is a co-end for the functor $F$. Given a dinatural transformation
$\alpha_A:S(A)\times \cA(B,A)\longrightarrow X$ there is a unique $g:S(B)\rightarrow X$ given by $g(y)=\alpha_B(y,1_B)$ that satisfies the definition.
\end{example}

From the uniqueness of the universal property we conclude that, up to isomorphism, all the ends are equal. This justifies the following notation to indicate an end $E$ with components $\lambda_A$:
$$\int_{A}F(A,A)\stackrel{\lambda_A}\longrightarrow F(A,A)$$
and in the same way the co-end:
$$F(A,A)\stackrel{\lambda_A}\longrightarrow \int^{A}F(A,A).$$
\begin{theorem}
\label{PROPERTIES OF COEND 1}
Let $\alpha:F\Rightarrow G:\cA^{op}\times\cA\rightarrow \cB$ be a natural transformation. Suppose also that there exists the ends induced by $F$ and $G$:
\begin{equation}
\int_{A}F(A,A)\stackrel{\lambda_A}\longrightarrow F(A,A)\,\,\,\,\mbox{and}\,\,\,\,
\int_{A}G(A,A)\stackrel{\mu_A}\longrightarrow G(A,A)
\end{equation}
then there is a unique map $\int_{A}\alpha_{A,A}$ in the category $\cB$
such that:
$$\xymatrix@=25pt{
\int_A F(A,A)\ar[r]^{\lambda_A}\ar[d]^{\int_{A}\alpha_{A,A}}&F(A,A)\ar[d]^{\alpha_{A,A}}\\
\int_A G(A,A)\ar[r]_{\mu_{A}}&G(A,A)
}$$
\end{theorem}
\begin{proof}
\cite{MacLane 91}
\end{proof}

\begin{theorem}
Let $F:\cA\times\cB^{op}\times\cB\rightarrow \cC$ be a functor such that for each $A\in |\cA|$ there exists an end $$\int_{B,B}F(A,B,B)\stackrel{\lambda_B^A}\longrightarrow F(A,B,B).$$
Then there is a unique functor $U:\cA\rightarrow \cC$ with $U(A)=\int_BF(A,B,B)$ making $\lambda_B^A$ natural in $A\in |\cA|$.
\end{theorem}
\begin{proof}
\cite{MacLane 91}
\end{proof}

\section{Indexed limits and colimits}
\begin{definition}\label{isbell functor}
Let $\cA$ be a small category and $G:\cA\rightarrow\cB$ be functors. We define a functor $\hat{G}:\cB^{op}\rightarrow [\cA,{\bf Set}]$ whose values on objects are functors $$\hat{G}(B)=\cB(B,G-):\cA\rightarrow {\bf Set}$$
 and whose value on a morphism $B\stackrel {f}\rightarrow B'$ is a natural transformation $$\cB(f,G-):\cB(B',G-)\rightarrow\cB(B,G-).$$
 \end{definition}
 Let $F:\cA\rightarrow {\bf Set}$ be a functor.
 Thus we have a composition of functors:
$$\cB^{op}\stackrel {\hat{G}}\rightarrow [\cA,{\bf Set}]\stackrel {[\cA,{\bf Set}](F,-)}\longrightarrow{\bf Set}.$$

Suppose now that this composition admits a representation:
$$\phi:\cB(-,C)\cong [\cA,{\bf Set}](F,\hat{G}(-)).$$
\begin{definition}[Indexed limit]
Let us denote $C=\{F,G\}$, so we have that
$$\cB(B,\{F,G\})\cong [\cA,{\bf Set}](F,\cB(B,G-))$$
natural in $B$ with counit $\mu=\phi_{\{F,G\}}(1_{\{F,G\}}):F\rightarrow \cB(\{F,G\},G-)$ which has the property of being a universal element. Following Kelly's definition~\cite{KELLY82}, we name this pair $(\{F,G\},\mu)$ the limit of $G$ indexed by $F$.
\end{definition}
Thus $\mu\in [\cA,{\bf Set}](F,\cB(\{F,G\},G-))$ and if there is another $\lambda\in [\cA,{\bf Set}](F,\cB(B',G-))$ then there exists a unique $\{F,G\}\stackrel {g^{op}}\rightarrow B'$ in the category $\cB^{op}$ such that $([\cA,{\bf Set}](F,\cB(g,G-)))(\mu)=\lambda$ which means that $\cB(g,G-)\comp\mu=\lambda$.

Therefore,

$$\xymatrix@=25pt{
F(A)\ar[rr]^{\mu_A}\ar[rrd]_{\lambda_{A}}& & \cB(\{F,G\},G(A))\ar[d]^{\cB(g,G(A))}\\
  && \cB(B',G(A))
}.$$
Thus, after evaluating at $x\in F(A)$ we obtain:

$$\xymatrix@=25pt{
B'\ar[r]^{g}\ar[rd]_{\lambda_{A}(x)}& \{F,G\}\ar[d]^{\mu_{A}(x))}\\
  & G(A)
}$$
There is a bijection:
$$[\cA,{\bf Set}](F,\cB(B,G-))\cong [El(F),\cB](\Delta B,G\comp \pi_{F})$$
natural in $B$, with the projection $\pi_{F}:El(F)\rightarrow \cA$.\\
By equation~(\ref{REPRESLIMIT}):
$$\cB(B,\mbox{lim}\,G\comp \pi_{F})\cong [El(F),\cB](\Delta B,G\comp \pi_{F})$$
we conclude that:
\begin{proposition}
$$\mbox{lim}\,G\comp\pi_{F}=\{F,G\}$$
\end{proposition}
\begin{proof}
To see this bijection we have that every natural transformation $\alpha\in[\cA,{\bf Set}](F,\cB(B,G-))$ and for every $f:A\rightarrow A'$ there is a diagram:

$$\xymatrix@=25pt{
 F(A)\ar[rr]^{\alpha_A}\ar[d]_{F(f)}& & \cB(B,G(A))\ar[d]^{\cB(B,G(f))}\\
 F(A')\ar[rr]^{\alpha_{A'}} && \cB(B,G(A'))
}$$
which translates into a diagram:
$$\xymatrix@=25pt{
 B\ar[d]_{\alpha_A(a)}\ar[drr]^{\alpha_{A'}(F(f)(a))}& & \\
 G(A)\ar[rr]^{G(f)} && G(A')
}$$
for every $a\in F(A)$.\\
\end{proof}
\begin{remark}
When we choose $F=\Delta 1$
$$\cB(B,\mbox{lim}\,G)\cong[\cA,\cB](\Delta B,G)\cong[\cA,{\bf Set}](\Delta 1,\cB(B,G-))$$
we obtain by definition that
$$\mbox{lim}\,G=\{\Delta 1,G\}$$
\end{remark}
\begin{definition}[Indexed colimit]\label{index colimits}
In the same way as above by duality we define the colimit of $G:\cA\rightarrow \cB$ indexed by $F:\cA^{op}\rightarrow {\bf Set}$ as the representing pair $(F\star G,\lambda)$ of the functor:
$$[\cA^{op},{\bf Set}](F,\tilde {G}(-)):\cB\rightarrow {\bf Set}$$
where $\tilde {G}:\cB\rightarrow [\cA^{op},{\bf Set}]$
whose values on objects are functors $$\tilde{G}(B)=\cB(G-,B):\cA^{op}\rightarrow {\bf Set}$$ and whose value on a morphism $B\stackrel {f}\rightarrow B'$ is a natural transformation $$\cB(G-,f):\cB(G-,B)\rightarrow\cB(G-,B').$$
\end{definition}
Therefore, we have that
\begin{equation}\label{ADJOINT COND ON INDEX COLIMIT}
\cB(F\star G,B)\cong [\cA^{op},{\bf Set}](F,\cB(G-,B))
\end{equation}
and after evaluating the representation isomorphism on the identity with $B=F\star G$ we obtain a unit  $\lambda:F\rightarrow\cB(G-,F\star G)$.
\begin{remark}
With enough conditions, for example when $\cB$ in cocomplete, there is a functor $\bullet\star G:[\cA^{op},{\bf Set}]\rightarrow \cB$. Also, from equation~(\ref{ADJOINT COND ON INDEX COLIMIT})
we conclude that $\bullet\star G$ is left adjoint of the functor $\cB(G-,\bullet):\cB\rightarrow [\cA^{op},{\bf Set}]$ where $\cB(G-,\bullet)(B)=\cB(G-,B):\cA^{op}\rightarrow{\bf Set}$. We write $\bullet\star G\dashv\cB(G-,\bullet)$.

The functor $\bullet\star G$ is the unique, up to isomorphism, colimit preserving functor such that the following diagram commutes:
$$\xymatrix@=15pt{
\cA\ar[rrr]^{Y}\ar[rrrd]_{G} && &[\cA^{op},{\bf Set}]\ar[d]^{\bullet\star G}\\
  &&& \cB
}$$
In the next section we shall discuss this construction in more detail in the context of a coproduct preserving Yoneda embedding.
\end{remark}
\begin{proposition}
If $F:\cA^{op}\rightarrow {\bf Set}$ and $G:\cA\rightarrow \cB$ then
$$\mbox{colim}\,G\comp\pi_{F}^{op}\cong F\star G$$
\end{proposition}
\begin{proof}
Analogously, there is a bijection:
$$[\cA^{op},{\bf Set}](F,\cB(G-,B))\cong [El(F)^{op},\cB](G\comp \pi_{F}^{op},\Delta B)$$
natural in $B$, with projection $\pi_{F}^{op}:El(F)^{op}\rightarrow \cA$.\\
From this since by equation~(\ref{REPRESCOLIMIT}):
$$\cB(\mbox{colim}\,G\comp \pi_{F}^{op},B)\cong [El(F)^{op},\cB](G\comp \pi_{F}^{op},\Delta B)$$
we conclude that:
$$\mbox{colim}\,G\comp\pi_{F}^{op}\cong F\star G.$$
\end{proof}
\begin{remark}
Since all colimits may be expressed in terms of coproducts and coequalizers we have the following explicit formula:
$$\xymatrix{
\coprod_{x\in F(A), f:A'\rightarrow A}G(A')\ar@<1ex>[rr]^{\theta}\ar@<-1ex>[rr]_{\tau}&&\coprod_{A,x\in F(A)}G(A)\ar[rr]^{\lambda}&&F\star G}
$$
where $\lambda$ is a coequalizer of the unique maps $\tau$ and $\theta$:
$$\begin{array}{cc}
\xymatrix@=14pt{
G(A')\ar[r]^{id}\ar[d]_{i_{(x,f)}}  & G(A')\ar[d]^{i_{(A',F(f)(x))}}\\
\coprod_{x\in F(A),A'\stackrel{f}\rightarrow A}G(A')\ar[r]_<>(.5){\theta}& \coprod_{A,x\in F(A)}G(A)
} \hspace{.0cm}&
\xymatrix@=14pt{
G(A') \ar[r]^{G(f)}\ar[d]_{i_{(x,f)}}& G(A) \ar[d]^{i_{(A,x)}}\\
\coprod_{x\in F(A), A'\stackrel{f}\rightarrow A}G(A')\ar[r]_<>(.5){\tau}& \coprod_{A,x\in F(A)}G(A)
}\end{array}$$
obtained by the coproduct definition.
\end{remark}
Now, suppose we take $F=\cA(-,A):\cA^{op}\rightarrow{\bf Set}$, then for every $B$ we have that:
$$\cB(\cA(-,A)\star G,B)\cong [\cA^{op},{\bf Set}](\cA(-,A),\cB(G-,B))=\cB(G(A),B)$$
by the Yoneda Lemma. Therefore $\cA(-,A)\star G\cong G(A)$. In the same way we obtain that $\{\cA(A,-),G\}\cong G(A)$.

\begin{proposition}
$$\int^{A}F(A)\otimes G(A)\cong F\star G$$
\end{proposition}
\begin{proof}
Let $F:\cA\rightarrow{\bf Set}$ and $G:\cA\rightarrow\cB$ be functors and suppose now that the category $\cB$ has copowers\footnote{If $X$ is a set and $B$ an object, the {\em
copower} $X\times B$ is defined to be a coproduct of $X$ copies of
$B$, i.e., $\coprod_{x\in X}B$.}. We denote by $X\otimes A=\coprod_{X} A\in |\cB|$ where $X$ is a set and $A\in |\cB|$.   Then we have
\begin{center}
$\cB(\int^{A} F(A)\otimes G(A),B)\cong \int_{A}\cB(F(A)\otimes G(A),B)\cong \int_{A}[F(A),\cB(G(A),B)]\cong
[\cA^{op},{\bf Set}](F,\cB(G(-),B))$
\end{center}
by properties of ends, copowers, hom as end in the functor category.

Thus, by definition this implies that
$$\int^{A}F(A)\otimes G(A)\cong F\star G.$$
\end{proof}
In particular when $G=Y:\cA\rightarrow [\cA^{op},{\bf Set}]$ we have that:
$$\int^{A}F(A)\otimes \cA(-,A)\cong F\star Y\cong F$$ as we already have proved (Example~\ref{EXAMPLE COEND AS DINATURAL MAP}).

\section{Idempotent adjunctions}
\begin{proposition}\label{F FULL FAITHFUL IFF ETA IS AN ISO}
Let
$\xymatrix{
\mathcal{A}\ar@<1ex>[r]^{F}& \mathcal{B} \ar@<1ex>[l]^{G}_{\bot}}$
be an adjunction with unit $\eta:1_{\cA}\Rightarrow GF$ and counit $\varepsilon:FG\Rightarrow 1_{\cB}$. Then
(i) $F$ is full and faithful if and only if (ii) $\eta$ is an isomorphism. When these conditions are satisfied,\, $\varepsilon\ast G$ and $F\ast\varepsilon$ are isomorphisms. \\
Dually, $G$ is full and faithful iff and only if $\varepsilon$ is an isomorphism. When this happens $\eta\ast G$ and $F\ast \eta$ are isomorphism as well.
\end{proposition}
\begin{proof}
(i)$\Rightarrow$(ii):
We have that $\phi:\cB(FA,B)\rightarrow\cA(A,GB)$, where $\phi^{-1}(g)=\varepsilon_{B}\circ F(g)$.
Since $F$ is full there is an $f$ such that $F(f)=\varepsilon_{FA}$. Hence since F is faithful, $F(f\circ\eta_{A})=F(f)F(\eta_{A})=\varepsilon_{FA}F(\eta_{A})=1_{FA}=F(1_{A})$ implies $f\circ\eta_{A}=1_{A}$ has a left inverse.\\
Therefore we have:
$\phi^{-1}(\eta_{A}\circ f)=\varepsilon_{FB}\circ F(\eta_{A}\circ f)=\varepsilon_{FB}\circ F(\eta_{A})\circ F(f)=\varepsilon_{FB}\circ F(\eta_{A})\circ\varepsilon_{FB}=1_{FA}\circ\varepsilon_{FB}=\varepsilon_{FB}\circ F(1_{GFA})=\phi^{-1}(1_{GFA})$. This implies that $\eta_{A}\circ f=1_{GFA}$ is also a right inverse.\\
(ii)$\Rightarrow$ (i): Consider the following isomorphism

$$\cA(A,A')\stackrel {\cA(A,\eta_{A'})}\longrightarrow \cA(A,GFA')\stackrel {\phi^{-1}}\longrightarrow \cB(FA,FA').$$

When we evaluate at $g:A\rightarrow A'$ we obtain that:

$$\phi^{-1}(\cA(A,\eta_{A'})(g))=\phi^{-1}(\eta_{A'}\circ g))=\varepsilon_{FA'}\circ F(\eta_{A}\circ g)=\varepsilon_{FA'}\circ F(\eta_{A})\circ F(g)=F(g)$$

by definition of adjunction. Thus $\phi^{-1}\circ\cA(A,\eta_{A'})=F$, is an isomorphism.

\end{proof}

\section{Lambek's completion for small categories}\label{LAMBEK COMPLETION}
In this section, we review some material from~\cite{Lambek66} relevant to
the following question: how to embed a small category as a full subcategory of a complete and cocomplete category in which the embedding preserves existing limits and colimits.
\begin{definition}
\rm
Let $G:\cA\rightarrow \cB$ be a functor, $\cA$ a small category.  Recall the functor $\tilde{G}$ defined in Definition~\ref{isbell functor} by $\tilde{G}(B)=\cB(G(-),B)$ on objects and $\tilde{G}(f)=\cB(G(-),f)$ on arrows. We say that $G$ is \textit{left adequate} for the category $\cB$ if the functor $\tilde{G}:\cB\rightarrow [\cA^{op},{\bf Set}]$ is fully faithful.
\end{definition}
\begin{proposition}
\label{PROPO4}
Suppose we have a functor $G:\cA\rightarrow \cB$, $\cA$ a small category, $\cB$ a co-complete category. If $G$ is a left adequate functor then for every $B\in\cB$ there exists a small category $\cI$ and a functor $H:\cI\rightarrow \cA$ such that $\mbox{colim}\,GH=B$.
\end{proposition}
\begin{proof}
For every $B\in\cB$ let us consider $F=\cB(G(-),B):\cA^{op}\rightarrow{\bf Set}$. Also consider the category $El(F)^{op}$ of elements of $F$, defined in Definition~\ref{CATEGORY OF ELEMENTS}. We claim that $H=\pi^{op}:El(F)^{op}\rightarrow \cA$, i.e.,
$$\mbox{colim}\,(El(F)^{op}\stackrel {\pi^{op}}\rightarrow \cA\stackrel{G}\rightarrow\cB)\cong B.$$
If $(A',x')\stackrel {f^{op}}\rightarrow(A,x)$ then $(A,x)\stackrel {f}\rightarrow(A',x')$ with

$$\xymatrix@=25pt{
G(A') \ar[r]^{x'}\ar[d]_{G(f)} & B\\
G(A) \ar[ru]^{x} &
}$$
since $x'=F(f^{op})(x)$.\\
We define the following set of arrows $G\pi^{op}(A,x)\stackrel{\beta_{(A,x)}}\rightarrow B$ with $\beta_{(A,x)}=x$.
Naturality follows from the previous diagram:
$$\xymatrix@=25pt{
G\pi^{op}(A,x) \ar[r]^{\beta_{(A,x)}}\ar[d]_{G\pi^{op}(f^{op})} & \Delta B(A',x')\ar[d]^{\Delta B(f^{op})} \\
G\pi^{op}(A,x) \ar[r]^{\beta_{(A,x)}} & \Delta B(A,x)
}$$
for every $(A',x')\stackrel {f^{op}}\rightarrow(A,x)$.
Now since $\cB$ is co-complete we have that there exists a co-cone $(C,u_{(A,x)}:G\pi^{op}(A,x)\rightarrow C)$ such that $\mbox{colim}\,G\pi^{op}=C$.
This implies, by definition of colimit, that there exists a unique $p:C\rightarrow B$ such that the following diagram commutes:
$$\xymatrix@=25pt{
G\pi^{op}(A,x) \ar[rr]^{u_{(A,x)}}\ar[rrd]_{\beta_{(A,x)}}&& C\ar[d]^{p}\\
 & & B
}$$
Actually $p$ is an epimorphism. If $fp=gp$ with $f:B\rightarrow B'$ and $g:B\rightarrow B'$ then we have that $fpu_{(A,x)}=gpu_{(A,x)}$ for every $g(A)\stackrel {x}\rightarrow B$. This implies $fx=f\beta_{(A,x)}=g\beta_{(A,x)}=gx$ for every $g(A)\stackrel {x}\rightarrow B$. Now we use the fact that by hypothesis $\tilde{G}$ is faithful. By definition we have $\tilde{G}(f)=\tilde{G}(g):\cB(G-,B)\rightarrow\cB(G-,B')$ since $\tilde{G}(f)(A)(x)=fx=gx=\tilde{G}(g)(A)(x)$, which implies $f=g$.

Now we define $\alpha_{A}:\cB(G(A),B)\rightarrow\cB(G(A),C)$ with $\alpha_{A}(x)=u_{(A,x)}$ for every $A\in \cA$ and $g(A)\stackrel {x}\rightarrow B$. We check that $\alpha$ is a natural transformation:
$$\xymatrix@=25pt{
\cB(G(A'),B) \ar[r]^{\alpha_{A'}}\ar[d]_{\cB(G(f),B)}& \cB(G(A'),C)\ar[d]^{\cB(G(f),C)}\\
\cB(G(A),B)\ar[r]^{\alpha_{A}} & \cB(G(A),C)
}$$
for every $A\stackrel {f}\rightarrow A'$.
$$\cB(G(f),C)(\alpha_{A'}(x'))=\cB(G(f),C)(u_{(A',x')})=u_{(A',x')}G(f)= \qquad (*)$$
$$u_{(A,x)}=\alpha_{A}(x'G(f))=\alpha_{A}(\cB(G(f),C)(x')).$$
This  equality $(*)$ is justified because $u$ is a co-cone, i.e., for every $(A,x)\stackrel {f^{op}}\rightarrow(A',x')$
$$\xymatrix@=25pt{
G\pi^{op}(A,x) \ar[rd]_{u_{(A,x)}}\ar[rr]^{G\pi^{op}(f^{op})}&& G\pi^{op}(A',x')\ar[ld]^{u_{(A',x')}} \\
 & C &
}$$
since we have that
$$\xymatrix@=25pt{
G(A) \ar[rd]_{u_{(A,x)}}\ar[rr]^{G(f)}&& G(A')\ar[ld]^{u_{(A',x')}} \\
 & C &
}$$
The rest of the proof follows now from the fact that $\tilde{G}$ is a full functor. Hence there exists a morphism $b:B\rightarrow C$ such that $\alpha=\cB(G-,b):\cB(G-,B)\rightarrow\cB(G-,C)$. Therefore using this representation we get that $u_{(A,x)}=\alpha_{A}(x)=\cB(G(A),b)(x)=bx$ for every $(A,x)\in \mbox{El}(F)^{op}$. Thus by definition of colimits we get that $bpu_{(A,x)}=bx=u_{(A,x)}$ for every $(A,x)\in \mbox{El}(F)^{op}$ implies that $bp=1_{C}$. But $p$ is an epimorphism, so we cancel
to obtain $pbp=p1=p$ and thus $pb=1_{B}$, which means it is an isomorphism. Therefore $\mbox{colim}\,G\pi^{op}=(C,u_{(A,x)})_{(A,x)\in El(F)}\cong(B,\beta_{(A,x)})_{(A,x)\in El(F)}$.
\end{proof}
\begin{corollary}
\label{CORO1}
For every $F\in [\cA^{op},{\bf Set}]$
$$F=\mbox{colim}\,(El(F)^{op}\stackrel {\pi^{op}}\rightarrow\cA\stackrel {Y}\rightarrow [\cA^{op},{\bf Set}]).$$
\end{corollary}
\begin{proof}
The Yoneda functor $A\stackrel {Y}\rightarrow [\cA^{op},{\bf Set}]$ is left adequate since we have that:\\
$\tilde{Y}:[\cA^{op},{\bf Set}]\rightarrow [\cA^{op},{\bf Set}]$ is defined  $\tilde{Y}(F)=[\cA^{op},{\bf Set}](Y-,F)= F$ on objects and $\tilde{Y}(\alpha)=[\cA^{op},{\bf Set}](Y-,F)\rightarrow [\cA^{op},{\bf Set}](Y-,F')=Id(\alpha)$ on arrows by the Yoneda Lemma.
\end{proof}

\begin{definition}
\rm A functor $F:\cA\rightarrow \cB$ \textit{reflects limits} when for each functor $G:\cI\rightarrow\cA$ with $\cI$ small and given a cone $(A,u_i)_{i\in\cI}$, $u_i:A\rightarrow G(i)$, if $(F(A),F(u_i))_{i\in\cI}$ is a limit of $FG$ then $(A,u_i)_{i\in\cI}$ is a limit of $G$.
\end{definition}

\begin{proposition}
\label{PROPO1}
Let $F:\mathcal{A}\rightarrow \mathcal{B}$ be a functor.
$F$ preserves colimits if and only if $\mathcal{B}(F-,B):\mathcal{A}^{op}\rightarrow {\bf Set}$ preserves limits for every $B\in \mathcal{B}$.
\end{proposition}
\begin{proof}
($\Rightarrow$)
Let us first observe that we have a composition of functors $\mathcal{B}(F-,B)=\mathcal{B}(-,B)\circ F^{op}$  where $ F^{op}:\mathcal{A}^{op}\rightarrow \mathcal{B}^{op}$ preserves limits since $F$ preserves colimits and $\mathcal{B}(-,B):\mathcal{B}^{op}\rightarrow {\bf Set}$ preserves limits~\cite{Borceux94}.\\
($\Leftarrow$)
Now consider the functor $G:\cI\rightarrow\cA$ with $\mbox{colim}\, G=(A,u_i)_{i\in \cI}$, $u_i:G(i)\rightarrow A$. Thus
 $\mbox{lim}\, G^{op}=(A,u_i^{op})_{i\in \cI^{op}}$ where $G^{op}:\cI^{op}\rightarrow\cA^{op}$. By hypothesis we know that $\mathcal{B}(F-,B):\mathcal{A}^{op}\rightarrow {\bf Set}$ preserves limits, hence for every $B\in\cB$ the limit takes the form $\mbox{lim}\,\mathcal{B}(F-,B)\circ G^{op}=(\mathcal{B}(F(A),B),\mathcal{B}(F(u_i^{op}),B))_{i\in \cI^{op}}$, so we have:
\begin{center}
 $\cI^{op}\stackrel {G^{op}}\rightarrow \cA^{op}\stackrel {F^{op}}\rightarrow\cB^{op}\stackrel {Y}\rightarrow[\cB,{\bf Set}]$\\
 $i\mapsto G(i)\mapsto F(G(i))\mapsto \mathcal{B}(F(G(i)),-)$,
\end{center}
where $Y(B')=\cB(B',-):\cB\rightarrow{\bf Set}$.\\
Therefore for any $B\in\cB$ it may be verified that $Y\circ F^{op}\circ G^{op}(-)(B):\cI^{op}\rightarrow{\bf Set}$ has a limit by hypothesis, since $\forall B\in\cB$:
\begin{center}
$\mbox{lim}\,Y\circ F^{op}\circ G^{op}(-)(B)=(\mathcal{B}(F(A),B),\mathcal{B}(F(u_i^{op}),B))_{i\in \cI^{op}}$.
\end{center}
Then, by proposition 2.15.1 of~\cite{Borceux94} we have $Y\circ F^{op}\circ G^{op}:\cI^{op}\rightarrow[\cB,{\bf Set}]$ has a limit being compute pointwise. Which means we have:
\begin{center}
$\mbox{lim}\,Y\circ F^{op}\circ G^{op}=(\mathcal{B}(F(A),-),\mathcal{B}(F(u_i^{op}),-))_{i\in \cI^{op}}$.
\end{center}
But $Y$ is a full and faithful functor, it reflects limits (see proposition 2.9.9~\cite{Borceux94}) which implies that (see definition 2.9.6~\cite{Borceux94}) since $(Y(F(A)),Y((F(u_i^{op}))_{i\in \cI^{op}}$ is the limit of $Y\circ F^{op}\circ G^{op}$ then $(F(A),F(u_i^{op}))_{i\in \cI^{op}}$ is the limit of $F^{op}\circ G^{op}$ in $\cB^{op}$. Equivalently, in view of this we are saying that $(F(A),F(u_i))_{i\in \cI}$ is the colimit of $F\circ G$ in $\cB$.
Summarizing, we started with $\mbox{colim}\, G=(A,u_i)_{i\in \cI}$ and we end with $\mbox{colim}\, FG=(FA,Fu_i)_{i\in \cI}$, i.e., $F$ preserves colimits.
 \end{proof}
\begin{proposition}
\label{COPRODUCT CONT EQUIVALENCE}
Let $F:\mathcal{A}\rightarrow \mathcal{B}$ be a functor.
$F$ preserves coproducts if and only if $\mathcal{B}(F-,B):\mathcal{A}^{op}\rightarrow {\bf Set}$ preserves products for every $B\in \mathcal{B}$.
\end{proposition}

\begin{proposition}
Let $F:\mathcal{A}\rightarrow \mathcal{B}$ be a functor.
$F$ preserves limits if and only if $\mathcal{B}(B,F-):\mathcal{A}\rightarrow {\bf Set}$ preserves limits for every $B\in \mathcal{B}$.\\
\end{proposition}

Let $F:\mathcal{A}\rightarrow \mathcal{C}$ be a fully faithful functor. Consider the full subcategories $\cB$ of $\cC$ such that $|F(\cA)|\subseteq |\cB|\subseteq |\cC|$ and define:
$$\cA\stackrel {F_{\cB}}\rightarrow\cB\stackrel {j_{\cB}}\rightarrow\cC$$
with $F=j_{\cB}F_{\cB}$, $F_{\cB}(A)=F(A)$, $F_{\cB}(f)=F(f)$ and $j$ the inclusion functor.
Define $\cB_0$ a full subcategory of $\cC$ in the following way:
 $$|\cB_0|=\{B\in |\cC|:\cC(F(-),B):\cA^{op}\rightarrow {\bf Set}\,\mbox{preserves limits}\}.$$
\begin{remark}
\rm
If $F:\mathcal{A}\rightarrow \mathcal{C}$ is a fully faithful functor then $|F(\cA)|\subseteq |\cB_0|$. To see this we have that  $\mathcal{C}(F-,F(A))\cong \mathcal{A}(-,A)$ are naturally isomorphic which implies that  $\mathcal{C}(F-,F(A))$ preserves limits.
\end{remark}

\begin{proposition}
\label{PROPO2}
Let $F:\mathcal{A}\rightarrow \mathcal{C}$ be a fully faithful functor. Then:\\
(a) if $j_{\cB}F_{\cB}$ preserves colimits then $\cB\subseteq \cB_0$\\
(b) let $\cJ$ be a small category, and consider the following composition of functors:
 $$\cJ\stackrel {\Delta}\rightarrow\cB_{0}\stackrel {j_{\cB_0}}\rightarrow\cC$$
 if $\mbox{lim}\, j_{\cB_0}\Delta=(C, v_j)$ then $C\in |\cB_0|$.
\end{proposition}

\begin{proof}
(a) Take $B\in |\cB|$, since $j_{\cB}F_{\cB}$ preserves colimits then by Proposition~\ref{PROPO1} $\mathcal{C}(j_{\cB}(F_{\cB}(-)),B)=\mathcal{C}(F(-),B)$ preserves limits, which by definition means that $B\in |\cB_0|$.\\
(b) We are going to prove that $|\cB'|=|\cB_0|\cup\{C\}$ also satisfies property of part (a) above. This implies that
$\cB'\subseteq\cB_0$ i.e., $C\in |\cB_0|$.
We have that
$$\cA\stackrel {F_{\cB'}}\rightarrow\cB'\stackrel {j_{\cB'}}\rightarrow\cC$$
and we want to show that if $\cI\stackrel {\Gamma}\rightarrow\cA$ with $\mbox{colim}\,\Gamma =(A,u)$, with $\Gamma(i)\stackrel{u_i}\rightarrow A$, $i\in\cI$ then
$\mbox{colim}\,F_{\cB'}\Gamma =(F_{\cB'}(A),F_{\cB'}(u_i))_{i\in \cI}=(F(A),F(u_i))_{i\in \cI}$.\\
Let $t:F_{\cB'}\Gamma \Rightarrow C$ be a co-cone.
Without loss of generality, we assume that $F_{\cB'}(\Gamma(i))\neq C$ for every $i\in \cI$. If there exists a $i_0$ with $F_{\cB'}(\Gamma(i_0))= C$ then since $|F(\cA)|\subseteq |\cB|$ this implies that $C\in |\cB|$. \\
We fix an object $j\in |\cJ|$. Therefore since $t$ is a co-cone we consider the following co-cone:
$$F(\Gamma(i)) \stackrel {t_{i}}\rightarrow C\stackrel {v_{j}}\rightarrow\Delta(j).$$
These arrows are contained in the category $\cB$ because $F(\Gamma(i))$ and $\Delta(j)$ are object of $\cB$.
We know by part (a) that $F_{\cB_0}$ has the property of preserving colimits:
$$\mbox{colim}\,F_{\cB_0}\Gamma =(F_{\cB_0}(A),F_{\cB_0}(u_i))_{i\in \cI}=(F(A),F(u_i))_{i\in \cI}.$$
For that reason there exists a unique $x_j:F(A)\rightarrow \Delta(j)$ such that

$$\xymatrix@=25pt{
F(\Gamma(i)) \ar[r]^{t_i}\ar[d]_{F(u_i)} & C\ar[d]^{v_j}\\
F(A) \ar[r]^{x_j} & \Delta(j)
}$$
for every $i\in |\cI|$.
We will show that $x_j$ is a cone in order to use the universal property of the limit.
Let $f:j\rightarrow j'$ be an arrow in $\cJ$. We want to prove that $\Delta(f)x_j=x_j'$.
This follows from the fact that $x_j$ is defined using $\mbox{colim}\,F_{\cB_0}\Gamma$.
We must check that $v_j't_i=\Delta(f)x_jF(u_i)$ for every $i\in |\cI|$. Then by uniqueness of the colimit definition we get that $\Delta(f)x_j=x_j'$.\\
But we know by definition of $x_j$ that: $x_jF(u_i)=v_jt_i$ for every $i\in |\cI|$, then composing with $\Delta(f)$ we obtain $\Delta(f)x_jF(u_i)=\Delta(f)v_jt_i$ for every $i\in |\cI|$. Therefore, it will be enough to prove that $\Delta(f)v_j=v'_j$, but this follows from the naturality of the cone $C\Rightarrow \Delta$. \\
 We have proved that $F(A)\stackrel {x}\Rightarrow \Delta$ is a cone in $\cB_0$. Then by definition of $\mbox{lim}\Delta=(C,v)$ there exists a unique $y:F(A)\rightarrow C$ such that:

$$\xymatrix@=25pt{
F(A)\ar[r]^{y}\ar[rd]_{x_j} & C\ar[d]^{v_j}\\
  & \Delta(j)
}$$
We therefore put all the equations together: $v_jt_i=x_jF(u_i)=v_jyF(u_i)$ for every $j\in |\cJ|$. Thus since this is true for every $j\in |\cJ|$, by definition of limit we have that $t_i=yF(u_i)$.\\
So now suppose there exists another $y'$ satisfying the same property as above: $t_i=y'F(u_i)$. We want to prove that $y=y'$. It will be enough to prove that: $v(j)y'=x_j$ for every $j\in |\cJ|$. For that purpose, we know by hypothesis that $t_i=y'F(u_i)$ for every $i\in\cI$. Then by composing we get $v(j)y'F(u_i)=v(j)t_i$ for every $i\in\cI$, and since $v_jt_i=x_jF(u_i)$ we replace it: $v_jy'F(u_i)=x_jF(u_i)$ for every $i\in\cI$. This implies by uniqueness of the colimit that $v_jy'=x_j$.\\
We proved that $\mbox{colim}\,F_{\cB'}\Gamma=(F_{\cB'}(A),F_{\cB'}(u_i))_{i\in \cI}$ where $|\cB'|=|\cB_0|\cup\{C\}$, i.e., for an arbitrary co-cone in $\cB'$, $(F(A),F(u_i))_{i\in \cI}$ is still a limit co-cone and this implication is the the property that characterizes the set $|\cB_0|$.
\end{proof}
\begin{corollary}\label{pipoqui}
Let $F:\cA\rightarrow\cC$ be a fully faithful functor such that for every $C\in\cC$ there exists a functor $G:\cI\rightarrow \cA$ with $\mbox{lim}\,FG=C$. Then $F$ preserves colimits.
\end{corollary}
\begin{proof}
We consider $\cB_0$ as above. Since $\mbox{lim}\,FG=C$ for some $G$, then by part (b) of the Proposition~\ref{PROPO2}  above we have that $C\in\cB_0$, therefore $F=F_{\cB_0}$ and it preserves colimits by Proposition~\ref{PROPO1}.
\end{proof}

\begin{remark}
\rm
To prove that $Y:\cA\rightarrow [\cA^{op},{\bf Set}]$ preserves limits is equivalent to proving that $Y^{op}:\cA^{op}\rightarrow [\cA^{op},{\bf Set}]^{op}$ preserves colimits and since by Corollary~\ref{CORO1}:
$$F=\mbox{colim}\,(El(F)^{op}\stackrel {\pi^{op}}\rightarrow\cA\stackrel {Y}\rightarrow [\cA^{op},{\bf Set}])$$
for every $F\in [\cA^{op},{\bf Set}]$ this implies that:
$$F=\mbox{lim}\,(El(F)\stackrel {\pi}\rightarrow\cA^{op}\stackrel {Y^{op}}\rightarrow [\cA^{op},{\bf Set}]^{op})$$
for every $F\in [\cA^{op},{\bf Set}]^{op}$. But we know by the Corollary~\ref{pipoqui} above that this implies that $\cA^{op}\stackrel {Y^{op}}\rightarrow [\cA^{op},{\bf Set}]^{op} $ preserves colimits.
\end{remark}
\begin{definition}\label{Y INF DEFINITION}
Let $[\cA,{\bf Set}]_{inf}\subseteq [\cA,{\bf Set}]$ be the full subcategory of limit preserving functors. Since the representable functors $\cA(-,A):\cA^{op}\rightarrow {\bf Set}$ preserve limits, we can define a functor
$\cA\stackrel {Y_{inf}}\rightarrow [\cA^{op},{\bf Set}]_{inf}$ by co-restriction induced by the Yoneda embedding.
\end{definition}
\begin{remark}
\label{OBS2}
\rm Let $\cA$ be a small category. The functor $\cA\stackrel {Y_{inf}}\rightarrow [\cA^{op},{\bf Set}]_{inf}$ is left adequate since the induced functor $[\cA,{\bf Set}]_{inf}\stackrel {\tilde{Y}_{inf}}\rightarrow [\cA^{op},{\bf Set}]$ is fully faithful.
To see this, we check that we have on objects: $$\tilde{Y}_{inf}(F)=[\cA^{op},{\bf Set}]_{inf}(Y_{inf}-,F)=[\cA^{op},{\bf Set}](Y_-,F)\cong F$$
since is a full subcategory and $Y_{inf}-=Y-\in[\cA^{op},{\bf Set}]_{inf}$. Thus we have that
$$[\cA^{op},{\bf Set}]_{inf}(F,G)=[\cA^{op},{\bf Set}](F,G)\cong[\cA^{op},{\bf Set}](\tilde{Y}_{inf}(F),\tilde{Y}_{inf}(G))$$
which means that $\tilde{Y}_{inf}$ is fully faithful, i.e., $Y_{inf}$ left adequate. Therefore,
using the same argument
we get that $\cA\stackrel {Y_{inf}}\rightarrow [\cA^{op},{\bf Set}]_{inf}$ preserves limits.
\end{remark}

\begin{proposition}
Let $\cB$ be a full subcategory of $\cC$ such that for every $C\in \cC$ there exists functor $G:\cI\rightarrow\cB$ with $\mbox{colim}\,j_{\cB}G=C$. If $\cB$ is a co-complete category then $\cB$ is a left reflective subcategory of $\cC$.
Conversely, suppose $\cB$ is a left reflective subcategory of $\cC$. If $\cC$ is co-complete then $\cB$ is co-complete.
\end{proposition}
\begin{proof}
We want to prove that the inclusion functor $\cB\stackrel {j}\hookrightarrow\cC$ has a left adjoint $\cC\stackrel {R}\rightarrow\cB$. It is enough to prove that for every $C\in\cC$ there is an object $R(C)\in \cB$, a map $C\stackrel {\eta_{C}}\rightarrow R(C)$ such that for every $f:C\rightarrow B'$ with $B'\in\cB$ there is a unique $g:R(C)\rightarrow B'$ such that the following diagram commutes:

$$\xymatrix@=25pt{
C \ar[rr]^{\eta_{C}}\ar[drr]_{f}& & R(C)\ar[d]^{g}\\
 & & B'.
}$$
Let us consider $C\in\cC$. By hypothesis we have that there exists a functor $G:\cI\rightarrow\cB$ with $\mbox{colim}\,j_{\cB}G=C$. But since $\cB$ is a co-complete category then there is an object $B\in \cB$ and a co-cone $\{G(i)\stackrel {u_i}\rightarrow B\}_{i\in I}$ with $\mbox{colim}\,G=(B,u)$.\\
We define $R(C)=B$, and since $\{jG(i)=G(i)\stackrel {u_i}\rightarrow B\}_{i\in I}$ is a co-cone of $jG$ in the category $\cC$ therefore there exists a unique $C\stackrel {\eta_{C}}\rightarrow R(C)$, such that:
$$\xymatrix@=25pt{
G(i) \ar[rr]^{v_i}\ar[drr]_{u_i}& & C\ar[d]^{\eta_C}\\
 & & B=R(C)
}$$
commutes for every $i\in \cI$.
Now suppose we have a map $f:C\rightarrow B'$ with $B'\in\cB$. Then since $G(i)$ is an object of $\cB$ for every $i\in \cI$ and $\{v_i\}_{i\in \cI}$ is a co-cone in $\cC$ this implies that $\{jG(i)=G(i)\stackrel {v_i}\rightarrow C\stackrel {f}\rightarrow B'\}_{i\in I}$ is a co-cone in the category $\cB$. Therefore by definition of $\mbox{colim}\,G=(B,u)$ there is a unique $g:R(C)\rightarrow B'$, $g\in \cB$ with $fv_i=gu_i$ for every $i \in \cI$.
Hence $fv_i=gu_i=g\eta_{C}v_i$ for every $i \in \cI$, and this implies by definition (uniqueness) of colimit that $f=g\eta_{C}$.\\
If there is a morphism $\tilde{g}:R(C)\rightarrow B'$ such that $f=\tilde{g}\eta_{C}$ then by composing with $v_i$ we get $fv_i=g\eta_{C}v_i$ for every $i \in \cI$ which means that $fv_i=\tilde{g}u_i$ for every $i \in \cI$ therefore $g=\tilde{g}$.
If $C\stackrel {f}\rightarrow C'$ a morphism in $\cC$ then $R(f)$ is defined as the unique arrow such that:
$$\xymatrix@=25pt{
C\ar[d]^{f} \ar[rr]^{\eta_C}& & R(C)\ar[d]^{R(f)}\\
C'\ar[rr]^{\eta_{C'}}& & R(C')
}$$
commutes. By uniqueness we obtain that $R$ is a functor and naturality of $Id\stackrel {\eta}\Rightarrow j\circ R$ follows from the diagram.\\
Conversely, let $G:\cI\rightarrow\cB$ be a functor. Since $\cC$ is co-complete there exists $\mbox{colim}\,jG=(C,v)$ with $j:\cB\rightarrow \cC$ the inclusion functor and $G(i)\stackrel {v_i}\rightarrow C$. By hypothesis we know that $R$ is a reflection of $j$, which means $\cB(R(A),B)\cong\cC(A,j(B))$ for every $A\in\cC$, and $B\in\cB$. When $A\in\cB$ then since $\cB$ is a full subcategory we have that $\cB(R(A),B)\cong\cB(A,B)$ for every $B\in\cB$. By the Yoneda Lemma this implies that $R(A)\cong A$. On the other hand $R$ preserves colimits because is a left adjoint. Thus $G(i)\cong R(G(i))=RjG(i)\stackrel {R(v_i)}\rightarrow R(C)$ is a colimit of $G$ with $R(C)\in\cB$.
\end{proof}
\begin{remark}
Notice that, from the proof above, colimits in $\cB$ are induced by the reflection, i.e., if $G:\cI\rightarrow \cB$ is a functor with $\cI$ small then:
$$\mbox{colim}_{\cB}\, G\cong R(\mbox{colim}_{\cC}\,j\circ G).$$
\end{remark}
\begin{remark}
\rm
For every $F\in [\cA^{op},{\bf Set}]$ there exists a functor $\cI\stackrel {G}\rightarrow [\cA^{op},{\bf Set}]_{inf}$ such that $\mbox{colim}\,jG=F$:
$$F=\mbox{colim}\,(El(F)^{op}\stackrel {\pi^{op}}\rightarrow\cA\stackrel {Y_{inf}}\rightarrow [\cA^{op},{\bf Set}]_{inf}\stackrel {j}\hookrightarrow[\cA^{op},{\bf Set}]).$$
\end{remark}
\begin{proposition}
\label{KENNISON THEOREM}\label{PROPO3}
Let $\cA$ be  a small category. Then $[\cA,{\bf Set}]_{inf}$ is a reflective subcategory of $[\cA,{\bf Set}]$.
\end{proposition}
\begin{proof}
\cite{KENNISON}.
\end{proof}
\begin{remark}
\label{OBS1}
\rm
This implies that $[\cA,{\bf Set}]_{inf}$  is a co-complete category.
\end{remark}
\begin{proposition}
Let $\cA\stackrel {Y_{inf}}\rightarrow [\cA^{op},{\bf Set}]_{inf}$ be the restricted Yoneda embedding from
Definition~\ref{Y INF DEFINITION} above. Then $Y_{inf}$ is a full and faithful, limit and colimit preserving functor such that for every $F\in[\cA^{op},{\bf Set}]_{inf}$ there exists a functor $G:\cI\rightarrow \cA$ with $\mbox{lim}\,Y_{inf}G=F$. Moreover, $[\cA^{op},{\bf Set}]_{inf}$ is a complete and co-complete category.
\end{proposition}
\begin{proof}
First, $[\cA^{op},{\bf Set}]_{inf}$ is a co-complete category by Remark~\ref{OBS1} above.
In view of the Remark~\ref{OBS2} above $\cA\stackrel {Y_{inf}}\rightarrow [\cA^{op},{\bf Set}]_{inf}$ preserves limits.\\
Using Proposition~\ref{PROPO1}:
\begin{center}
$\cA\stackrel {Y_{inf}}\rightarrow [\cA^{op},{\bf Set}]_{inf}$  preserves co-limits if and only if $[\cA^{op},{\bf Set}]_{inf}(Y_{inf}-,F):\cA^{op}\rightarrow {\bf Set}$ preserves limits for all $F\in[\cA^{op},{\bf Set}]_{inf}$.
\end{center}
But by the Yoneda Lemma we have that $$[\cA^{op},{\bf Set}]_{inf}(Y_{inf}-,F)=[\cA^{op},{\bf Set}](Y_-,F)\cong F$$
 which is the condition that defines the subcategory. Notice that we used the fact that $[\cA^{op},{\bf Set}]_{inf}$ is a full subcategory. \\
 Now, in view of Proposition~\ref{PROPO2}, consider the fully faithful functor $F:\cA\rightarrow \cC$, with $F=Y$, $\cB_0=[\cA^{op},{\bf Set}]_{inf}$ and $\cC=[\cA^{op},{\bf Set}]$. By part (b) when there is a functor
 $$\cJ\stackrel {\Delta}\rightarrow\cB_{0}\stackrel {j_{\cB_0}}\rightarrow\cC$$
 since $[\cA^{op},{\bf Set}]$ is a complete category then $\mbox{lim}\, j_{\cB_0}\Delta=(C, v_j)$ exists. But this implies that $C\in |\cB_0|$ which means that $\cB_0=[\cA^{op},{\bf Set}]_{inf}$ is complete.\\
To see why $\cB_0=[\cA^{op},{\bf Set}]_{inf}$, consider $\cB=[\cA^{op},{\bf Set}]_{inf}$ and $j_{\cB}F_{\cB}=Y$ with $F_{\cB}=Y_{inf}$. Since it preserves colimits then $\cB\subseteq \cB_0$. On the other hand if $B\in \cB_0$ such that $Y_{B_0}:\cA\rightarrow \cB_0$ preserves colimits then by Proposition~\ref{PROPO1}  this implies that:
$B_0(Y_{B_0}-,B):\cA^{op}\rightarrow{\bf Set}$ preserves limits. But
$$B_0(Y_{B_0}-,B)=[\cA^{op},{\bf Set}](Y-,B)\cong B.$$
Thus it means that $B$ preserves limits, i.e.,  $B\in[\cA^{op},{\bf Set}]_{inf}$.\\
It remains to show that if
$F\in[\cA^{op},{\bf Set}]_{inf}$ then there exists a functor $G:\cI\rightarrow \cA$ with $\mbox{lim}\,Y_{inf}G=F$.
For this, it is enough to prove that $Y_{inf}$ is left adequate, which was done on Remark~\ref{OBS2}.
\end{proof}
\begin{remark}
This amounts to proving that for every $F\in [\cA^{op},{\bf Set}]$ there is an object $R(F)\in[\cA^{op},{\bf Set}]_{inf}$, a co-cone $Y_{inf}\pi_F\stackrel {u}\Rightarrow \Delta R(F)$, and a co-cone $jY_{inf}\pi_F\stackrel {v}\Rightarrow \Delta F$ such that
$\mbox{colim}\,Y_{inf}\pi_F=(R(F),u)$ and $\mbox{colim}\,jY_{inf}\pi_F=(F,v)$.
Therefore there is a unique $F\stackrel {\eta_F}\rightarrow R(F)$ such that
$$\xymatrix@=25pt{
F \ar[rr]^{\eta_F}& & R(F)\\
 &jY_{inf}\pi_{F}(A,a)=Y_{inf}\pi_{F}(A,a)=\cA(-,A)\ar[ul]_{v_{(A,a)}}\ar[ur]^{u_{(A,a)}} &
}$$
commutes for every $i\in \cI$.
\end{remark}

To conclude this section, we briefly comment on the reflective adjoint pair $i\vdash R$ of Proposition~\ref{KENNISON THEOREM}.
Since $[\cA^{op},{\bf Set}]_{inf}$ is a co-complete category, all small colimits exists and we are in a position to consider co-powers $A\otimes_{inf}B$ where $A\in {\bf Set}$ and $B\in [\cA^{op},{\bf Set}]_{inf}$. On the other hand, co-powers in the category $[\cA^{op},{\bf Set}]_{inf}$ are induced by copowers in $[\cA^{op},{\bf Set}]$ using the reflection above:
$$A\otimes_{inf} B=R(A\otimes i(B)).$$

Therefore, since $R$ preserves coends we have that we can express $R(F)=F\star Y_{inf}$ as an indexed colimit where the definition of the operation $\star$, taken from~\cite{KELLY82} (see Definition~\ref{index colimits}) is given by the next first equation:
$$F\star Y_{inf}=\int^x F(x)\otimes_{int} Y_{inf}(x)=\int^x R(F(x)\otimes A(-,x))\cong R(\int^x F(x)\otimes A(-,x))\cong R(F)$$
Notice that we are using the fact that every representable functor is included in the category $[\cA^{op},{\bf Set}]_{inf}$.
Thus, in terms of left Kan extension (see Section~\ref{KAN EXTENSION}) or indexed colimits we have the following diagram:
$$\xymatrix@=25pt{
\cA\ar[rr]^{Y}\ar[rrd]_{Y_{inf}}& & [\cA^{op},{\bf Set}]\ar@<1ex>[d]^{\tilde{Y}_{inf}}\\
& &[\cA^{op},{\bf Set}]_{inf}\ar@<1ex>[u]^{\hat{Y}_{inf}}_{\vdash}
}$$
where $\tilde{Y}_{inf}=R=-\star {Y}_{inf}=Lan_{Y}(Y_{inf})$ since $$Lan_{Y}(Y_{inf})(F)=\int^{A}[\cA^{op},{\bf Set}](Y(A),F)\otimes_{inf}Y_{inf}(A)\cong\int^{A}F(A)\otimes_{inf}Y_{inf}(A)$$
and $\hat{Y}_{inf}\cong i$  the inclusion functor since

$$\hat{Y}_{inf}(F)=[\cA^{op},{\bf Set}](Y(-),F)\cong F.$$

\section{Kan extensions}
\label{KAN EXTENSIONS}\label{KAN EXTENSION}
This section provides a brief overview of the left Kan extension. A large portion of Chapter~\ref{PREASHEAVES MODELS OF Q C} depends on this central notion. To mention two examples: the definition of a left adjoint of a certain functor and the monoidal enrichment of the functor
category.
\begin{definition}
\rm
Let $F:\mathcal{A}\rightarrow \mathcal{B}$ and $G:\mathcal{A}\rightarrow \mathcal{C}$ be two functors. The \textit{left Kan extension} of the functor $G$ along $F$, if it exists, is a functor $K:\mathcal{B}\rightarrow \mathcal{C}$ together with a natural transformation $\alpha:G\Rightarrow KF$ satisfying the following universal property:
if $H:\mathcal{B}\rightarrow \mathcal{C}$ and $\beta:G\Rightarrow HF$ then there is a unique natural transformation $\gamma:K\Rightarrow H$ satisfying $(\gamma \ast F)\circ \alpha =\beta$.
\end{definition}

\texttt{Notation:} We denote the functor $K$ by $Lan_{F}(G)$.

Let $F:\cA\rightarrow \cB$, and consider the functor between functor
 categories
\begin{equation}
\label{RIGHT ADJOINT KAN EXTENSION}
[\mathcal{B},\mathcal{C}]\stackrel {F^{*}}\rightarrow  [\mathcal{A},\mathcal{C}]
\end{equation}
defined by precomposition with $F$, i.e.,
 $F^*(G) = G\circ F$ for any functor $G:\cB\to\cC$.

 \begin{corollary}
If $Lan_F(G)$ exists for all $G$, then $Lan_F\dashv F^*$.
\end{corollary}

 \begin{proof}
The definition above turns out to be the following:
for every $\beta:G\Rightarrow F^{*}(K)$ there exists a unique $\gamma:Lan_{F}(G)\rightarrow H$ such that:

$$\xymatrix@=25pt{
G \ar[rrd]^{\beta}\ar[d]_{\alpha}&\\
 F^{*}(Lan_{F}(G))\ar[rr]_{F^{*}(\gamma)}&& F^{*}(H)
}$$
which means that:

$$[\mathcal{B},\mathcal{C}](Lan_{F}(G),H]\cong [\mathcal{A},\mathcal{C}](G,F^{*}(H))$$
with unit $\alpha=\eta_{G}:G\Rightarrow F^{*}(Lan_{F}(G))$.\\
\end{proof}
\begin{proposition}
If $\mathcal{A}$ is a small category and $\mathcal{C}$ is co-complete then the left Kan extension of $G$ along F exists.
\end{proposition}
\begin{remark}
We can also formulate the left Kan extension as a coend.
If $\forall a,a'\in \mathcal{A}$ and $b\in \mathcal{B}$ the copowers $\mathcal{B}(F(a'),b)\times G(a)$ exist in $\mathcal{C}$; and the following coend exists $\forall b\in \mathcal{B}$ then:
$$Lan_{F}(G)(b)=\int^{a}\mathcal{B}(F(a),b)\times G(a).$$
\end{remark}
\texttt{Notation:} For the sake of brevity we sometimes write $Lan_{F}$ instead of $Lan_{F^{op}}$ when the extension is along the opposite functor $F^{op}:\cA^{op}\rightarrow \cB^{op}$.
\begin{remark}
Notice that for a functor $\Phi:\cA\rightarrow \cB$ we can express the adjunction $Lan_{\Phi}\dashv \Phi^{*}$ as a left Kan extension of $Y\circ\Phi:\cA\rightarrow[\cB^{op},{\bf Set}]$ along $Y:\cB\rightarrow[\cB^{op},{\bf Set}]$ in the following way: for some $F:\cA^{op}\rightarrow{\bf Set}$ we have
\begin{center}
$Lan_{Y}(Y\circ\Phi)(F)=(\int^{a}[\cA^{op}{\bf Set}](Y(a),-)\times Y\circ\Phi(a))(F)=\int^{a}[\cA^{op}{\bf Set}](Y(a),F)\times Y\circ\Phi(a)\cong\int^{a}F(a)\times \cB(-,\Phi(a))=Lan_{\Phi^{op}}(F)$
\end{center}
and also for some $G:\cB^{op}\rightarrow{\bf Set}$:
\begin{center}
$[\cB^{op},{\bf Set}](Y(\Phi(-)),G)\cong G(\Phi(-))=\Phi^{*}(G).$
\end{center}
\end{remark}

\section{Day's closed monoidal convolution}\label{DAY S CLOSED MONO CONVO}
A symmetric monoidal category can be fully and faithfully embedded in a symmetric monoidal closed category in such a way that the tensor is preserved.
This construction is a particular instance of a more general notion called promonoidal categories defined by Day~\cite{B.Day70}. In fact there is a correspondence between promonoidal categories and biclosed monoidal structures defined on the functor categories.

\begin{proposition}\label{monoidal convolution}
 Let $\cA$ be a symmetric monoidal category. Then $[\cA^{op},{\bf Set}]$ can be equipped with a symmetric monoidal structure (called the {\em Day tensor}~\cite{B.Day70}), such that the Yoneda embedding $Y:\cA\to [\cA^{op},{\bf Set}]$ is a strong monoidal functor. Moreover, $[\cA^{op},{\bf Set}]$ is monoidal closed.
\end{proposition}
\begin{proof}(sketch)

We consider the monoidal closed case on functor categories
$$([\mathcal{A}^{op},{\bf Set}],\otimes_D,I_D,\multimap).$$
This structure is obtained by using the Kan extension to closed functor categories:

$$\xymatrix@=25pt{
\cA\times\cA\ar[rrr]^{Y\times Y}\ar[d]_{\otimes}&&& [\mathcal{A}^{op},{\bf Set}]\times [\mathcal{A}^{op},{\bf Set}]\ar[d]^{\textit{Lan}_{Y\times Y}(-\otimes-)}\\
\cA\ar[rrr]^{Y}&&& [\mathcal{A}^{op},{\bf Set}]
}$$
In more detail the following data is obtained:

\begin{itemize}
\item $-\otimes_D-:[\mathcal{A}^{op},{\bf Set}]\times[\mathcal{A}^{op},{\bf Set}]\rightarrow[\mathcal{A}^{op},{\bf Set}]$ is defined by
$$S\otimes_DT=\int^a S(a)\times\int^b T(b)\times\mathcal{A}(-,a\otimes b)$$
This operation is also called the {\em convolution} of
 $S$ and $T$.
\item $I_D=\mathcal{A}(-,I)$
\item $l:I_D\otimes_D T\rightarrow T$ is given by:\\
$\int^xI_D(x)\times(\int^a T(a)\times\mathcal{A}(-,x\otimes a))\cong\int^a(\int^xI_D(x)\times\mathcal{A}(-,x\otimes a))\times T(a)\stackrel {\int\lambda^{*}\times 1}\rightarrow \int^a\mathcal{A}(-,a)\times T(a)\cong T $\\
where $\lambda^{*}:\int^x\mathcal{A}(x,I)\times\mathcal{A}(-,x\otimes a))\cong \mathcal{A}(-,I\otimes a)\stackrel {\mathcal{A}(-,\lambda)}\rightarrow \mathcal{A}(-,a)$
\item $r:T\otimes_D I_D\rightarrow T$: analogous.
\item $a:(R\otimes_DS)\otimes_DT\rightarrow R\otimes_D(S\otimes_DT)$\\
$(R\otimes_DS)\otimes_DT={}$\\
$\int^x(\int^aR(a)\times(\int^bS(b)\times\mathcal{A}(x,a\otimes b))\times(\int^cT(c)\times \mathcal{A}(-,x\otimes c)))$\\
$\stackrel {\cong\circ(\int 1\times\int 1\times\int(1\times\alpha))\circ\cong}\rightarrow$\\
$\int^aR(a)\times\int^x(\int^bS(b)\times(\int^cT(c)\times\mathcal{A}(x,b\otimes c)))\times \mathcal{A}(-,a\otimes x))=R\otimes_D(S\otimes_DT)$
\item $c:S\otimes_D T\rightarrow T\otimes_D S$ is\\
$$S\otimes_D T=\int^a S(a)\times(\int^b T(b)\times\mathcal{A}(-,a\otimes b))\cong\int^b T(b)\times(\int^a S(a)\times\mathcal{A}(-,a\otimes b))$$
$$\stackrel {\int1\times(\int1\times\sigma)}\rightarrow \int^b T(b)\times(\int^a S(a)\times\mathcal{A}(-,b\otimes a))=T\otimes_D S$$
\item the internal hom is:
$$[S,T]_D\cong\int_b[S(b),T(-\otimes b)]$$
\end{itemize}
\end{proof}

For more details on this construction we refer the reader to~\cite{B.Day70}.

\section{The reflective subcategory $[\cC,\cA]_{\Gamma}$}\label{THE REFLECTIVE SUBCAT SECTION}

In this section we give a brief overview the methodology of Freyd and Kelly~\cite{Freyd-Kelly 1972} in order to build reflections in a more general way using the notion of orthogonality. In particular, we are interested in some full subcategories of presheaves.
This construction generalizes Lambek's presentation in Section~\ref{LAMBEK COMPLETION} by regarding the condition of preserving limits as a special case of the continuity of functors over a certain class of cylinders.

Given an object $A\in \cA$, we define a preorder among the class of monomorphisms with codomain $A$: if $f:B\rightarrow A $, $g:C\rightarrow A$ are two monomorphisms $f$ is said to be smaller than $g$ $(f\leq g)$ when $f$ factors through $g$ i.e., $f=gk$ for some $k:B\rightarrow C$. Note that $k$ is unique and also a monomorphism.

We have an equivalence relation $f\equiv g$ iff $f\leq g$ and $g\leq f$.
\begin{definition}
\rm A \textit{subobject} of $A$ is an equivalence class of these monomorphisms.
\end{definition}
The class of subobjects is partially ordered by the order induced by the representatives.
\begin{definition}
\rm We say that a category $\cA$ is \textit{well-powered} when for every $A\in \cA$ the class of subobjects of $A$ is a set.
\end{definition}
The dual notions applied to epimorphisms are called \textit{quotient} for an equivalence class of epimorphisms, and \textit{co-well-powered}.
\begin{definition}
\rm Let $A$ be an object. The \textit{intersection} of a family of subobjects of $A$, if it exists, is the greatest lower bound defined in the partially ordered class of subobjects of $A$. Analogously, by the \textit{union} we mean the least upper bound, if it exists.
\end{definition}
Concretely, we mean the following: if $\{A_i\stackrel{f_i}\rightarrow A\}_{i\in I}$ are subobjects of $A$ then there exists an arrow $\cap_{i\in I}A_i\stackrel{f}\rightarrow A$ satisfying the following properties:
\begin{itemize}
\item[-]$f\leq f_i\,\,\, \forall i\in I$, i.e., for every $i\in I$ there exists an arrow $\cap_{i\in I}A_i\stackrel{t_i}\rightarrow A_i$ such that $f_i\circ t_i=f$.
\item[-]if there exists a $p$ such that $p\leq f_i\,\,\, \forall i \in I$ then $p\leq f$, i.e., if there are maps $B\stackrel{p}\rightarrow A$ and $B\stackrel{p_i}\rightarrow A_i$ with the property $f_i\circ p_i=p\,\,\, \forall i\in I$ then there exists a unique $h:B\rightarrow \cap_{i\in I}A_i$ such that $p=f\circ h$.
\end{itemize}

\begin{definition}
\rm An infinite limit cardinal $\alpha$ is \textit{regular} when it is equal to its cofinality: ${\rm cf}(\alpha)=\alpha$.
Here ${\rm cf}(\alpha)$ is the least limit ordinal $\beta$ such that there
exists an increasing sequence $\{\alpha_{\eta}\}_{\eta < \beta}$ with ${\rm lim}_{\eta\rightarrow \beta}\,\,\alpha_{\eta}=\alpha.$
\end{definition}
The fact that $\alpha$ is regular means cannot be written as a sum of a lesser number of cardinals less than $\alpha$.
\begin{definition}
\rm Let $\alpha$ be a regular cardinal. An ordered set $J$ is \textit{$\alpha$-directed} when for every subset $I\subseteq J$ with $|I|\leq \alpha$ there exists an upper bound in $J$.
\end{definition}
\begin{definition}
\rm Let $S=\{f_{\xi}:C_{\xi}\rightarrow B\}$ be a family of subobjects of $B$ with the monotonic property: $f_{\xi}\leq f_{\zeta}$ whenever $\xi\leq\zeta$. The family $S$ is called \textit{$\alpha$-directed} provided that the set $J$ is \textit{$\alpha$-directed}.
\end{definition}
\begin{definition}
\rm We say that an object $A\in \cA$ is \textit{bounded} by a regular cardinal $\alpha$ when for every morphism from $A$ to a $\alpha$-directed union $\cup_{\xi\in J}C_{\xi}$ factors through a union $\cup_{\xi\in K}C_{\xi}$ for some $K\subseteq J$ with $|K|<\alpha$. We call $\cA$ \textit{bounded} if each $A\in \cA$ is bounded.
\end{definition}

\begin{definition}
\rm Let $E,M\subseteq {\rm Mor}(\cA)$ be two classes of morphisms. A \textit{factorization system} $(E,M)$ on a category $\cA$ consists of the following data:
\begin{itemize}
\item[-]${\rm Isos}(\cA)\subseteq E\cap M$, isomorphisms belong to the intersection of the two classes
\item[-]$E$ and $M$ are closed under composition
\item[-]for every morphism $f$ there is a factorization $f=m\circ e$ with $e\in E$ and $m\in M$
\item[-]for every $f$ and $g$ if $m'\circ e'\circ f=g\circ m\circ e$ with $e,e'\in E$ and $m,m'\in M$ then there exists a unique $w$ making the whole diagram
$$\xymatrix@=25pt{
\bullet\ar[d]_{f}\ar[r]^{e}&\bullet\ar@{-->}[d]_{w}\ar[r]^{m}&\bullet\ar[d]^{g}\\
\bullet\ar[r]_{e'}&\bullet\ar[r]_{m'}&\bullet
}$$
commutative. A factorization system $(E,M)$ is called a \textit{proper factorization} when $E\subseteq {\rm Epis}(\cA)$, $M\subseteq {\rm Monos}(\cA)$ where ${\rm Epis}(\cA)$ is the class of all epimorphisms of $\cA$ and ${\rm Monos}(\cA)$ is the class of all monomorphisms of $\cA$.
\end{itemize}
\end{definition}

\begin{definition}
\rm An epimorphism $p$ is called \textit{extremal} provided that whenever we have $p=m\circ g$, where $m$ is a monomorphism then $m$ is also an isomorphism. Dually we define the notion of \textit{extremal monomorphism}. ${\rm Epi}^{*}$ denotes the class of extremal epimorphism and ${\rm Mon}^{*}$ the class of extremal monomorphism.
\end{definition}
\begin{proposition}
If one of these two conditions below are satisfied
\begin{itemize}
\item[-]the category $\cA$ is finitely complete and has arbitrary intersections of monomorphisms
\item[-]the category $\cA$ is finitely co-complete and has arbitrary co-intersections of extremal epimorphisms
\end{itemize}
then $({\rm Epi}^{*},{\rm Mon})$ is a proper factorization system.
\end{proposition}
\begin{proof}
\cite{Freyd-Kelly 1972}
\end{proof}
In the case of the category of sets a direct calculation shows that ${\rm Epi}^{*}={\rm Epi}$ and ${\rm Mon}^{*}={\rm Mon}$ since we have:
if $p\in {\rm Epi}$ with $p=m\circ g$ then $a\circ m=b\circ m$ implies $a\circ m\circ g=b\circ m\circ g$ which is $a\circ p=b\circ p$ and this $a=b$.

\begin{definition}
\rm
Given factorization system $({E},{M})$ a \textit{generator} of the category $\cA$ is a small full subcategory $\cG$ such that for each $A\in\cA$,  $\,\,\,\cup_{G\in\cG}\cA(G,A)\subseteq {E}$.
\end{definition}
When a factorization system $({E},{M})$ is proper and $\cG$ a generator then given any pair of morphisms $f,g:A\rightarrow B$ then for every $p:G\rightarrow A$ with $G\in\cG$ we have that $f\circ p=g\circ p\,\,\Rightarrow\,\,f=g$.

If $\cA$ has coproducts then $\cG$ is a generator if and only if for every $A\in\cA$ the map:
$$k_A:\coprod_{G\in\cG}(\coprod_{\cA(G,A)})\rightarrow A$$
is in ${E}$;
where $k_A$ is defined by the universal property of the coproduct, i.e., $k_A\circ i_{G,f}=f:G\rightarrow A$ and $i_{G,f}:G\rightarrow\coprod_{G\in\cG}(\coprod_{\cA(G,A)})$ is the coproduct injection.

\begin{definition}
\rm
Let $P,Q:\cK\rightarrow \cC$ be functors with $\cK$ a small category. A \textit{cylinder} in $\cC$ is just a natural transformation $\alpha:P\rightarrow Q$.
\end{definition}
\begin{definition}
\rm
A functor $T:\cC\rightarrow \cA$ is \textit{continuous} with respect to the cylinder $\alpha$ when:
\begin{itemize}
\item[-] there exists ${\rm lim}\, TP$ and ${\rm lim}\, TQ$ as cones in $\cA$.
\item[-] the unique morphism ${\rm lim}\, T\alpha: {\rm lim}\, TP\rightarrow {\rm lim}\, TQ$ determined by the definition of limit $({\rm lim}\, TQ,\pi_K^{\cQ})$:
$$\xymatrix@=25pt{
TPK\ar[rr]^{T\alpha_K}&&TQK\\
\lim\, TP\ar[u]^{\pi_K^{\cP}}\ar@{-->}[rr]^{{\rm lim}\,T\alpha}&&{\rm lim}\, TQ\ar[u]^{\pi_K^{\cQ}}
}$$
is an isomorphism.
\end{itemize}
\end{definition}

\begin{remark}
\rm
In the case when $P=\Delta C$ is a constant functor, $C\in\cC$,  then $\alpha$ is just a cone in the usual sense and continuity is the standard definition of continuity of functors.
\end{remark}
\begin{definition}
Let $\Gamma$ be a class of cylinders in the category $\cC$. Then $[\cC,\cA]_\Gamma$ is the full
subcategory of $[\cC,\cA]$ of functors $T$ that are continuous
w.r.t. each $(P,Q,\alpha)\in\Gamma$.
\end{definition}

\begin{definition}
\rm
Consider an arrow $f:A\rightarrow B$ and an object $C\in\cA$. We say that $f$ is \textit{orthogonal} to $C$, and we write $C\perp f$, if for every morphism $y:A\rightarrow C$ there exists a unique $x:B\rightarrow C$ such that $x\circ f=y$.
\end{definition}
This definition is basically the definition of a bijective function since is equivalent to the fact that the representables $\cA(B,C)\stackrel{\cA(f,C)}\longrightarrow \cA(A,C)$ are isos in the category of sets.

Dually we consider $f\perp C$.
\begin{definition}
Given a class $\Delta$ of morphisms in a category $\cA$, let us consider the full subcategory of $\cA$ defined by the following object: $\Delta^{\perp}=\{B\in \cA: B\perp f, \forall f\in\Delta \}$.
\end{definition}
\begin{definition}
Let us consider $X\in {\bf Set}$, where $A\in \cA$. The tensor product $X\otimes A\in \cA$ is the co-power, i.e., the coproduct of $|X|$ copies of the object $A$ in the category $\cA$ characterized by the following natural isomorphism:
$$\cA(X\otimes A,B)\cong{\bf Set}(X,\cA(A,B)).$$
\end{definition}
Now, to each cylinder $\alpha:P\rightarrow Q: \cK \rightarrow \cC$ we associate an arrow $\tilde{\alpha}:\tilde{Q}\rightarrow\tilde{P}$ in the presheaf category $[\cC,\cA]$ in the following way.

First we consider the functor $\hat{P}:\cK^{op}\rightarrow [\cC,{\bf Set}]$ defined by:
$$\cK^{op}\stackrel{P^{op}}\longrightarrow\cC^{op}\stackrel{Y}\longrightarrow [\cC,{\bf Set}]$$

thus $\hat{P}(K)=\cC(P(K),-)$, and $\cC(f,-)$ if $f^{op}\in\cK^{op}$.

Then, we take $\tilde{P}=colim\,\hat{P}$ the pointwise colimit in the category $[\cC,{\bf Set}]$, i.e., $\tilde{P}\in[\cC,{\bf Set}]$.

In the same way, at the level of arrows we get:
$$\hat{\alpha_K}=\cC(\alpha_K,-):\cC(QK,-)\longrightarrow \cC(PK,-)$$
and then we obtain:
$$\xymatrix@=25pt{
\hat{Q}(K)\ar[d]^{\pi^{\cQ}_K}\ar[rr]^{\hat{\alpha}_K}&&\hat{P}(K)\ar[d]^{\pi^{\cP}_K}\\
\tilde{Q}\ar@{-->}[rr]^{\tilde{\alpha}}&&\tilde{P}
}$$
by definition of colimit $(\tilde{Q},\pi^{\cQ}_K)$, since $\pi^{\cP}_K\circ\hat{\alpha}_K$ is natural in $K$. So,  $\tilde{\alpha}$ is given as the unique arrow in $[\cC,{\bf Set}]$ making the previous diagram commute.
Now we consider the class of morphisms $\Delta\subseteq[\cC,\cA]$ depending on a choice of a class of cylinders $\Gamma$:
$$\Delta=\{\tilde{\alpha}\otimes A:\tilde{Q}\otimes A\rightarrow\tilde{P}\otimes A, {\rm with}\,\,\, A\in\cA,\,\,\alpha \in \Gamma\}$$
where $\tilde{Q}\otimes A:\cC\rightarrow\cA$ and $\tilde{\alpha}\otimes A$ are defined using the pointwise co-power as $(\tilde{Q}\otimes A)(C)=\tilde{Q}(C)\otimes A$.

\begin{proposition}
Let $\cA$ be a complete and co-complete category and let $\Gamma$ be a class of cylinders in the small category $\cC$.
Then $[\cC,\cA]_{\Gamma}=\Delta^{\perp}$.
\end{proposition}
\begin{proof}
Since both categories are full it is enough to check that they contain the same objects. We want to prove that $T\in \Delta^{\perp}$ if and only if $T\in [\cC,\cA]_{\Gamma}$.

By definition of the orthogonal class, $T\in \Delta^{\perp}$ if and only if for every $\tilde{\alpha}\otimes A\in\Delta$ we have that $[\cC,\cA](\tilde{\alpha}\otimes A,T)$ is a bijective map, i.e., for every $\mu:\tilde{Q}\otimes A\rightarrow T$ there exists a unique $\nu$ such that,
$$\xymatrix@=25pt{
\tilde{Q}\otimes A\ar[d]_{\tilde{\alpha}\otimes A}\ar[rr]^{\mu}&&T\\
\tilde{P}\otimes A \ar[rru]_{\nu}&&
}$$
But since when $\cA(F(X),B)\cong{\bf Set}(X,G(B))$ is natural with $F:{\bf Set}\rightarrow\cA$, $F(X)=X\otimes A$ and $G:\cA\rightarrow{\bf Set}$, $G(B)=\cA(A,B)$. Then we have that:
$$\xymatrix@=25pt{
\cA(F(X),B)\ar[d]_{\cA(F(f),g)}\ar[rr]^{\phi_{X,B}}&&{\bf Set}(X,G(B))\ar[d]^{{\bf Set}(f,G(g))}\\
\cA(F(X'),B')\ar[rr]_{\phi_{X',B'}}&&{\bf Set}(X',G(B'))
}$$
with $X'\stackrel{f}\rightarrow X$ and $B\stackrel{g}\rightarrow B'$.

This implies that $G(g)\circ \phi_{X,B}(x)\circ f=\phi_{X',B'}(g\circ x\circ F(f))$ for every $x:F(X)\rightarrow B$.
Therefore choosing $g=1$, $x=\nu$, $X=\tilde{P}$, $X'=\tilde{Q}$, $f=\tilde{\alpha}$, $B=T$ we have that since $\nu\circ F(\tilde{\alpha})=\mu$ then $\phi_{X',B'}(\nu\circ F(\tilde{\alpha}))=\phi_{X',B'}(\mu)$ and then $G(1)\circ \phi_{X,B}(\nu)\circ \tilde{\alpha}=\phi_{X',B'}(\mu)$.

Using the natural isomorphism let us call $\nu'=\phi_{X,B}(\nu):\tilde{P}\rightarrow\cA(A,T-)$ where $\nu:F(\tilde{P})\rightarrow T$ and
$\mu'=\phi_{X',B'}(\mu):\tilde{Q}\rightarrow\cA(A,T-)$ where $\mu:F(\tilde{Q})\rightarrow T$.
So this turns out to be $\nu'\circ \tilde{\alpha}=\mu$,
$$\xymatrix@=25pt{
\tilde{Q}\ar[d]_{\tilde{\alpha}}\ar[rr]^{\mu'}&&\cA(A,T-)\\
\tilde{P}\ar[rru]_{\nu'}&&
}$$
Then by definition of $\tilde{Q}={\rm colim}\,\hat{Q}$ with injection $\hat{Q}K\stackrel{i_{K}^{Q}}\rightarrow \tilde{Q}$ and $\tilde{P}={\rm colim}\,\hat{P}$ with injection $\hat{P}K\stackrel{i^{P}_{K}}\rightarrow \tilde{P}$ we define $\mu''$ and $\nu''$ by the following compositions:
$\mu''=\mu'\circ i^{Q}_K$ where $\cC(QK,-)=\hat{Q}K\stackrel{i^{Q}_K}\longrightarrow\tilde{Q}\stackrel{\mu'}\longrightarrow\cA(A,T-)$
and
$\nu''=\nu'\circ i^{P}_K$ where $\cC(PK,-)=\hat{P}K\stackrel{i^{P}_K}\longrightarrow\tilde{P}\stackrel{\mu'}\longrightarrow\cA(A,T-)$.
Therefore we have

$$\xymatrix@=25pt{
\cC(QK,-)\ar[d]_{\cC(\alpha_K,-)}\ar[rr]^{i^Q_K}&&\tilde{Q}\ar[d]_{\tilde{\alpha}}\ar[rr]^{\mu'}&&\cA(A,T-)\\
\cC(PK,-)\ar[rr]_{i^P_K}&&\tilde{P}\ar[rru]_{\nu'}&&
}$$
Let us call $F=\cA(A,T-)$, then by naturality of the Yoneda Lemma with respect to $\alpha_K$ we have that:
$$\xymatrix@=25pt{
[\cC,{\bf Set}](\cC(PK,-),F)\ar[d]_{[\cC,{\bf Set}](\cC(\alpha_K,-),F)}\ar[rrr]^{\theta_P}&&&F(PK)\ar[d]_{F(\alpha_K)}\\
[\cC,{\bf Set}](\cC(QK,-),F)\ar[rrr]^{\theta_Q}&&&F(QK)
}$$
Thus if we evaluate $\nu:\cC(PK,-)\rightarrow F$ we obtain:
$$\theta_Q(\nu\circ \cC(\alpha_K,-))=F(\alpha_K)(\theta_P(\nu))$$
and since $F=\cA(A,T-)$ then we get
$$\theta_Q(\nu\circ \cC(\alpha_K,-))=T(\alpha_K)\circ\theta_P(\nu)$$
Therefore since $\mu''=\nu''\circ \cC(\alpha_K,-)$ we have by choosing $\nu=\nu''$:
$$\theta _Q(\mu'')=T(\alpha_K)\circ\theta_P(\nu'')$$
where $\theta _Q(\mu'')\in F(QK)=\cA(A,TQK)$, $\theta _Q(\mu''):A\rightarrow TQK$ and
$\theta _Q(\nu'')\in F(PK)=\cA(A,TPK)$, $\theta _P(\nu''):A\rightarrow TPK$.

So by naturality of $K$ and the definition of limit we obtain the following diagram:
$$\xymatrix@=25pt{
&&&&{\rm lim}\,TQ\ar[ld]^{\pi^Q}\\
A\ar[rrrru]^{\bar{\mu}}\ar@{=}[d]_{}\ar[rrr]_{\theta _Q(\mu'')}&&&TQK&\\
A\ar[rrrrd]_{\bar{\nu}}\ar[rrr]^{\theta _P(\nu'')}&&&TPK\ar[u]_{T(\alpha_K)}&\\
&&&&{\rm lim}\,TP\ar[uuu]_{{\rm lim}\,T\alpha}\ar[lu]_{\pi^P}
}$$
Thus the condition of $T\in[\cC,\cA]_{\Gamma}$ (continuity) is by definition that ${\rm lim}\,T\alpha$ is an isomorphism  and $T\in \Delta^{\perp}$ (orthogonality) iff $[\cC,\cA](\tilde{\alpha}\otimes A,T)$ is an isomorphism.
\end{proof}
\begin{theorem}
Let $\cA$ be a complete and co-complete category with a given proper factorization system $(E,M)$. Let $\cA$ be bounded and co-well-powered. Let us consider the class $\Delta=\Phi\cup\Psi$ where $\Phi$ is small and where $\Psi\subseteq{E}$. Then $\Delta^{\perp}$ is a reflective subcategory of $\cA$.
\end{theorem}
\begin{proof}
\cite{Freyd-Kelly 1972}
\end{proof}
\begin{theorem}\label{abstract1}
Let $\cA$ be a complete and co-complete category with a given proper factorization system $(\textit{E},\textit{M})$. Let $\cA$ be bounded with a generator, and co-well-powered. Let $\Gamma$ be a class of cylinders in the small category $\cC$, and let all but a set of these cylinders be cones. Then $[\cC,\cA]_{\Gamma}$ is a reflective subcategory of $[\cC,\cA]$.
\end{theorem}
\begin{proof}
\cite{Freyd-Kelly 1972}
\end{proof}

\section{Day's reflection theorem}\label{Day s reflection theorem}
Let $\cB$ be a symmetric monoidal closed category. Day's so-called
 reflection theorem~\cite{B.Day72} can be used to derive a monoidal closed
 structure in a reflective subcategory of $[\cB^{op},{\bf Set}]$.
In Chapter~\ref{PREASHEAVES MODELS OF Q C}, we shall utilize this to determine a strong
 monoidal functor which, in turns, determines a monoidal adjunction.
 Here, we review Day's reflection theorem.

\begin{definition}
\rm A class of objects $\cA\subseteq |\cB|$ is \textit{strongly generating} when $\cB(1,f):\cB(A,B)\rightarrow\cB(A,B')$ is an isomorphism for every $A\in\cA$ implies that $f:B\rightarrow B'$ is an isomorphism in $\cB$.

Dually we define the notion of \textit{strongly cogenerating} class of object by considering the maps $\cB(f,1)$.
\end{definition}
\begin{example}
The class $\cA\subseteq [\cB^{op},{\bf Set}]$, where $\cA=\{\cB(-,B): B\in |\cB|\}$ are representables, is strongly generating. To see this we must prove that if $(1,\alpha):[\cB^{op},{\bf Set}](\cB(-,B),F)\rightarrow [\cB^{op},{\bf Set}](\cB(-,B),G)$ is an isomorphism for every $B\in \cB$, where $(1,\alpha)=[\cB^{op},{\bf Set}](\cB(-,B),\alpha)$ acts on natural transformations as $(1,\alpha)(\beta)=\alpha\circ\beta$, then $\alpha:F\Rightarrow G$ is an isomorphism. To prove this, consider the following diagram:

$$\xymatrix@=25pt{
[\cB^{op},{\bf Set}](\cB(-,B),F)\ar[rrr]^{(1,\alpha)} & & &[\cB^{op},{\bf Set}](\cB(-,B),G)\ar[d]_{\phi_{G}}\\
 F(B)\ar[u]_{\phi_{F}^{-1}}\ar[rrr]_{\alpha_{B}}&& & G(B)
}$$
where
$\phi_{F}^{-1}:F(B)\rightarrow [\cB^{op},{\bf Set}](\cB(-,B),F)$ is defined $\phi_{F}^{-1}(x):\cB(-,B)\Rightarrow F$ as
$(\phi_{F}^{-1}(x))_{C}(g)=F(g)(x)$ and $\phi_{G}:[\cB^{op},{\bf Set}](\cB(-,B),G)\rightarrow G(B)$ is defined as $\phi_{G}(\beta)=\beta_{B}(1_{B})$.

Therefore, we have
$$(\phi_{G}\circ(1,\alpha)\circ\phi_{F}^{-1})(x)=\phi_{G}((1,\alpha)((\phi_{F}^{-1}(x))))=$$
$$=\phi_{G}(\alpha\circ \phi_{F}^{-1}(x))=(\alpha\circ \phi_{F}^{-1}(x))_{B}(1_{B})=\alpha_{B}\circ (\phi_{F}^{-1}(x))_{B}(1_{B})=$$
$$=\alpha_{B}((\phi_{F}^{-1}(x))_{B}(1_{B}))=\alpha_{B}(F(1_{B})(x))=\alpha_{B}(1_{FB})(x))=\alpha_{B}(x),$$
which means $\phi_{G}\circ(1,\alpha)\circ\phi_{F}^{-1}=\alpha_{B}$.
\end{example}
\begin{theorem}
(Day's reflection theorem)\label{DAY'S REFLECTION THEOREM}
Let $(\cB,\otimes,I,[-])$ be a symmetric monoidal closed category, and let
$\xymatrix{
\mathcal{B}\ar@<1ex>[r]^{F}& \mathcal{C} \ar@<1ex>[l]^{G}_{\bot}}$
be an adjunction from $\cB$ to $\cC$, where $G$ is full and faithful. Let $\cA\subseteq |\cB|$ be a strongly generating class in $\cB$ and $\cD\subseteq |\cC|$ be a strongly cogenerating class in $\cC$. Then the following are equivalent:
\begin{itemize}
\item[(0)] there exists a monoidal closed structure on $\cC$ for which $F$ is a monoidal strong functor.
\item[(a)]   $\eta:[B,GC]\rightarrow GF[B,GC]$, is an isomorphism for all $C\in\cC$, $B\in\cB$.
\item[(b)]   $\eta:[A,GD]\rightarrow GF[A,GD]$, is an isomorphism for all $A\in\cA$, $D\in\cD$.
\item[(c)]   $[\eta,1]:[GFB,GC]\rightarrow [B,GC]$, is an isomorphism for all $C\in\cC$, $B\in\cB$.
\item[(d)]   $F(\eta\otimes 1):F(B\otimes B')\rightarrow F(GFB\otimes B')$, is an isomorphism for all $B,B'\in\cB$.
\item[(e)]   $F(\eta\otimes 1):F(B\otimes A)\rightarrow F(GFB\otimes A)$, is an isomorphism for all $A\in\cA$, $B\in\cB$.
\item[(f)]   $F(\eta\otimes \eta):F(B\otimes B')\rightarrow F(GFB\otimes GFB')$, is an isomorphism for all $B,B'\in\cB$.
\end{itemize}

\end{theorem}
\begin{proof}
$(a)\Rightarrow (b)$ Since $\cA\subseteq\cB$ and $\cD\subseteq \cC$.\\
$(b)\Rightarrow (e)$
$$\xymatrix@=25pt{
\cC(F(GFB\otimes A),D)\ar[d]_{\mbox{\tiny{adjunction}}} \ar[rrr]^{\cC(F(\eta\otimes 1),1)} && &\cC(F(B\otimes A),D)\ar[d]_{\mbox{\tiny{adjunction}}}\\
 \cB(GFB\otimes A,GD)\ar@{}[urrr]^{(4)}\ar[d]_{\mbox{\tiny{closed}}} \ar[rrr]^{\cB(\eta\otimes 1,1)} &&  &\cB(B\otimes A,GD)\ar[d]_{\mbox{\tiny{closed}}}\\
 \cB(GFB,[A,GD])\ar@{}[urrr]^{(3)}\ar[d]_{\mbox{\tiny{$\eta$ iso by hypothesis}}}\ar@{}[d]^{\cB(1\otimes\eta)} \ar[rrr]^{\cB(\eta,1)} && &\cB(B,[A,GD])\ar[d]_{\mbox{\tiny{$\eta$ iso by hypothesis}}}\ar@{}[d]^{\cB(1\otimes\eta)}\\
 \cB(GFB,GF[A,GD])\ar@{}[urrr]^{(2)}\ar[d]_{\mbox{\tiny{$G$ fully faithful}}} \ar[rrr]^{\cB(\eta,1)} & & &\cB(B,GF[A,GD])\\
 \cC(FB,F[A,GD])\ar@{}[urr]^{(1)}\ar[urrr]_{\mbox{\tiny{adjunction}}}&& &
}$$
(1) commutes since we have $\theta(f)=G(f)\circ \eta_{B}=(\cB(\eta,1)\circ G_{FB,F[A,GD]})(f)$
$$\xymatrix@=25pt{
\cB(GFB,GF[A,GD])\ar[rrr]^{\cB(\eta,1)} & & &\cB(B,GF[A,GD])\\
 \cC(FB,F[A,GD])\ar[u]_{G}\ar[urrr]_{\theta}&& &
}$$
(2) by functoriality; (3) and (4) by naturality. The vertical and bottom arrows are isos then the top is an isomorphism. Hence since $D\subseteq\cC$ is strongly cogenerating we have that $F(\eta\otimes 1):F(B\otimes A)\rightarrow F(GFB\otimes A)$ is an isomorphism for every $A\in\cA$ and $B\in\cB$.\\

$(e)\Rightarrow (c)$\\
$$\xymatrix@=25pt{
\cC(F(GFB\otimes A),C)\ar[d]_{\mbox{\tiny{adjunction}}} \ar[rrr]^{\cC(F(\eta\otimes 1),1)} && &\cC(F(B\otimes A),C)\ar[d]_{\mbox{\tiny{adjunction}}}\\
 \cB(GFB\otimes A,GC)\ar@{}[urrr]^{(2)}\ar[d]_{\mbox{\tiny{closed}}} \ar[rrr]^{\cB(\eta\otimes 1,1)} &&  &\cB(B\otimes A,GC)\ar[d]_{\mbox{\tiny{closed}}}\\
 \cB(A,[GFB,GC])\ar@{}[urrr]^{(1)} \ar[rrr]^{\cB(1,[\eta,1])} && &\cB(A,[B,GC])
 }$$

(1) and (2) commute by naturality. The top arrow is an isomorphism by hypothesis, also the vertical arrows are isomorphism, this implies that the bottom arrow is an iso and since $\cA$ is strongly generating then $[\eta,1]:[GFB,GC]\rightarrow[B,GC]$ is an isomorphism as well.\\

$(c)\Rightarrow (d)$ We use the same diagram with $A\in\cB$.\\

$(d)\Rightarrow (f)$\\
By functoriality
$$\xymatrix@=25pt{
F(B\otimes B')\ar[rr]^{F(\eta\otimes \eta)}\ar[rd]_{F(\eta\otimes 1)} & &F(GFB\otimes GFB')\\
  & F(GFB\otimes B')\ar[ru]_{F(1\otimes \eta)}&
}$$

$(f)\Rightarrow (a)$\\
We want to find an arrow $\nu:GF[B,GC]\rightarrow [B,GC]$ such that $\eta\circ\nu=\nu\circ\eta=1$. From naturality of the following diagram

$$\xymatrix@=25pt{
\cB(GF[B,GC]\otimes B,GC)\ar[rr]^{\phi} \ar[d]^{\cB(\eta\otimes 1,1)}& &\cB(GF[B,GC],[B,GC])\ar[d]^{\cB(\eta,1)}\\
\cB([B,GC]\otimes B,GC)\ar[rr]^{\phi} &  &\cB([B,GC],[B,GC])
}$$

\noindent
we obtain $\cB(\eta\otimes 1,1)(\phi^{-1}(\nu))=\phi^{-1}(\cB(\eta,1)(\nu))$ which implies that $\phi^{-1}(\nu)\circ(\eta\otimes 1)=\phi^{-1}(\nu\circ\eta)$. On the other hand we have that
\begin{center}
$1= \nu\circ\eta$ if and only if $\phi^{-1}(1)=\phi^{-1}( \nu\circ\eta)$ if and only if $ev=\phi^{-1}(\nu)\circ(\eta\otimes 1).$
\end{center}
 Therefore by uniqueness it is enough to find an arrow $x$ of the correct type which is a solution of the following equation
$$ev=x\circ(\eta\otimes 1)$$
for then $x=\phi^{-1}(\nu)$, i.e., $\phi(x)=\nu$. We choose $x=G(\theta^{-1}(ev))GF(\eta\otimes\eta)\eta (1\otimes\eta)$ satisfying the following diagram

$$\xymatrix@=25pt{
[B,GC]\otimes B\ar[dd]_{\eta\otimes 1}\ar[rdd]^{\eta\otimes\eta}\ar[rrd]_{\eta} \ar[rr]^{ev}& &GC\ar@{}[dl]_{(1)}\\
  & &GF([B,GC]\otimes B)\ar[u]_{G(\theta^{-1}(ev))}\\
 GF[B,GC]\otimes B \ar[r]_{1\otimes\eta}&GF[B,GC]\otimes GFB\ar[r]_{\eta}&GF(GF[B,GC]\otimes GFB)\ar[u]_{GF(\eta\otimes\eta)}
}$$

To justify (1), let $\phi:\cB([B,GC]\otimes B,GC)\rightarrow \cB([B,GC],[B,GC])$ be the tensor adjunction. By definition we have $ev=\phi^{-1}(1_{[B,GC]})$.
Now consider the adjunction between functors $F$ and $G$,
$$\theta:\cC(F([B,GC]\otimes B),C)\rightarrow \cB([B,GC]\otimes B,GC)$$
and take $e'=\theta^{-1}(ev)$. Then we have that $G(e')\circ \eta_{{[B,GC]\otimes B}} =\theta(e')=\theta(\theta^{-1}(ev))=ev$.\\
It remains to prove that $\eta\circ\nu=1$. Since $G:\cC\rightarrow\cB$ is  a fully faithful functor, there is a unique $f$ such that $G(f)=\eta\circ\nu$. Also we know that $\nu\circ\eta=1$
Hence, we have
$$G(f)\circ\eta=(\eta\circ\nu)\circ\eta=\eta\circ(\nu\circ\eta)=\eta\circ(1)=\eta=G(1)\circ\eta.$$
Finally, from the adjunction
$\theta:\cC(F[B,GC],F[B,GC]\rightarrow\cB([B,GC],GF[B,GC])$
we obtain $\theta(f)=G(f)\circ\eta_{[B,GC]}$, which implies that $\theta(f)=\theta(1)$, i.e., $f=1$. Therefore $\eta\circ\nu=G(1)=1$.

$(0)\Rightarrow (f)\,\,\,$ See~\cite{KELLY 1969}.\\

$(f)\Rightarrow (0)$\\
\noindent
\textbf{The monoidal closed structure induced on $\cC$:}\\
Now using Theorem~\ref{DAY'S REFLECTION THEOREM} we are able to induce a monoidal structure on the category $\cC$.
Define
$C\tilde{\otimes}C'=F(GC\otimes GC')$ \ and \
$f\tilde{\otimes}g=F(Gf\otimes Gg)$.
Also define
$\tilde{I}=FI$ and
$(F,m)$ is monoidal functor, where
$m_{A,B}:F(A)\tilde{\otimes}F(B)\rightarrow F(A\otimes B)$ is given by: $m_{A,B}=(F(\eta\otimes\eta))^{-1}$ with
$(F(\eta\otimes\eta))^{-1}:F(GFA\otimes GFB)\rightarrow F(A\otimes B)$.\\
The tensor has right adjoint given by the following formula $[C,E]_{\cC}=F[GC,GE]$, $C,E\in |\cC|$
$$\cC(D\tilde{\otimes}C,E)=\cC(F(GD\otimes GC),E)\cong\cB(GD\otimes GC,GE)$$
$$\cong\cB(GD,[GC,GE])\cong\cB(GD, GF[GC,GE])\cong\cC(D,F[GC,GE])\cong\cC(D,[C,E]_{\cC}).$$
In order to obtain a monoidal structure on the category $\cC$ we define natural isomorphisms $\tilde{\lambda}$, $\tilde{\rho}$ and $\tilde{\alpha}$ determined by the following diagrams:
$$\xymatrix@=25pt{
\tilde{I}\tilde{\otimes}C=F(GFI\otimes GC)\ar[rr]^{\tilde{\lambda}}\ar[d]_{F(\eta_{I}\otimes 1)^{-1}}& & C\\
 F(I\otimes GC)\ar[rr]_{F(\lambda)} && FGC\ar[u]_{\epsilon_{C}}
}$$

$$\xymatrix@=25pt{
C\tilde{\otimes}\tilde{I}=F(GC\otimes GFI)\ar[rr]^{\tilde{\rho}}\ar[d]_{F(1\otimes\eta_{I})^{-1}}& & C\\
 F(GC\otimes I)\ar[rr]_{F(\rho)} && FGC\ar[u]_{\epsilon_{C}}
}$$

$$\xymatrix@=25pt@u@R=25pt{&\\
\makebox[0in][r]{$(C\tilde{\otimes} C')\tilde{\otimes} C''={}$}F(GF(GC\otimes GC')\otimes GC'')\ar[r]^{\tilde{\alpha}}\ar[d]_<>(.5){F(\eta\otimes 1)^{-1}} & \makebox[0in][r]{$C\tilde{\otimes} (C'\tilde{\otimes} C'')={}$}F((GC\otimes GF(GC'\otimes GC'))\\
F((GC\otimes GC')\otimes GC'')\ar[r]_{F(\alpha)} & F((GC\otimes (GC'\otimes GC''))\ar[u]_<>(.5){F(1\otimes\eta)}
}$$
For example we want to check that:

$$\xymatrix@=25pt{
(C\tilde{\otimes}\tilde{I})\tilde{\otimes}C'\ar[rr]^{\tilde{\alpha}}\ar[dr]_{\tilde{\rho}\tilde{\otimes}1}& &C\tilde{\otimes} (\tilde{I}\tilde{\otimes} C')\ar[dl]^{1\tilde{\otimes}\tilde{\lambda}}\\
 & C\tilde{\otimes}C'&
}$$
\nopagebreak
this diagram is the center (F) of the following diagram:

\begin{tiny}
\[
\xymatrix@u@R=-85pt@C=25pt{
 &&&&&&&&\\
 &&&&&&&&\\
  &F((GC\otimes I)\otimes
GC')\ar@{}[dddddddddrr]^{A}\ar@{}[uurrrrrr]^{B}\ar@/^20ex/[rrrrrr]^{F(\alpha)}\ar[ddddddddd]^{F(\eta_{GC\otimes
I}\otimes 1)}\ar[ldddddddddddddddddd]_{F(\rho\otimes 1)}&&F((GC\otimes GFI)\otimes
GC')\ar@{}[dddddddddrr]^{C}\ar[ll]^{F((1\otimes\eta)\otimes
1)^{-1}}\ar[rr]_{F(\alpha)}&&F(GC\otimes (GFI\otimes
GC'))\ar@{}[dddddddddrr]^{E}\ar[ddddddddd]_{F(1\otimes\eta)}&&F(GC\otimes
(I\otimes GC'))\ar[ll]^{F(1\otimes(\eta\otimes
1))}\ar[ddddddddd]_{F(1\otimes\eta_{I\otimes
GC'})}\ar[rdddddddddddddddddd]^{F(1\otimes\lambda)}&\\
  &&&&&&&&\\
  &&&&&&&&\\
  &&&&&&&&\\
  &&&&&&&&\\
  &&&&&&&&\\
  &&&&&&&&\\
  &&&&&&&&\\
  &&&&&&&&\\
&F(G(F(GC\otimes I))\otimes
GC')\ar@{}[ddddddddd]^{D}\ar@{}[dddddddddrrrrrr]^{F}\ar[rddddddddd]^{F(G(F(\rho))\otimes
1)}&&F(G(F(GC\otimes GFI))\otimes
GC')\ar@{}[dddddddddl]^{H}\ar[rddddddddd]_{\tilde{\rho}\tilde{\otimes}1}\ar[ll]^{F(G(F(1\otimes\eta_{I})^{-1})\otimes
1)}\ar[rr]_{\tilde{\alpha}}\ar[uuuuuuuuu]_{F(\eta\otimes 1)^{-1}}&&F(GC\otimes
GF(GFI\otimes
GC'))\ar@{}[dddddddddr]^{I}\ar[lddddddddd]^{1\tilde{\otimes}\tilde{\lambda}}\ar[rr]_{F(1\otimes
G(F(\eta\otimes 1)^{-1}))}&&F(GC\otimes GF(I\otimes GC'))\ar[lddddddddd]_{F(1\otimes
GF(\lambda))}\ar@{}[ddddddddd]^{G}&\\
&&&&&&&&\\
&&&&&&&&\\
&&&&&&&&\\
&&&&&&&&\\
&&&&&&&&\\
&&&&&&&&\\
&&&&&&&&\\
&&&&&&&&\\
F(GC\otimes GC')\ar@/_15ex/[rrrr]_<>(.8){1}\ar[rr]^{F(\eta_{GC}\otimes
1)}&&F(G(F(GC))\otimes GC')\ar[rr]^{F(G(\varepsilon_{C})\otimes 1)}&&F(GC\otimes
GC')&&F(GC\otimes GF(GC'))\ar[ll]_{F(1\otimes G(\varepsilon_{C'}))}&&F(GC\otimes
GC')\ar[ll]_{F(1\otimes\eta_{GC'})}\ar@/^15ex/[llll]^<>(.8){1}\\
&&&&&&&&
}
\]
\end{tiny}

Diagram A: By naturality of $\eta$ with $1\otimes\eta_{I}:GC\otimes I\rightarrow GC\otimes GFI$, then by functoriality of $-\otimes GC'$ and F we obtain:
$$\xymatrix@=25pt{
F((GC\otimes I)\otimes GC')\ar[rrr]^{F(\eta_{GC\otimes I}\otimes 1)}\ar[d]_{F((1\otimes\eta_{I})\otimes 1)}&& & F(GF(GC\otimes I)\otimes GC')\ar[d]^{F(GF(1\otimes\eta_{I})\otimes 1)}\\
 F((GC\otimes GFI)\otimes GC')\ar[rrr]_{F(\eta_{GC\otimes GFI}\otimes 1)} &&& F(GF(GC\otimes GFI)\otimes GC')
}$$

Since $F(1_{GC}\otimes\eta_{I})$, $F(\eta_{GC\otimes I}\otimes 1_{GC'})$ and $F(\eta_{GC\otimes GFI}\otimes 1_{GC'})$ are invertible map this implies that $F((1_{GC}\otimes\eta_{I})\otimes 1_{GC'})$ is invertible as well.

Diagram D: by naturality of $\eta$ with $\rho:GC\otimes I \rightarrow GC$ we have that $GF(\rho)\circ \eta_{GC\otimes I}=\eta_{GC}\circ \rho$ then by functoriality of $-\otimes GC'$ and $F$.

Diagram H: by definition we have $\tilde{\rho}=F(1\otimes\eta)^{-1};F(\rho);\varepsilon$, then we apply functor $-\tilde{\otimes}-=F(G(-)\otimes G(-))$ to the pair of arrows $(\tilde{\rho}, 1_{C'})$.

Diagram C: by definition of $\tilde{\alpha}$.

Diagram B: by considering the diagram $A$,  the map $F((1_{GC}\otimes\eta_{I})\otimes 1_{GC'})^{-1}$ makes sense, also by naturality of $\alpha$ with $1_{GC}$, $\eta_{I}$ and $1_{GC'}$, and then compose with $F$.

Diagram E:  this is analogous to diagram A.  We consider naturality of $\eta$ with the map $\eta_{I}\otimes 1:I\otimes GC'\rightarrow GFI\otimes GC'$, then compose with the functor $GC\otimes -$ and  $F$. Since $F(\eta_{I}\otimes 1)$ is invertible then $F(1\otimes GF(\eta_{I}\otimes 1))$ is invertible and we have that:
$$F(1\otimes \eta_{I\otimes GC'})=F(1\otimes G((F(\eta_{I}\otimes 1)^{-1}))\circ F(1\otimes \eta_{GFI\otimes GC'})\circ F(1\otimes(\eta_{I}\otimes 1))$$

Diagram G: this is analogous to diagram D. Naturality of $\eta$ with $\lambda :I\otimes GC'\rightarrow GC'$ then compose with $GC\otimes-$ and $F$.

At the bottom of the diagram we have an adjoint equation: $\eta_G\circ G(\varepsilon)=1$.

We can also define $\rho$ on the image of $F$ in the following way:

$$\xymatrix@=25pt{
FB\tilde{\otimes}\tilde{I}=F(GFB\otimes GFI)\ar[rr]^{\tilde{\rho}}\ar[d]_{F(\eta_{B}\otimes\eta_{I})^{-1}}& & FB\\
 F(B\otimes I)\ar[rru]_{F(\rho)} &&
}$$

\noindent
This coincides with the above definition:

$$\xymatrix@=25pt{
F(GFB\otimes GFI)\ar[dd]_{F(\eta_{B}\otimes\eta_{I})^{-1}}\ar[rr]^{F(1_{GFB}\otimes\eta_{I})^{-1}}& & F(GFB\otimes I)\ar[d]_{F(\rho)}\\
  &&F(G(F(B)))\ar[d]_{\varepsilon_{FB}} \\
 F(B\otimes I)\ar[rr]_{F(\rho)}&&FB
}$$

\noindent
To see this we have that:\\
$F(\rho)\circ F(\eta_{B}\otimes\eta_{I})^{-1}=\varepsilon_{FB}\circ F(\rho)\circ F(1_{GFB}\otimes\eta_{I})^{-1}$ iff\\
$F(\rho)=\varepsilon_{FB}\circ F(\rho)\circ F(1_{GFB}\otimes\eta_{I})^{-1}\circ F(\eta_{B}\otimes\eta_{I})$ iff\\
$F(\rho)=\varepsilon_{FB}\circ F(\rho)\circ F(1_{GFB}\otimes\eta_{I})^{-1}\circ F(1_{GFB} \otimes\eta_{I})\circ F(\eta_{B}\otimes 1_{I})$ iff\\
$F(\rho)=\varepsilon_{FB}\circ F(\rho)\circ F(\eta_{B}\otimes 1_{I})$ iff\\
$\varepsilon_{FB}^{-1}\circ F(\rho)=F(\rho)\circ F(\eta_{B}\otimes 1_{I})$ iff\\
$F(\eta_{B})\circ F(\rho)=F(\rho)\circ F(\eta_{B}\otimes 1_{I})$ iff\\
$F(\eta_{B}\circ\rho)=F(\rho\circ(\eta_{B}\otimes 1_{I}))$ \\
where the last two equations are justified by naturality of $\rho$ with $\eta_{B}:B\rightarrow GFB$ and, since $G$ is full and faithful,  we have that $\varepsilon$ is an isomorphism and $\varepsilon_{FB}^{-1}=F(\eta_{B})$.

We can also define an associativity isomorphism on the image of $G$
$$\tilde{\alpha}:(GC\tilde{\otimes} GC')\tilde{\otimes}GC''\rightarrow GC\tilde{\otimes}( GC'\tilde{\otimes}GC'')$$
in the following way:

$$\xymatrix@=25pt@C=10pt{
F(G(F(G(FB)\otimes G(FB')))\otimes G(FB''))\ar[rrr]^{\tilde{\alpha}}\ar[d]_{F(G(F(\eta\otimes\eta)^{-1})\otimes 1)}&& & F(GFB\otimes GF(GFB'\otimes GFB''))\\
F(GF(B\otimes B')\otimes GFB'')\ar[d]_{F(\eta\otimes\eta)^{-1}}& && F(GFB\otimes GF(B'\otimes B''))\ar[u]_{F(1\otimes GF(\eta\otimes\eta))}\\
F((B\otimes B')\otimes B'')\ar[rrr]_{F(\alpha)}&&&F(B\otimes(B'\otimes B''))\ar[u]_{F(\eta\otimes\eta)}
}$$

$$\xymatrix@=25pt@C=10pt{
F(G(F(G(FB)\otimes G(FB')))\otimes G(FB''))\ar[d]_{F(G(F(\eta\otimes\eta)^{-1})\otimes 1)}\ar[rrr]^<>(.5){F(\eta\otimes 1)^{-1}}&&&F((GFB\otimes GFB')\otimes GFB'')\ar[d]_{F(\alpha)}\\
F(GF(B\otimes B')\otimes GFB'')\ar[d]_{F(\eta\otimes\eta)^{-1}}\ar@{}[urr]_{A}&&&F(GFB\otimes (GFB'\otimes GFB''))\ar[d]_{F(1\otimes\eta)}\\
F((B\otimes B')\otimes B'')\ar[d]_{F(\alpha)}\ar[uurrr]_{\quad F((\eta\otimes\eta)\otimes\eta)}&&&F(GFB\otimes GF(GFB'\otimes GFB''))\\
F(B\otimes(B'\otimes B''))\ar@{}[uurr]^{B}\ar[rrr]_{F(\eta\otimes\eta)}\ar[uurrr]_{\quad F(\eta\otimes(\eta\otimes\eta))}&&&F(GFB\otimes GF(B'\otimes B''))\ar[u]_{F(1\otimes GF(\eta\otimes\eta))}\ar@{}[ul]_{C}
}$$

\noindent
Diagram A  commutes by naturality of $\eta$ with $\eta\otimes\eta:B\otimes B'\rightarrow GFB\otimes GFB'$: we apply $-\otimes\eta_{B''}$
$$\xymatrix@=25pt{
(B\otimes B')\otimes B''\ar[r]_{\eta\otimes\eta}\ar[d]_{(\eta\otimes\eta)\otimes\eta}&GF(B\otimes B')\otimes GFB''\ar[d]_{GF(\eta\otimes\eta)\otimes 1}\\
(GFB\otimes GFB')\otimes GFB''\ar[r]_{\eta\otimes 1}&GF(GFB\otimes GFB')\otimes GFB''
}$$
and then we apply functor $F$.\\
Diagram B commutes by naturality of the  isomorphism $\alpha$\\
Diagram C is analogous to diagram A: it commutes by naturality of $\eta$ with $\eta\otimes\eta:B'\otimes B''\rightarrow GFB'\otimes GFB''$, then we apply $\eta\otimes-$
and finally we evaluate the functor $F$ on this diagram.
\end{proof}

\section{Application of Day's reflection theorem to presheaves}\label{abstract4}
Now we consider a particular case of Theorem~\ref{DAY'S REFLECTION THEOREM} studied in~\cite{B.Day73}. Let us consider
$\xymatrix{
[\cB^{op},{\bf Set}]\ar@<1ex>[rr]^{F}& &\mathcal{C} \ar@<1ex>[ll]^{G}_{\bot}}$ with $G$ fully faithful and
where $([\cB^{op},{\bf Set}],\otimes, I)$ has the monoidal structure induced by the convolution product (defined in Proposition~\ref{monoidal convolution}).
When $A=\cB(-,B)$ is a representable functor, by the Yoneda Lemma we have that:
\begin{equation}
[A,G(C)]=\int_{B'}[\cB(B',B),G(C)(-\otimes B')]\cong G(C)(-\otimes B)\label{CONDITION OF ENRICHMENT HOM}
\end{equation}
Now suppose there exists $C'\in\cC$ such that
\begin{equation}
 G(C)(-\otimes B)\cong G(C')\label{CONDITION OF ENRICHMENT}
\end{equation}

\noindent
is a natural isomorphism between functors.
Let us explicitly call $\phi$ the composition of these two isomorphisms~(\ref{CONDITION OF ENRICHMENT HOM}) and~(\ref{CONDITION OF ENRICHMENT}) above: $\phi:[A,G(C)]\rightarrow G(C')$. Then we have:
$$\xymatrix@=25pt{
[A,G(C)]\ar[rr]^{\phi}\ar[d]^{\eta_{[A,G(C)]}}&&G(C')\ar[d]^{\eta_{G(C')}}\\
GF[A,G(C)]\ar[rr]^{GF(\phi)}&&GF(G(C'))
}$$
From this diagram we conclude that the condition of $\eta_{[A,G(C)]}$ being an isomorphism is equivalent to the condition of $\eta_{G(C')}$ of being an isomorphism. Thus, since $G$ is fully faithful we have by Proposition~\ref{F FULL FAITHFUL IFF ETA IS AN ISO} that $\eta\ast G$ is always an isomorphism which implies that $\eta_{[A,G(C)]}$ is an isomorphism as well. Therefore, the adjunction is monoidal if and only if condition~(\ref{CONDITION OF ENRICHMENT}) is satisfied.\\

In the particular case when $G$ is an inclusion this translates to the condition that there exists an isomorphism $C(-\otimes B)\cong C'$ where $C\in\cC\subseteq [\cB^{op},{\bf Set}]$, $B\in\cB$ for some $C'\in \cC$.

\begin{remark}\label{GAMMA CLASS OF CONES}
\rm
Consider $\cC=[\cB^{op},{\bf Set}]_{inf}$. Suppose we have two functors $F$ and $H$  isomorphic in  $[\cB^{op},{\bf Set}]$. Then $F$ preserves limits if and only if $H$ preserves limits. Therefore the condition $C(-\otimes B)\cong C'\in \cC$ implies that $C(-\otimes B)$ preserves limits, i.e., $C(-\otimes B)\in [\cB^{op},{\bf Set}]_{inf}$. We have by hypothesis that $C\in\cC$ and hence it depends on whether the functor $-\otimes B:\cB^{op}\rightarrow \cB^{op}$ preserves limits. The same is valid if we consider not all but some specific limits: a certain class $\Gamma$.
\end{remark}

\chapter[Presheaf models]
         {Presheaf models of a quantum lambda calculus}\label{PREASHEAVES MODELS OF Q C}

In this chapter we study a categorical model for the quantum lambda calculus of
Selinger and Valiron~\cite{SelVal 2006A}. We focus on exploring the existence of such a model using presheaf categories.

In \cite{Sel 2004}, Selinger defined an elementary quantum flow chart language and gave a denotational model in terms of superoperators. This axiomatic framework captures the behavior and interconnection between the basic quantum computation concepts such as the manipulation of quantum bits by considering two basic operations: measurement and unitary transformation in a lower-level language. In particular, the semantics of this framework is very well understood: each program corresponds to a concrete superoperator.

Higher-order functions are functions that can input or output other functions. In order to deal with higher-order functions, Selinger and Valiron introduced, in several papers \cite{SelVal 2006B}, \cite{SelVal 2008}, \cite{SelVal 2009} a typed lambda calculus for quantum computation and investigated several aspects of its semantics. In this context, they combined two very well-established languages in the literature of computer science: the intuitionistic fragment of Girard's linear logic~\cite{GIRARD 87} and the computational monads introduced by Moggi in~\cite{Moggi88}.

The type system of Selinger and Valiron's quantum lambda calculus is based on intuitionistic linear logic, where the rules of weakening and contraction are controlled in a sensitive way by an operator $``!"$ called ``of course" or ``exponential". This operator creates a bridge between two different kinds of computation.
More precisely, a value of a general type $A$ can only be used once,
 whereas a value of type $!A$ can be copied and used multiple
 times. The impossibility of copying quantum information is one of the
 fundamental differences between quantum information and classical
 information, and is known as the {\em no-cloning property}.
 From the logical perspective, it therefore seems natural to relate quantum computation and linear logic. Note that the operator
``$!$'' satisfies the properties of a {\em comonad}.

Since we have higher-order functions, as well as probabilistic operations
(namely quantum measurement), the language needs to address the question of evaluation strategies. Otherwise, in some concrete situation, it would be impossible to give a coherent outcome every time for identical circumstances. In order to deal with this issue, Selinger and Valiron chose to incorporate a methodology \textit{\`a la} Moggi by making the distinction between values and computations. Moggi \cite{Moggi88} proposed the notion of a monad as an appropriate tool for interpreting computational behavior. At the level of the denotational model, this will be reflected by a strong monad.

To summarize, let us say that the exponential operator $!$ will be modelled by a monoidal comonad arising from an adjunction between a cartesian category (accounting for classical duplicability) and a symmetric monoidal category (accounting for quantum non-duplicability) while the manipulation of the probabilistic aspect of the quantum computation is handled by a monoidal monad. The result of combining these two methodologies is what Selinger and Valiron call a {\em linear category for duplication}.

This is not the first time that this interaction between a monad and a comonad has been invoked in order to express denotational aspects of a system in computer science (see \cite{BENTON 1996} for example). But what is new in Selinger and Valiron's work, is putting this interaction in the context of quantum computation.

 In this thesis, we will focus exclusively on the categorical aspects
 of the model construction. Thus, we will not review the syntax of the
 quantum lambda calculus itself. Instead, we will take as our starting
 point Selinger and Valiron's definition of a {\em categorical model
 of the quantum lambda calculus}~\cite{SelVal 2009}. It was already proven in~\cite{SelVal 2009} that the quantum lambda calculus forms an internal language for the class of such models. This is similar to the well-known interplay
 between typed lambda calculus and cartesian closed categories~\cite{LambekScott86}. What was left open in~\cite{SelVal 2009} was the construction of a concrete
 such model (other than that given by the syntax itself). This is the
 question we answer here.

The use of category theory to model and to explain formal languages has an established tradition in logic, but in quantum computation it constitutes a relatively recent trend. We finish this introduction by stressing that the field of quantum computation in connection with category theory is fast-growing.
 The ability to create bridges among these different branches of mathematics that are apparently far from one another is one of the motivating goals of this thesis and we hope to contribute in this direction.

\section{Definition of a categorical model for quantum lambda calculus}\label{CAT MODEL QLC}

In the introduction we informally described the main ideas and motivation of what should be a categorical model for quantum lambda calculus. Here we shall take the formal definition in~\cite{SelVal 2009} as our starting point. However, before presenting it, we will give some preliminary definitions and we shall make some remarks about how to simplify its presentation. Several of the definitions sketched here will be made more precise in Section~\ref{CAT MODELS OF LL} and
beyond.\\
Let $(\cC,\otimes,I,\alpha,\rho,\lambda,\sigma)$ be a symmetric monoidal category.
\begin{definition}
\rm
A symmetric \textit{monoidal comonad} $(!,\delta,\varepsilon,m_{A,B},m_I)$ is a comonad $(!,\delta,\varepsilon)$ where the functor $!$ is a monoidal functor $(!,m_{A,B},m_I)$, i.e., with natural transformations $m_{A,B}:{!}A\otimes{!}B\to{!}(A\otimes B)$ and $m_I:I\to{!}I$ satisfying the coherence axioms
of Definition~\ref{MONOIDAL FUNCTOR}, such that  $\delta$ and $\varepsilon$ are symmetric monoidal natural transformations.
\end{definition}
\begin{definition}
\label{LINEAR EXPONENTIAL COMONAD}
\rm
A \textit{linear exponential comonad} is a symmetric monoidal comonad $(!,\delta,\varepsilon,m_{A,B},m_I)$ in which the following conditions hold:
\begin{itemize}
\item[-]for every $A\in \cC$ there exists a commutative comonoid, with
$d_A:!A\rightarrow !A\otimes !A$ and $e_A:!A\rightarrow I$ as associated maps,
\item[-] $d_A$ and $e_A$ are monoidal natural transformation with respect to the natural transformations $m$,
\item[-]$d_A$ and $e_A$ are coalgebra morphisms when we consider $(!A,\delta_A)$, $(!A\otimes !A,m_{!A,!A}\circ (\delta_A\otimes\delta_A))$, and $(I,m_I)$ as coalgebras,
\item[-] the maps $\delta_A:(!A,e_A,d_A)\rightarrow (!!A,e_{!A},d_{!A})$ are comonoid morphisms.
\end{itemize}
\end{definition}
\begin{definition}
\rm
Let $(T,\eta,\mu)$ be a strong monad. We say that $\cC$ has \textit{Kleisli exponentials} if there exists a functor $[-,-]_{k}:\cC^{op}\times \cC\rightarrow \cC$ and a natural isomorphism:
$$\cC(A\otimes B,TC)\cong\cC(A,[B,C]_{k}))$$
\end{definition}
\begin{remark}
When the category $(\cC,\otimes,[-,-])$ is a monoidal closed category then it certainly has Kleisli exponentials just by putting $[B,C]_{k}=[B,TC]$.
\end{remark}

\begin{definition}[Linear category for duplication~\cite{SelVal 2009}]
\label{LINEAR CATEGORY FOR DUPLICATION}
\rm
A \textit{linear category for duplication} consists of a symmetric monoidal category $(\cC,\otimes,I)$ satisfying the following data:
\begin{itemize}
\item[-]an idempotent, strongly monoidal, linear exponential comonad $(!,\delta,\varepsilon,d,e)$,
\item[-]a strong monad $(T,\mu,\eta,t)$,
\item[-]$\cC$ has Kleisli exponentials.
\end{itemize}
\end{definition}
Further, if the unit $I$ is a terminal object we shall speak of an {\em affine linear category for duplication}, cf. Definition~\ref{AFFINE CATEGORY}.
\begin{remark}
\label{ADJUNCTION PRESENTATION}
\rm
The definition of a linear category for duplication (Definition~\ref{LINEAR CATEGORY FOR DUPLICATION}) is equivalent to the existence
 of a pair of monoidal adjunctions (\cite{BENTON 1994},~\cite{MELLIES} and~\cite{KOCK72}):
\[\xymatrix{
(\mathcal{B},\times,1)\ar@<1ex>[rr]^{(L,l)}&&(\mathcal{C},\otimes,I)\ar@<1ex>[rr]^{(F,m)}\ar@<1ex>[ll]^{(I,i)}_{\bot}&&(\mathcal{D},\otimes,I)\ar@<1ex>[ll]^{(G,n)}_{\bot}}\]
where the category $\cB$ has finite products and $\cC$ and $\cD$ are symmetric monoidal closed categories.
The monoidal adjoint pair of functors on the left represents a linear-non-linear model in the sense of Benton~\cite{BENTON 1994} in which we obtain a monoidal comonad by $!=L\circ I$. The monoidal adjoint on the right gives rise to $T=G\circ F$ a strong monad in the sense of Kock~\cite{KOCK70},~\cite{KOCK72} which is also a computational monad in the sense of Moggi~\cite{Moggi88}.
\end{remark}
We now state the main definition of a model of the quantum lambda calculus.

\begin{definition}[Model of the quantum lambda calculus~\cite{SelVal 2009}]
\label{DEF MODEL OF QUANTUM LAMBDA CALCULUS}
\rm
An \textit{abstract model of the quantum lambda calculus} is an affine linear category for duplication $\cC$ with finite coproducts, preserved by the comonad $!$. Moreover, a {\em concrete model of the quantum lambda calculus} is an abstract model of the
quantum lambda calculus such  that there exists a full and faithful embedding $\textbf{Q}\hookrightarrow \cC_{T}$, preserving tensor $\otimes$ and coproduct $\oplus$, from the category $\textbf{Q}$ of norm non-increasing superoperators (see Definition~\ref{Q OPLUS CATEGORY}) into the Kleisli category generated by the monad $T$.
\end{definition}
\begin{remark} To make the connection to quantum lambda calculus: the
 category $\cC$, the Kleisli category $\cC_T$, and the co-Kleisli
 category $\cC_{!}$ all have the same objects, which correspond to
 {\em types} of the quantum lambda calculus. The morphism $f:A\to B$
 of $\cC$ correspond to {\em values} of type $B$ (parameterized by
 variables of type $A$). A morphism $f:A\to B$ in $\cC_T$, which is
 really a morphism $f:A\to TB$ in $\cC$, corresponds to a {\em
 computation} of type $B$ (roughly, a probability distribution of
 values). Finally, a morphism $f:A\to B$ in $\cC_{!}$, which is really
 a morphism $f:{!A}\to B$ in $\cC$, corresponds to a {\em classical
 value} of type $B$, i.e., one which only depends on classical
 variables. The idempotence of ``$!$'' implies that morphisms ${!A}\to
 B$ are in one-to-one correspondence with morphisms ${!A}\to{!B}$,
 i.e., classical values are duplicable. For details, see
 ~\cite{SelVal 2009}.
 \end{remark}

\section{Outline of the procedure for obtaining the model}\label{OUTLINE OF THE PROCEDURE}

 Our complete process for obtaining a categorical model of the quantum
 lambda calculus consists of two stages. In the first stage, we will
 construct {\em abstract} models of the quantum lambda calculus by
 applying a certain presheaf construction to suitable sequences of
 functors $\cB\rightarrow\cC\rightarrow\cD.$
 This construction is very general, and the
 base categories $\cB$, $\cC$, and $\cD$ can be viewed as parameters. We will
 identify the precise conditions required of the base categories (and
 the functors connecting them) in order to obtain a valid abstract
 model. This is the content of Chapter~\ref{PREASHEAVES MODELS OF Q C}.

In the second stage, we will construct a {\em concrete} model of the
 quantum lambda calculus by identifying particular base categories so
 that the remaining conditions of Definition~\ref{DEF MODEL OF QUANTUM LAMBDA CALCULUS} are satisfied. This
 is the content of Chapter~\ref{A CONCRETE MODEL}.

 We briefly outline the main steps of the construction; full details
 will be given in later sections.

\begin{itemize}
\item The basic idea of the construction is to lift a sequence of functors
      $$\cB\stackrel{\Phi}\rightarrow\cC\stackrel{\Psi}\rightarrow\cD$$ into a pair of adjunctions between presheaf categories

      \[\xymatrix{
[\cB^{op},{\bf Set}]\ar@<1ex>[rr]^{L}&&[\cC^{op},{\bf Set}]\ar@<1ex>[rr]^{F_1}\ar@<1ex>[ll]^{\Phi^{*}}_{\bot}&&[\cD^{op},{\bf Set}]
\ar@<1ex>[ll]^{\Psi^{*}}_{\bot}}\]

   Here, $\Phi^{*}$ and $\Psi^{*}$ are the precomposition functors, and $L$ and
  $ F_1$ are their left Kan extensions.  By Remark~\ref{ADJUNCTION PRESENTATION}, such a pair of
   adjunctions potentially yields a linear category for duplication,
   and therefore, with additional conditions, an abstract model of
   quantum computation. Our goal is to identify the particular
   conditions on $\cB$, $\cC$, $\cD$, $\Phi$, and $\Psi$, that make this construction
   work correctly.
  \item By Day's construction, the requirement that $[\cC^{op},{\bf Set}]$ and $[\cD^{op},{\bf Set}]$
   are monoidal closed can be achieved by requiring $\cC$ and $\cD$ to be
   monoidal. The requirement that the adjunctions $L \dashv \Phi^{*}$ and
   $F_1 \dashv \Psi^{*}$ are monoidal is directly related to the fact
   that the functors $\Psi$ and $\Phi$ are strong monoidal. More precisely,
   this implies that the left Kan extension is a strong monoidal
   functor which in turn determines the enrichment of the adjunction.
   We also note that the category $\cB$ must be cartesian.

    \item One important complication with the model, as discussed so far, is
   the following. The Yoneda embedding $Y : \cD\rightarrow [\cD^{op},{\bf Set}]$ is full
   and faithful, and by Day's result, also preserves the monoidal
   structure $\otimes$. Therefore, if one takes $\cD={\bf Q}$, all but one of
   the conditions of a concrete model (from Definition~\ref{DEF MODEL OF QUANTUM LAMBDA CALCULUS}) are
   automatically satisfied. Unfortunately, the Yoneda embedding does
   not preserve coproducts, and therefore the remaining condition of
   Definition~\ref{DEF MODEL OF QUANTUM LAMBDA CALCULUS} fails. For this reason, we modify the construction
   and use the modified presheaf category and coproduct-preserving
   Yoneda embedding from Section~\ref{THE REFLECTIVE SUBCAT SECTION}. Our adjunctions, and the
   associated Yoneda embeddings, now look like this:

      $$\xymatrix@=25pt{
[\cB^{op},{\bf Set}]\ar[rr]^{L \dashv \Phi^{*}}&& [\cC^{op},{\bf Set}]\ar[rr]^{F \dashv G}&&[\cD^{op},{\bf Set}]_{\Gamma}\\
\cB\ar[u]^{Y}\ar[rr]^{\Phi}&&\cC\ar[u]^{Y}\ar[rr]^{\Psi}&&\cD\ar[u]^{Y_{\Gamma}}
}$$

   The second pair of adjoint functors $F \dashv G$ is generated by the
   composition of two adjunctions:

   \[\xymatrix{
[\cC^{op},{\bf Set}]\ar@<1ex>[rr]^{F_1}&&[\cD^{op},{\bf Set}]\ar@<1ex>[rr]^{F_2}\ar@<1ex>[ll]^{\Psi^{*}}_{\bot}&&[\cD^{op},{\bf Set}]_{\Gamma}\ar@<1ex>[ll]^{G_2}_{\bot}}\]

   Here, the pair of functors $F_2 \dashv G_2$ arises as a reflection of
   $[\textbf{Q}^{op},{\bf Set}]_{\Gamma}$ in $[\textbf{Q}^{op},{\bf Set}]$,
   and depends on a choice of a
   certain class $\Gamma$ of cones. The structural aspects of the modified
   Yoneda embedding $\textbf{Q} \to [\textbf{Q}^{op},{\bf Set}]_{\Gamma}$ depend crucially on general
   properties of the functor categories, which go back to the study of
   continuous functors by Lambek (see Section~\ref{LAMBEK COMPLETION}) and Freyd and Kelly
   (see Section~\ref{THE REFLECTIVE SUBCAT SECTION}).

   But, as we mentioned before, at the same time we still require that
   the reflection functor remain strongly monoidal. Here will will use
   Day's results (see Section~\ref{Day s reflection theorem}) on the conditions that are needed
   for the reflection to be strong monoidal, by inducing a monoidal
   structure from the category $[\textbf{Q}^{op},{\bf Set}]$ into its subcategory $[\textbf{Q}^{op},{\bf Set}]_{\Gamma}$ (see Section~\ref{Day s reflection theorem}). In particular, this induces a
   constraint on the choice of $\Gamma$ considered above: all the cones
   considered in $\Gamma$ must be preserved by the opposite functor of
   the tensor function in $\cD$ (see Remark~\ref{GAMMA CLASS OF CONES}).

   \item Notice that the above adjunctions are examples of what in topos
   theory is named an {\em essential geometric} morphism, in which
   both functors are left adjoint to some other two functors: $L \dashv
   \Phi^{*} \dashv \Phi_{*}$. Therefore, this shows that the comonad ``$!$''
   obtained will preserve finite coproducts.

   \item The condition for the comonad $``!"$ to be idempotent turns out to
   depend on the fact that the functor $\Phi$ is full and faithful.

 \item In addition to the requirement that $``!"$ preserves coproducts, we
 also need $``!"$ to preserve the tensor, i.e., to be strongly
 monoidal, as required in Definition~\ref{DEF MODEL OF QUANTUM LAMBDA CALCULUS}.
This property is unusual for models of intuitionistic linear logic and puts some restriction on the range of possible choices we have for the category $\cC$. In brief, since the left Kan extension along $\Phi$ is a strong monoidal functor we find that a concrete condition in the category $\cC$ is necessary to ensure that this property holds when we lift the functor $\Phi$ to the category of presheaves; see Section~\ref{A STRONGLY COMONAD}.

  \item Once we have constructed this categorical environment our next task is to translate these properties to the Kleisli category. To achieve this we use the comparison Kleisli functor for passing from the framework we have already established to the Kleisli monoidal adjoint pair of functors. Also, at the same time in this context, we shall find it convenient to characterize the  functor $H:\cD\rightarrow [\cC^{op},{\bf Set}]_T$ as a strong monoidal functor.

 All of the above steps yield an abstract model of quantum
 computation, parametric on the sequence of functors $\cB\rightarrow \cC \rightarrow \cD$.

 \item Finally, as we shall see in Section~\ref{Q''CATEGORY AND FUNCTORS PSI AND PHI}, we will identify specific
   categories $\cB$, $\cC$, and $\cD$ that yield a concrete model of quantum
   computation. We let $\cD=\textbf{Q}$, the category of superoperators. The
   categories $\cB$ and $\cC$ must be chosen in such a way as to satisfy all
   of the properties outlined above. For $\cB$, we take the category of
   finite sets.

   Identifying a suitable candidate for the category $\cC$ is more
   tricky. For example, here are two of the requirements directly
   concerning the semantics: $\cC$ must be affine monoidal and must
   satisfy the condition of equation~(\ref{EQUATION MULTIPLICATIVE KERNEL}) in Section~\ref{A STRONGLY COMONAD}.

   In a series of intermediate steps, with the help of some universal
   constructions, we introduce a category $\cC={\bf Q}''$
   related to the category $\textbf{Q}$ of superoperators.

\end{itemize}

As we have noted, the category ${\bf Q}''$ plays a central role in our construction. It is in some sense the ``barycenter" of our model.
While the basic structural properties occur at the level of the functor categories, providing a general mathematical setting, the development of the concrete quantum meaning of the model occurs mostly at this base level.

\section{Categorical models of linear logic}\label{CAT MODELS OF LL}

The first definition of a categorical model of linear logic was given by Seely~\cite{SEELY}. Other pioneering studies in this area were Lafont's thesis~\cite{LAFONT} and Abramsky's paper~\cite{ABRAMSKY93}. Also, Melli\`{e}s' survey~\cite{MELLIES} is an excellent introduction to the topic.\\
Now we formulate Bierman's definition of linear category~\cite{BIERMAN95} which is based upon the above-mentioned previous work on the Topic. We also state an equivalent alternative simplified version that we take from Benton~\cite{BENTON 1994} (this is the notion we outlined in Remark~\ref{ADJUNCTION PRESENTATION}).
For the purpose of this thesis, since it is clear that the linear fragment of Definition~\ref{DEF MODEL OF QUANTUM LAMBDA CALCULUS} does not impose any constraints on the rest of the definition, it follows that it will be more helpful to work with Benton's version representing the underlying linear fragment. In any case, to appreciate
the details behind these categorical models, Bierman's definition will occupy the rest of the present section.
\begin{definition}[Bierman]\label{LINEAR CATEGORY BIERMAN}
A \textit{linear category} $\mathcal{C}$ consists of a symmetric monoidal closed
category $(C,I,\otimes,\multimap,\alpha,\lambda,\rho,\gamma)$ with
a symmetric monoidal comonad $({!},\varepsilon,\delta,m_I,m_{A,B})$ defined on $\cC$
and monoidal natural transformations $e:!(-)\rightarrow I$,
$d:{!}(-)\rightarrow {!}(-)\otimes {!}(-)$ such that:
\begin{itemize}
\item[-] $e_A:{!}(A)\rightarrow I$, $d_A:{!}(A)\rightarrow {!}(A)\otimes {!}(A)$ are coalgebra
morphisms for each $A$;
\item[-] $(({!}A,\delta _A),e_A,d_A)$ is a commutative comonoid
for every free coalgebra $({!}A,\delta _A)$ and
\item[-] morphisms between
free coalgebras $f:({!}A,\delta _A)\rightarrow ({!}B,\delta _B)$ are
also comonoid commutative morphisms.
\end{itemize}
\end{definition}

We will now consider the meaning of each of these conditions:
\begin{itemize}
\item[-]for every $A\in \cC$ there exists a commutative comonoid, with
$d_A:!A\rightarrow !A\otimes !A$ and $e_A:!A\rightarrow I$ as associated maps. This means the following:\\
The assumption that $((!A,\delta _A),e_A,d_A)$ is a
commutative comonoid for every free coalgebra $(!A,\delta _A)$
means that:
$$
\xymatrix@=40pt{ !A \ar[rr]^{d_A} \ar[d]_{d_A}
  && !A\ot !A \ar[d]^{d_A\ot 1_{!A}}\\
!A\ot !A \ar[r]_{1_{!A}\ot d_A}
  & !A\ot (!A\ot !A) \ar[r]^{\alpha}
    & (!A\ot !A)\ot !A
}$$

$$\begin{array}{cc}
 \xymatrix@=30pt{
 !A
  & !A \ar[r]^{1_{!A}} \ar[d]_{d_A} \ar[l]_{1_{!A}}
    & !A \\
 I\ot !A \ar[u]^{\lambda_{!A}}
  & !A\ot !A \ar[r]^{1_{!A}\ot e_A} \ar[l]_{e_A\ot 1_{!A}}
    & !A\ot I \ar[u]_{\rho_{!A}}
 }\hspace{2cm} & \xymatrix@=25pt{
!A \ar[r]^{d_A}\ar[rd]_{d_A}
  & !A\ot !A \ar[d]^{\gamma_{!A,!A}}\\
  & !A\ot !A
} \end{array}$$

\item[-] $d_A$ and $e_A$ are monoidal natural transformation with respect to the natural transformation $m$.

The transformations $e:!(-)\rightarrow I$ and $d:!(-)\rightarrow
!(-)\otimes (-)$ are monoidal natural transformations between
monoidal functors; if $f:A\rightarrow B$ then
$e:(!,m_{A,B},m_I)\rightarrow (I,\lambda_I,1_I)$ is the
statement that the following diagrams commute:

 $$\xymatrix@=25pt{ !A \ar[r]^{!f}\ar[rd]_{e_A}
  & !B \ar[d]^{e_B}\\
  & I
}$$

$$\begin{array}{cc}
\xymatrix@=25pt{ !A\ot !B \ar[rr]^{m_{A,B}}\ar[d]_{e_A\ot e_B}
  && !(A\ot B) \ar[d]^{e_{A\ot B}}\\
I\ot I \ar[rr]_{\lambda_I}
  && I
} & \xymatrix@=25pt{ I \ar[r]^{m_I}\ar[rd]_{id_I}
  & !I \ar[d]^{e_I}\\
  & I
} \end{array}$$

and $d:(!,m_{A,B},m_I)\rightarrow (!\ot !,t_{A,B},t_I)$:

$$
\xymatrix@=25pt{ !A \ar[rr]^{d_A}\ar[d]_{!f}
  && !A\ot !A  \ar[d]^{!f\ot !f}\\
!B \ar[rr]_{d_B}
  && !B\ot !B
}$$

$$
\xymatrix@=70pt{ !A\ot !B \ar[rr]^{m_{A,B}} \ar[d]_{d_A\ot d_B}
  && !(A\ot B) \ar[d]^{d_{A\ot B}}\\
!A\ot !A\ot !B\ot !B \ar[r]_{Id_{!A}\ot \gamma_{!A,!B}\ot Id_{!B}}
  & !A\ot !B\ot !A\ot !B  \ar[r]^{m_{A,B}\ot m_{A,B}}
    & !(A\ot B)\ot !(A\ot B)
}$$

with $t_{A,B}=(m_{A,B}\ot m_{A,B})\comp \,\, Id_{!A}\ot
\gamma_{!A,!B}\ot Id_{!B}$ and

$$\xymatrix@=25pt{
I \ar[rr]^{m_I}\ar[d]_{\lambda_I^{-1}}
  && !I \ar[d]^{d_I}\\
I\ot I \ar[rr]_{m_I\ot m_I}
  && !I\ot !I
}$$

with $t_I=(m_I\ot m_I)\comp \,\, \lambda_I^{-1}$.

\item[-]$d_A$ and $e_A$ are coalgebra morphisms when we consider $(!A,\delta_A)$, $(!A\otimes !A,m_{!A,!A}\circ (\delta_A\otimes\delta_A))$, and $(I,m_I)$ as coalgebras:

The definition of linear category characterizes
$e_A:(!A,\delta_A)\rightarrow (I,m_I)$ and
$d_A:(!A,\delta_A)\rightarrow (!A\ot !A, m_{!A,!A}\circ
(\delta_A\ot \delta_A ))$ as coalgebra morphisms which means
that the following diagrams commute:

$$\begin{array}{cc}
 \xymatrix@=25pt{ !A \ar[rr]^{\delta_A}\ar[d]_{e_A}
  && !!A \ar[d]^{!e_A}\\
I \ar[rr]_{m_I}
  && !I
}\hspace{2cm} & \xymatrix@=35pt{ !A \ar[rr]^{\delta_A}\ar[d]_{d_A}
  && !!A \ar[d]^{!d_A}\\
!A\ot !A \ar[rr]_{m_{!A,!A}\circ (\delta_A\ot \delta_A)}
  && !(!A\ot !A)
} \end{array}$$

\item[-]
Morphisms between free coalgebras $f:(!A,\delta _A)\rightarrow
(!B,\delta _B)$ are also comonoid commutative morphisms. This
means that if $f:!A\rightarrow !B$ is an arrow with $!f\circ
\delta_A=\delta_B \circ f$ then is also true that
$f:(!A,d_A,e_A)\rightarrow (!B,d_B,e_B)$ is a map between
commutative comonoids that is $f:!A\rightarrow !B$ is an arrow
that satisfies:

$$\begin{array}{cc}
\xymatrix@=25pt{ !A \ar[rr]^{d_A}\ar[d]_{f}
  && !A\ot !A  \ar[d]^{f\ot f}\\
!B \ar[rr]_{d_B}
  && !B\ot !B
} \hspace{2cm} & \xymatrix@=25pt{ !A \ar[r]^{f}\ar[rd]_{e_A}
  & !B \ar[d]^{e_B}\\
  & I
} \end{array}$$

\end{itemize}

To complete the list of conditions let us show the structural conditions.
The natural transformations $\varepsilon :!(-)\rightarrow I$ and
$\delta : !(-)\rightarrow !!(-)$ are monoidal. If
$(!,m_{A,B},m_I)$ and $(Id,1_{A\otimes B},1_I)$ are monoidal
functors then $\varepsilon:(!,m_{A,B},m_I)\rightarrow
(Id,1_{A\otimes B},1_I)$ is a monoidal natural transformation
which is compatible in the sense that the following diagrams
commute:
$$\begin{array}{cc}
\xymatrix@=25pt{
!A\otimes !B \ar[rr]^{m_{A,B}}\ar[d]_{\varepsilon_A\ot
\varepsilon_B}
  && !(A\otimes B) \ar[d]^{\varepsilon_{A\ot B}}\\
A\ot B \ar[rr]_{1}
  && A\ot B
} \hspace{2cm}&
\xymatrix@=25pt{
I \ar[r]^{m_I}\ar[rd]_{1}
  & !I \ar[d]^{\varepsilon_I}\\
  & I
}\end{array}$$

Also $\delta :(!,m_{A,B},m_I)\rightarrow (!!,t_{A,B},t_I)$ is a
monoidal natural transformation between monoidal functors; with
$t_{A,B}=!(m_{A,B})\comp\, m_{!A,!B}$ and $t_I=!(m_I)\comp\, m_I$:

$$
\xymatrix@=40pt{ !A\ot !B \ar[rr]^{m_{A,B}} \ar[d]_{\delta _A \ot
\delta_B}
  && !(A\ot B) \ar[d]^{\delta _{A\ot
B}}\\
!!A\ot !!B \ar[r]_{m_{!A,!B}}
  & !(!A\ot !B) \ar[r]^{!(m_{A,B})}
    & !!(A\ot B)
}$$

$$\xymatrix@=35pt{
I \ar[r]^{m_I}\ar[rd]_{t_I}
  & !I \ar[d]^{\delta _I}\\
  & !!I
}$$

Recalling that a symmetric monoidal comonad $(!,\varepsilon ,\delta
,m_{A,B},m_I)$ is a comonad $(!,\varepsilon ,\delta )$ equipped
with a symmetrical monoidal functor $(!,m_{A,B},m_I)$, where:
$!:C\rightarrow C$ is a functor, for every object $A$ and $B$
there is a morphism $m_{A,B}:!A\otimes !B\rightarrow !(A\otimes
B)$ natural in $A$ and $B$, for the unit $I$ there is a morphism
$m_I :I\rightarrow !I$.

 These morphisms with the structural maps
$\alpha ,\lambda ,\rho ,\gamma $ must make the following diagrams
commute:
$$
\xymatrix@=40pt{ !A\otimes (!B\otimes !C) \ar[r]^{id_{!A}\otimes
m_{B,C}} \ar[d]_{\alpha}
  & !A\otimes !(B\otimes C) \ar[r]^{m_{A,B\otimes C}}
    & !(A\otimes (B\otimes C) \ar[d]^{!\alpha}\\
(!A\otimes !B)\otimes !C \ar[r]_{m_{A,B}\otimes id_{!C}}
  & !(A\otimes B)\otimes !C \ar[r]^{m_{A\otimes B,C}}
    & !((A\otimes B)\otimes C)
}$$

$$\begin{array}{cc}
\xymatrix@=25pt{
!B\otimes I \ar[rr]^{\rho_{!B}}\ar[d]_{id_{!B}\otimes m_I}  && !B \\
!B\otimes !I \ar[rr]_{m_{B,I}}  && !(B\otimes I)
\ar[u]^{!(\rho_B)}} & \xymatrix@=25pt{
I\otimes !B \ar[rr]^{\lambda_{!B}}\ar[d]_{m_I\otimes id_{!B}}  && !B \\
!I\otimes !B \ar[rr]_{m_{I,B}}  && !(I\otimes B)
\ar[u]^{!(\lambda_B)}} \end{array}$$

$$\xymatrix@=25pt{
!A\otimes !B \ar[rr]^{\gamma_{!A,!B}}\ar[d]_{m_{A,B}}
  && !B\otimes !A \ar[d]^{m_{B,A}}\\
!(A\otimes B) \ar[rr]_{!(\gamma_{A,B})}
  && !(B\otimes A)
}$$

\begin{definition}[Benton~\cite{BENTON 1994}]
\label{DEF L-N-L BENTON MODEL}
A \textit{linear-non-linear category} consists of:
\begin{itemize}
\item[(1)] a symmetric monoidal closed category $(\mathcal{C},\otimes ,I, \multimap )$
\item[(2)] a category $(\mathcal{B},\times,1)$ with finite product
\item[(3)] a symmetric monoidal adjunction:
\[\xymatrix{
(\mathcal{B},\times,1)\ar@<1ex>[rr]^{(F,m)}&&(\mathcal{C},\otimes,I)\ar@<1ex>[ll]^{(G,n)}_{\bot}}.\]
\end{itemize}
\end{definition}

Note that Definition~\ref{DEF L-N-L BENTON MODEL} is far simpler than Definition~\ref{LINEAR CATEGORY BIERMAN}. Its significance is in the following:

\begin{proposition}
Every linear-non-linear category gives rise to a linear category. Every linear category defines a linear-non-linear category, where $(\cB,\times,1)$ is the category of coalgebras of the comonad $(!,\varepsilon, \delta)$.
\end{proposition}
\begin{proof}
See~\cite{BENTON 1994} or ~\cite{MELLIES}.
\end{proof}

\begin{remark}
Kelly's characterization of monoidal adjunctions (see Proposition~\ref{KELLY STRONG-MONOIDAL ADJ})
   allows us to replace condition $(3)$ in the definition of linear-non-linear categories by the following new statement in Definition~\ref{DEF L-N-L BENTON MODEL}:
\begin{itemize}
\item[(3')] an adjunction:
\[\xymatrix{
(\mathcal{B},\times,1)\ar@<1ex>[rr]^{F}&&(\mathcal{C},\otimes,I)\ar@<1ex>[ll]^{G}_{\bot}}\]
and there exist isomorphisms
$$m_{A,B}:FA\otimes FB\rightarrow F(A\times B)\,\,\,,\,\,\,  m_I:I\rightarrow F(1)$$
making $(F,m_{A,B},m_I):(\mathcal{B},\times,1)\rightarrow (\mathcal{C},\otimes ,I) $ a strong symmetric monoidal functor.
\end{itemize}
\end{remark}

\section{Linear-non-linear models on presheaf categories}\label{abstract3}
Our purpose here is to characterize Benton's linear-non-linear models
 of intuitionistic linear logic, in the sense of Definition~\ref{DEF L-N-L BENTON MODEL}, on
 presheaf categories using Day's monoidal structure from Section~\ref{DAY S CLOSED MONO CONVO}. This is an application of monoidal enrichment of the Kan
 extension see~\cite{B.Day-Street95}. We use Kelly's equivalent
 formulation of monoidal adjunctions from Proposition~\ref{KELLY STRONG-MONOIDAL ADJ}.

\begin{proposition}
Suppose we have a strong monoidal functor $\Phi:(\mathcal{A},\times,1)\rightarrow(\mathcal{B},\otimes,I)$ from a cartesian category to a monoidal category, i.e., we have a natural isomorphism $\Phi(a)\otimes\Phi(b)\cong\Phi(a\times b)$ and $I\cong \Phi(1)$.

Let us consider the left Kan extension along $\Phi$ in the functor
category $[\cB^{op},{\bf Set}]$ where the copower is product on sets:
$$Lan_{\Phi}(F)=\int^{a}\mathcal{B}(-,\Phi(a))\times F(a)$$
Then $Lan_{\Phi}$ is strong monoidal.
\end{proposition}
\begin{proof}
By the Yoneda Lemma, the strong functor $\Phi$, Fubini and coend properties:\\
$Lan_{\Phi}(F\times G)=Lan_{\Phi}(\int^a\mathcal{A}(-,a)\times F(a)\times \int^b\mathcal{A}(-,b)\times G(b))$ by the Yoneda Lemma and pointwise product

$=Lan_{\Phi}(\int^{a\,\, b}\mathcal{A}(-,a)\times\mathcal{A}(-,b)\times F(a)\times G(b))=
Lan_{\Phi}(\int^{a\,\, b}\mathcal{A}(-,a\times b)\times F(a)\times G(b))$ cartesian product

$=\int^c\mathcal{B}(-,\Phi(c))\times\int^{a\,\, b}\mathcal{A}(c,a\times b)\times F(a)\times G(b)=$
$\int^{a\,\, b}(\int^c\mathcal{B}(-,\Phi(c))\times\mathcal{A}(c,a\times b))\times F(a)\times G(b)$ definition of Kan extension

$=\int^{a\,\, b}\mathcal{B}(-,\Phi(a\times b))\times F(a)\times G(b)=
\int^{a\,\, b}\mathcal{B}(-,\Phi(a)\otimes \Phi(b))\times F(a)\times G(b)
$ $\Phi$ strong functor

$=\int^{a\,\, b}(\int^y\mathcal{B}(y,\Phi(a))\times\mathcal{B}(-,y\otimes \Phi(b)))\times F(a)\times G(b)=
\int^{a\,\, b}\int^y\mathcal{B}(y,\Phi(a))\times(\int^z\mathcal{B}(-,y\otimes z)\times\mathcal{B}(z, \Phi(b)))\times F(a)\times G(b)
$ by the Yoneda Lemma

$=
\int^{y\,\, z}(\int^a\mathcal{B}(y,\Phi(a))\times F(a))\times (\int^b\mathcal{B}(z,\Phi(b))\times G(b))\times \mathcal{B}(-,y\otimes z)=
\int^{y\,\, z}((Lan_{\Phi}(F))(y))\times ((Lan_{\Phi}(G))(z))\times \mathcal{B}(-,y\otimes z)$ by Fubini and copower preserves colimits

$=Lan_{\Phi}(F)\otimes_{D}Lan_{\Phi}(G)$ by definition of Kan extension and convolution\\
and also the units:\\
$Lan_{\Phi}(I_D^{\mathcal{A}})=Lan_{\Phi}(\mathcal{A}(-,1))=\int^a\mathcal{B}(-,\Phi(a))\times\mathcal{A}(a,1)=\mathcal{B}(-,\Phi(1))=
\mathcal{B}(-,I)=I_D^{\mathcal{B}}$.
\end{proof}
\begin{remark}
\rm
Note that, in view of the line of arguments used above, the case where $\cA$ is monoidal has the same proof, i.e., if we have
$\Phi(a)\otimes\Phi(b)\cong\Phi(a\otimes b)$
and $I\cong \Phi(I')$
we start directly from the convolution product:
$$Lan_{\Phi}(F\otimes_D G)=Lan_{\Phi}(\int^{a\,\, b}\mathcal{A}(-,a\otimes b)\times F(a)\times G(b))$$
and we repeat the same proof.
Also notice that when we have a product in $\cA$ the convolution is a pointwise product of functors:
$$F\times G=\int^{a\,\, b}\mathcal{A}(-,a\times b)\times F(a)\times G(b)).$$
\end{remark}

\begin{remark}\rm
If the unit of a monoidal category $\cC$ is a terminal object then the unit of the convolution is also terminal.
Let us consider a morphism $\alpha:F\rightarrow \cC(-,I)$ in the functor
category $[\cC^{op},{\bf Set}]$.
Then for every $V$ there is only one way to define the map $\alpha_V:F(V)\rightarrow \cC(V,I)$ which is $\alpha_V(x)={!}$ for every $x\in F(V)$ in the category of sets. Hence there is a unique $\alpha$. Therefore it is a terminal object in the functor
category.
\end{remark}

\section{Idempotent comonad in the functor category}\label{abstract2}
A comonad $({!},\epsilon,\delta)$ is said to be {\em idempotent} if $\delta:{!}\Rightarrow {!!}$ is an isomorphism.
Let $(!,\epsilon,\delta)$ be the comonad generated by the adjunction:

\[\xymatrix{
(\mathcal{D},\times ,1)\ar@<1ex>[r]^{F}& (\mathcal{V},\otimes,I,\multimap) \ar@<1ex>[l]^{G}_{\bot}}\]
then $\delta=F\eta_G$ with $\eta:A\rightarrow GFA$.
Thus if $\eta$ is an isomorphism then $\delta$ is also an isomorphism.
Now consider the unit of the Kan extension:
$$G\Rightarrow F^{*}(Lan_F(G)).$$
It is given by:

$$\xymatrix@=25pt{
G(a) \ar[rr]^{i_{1_{F(a)}}}\ar[rrdd]_{(\eta_G)(a)}
 & & \mathcal{B}(F(a),F(a))\times G(a) \ar[dd]^{(w_a)_{F(a)}}\\
 & &\\
&  & \int^{a'}\mathcal{B}(F(a'),F(a))\times G(a')
},$$
where $i$ is the injection of the copower and $w$ is the wedge of the coend.

\begin{proposition}\label{F FULL AND FAITHFUL ETA ISO}
If $F$ is a full and faithful functor then $\eta_G:G\Rightarrow F^{*}(Lan_F(G))$ is an isomorphism.
\end{proposition}
\begin{proof}
\cite{Borceux94}
\end{proof}

\section{A strong comonad}
\label{A STRONGLY COMONAD}
In this section we study conditions that allow us to force the idempotent comonad to be a strong monoidal functor.
This property, part of the model we are building, is a main difference with other previously intuitionistic linear models.
In order to achieve this, consider a full and faithful functor
$\Phi:\cA\to\cB$ as in Proposition~\ref{F FULL AND FAITHFUL ETA ISO}. Let $\Phi^*$ be the functor we had seen earlier in Section~\ref{KAN EXTENSION}: $$[\cB^{op},{\bf Set}]\stackrel{\Phi^{*}}\longrightarrow[\cA^{op},{\bf Set}],$$
i.e., the right adjoint of the left Kan extension.
\begin{lemma}\label{MULTIPLICATIVE KERNEL CONDITIONS}
If there exists a natural isomorphism:
\begin{equation}
\label{EQUATION MULTIPLICATIVE KERNEL}
\cB(\Phi(a),b)\times\cB(\Phi(a),b')\cong\cB(\Phi(a),b\otimes b')
\end{equation}
where $a\in \cA$ and $b,b'\in\cB$ and $\Phi$ is a fully faithful, strong monoidal functor
then $\Phi^{*}$ is a strong monoidal functor.
\end{lemma}
\begin{proof}
To see this:
$\Phi^{*}(F)\times\Phi^{*}(G)=F(\Phi(-))\times G(\Phi(-))\cong\int^{b}F(b)\times\cB(\Phi(-),b)\times \int^{b'}G(b')\times\cB(\Phi(-),b')\cong$ by the Yoneda Lemma, definition of $\Phi^{*}$ and the fact that convolution in $[\cA^{op},{\bf Set}]$ is pointwise cartesian product

$\cong\int^{b\,b'}F(b)\times G(b')\times\cB(\Phi(-),b)\times\cB(\Phi(-),b')\cong\,\,\,\,\,\,\,$ by properties of coends (preservation)

$\cong\int^{b\,b'}F(b)\times G(b')\times\cB(\Phi(-),b\otimes b')=\,\,\,\,\,\,\,\,$ by hypothesis~(\ref{EQUATION MULTIPLICATIVE KERNEL}) and Lemma~\ref{PROPERTIES OF COEND 1}

$=(F\otimes G)(\Phi(-))=\Phi^{*}(F\otimes G)\,\,\,\,\,\,\,$ by definition of convolution in $[\cB^{op},{\bf Set}]$ and definition of $\Phi^{*}$.

Moreover the units are isomorphic,\\
$\Phi^{*}(\cB(-,I))=\cB(\Phi(-),I)\cong\,\,\,\,\,\,\,$ by definition of $\Phi^{*}$

$\cong\cB(\Phi(-),\Phi(1))\,\,\,\,\,\,$  since $\Phi$ is strong

$\cA(-,1)\,\,\,\,\,\,$ since $\Phi$ is fully faithful.
\end{proof}

\begin{remark}
\rm
At this point it is useful to mention that the conditions of Lemma~\ref{MULTIPLICATIVE KERNEL CONDITIONS} are an example of a \textit{multiplicative kernel} $K:\cB\times\cA^{op}\rightarrow {\bf Set}$ from the monoidal category $\cB$ to $\cA$ in the sense of \cite{B.Day06}. In fact, $K$ is explicitly defined as $K(b,a)=\cB(\Phi(a),b)$ satisfying the two following equations as part of the definition:
$$\int^{yz}K(a,y)\times K(b,z)\times \cA(x,y\times z)\cong \int^c K(c,x)\times\cB(c,a\otimes b)$$
$$\int^b\cB(\Phi(x),b)\times \cB(b,I)\cong \cA(x,1)$$
\end{remark}
In Section~\ref{Q double prime AND CATEGORY AND FUNCTOR PSI} we shall built a category satisfying this specific requirement among others. More precisely, from our viewpoint this will depend on the construction of a certain category that we will name $\textbf{Q}''$ which is a modification of the category $\textbf{Q}$ of superoperators. Also we consider the functor $\Phi$ of Section~\ref{THE FUNCTOR PHI FROM SETS} where $\cC^{+}=\textbf{Q}''$.

\section{If $\cC$ has finite coproducts then $\cC_{T}$ has finite coproducts}
\label{COPRODUCT INDUCED KLEISLI}
An important property of the Kleisli construction is that if we assume that the original category has finite coproducts then we can define finite coproducts in the Kleisli category.
\begin{proposition}Kleisli categories inherit
coproducts, i.e., if $\cC$ has finite coproducts then $\cC_{T}$ also has finite coproducts.
\end{proposition}
\begin{proof}
Suppose we have that $A\stackrel{f^{K}}\rightarrow C$ and $B\stackrel{g^{K}}\rightarrow C$ two arrows in the category $\cC_{T}$. We take $A\oplus_{K}B=A\oplus B$ on objects, and

$$\xymatrix@=25pt{
A\ar[dr]_{i_{A}}\ar[rr]^{i^{T}_{A}}&&T(A\oplus B)&&&B\ar[dr]_{i_B}\ar[rr]^{i^{T}_{B}}&&T(A\oplus B)\\
&A\oplus B\ar[ur]_{\eta_{A\oplus B}}&&&&&A\oplus B\ar[ur]_{\eta_{A\oplus B}}&\\
}$$
as injections in the category $\cC_{T}$.

We want to find a unique $A\oplus_{K}B\stackrel{[f^{K},g^{K}]_{K}}\longrightarrow C$ such that $f^{K}=[f^{K},g^{K}]_{K}\circ_{K} i^{T}_{A}$ and $g^{K}=[f^{K},g^{K}]_{K}\circ_{K} i^{T}_{B}$ commute.
This is verified by the following diagram:
$$\xymatrix@=25pt{
A\ar[rrd]_{f}\ar[rr]^{i_{A}}&&A\oplus B\ar[d]^{[f,g]}\ar[rr]^{\eta_{A\oplus B}}&&T(A\oplus B)\ar[d]^{T[f,g]}&&A\oplus B\ar[d]^{[f,g]}\ar[ll]_{\eta_{A\oplus B}}&&B\ar[dll]^{g}\ar[ll]_{i_{B}}&\\
&&TC\ar[drr]_{1_{TC}}\ar[rr]^{\eta_{TC}}&&T^{2}C\ar[d]^{\mu_{C}}&&TC\ar[ll]_{\eta_{TC}}\ar[lld]^{1_{TC}}&&&\\
&&&&TC&&&\\
}$$
where $[f,g]$ is the unique morphism that defines coproduct in $\cC$. This last diagram commutes by naturality of $\eta$ with respect to $[f,g]$ and by definition of monad.\\
Uniqueness follows from the following reasoning: suppose there is an arrow $A\oplus_{K}B\stackrel{h_{K}}\longrightarrow C$, i.e.,
$A\oplus B\stackrel{h}\longrightarrow TC$, such that $\mu_{C}\circ T(h)\circ \eta_{A\oplus B}\circ i_A=f$ and $\mu_{C}\circ T(h)\circ \eta_{A\oplus B}\circ i_B=g$ then by naturality and monad definition we have that $h\circ i_A=f$ and $h\circ i_A=g$, thus by uniqueness in $\cC$ we have that $h=[f,g]$.
\end{proof}

We notice that $C:\cC_{T}\rightarrow\cD$ preserves finite coproducts. To see this, by definition we have that $i^{T}_{A}=\eta_{A\oplus B}\circ i_A$ and $i^{T}_{B}=\eta_{A\oplus B}\circ i_B$. Then
$$C(i^{T}_{A})=c(\eta_{A\oplus B}\circ i_A)=\varepsilon_{F(A\oplus B)}\circ F(\eta_{A\oplus B}\circ i_A)=$$
$$=\varepsilon_{F(A\oplus B)}\circ F(\eta_{A\oplus B})\circ F(i_A)=1_{A\oplus B}\circ F(i_A)=F(i_A).$$
In the same way $C(i^{T}_{B})=F(i_B)$. \\
Given that right adjoint preserves coproducts then  $C(A\oplus_{K}B)=C(A\oplus_{\cC}B)=F(A\oplus_{K}B)=FA\oplus_{\cD} FB$ and
$$C(A\stackrel{i^{T}_{A}}\longrightarrow A\oplus_{K} B)=C(A)\stackrel{C(i^{T}_{A})}\longrightarrow C(A\oplus_{K} B)=F(A)\stackrel{F(i_{A})}\longrightarrow F(A)\oplus_{\cD} F(B)$$
which is a coproduct in $\cD$.\\
In the same way we can apply a similar reasoning with $B$.

\section{The functor $H:\cD\rightarrow \hat{\cC}_{T}$}
\subsection{Preliminaries}
Let $\cC$ and $\cD$ be categories, and let $\hat\cC$ and $T=G\circ F$ be defined as
 in Section~\ref{OUTLINE OF THE PROCEDURE}.

 \[\xymatrix{
[\cC^{op},{\bf Set}]\ar@<1ex>[rr]^{F}&&[\cD^{op},{\bf Set}]_{\Gamma}\ar@<1ex>[ll]^{G}_{\bot}}\]

In this section we consider the construction of a coproduct and tensor preserving functor $H:\cD\rightarrow \hat{\cC}_{T}$ with properties similar to the Yoneda embedding. We investigate the role of a general category $\cD$ fully embedded into a Kleisli category $\hat{\cC}_{T}$. Certain properties of this functor are introduced in order to apply this to the category of superoperators $\textbf{Q}$ as well as to develop a methodology for obtaining higher-order models
in the sense of Section~\ref{CAT MODEL QLC}.

Let  $F_{1}\dashv G_{1}$ and $F_{2}\dashv G_{2}$ be two monoidal adjoint pairs with associated natural transformations $(F_{1},m_{1})$, $(G_{1},n_{1})$ and $(F_{2},m_{2})$, $(G_{2},n_{2})$.
We shall use the following notation $F=F_2\circ F_1$, $G=G_1\circ G_2$, $T=G\circ F$.
We now describe a typical situation of this kind generated by a functor $\Phi:\cC\rightarrow\cD$.

Let us consider $F_1=Lan_{\Phi}$ and $G_1=\Phi^{*}$. With some co-completeness condition assumed, we can express
$F_1(A)=\int^{c}\cD(-,\Phi(c))\otimes A(c)$ and $G_1=\Phi^{*}$.

On the other hand we consider
$$\xymatrix@=25pt{
 \cD\ar[rrdd]_{Y_{\Gamma}}\ar[rr]^{Y}&&[\cD^{op},{\bf Set}]\ar@<1ex>[dd]^{F_2} \\
 &&\\
 &&[\cD^{op},{\bf Set}]_{\Gamma}\ar@<1ex>[uu]^{G_{2}}_{\vdash}
}$$

where we take $F_2=Lan_{Y}(Y_{\Gamma}):[\cD^{op},{\bf Set}]\rightarrow[\cD^{op},{\bf Set}]_{\Gamma}$, and
$Y_{\Gamma}:\cD\rightarrow [\cD^{op},{\bf Set}]_{\Gamma}$ is given by $Y_{\Gamma}(d)=\cD(-,d)$. Thus we have that
$F_2(D)=D\star Y_{\Gamma}=\int^{d}D(d)\otimes Y_{\Gamma}(d)$.

Assuming that $[\cD^{op},{\bf Set}]_{\Gamma}$ is co-complete and contains the representable presheaves then the right adjoint is given by
$$G_2(F)=[\cD^{op},{\bf Set}]_{\Gamma}(Y_{\Gamma}-,F)=[\cD^{op},{\bf Set}](Y-,F)\cong F$$
since it is a full subcategory and by the Yoneda Lemma. Therefore we consider $G_2$ as the inclusion functor up to isomorphism.

\subsection{Definition of $H$.}
\label{CANONICAL CHOICE}
We want to study the following situation:
\vspace{5pt}

$$\xymatrix@=25pt{
&\hat{\cC}\ar@<1ex>[dd]^{F_{T}}\ar@<1ex>[rr]^{F_{1}}&&\hat{\cD}\ar@<1ex>[ll]^{G_{1}}_{\bot}\ar@<1ex>[rr]^{F_{2}}&&\hat{\cD}_{\Gamma}\ar@<1ex>[ll]^{G_{2}}_{\bot} \\
 &&&&&\\
 \cC\ar@/_/[rd]_{\Phi}\ar@/^/[ruu]^{Y}&\hat{\cC}_{T}\ar@<1ex>[uu]^{G_{T}}_{\vdash}\ar[rrrruu]_{C}&&&&&\\
 &\cD\ar@{-->}[u]_{H}\ar@/_/[rrrruuu]_{Y_{\Gamma}}&&&&&
}$$
\vspace{5pt}

The goal is to determine a fully faithful functor, $H$ in this diagram, that preserves tensor and coproduct.

First, notice that the perimeter of this diagram
commutes on objects:
$$F_1(\cC(-,c))=\int^{c'}\cD(-,\Phi(c'))\otimes \cC(c',c)=\cD(-,\Phi(c))$$
When we evaluate again we obtain:
$$F_2(\cD(-,\Phi(c)))=\int^{d'}\cD(d',\Phi(c))\otimes Y_{\Gamma}(d')=Y_{\Gamma}(\Phi(c))=\cD(-,\Phi(c))$$
Summing up we have that $F(\cC(-,c))=\cD(-,\Phi(c))$ up to isomorphism.

Suppose now that $\Phi$ is onto on objects. We have that:
$$\cD(-,d)=\cD(-,\Phi(c))$$
for some $c\in\cC$, i.e.,
 we can make a choice, for every
 $d\in |\cD|$, of some $c\in |\cC|$ such that $\Phi(c)\cong d$. Let us
 call this choice a ``choice of preimages''. We can therefore define
 a map $H:|\cD|\rightarrow |\hat{\cC}_{T}|$ by $H(d)=\cC(-,c)$ on objects.

Hence, we can define a functor $H:\cD\rightarrow \hat{\cC}_{T}$ in the following way: \\
let $d\stackrel {f}\rightarrow d'$ be an arrow in the category $\cD$, then we apply $Y_{\Gamma}$ obtaining $\cD(-,d)\stackrel {Y_{\Gamma}(f)}\rightarrow \cD(-,d')$. This arrow is equal to $\cD(-,\Phi(c))\stackrel {Y_{\Gamma}(f)}\rightarrow \cD(-,\Phi(c'))$ for some $c,c'\in\cC$ and for the reason stipulated above is equal to  $F(\cC(-,c))\stackrel {Y_{\Gamma}(f)}\rightarrow F(\cC(-,c'))$. Now we use the fact that the comparison functor
$C:\hat{\cC}_{T}\rightarrow \hat{\cD}_{\Gamma}$,
$$C:\hat{\cC}_{T}(\cC(-,c),\cC(-,c'))\rightarrow \hat{\cD}_{\Gamma}(F(\cC(-,c)),F(\cC(-,c')))$$
is fully faithful, i.e, there is a unique $\gamma:\cC(-,c)\rightarrow\cC(-,c')$ such that $C(\gamma)=Y_{\Gamma}(f)$. Then we define: $H(f)=\gamma$ on morphisms and $H(d)=\cC(-,c)$ on objects, where $c$ is given by our choice of preimages.

Explicitly on arrows we have that $H:\cD\rightarrow \hat{\cC}_{T}$ is given by $H(f)=G(Y_{\Gamma}(f))\circ \eta_{\cC(-,c)}$ i.e.,
$$\cC(-,c)\stackrel {\eta_{\cC(-,c)}}\longrightarrow GF(\cC(-,c))\stackrel {G(Y_{\Gamma}(f))}\longrightarrow GF(\cC(-,c'))$$

\begin{remark}
\rm
We notice that:
$$C\circ H(d)=C(\cC(-,c))=F(\cC(-,c))=\cD(-,\Phi(c))=\cD(-,d)=Y_{\Gamma}(d)$$
Also since $\Phi(c)=d\stackrel{f}\longrightarrow d'=\Phi(c')$ then
$$F(\cC(-,c))=\cD(-,\Phi(c))\stackrel{Y_{\Gamma}(f))}\longrightarrow\cD(-,\Phi(c'))=F(\cC(-,c'))$$
Moreover,
$$C\circ H(f)=C(\,\,\,\cC(-,c)\stackrel {\eta_{\cC(-,c)}}\longrightarrow GF(\cC(-,c))\stackrel {G(Y_{\Gamma}(f))}\longrightarrow GF(\cC(-,c'))\,\,\,)=Y_{\Gamma}(f)$$
since
$$\xymatrix@=25pt{
 F(\cC(-,c))\ar[rr]^{F(\eta_{\cC(-,c)})}\ar[drr]_{1}& & FGF(\cC(-,c))\ar[d]^{\varepsilon_{F(\cC(-,c))}}\ar[rr]^{FG(Y_{\Gamma}(f))}&& FGF(\cC(-,c'))\ar[d]^{\varepsilon_{F(\cC(-,c'))}}\\
 &&F(\cC(-,c))\ar[rr]^{Y_{\Gamma}(f)} &&F(\cC(-,c'))
}$$

Thus $C\circ H=Y_{\Gamma}$.
\end{remark}

\begin{remark}
\rm
Suppose that we are in the above situation where $(F_{2},m_{2})\dashv (G_{2},n_{2})$ is a monoidal adjunction. The Yoneda embedding is a strong monoidal functor respecting  the Day's convolution monoidal structure. Then we have:

$$\xymatrix@=25pt{
&\hat{\cD}\ar@<1ex>[rr]^{(F_{2},m_{2})}&&\hat{\cD}_{\Gamma}\ar@<1ex>[ll]^{G_{2}}_{\bot} \\
&\cD\ar[u]^{(Y,y)}\ar@/_/[rru]_{Y_{\Gamma}}&&
}$$
Since the adjunction is monoidal $F_{2}$ is a strong monoidal functor. This implies that $Y_{\Gamma}$ is a strong monoidal functor by composition.

\end{remark}

\subsection{$C:\hat{\cC}_{T}\rightarrow\hat{\cD}_{\Gamma}$ is a strong monoidal functor}

We define $C(A)\otimes_{\hat{\cD}_{\Gamma}}C(B)\stackrel {u_{AB}}\longrightarrow C(A\otimes_{\cC_{T}}B)$ by the following arrow:
$F(A)\otimes_{\hat{\cD}_{\Gamma}}F(B)\stackrel {m_{AB}}\longrightarrow F(A\otimes B)$. We want to check naturality:
for every $A\stackrel {f}\rightarrow A'$, $B\stackrel {g}\longrightarrow B'$, where $f$, $g\in \hat{\cC}_{T}$

$$\xymatrix@=25pt{
C(A)\otimes_{\hat{\cD}_{\Gamma}}C(B)\ar[rr]^{u_{AB}}\ar[d]^{C(f)\otimes_{\hat{\cD}_{\Gamma}}C(g)}& &C(A\otimes_{\cC_{T}}B)\ar[d]^{C(f\otimes_{\cC_{T}}g)}\\
C(A')\otimes_{\hat{\cD}_{\Gamma}}C(B')\ar[rr]_{u_{A'B'}}& & C(A'\otimes_{\cC_{T}}B')
}.$$

This turns out to be

$$\xymatrix@=25pt{
F(A)\otimes_{\hat{\cD}_{\Gamma}}F(B)\ar[rr]^{m_{AB}}\ar[d]^{\varepsilon_{FA'}F(f)\otimes\varepsilon_{FB'}F(g)}& &F(A\otimes_{\cC}B)\ar[d]^{\varepsilon_{F(A'\otimes B')}F(G(m_{A'B'})n(f\otimes g))}\\
F(A')\otimes_{\hat{\cD}_{\Gamma}}F(B')\ar[rr]_{m_{A'B'}}& & F(A'\otimes_{\cC} B')
}$$

where $f^{K}\otimes_{\cC_{T}}g^{K}$ is equal to

$$A\otimes B\stackrel {f\otimes g}\longrightarrow GFA'\otimes_{\cC} GFB'\stackrel {n_{FA'FB'}}\longrightarrow G(FA'\otimes_{\hat{\cD}_{\Gamma}} FB')\stackrel {G(m_{A'B'})}\longrightarrow GF(A'\otimes_{\hat{\cC}} B').$$

We define $I\stackrel {u_{I}=m_{I}}\longrightarrow C(I)=F(I)$.

$$\xymatrix@=25pt{
F(A)\otimes F(B)\ar[r]^{m_{AB}}\ar[d]^{F(f)\otimes F(g)} &F(A\otimes B)\ar[r]^{F(f\otimes g)}& F(TA'\otimes TB')\ar@{}[dl]^{(a)}\ar[r]^{F(n)}&F(G(FA'\otimes FB'))\ar[d]^{F(G(m_{A'B'}))}\ar[llldd]^{\varepsilon_{FA'\otimes FB'}}\\
FGF(A')\otimes FGF(B')\ar[urr]_{m_{GFA' GFB'}}\ar[d]^{\varepsilon_{FA'}\otimes\varepsilon_{FB'}} & &&FG(F(A'\otimes B'))\ar[d]^{\varepsilon_{F(A'\otimes B')}}\ar@{}[ll]^{(b)}\\
FA'\otimes FB'\ar[rrr]_{m_{A'B'}}&&&F(A'\otimes B')
}$$
(a) commutes since $\varepsilon$ is a monoidal natural transformation of the monoidal adjunction $(F,m)\dashv (G,n)$.\\
(b) $\varepsilon$ is natural with $m$.\\

Since $m_{AB}$ and $m_{I}$ are invertible in $\hat{\cD}_{\Gamma}$ then $u_{AB}$ and $u_{I}$ are invertible. This implies that $(C,m)$ is a strong functor.

Now we want to check that

$$\xymatrix@=25pt{
 C(A)\otimes_{\hat{\cD}_{\Gamma}} I\ar[rr]^{\rho}\ar[d]^{1\otimes u_{I}}&&C(A)\\
 C(A)\otimes_{\hat{\cD}_{\Gamma}} C(I)\ar[rr]^{u_{AI}}&&C(A\otimes_{\cC_{T}} I)\ar[u]^{C(\rho^{T})}
}$$

since $A\otimes_{\cC_{T}} I\stackrel {\rho^{T}}\rightarrow A$ is by definition $A\otimes_{\cC} I\stackrel {\rho}\rightarrow A\stackrel {\eta}\rightarrow GFA$ this implies that $C(\rho^{T})$ is

$$F(A\otimes I)\stackrel {F(\rho)}\rightarrow FA\stackrel {F(\eta)}\rightarrow FGFA\stackrel {\varepsilon_{FA}}\rightarrow FA$$
i.e., $C(\rho^{T})=F(\rho)$.
Thus we obtain that

$$\xymatrix@=25pt{
 F(A)\otimes I\ar[rr]^{\rho}\ar[d]^{1\otimes m_{I}}&&F(A)\\
 F(A)\otimes F(I)\ar[rr]^{m_{AI}}&&F(A\otimes I)\ar[u]_{F(\rho)}
}$$
and this is satisfied since $(F,m)$ is a monoidal functor. The same is true for the $\lambda$ axiom.

$$\xymatrix@=25pt{
C(A)\otimes (C(A')\otimes C(A''))\ar[rr]^{\alpha}\ar[d]^{1\otimes u_{A'A''}}&&(C(A)\otimes C(A'))\otimes C(A'')\ar[d]^{u_{AA'}\otimes 1}\\
C(A)\otimes C(A'\otimes A'')\ar[d]^{u_{A, A'\otimes A''}}&&C(A\otimes A')\otimes C(A'')\ar[d]^{u_{A\otimes A', A''}}\\
C(A\otimes (A'\otimes A''))\ar[rr]_{C(\alpha)}&&C((A\otimes A')\otimes A'')
}$$
For the same reasons as above we have that $C(\alpha^{T})=F(\alpha)$,
since $\alpha^{T}=\eta_{A\otimes A',A''}\circ \alpha$ by definition.

\subsection{$H$ is a strong monoidal functor}

We want to define a natural transformation $H(A)\otimes_{\cC^{T}} H(B)\stackrel {\psi_{A,B}}\longrightarrow H(A\otimes_{\cD} B)$ that makes $H$ into a strong monoidal functor.\\

Definition of $\psi$.

We begin by recalling that $(C,u)$ and $(Y_{\Gamma},y)$ are strong monoidal functors, i.e., $u$ and $y$ are isomorphisms, and since $C$ is a fully faithful functor this allows us to define $\psi_{A,B}$ as the unique map making the following diagram commute:
$$\xymatrix@=25pt{
 Y_{\Gamma}(A)\otimes Y_{\Gamma}(B)\ar[rr]^{y_{A,B}}\ar[dr]^{u_{HA,HB}}&&Y_{\Gamma}(A\otimes B)=C\circ H(A\otimes B)\\
 &C(H(A)\otimes H(B))\ar[ru]^{C(\psi_{A,B})}&
}$$

In the same way we define $\psi_{I}$ as the unique map $\psi_{I}:I\rightarrow H(I)$ making the following diagram commute:
$$\xymatrix@=25pt{
 I\ar[rr]^{y_{I}}\ar[dr]^{u_{I}}&&Y_{\Gamma}(I)=C\circ H(I)\\
 &C(I)\ar[ru]^{C(\psi_{I})}&
}$$
i.e., since $C$ is fully faithful the unique $\psi_{I}$ such that $C(\psi_{I})=y_{I}\circ u^{-1}_{I}$.

Notice that since $\cC_{T}\stackrel{C}{\rightarrow}\cD_{\Gamma}$ is fully faithful and $u$ and $y$ are invertible maps this implies that $\phi$ is an invertible map.

We shall prove naturality of $\phi$.
\[
\xymatrix@=25pt@C-3ex{
 *{CH(A)\otimes CH(B)\makebox[0cm][l]{${}=Y_{\Gamma}(A)\otimes Y_{\Gamma}(B)$}}\ar@{}[drr]|{(a)}\ar[ddd]^{Y_{\Gamma}(f)\otimes Y_{\Gamma}(g)}_{CH(f)\otimes CH(g)=}\ar[rrrd]^{y_{A,B}}\ar[dr]_{u_{HA,HB}}&&&\\
 &C(H(A)\otimes H(B))\ar@{}[dl]^<>(.2){(c)}\ar[d]^{C(H(f)\otimes H(g))}\ar[rr]_<>(.5){C(\psi_{A,B})}&&CH(A\otimes B)=Y_{\Gamma}(A\otimes B)\ar@{}[dl]_<>(.6){(d)}\ar[d]_{}\ar[d]^{CH(f\otimes g)=Y_{\Gamma}(f\otimes g)}\\
 &C(H(A')\otimes H(B'))\ar[rr]^<>(.5){C(\psi_{A',B'})}&&CH(A'\otimes B')=Y_{\Gamma}(A'\otimes B')\\
 CH(A')\otimes CH(B')\makebox[0cm][l]{${}=Y_{\Gamma}(A')\otimes Y_{\Gamma}(B')$}\ar@{}[urr]|{(b)}\ar[rrru]_{y_{A',B'}}\ar[ur]^{u_{HA',HB'}}&&&
}
\]
(a) and (b) by definition of $\psi$.\\
(c) naturality of $u$ where $(C,u)$ is a monoidal functor.\\
The perimeter of the diagram commutes by naturality of $y$ where $(Y_{\Gamma},y)$ is a monoidal functor.
Using the fact that $u_{HA,HB}$ is an iso, all this implies that the interior square $(d)$ commutes.
Thus we obtain that $C(\psi_{A',B'})\circ C(H(f)\otimes H(g))=CH(f\otimes g)\circ C(\psi_{A,B})$ therefore since $C$ is faithful $\psi_{A',B'}\circ (H(f)\otimes H(g))=H(f\otimes g)\circ \psi_{A,B}$.

Now we want to prove that this natural transformation satisfies all the axioms of a monoidal structure. We start with the following axiom:
\begin{equation}
\xymatrix@=25pt{
 H(A)\otimes H(I)\ar[rr]^{\rho}\ar[d]^{1\otimes \Psi_{I}}&&H(A)\\
 H(A)\otimes H(I)\ar[rr]^{\psi_{AI}}&&H(A\otimes I)\ar[u]^{H(\rho)}
}\label{DIAGRAMA 1}
\end{equation}

This turns to be the following diagram:

\[
\xymatrix@=25pt@C-3ex{
CH(A)\otimes CH(I)\ar@{}[drrr]_{(d)}\ar[ddddr]_{1}&&&&&\\
&CH(A)\otimes C(I)\ar@{}[ddr]_{(c)}\ar[dr]^{u_{HA,I}}\ar[ul]^{C1\otimes C\psi_{I}}\ar[ddd]^{C1\otimes C\psi_{I}} &&&&CH(A)\otimes I\ar[llll]^{1\otimes u_{I}}\ar[d]^{\rho}\ar[lllllu]_{1\otimes y_{I}}\ar@{}[dll]_{(a)}\\
& &C(H(A)\otimes I))\ar@{}[drrr]^{(e)}\ar[d]^{C(1\otimes \psi_{I})}\ar[rrr]^{C(\rho)}&&&CH(A)\\
&&C(H(A)\otimes H(I))\ar[rrr]^{C\psi_{A,I}}&&&CH(A\otimes I)\ar[u]^{CH(\rho)=}\ar[u]_{Y_{\Gamma(\rho)}}\\
&CH(A)\otimes CH(I)\ar@{}[urr]|<>(.6){(b)}\ar[ur]^{u_{HA,HI}} \ar[urrrr]_{y_{A,I}}&&&&
}
\]
We use the same argument again and we show that it satisfies the required equation.\\
(a) $C$ is a monoidal functor.\\
(b) by definition of $\psi_{A,I}$.\\
(c) naturality of $u$ where $(C,u)$ is a monoidal functor.\\
(d) definition of $\psi_{I}$.\\
The exterior diagram commutes because $(Y_{\Gamma},y)$ is a monoidal functor. Using the fact that $u_{HA,I}$ is an iso, all this implies that the interior square $(e)$ commutes. Again, since $C$ is faithful we get $H(\rho)\circ\psi_{A,I}\circ(1\otimes  \psi_{I})=\rho$ which is diagram~(\ref{DIAGRAMA 1}).

In the same way we can verify that $H(\lambda)\circ\psi_{I,A}\circ(\psi_{I}\otimes 1)=\lambda$.

Now we move to proving the associativity axiom.
\begin{small}
\[\hspace{-80pt}\raisebox{1.6in}[0in][4.6in]{x$
\xymatrix@dr@=65pt{
&&CHA\otimes (CHA'\otimes CHA'')\ar@{}[ddr]|{(a)}\ar[ddll]_{1\otimes y_{A',A''}}\ar[ddl]_{1\otimes u}\ar[d]_{1\otimes u}\ar[r]^{\alpha}&(CHA\otimes CHA')\otimes CHA''\ar[d]_{u\otimes 1}\ar@/^4ex/@<1ex>[ddr]_{u\otimes 1}\ar@<2ex>[rrdd]^{y_{A,A'}\otimes 1}&&\\
&&CHA\otimes C(HA'\otimes HA'')\ar[d]_{u}&C(HA\otimes HA')\otimes CHA'')\ar[d]_{u}&&\\
CHA\otimes CH(A'\otimes A'')\ar[rd]_{1}&CHA\otimes C(HA''\otimes HA'')\ar@{}[l]_<>(.4){(i)}\ar[d]|{1\otimes C(\psi_{A',A''})}\ar[r]^{u}&C(HA\otimes (HA'\otimes HA''))\ar@{}[ddr]|{(k)}\ar@{}[dl]_{(b)}\ar[d]|{C(1\otimes \psi_{A',A''})}\ar[r]^{C(\alpha)}\ar@{}[ul]|<>(.3){(f)}&C((HA\otimes HA')\otimes HA'')\ar@{}[ur]^{(j)}\ar[d]|{C(\psi_{A,A'}\otimes 1)}&C(HA\otimes HA')\otimes CHA''\ar@{}[u]_{(g)}\ar@{}[dl]_{(h)}\ar[l]_{u}\ar[d]|{C(\psi_{A,A'})\otimes 1}&CH(A\otimes A')\otimes CHA''\ar[dl]_{1}\\
&CHA\otimes CH(A'\otimes A'')\ar[rd]_{y_{A,A'\otimes A''}}\ar[r]^{u}\ar@{}[drr]|<>(.2){(c)}&C(HA\otimes H(A'\otimes A''))\ar[d]|{C(\psi_{A,A'\otimes A''})}&C(H(A\otimes A')\otimes HA'')\ar[d]|{C(\psi_{A\otimes A',A''})}&CH(A\otimes A')\otimes CHA''\ar[dl]_{y_{A\otimes A',A''}}\ar@{}[dll]_<>(.25){(d)}\ar[l]_{u}&\\
&&CH(A\otimes (A'\otimes A''))\ar[r]_{CH(\alpha)}&CH((A\otimes A')\otimes A'')&&
}$}
\]
\end{small}
The goal is to prove that the diagram $(k)$ commutes. We have that:\\
\begin{itemize}
\item[-] (a): $(C,u)$ is a monoidal functor.
\item[-] (b) and (h): $u$ is natural with $\Psi$ and $1$.
\item[-] (c) and (d): definition of $\psi_{A,A'\otimes A''}$ and $\psi_{A\otimes A',A''}$.
\item[-] (i) and (g): definition of $\psi_{A', A''}$ and $\psi_{A,A'}$ and functoriality of the tensor.
\item[-] (f) and (j): are equal.
\item[-] The exterior diagram commutes because $(Y_{\Gamma},y)$ is a monoidal functor.
\end{itemize}
Since $1\otimes u_{HA',HA''}$ and $u_{HA,HA'\otimes HA''}$ are isos it is enough to check that\\
(top leg of (h))${}\circ u\circ (1\otimes u)=$(bottom leg of (h))${}\circ u\circ (1\otimes u)$. Then we use the fact that $C$ is a faithful functor.

\begin{remark}
Notice that since $C$ and $Y_{\Gamma}$ are fully faithful functor, $H$ is fully-faithful as well.
\end{remark}

\subsection{$H$ preserves coproducts}
\label{H PRESERVES COPRODUCTS}
In this section we focus on the specific problem of the preservation of finite coproducts of the functor $H$ defined in Section~\ref{CANONICAL CHOICE}. First, we notice that the category $[\cC^{op},{\bf Set}]$ has finite coproducts. These coproducts are computed pointwise: if $F$ and $G$ are in $[\cC^{op},{\bf Set}]$ then
$(F\oplus G)(C)=F(C)\oplus G(C)$ for every $C\in\cC$ and with injections as in the category ${\bf Set}$.

Also these coproducts are preserved going to the category $[\cD^{op},{\bf Set}]_{\Gamma}$ via the left adjoint $F=F_2\circ F_1$, where $F_2=R$ is the left adjoint of the reflection determined by the class $\Gamma$. The coproducts in  $[\cD^{op},{\bf Set}]_{\Gamma}$ are induced by this reflection $i\vdash R$.
More precisely: $A\oplus_{\Gamma}B=R(i(A)\oplus i(B))$ and $in_{\Gamma}=R(in)$ where $A$ and $B$ are in $[\cD^{op},{\bf Set}]_{\Gamma}$. Then, it makes sense to think about finite coproducts in $[\cD^{op},{\bf Set}]_{\Gamma}$.

Finally, the Kleisli category $\hat{\cC}_{T}$ inherits the coproduct structure from $\hat{\cC}$ as we proved in Section~\ref{COPRODUCT INDUCED KLEISLI}. Therefore, $[\cC^{op},{\bf Set}]_{T}$ has finite coproducts.
Recall that the comparison functor
$$C:[\cC^{op},{\bf Set}]_{T}\rightarrow[\cD^{op},{\bf Set}]_{\Gamma}$$ is fully faithful.
Also, by Corollary~\ref{COPRODUCT CONT EQUIVALENCE},
$H:\cD\rightarrow [\cC^{op},{\bf Set}]_{T}$ preserves coproducts iff $[\cC^{op},{\bf Set}]_{T}(H-,A):\cD^{op}\rightarrow{\bf Set}$ preserves products for every $A\in [\cC^{op},{\bf Set}]_{T}$.

But, we have
$$[\cC^{op},{\bf Set}]_{T}(H-,A)\cong[\cD^{op},{\bf Set}]_{\Gamma}(CH-,CA).$$
More precisely, since $C$ is fully faithful then the following functors:
$$\cD^{op}\stackrel {H^{op}}\rightarrow\hat{\cC}_{T}^{op}\stackrel {\hat{\cC}_{T}(-,A)}\rightarrow{\bf Set}$$
and
$$\cD^{op}\stackrel {(CH)^{op}}\rightarrow\hat{\cD}_{T}^{op}\stackrel {\hat{\cD}_{\Gamma}(-,CA)}\rightarrow{\bf Set}$$
are naturally isomorphic.

Therefore, we have:
$$[\cD^{op},{\bf Set}]_{\Gamma}(CH-,CA)=[\cD^{op},{\bf Set}]_{\Gamma}(Y_{\Gamma}-,CA)$$
because $CH=Y_{\Gamma}$. Also,
$$[\cD^{op},{\bf Set}]_{\Gamma}(Y_{\Gamma}-,CA)\cong[\cD^{op},{\bf Set}](Y-,CA)$$
because $[\cD^{op},{\bf Set}]_{\Gamma}$ is a full subcategory of $[\cD^{op},{\bf Set}]$
and $Y_{\Gamma}-=Y-$ evaluated on $-$. Finally,
$$[\cD^{op},{\bf Set}](Y-,CA)\cong CA$$
holds because by the Yoneda Lemma~\ref{YONEDA LEMMA} there is a bijection which is natural in $-$, i.e., these functors are naturally isomorphic.
But $CA\in [\cD^{op},{\bf Set}]_{\Gamma}$ for every $A\in [\cC^{op},{\bf Set}]_{T}$, which by definition means that $CA$ satisfies the property of continuity i.e., preserves all the cylinders
and limit cones that are in $\Gamma$. In particular, since natural isomorphisms preserve limits, it will be enough to impose that condition on the class $\Gamma$. From this, we conclude that $\Gamma$ contains all the finite products. This is another requirement to obtain a model.

\section{$F_{T}\dashv G_{T}$ is a monoidal adjunction}
In this section we show how a monoidal adjoint pair $(F,m)\dashv
 (G,n)$ induces a monoidal structure for the adjunction $F_{T}\dashv
 G_{T}$ associated with the Kleisli construction, where $T=GF$.

 \begin{lemma}
   Let $F\dashv G$ be a monoidal adjunction, let $T=GF$, and consider
   the Kleisli adjunction $\xymatrix{\mathcal{C}\ar@<1ex>[r]^{F_{T}}&
     \mathcal{C}_{T}\ar@<1ex>[l]^{G_{T}}_{\bot}}$ as in
   Definition~\ref{Kleisli category}. Then $\cC_{T}$ is a monoidal
   category and $F_{T}\dashv G_{T}$ is a monoidal adjunction.
 \end{lemma}

 \begin{proof}
   Since $F\dashv G$ is a monoidal adjunction, it follows that $T=GF$
   is a monoidal monad. The result then follows from Lemma~\ref{2.3.2a}.
 \end{proof}

\section{Abstract model of the quantum lambda calculus}\label{MAIN THEOREM OF CHAPTER 4}
To sum up the sections of this chapter we have the following theorem.
\begin{theorem} \label{the theorem MAIN THEOREM OF CHAPTER 4}
Given categories $\cB$, $\cC$ and $\cD$, and functors $\Phi:\cB\rightarrow\cC$, and $\Psi:\cC\rightarrow\cD$, satisfying
\begin{itemize}
\item[-]$\cB$ has finite products, $\cC$ and $\cD$ are symmetric monoidal,
\item[-]$\cB$, $\cC$, and $\cD$ have coproducts, and they are distributive w.r.t. tensor,
\item[-]$\cC$ is affine,
\item[-]$\Phi$ and $\Psi$ are strong monoidal,
\item[-]$\Phi$ and $\Psi$ preserve coproducts,
\item[-]$\Phi$ is full and faithful,
\item[-]$\Psi$ is essentially surjective on objects,
\item[-]for every $b\in\cB$, $c,c'\in\cC$ we have
$$\cC(\Phi(b),c)\times\cC(\Phi(b),c')\cong\cC(\Phi(b),c\otimes c').$$
\end{itemize}

\noindent
Let $\Gamma$ be any class of cones preserved by the opposite tensor
functor, including all
the finite product cones and $Lan_{\Phi}$, $\Phi^{*}$, $F$ and $G$ be defined as
 in Section~\ref{OUTLINE OF THE PROCEDURE}. Then
 \[\xymatrix{
[\cB^{op},{\bf Set}]\ar@<1ex>[rr]^{Lan_{\Phi}}&&[\cC^{op},{\bf Set}]\ar@<1ex>[rr]^{F}\ar@<1ex>[ll]^{\Phi^{*}}_{\bot}&&[\cD^{op},{\bf Set}]_{\Gamma}\ar@<1ex>[ll]^{G}_{\bot}}\]

 \noindent
 forms an abstract model of the quantum lambda calculus.
\end{theorem}
 \begin{proof} Relevant propositions from
 sections~\ref{KAN EXTENSION},~\ref{DAY S CLOSED MONO CONVO},~\ref{THE REFLECTIVE SUBCAT SECTION},~\ref{Day s reflection theorem},~\ref{abstract4},~\ref{CAT MODELS OF LL},~\ref{abstract3},~\ref{abstract2},~\ref{A STRONGLY COMONAD}.
 \end{proof}
\chapter[A concrete model]
         {A concrete model of the quantum lambda calculus}\label{A CONCRETE MODEL}

\section{An example: $\textbf{Srel}_{fn}$}
Before we give the main model for higher-order quantum computation,
 it is instructive to consider a simpler model for higher-order {\em
 probabilistic} computation. In the sense of Section~\ref{OUTLINE OF THE PROCEDURE}, we let
 $\cD$ be the category ${\bf Srel}_{\it fn}$ of sets and stochastic
 relations, see Definition~\ref{STOCHASTIC RELATION}, and we let $\cC$ be the category of finite sets and
 functions. In this setting, we let ``$!$'' be the identity comonad,
 i.e., $\cB=\cC$. The latter is justified because in the context of
 classical probabilistic computation, there are no quantum types and
 no no-cloning property; all types are classical and hence ${!A}=A$.

 \begin{lemma}\label{DA}
   ${\bf Srel}_{fn}$ has finite coproducts, satisfying
 distributivity $(A\oplus B)\otimes C \cong A\otimes C\oplus B\otimes C$.
  \end{lemma}
 \begin{proof}
  The coproduct of two objects is given by their disjoint union,
 $A\oplus B=(A\times {1})\cup (B\times {2})$. Injections are given by the following stochastic
 maps: $i_1 : A \to A\oplus B$ and $i_2 : B \to A\oplus B$, where
 \[ i_{1}((x,j),y)=\left \{ \begin{array}{ll}
1& \textrm{if $j=1$ and $x=y$}\\
0& \textrm{otherwise.}
\end{array} \right.\]

  and
  \[ i_{2}((x,j),y)=\left \{ \begin{array}{ll}
1& \textrm{if $j=2$ and $x=y$}\\
0& \textrm{otherwise.}
\end{array} \right.\]
 It is easy to verify that these satisfy the required universal
 property. The natural map
 \[ d_{A,B,C}:A\otimes C\oplus B\otimes C \to (A\oplus B)\otimes C
 \]
 is defined as $[i_1\otimes C, i_2\otimes C]$. The map $d$ is easily seen to
 be a natural isomorphism by precomposing with injections $i_1$, $i_2$ and using the universal property for coproducts.
 \end{proof}

\begin{definition}\label{STRONG FUNCTOR ON SREL}
  Let
 $\Psi:{\bf FinSet}\to{\bf Srel}_{fn}$ be the functor that is the identity on objects, and
defined on morphisms by
 \[ \Psi(f)(x,y) = \left\{\begin{array}{ll}1&\mbox{if $x=y$}\\
                                           0&\mbox{otherwise.}
                   \end{array}\right.
 \]
\end{definition}
\begin{remark}\label{strong functor on srel preserves coproducts}
The functor $\Psi$ is strong monoidal and preserves coproducts.
\end{remark}

\begin{theorem}\label{srel-theorem}
The choice $\cB={\bf FinSet}$, $\cC={\bf FinSet}$,
$\cD={\bf Srel}_{fn}$ with $\Phi=id$ and $\Psi$ as in
Definition~\ref{STRONG FUNCTOR ON SREL}. Let $\Gamma$ be the class of all
 finite product cones in $\cD^{op}$. This choice satisfies all the properties required by the
Theorem~\ref{the theorem MAIN THEOREM OF CHAPTER 4}. Therefore, this gives an abstract model of the quantum lambda
calculus.
\end{theorem}
\begin{proof}By Lemma~\ref{DA}
 and Remark~\ref{strong functor on srel preserves coproducts}.
\end{proof}
\begin{remark}
 Such a model could be considered to be a
concrete model of ``probabilistic lambda calculus", i.e., of
higher-order probabilistic computation.
\end{remark}

 \begin{remark}
 By Lemma~\ref{DA}, the functor $-\otimes X$ preserves
 finite coproducts for all $X\in {\bf Srel}_{fn}$. It is possible to show
that
 this functor in fact preserves all existing colimits (due to the
 natural isomorphism $A\otimes X \cong A\oplus A\oplus\ldots\oplus A$,
$|X|$ times
 for any fixed $X$). Therefore, in Theorem~\ref{srel-theorem}, we could have
 alternatively defined $\Gamma$ to be the class of all limit cones. In
 fact, any class of limit cones that contains at least all finite
 product ones would do. Each such choice yields an a priori different
 model.
\end{remark}

\section{The category ${\bf Q}''$ and the functors $\Phi$ and $\Psi$}
\label{Q''CATEGORY AND FUNCTORS PSI AND PHI}\label{Q double prime AND CATEGORY AND FUNCTOR PSI}
Recall the definition of the
category ${\bf Q}$ of superoperators from Section~\ref{SUPEROPERATORS}.
 In this section, we discuss a category ${\bf Q}''$
 related to superoperators $\textbf{Q}$, together with functors ${\bf FinSet}\stt{\Phi}{\bf Q}''\stt{\Psi}\textbf{Q}$. Here, the goal is to choose ${\bf Q}''$ and the functors $\Phi$ and $\Psi$
 carefully so as to satisfy the requirement of Theorem~\ref{the theorem MAIN THEOREM OF CHAPTER 4}.

Recall the definition of the free affine monoidal category $\cF wm(\cK)$ from Section~\ref{FREE AFFINE MONOIDAL CATEGORY}.
We apply this universal construction to situation where $\cK$ is a discrete category. For later convenience, we let $\cK$ be the discrete category
with finite dimensional Hilbert spaces as objects.
Then $\cF wm(\cK)$ has sequences of Hilbert spaces as objects and dualized, compatible, injective functions as arrows:
\begin{itemize}
\item[-] objects: finite sequences of finite dimensional Hilbert spaces
\item[-] a morphism from $\{V_1, \ldots, V_n\}$ to $\{W_1, \ldots, W_m\}$ is given by an
injective function $f:\{1,\ldots, m\}\rightarrow \{1,\ldots, n\}$, such that for all $i,$
 $V_{f(i)} = W_i$.
\end{itemize}
\begin{remark}
Since the objects of ${\bf Q}$ and $\cF wm(\cK)$ are finite sequences of finite-dimensional Hilbert
spaces, and there are only countably many finite-dimensional Hilbert
spaces up to isomorphism, we may w.l.o.g. assume that ${\bf Q}$ and $\cF wm(\cK)$ are small categories.
\end{remark}

Now consider the identity-on-objects inclusion functor $F:\cK\rightarrow \cQ'_s$ where $\cQ'_s$ is the category of simple trace-preserving superoperator defined in Section~\ref{SUPEROPERATORS}. Since $\cQ'_s$ is affine, by Proposition~\ref{FREE WEEK SYM MON} there exists a unique (up to natural isomorphism) strong monoidal functor $\hat{F}$ such that:
$$\xymatrix@=25pt{
\cK\ar[d]_{I}\ar[r]^{F}&\cQ'_s\\
\cF wm(\cK)\ar[ru]_{\hat{F}}&
}$$

\begin{remark}
\rm
This reveals the purpose of using the equality instead of
 $\leq$ in the definition of a trace-preserving superoperator
 (Definition~\ref{TRACE PRESERVING CATEGORY}). When the codomain is the unit,
 there is only one map $f(\rho) = \tr(\rho)$, and therefore $\cQ'_s$ is
 affine.
\end{remark}

\begin{remark}\label{IN AND E EMBEDDING}
\rm

 By definition, $\textbf{Q}_s$ is a full subcategory of $\textbf{Q}$, and the
 inclusion functor $In:\textbf{Q}_s\rightarrow \textbf{Q}$ is strong monoidal.
 Also, since every trace preserving superoperator is trace non-increasing,
 $\cQ'_s$ is a subcategory of $\textbf{Q}_s$, and the inclusion functor $E:\cQ'_s\rightarrow \textbf{Q}_s$
 is strong monoidal as well.
\end{remark}

Then we apply the machinery of Proposition~\ref{MONOIDAL FUNCTOR PSI} to the functor:
$$\cF wm(\cK)\stackrel{\hat{F}}\rightarrow\cQ'_s\stackrel{E}\rightarrow \textbf{Q}_s\stackrel{In}\rightarrow \textbf{Q}.$$
where $In$ and $E$ are as defined in Remark~\ref{IN AND E EMBEDDING}.
\begin{definition}\label{DEFINITION OF THE FUNCTOR PSI}
Let $\textbf{Q}''=(\cF wm(\cK))^{+}$ and let $\Psi$ be the unique finite
coproduct preserving functor making the following diagram
commute:
\begin{equation}
\label{THE FUNCTOR PSI}
\xymatrix@=25pt{
\cF wm(\cK)\ar[d]_{I}\ar[r]^{\hat{F}}&\cQ'_s\ar[r]^{E}&\textbf{Q}_s\ar[r]^{In}&\textbf{Q}&\\
(\cF wm(\cK))^{+}\ar[rrru]_{\Psi}&&&
}
\end{equation}
Note that such a functor exists by Proposition~\ref{COCOMPLETION OF A CATEGORY}, and it is strong monoidal by Proposition~\ref{MONOIDAL FUNCTOR PSI}.
\end{definition}

\begin{remark}
\rm
Since
$$\Psi\{\{V_{i}^{a}\}_{i\in[n_a]}\}_{a\in A}=\coprod_{a\in A}\{(V_1^a\otimes\ldots\otimes V_{n_a}^a)_{*}\}_{*\in 1}$$
the functor $\Psi$ is essentially onto objects. Specifically,
 given any object $\{V_a\}_{a\in A}\in |{\bf Q}|$, we can choose a preimage
 (up to isomorphism) as follows:
\begin{equation}
\Psi\{\{V_{i}^{a}\}_{i\in[1]}\}_{a\in A}=\coprod_{a\in A}\{(V_1^a)_{*}\}_{*\in 1}\cong\{V_a\}_{a\in A}.
\end{equation}
\end{remark}

Here is the full picture of categories and functors:
$$\xymatrix@=25pt{
\cK\ar[r]_{F}\ar[d]_{I}&\cQ'_s\ar[r]^{E}&{\bf Q}_s\ar[d]^{In}\\
\cF wm(\cK)\ar[d]_{I}\ar[ru]_{\hat{F}}&&\textbf{Q}\\
(\cF wm(\cK))^{+}\ar[rru]_{\Psi}
}$$

 \begin{remark}\label{quequeco}
 Since $Fwm(\cK)$ is an affine category and ${\bf Q}''=Fwm(\cK)^{+}$, let us consider the functor $$\Phi:{\bf Finset}\rightarrow{\bf Q}''$$ defined by Lemma~\ref{fully-faithful strong monoidal functor PHI}.
\end{remark}
\begin{theorem}\label{THE MODEL}
The choice $\cB={\bf FinSet}$, $\cC={\bf Q}''$,
$\cD={\bf Q}$ with the functors $\Phi$ as in Remark~\ref{quequeco} and $\Psi$ as in Definition~\ref{DEFINITION OF THE FUNCTOR PSI} satisfies all the properties required by Theorem~\ref{the theorem MAIN THEOREM OF CHAPTER 4}.
\end{theorem}
\begin{proof}
By relevant propositions from Section~\ref{Q''CATEGORY AND FUNCTORS PSI AND PHI}.
\end{proof}

\section{A concrete model}
\begin{theorem}
 Let ${\bf Q}$ and ${\bf Q}''$ be defined as in Sections~\ref{SUPEROPERATORS} and~\ref{Q double prime AND CATEGORY AND FUNCTOR PSI}. Let $\Gamma$ be the class of all finite product cones in $\cD^{op}$ where $\cD={\bf Q}$. Then

\[\xymatrix{
[{\bf FinSet}^{op},{\bf Set}]\ar@<1ex>[rr]^{Lan_{\Phi}}&&[{\bf Q}''^{op},{\bf Set}]\ar@<1ex>[rr]^{F}\ar@<1ex>[ll]^{\Phi^{*}}_{\bot}&&[{\bf Q}^{op},{\bf Set}]_{\Gamma}\ar@<1ex>[ll]^{G}_{\bot}}\]
  forms a concrete model of the quantum lambda calculus.
 \end{theorem}
\begin{proof}
 The proof is by Theorem~\ref{THE MODEL}, by and by Theorem~\ref{the theorem MAIN THEOREM OF CHAPTER 4}.
 \end{proof}
\chapter{Conclusions and future work}

In the first part of this thesis, we established that the partially
 traced categories, in the sense of Haghverdi and Scott, are precisely
 the monoidal subcategories of totally traced categories.  This was
 proved by a partial version of Joyal, Street, and Verity's
 ``Int''-construction, and by considering a strict symmetric compact
 closed version of Freyd's paracategories.

 We also introduced some new examples of partially traced categories,
 in connection with some standard models of quantum computation such
 as completely positive maps and superoperators.

 One question that we did not answer is whether specific partially
 traced categories can always be embedded in totally traced categories
 in a ``natural'' way. For example, the category of finite dimensional
 vector spaces, with the biproduct $\oplus$ as the tensor, carries a  partial trace. By our proof, it follows that it can be
 faithfully embedded in a totally traced category. However, we do not
 know any concrete ``natural'' example of such a totally traced category
 (i.e., other than the free one constructed in our proof) in which it
 can be faithfully embedded.

 In the second part, we constructed mathematical (semantical) models
 of higher-order quantum computation, and more specifically, for the
 quantum lambda calculus of Selinger and Valiron. The central idea of
 our model construction was to apply the presheaf construction to a
 sequence of three categories and two functors, and to find a set of
 sufficient conditions for the resulting structure to be a valid
 model. The construction depends crucially on properties of presheaf
 categories, using Day's convolution theory, Lambek's modified Yoneda
 embedding, and Kelly and Freyd's notion of continuity of functors.

 We then identified specific base categories and functors which
 satisfy these abstract conditions, based on the category of
 superoperators. Thus, our choice of base categories ensures that the
 resulting model has the ``correct'' morphisms at base types, whereas
 the presheaf construction ensures that it has the ``correct''
 structure at higher-order types.

 Our work has concentrated solely on the existence of such a model.
 One question that we have not yet addressed is specific properties of
 the interpretation of quantum lambda calculus in this model. It would
 be interesting, in future work, to analyze whether this particular
 interpretation yields new insights into the nature of higher-order
 quantum computation, or to use this model to compute properties of
 programs.

\end{document}